\documentclass[12pt, reqno]{amsart}

\usepackage{amssymb}
\usepackage{amsthm}
\usepackage{amscd}
\usepackage{longtable,array,booktabs}

\usepackage{mathtools}
\mathtoolsset{showonlyrefs}
\numberwithin{equation}{section}

\usepackage{mathrsfs}

\usepackage{hyperref, graphicx}

\usepackage{amsmath}
\usepackage{epic,eepic}
\usepackage{graphicx}

\DeclareMathAlphabet{\mathpzc}{OT1}{pzc}{m}{it}

\usepackage{latexsym}

\usepackage{pifont}

\theoremstyle{plain}
\newtheorem{lemma}{Lemma}[section]
\newtheorem{prop}[lemma]{Proposition}
\newtheorem{thm}[lemma]{Theorem}
\newtheorem{cor}[lemma]{Corollary}
\newtheorem{aplemma}{Lemma~A.\hspace{-1.5mm}}
\newtheorem{approp}{Proposition~A.\hspace{-1.5mm}}
\newtheorem{apthm}{Theorem~A.\hspace{-1.5mm}}
\newtheorem{apcor}{Corollary~A.\hspace{-1.5mm}}
\newtheorem{intthm}{Theorem}
\newtheorem{prdef}[lemma]{Proposition-Definition}
\newtheorem{tdef}[lemma]{Theorem-Definition}

\usepackage[all]{xy}

\newtheorem{conj}[lemma]{Conjecture}

\theoremstyle{definition}

\newtheorem{rema}[lemma]{Remark}

\newtheorem{remb}{Remark}

\newtheorem{defi}[lemma]{Definition}
\newtheorem{exa}[lemma]{Example}
\newtheorem{aprem}{Remark~A.\hspace{-1.5mm}}
\newtheorem{apdefi}{Definition~A.\hspace{-1.5mm}}
\newcommand{\bde}{\begin{defi}}
\newcommand{\ede}{\end{defi}\vspace{1mm}}
\newcommand{\ble}{\begin{lemma}}
\newcommand{\ele}{\end{lemma}}
\newcommand{\bpr}{\begin{prop}}
\newcommand{\epr}{\end{prop}}
\newcommand{\bt}{\begin{thm}}
\newcommand{\et}{\end{thm}}
\newcommand{\bco}{\begin{cor}}
\newcommand{\eco}{\end{cor}}
\newcommand{\bre}{\begin{rema}}
\newcommand{\ere}{\end{rema}}
\newcommand{\brea}{\begin{rema}}
\newcommand{\erea}{\end{rema}\vspace{1mm}}
\newcommand{\breb}{\begin{remb}}
\newcommand{\ereb}{\end{remb}\vspace{1mm}}
\newcommand{\bex}{\begin{exa}}
\newcommand{\eex}{\end{exa}}
\newcommand{\bpf}{\begin{proof}}
\newcommand{\epf}{\end{proof}\vspace{1mm}}

\newcommand{\bade}{\begin{apdefi}}
\newcommand{\eade}{\end{apdefi}}
\newcommand{\bale}{\begin{aplemma}}
\newcommand{\eale}{\end{aplemma}}
\newcommand{\bapr}{\begin{approp}}
\newcommand{\eapr}{\end{approp}}
\newcommand{\bat}{\begin{apthm}}
\newcommand{\eat}{\end{apthm}}
\newcommand{\baco}{\begin{apcor}}
\newcommand{\eaco}{\end{apcor}}
\newcommand{\bare}{\begin{aprem}}
\newcommand{\eare}{\end{aprem}}

\newcommand{\be}{\begin{enumerate}}
\newcommand{\ee}{\end{enumerate}}
\newcommand{\bcd}{\[\begin{CD}}
\newcommand{\ecd}{\end{CD}\]}
\newcommand{\bit}{\begin{itemize}}
\newcommand{\eit}{\end{itemize}}
\newcommand{\bq}{\begin{quote}}
\newcommand{\eq}{\end{quote}}
\newcommand{\ba}{\begin{array}}
\newcommand{\ea}{\end{array}}
\newcommand{\mcA}{\mathcal{A}}
\newcommand{\mcB}{\mathcal{B}}
\newcommand{\mcC}{\mathcal{C}}
\newcommand{\mcD}{\mathcal{D}}
\newcommand{\mcE}{\mathcal{E}}
\newcommand{\mcF}{\mathcal{F}}
\newcommand{\mcG}{\mathcal{G}}
\newcommand{\mcH}{\mathcal{H}}
\newcommand{\mcI}{\mathcal{I}}
\newcommand{\mcJ}{\mathcal{J}}
\newcommand{\mcK}{\mathcal{K}}
\newcommand{\mcL}{\mathcal{L}}
\newcommand{\mcM}{\mathcal{M}}
\newcommand{\mcN}{\mathcal{N}}
\newcommand{\mcO}{\mathcal{O}}
\newcommand{\mcP}{\mathcal{P}}
\newcommand{\mcQ}{\mathcal{Q}}
\newcommand{\mcR}{\mathcal{R}}
\newcommand{\mcS}{\mathcal{S}}
\newcommand{\mcT}{\mathcal{T}}
\newcommand{\mcU}{\mathcal{U}}
\newcommand{\mcV}{\mathcal{V}}
\newcommand{\mcW}{\mathcal{W}}

\newcommand{\mcZ}{\mathcal{Z}}

\newcommand{\mbC}{\mathbb{C}}

\newcommand{\mbF}{\mathbb{F}}
\newcommand{\mbG}{\mathbb{G}}
\newcommand{\mbH}{\mathbb{H}}

\newcommand{\mbM}{\mathbb{M}}
\newcommand{\mbN}{\mathbb{N}}

\newcommand{\mbP}{\mathbb{P}}
\newcommand{\mbQ}{\mathbb{Q}}
\newcommand{\mbR}{\mathbb{R}}
\newcommand{\mbS}{\mathbb{S}}

\newcommand{\mbY}{\mathbb{Y}}
\newcommand{\mbZ}{\mathbb{Z}}

\newcommand{\mfS}{\mathfrak{S}}

\newcommand{\mfa}{\mathfrak{a}}
\newcommand{\mfb}{\mathfrak{b}}

\newcommand{\mfe}{\mathfrak{e}}

\newcommand{\mfg}{\mathfrak{g}}

\newcommand{\mfl}{\mathfrak{l}}
\newcommand{\mfm}{\mathfrak{m}}

\newcommand{\mfo}{\mathfrak{o}}

\newcommand{\mfs}{\mathfrak{s}}
\newcommand{\mft}{\mathfrak{t}}

\newcommand{\mpa}{\mathpzc{a}}

\newcommand{\mpD}{\mathpzc{D}}

\newcommand{\msD}{\mathscr{D}}
\newcommand{\msE}{\mathscr{E}}
\newcommand{\msF}{\mathscr{F}}
\newcommand{\msG}{\mathscr{G}}

\newcommand{\msL}{\mathscr{L}}

\newcommand{\msN}{\mathscr{N}}
\newcommand{\msO}{\mathscr{O}}
\newcommand{\msP}{\mathscr{P}}

\newcommand{\msU}{\mathscr{U}}
\newcommand{\msV}{\mathscr{V}}

\newcommand{\msX}{\mathscr{X}}
\newcommand{\msY}{\mathscr{Y}}

\newcommand{\migi}{\rightarrow}

\newcommand{\isom}{\stackrel{\sim}{\migi}}

\newcommand{\migiincl}{\hookrightarrow}

\newcommand{\migisurj}{\twoheadrightarrow}

\newcommand{\N}{N}
\newcommand{\M}{m}

\newcommand{\mr}{\mathrm}
\newcommand{\hidden}[1]{\,}

\newcommand{\SSP}{\vspace{0mm}}
\newcommand{\LSP}{\vspace{0mm}}

\newcommand{\Fus}{\rotatebox[origin=c]{180}{$\mbY$}}

\newcommand{\qq}{q_1\hspace{-4.5mm}\rotatebox[origin=c]{-10}{$\swarrow$}\hspace{0.5mm}}
\newcommand{\Dual}{\rotatebox[origin=c]{180}{$C$}\hspace{-0.5mm}}

\newcommand{\GR}{{^{\rotatebox[origin=c]{-20}{$\vartriangle$}}}\hspace{-2.6mm}G}

\newcommand{\DDual}{\rotatebox[origin=c]{180}{$D$}\hspace{-0.5mm}}

\newcommand{\ZZZ}{{^{\mr{Zzz...}}}}
\newcommand{\sss}{\spadesuit}
\newcommand{\ST}{\blacklozenge}

\newcommand{\Diag}{\rotatebox[origin=c]{45}{$\Leftarrow$}}
\newcommand{\Diagg}{\rotatebox[origin=c]{45}{$\Rightarrow$}}
\newcommand{\DMO}{\nabla}
\newcommand{\STR}{\phi}
\newcommand{\CH}{\ell}
\newcommand{\CL}{L}
\newcommand{\EX}{d}
\newcommand{\IN}{a}
\newcommand{\BB}{\varTheta}

\newcommand{\bb}{\vartheta}
\newcommand{\CC}{{^{\rotatebox{180}{{\tiny  $\clubsuit$}}}}}
\newcommand{\ep}{a}

\usepackage{bm}
\usepackage {cancel}

\begin{document}

\title[Arithmetic liftings and $2$d TQFT for dormant opers  of higher level]{
Arithmetic liftings and $2$d TQFT \\ for dormant opers  of higher level}
\author{Yasuhiro Wakabayashi}
\date{}
\markboth{Yasuhiro Wakabayashi}{}
\maketitle
\footnotetext{Y. Wakabayashi: 
Graduate School of Information Science and Technology, The University of Osaka, Suita, Osaka 565-0871, Japan;}
\footnotetext{e-mail: {\tt wakabayashi@ist.osaka-u.ac.jp};}
\begin{abstract}

This manuscript represents an advance in the enumerative geometry of opers that takes the subject beyond our previous work. Motivated by a counting problem of linear differential equations in positive characteristic, we investigate the moduli space of opers from arithmetic and combinatorial points of view. We construct and investigate the moduli space classifying dormant $\mathrm{PGL}_n^{(N)}$-opers (i.e., dormant $\mathrm{PGL}_n$-opers of level $N$) on pointed stable curves in characteristic $p>0$. One of the key results is the generic \'{e}taleness of that space for $n=2$, which is proved by obtaining a detailed understanding of relevant deformation spaces. This fact induces a certain arithmetic lifting of each dormant $\mathrm{PGL}_2^{(N)}$-oper on a general curve to characteristic $p^N$; this lifting is called the canonical diagonal lifting. On the other hand, the generic \'{e}taleness also implies that the degree function for the moduli spaces in the rank $2$ case  satisfies factorization properties determined by various gluing  morphisms of the underlying curves. That is to say, the degree function forms a $2$d TQFT (= a $2$-dimensional topological quantum field theory); it leads us to describe dormant $\mathrm{PGL}_2^{(N)}$-opers in terms  of edge numberings on trivalent graphs, as well as lattice points inside generalized rational polytopes. These results yield an effective way of computing the numbers of such objects and $2$nd order differential equations in characteristic $p^N$ with a full set of solutions.

\end{abstract}
\setcounter{tocdepth}{2}
\tableofcontents 

\section{Introduction}

\LSP
\subsection{Linear differential equations with a full set of solutions} \label{WW44}

 Let us consider the  linear differential equation  $Dy =0$ 
 on a smooth complex algebraic curve $X$ 
 associated to an operator  $D$ 
 expressed locally  as 
\begin{align} \label{ee650}
D := \frac{d^n}{d x^n}  +q_1 \frac{d^{n-1}}{d x^{n-1}} + \cdots + q_{n-1} \frac{d}{dx} + q_n
\end{align}
($n >1$).
Here,   $x$ denotes a local coordinate in $X$ and  $q_1, \cdots, q_n$ are variable coefficients.
To each such  differential equation,
 one can  associate
a flat connection on a vector bundle whose matrix representation is locally given by 
\begin{align} \label{ee300}
\nabla = \frac{d}{dx} -   
\begin{pmatrix} 
-q_1 & -q_2  & -q_3 & \cdots  & -q_{n-1} & -q_n \\
 1 & 0 & 0 & \cdots    & 0 & 0 \\
  0 & 1 & 0&  \cdots   & 0 & 0 \\
   0 & 0 & 1&  \cdots  & 0 & 0 \\ 
  \vdots    & \vdots  & \vdots & \ddots  &  \vdots & \vdots \\ 
       0 & 0  & 0 & \cdots    & 1 & 0
     \end{pmatrix}.
\end{align}

A vector bundle equipped with such a flat connection is known as  a {\it $\mr{GL}_n$-oper}.
 Conversely, any flat bundle (i.e., a vector bundle equipped with a flat connection) defining a $\mr{GL}_n$-oper becomes a flat connection of that  form after possibly carrying out a suitable gauge transformation.
See ~\cite{BD1} or ~\cite{BD2} for the precise  definition of a $\mr{GL}_n$-oper, or more generally a $G$-oper for an algebraic group $G$ of a certain sort.

Note that the assignment $y \mapsto {^t(} \frac{d^{n-1}y}{dx^{n-1}}, \cdots, \frac{dy}{dx},  y)$ gives  a bijective  correspondence between  
 solutions of the equation $D y=0$ with $D$ as above 
and  horizontal  sections of
the associated 
$\mr{GL}_n$-oper.
If the associated connection has only regular singularities on a smooth projective compactification of $X$, then the equation has a full set of algebraic solutions  (i.e., has $n$ linearly independent algebraic solutions)  if and only if the corresponding  $\mr{GL}_n$-oper
has  finite monodromy.


The problem of classifying and counting  linear ODE's
with a full set of algebraic  solutions
  has long been one of the fundamental topics in mathematics.
The study of them
  was 
tackled and developed  from  the 
19th century onwards 
 by many mathematicians:
 H. A. Schwarz (for the hypergeometric equations), L. I. Fuchs, P. Gordan, and C. F.  Klein (for the second order equations), C. Jordan (for the $n$-th order equations) et al.

Regarding  the relationship 
with the case of characteristic $p$ (where $p$ is a prime number),
there is a well-known conjecture, i.e., the so-called Grothendieck-Katz $p$-curvature conjecture; it  
predicts  the existence of 
a full set of algebraic solutions  of a given linear differential equation
 in terms of  its reductions modulo $p$ for various primes $p$.
See, e.g.,   ~\cite{And} or  ~\cite{Kat4} for  detailed accounts of this topic.

\LSP
\subsection{Counting problem of dormant  opers} \label{WW60}

We want to focus on  the situation  where the underlying curve has  prime-power characteristic $p^{\N}$ ($\N \in \mbZ_{>0}$).
One central theme of our study stated in the most primitive form  is 
to answer the following natural question:
\begin{quote}
{\it How many homogeneous  linear differential equations in characteristic $p^{\N}$ 
 (associated to  differential operators $D$ as in  \eqref{ee650})  have a full set of solutions?}
\end{quote}
We can find some previous studies related to this question  for $\N =1$  under the correspondence with $\mr{GL}_n$-opers
  (cf. ~\cite{Ihara1}, ~\cite{JP},   ~\cite{Jo14}, \  ~\cite{LP}, ~\cite{Mzk2}, ~\cite{O4}, ~\cite{Wak}, ~\cite{Wak3}, ~\cite{Wak2}, and ~\cite{Wak8}). 
 Here, note that one may define the notion of a $G$-oper (for an arbitrary $G$)  on a curve in positive characteristic 
 because of the  algebraic nature of  its formulation.
For example, $G$-opers in characteristic $p$ have been investigated in the context of $p$-adic Teichm\"{u}ller theory.

One of the common key ingredients in these developments is the study of {\it $p$-curvature}.
The $p$-curvature of a flat connection in characteristic $p$ is an invariant that measures the obstruction to the compatibility of $p$-power structures appearing in certain associated spaces of infinitesimal symmetries.
A $G$-oper  is called {\it dormant} if it has vanishing $p$-curvature.
It follows from a classical result by Cartier (cf. ~\cite[Theorem (5.1)]{Kal})   that 
a $\mr{GL}_n$-oper,  or a $\mr{PGL}_n$-oper,   in characteristic $p$ is dormant if and only if it arises from a differential  equation $D y=0$ having  a full set of  solutions.
Thus, the above question  for $\N =1$ 
can be 
reduced to asking  
the number of 
all possible  dormant $\mr{GL}_n$-opers or their projectivizations, i.e., dormant $\mr{PGL}_n$-opers.

In what follows, we briefly review 
previous results 
concerning the counting problem of dormant opers.
  Let $X$ be a connected proper  smooth curve of genus $g >1$ over an algebraically closed field of  characteristic $p >2$.
In the case of $g=2$,
S. Mochizuki (cf. ~\cite[Chapter V, Corollary 3.7]{Mzk2}), H. Lange-C. Pauly (cf. ~\cite[Theorem 2]{LP}), and B. Osserman (cf. ~\cite[Theorem 1.2]{O4})
computed, by applying different methods,   
the total number of dormant $\mr{PGL}_2$-opers on a general $X$, as follows:
 \begin{align} \label{fffrr}
 \sharp \left\{ \begin{matrix} \text{dormant $\mr{PGL}_2$-opers on $X$} 
 \end{matrix}\right\}
 =
 \frac{p^3-p}{24}.
 \end{align}
For example,  this computation was made by using a correspondence with the base locus of the Verschiebung rational map on the moduli space of rank $2$ semistable  bundles (cf. Section \ref{SS101}). 

Moreover, by extending the relevant formulations to the case where $X$ admits marked points or  nodal singularities,
   S. Mochizuki also gave
  an explicit description of dormant $\mr{PGL}_2$-opers on each totally degenerate curve
   in terms of  {\it radius}   (cf. ~\cite[Introduction, Theorem 1.3]{Mzk2}).
The notion of radius  is an invariant determining a sort of boundary condition at a marked point to glue together dormant opers  in accordance with  the attachment of underlying curves at this point.
Mochizuki's description   is essentially equivalent to  a previous result obtained by Y. Ihara (cf. ~\cite[Section 1.6]{Ihara1}), who investigated the situation when a given Gauss hypergeometric differential equation in characteristic $p$ had a full set of solutions.
  As a result of this  explicit description, we  obtained 
   a combinatorial procedure for explicitly computing   the number  of dormant $\mr{PGL}_2$-opers.

This description also   leads to a work by
  F. Liu and B. Osserman (cf. ~\cite[Theorem 2.1]{LO}), who  have shown  that
the value in question
may be expressed as  a degree $3g-3$ polynomial with respect to ``$p$''.
They did so by applying
Ehrhart's theory, which concerns computing   the cardinality of the set of lattice points inside a polytope.

For general rank cases,
K. Joshi and C. Pauly showed that there are only finitely many  dormant $\mr{PGL}_n$-opers  on a fixed curve (cf. ~\cite[Corollary 6.1.6]{JP}), and  later 
  Joshi conjectured an explicit description 
of their total  number (cf. ~\cite[Conjecture 8.1]{Jo14}).
As a consequence of developing  the moduli theory of dormant $G$-opers on pointed stable curves,
we solved affirmatively this conjecture for a general curve (cf.  ~\cite[Theorem H]{Wak8}), which  is described as 
\begin{align} \label{eeQQ782}
 \sharp \left\{ \begin{matrix} \text{dormant $\mr{PGL}_n$-opers on} \\ \text{a general stable curve} \\ \text{of genus $g$ in characteristic $p$}\end{matrix}\right\}
 = \frac{ p^{(n-1)(g-1)-1}}{n!} \cdot  \hspace{-5mm}
 \sum_{\genfrac{.}{.}{0pt}{}{(\zeta_1, \cdots, \zeta_n) \in \mbC^{\times n} }{ \zeta_i^p=1, \ \zeta_i \neq \zeta_j (i\neq j)}}
 \frac{(\prod_{i=1}^n\zeta_i)^{(n-1)(g-1)}}{\prod_{i\neq j}(\zeta_i -\zeta_j)^{g-1}} \hspace{3mm}
 \end{align}
 for $p > n \cdot \mr{max}\{ g-1, 2 \}$.
 Based on the idea of Joshi
 et al. (cf. ~\cite{JP}, ~\cite{Jo14}) and a work  by Holla (cf. ~\cite{Hol}),
this formula was proved  
   by establishing
 a relationship  with the Gromov-Witten theory of Grassmann varieties  
 and then applying an explicit computation of their Gromov-Witten invariants, called  the Vafa-Intriligator formula.

When $n=2$,  we can extend this result to {\it pointed stable} curves
by establishing the relationship   with   the $\mfs \mfl_2 (\mbC)$-WZW (= Wess-Zumino-Witten) conformal field theory (cf. ~\cite[Theorem 7.41]{Wak8}).
As a result, the Verlinde formula for that CFT yields  the following formula: 
 \begin{align} \label{QH882}
 \sharp \left\{ \begin{matrix} \text{dormant $\mr{PGL}_2$-opers on} \\ \text{a general $r$-pointed  stable curve} \\ \text{of genus $g$ in characteristic $p$}\end{matrix}\right\}
 =
  \frac{p^{g-1}}{2^{2g-1+r}} \cdot \sum_{j =1}^{p-1} 
\frac{\left(1-(-1)^j \cdot \cos \left(\frac{j\pi}{p} \right)\right)^r}{\sin^{2(g-1+r)} \left(\frac{j \pi}{p} \right)},
 \end{align}
 which is consistent with \eqref{eeQQ782} in the case of  $r = 0$ and $n=2$.

Including such discussions,  we refer to the endeavor of describing the behavior of the moduli space under deformations and degenerations of the underlying algebraic curve, based on various methods and perspectives originating from $p$-adic Teichm\"{u}ller theory, while establishing  correspondences with other enumerative geometries to explicitly address the aforementioned problem, as the ``{\bf enumerative geometry of dormant opers}." Our ultimate goal is to further develop this theory.

\LSP
\subsection{Generalization to characteristic $p^\N$} \label{WW61}

The purpose of this manuscript is  
 to develop   the enumerative geometry of dormant $\mr{PGL}_n$-opers in prime-power characteristic so that \eqref{eeQQ782} and \eqref{QH882} are generalized.
From now on, let $X$ be  a geometrically connected, proper, and smooth curve in characteristic $p^\N$, where $\N$ is a positive integer.
Just as in  the case of characteristic $p$, 
an oper, or more generally a flat bundle, on $X$
   will be called {\it dormant} if  it is spanned by its horizontal sections, i.e., isomorphic locally  to the trivial flat bundle (cf. Definitions \ref{D019},  \ref{DD45YY}, and Proposition \ref{P016dd}, (ii)).
 In particular,
\begin{quote}
{\it Our central theme posed earlier can be rephrased  essentially as the issue of counting  dormant opers on a (general) curve in characteristic $p^\N$.}
\end{quote}
(Note that we also discuss, at the same time,  the case where the underlying curve has nodal singularities or marked points; in such a generalized situation, the dormancy condition has to be formulated in a different and more complicated manner.)

Unfortunately, the lack of a reasonable  invariant
exactly like the $p$-curvature for $\N =1$
makes it difficult to handle with dormant opers in characteristic $p^\N$.
To overcome this difficulty, we  relate, partly on the basis  of  the argument of ~\cite[Section 2.1, Chapter II]{Mzk2},  dormant opers on $X$ to certain objects defined on the mod $p$ reduction $X_0$ of $X$;
we shall refer to the operation resulting from  this argument as   {\bf diagonal reduction/lifting}.

One remarkable observation is that each dormant $\mr{GL}_n$-oper $\msE^\spadesuit$ on $X$ induces,
via diagonal reduction, a dormant $\mr{GL}_n$-oper ${^{\Diag}}\!\!\msE^\spadesuit$ on $X_0$ equipped with a $\mcD^{(\N-1)}$-module structure extending its flat structure.
Here, $\mcD^{(\N -1)}$ denotes the sheaf of differential operators of level $\N -1$ introduced by P. Berthelot in ~\cite{PBer1} and ~\cite{PBer2}.
A dormant $\mr{GL}_n$-oper 
on $X_0$ equipped  with a compatible $\mcD^{(\N -1)}$-module structure whose $p^\N$-curvature vanishes  is called a {\bf dormant $\mr{GL}_n$-oper of level $\N$},
 or  a {\bf  dormant  $\mr{GL}_n^{(\N)}$-oper}, 
  for short.
Similarly, we obtain  the notion of a dormant $\mr{PGL}_n^{(\N)}$-oper such that $\mr{PGL}_n^{(1)}$-opers are equivalent to  $\mr{PGL}_n$-opers in the classical sense.
See Definitions \ref{D39}, \ref{DD45YY}, \ref{D01234}, and \ref{D50} for their precise definitions.

We expect that, for a general curve $X$,  the assignment $\msE^\spadesuit \mapsto {^{\Diag}}\!\!\msE^\spadesuit$ given by taking the diagonal reductions is {\it invertible}, i.e., 
 each dormant $\mr{PGL}_n^{(\N)}$-oper $\msE_0^\spadesuit$ on $X_0$ may be  lifted {\it uniquely} to a dormant $\mr{PGL}_n$-oper ${^{\Diagg}}\!\!\msE_0^\spadesuit$ on $X$.
One of  the consequences in this manuscript shows    that this  is in fact true for $n=2$.
  The results of the lifting $\msE_0^\spadesuit \mapsto {^{\Diagg}}\!\!\msE_0^\spadesuit$   will be called  the {\bf canonical diagonal liftings}.
 The bijective correspondence given by $\msE^\spadesuit \mapsto {^{\Diag}}\!\!\msE^\spadesuit$ and $\msE_0^\spadesuit \mapsto {^{\Diagg}}\!\!\msE_0^\spadesuit$ enables  us to obtain   a detailed understanding 
of    dormant $\mr{PGL}_2$-opers in characteristic $p^\N$
by applying, via diagonal reduction,  various methods and perspectives inherent in characteristic-$p$-geometry  established in ~\cite{Wak8}.

In particular, because of  a  factorization property on the moduli space classifying  dormant $\mr{PGL}_2^{(\N)}$-opers  in accordance with  clutching morphisms,
we 
 reduce the counting problem under consideration to the cases of small $g$  and $r$.
Some explicit computations based on this argument will be made in, e.g.,   Corollary \ref{C19} and Example \ref{Example33}.

\hspace{-5mm} 
 \includegraphics[width=18cm,bb=0 0 1012 230,clip]{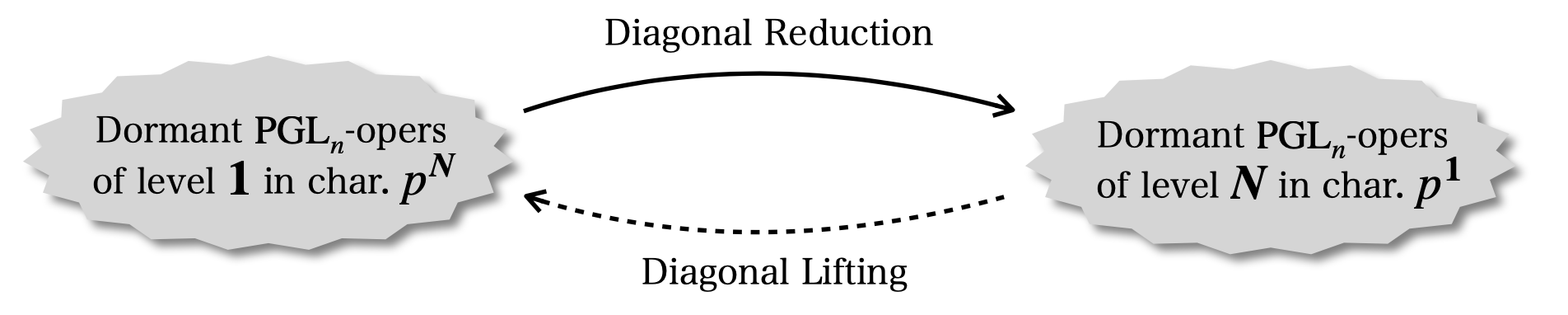}

\LSP
\subsection{Part I: $\mcD^{(\M)}$-modules and dormant flat bundles} \label{S02}

In the rest of this Introduction, we shall describe the contents and  some (relatively important) results of this manuscript.
Note that we substantially apply various results and discussions of ~\cite{Wak8}, in which the author developed the theory of (dormant) opers on pointed stable curves from the viewpoint of logarithmic geometry.
Our study could be placed in a higher-level  generalization of  that theory.

The discussions in this manuscript  may be divided into three parts.
The first part  is devoted to a general study of  flat bundles  formulated in terms of logarithmic geometry.
In Chapter \ref{S145}, 
we  discuss  flat bundles on a log curve.
To do this, we use  the logarithmic  generalization of the sheaf ``$\mcD^{(\M)}$" ($\M \in \mbZ_{\geq 0}$) introduced  by C. Montagnon   (cf. \cite{Mon}).   
This formulation is essential in investigating   how the related moduli spaces behave  when
 the underlying curve degenerates.
 Also, the Cartier operator  associated to a $p^{\M +1}$-flat $\mcD^{(\M)}$-module (i.e., a $\mcD^{(\M)}$-module with vanishing $p^{\M +1}$-curvature)  is defined by composing the usual Cartier operators of the flat bundles constituting that $\mcD^{(\M)}$-module at the respective levels   (cf. Definition \ref{DDDf}).

In Chapter \ref{S200}, 
we generalize the classical notion of a dormant flat bundle (i.e., a flat bundle with vanishing $p$-curvature) 
   to characteristic $p^{\M +1}$ (cf. Definition \ref{D019}).
 At the same time,  we define the diagonal reduction of a dormant flat bundle, as well as a diagonal lifting of a $p^{\M +1}$-flat $\mcD^{(\M)}$-module.
Roughly speaking, the diagonal reduction ${^{\Diag}}\!\!\msF$ of a given dormant flat bundle $\msF$ is obtained in
such a way that, for each  $\CH \leq \M$,  the reduction of $\msF$  modulo $p^{\CH +1}$ corresponds, via the Cartier operator,   to the reduction of   the  level of ${^{\Diag}}\!\!\msF$  to $\CH$.
If the log structure of the underlying curve is trivial, then the dormancy condition can be characterized  by  the  vanishing of the   $p$-curvature  of  the flat  bundles appearing in the inductive procedure for constructing  the diagonal reduction (cf. Proposition \ref{P016dd}, (i)).
In this sense, a dormant flat bundle in characteristic $p^{\M +1}$ may be seen as something like a successive {\it $p$-flat} deformation of 
a $p^{\M +1}$-flat $\mcD^{(\M)}$-module.

In Chapter \ref{S0111},
we discuss the local  description of a dormant flat bundle around a marked/nodal point  of the underlying curve.
For each affine  scheme $S =\mr{Spec}(R)$ flat over $\mbZ/p^{\M +1}\mbZ$,
we shall set
\begin{align}
U_\oslash := \mr{Spec}(R [\![t]\!])
\end{align}
(cf. \eqref{YY124}), where $t$ denotes a formal parameter.
By equipping $U_\oslash$ with the log structure determined by the divisor $t = 0$,
we obtain a log scheme  $U_\oslash^\mr{log}$,
regarded as a formal neighborhood of each marked point in   a pointed curve. 
One of the results in Chapter \ref{S0111} generalizes  ~\cite[Proposition 1.1.12]{Kin1},  ~\cite[Corollary 2.10]{O2}, and it tells us by what data a dormant flat bundle can be exactly characterized.
The statement is described  as follows:

\begin{intthm}[cf. Proposition-Definition \ref{P724}]  \label{TT3fA}
Suppose that the reduction $S_0$ of $S$ modulo $p$  is isomorphic to $\mr{Spec}(R_0)$ for a local ring $(R_0, \mfm_0)$ over $\mbF_p$ whose  residue field $R_0/\mfm_0$ is algebraically closed. 
Let $\msF = (\mcF, \DMO)$ be a dormant flat bundle on $U_\oslash^\mr{log}/S$ of rank $n >0$.
Then, there exists
 an isomorphism of  flat bundles
\begin{align}
\bigoplus_{i=1}^n \msO_{\oslash, \EX_i} \isom \msF
\end{align}
(cf. \eqref{WW102} for the definition of the dormant flat line bundle $\msO_{\oslash, (-)}$)
for some  $\EX_1, \cdots, \EX_n \in \mbZ/p^{\M +1} \mbZ$.
In particular, the monodromy operator $\mu (\DMO)$ of $\DMO$ in the sense of Definition \ref{D98} can be transposed, via conjugation, into the diagonal matrix with diagonal entries $-\EX_1, \cdots, - \EX_n$, and
 the resulting multiset $[\EX_1, \cdots, \EX_n]$ depends only on the isomorphism class of $\msF$.
\end{intthm}

We remark that another  main  result in that chapter gives an explicit description of a $p^{\M +1}$-flat  $\mcD^{(\M)}$-module around  a nodal/marked point  (cf. Proposition-Definition \ref{P022}); in a certain sense, that  result commutes  with Theorem \ref{TT3fA}  via diagonal reduction.

\LSP
\subsection{Part II: Moduli space  and $2$d TQFT of dormant $\mr{PGL}_n^{(\N)}$-opers} \label{SWW02}

In Chapters \ref{DDe345}-\ref{S0133}, we discuss 
dormant opers of finite level  on pointed curves and their moduli space.
For  a pair of nonnegative integers  $(g, r)$  with $2g-2+r>0$, 
we denote by  $\overline{\mcM}_{g,r}$ the moduli stack classifying $r$-pointed stable curves of genus $g$ over $\mbF_p$.
Also, let  us fix integers $n$,  $\N$  with $1 < n < p$ and $\N > 0$.

An important achievement of our discussion is to extend  the moduli problem to pointed stable curves,  allowing nodal singularities. 
This extension   enables us to  investigate how the moduli space deforms when the underlying curve degenerates.
 Additionally,  when it has actually occurred, 
 elements of a certain finite  set
 \begin{align}
 \Xi_{n, \N}
  \end{align}
 (cf. \eqref{e401YY}) provide  
   a boundary condition (i.e., the coincidence of {\it radii}) for gluing dormant $\mr{PGL}_n^{(\N)}$-opers in accordance with an attachment of two curves along respective marked points
    (cf. Definition \ref{epaddd}).

  With that in mind, the moduli space we have to deal with is 
  defined as the category classifying  pairs $(\msX, \msE^\spadesuit)$ consisting of   an $r$-pointed stable curve $\msX$ of genus $g$ in characteristic $p$ (i.e., an object  classified by $\overline{\mcM}_{g,r}$) and  a dormant $\mr{PGL}_n^{(\N)}$-oper  $\msE^\spadesuit$ on $\msX$ {\it  of fixed radii $\rho \in \Xi_{n, \N}^{\times r} \left(=  \Xi_{n, \N} \times \cdots \times \Xi_{n, \N} \right)$}; this category will be denoted by 
\begin{align} \label{dEE9091}
\mcO p_{n, \N, \rho, g,r, \mbF_p}^{^\mr{Zzz...}}, \text{or simply} \ \mcO p_{\rho, g,r}^{^\mr{Zzz...}}
\end{align}
(cf. \eqref{e4000}).
 Forgetting the data of dormant $\mr{PGL}_n^{(\N)}$-opers yields  
  the projection
  \begin{align} \label{dEE9090}
  \Pi_{n, \N, \rho, g, r, \mbF_p} \ (\text{or simply} \ \Pi_{\rho, g, r}) :  \mcO p_{\rho, g,r}^{^\mr{Zzz...}} \migi \overline{\mcM}_{g,r}
  \end{align}
 (cf. \eqref{TTT15}).
 Our results concerning the geometric structures of
  $\mcO p_{\rho, g,r}^{^\mr{Zzz...}}$ and $\Pi_{\rho, g, r}$ are summarized as  follows. 

\SSP
\begin{intthm}[cf.  Theorems \ref{P91} and 
 \ref{y0176}] \label{TT3f}
Let $\rho$ be an element of $\Xi_{n, \N}^{\times r}$ (where $\rho := \emptyset$ if $r = 0$).
Then, the category   $\mcO p_{\rho, g,r}^{^\mr{Zzz...}}$ may be represented by  a  Deligne-Mumford stack of finite type  over $\mbF_p$, and the projection $\Pi_{\rho, g, r}$
 is affine.

Moreover, for  clutching data $\mbG := (\GR, \{ (g_j, r_j)\}_{j=1}^J, \{ \lambda_j \}_{j=1}^J)$  of type $(g,r)$ (cf. Definition \ref{Def112}) and 
  a set of $\mbG$-$\Xi_{n, \N}$-radii $\rho_\mbG:= \{ \rho^j \}_{j=1}^J$ with $\rho_{\mbG \Rightarrow \emptyset} = \rho$ (cf. Definition \ref{y0175} and \eqref{fff036}),
 there exists a canonical  clutching morphism 
\begin{align} \label{WW82}
\mr{Clut}_{\mbG, \rho_\mbG} : \prod_{j =1}^J \mcO p_{\rho^j, g_j,r_j}^{^\mr{Zzz...}} \migi  
\mcO p_{\rho, g,r}^{^\mr{Zzz...}}
\end{align}
 such that
 the square diagram
\begin{align} \label{1050JJ}
\vcenter{\xymatrix@C=76pt@R=36pt{
\coprod_{ \rho_\mbG := \{ \rho^j\}_{j=1}^J}\prod_{j =1}^J \mcO p^\ZZZ_{\rho^j, g_j,r_j} 
\ar[r]^-{\coprod_{\rho_\mbG} \mr{Clut}_{\mbG, \rho_\mbG}} \ar[d]_-{\coprod_{\rho_\mbG} \prod_j \Pi_{\rho^j, g_j, r_j}}& \mcO p^\ZZZ_{\rho, g,r} 
\ar[d]^-{\Pi_{\rho, g, r}}\\
 \prod_{j=1}^J \overline{\mcM}_{g_j, r_j}  \ar[r]_-{\mr{Clut}_\mbG}& \overline{\mcM}_{g, r}
}}
\end{align}
  is commutative and  Cartesian,  where 
  \begin{itemize}
  \item
  the products ``$\prod$'' are all taken over $\mbF_p$ and 
  the disjoint union on the upper-left corner runs  over  the  sets of $\mbG$-$\Xi_{n, \N}$-radii $\rho_\mbG$  with $\rho_{\mbG \Rightarrow \emptyset} = \rho$;
  \item
   $\mr{Clut}_\mbG$ denotes the usual clutching morphism associated to $\mbG$ (cf. \eqref{WW81}).
  \end{itemize}
 \end{intthm}
 \SSP

In Chapters \ref{S18}-\ref{S44}, we focus on the case of $n =2$ and develop the deformation theory   of dormant $\mr{PGL}_2^{(\N)}$-opers.
The deformation space of a fixed  dormant $\mr{PGL}_2^{(\N)}$-oper  is described by using  the hypercohomology group of a certain complex associated to that oper (cf. Theorem \ref{P47}). 
By this description  together with the local study of  finite-level $\mcD$-modules discussed  in the first part, 
we can show more than the facts stated in Theorem \ref{TT3f}, as described below (cf. \cite[Chapter II, Theorem 2.8]{Mzk2} for the case of $\N =1$).

The most important of our results would be  the generic \'{e}taleness of $\Pi_{\rho, g, r}$, since  it  deduces decompositions of  its  degree $\mr{deg}(\Pi_{\rho, g, r})$  with respect to various   clutching morphisms $\mr{Clut}_{\mbG, \rho_\mbG}$.
According to  the discussion in  Chapters \ref{S18}-\ref{S44}, these  decompositions can be collectively explained by the notion of a $2$d TQFT, which is by definition a symmetric monoidal functor from 
the category of $2$-dimensional cobordisms $2\text{-}\mcC ob$ to 
the category of $K$-vector spaces  $\mcV ect_K$ for a field $K$, say, $\mbQ$ or $\mbC$ (cf. Definition \ref{Def5932}).
Applying standard results of  $2$d TQFTs, 
we obtain an approach to compute $\mr{deg}(\Pi_{\rho, g, r})$'s  by means of the ring-theoretic structure of the corresponding Frobenius algebra, or equivalently, the fusion ring of the associated fusion rule.
In ~\cite{Wak34}, ~\cite{Wak8}, ~\cite{Wak33}, the generic \'{e}taleness is used to compute the degrees of moduli spaces of dormant opers in suitable settings.

\SSP
\begin{intthm}[cf. Corollary \ref{T34},  Theorems \ref{Theorem4f4}, and \ref{c43} for the full statement]\label{TTTB2}
Suppose 
that $n=2$ (and $p>2$).
Then, the following assertions hold:
\begin{itemize}
\item[(i)]
Let
$\rho := (\rho_i)_{i=1}^r$ be an element of the set $((\mbZ/p^\N \mbZ)^\times/\{ \pm 1 \})^{\times r}$ (which is identified with $\Xi_{2, \N}^{\times r}$ via \eqref{wpaid}), where $\rho := \emptyset$ if $r = 0$.
Then, 
the  Deligne-Mumford stack  $\mcO p_{\rho, g,r}^\ZZZ$
is smooth over $\mbF_p$ and equidimensional of dimension $3g-3+r$.
Moreover,   
 the projection $\Pi_{\rho, g, r}$ is  finite and \'{e}tale over an \'{e}tale neighborhood of the point in $\overline{\mcM}_{g, r}$ classifying each totally degenerate curve.
 In particular, $\Pi_{\rho, g, r}$ is generically \'{e}tale and generically finite.
 \item[(ii)]
  There exists a  $2$d TQFT 
 \begin{align}
 \mcZ_{2, \N} : 2\text{-}\mcC ob \rightarrow \mcV ect_\mbQ
 \end{align}
  over $\mbQ$  satisfying  the following rules:
  \begin{itemize}
  \item
 If $\mbS^r$  (for $r \in \mbZ_{> 0}$) denotes  the disjoint union of $r$ copies of the circle  $\mbS := \{ (x, y) \in \mbR^2 \, | \, x^2 + y^2 =1 \}$,  then we have
 \begin{align}
 \mcZ_{2, \N} (\mbS^r) = (\mbQ^{\Xi_{2, \N}})^{\otimes r},
 \end{align}
 i.e., the $r$-fold  tensor product of  the $\mbQ$-vector space  $\mbQ^{\Xi_{2, \N}}$ with basis $\Xi_{2, \N}$;
  \item
   If  $\mbM_g^{r \Rightarrow s}$ (for $(g, r, s) \in \mbZ_{\geq 0}^{\times 3}$ with $2g-2 + r+s >0$) denotes   a connected, compact oriented surface whose in-boundary and out-boundary are $\mbS^r$ and $\mbS^s$, respectively, then 
the $\mbQ$-linear map $\mcZ_{2, \N} (\mbM_g^{r \Rightarrow s}) : (\mbQ^{\Xi_{2, \N}})^{\otimes r} \rightarrow (\mbQ^{\Xi_{2, \N}})^{\otimes s}$ is given  by 
\begin{align} \label{eeQQ4r215}
\mcZ_{2, \N} (\mbM_g^{r \Rightarrow s}) (\bigotimes_{i=1}^r \rho_i) = \sum_{(\lambda_j)_j \in \Xi_{2, \N}^{\times s}} \mr{deg}(\Pi_{((\rho_i)_i, ( \lambda_{j})_{j}), g, r +s}) \bigotimes_{j=1}^{s} \lambda_{j}.
\end{align}
  \end{itemize}
 \end{itemize}
 \end{intthm}


\LSP
\subsection{Part III: Canonical diagonal liftings and combinatorics of dormant $\mr{PGL}_2^{(\N)}$-opers} \label{S03}

In Chapter \ref{dEr45},
 the  canonical diagonal lifting of a dormant $\mr{PGL}_2^{(\N)}$-oper
is constructed by applying again the generic \'{e}taleness of $\Pi_{\rho, g, r}$'s.
As discussed in Proposition \ref{Pr459},
there exists a direct linkage between raising  the level of a dormant $\mr{PGL}_2$-oper  in characteristic $p$ to $\N$ and lifting  that oper to characteristic $p^\N$.
To be more precise,  
we show that 
{\it the  space of
 geometric deformations (which mean  deformations classified by   the moduli space $\mcO p^{^\mr{Zzz...}}_{\rho, g, r}$) of a given dormant $\mr{PGL}_2^{(\N)}$-oper 
  is, in a certain sense, dual to the space of 
 its arithmetic deformations (which mean liftings 
  to characteristic $p^\N$ forgetting the data of higher-level structures)}.
 See 
Corollary \ref{T34}, (ii), Propositions \ref{P29},  \ref{P037}, and the picture displayed below.

In particular, 
when a given smooth curve  is $\N$-ordinary  in the sense of   Definition \ref{DD456},
any dormant $\mr{PGL}_2^{(\N)}$-oper on it admits  a canonical lifting.
This construction of  liftings reverses the operation of taking the diagonal reductions, and it is available  to general curves because the  $\N$-ordinary locus  in $\overline{\mcM}_{g, r}$  is dense.

Now, let $S$ be a flat scheme over $\mbZ/p^\N \mbZ$ (with $p>2$) and $X$ a geometrically connected,  proper, and  smooth curve over $S$ of genus $g >1$.
Denote by $S_0$ and $X_0$ the mod $p$  reductions  of $S$ and $X$, respectively.
Also, denote by
\begin{align} \label{TTT556}
\mr{Op}_{1, X}^{^\mr{Zzz...}}  \left(\text{resp.,} \  \mr{Op}_{\N, X_0}^{^\mr{Zzz...}} \right)
\end{align}
(cf. \eqref{e209}) the set of isomorphism classes of dormant $\mr{PGL}_2$-opers  on $X/S$ (resp.,  dormant $\mr{PGL}_2^{(\N)}$-opers on $X_0/S_0$).
Taking the diagonal reductions induces  a map of sets
\begin{align}\label{WW3343}
\Diag_{\!\!\spadesuit}
 : \mr{Op}^{^\mr{Zzz...}}_{1, X} \migi  \mr{Op}^{^\mr{Zzz...}}_{\N, X_0}.
\end{align}
(cf. \eqref{e210www}). Then, we obtain the following assertion.

\SSP
\begin{intthm}[cf. Theorem-Definition \ref{T04}, Remark \ref{WW90}, and Corollary \ref{C998}] \label{TTTB}
 Suppose that $X_0/S_0$ is general in $\overline{\mcM}_{g, 0}$.
 Then, the map $\Diag_{\!\!\spadesuit}$  has an inverse map
 \begin{align} \label{e191}
\Diagg_{\!\!\spadesuit}
 :   \mr{Op}^{^\mr{Zzz...}}_{\N, X_0} \isom \mr{Op}^{^\mr{Zzz...}}_{1, X}.
\end{align}
That is to say, for each dormant $\mr{PGL}_2^{(\N)}$-oper $\msE^\spadesuit_0$  on $X_0/S_0$,
there exists a unique (up to isomorphism) dormant $\mr{PGL}_2$-oper ${^{\Diagg}}\!\!\msE^\spadesuit_0$ on $X/S$ whose diagonal reduction is isomorphic to $\msE^\spadesuit_0$.
In particular, if $S_0 = \mr{Spec}(k)$ for an algebraically closed field $k$ over $\mbF_p$, then  the set  $\mr{Op}^{^\mr{Zzz...}}_{1, X}$ is finite and  its cardinality
 $\sharp (\mr{Op}^{^\mr{Zzz...}}_{1, X})$ satisfies 
\begin{align}
\sharp (\mr{Op}^{^\mr{Zzz...}}_{1, X}) = \left(\sharp (\mr{Op}^{^\mr{Zzz...}}_{\N, X_0}) =\right) \mr{deg} (\Pi_{2, \N, \emptyset, g, 0, \mbF_p}).
\end{align}
 \end{intthm}
\SSP

For each $\msE_0^\spadesuit$ as in the above theorem, 
we shall refer to 
the resulting dormant $\mr{PGL}_2$-oper  ${^{\Diagg}}\!\!\msE^\spadesuit_0$
as  the {\bf canonical diagonal lifting} of $\msE_0^\spadesuit$. 
This fact will  also be  formulated as  
the existence of canonical liftings of
 {\it  Frobenius-projective structures} (cf. Theorem-Definition \ref{T1dddd}), which are 
 characteristic-$p$ analogues of  complex projective structures\footnote{A {\it (complex) projective structure} is defined as  an additional structure on a 
Riemann surface
consisting  of local coordinate charts defining its complex structure such that on any two overlapping patches, the change of coordinates may be described as a M\"{o}bius transformation. Projective structures 
are in bijection with  $\mr{PGL}_2$-opers (in our sense)  via the Riemann-Hilbert correspondence and the algebraization of the underlying Riemann surface.} (cf. ~\cite[Definition 2.1]{Hos2}, or Definition \ref{D0188}).

We expect  that 
canonical liftings exist even for a general  rank $n$ (sufficiently small relative to $p$).
To prove it in accordance with the arguments of this manuscript, we will have to prove the generic \'{e}taleness of 
$\Pi_{n, \N, \rho, g, r}$'s.
(Other constructions of  liftings of $\mr{PGL}_2$-opers in characteristic $p$ can be found in 
~\cite{LSYZ},  ~\cite{Mzk1}, and ~\cite{Mzk2}; but they differ from ours, as mentioned in   Remark \ref{WW92}).

\hspace{15mm} 
 \includegraphics[width=14cm,bb=0 0 1012 350,clip]{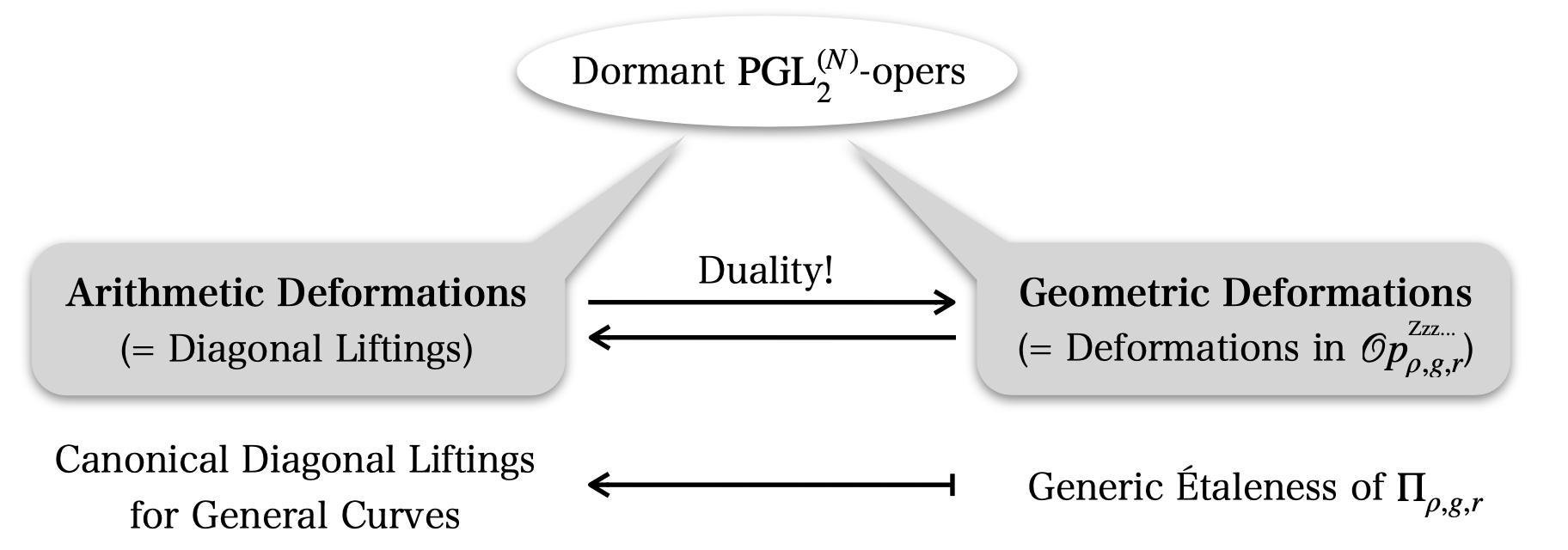}

\vspace{5mm}

In  Chapter \ref{S10},  we translate 
the  $2$d TQFT $\mcZ_{2, \N}$
   into combinatorial objects in order to solve our counting problem in a practical manner.
This is possible because 
the factorization property of $\mr{deg}(\Pi_{\rho, g, r})$'s  
reduces the problem to the case where the underlying curve is totally degenerate (cf. Definition \ref{DD3WW}).
In this case, the normalizations of its  components are $3$-pointed projective lines, so dormant $\mr{PGL}_2^{(\N)}$-opers are  determined by purely combinatorial patterns of their radii (cf. ~\cite{LO} in the case of $\N =1$).

To make the discussion clearer, we set  ${^\dagger}C_\N$ (cf. \eqref{eeQQ112}) to be the collection  of triples of  nonnegative integers $(s_1, s_2, s_3)$ satisfying the following conditions\footnote{The second one can be rephrased as the condition that, for every positive integer $\N' < \N$, the inequalities $\sum_{i=1}^3 \hat{s}_i \leq p^{\N'} -2$ and $|\hat{s}_2 - \hat{s}_3| \leq \hat{s}_1 \leq \hat{s}_2 + \hat{s}_3$ hold, where $\hat{a} := \frac{p^{\N'}-1}{2} - \left| a- p^{\N'} \cdot \left\lfloor \frac{a}{p^{\N'}}\right\rfloor - \frac{p^{\N'}-1}{2} \right|$ for each nonnegative integer $a$.}:
\begin{itemize}
\item
$\sum_{i=1}^3 s_i \leq p^\N -2$ and $|s_2 -s_3| \leq s_1 \leq s_2 + s_3$;
\item
For every positive integer $\N' < \N$, we can choose a triple of integers  $(s'_1, s'_2, s'_3)$ with  $s'_i \in \{ [s_i]_{\N'}, p^{\N'}-1-[s_i]_{\N'} \}$ ($i=1, 2, 3$)  that satisfies  $\sum_{i=1}^3 s'_i \leq p^{\N'} -2$ and $|s'_2 -s'_3| \leq s'_1 \leq s'_2 + s'_3$.
\end{itemize}
Here, for each nonnegative integer $a$, 
we set 
$[a]_{\N'}$ to be 
the remainder obtained by dividing $a$ by $p^{\N'}$.
Also, 
fix   trivalent clutching  data $\mbG$  of type $(g, 0)$ (cf. Definition \ref{Def112}, (iii)), and denote by $X_\mbG$ the totally degenerate curve corresponding to $\mbG$. 

A {\it balanced $(p, \N)$-edge numbering} on $\mbG$ (cf. Definition \ref{De113}, (i)) is a collection of nonnegative integers  $(a_e)_{e \in E}$ indexed by the set $E$ of edges of $\mbG$ such that, for 
each triple of edges $(e_1, e_2, e_3)$ (with multiplicity) incident to a common vertex, 
the triple  $(a_{e_1}, a_{e_2}, a_{e_3})$  belongs to ${^\dagger}C_\N$.
The set of  balanced $(p, \N)$-edge numberings on $\mbG$ is denoted by 
\begin{align}
\mr{Ed}_{p, \N, \mbG}
\end{align}
(cf. \eqref{eeQQ79}).
This set is  finite, and it is possible to explicitly find out  which combinations of nonnegative integers belong to it.

\hspace{20mm} 
 \includegraphics[width=14cm,bb=0 0 1212 280,clip]{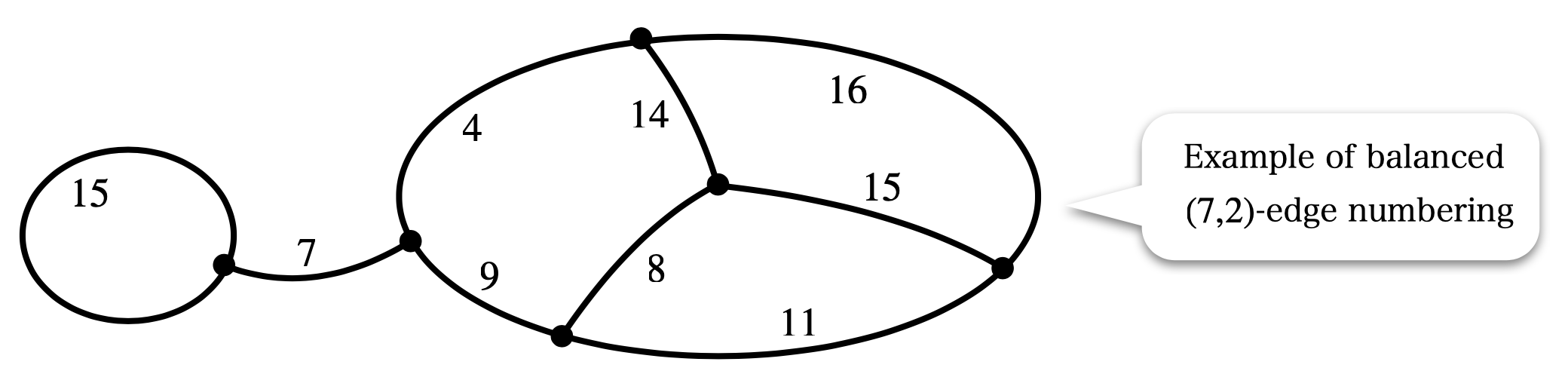}

\vspace{5mm}

The study  of  Gauss hypergeometric differential operators with a full set of solutions shows that 
 dormant $\mr{PGL}_2^{(\N)}$-opers on a $3$-pointed projective line  are verified to be parametrized by  ${^\dagger}C_\N$ (cf. Propositions \ref{Prop21}, \ref{P05}).
Thus, 
dormant  $\mr{PGL}_2^{(\N)}$-opers on $X_\mbG$  correspond bijectively to 
elements of $\mr{Ed}_{p, \N, \mbG}$.
Moreover, such numberings on $\mbG$ may be identified with 
   lattice points inside a certain subset   of a Euclidean space of dimension $(3g-3)\cdot \N$   (cf. Proposition \ref{Prop4378}).
This subset  is constructed as the union of a finite number of rational convex polytopes without boundary.
It follows that  we can apply  Ehrhart's theory for lattice-point counting of  polytopes, and obtain a quasi-polynomial
  realizing its lattice-point function.
The consequence  of this argument is as follows.

\SSP
 \begin{intthm}[cf. Theorem \ref{Thm5} for the full statement] \label{Th4ke2}
Let $\mbG$ be trivalent clutching data of type $(g, 0)$.
Then, there exists a canonical bijection between $\mr{Ed}_{p, \N, \mbG}$ and the set of dormant $\mr{PGL}_2^{(\N)}$-opers on $X_\mbG$.
In particular, 
 one can find a  quasi-polynomial  $H_{\N, \mbG} (t)$ (independent of $p$) with coefficients in $\mbQ$ of degree $(3g-3) \cdot \N$ such that its minimum period divides $4$ and satisfies  the equalities
\begin{align} \label{eeQQ4j0}
\mr{deg}(\Pi_{2, \N, g, 0, \mbF_p}) = \sharp (\mr{Ed}_{p, \N, \mbG})  = H_{\N, \mbG} (p).
\end{align}
(Note that the odd constituents of $H_{\N, \mbG} (t)$ do not depend on the choice of $\mbG$, i.e., depend only on the genus  ``$g$" and the positive integer ``$\N$").
  \end{intthm}
\SSP

Finally,
as a consequence of this fact, 
a partial  answer to  the question displayed  at the beginning of Section \ref{WW60} can be given as follows.

\SSP
 \begin{intthm}[cf. Theorem \ref{Theorem895}] \label{Th4zz}
Let  $X/S$ be  as in Theorem \ref{TTTB}, and suppose that $S = \mr{Spec}(R)$ for a flat $\mbZ /p^\N \mbZ$-algebra $R$ whose mod $p$ reduction $R/pR$ is an algebraically closed field over $\mbF_p$.
Also, let $\mcL$ be a line  bundle on $X$, and denote by 
\begin{align}
\mr{Diff}_{2, \mcL}^\mr{full}
\end{align}
 (cf. \eqref{eeQQ571}) the set of $2$nd order linear differential operators on $\mcL$ with unit principal symbol and having a full set of root functions.
\begin{itemize}
\item
Suppose that  $\mr{obs}(\mcL^{\otimes 2} \otimes \Omega) \neq 0$ (cf.  \eqref{eeQQ468} for the definition of $\mr{obs}(-)$).
Then, we have $\mr{Diff}_{2, \mcL}^\mr{full} = \emptyset$.
\item
Suppose 
$\mr{obs}(\mcL^{\otimes 2} \otimes \Omega) = 0$ and 
$X_0$ is sufficiently general in $\overline{\mcM}_{g, 0}$.
Then, 
$\mr{Diff}_{2, \mcL}^\mr{full}$ is finite and its cardinality $\sharp (\mr{Diff}_{2, \mcL}^\mr{full})$ satisfies
\begin{align}
\sharp (\mr{Diff}_{2, \mcL}^\mr{full}) = p^{g\N} \cdot 
\sharp (\mr{Ed}_{p, \N, \mbG}) = p^{g\N} \cdot H_{\N, \mbG}(p)
\end{align}
for any trivalent clutching data $\mbG$ of type $(g, 0)$, where  $H_{\N, \mbG}(t)$ denotes the quasi-polynomial resulting from Theorem \ref{Th4ke2}.
In particular,  $\sharp (\mr{Diff}_{2, \mcL}^\mr{full})$ may be expressed as a rational  quasi-polynomial in $p$  of degree $(4g-3) \cdot \N$. 
\end{itemize}
  \end{intthm}

\LSP
\subsection{Notation and Conventions} 

Throughout this manuscript, all schemes
 are assumed to
be locally noetherian.
 We fix a prime $p$,  and write $\mbF_p := \mbZ/p\mbZ$.

Let $S$ be a scheme.
 Given  a sheaf $\mcV$ on $S$, 
 we use the notation ``\,$v \in \mcV$\,''
   for a local section $v$ of  $\mcV$.
If $\mcV$ is an $\mcO_S$-module,  then we denote by $\mcV^\vee$ \index{$\mcV^\vee$, dual of $\mcV$} the dual of $\mcV$, i.e., $\mcV^\vee := \mcH om_{\mcO_S}(\mcV, \mcO_S)$.
By a {\bf vector bundle} \index{vector bundle} on $S$, we mean a locally free $\mcO_S$-module of finite rank. 
 If $X$ is a scheme 
  over $S$,
 then we shall write $\Omega_{X/S}$   for 
 the sheaf of $1$-forms on $X$ over $S$, and $\mcT_{X/S}$ for its dual.

 For  basic properties on log schemes,
we refer the reader to  ~\cite{KaKa}, ~\cite{ILL}, and ~\cite{KaFu}.
Given a log scheme (or more generally, a log stack) indicated, say,  by   $S^\mr{log}$, we shall write $S$ for the underlying scheme (stack) of $S^\mr{log}$, and 
$\alpha_{S^\mr{log}} : \mcM_{S^\mr{log}} \migi \mcO_S$ for the morphism of sheaves of monoids defining the log structure of $S^\mr{log}$.
For any morphism of log schemes $f^\mr{log} : X^\mr{log} \rightarrow S^\mr{log}$, we write $\overline{\mcM}_{X^\mr{log}/S^\mr{log}} := \mcM_{X^\mr{log}}/\mr{Im}(f^*(\mcM_{S^\mr{log}}) \rightarrow \mcM_{X^\mr{log}})$ (cf. ~\cite[Introduction]{KaFu}), and call it the {\bf relative characteristic} of $f^\mr{log}$ (or,  of $X^\mr{log}/S^\mr{log}$).

Let $\nabla : \mcK^0 \migi \mcK^1$ be a morphism of sheaves of abelian groups on a scheme $X$.
It may be regarded as a complex concentrated at degrees $0$ and $1$; we denote this complex by 
$\mcK^\bullet[\nabla]$. 
\index{$\mcK^\bullet[\nabla]$, complex defined by $\nabla$}
Next, let $f : X \migi S$ be  a morphism of schemes and
$\nabla : \mcK^0 \rightarrow \mcK^1$ an $f^{-1}(\mcO_S)$-linear morphism between quasi-coherent $\mcO_X$-modules.
 Then, for each integer $i \geq 0$,
one may define the sheaf
  $\mbR^i f_*(\mcK^\bullet [\nabla])$ 
on $S$ obtained from $\mcK^\bullet [\nabla]$ by applying 
the $i$-th hyper-derived functor $ \mbR^i f_*(-)$  of $f_*(-)$ (cf. ~\cite{Kal}, (2.0)).
In particular, $\mbR^0f_*(\mcK^\bullet [\nabla]) = f_*(\mr{Ker}(\nabla))$.
If  $S$ is affine and $f$ is quasi-compact and quasi-separated, then $H^0 (S,   \mbR^i f_*(\mcK^\bullet [\nabla]))$ may be identified with the $i$-th hypercohomology group 
$\mbH^i (X, \mcK^\bullet [\nabla])$.
Given an integer $n$ and a sheaf $\mcF$ on $S$,
  we define the complex
$\mcF [n]$  to be  $\mcF$ (considered as a complex concentrated at degree $0$) shifted down by $n$, so that $\mcF[n]^{-n} = \mcF$ and $\mcF[n]^i = 0$ ($i \neq -n$).

Denote by   $\mbG_m$
  the multiplicative group.
Also, for a positive  integer $n$, 
we shall write $\mr{GL}_n$ 
 (resp., $\mr{PGL}_{n}$) 
  for the general (resp.,  projective) linear group  of rank $n$.

\vspace{10mm}
\section{$\mcD$-modules of finite level} \label{S145}\SSP

In this section, we discuss  some basics related to sheaves of logarithmic differential operators of finite level and  higher-level flat bundles  on a log curve.

\LSP
\subsection{Logarithmic differential operators of finite  level} \label{SS040}

First, we briefly recall sheaves  of logarithmic differential operators of finite level discussed in ~\cite{PBer1}, ~\cite{PBer2},  and ~\cite{Mon}.

Let $\M$ be a nonnegative integer.
Also, let  $S^\mr{log}$ be   an fs log scheme over 
$\mbZ_{(p)}$
and $X^\mr{log}$ an fs log scheme equipped with a morphism of log schemes $f^\mr{log} : X^\mr{log} \migi S^\mr{log}$ which is log smooth  (i.e., ``smooth" in the sense of ~\cite{KaKa}, (3.3)).
 Denote by $\Omega_{X^\mr{log}/S^\mr{log}}$
  the sheaf of logarithmic $1$-forms  on $X^\mr{log}/S^\mr{log}$ (cf. ~\cite{KaKa}, (1.7))
  and by 
 $\mcT_{X^\mr{log}/S^\mr{log}}$
  the sheaf of logarithmic vector fields on $X^\mr{log}/S^\mr{log}$, i.e., the dual of $\Omega_{X^\mr{log}/S^\mr{log}}$.
 For simplicity, we occasionally write $\Omega$ and $\mcT$ instead of $\Omega_{X^\mr{log}/S^\mr{log}}$ and $\mcT_{X^\mr{log}/S^\mr{log}}$, respectively.
 Since $f^\mr{log}$ is log smooth, both $\Omega$ and $\mcT$ are vector bundles (cf ~\cite{KaKa}, Proposition (3.10)).
 
 Suppose that $p$ is locally nilpotent on $S$ and 
 $S$ is  equipped with an $\M$-PD structure 
 that extends to 
    $X$ via $f$. 
Denote by $P_{(\M)}^\mr{log}$
  the  log $\M$-PD envelope 
  of the diagonal embedding  $X^\mr{log} \migi X^\mr{log} \times_{S^\mr{log}} X^\mr{log}$ (cf. ~\cite[Proposition 2.1.1]{Mon}) and 
 by $\mcP_{(\M)}$  the structure sheaf of $P_{(\M)}^\mr{log}$.
 The defining ideal $\overline{\mcI}$ of  the strict closed immersion  $X^\mr{log} \migiincl P_{(\M)}^\mr{log}$
admits   the $\M$-PD-adic filtration $\{ \overline{\mcI}^{\{ \ell \}} \}_{\ell \in \mbZ_{\geq 0}}$ constructed in the manner of ~\cite[Definition A.\,3]{PBer2}.
For each  $\ell \in \mbZ_{\geq 0}$,   the quotient sheaf $\mcP_{(\M)}^\ell$ of $\mcP_{(\M)}$   by $\overline{\mcI}^{\{ \ell+1 \}}$ 
determines a strict closed subscheme
$P^{\ell, \mr{log}}_{(\M)}$  of  $P_{(\M)}^\mr{log}$.
We shall write 
$\mr{pr}_1$ and $\mr{pr}_2$ (resp., $\mr{pr}_1^{\ell}$ and $\mr{pr}_2^\ell$) for   the morphisms $P_{(\M)}^{\mr{log}} \migi X^\mr{log}$ (resp., $P_{(\M)}^{\ell, \mr{log}} \migi X^\mr{log}$) induced by 
the first and second projections  $X^\mr{log} \times_{S^\mr{log}} X^\mr{log} \migi X^\mr{log}$,
  respectively.
Note that $\mcP^\ell_{(\M)}$ may be regarded as a sheaf on $X$.
Moreover, it  has  an $\mcO_X$-module structure via $\mr{pr}_1^\ell$ (resp., $\mr{pr}_2^\ell$); it  will be  applied whenever  we are  considering  an action of  $\mcO_X$ 
 by  left (resp., right) multiplication.
 When there is a fear of confusion, we use the notation $\mr{pr}_{1*}^\ell (\mcP_{(\M)}^{\ell})$ (resp., $\mr{pr}_{2*}^\ell (\mcP_{(\M)}^{\ell})$) for  writing the sheaf $\mcP_{(\M)}^{\ell}$ equipped with this $\mcO_X$-module structure.

Given  an integer $\M' \geq  \M$, we obtain  a natural morphism
\begin{align}
\varsigma^\ell_{\M', \M} : \mcP_{(\M')}^{\ell} \migi \mcP_{(\M)}^{\ell}
\end{align}
 preserving  both the left and right $\mcO_X$-module structures.
As 
mentioned in ~\cite[Section 2.3]{Mon},
there exists a canonical  $\mcO_X$-algebra morphism
\begin{align} \label{EOFOSQ98}
\delta_\M^{\ell, \ell'}: \mcP_{(\M)}^{\ell+\ell'} \migi \mcP^\ell_{(\M)} \otimes_{\mcO_X} \mcP^{\ell'}_{(\M)}
\end{align}
for each pair of nonnegative integers $(\ell, \ell')$ such that  if $\M'$ is an 
integer 
 $\geq \M$, then 
 the following diagram  is commutative:
\begin{align} \label{COWQ231}
\vcenter{\xymatrix@C=46pt@R=36pt{
\mcP_{(\M')}^{\ell+\ell'} \ar[r]^-{\delta_{\M'}^{\ell, \ell'}} \ar[d]_-{\varsigma^{\ell+\ell'}_{\M', \M}}& \mcP_{(\M')}^\ell \otimes_{\mcO_X} \mcP_{(\M')}^{\ell'} \ar[d]^-{\varsigma^{\ell}_{\M', \M} \otimes \varsigma^{\ell'}_{\M', \M}}
\\
\mcP_{(\M)}^{\ell+\ell'} \ar[r]_-{\delta_{\M}^{\ell, \ell'}}& \mcP_{(\M)}^\ell \otimes_{\mcO_X} \mcP_{(\M)}^{\ell'}.
}}
\end{align}
 The morphism $\delta_\M^{\ell, \ell'}$ determines  a morphism  of log schemes
 \begin{align} \label{MROD3}
 \delta_\M^{\ell, \ell' \sharp}:  P_{ (\M)}^{\ell, \mr{log}}  \times_{X^\mr{log}} P_{(\M)}^{\ell', \mr{log}} \migi P_{(\M)}^{\ell+\ell', \mr{log}}.
 \end{align}
 over $X^\mr{log}$.

For each $\ell \in \mbZ_{\geq 0}$,
we shall  set
 \begin{align} \label{MROD2}
\mcD_{X^\mr{log}/S^\mr{log}, \leq \ell}^{(m)} := \mcH om_{\mcO_X} (\mr{pr}_{1*}^{\ell}(\mcP^\ell_{(\M)}), \mcO_X)
\end{align}
(cf. ~\cite[Definition 2.3.1]{Mon}).
In particular, we have natural identifications $\mcD_{X^\mr{log}/S^\mr{log}, \leq 0}^{(\M)} = \mcO_X$.
The  {\it sheaf of logarithmic differential operators of level $m$} is defined by 
\begin{align} \label{MROD1}
\mcD_{X^\mr{log}/S^\mr{log}}^{(\M)} := \bigcup_{\ell \in \mbZ_{\geq 0}} \mcD_{X^\mr{log}/S^\mr{log}, \leq \ell}^{(\M)}.
\end{align}
For simplicity, we occasionally write  $\mcD_{\leq \ell}^{(\M)}$ and  $\mcD^{(\M)}$
instead of $\mcD_{X^\mr{log}/S^\mr{log}, \leq \ell}^{(\M)}$ and $\mcD_{X^\mr{log}/S^\mr{log}}^{(\M)}$, respectively.
The morphisms  $\delta_{\M}^{\ell, \ell'}$ (for $\ell, \ell' \in \mbZ_{\geq 0}$) determine   a structure of  (possibly noncommutative) $f^{-1}(\mcO_S)$-algebra
$\mcD^{(\M)} \otimes_{\mcO_X} \mcD^{(\M)} \migi \mcD^{(\M)}$ on $\mcD^{(\M)}$.
The collection of morphisms $\varsigma_{\M, \M'}^\ell$ 
(with $\M' \leq \M$) induces an inductive system of sheaves 
  $\{ \mcD^{(\M)}_{\leq \ell} \}_{\M \in \mbZ_{\geq 0}}$.
We shall write ${^L \mcD}^{(\M)}_{\leq \ell}$ (resp., ${^R \mcD}^{(\M)}_{\leq \ell}$) for the sheaf $\mcD^{(\M)}_{\leq \ell}$ endowed with a structure of $\mcO_X$-module  arising from  left (resp., right) multiplication by sections of $\mcD_{\leq 0}^{(\M)}  \left(=\mcO_X \right)$.

Given an $\mcO_X$-module $\mcF$, we equip the tensor product $\mcD_{\leq \ell}^{(\M)} \otimes \mcF := {^R}\mcD_{\leq \ell}^{(\M)} \otimes \mcF$ (resp., $\mcF \otimes \mcD^{(\M)}_{\leq \ell} :=\mcF \otimes {^L}\mcD^{(\M)}_{\leq \ell}$) with the $\mcO_X$-module structure given by left (resp., right) multiplication.

Note that $\mcD^{(0)}$ coincides with the sheaf of noncommutative rings ``$\msD_{\hslash, Y^\mr{log}/T^\mr{log}}^{<\infty}$" defined in ~\cite[Eq.\,(463)]{Wak8} such that the triple $(\hslash, T^\mr{log}, Y^\mr{log})$ is taken to be $(1, S^\mr{log}, X^\mr{log})$.
With this in mind, 
whenever  we deal with  
the sheaf $\mcD^{(0)}$ (as well as the sheaves $\mcD_{\leq \ell}^{(0)}$'s for $\ell \in \mbZ_{\geq 0}$) 
 without choosing a ($0$-)PD structure on $S$,
 we take it 
 in the sense of ~\cite{Wak8}. 
On the other hand,  $S$ is always assumed to be  equipped with the trivial $\M$-PD structure when it is defined  over $\mbF_p$.

\LSP
\subsection{$\mcD^{(\M)}$-modules on log curves}
\label{SS04f4}

To simplify the discussion,  
we restrict ourselves to the case where 
 $X^\mr{log}/S^\mr{log}$ is a log curve, in the following sense\footnote{For the purposes of the various discussions to be made in ~\cite{Wak8} and this monograph, it is desirable to include the situation where a log curve admits an empty fiber.
 Hence, to be precise, the description of  ~\cite[Definition 1.40]{Wak8} needs to be modified slightly as in Definition \ref{pY5019}.
 Note that our definition also differs from   ~\cite[Definition 1.1]{KaFu} and ~\cite[Definition 4.5]{ACGH}.}.

\SSP
\bde \label{pY5019} 
We say that $f^\mr{log} : X^\mr{log} \migi S^\mr{log}$ is 
a {\bf log curve}  (over $S^\mr{log}$) if it is a  log smooth  integral morphism 
 such that each geometric fiber of the underlying morphism of schemes $f : X \migi S$  is either empty or a reduced scheme of pure dimension one.
 (In particular, both $\Omega$ and $\mcT$ are line bundles, and  the underlying morphism $f : X \migi S$ is flat,  according to  ~\cite[Corollary 4.5]{KaKa}.)
   \ede

Hereinafter,  suppose that $f^\mr{log} : X^\mr{log} \migi S^\mr{log}$ is a log curve.
Recall from ~\cite[Section 2.2.2]{Mon} that 
there exists  a short  exact sequence
\begin{align} \label{ddfg1}
0 \migi \varDelta^{-1}(1 + \overline{\mcI}) \xrightarrow{\lambda} \varDelta^{-1} (\mcM_{P_{(\M)}^\mr{log}}) \xrightarrow{\varDelta^*} \mcM_{X^\mr{log}} \migi 0,
\end{align}
where
$\varDelta$ denotes the natural morphism   $X^\mr{log} \migiincl P_{(\M)}^\mr{log}$ and  
 $\lambda$  denotes  the morphism obtained by  restricting $\varDelta^{-1}(\alpha^{-1}_{P_{(\M)}^\mr{log}}) : \varDelta^{-1}(\mcO^\times_{P_{(\M)}}) \migiincl   \varDelta^{-1}(\mcM^\times_{P^\mr{log}_{(\M)}})$.

For any local section $a \in \mcM_{X^\mr{log}}$,
$\mr{pr}_1^* (a)$ and $\mr{pr}_2^* (a)$ have the same image in $\mcM_{X^\mr{log}}$.
From the above exact sequence, there exists a unique section $\mu_{(\M)} (a) \in \varDelta^{-1}(1 + \overline{\mcI})$ with  $\mr{pr}_2^* (a) = \mr{pr}_1^* (a) \cdot \lambda (\mu_{(\M)} (a))$.
Thus, the assignment $a \mapsto \mu_{(\M)} (a)$ defines a well-defined morphism $\mu_{(\M)} : \mcM_{X^\mr{log}} \migi \varDelta^{-1}(1 + \overline{\mcI})$.

Now, let us take a (locally defined) logarithmic coordinate $x \in \mcM_{X^\mr{log}}^\mr{gr}$, i.e., 
the section $d \mr{log} (x)$ forms a local generator of $\Omega$.
We shall set $\eta := \mu_{(\M)}(x) -1$.
For each $\ell \in \mbZ_{\geq 0}$,
the collection $\{ \eta^{\{ j \}} \, | \, j \leq \ell \}$ (cf. ~\cite[Section 1.2.3]{Mon} for the definition of $(-)^{\{ j \}}$) forms  a local basis of $\mcP^\ell_{(\M)}$.
Hence, by setting   $\partial^{\langle j \rangle}$ as the dual of $\eta^{\{ j \}}$,
we obtain a  local basis $\{ \partial^{\langle  j \rangle}\}_{j \leq \ell}$ of $\mcD_{\leq \ell}^{(\M)}$.
For any nonnegative  integers $j'$ and $j''$,
the following equality  holds:
\begin{align} \label{e364}
\partial^{\langle j' \rangle} \cdot \partial^{\langle j'' \rangle} = \sum_{j = \mr{max} \{ j', j'' \} }^{j' + j''} \frac{j!}{(j' + j'' - j)! \cdot (j - j')! \cdot  (j-j'')!} \cdot \frac{q_{j'}! \cdot q_{j''}!}{q_j !} \cdot \partial^{\langle j \rangle}
\end{align}
(cf. ~\cite[Lemme 2.3.4]{Mon}), 
where, 
for  each  $j \in \mbZ_{\geq 0}$,  let $(q_j, r_j)$ be the pair of nonnegative integers uniquely determined by the condition that $j = p^\M \cdot q_j + r_j$ and $0 \leq r_j < p^\M$.
In particular, we have 
\begin{align} \label{Equu88}
\partial^{\langle j' \rangle} \cdot \partial^{\langle j'' \rangle} = \partial^{\langle j'' \rangle} \cdot \partial^{\langle j' \rangle}.
\end{align}
The assignment $\partial^{\langle \ell \rangle} \mapsto \partial^{\otimes \ell}$ determines an identification 
 $\mcD_{ \leq \ell}^{(\M)}/\mcD_{\leq \ell-1}^{(\M)}= \mcT^{\otimes \ell}$ ($l \geq 1$).

The structure sheaf $\mcO_X$ carries a natural  left $\mcD^{(\M)}$-action 
 determined uniquely by the condition that
 if we write $t := \alpha_{X^\mr{log}} (x)$, where $x$ is as above,
  then
\begin{align} \label{dd123}
\partial^{\langle j \rangle} (t^n) = q_j! \cdot \binom{n}{j} \cdot t^n
\end{align}
for every $j \in \mbZ_{\geq 0}$ (cf. ~\cite[Lemme 2.3.3]{Mon}).
For any local sections $a_1, a_2 \in \mcO_X$, the following Leibniz rule holds:
\begin{align} \label{EE658}
\partial^{\langle j \rangle} (a_1 \cdot a_2) = \sum_{j' + j'' = j} \frac{q_j !}{q_{j'}! \cdot q_{j''}!} \cdot 
\partial^{\langle j' \rangle}(a_1) \cdot  \partial^{\langle j'' \rangle}  (a_2) 
\end{align}
(cf. ~\cite[Eq.\,(2.8)]{Mon}).

The following assertion will be applied in the subsequent discussion (cf. Lemma \ref{P238}).

\SSP
\ble \label{L99}
Let us keep the above notation.
Moreover,  suppose that $S$ is a scheme over $\mbF_p$.
Then, 
for each nonnegative  integer $\ell\leq \M$,
the following equality holds:
\begin{align}
\sum_{j=0}^{p^\ell-1} (-1)^j \cdot \partial^{\langle j \rangle} = 
 \prod_{s = 0}^{\ell -1} (1 - (\partial^{\langle p^s \rangle})^{p-1}),
\end{align}
where we set $\prod_{s = 0}^{-1} (-) := 1$.
\ele
\begin{proof}
Let $a$,  $b$, and $d$  be nonnegative  integers with
$1 \leq a \leq p-1$ and $p^d (p-1 +pb) < p^\M$ (which implies $q_{j} = 0$ for any $j \leq p^d (a-1) + p^{d+1}b$).
Since $S$ is a scheme over $\mbF_p$,
  we have
\begin{align}
& \hspace{8mm} \partial^{\langle p^{d}(a-1) + p^{d+1}b\rangle} \cdot  \partial^{ \langle p^d \rangle} \\
& =  \sum_{j=\max \{p^{d}(a-1)+ p^{d+1}b, p^d \}}^{p^d a + p^{d+1}b} \frac{j!}{(p^da + p^{d+1}b -j)! \cdot  (j - p^d (a-1)- p^{d+1}b)!  \cdot (j - p^d)!}
 \cdot \partial^{\langle j \rangle}  \notag \\
& \, = (a-1) \cdot \partial^{\langle p^d(a-1) + p^{d+1}b\rangle} + a  \cdot  \partial^{\langle p^d a+ p^{d+1}b\rangle}, \notag
\end{align}
where the first equality follows from \eqref{e364}.
By using this, we have
\begin{align} \label{dGG4}
\partial^{\langle p^d a + p^{d+1}b\rangle} & = \frac{1}{a } \cdot \left(\partial^{\langle p^d \rangle} - (a-1) \right) \cdot  \partial^{\langle  p^d (a-1) + p^{d+1}b\rangle}  \\
& =  \frac{1}{a} \cdot \left(\partial^{\langle p^d \rangle} - (a-1) \right) \cdot \left(  \frac{1}{a-1} \cdot  \left(\partial^{\langle p^d \rangle} - (a-2) \right) \cdot \partial^{\langle  p^d (a-2) + p^{d+1}b\rangle}  \right) \notag \\
&  \ \  \vdots \notag \\
& =  \frac{1}{a!} \cdot \partial^{\langle p^{d+1}b \rangle} \cdot \prod_{j = 1}^{a}(\partial^{\langle p^d \rangle} - (a-j)). \notag
\end{align}

Here, note that the equality 
\begin{align} \label{dGG2}
1-x^{p-1} = \sum_{a = 0}^{p-1} \frac{(-1)^a}{a!} \cdot \prod_{j=1}^{a} (x- (a-j))
\end{align}
 holds in $\mbF_p [x]$.
Indeed,  if $h (x)$ denotes the right-hand side of \eqref{dGG2}, then, for each  $n =1, \cdots, p-1$, we have
\begin{align}
h(n)  
= \sum_{a = 0}^{n} \frac{(-1)^a}{a!} \cdot \prod_{j=1}^{a} (x- (a-j)) \Bigl|_{x = n}
= \sum_{a=0}^n (-1)^a \cdot \binom{n}{a} = (1-1)^n = 0.
\end{align}
Moreover, since the equality  $h (0) = 1$ is immediately verified, we obtain \eqref{dGG2}, as desired.

The equalities   \eqref{dGG4} and \eqref{dGG2} together imply
\begin{align} \label{dGG1}
\sum_{a=0}^{p-1} (-1)^{p^{d}a + p^{d+1}b} \cdot \partial^{\langle p^{d} a + p^{d+1}b \rangle} &= \partial^{\langle p^{d+1}b \rangle} \cdot \left( \sum_{a = 0}^{p-1}\frac{(-1)^{p^d a + p^{d+1}b}}{a!} \cdot \prod_{j = 1}^{a}(\partial^{\langle p^d \rangle} - (a-j))\right)\\
&= (-1)^{p^{d+1}b} \cdot  \partial^{\langle p^{d+1}b \rangle} \cdot \left( \sum_{a = 0}^{p-1}\frac{(-1)^{a}}{a!} \cdot \prod_{j = 1}^{a}(\partial^{\langle p^d \rangle} - (a-j))\right)\notag\\
& =  (-1)^{p^{d+1}b} \cdot \partial^{\langle p^{d+1}b \rangle} \cdot (1- (\partial^{\langle p^d \rangle})^{p-1}). \notag
\end{align}
By applying \eqref{dGG1} for various $b$'s and $d$'s, we obtain the following sequence of equalities:
\begin{align}
\sum_{j =0}^{p^\ell-1} (-1)^j  \cdot \partial^{\langle j \rangle} &= \sum_{b = 0}^{p^{\ell -1}-1}
\sum_{a = 0}^{p-1}   (-1)^{a+ pb} \cdot \partial^{\langle a + p b \rangle} \\
& \stackrel{\eqref{dGG1}}{=} (1-(\partial^{\langle 1\rangle})^{p-1}) \cdot \sum_{b = 0}^{p^{\ell-1}-1} (-1)^{pb} \cdot \partial^{\langle pb \rangle} \notag  \\
& = (1-(\partial^{\langle 1\rangle})^{p-1}) \cdot 
\sum_{b=0}^{p^{\ell-2}-1} \sum_{a=0}^{p-1} (-1)^{pa + p^2 b} \cdot \partial^{\langle pa + p^2 b \rangle} \notag \\
& \stackrel{\eqref{dGG1}}{=}  (1-(\partial^{\langle 1\rangle})^{p-1}) (1-(\partial^{\langle p\rangle})^{p-1}) \cdot \sum_{b = 0}^{p^{\ell-2}-1} (-1)^{p^2 b} \cdot \partial^{\langle p^2 b \rangle}
\notag \\
&  \ \,  \vdots \notag \\
& = \prod_{s=0}^{\ell-1} (1-(\partial^{\langle p^s\rangle})^{p-1}). \notag
\end{align}
This completes the proof of this assertion.
\end{proof}
\SSP

By  a {\bf (left) $\mcD^{(\M)}$-module}, we shall mean  a pair $(\mcF, \DMO)$ consisting of  an $\mcO_X$-module $\mcF$
and an $\mcO_X$-linear morphism of  $f^{-1}(\mcO_S)$-algebras $\DMO : {^L \mcD}^{(\M)} \migi \mcE nd_{\mcO_S} (\mcF) \left(:= \mcE nd_{f^{-1}(\mcO_S)}(\mcF) \right)$.
We refer to $\DMO$ as  a {\bf $\mcD^{(\M)}$-module structure} on $\mcF$.
When we want to clarify the level ``$\M$",
we  write $\DMO^{(\M)}$ instead of $\DMO$. 
By \eqref{EE658},
the following Leibniz rule holds for any local sections $a \in \mcO_X$, $v \in \mcF$:
\begin{align} \label{EE665}
\nabla (\partial^{\langle j \rangle}) (a \cdot  v) = \sum_{j' + j'' = j} \frac{q_j !}{q_{j'}! \cdot q_{j''}!} \cdot 
\partial^{\langle j' \rangle}(a) \cdot  \nabla (\partial^{\langle j'' \rangle})  (v). 
\end{align}

Also, an {\bf invertible $\mcD^{(\M)}$-module} is  a $\mcD^{(\M)}$-module $(\mcL, \DMO)$ such that $\mcL$ is a line bundle.
The structure sheaf $\mcO_X$ equipped with the $\mcD^{(\M)}$-module structure $\DMO_{X, \mr{triv}}^{(\M)}$ given by the natural $\mcD^{(\M)}$-action (i.e., $\DMO_{X, \mr{triv}}^{(\M)}(\partial^{\langle j \rangle}) (a) := \partial^{\langle j \rangle} (a)$ for $a \in \mcO_X$) defines the trivial invertible $\mcD^{(\M)}$-module
\begin{align}\label{dGG6}
(\mcO_X, \DMO_{X, \mr{triv}}^{(\M)}).
\end{align}

Given   a $\mcD^{(\M)}$-module $(\mcF, \DMO)$,
we shall write
\begin{align}\label{dE23}
\mcS ol (\DMO)
\end{align}
for the subsheaf of $\mcF$ on which $\mcD_{ +}^{(\M)}$ acts as zero via $\DMO$, where $\mcD_{+}^{(\M)}$ denotes the kernel of the canonical projection $\mcD^{(\M)} \migisurj \mcO_X$.

Recall from ~\cite[Definition 4.1]{Wak8}  that an {\bf $S^\mr{log}$-connection} on  an $\mcO_X$-module $\mcF$ is an $f^{-1}(\mcO_S)$-linear morphism 
$\nabla : \mcF \rightarrow \Omega  \otimes \mcF$ satisfying $\nabla (a \cdot v) = da \otimes v + a \cdot \nabla (v)$ for any local sections $a \in \mcO_X$ and $v \in \mcF$. 
(Since we have assumed that $X^\mr{log}/S^\mr{log}$ is a log curve, any $S^\mr{log}$-connection is automatically flat, in the sense of ~\cite[Definition 4.3]{Wak8}.)
Under the bijective correspondence mentioned in ~\cite[Eq.\,(468)]{Wak8},
we will not distinguish between a $\mcD^{(0)}$-module structure on $\mcF$ and an $S^\mr{log}$-connection   on $\mcF$.
In particular, if $\nabla$ is an $S^\mr{log}$-connection on $\mcF$,
then we have $\mcS ol (\nabla) = \mr{Ker}(\mcF \xrightarrow{\nabla} \Omega\otimes \mcF)$; each local section of $\mcS ol (\nabla)$ is called {\bf horizontal}.

By a {\bf flat module} (resp., a {\bf flat bundle}) on $X^\mr{log}/S^\mr{log}$,
we mean 
 an $\mcO_X$-module (resp., a vector bundle on $X$) together with an $S^\mr{log}$-connection on it.
 (Note that flat modules/bundles can be discussed without choosing a PD structure on $S$.)

Let  $(\mcF_\circ, \DMO_{\circ})$ and $(\mcF_\bullet, \DMO_{\bullet})$ be $\mcD^{(\M)}$-modules.
Then,   a {\bf morphism} from  $(\mcF_\circ, \DMO_{\circ})$ to $(\mcF_\bullet, \DMO_{\bullet})$ is  defined as an $\mcO_X$-linear morphism $\mcF_\circ \migi \mcF_\bullet$  compatible with  the respective $\mcD^{(\M)}$-module structures $\DMO_{\circ}$, $\DMO_{\bullet}$.


Let 
$\M'$ be an integer $\geq \M$.
Each    $\mcD^{(\M')}$-module
$(\mcF, \DMO^{(\M')})$ induces   a $\mcD^{(\M)}$-module structure
\begin{align} \label{dGG10}
\DMO^{(\M')\Rightarrow (\M)} : {^L}\mcD^{(\M)} \migi \mcE nd_{\mcO_S} (\mcF)
\end{align}
  on $\mcF$, i.e., the composite of $\DMO^{(\M')}$ and the natural morphism   ${^L}\mcD^{(\M)} \migi {^L}\mcD^{(\M')}$.
 In particular, we obtain 
 a $\mcD^{(\M)}$-module 
$(\mcF, \nabla^{(\M')\Rightarrow (\M)})$.

\LSP
\subsection{$m$-PD stratifications  on  $G$-bundles} \label{SS0239}

Let us fix a smooth  affine algebraic group $G$ over $S$.
Denote by $\mcO_G$ the $\mcO_S$-algebra defined as the coordinate ring of $G$.
Also, let us fix $m \in \mbZ_{\geq 0}$.

\SSP
\bde  \label{D057}
\begin{itemize}
\item[(i)]
Let $\mcE$ be a $G$-bundle on $X$.
An {\bf $m$-PD stratification} on $\mcE/X^\mr{log}/S^\mr{log}$ (or simply, an $\M$-PD stratification on $\mcE$)  is
a collection
\begin{align}
\STR :=  \{\STR_\ell  \}_{\ell \in \mbZ_{\geq 0}},
\end{align}
where each $\STR_\ell$ denotes an isomorphism $P_{(\M)}^{\ell}  \times_X \mcE
\left(=\mr{pr}_2^{\ell  *}(\mcE)\right)
 \isom \mcE \times_X P^{\ell}_{(\M)}
  \left(=\mr{pr}^{\ell*}_1 (\mcE)\right)$
   of $G$-bundles  on $P_{(\M)}^{\ell}$
satisfying the following conditions:
\begin{itemize}
\item
 $\STR_0$ coincides with the identity morphism $\mr{id}_\mcE$ of $\mcE \left(= \mr{pr}_1^{0 *}(\mcE) = \mr{pr}_{2}^{0 *}(\mcE) \right)$, and  the equality  $\STR_{\ell'} |_{P_{(\M)}^{\ell}} = \STR_{\ell}$  holds  for  any pair of integers 
 $(\ell, \ell')$ with $\ell \leq \ell'$;
\item
The cocycle condition holds:
to be precise, for any pair of nonnegative  integers $(\ell, \ell')$,  the following diagram is commutative:
\begin{align} \label{E00235}
\vcenter{\xymatrix@C=6pt@R=36pt{
P_{(\M)}^\ell \times_X P_{(\M)}^{\ell'} \times_X  \mcE
\ar[rr]_-{\sim}^{\delta_{\M}^{\ell, \ell' \sharp *}(\STR_{\ell+\ell'})} \ar[rd]^-{\sim}_-{q_2^{\ell, \ell'*}(\STR_{\ell+\ell'})}&&  
\mcE\times_X P_{(\M)}^\ell \times_X P_{(\M)}^{\ell'}
\\
& P_{(\M)}^{\ell} \times_X \mcE \times_X P_{(\M)}^{\ell'} \ar[ru]^-{\sim}_-{q_1^{\ell, \ell' *}(\STR_{\ell+\ell'})}, &
 }}
\end{align}
where $q_1^{\ell, \ell'}$ and $q_2^{\ell, \ell'}$  denote  the morphisms defined as
 \begin{align} \label{Ee1}
 q_1^{\ell, \ell'} : P_{(\M)}^\ell \times_X P_{(\M)}^{\ell'} \xrightarrow{\mr{pr}_1} P_{(\M)}^\ell \migiincl  P_{(\M)}^{\ell+\ell'} \ \ \text{and} \ 
 \
  q_2^{\ell, \ell'} : P_{(\M)}^\ell \times_X P_{(\M)}^{\ell'} \xrightarrow{\mr{pr}_2} P_{(\M)}^{\ell'} \migiincl  P_{(\M)}^{\ell+\ell'},
 \end{align}
 respectively.
\end{itemize}

Also, by an {\bf $m$-PD stratified $G$-bundle}
   on $X^\mr{log}/S^\mr{log}$, we mean    a pair 
\begin{align}
(\mcE, \STR)
\end{align}
 consisting of a $G$-bundle $\mcE$ on $X$ and  an $m$-PD stratification $\STR$ on $\mcE$.
\item[(ii)]
Let $(\mcE_\circ, \STR_\circ)$ and 
$(\mcE_\bullet, \STR_\bullet)$ be $m$-PD stratified   $G$-bundles on $X^\mr{log}/S^\mr{log}$, where $\STR_\circ := \{ \STR_{\circ, \ell} \}_\ell$,  $\STR_\bullet :=\{ \STR_{\bullet, \ell} \}_\ell$.
An {\bf isomorphism of $m$-PD stratified   $G$-bundles} from 
$(\mcE_\circ, \STR_\circ)$ to  $(\mcE_\bullet, \STR_\bullet)$
is defined as an isomorphism of $G$-bundles $h : \mcE_\circ \isom  \mcE_\bullet$   such that, for each $\ell \in \mbZ_{\geq 0}$, the following square diagram is commutative:
\begin{align} \label{E00237}
\vcenter{\xymatrix@C=46pt@R=36pt{
  P_{(\M)}^\ell \times_X \mcE_\circ 
 \ar[r]_-{\sim}^-{\STR_{\circ, \ell}} \ar[d]^-{\wr}_-{\mr{id} \times h} &
 \mcE_\circ \times_X P_{(\M)}^\ell
\ar[d]_-{\wr}^-{h\times \mr{id}}
\\
  P_{(\M)}^\ell \times_X \mcE_\bullet
 \ar[r]^-{\sim}_{\STR_{\bullet, \ell}}&
 \mcE_\bullet \times_X P_{(\M)}^{\ell}.
 }}
\end{align}

\end{itemize}
 \ede
\SSP

\begin{exa}[Trivial $\M$-PD stratified $G$-bundle] \label{NN601}
The trivial $G$-bundle $X \times_S G$ on $X$ admits an $\M$-PD stratification.
Indeed, for each $\ell\in \mbZ_{\geq 0}$,  let $\STR_{\ell, \mr{triv}}$ denote the composite of natural morphisms
\begin{align}
P^\ell_{(\M)} \times_X (X \times_S G) \isom P^\ell_{(\M)} \times_S G
\xrightarrow{(a, b) \mapsto (b, a)} G \times_S P^\ell_{(\M)}  \isom 
(X \times_S G) \times_X P^\ell_{(\M)}.
\end{align}
Then, the resulting collection
\begin{align} \label{NN600}
\STR_{\mr{triv}} := \{ \STR_{\ell, \mr{triv}} \}_{\ell \in \mbZ_{\geq 0}}
\end{align}
forms an  $\M$-PD stratification on $X \times_S G$, which will be called the {\bf trivial $\M$-PD stratification}.
\end{exa}

\SSP
\begin{rema}[Case of $G = \mr{GL}_n$] \label{dGGg5}
Let  $\mcF$ be a rank $n \left(\in \mbZ_{> 0}\right)$ vector bundle on $X$, and denote by
$\mcE$ the $\mr{GL}_n$-bundle corresponding to $\mcF$.
Suppose that we are given an $\M$-PD stratification $\STR := \{ \STR_\ell \}_{\ell}$ on $\mcE$.
Since $\mcE$ represents  the  sheaf of isomorphisms $\mcI som(\mcO_X^{\oplus n}, \mcF)$,
each $\STR_\ell$ ($\ell \in \mbZ_{\geq 0}$) may be regarded as  an isomorphism
\begin{align}
\mcI som (\mcO_{P_{(\M)}^{\ell}}^{\oplus n}, \mr{pr}_2^{\ell*}(\mcF)) \isom \mcI som (\mcO_{P_{(\M)}^{\ell}}^{\oplus n}, \mr{pr}_1^{\ell*}(\mcF)),
\end{align}
 which is given by composition with  some  isomorphism $\STR^\natural_\ell :  \mr{pr}_2^{\ell*}(\mcF) \isom \mr{pr}_1^{\ell*}(\mcF)$.
The resulting collection $\STR^\natural := \{ \STR^\natural_\ell\}_\ell$ defines 
      an $\M$-PD stratification on $\mcF$ in the classical sense.
 Moreover, it determines a $\mcD^{(\M)}$-module structure on $\mcF$ (cf. ~\cite[Proposition 2.6.1]{Mon}).
 The assignment $\STR \mapsto \STR^\natural$ gives  an equivalence  between
  $\M$-PD stratifications on $\mcE$  and $\mcD^{(\M)}$-module structures on $\mcF$.
\end{rema}
\SSP

In what follows, we describe an $m$-PD stratification on a $G$-bundle 
not only a $\mr{GL}_n$-bundle (cf. Remark \ref{dGGg5} above)
 by using 
 the classical notion of an $m$-PD stratification defined for  an $\mcO_X$-module.

Let $\mcE$ be a $G$-bundle on $X$.
Since  $\mcE$ is affine over $X$ because of the affineness assumption on $G$,   it determines  an $\mcO_X$-algebra;
we shall  write $\mcO_\mcE$ for   this $\mcO_X$-algebra  by abuse of notation (hence, $\mcS pec (\mcO_\mcE) = \mcE$).
If $\mr{R}_\mcE : \mcE \times_S G \migi \mcE$ denotes
the $G$-action  on $\mcE$,  then it corresponds to 
an $\mcO_X$-algebra morphism $\mr{R}_\mcE^\sharp : \mcO_\mcE \migi \mcO_\mcE \otimes_{\mcO_S}  \mcO_G$. (We shall refer to  $\mr{R}_\mcE^\sharp$ as the {\it $G$-coaction} on $\mcO_\mcE$.)

Now, let
  $\STR := \{ \STR_\ell \}_{\ell}$ be an $m$-PD stratification on $\mcE$.
For each  $\ell \in \mbZ_{\geq 0}$, (the inverse of) the isomorphism $\STR_\ell$ defines a $\mcP_{(\M)}^\ell$-algebra isomorphism 
\begin{align}
\STR^\natural_\ell : \mcP_{(\M)}^{\ell} \otimes_{\mcO_X} \mcO_{\mcE} \isom \mcO_{\mcE} \otimes_{\mcO_X} \mcP_{(\M)}^{\ell}.
\end{align}
The  $G$-equivariance condition on  $\STR_\ell$ can be  interpreted as the commutativity of the following square diagram:
\begin{align} \label{E00100}
\vcenter{\xymatrix@C=28pt@R=36pt{
\mcP_{(\M)}^{\ell} \otimes_{\mcO_X} \mcO_{\mcE}  \ar[r]_-{\sim}^{\STR^\natural_\ell} \ar[d]_-{\mr{id} \otimes\mr{R}_\mcE^\sharp} & \mcO_{\mcE} \otimes_{\mcO_X} \mcP_{(\M)}^{\ell}  \ar[d]^-{\mr{R}_\mcE^\sharp \otimes \mr{id}}
\\
(\mcP_{(\M)}^{\ell} \otimes_{\mcO_X} \mcO_{\mcE}) \otimes_{\mcO_S} \mcO_G \ar[r]^-{\sim}_-{\STR_\ell^\natural \otimes \mr{id}}& (\mcO_{\mcE} \otimes_{\mcO_X} \mcP_{(\M)}^{\ell}) \otimes_{\mcO_S} \mcO_G \left(=  (\mcO_{\mcE} \otimes_{\mcO_S}  \mcO_G) \otimes_{\mcO_X} \mcP_{(\M)}^{\ell}\right).
 }} \hspace{-10mm}
\end{align}
Moreover, the commutativity of \eqref{E00235}  reads the commutativity of the following diagram:
\begin{align} \label{E00121}
\vcenter{\xymatrix@C=-32pt@R=36pt{
\mcP_{(\M)}^\ell \otimes_{\mcO_X} \mcP_{(\M)}^{\ell'} \otimes_{\mcO_X} \mcO_\mcE\ar[dr]^-{\sim}_-{q_2^{\ell,  \ell'  *}( \STR^\natural_{\ell+\ell'})}
\ar[rr]_-{\sim}^-{\delta_{m}^{\ell, \ell'\sharp *} (\STR^\natural_{\ell+\ell'})} 
&& 
\mcO_\mcE \otimes_{\mcO_X} \mcP_{(\M)}^\ell \otimes_{\mcO_X} \mcP_{(\M)}^{\ell'}
\\
& \mcP^\ell_{(\M)} \otimes_{\mcO_X} \mcO_\mcE \otimes_{\mcO_X} \mcP^{\ell'}_{(\M)}. \ar[ur]^-{\sim}_-{q_1^{\ell, \ell'  *}( \STR^\natural_{\ell+\ell'})}&
}}
\end{align}
Thus, 
the resulting collection
\begin{align} \label{Ertfq}
\STR^{\natural} := \{ \STR^\natural_\ell \}_{\ell \in \mbZ_{\geq 0}}
\end{align}
 forms an $m$-PD stratification on $\mcO_\mcE$ in the usual sense (cf. ~\cite[Definition 2.3.1]{PBer1}).

Conversely, suppose that we are given an $m$-PD stratification  $\STR^{\natural} := \{ \STR_\ell^\natural\}_{\ell \geq 0}$ on $\mcO_\mcE$ such that  each $\STR^\natural_\ell$ is   a  $\mcP_{(\M)}^\ell$-algebra  isomorphism $\mcP_{(\M)}^\ell \otimes_{\mcO_X} \mcO_\mcE \isom \mcO_\mcE \otimes_{\mcO_X} \mcP_{(\M)}^\ell$ and
the square diagram  \eqref{E00100} 
  for this collection is  commutative.
By applying the functor $\mcS pec ((-)^{-1})$ to various  $\STR_\ell^\natural$'s,
we obtain a collection
of isomorphisms $P_{(\M)}^\ell \times_X \mcE \isom \mcE \times_X P_{(\M)}^\ell$ ($\ell  \in \mbZ_{\geq 0}$), forming 
  an  $m$-PD stratification on $\mcE$.
\SSP

\begin{rema}[$\mcD^{(\M)}$-module associated to an $\M$-PD stratification] \label{Eruy78}
One may verify that giving a collection as in \eqref{Ertfq} is equivalent to giving 
a {\it compatible} collection
\begin{align} \label{Ler46}
\STR^{\natural \natural} := \{ \STR_\ell^{\natural \natural}\}_{\ell \in \mbZ_{\geq 0}},
\end{align}
consisting of  $\mcO_X$-algebra  morphisms $\STR_\ell^{\natural \natural} : \mcO_\mcE \migi \mcO_\mcE \otimes_{\mcO_X} \mcP_{(\M)}^\ell$ (where the $\mcO_X$-algebra structure of $\mcO_\mcE \otimes_{\mcO_X} \mcP_{(\M)}^\ell$ comes from that of   $\mcP_{(\M)}^\ell$ by right multiplication) such that the following two conditions are fulfilled:
\begin{itemize}
\item
For each $\ell \in \mbZ_{\geq 0}$, the following square diagram is commutative:
\begin{align} \label{E0d0df1}
\vcenter{\xymatrix@C=32pt@R=36pt{
\mcO_\mcE \ar[r]^-{\STR_\ell^{\natural \natural}} \ar[d]_-{\mr{R}_\mcE^\sharp} &  \mcO_\mcE \otimes_{\mcO_X} \mcP_{(\M)}^\ell\ar[d]^-{\mr{R}_\mcE^\sharp \otimes \mr{id}}
\\
\mcO_\mcE \otimes_{\mcO_S} \mcO_G \ar[r]_-{\STR_\ell^{\natural \natural} \otimes \mr{id}} & (\mcO_\mcE \otimes_{\mcO_X} \mcP_{(\M)}^\ell) \otimes_{\mcO_S} \mcO_G \left(= (\mcO_\mcE \otimes_{\mcO_S} \mcO_G) \otimes_{\mcO_X} \mcP_{(\M)}^\ell \right).
}}
\end{align}
\item
The equality $\STR_0^{\natural \natural} = \mr{id}_{\mcO_\mcE}$ holds, and for each pair of nonnegative integers $(\ell, \ell')$, the following square diagram is commutative:
\begin{align} \label{E00df1}
\vcenter{\xymatrix@C=46pt@R=36pt{
\mcO_\mcE \otimes_{\mcO_X} \mcP_{(\M)}^{\ell+\ell'} \ar[r]^-{\mr{id} \otimes \delta_\M^{\ell, \ell' }} &  \mcO_\mcE \otimes_{\mcO_X} \mcP_{(\M)}^{\ell} \otimes_{\mcO_X} \mcP_{(\M)}^{\ell'}
\\
\mcO_\mcE\ar[r]_-{\STR_{\ell'}^{\natural \natural}} \ar[u]^-{\STR_{\ell+\ell'}^{\natural \natural}}& \mcO_\mcE \otimes_{\mcO_X} \mcP_{(\M)}^{\ell'}\ar[u]_-{\STR_\ell^{\natural \natural}\otimes \mr{id}}.
}}
\end{align}
\end{itemize}
In particular,  the collection  $\STR^{\natural \natural}$ may be regarded as  a left $\mcD^{(\M)}$-module structure  on $\mcO_\mcE$ (cf. ~\cite[Proposition 2.3.2]{PBer1}).
\end{rema}
\SSP

\begin{rema}[Case of $\M =0$] \label{WW411}
Recall from ~\cite[Definition 1.28]{Wak8} the notion of a  {\it flat $G$-bundle} on $X^\mr{log}/S^\mr{log}$.
(Although all the schemes and algebraic groups appearing in ~\cite{Wak8} are  defined over a field,  the various   formulations related to flat $G$-bundles can be generalized to  our situation here.)
By an argument similar to the proof of 
~\cite[Proposition 7.8.1]{Wak7} (which deals with the non-logarithmic case),  there exists 
 an equivalence of categories 
  \begin{align} \label{Efjj2}
\begin{pmatrix}
\text{ the groupoid  of } \\
\text{$0$-PD stratified $G$-bundles on $X^\mr{log}/S^\mr{log}$} 
\end{pmatrix}
\isom \begin{pmatrix}
\text{ the groupoid of } \\
\text{flat $G$-bundles  on $X^\mr{log}/S^\mr{log}$} \\
\end{pmatrix}.
\end{align}
With this in mind,
whenever we deal with a $0$-PD stratification  without fixing a ($0$-)PD structure 
on $S$, it is assumed to mean a flat connection.
\end{rema} 
\SSP

Let $(\mcE, \STR)$ (where $\STR := \{ \STR_\ell \}_\ell$) be an $\M$-PD stratified $G$-bundle on $X^\mr{log}/S^\mr{log}$.
Also, let $G'$ be another smooth affine algebraic group over $S$ and  $w : G \migi G'$ a morphism of $S$-groups.
Denote by $\mcE \times^{G, w} G'$, or simply by $\mcE \times^G G'$, the $G'$-bundle induced from $\mcE$ via change of structure group along $w$.
Then, the isomorphism $\STR_\ell$ (for each $\ell \in \mbZ_{\geq 0}$) induces a $G'$-equivariant  isomorphism 
\begin{align}\label{dFFgy7}
\STR_\ell \times^G G' : P_{(\M)}^\ell \times_X (\mcE \times^{G} G') \isom (\mcE \times^{G} G') \times_X P_{(\M)}^\ell,
\end{align}
and the collection
\begin{align} \label{dFFgy6}
\STR \times^G G' := \{ \STR_\ell \times^G G' \}_{\ell \in \mbZ_{\geq 0}}
\end{align}
forms 
 an $\M$-PD stratification 
  on $\mcE \times^G G'$.

\LSP
\subsection{Relative Frobenius morphisms} \label{SS058}

In the rest of this section,
we suppose that
 $S$ is   a scheme over $\mbF_p$ (equipped with the trivial $\M$-PD structure).
We shall  write 
 $F_{X}$, $F_S$, and $F_{S^\mr{log}}$ for
the absolute Frobenius (i.e., $p$-power) endomorphisms of $X$, $S$, $S^\mr{log}$,  respectively.
Also, let us fix a nonnegative integer $\M$.

The  {\bf $(\M +1)$-st Frobenius twist of $X$ over $S$} is, by definition,  the base-change $X^{(\M +1)}$ $\left(:= X \times_{S, F_S^{\M +1}} S \right)$ of $X$ along 
 the $(\M +1)$-st iterate $F_S^{\M +1}$ of $F_S$.
Denote by $f^{(\M +1)} : X^{(\M +1)} \migi S$ the structure morphism of 
$X^{(\M +1)}$, which extends to   a log curve 
$X^{(\M +1)\mr{log}} := X^{\mr{log}} \times_{S^\mr{log}, F^{\M +1}_{S^\mr{log}}} S^\mr{log}$ over $S^\mr{log}$.

The {\bf $(\M +1)$-st relative Frobenius morphism of $X$ over $S$} is  the unique morphism $F_{X/S}^{(\M +1)} : X \migi X^{(\M +1)}$ over $S$ that makes the following diagram commute:
\begin{align}
\vcenter{\xymatrix@C=36pt@R=36pt{
X \ar@/^10pt/[rrrrd]^{F_X^{\M +1}}\ar@/_10pt/[ddrr]_{f} \ar[rrd]_{  \ \ \ \  \  \   F_{X/S}^{(\M +1)}} & & &   &   \\
& & X^{(\M +1)}  \ar[rr]_{\mr{id}_X \times F_S^{\M +1}} \ar[d]^-{f^{(\M +1)}}  \ar@{}[rrd]|{\Box}  &  &  X \ar[d]^-{f} \\
&  & S \ar[rr]_{F_S^{\M +1}} & &  S.
}}
\end{align}
 This extends to a morphism of log schemes $X^\mr{log} \rightarrow X^{(\M +1)\mr{log}}$.
To simplify the notation,  
 we write $F_{X/S} := F_{X/S}^{(1)}$, and write $\Omega^{(\M +1)} := \Omega_{X^{(\M +1)\mr{log}}/S^\mr{log}}$,   $\mcT^{(\M +1)} := \mcT_{X^{(\M +1)\mr{log}}/S^\mr{log}}$.
Also, for convenience, we occasionally write  $X^{(0)}$, $f^{(0)}$,  $\Omega^{(0)}$, and $\mcT^{(0)}$ instead of $X$, $f$,  $\Omega$, and $\mcT$, respectively.

\SSP
\begin{rema}[Cartier type] \label{QH45}
Recall from ~\cite[Theorem 1.1]{KaFu} and ~\cite[Remark 1.2.3]{Og}  (or  the comment following ~\cite[Definition (4.8)]{KaKa}) that the log curve  $X^\mr{log}/S^\mr{log}$
  is {\it of Cartier type}.
  Hence,  the exact relative Frobenius map in the statement  of ~\cite[Theorem 3.1.1]{Og3}  coincides with the usual relative Frobenius morphism $F_{X/S}$.
\end{rema}
\SSP

\bpr \label{Prr4e}
Let $\mcF$ be 
a relatively torsion-free sheaf  on  $X$ of rank $n \in\mbZ_{> 0}$ (cf. ~\cite[Chapter 7, D\'{e}finition 1]{Ses} or ~\cite[Definition 3.1]{Wak8} for the definition of a relatively torsion-free sheaf).
Then, the $\mcO_{X^{(\M +1)}}$-module $F_{X/S*}^{(\M +1)}(\mcF)$
is a relatively torsion-free sheaf of rank $n \cdot p^{\M +1}$.
\epr
\begin{proof}
First, let us consider the case of  $\M = 0$.
Since $F_{X/S*}(\mcF)$ is  flat over $S$,
we may assume, after restricting to  the fiber over each geometric point of $S$, that
$S = \mr{Spec}(k)$ for an algebraically closed field $k$ over $\mbF_p$.
Recall from ~\cite[Theorem 1.1]{KaFu} that $X$ has at most nodal singularities.
When restricted to the smooth locus $X^\mr{sm}$ of $X$,  $\mcF$ becomes  locally free and $F_{X/k}$ becomes  finite and flat of degree $p$.
This implies that  $F_{X/k*}(\mcF) |_{X^{\mr{sm}(1)}}$ is locally free (and hence, relatively torsion-free) of rank $n \cdot p$.
Hence, the problem is reduced to proving  that $F_{X/k*}(\mcF)$ is (relatively) torsion-free of rank $n \cdot p$ at each  nodal point of $X$.

Let us take 
 a nodal point $q$ of  $X$, and write $q^{(1)} := F_{X/k}(q) \in X^{(1)}(k)$.
Denote by $\widehat{\mcF}_{q}$ (resp., $\widehat{\mcO}_{X, q}$; resp., $\widehat{\mcO}_{X^{(1)}, q^{(1)}}$) the completion of the stalk of $\mcF$ (resp., $\mcO_{X}$; resp., $\mcO_{X^{(1)}}$) at $q$ (resp., $q$; resp., $q^{(1)}$) with respect to its maximal ideal 
 and  by $\widehat{\mfm}_{X, q}$ (resp.,   $\widehat{\mfm}_{X^{(1)},q^{(1)}}$)  the maximal ideal of $\widehat{\mcO}_{X, q}$ (resp., $\widehat{\mcO}_{X^{(1)}, q^{(1)}}$).
It follows from ~\cite[Chapter 8, Proposition 2]{Ses}
that $\widehat{\mcF}_{q} \cong \widehat{\mcO}_{X, q}^{\oplus a} \oplus \widehat{\mfm}_{X,q}^{\oplus (n-a)}$ for some $a \in \{0, \cdots, n \}$.
Hence, it suffices to consider the cases where $\widehat{\mcF}_{q} \cong \widehat{\mcO}_{X, q}$ and  $\widehat{\mcF}_{q} \cong \widehat{\mfm}_{X,q}$.
Since the first case was essentially proved in ~\cite[Proposition 3.2]{Wak8}, we only consider the second case, i.e., $\widehat{\mcF}_q \cong \widehat{\mfm}_{X,q}$.

Recall that  $\widehat{\mcO}_{X, q} \cong k [\![x,y]\!]/(xy)$ and $\widehat{\mfm}_{X,q} = x \cdot k[\![x]\!] \oplus y \cdot k[\![y]\!]$.
By using  the injection  $F_{X/k}^* : \widehat{\mcO}_{X^{(1)}, q^{(1)}} \migi \widehat{\mcO}_{X, q}$ induced by $F_{X/k}$, we shall identify $\widehat{\mcO}_{X^{(1)}, q^{(1)}}$ with the subring $k [\![x^p,y^p]\!]/(x^py^p)$ of $k [\![x,y]\!]/(xy)$.
In particular,  we have $\widehat{\mfm}_{X^{(1)}, q^{(1)}} \cong x^p \cdot k[\![x^p]\!] \oplus y^p \cdot k[\![y^p]\!]$.
Let us consider the  $k [\![x^p,y^p]\!]/(x^py^p)$-linear  morphism 
\begin{equation}
(x^p \cdot k[\![x^p]\!] \oplus y^p \cdot k[\![y^p]\!])^{\oplus p} \migi 
x \cdot k[\![x]\!] \oplus y \cdot k[\![y]\!]
\end{equation}
given by 
\begin{equation}
(x^p \cdot A_i + y^p \cdot B_i)_{i=1}^{p} \mapsto  \sum_{i =1}^{p}(x^i \cdot A_i+ y^i \cdot B_i)
 \end{equation}
for 
 any $A_i \in k[\![x^p]\!]$ and  $B_i \in k[\![y^p]\!]$ ($i = 1, \cdots, p$).
This morphism is verified to be bijective.
It follows that  $\widehat{\mfm}_{X, q} $ is isomorphic to $\widehat{\mfm}_{X^{(1)},q^{(1)}}^{\oplus p}$, which is relatively torsion-free.
This completes the proof of the assertion for $\M =0$.

Moreover, since $F_{X/S}^{(\M +1)} = F_{X^{(\M)}/S} \circ  \cdots \circ F_{X^{(1)}/S}\circ F_{X/S}$, the assertion  for $\M \in \mbZ_{> 0}$ can be proved by  successively applying  the assertion just proved.
Thus, we have finished the proof of this proposition.
\end{proof}

\LSP
\subsection{$p^{\M +1}$-curvature} \label{SS0dd58}

The image of the natural morphism
$\mcD^{(\M)}_{\leq p^{\M +1}} \migi 
\mcD^{(\M +1)}_{\leq p^{\M +1}}$ coincides with  $\mcD^{(\M +1)}_{\leq p^{\M +1}-1}$.
If $\varpi : \mcD^{(\M)}_{\leq p^{\M +1}} \migi \mcD^{(\M +1)}_{\leq p^{\M +1}-1}$ denotes  the resulting surjection, then the composite
\begin{align} \label{UU555}
\mcD^{(\M)}_{\leq p^{\M +1}-1} \xrightarrow{\mr{inclusion}}  \mcD^{(\M)}_{\leq p^{\M +1}} \xrightarrow{\varpi}
\mcD^{(\M +1)}_{\leq p^{\M +1}-1}
\end{align}
is an isomorphism.
The composite $\varpi' :  \mcD^{(\M)}_{\leq p^{\M +1}} \migi \mcD^{(\M)}_{\leq p^{\M +1}-1}$ of $\varpi$ and the inverse of  \eqref{UU555}
defines a split surjection of the short exact sequence
\begin{align} \label{WW700}
0 \migi \mcD_{\leq p^{\M +1}-1}^{(\M)} \migi \mcD_{\leq p^{\M +1}}^{(\M)} \migi \left(\mcD_{\leq p^{\M +1}}^{(\M)}/\mcD_{\leq p^{\M +1}-1}^{(\M)} =  \right) \mcT^{\otimes p^{\M +1}} \migi 0.
\end{align} 
Thus, we obtain the $\mcO_X$-linear composite
\begin{align} \label{dGG50}
\psi_{X^\mr{log}/S^\mr{log}} :  F_{X/S}^{(\M +1)*} (\mcT^{(\M +1)}) \left(=\mcT_{}^{\otimes p^{\M +1}} \right) \migiincl \mcD_{\leq p^{\M +1}}^{(\M)} \xrightarrow{\mr{inclusion}}  \mcD_{}^{(\M)},
\end{align}
where the first arrow denotes the split injection of \eqref{WW700} corresponding  to $\varpi'$.
Note that this morphism coincides, via the adjunction relation $F_{X/S}^{(\M +1)*}(-) \dashv F_{X/S*}^{(\M +1)}(-)$, with   the $p^{\M +1}$-curvature map $\mcT^{(\M +1)} \migi F_{X/S*}^{(\M +1)}(\mcD_{}^{(\M)})$  discussed  in ~\cite[Definition 3.10]{Ohk}.

If 
$x$ is a logarithmic coordinate as in Section \ref{SS04f4} and $\partial$ denotes the dual base of $d \mr{log} (x)$, then 
$\psi_{X^\mr{log}/S^\mr{log}}$ sends $\partial^{\otimes p^{\M+1}}$ to $\partial^{\langle p^{\M +1} \rangle}$ (cf. ~\cite[Proposition 3.11]{Ohk}).

\SSP
\bde \label{dGGd9}
\begin{itemize}
\item[(i)]
Let $(\mcF, \DMO)$ be a $\mcD^{(\M)}$-module.
The composite
\begin{align} \label{dGG55}
\psi (\DMO) : F_{X/S}^{(\M +1)*} (\mcT^{(\M +1)}) \left( = \mcT_{}^{\otimes p^{\M +1}}\right) \xrightarrow{\psi_{X^\mr{log}/S^\mr{log}}} \mcD_{}^{(\M)}  \xrightarrow{\DMO} \mcE nd_{\mcO_S} (\mcF)
\end{align}
is called  the {\bf $p^{\M +1}$-curvature (map)} of $\DMO$.
(One can verify  that $\psi (\nabla)$ lies in $\mcE nd_{\mcO_X} (\mcF) \subseteq \mcE nd_{\mcO_S} (\mcF)$.)
Also,  we shall say that $(\mcF, \DMO)$  is {\bf $p^{\M +1}$-flat}, or {\bf dormant},  if $\psi (\DMO) = 0$.
\item[(ii)]
Let  $(\mcE, \STR)$ be an $\M$-stratified $G$-bundle on $X^\mr{log}/S^\mr{log}$.
The {\bf $p^{\M +1}$-curvature (map)}   of  $\STR$ is defined  to be  the $p^{\M +1}$-curvature
 \begin{align}
 \psi (\STR) := \psi (\STR^{\natural \natural}) : F_{X/S}^{(\M +1)*} (\mcT^{(\M +1)}) \left(=\mcT_{}^{\otimes p^{\M +1}}\right) \migi \mcE nd_{\mcO_S} (\mcO_\mcE)
 \end{align}
   of  
the $\mcD_{}^{(\M)}$-module structure $\STR^{\natural \natural}$ on   $\mcO_\mcE$ corresponding to $\STR$ (cf. Remark \ref{Eruy78}).
Also, we shall say that  $(\mcE, \STR)$ is {\bf $p^{\M +1}$-flat}, or {\bf dormant},  if  $\psi (\STR) = 0$ (cf. ~\cite[Definition 3.8]{Wak8} for the case of $\M = 0$).
\end{itemize}
\ede
\SSP

\begin{rema}[The $p^{\M +1}$-curvature of a $\mr{GL}_n$-bundle] \label{dGGj1}
Let $\mcF$ and $\mcE$ be as in Remark \ref{dGGg5}.
Also, let $\DMO$ be a $\mcD^{(\M)}$-module structure on  $\mcF$, and denote by $\STR$ the $\M$-PD stratification on $\mcE$ corresponding to $\DMO$ via the equivalence mentioned in  Remark \ref{dGGg5}.
Then,  it is immediately verified that $\DMO$ has vanishing $p^{\M +1}$-curvature
 if and only if $\STR$ has vanishing  $p^{\M +1}$-curvature.
 In the case where $\M =0$ and $G$ is an algebraic group as before,
 \eqref{Efjj2} restricts to 
an equivalence of categories
 \begin{align} \label{Efjj32}
\begin{pmatrix}
\text{ the groupoid  of $p$-flat $0$-PD stratified} \\
\text{$G$-bundles on $X^\mr{log}/S^\mr{log}$} 
\end{pmatrix}
\isom \begin{pmatrix}
\text{ the groupoid of } \\
\text{$p$-flat $G$-bundles  on $X^\mr{log}/S^\mr{log}$} \\
\end{pmatrix}.
\end{align}
\end{rema}
\SSP

\begin{rema}[Classical definition] \label{RRfg5}
When the log structures of $S$ and $X$ are trivial,
the notion of  $p^{\M +1}$-curvature defined above is essentially the same as the notion of {\it $p$-$\M$-curvature} introduced   in ~\cite[Definition 3.1.1]{LSQ} (for $\mcD^{(\M)}$-modules) and ~\cite[Definition 7.7.1]{Wak7} (for $\M$-PD stratified $G$-bundles).
In particular,  a $\mcD^{(\M)}$-module or an $\M$-PD stratified $G$-bundle is $p^{\M +1}$-flat in the sense of Definition \ref{dGGd9} if and only if it has vanishing $p$-$\M$-curvature in the classical sense.
\end{rema}
\SSP

Let $(\mcF_i, \nabla_i)$ ($i=1, 2$) be $\mcD^{(\M)}$-modules.
One can define a $\mcD^{(\M)}$-module structure $\nabla_1 \otimes \nabla_2$ on the tensor product $\mcF_1 \otimes \mcF_2$ determined locally by
\begin{align}
(\nabla_1 \otimes \nabla_2) (\partial^{\langle j \rangle}) (v_1 \otimes v_2) = \sum_{j' + j'' = j} \frac{q_j !}{q_{j'}!  \cdot q_{j''}!} \cdot  \nabla_1 (\partial^{\langle j' \rangle}) (v_1) \otimes \nabla_2 (\partial^{\langle j'' \rangle}) (v_2)
\end{align}
for $v_1 \in \mcF_1$, $v_2 \in \mcF_2$.
Since the right-hand side of the above equality 
implies
\begin{align}
(\nabla_1 \otimes \nabla_2) (\partial^{\langle p^{\M +1}\rangle})(v_1 \otimes v_2) = \nabla_1 (\partial^{\langle p^{\M+1}\rangle}) (v_1) \otimes v_2 + v_1 \otimes \nabla_2 (\partial^{\langle p^{\M+1} \rangle})(v_2),
\end{align}
the  $p^{\M +1}$-curvature of $\nabla_1 \otimes \nabla_2$ satisfies
\begin{align} \label{EEE498}
\psi (\nabla_1 \otimes \nabla_2)  = \psi (\nabla_1)  \otimes \mr{id}_{\mcF_2} + \mr{id}_{\mcF_1} \otimes \psi (\nabla_2).
\end{align}
In particular, if $(\mcL, \nabla_\mcL)$ is a $\mcD^{(\M)}$-module such that $\mcL$ is a line bundle,
then the $\ell$-fold tensor product $\nabla_{\mcL}^{\otimes \ell} : {^L}\mcD^{(\M)} \rightarrow \mcE nd_{\mcO_S} (\mcL^{\otimes \ell})$ of $\nabla_\mcL$ ($\ell \geq 1$) satisfies
\begin{align} \label{EEE499}
\psi (\nabla_\mcL^{\otimes \ell}) = \ell \cdot \psi (\nabla_\mcL)
\end{align}
under the natural identification $\mcE nd_{\mcO_X} (\mcL) = \mcE nd_{\mcO_X} (\mcL^{\otimes \ell})$.

Also, let $(\mcF, \nabla)$ be a $\mcD^{(\M)}$-module such that $\mcF$ is a vector bundle of rank $n \geq 1$.
The $\mcD^{(\M)}$-module structure $\nabla$ gives rise to a $\mcD^{(\M)}$-module structure $\mr{det}(\nabla)$ on the determinant $\mr{det} (\mcF)$ of $\mcF$ determined locally by
\begin{align} \label{EEE500}
\mr{det} (\nabla) (\partial^{\langle j \rangle}) (v_1 \wedge \cdots \wedge v_n) = \sum_{\substack{j_1 + \cdots + j_n = j,}{ j_i \geq 0}} \frac{q_j !}{q_{j_1}! \cdots q_{j_n}!} \cdot \nabla (\partial^{\langle j_1 \rangle}) (v_1) \wedge \cdots \wedge \nabla (\partial^{\langle j_n \rangle}) (v_n).
\end{align}
This is obtained from the $n$-fold tensor product $(\mcF^{\otimes n}, \nabla^{\otimes n})$ via the quotient $\mcF^{\otimes n} \twoheadrightarrow \mr{det} (\mcF) \left(= \bigwedge^n \mcF \right)$,  and  one can apply the equality  \eqref{EEE498} to obtain 
\begin{align} \label{EEE501}
\psi (\mr{det} (\nabla)) = \mr{tr} \circ \psi (\nabla),
\end{align}
where $\mr{tr}$ denotes the trace map $\mcE nd_{\mcO_X} (\mcF) \rightarrow \mcO_X$.

Next, let $\mcG$ be an $\mcO_{X^{(\M +1)}}$-module.
The comment at the beginning of  ~\cite[Section 4.10.3]{Wak8} says that 
there exists a canonical $\mcD_{X^{(\M)\mr{log}}/S^\mr{log}}^{(0)}$-module structure on 
$F_{X^{(\M)}/S}^*(\mcG)$.
By applying  ~\cite[Corollaire 3.3.1]{Mon} to this,
we obtain  a  $\mcD^{(\M)}$-module  structure
\begin{align} \label{QQwkko}
\nabla_{\mcG, \mr{can}}^{(\M)} : {^L}\mcD^{(\M)} \migi \mcE nd_{\mcO_S} (F_{X/S}^{(\M +1)*}(\mcG))
\end{align}
on  the pull-back $F^{(\M +1)*}_{X/S} (\mcG) \left(= F^{(\M)*}_{X/S} (F_{X^{(\M)}/S}^*(\mcG)) \right)$.
It is immediately verified that $\nabla_{\mcG, \mr{can}}^{(\M)}$ has vanishing $p^{\M +1}$-curvature.

\SSP
\bde \label{dGGe11}
We shall refer to  $\nabla_{\mcG, \mr{can}}^{(\M)}$ as the {\bf canonical $\mcD^{(\M)}$-module  structure} on $F^{(\M +1)*}_{X/S} (\mcG)$.
\ede
\SSP

Note that, for   a $\mcD^{(\M)}$-module $(\mcF, \DMO)$, the subsheaf $\mcS ol (\DMO)$ of $\mcF$ (cf. \eqref{dE23}) may be regarded as an $\mcO_{X^{(\M +1)}}$-module  via the underlying homeomorphism of $F_{X/S}^{(\M +1)}$.
Here, suppose that
 the relative characteristic  of $X^\mr{log}/S^\mr{log}$
    is trivial, which implies that $X/S$ is smooth and $\mcD^{(\M)} = \mcD_{X/S}^{(\M)}$.
Then,   the assignments $\mcG \mapsto (F_{X/S}^{(\M +1)*}(\mcG), \DMO_{\mcG, \mr{can}}^{(\M)})$ and $(\mcF, \DMO) \mapsto \mcS ol (\DMO)$ determine an equivalence of categories 
  \begin{align} \label{YY6}
\begin{pmatrix}
\text{ the category  of } \\
\text{$\mcO_{X^{(\M +1)}}$-modules} 
\end{pmatrix}
\isom \begin{pmatrix}
\text{ the category of $\mcD^{(\M)}$-modules} \\
\text{ with vanishing $p^{\M +1}$-curvature} \\
\end{pmatrix}
\end{align}
(cf. ~\cite[Corollary 3.2.4]{LSQ}).
This equivalence commutes with taking  internal Hom's  ``$\mcH om$" and tensor products ``$\otimes$".

\LSP
\subsection{$\mu_{p^{\M+1}}$-actions arising from $p^{\M+1}$-flat $\M$-PD stratifications} \label{SSrtq}

Suppose that there exists a global logarithmic coordinate $x$, i.e., $\Omega = \mcO_X d \mr{log} (x)$.
We set $\eta := \mu_{(m)} (x)-1$.

Let $(\mcF, \nabla)$ be a $\mcD^{(\M)}$-module.
The $m$-PD stratification $\phi^\natural := \{ \phi_\ell^\natural \}_\ell$  corresponding to $\nabla$ is given by 
\begin{align} \label{EE23}
\phi^\natural_\ell (v) = \sum_{j=0}^\ell \nabla (\partial^{\langle j \rangle}) (v) \otimes \eta^{\{ j \}}
\end{align}
 for $v \in \mcF$.
By definition,  $\psi (\nabla) = 0$ if and only if $\nabla (\partial^{\langle p^{m+1}\rangle}) =0$.

Now suppose that  $\psi (\nabla) = 0$.
Since $\partial^{\langle q p^{m+1} +r\rangle} = (\partial^{\langle p^{m+1} \rangle})^q \partial^{\langle r \rangle}$ (cf. ~\cite[Lemma 3.8]{Ohk}), the equality $\nabla (\partial^{\langle p^{m+1}\rangle}) =0$ implies $\nabla (\partial^{\langle j \rangle}) =0$ for $j \geq p^{m+1}$. 
Hence, \eqref{EE23}   can be expressed as 
\begin{align} \label{EE24}
\phi^\natural_\ell (v) = \sum_{j=0}^{p^{m+1}-1} \nabla (\partial^{\langle j \rangle}) (v) \otimes \eta^{\{ j \}} = 
\sum_{j=0}^{p^{m+1}-1} \frac{1}{q_j !} \cdot \nabla (\partial^{\langle j \rangle}) (v) \otimes \eta^{j},
\end{align}
where the second equality follows from ~\cite[Proposition 3.2, (4)]{Ohk}.
Note that $\eta^{p^{m+1}} = p! \eta^{\{ p^{m+1}\}} = 0$, and that the $\mcO_X$-subalgebra of  $\mcP_{(m)}$ generated by $\eta$ can be identified with  $\overline{\mcP}_{(m)}:= \mcO_X [\eta]/(\eta^{p^{m+1}})$.
In particular,  \eqref{EE24} is valued in $\mcF \otimes \overline{\mcP}_{(m)}$.

The section $z := 1 + \eta$ satisfies $z^{p^{\M+1}} =1$ and $\varDelta (z) = z \otimes z$.
It coincides with  a structure of Hopf algebra on  the coordinate ring $H := \mbF_p [z]/(z^{p^{m+1}-1})$ of the group scheme $\mu_{p^{m+1}}$.
By using \eqref{EE24},  $\phi^\natural$ defines,  a morphism 
\begin{align} \label{EE26}
\rho_{(\mcF, \nabla)} : \mcF \rightarrow \mcF \otimes H.
\end{align}
Since the trivial $\mcD^{(\M)}$-module structure $\nabla^{(\M)}_{X, \mr{triv}}$ on $\mcO_X$ has vanishing $p^{\M +1}$-curvature,
it induces a morphism 
\begin{align}
\rho_{\mr{triv}} := \rho_{(\mcO_X, \nabla_{X, \mr{triv}}^{(\M)})} : \mcO_X \rightarrow \mcO_X \otimes H.
\end{align}
It follows from the definition of an $\M$-stratification that
\begin{align} \label{EE25}
\rho_{(\mcF, \nabla)} (av) &= \rho_{\mr{triv}} (a) \rho_{(\mcF, \nabla)} (v), \\
(\rho_{(\mcF, \nabla)} \otimes \mr{id}_H) \circ \rho_{(\mcF, \nabla)} &= (1 \otimes \varDelta) \circ \rho_{(\mcF, \nabla)}, \\
(1 \otimes \epsilon) \circ \rho_{(\mcF, \nabla)} &= \mr{id}_\mcF,
\end{align}
where $\epsilon$ denotes the counit $H \rightarrow \mbF_p$.
That is, $\mcF$ is $\mu_{p^{\M+1}}$-equivariant with respect to the $\mu_{p^{\M+1}}$-coaction of $\mcO_X$.

Since $1, z, \cdots, z^{p^{\M +1}-1}$ form a basis of $H$ over $\mbF_p$,  there exists a unique decomposition  
\begin{align}
\rho_{(\mcF, \nabla)} (v) = \sum_{j=0}^{p^{\M +1}-1} P_j (v) \otimes z^j
\end{align}
for each $v \in \mcF$.
By comparing the coefficients in   the conditions  \eqref{EE25} concerning coassociativity and counit, 
we have 
\begin{align}
\sum_{j=0}^{p^{\M +1}-1} P_j = \mr{id}_\mcF  \ \ \ \text{and} \ \ \ P_j \circ P_{j'} = \begin{cases} P_j & \text{if $j= j'$} \\ 0 & \text{if $j \neq j'$}. \end{cases}
\end{align}
Therefore,   $\mcF$ decomposes as 
\begin{align} \label{EE27}
\mcF = \bigoplus_{j \in \mbZ/p^{\M +1}\mbZ} \mcF_j, \hspace{5mm}\mcF_j := \{ v \in \mcF \, | \, \rho_{(\mcF, \nabla)} (v) = v \otimes z^j\}.
\end{align}
This decomposition for $(\mcF, \nabla) = (\mcO_X, \nabla^{(\M)}_{X, \mr{triv}})$ defines a decomposition $\mcO_X = \bigoplus_{j \in \mbZ/p^{\M+1}\mbZ} \mcO_{X, j}$, and the first condition in \eqref{EE25} implies $\mcO_{X, j} \mcF_{j'} \subseteq \mcF_{j+ j'}$.
This means that \eqref{EE27} is a decomposition as an $\mcO_{X, 0}$-module (but, not as an $\mcO_X$-module, in general).
For example, if $(\mcF, \nabla) = (\mcO_X, \nabla_{X, \mr{triv}}^{(\M)})$ and $X = \mr{Spec} (\mbF_p [t])$, then $\rho_\mr{triv} (t) = t \otimes z$ and the corresponding decomposition is given by $\mcO_X = \bigoplus_{j=0}^{p^{\M +1}-1}t^j \mcO_{X^{(\M +1)}}$.

By construction, the equality $\mcF_0 = \mcS ol (\nabla)$ holds.
More generally, if we write $j = j_0 + j_1 p + \cdots + j_{\M}p^{\M}$ with $0 \leq j_\ell < p$,
then
we have 
\begin{align}
\mcF_j= \bigcap_{\ell =0}^\M \mr{Ker} \left( \nabla (\partial^{\langle p^{\ell} \rangle}) -j_{\ell}  \cdot \mr{id}_\mcF)\right).
\end{align}
Indeed, for $v \in \mcF_j$, the equality $\rho_{(\mcF, \nabla)} (v) = v\otimes (1 + \eta)^j$ holds.
By comparing it with  \eqref{EE24},  
we have $\nabla (\partial^{\langle j' \rangle}) (v) = q_{j'} ! \binom{j}{j'} v$ ($0\leq j' < p^{\M+1}$), in particular, $\nabla (\partial^{\langle p^\ell \rangle}) (v) =j_{\ell} v$.

Conversely,  given a decomposition  $\mcF = \bigoplus_{j \in \mbZ/p^{\M +1}\mbZ} \mcF_j$ of $\mcF$ satisfying $\mcO_{X, j} \mcF_{j'} \subseteq \mcF_{j + j'}$,
we can define 
$\rho_{\{\mcF_j \}_j}  : \mcF \rightarrow \mcF \otimes H$ as 
$\rho_{\{\mcF_j \}_j} (v) := v \otimes z^j$ for  $v \in \mcF_j$.
It corresponds, via the equality $z = 1 + \eta$, an $\M$-PD stratification, and its $p^{\M +1}$-curvature vanishes.
This construction is inverse to the assignment $(\mcF, \nabla) \mapsto \rho_{(\mcF, \nabla)}$. 

Next, let $\mcE$ be a $G$-bundle on $X$ and $\phi := \{ \phi_\ell \}_{\ell}$ an $\M$-PD stratification on $\mcE/X^\mr{log}/S^\mr{log}$ with vanishing $p^{\M +1}$-curvature.
By  the above discussion,
 $\phi$ gives rise to a decomposition $\mcO_\mcE = \bigoplus_{j \in \mbZ/p^{\M +1}\mbZ} \mcO_{\mcE, j}$.
It follows from \eqref{E00100} and \eqref{E00121} that
$\mcO_{X, j} \mcO_{\mcE, j'} \subseteq \mcO_{\mcE, j+ j'}$, $\mcO_{\mcE, j} \mcO_{\mcE, j'} \subseteq \mcO_{\mcE, j+j'}$, and $R^\sharp_\mcE (\mcO_{\mcE, j}) \subseteq \mcO_{\mcE, j} \otimes_{\mcO_S} \mcO_G$.
It corresponds to 
a $\mu_{p^{\M +1}}$-action
\begin{align} \label{EE30}
\rho_{\phi} : \mu_{p^{\M +1}} \times \mcE \rightarrow \mcE
\end{align}
 on $\mcE$ commuting  with the $G$-action and extending the $\mu_{p^{\M +1}}$-action on $X$ arising from $\nabla^{(\M)}_{X, \mr{triv}}$. 
Conversely, such 
an action determines an $\M$-PD stratification with vanishing $p^{\M +1}$-curvature.
Thus, the assignment $(\mcE, \phi) \mapsto (\mcE, \rho_\phi)$ defines an equivalence of categories 
\begin{align} \label{EE36}
\begin{pmatrix} \text{the groupoid of} \\ \text{$p^{\M +1}$-flat $\M$-stratified} \\ \text{ $G$-bundle on $X^\mr{log}/S^\mr{log}$} \end{pmatrix}
\xrightarrow{\sim}
\begin{pmatrix} \text{the groupoid of $G$-bundles on $X$} \\ \text{equipped with a $\mu_{p^{\M +1}}$-action} \\ \text{extending  the $\mu_{p^{\M +1}}$-action on $X$ and} \\ \text{commuting with the natural $G$-action} \end{pmatrix}.
\end{align}

\LSP
\subsection{Cartier operator of a $p^{\M+1}$-flat $\mcD^{(\M)}$-module} \label{SS0ddd58}

Let $(\mcF, \DMO^{(\M)})$ be a $\mcD^{(\M)}$-module. 
For an  integer $a$ with $0\leq a \leq m+1$,  we set 
\begin{align} \label{dE48}
\mcF^{[a]} := \begin{cases} \mcF & \text{if $a = 0$};  \\
\mcS ol (\DMO^{(\M) \Rightarrow (a-1)}) & \text{if $a > 0$}.
\end{cases}
\end{align}
In particular, $\mcF^{[a]}$ is an $\mcO_{X^{(a)}}$-module.

In what follows, let us define an $S^\mr{log}$-connection on the $\mcO_{X^{(a)}}$-module $\mcF^{[a]}$ for $a = 0, \cdots, \M$.
First, we shall set $\nabla^{[0]}$ (or $(\nabla^{(\M)})^{[0]}$) $:= \nabla^{(\M) \Rightarrow (0)}$.

Next, let us choose  $a \in \{1, \cdots, \M\}$.
Since  
$\mcD_{\leq p^a -1}^{(a)} = \mr{Im}\left(\mcD^{(a-1)}_{\leq p^{a}} \migi \mcD^{(a)}_{\leq p^{a}}\right)$ and 
 $\mcD_{\leq p^a }^{(a)}/\mcD_{\leq p^a -1}^{(a)} = \mcT^{\otimes p^a}$, we obtain 
an exact sequence
\begin{align} \label{dE45}
\mcD^{(a-1)}_{\leq p^{a}}\cap \mcD^{(a-1)}_{+} \migi
 \mcD^{(a)}_{\leq p^{a}}\cap \mcD^{(a)}_{+}  \xrightarrow{\delta} \mcT^{\otimes p^{a}} \left(= F_{X/S}^{(a)*}(\mcT^{(a)}) \right) \migi 0,
\end{align}
where the first arrow is the morphism obtained by restricting the  morphism $\mcD^{(a-1)} \migi \mcD^{(a)}$.
Let us take a local section $\partial$ of $\mcT^{(a)}$.
There exists locally  a section $\widetilde{\partial}$ in $\delta^{-1} ((F_{X/S}^{(a)})^{-1}(\partial)) \left(\subseteq \mcD^{(a)}_{\leq p^{a}}\cap \mcD^{(a)}_{+} \right)$.
The exactness of \eqref{dE45} and the definition of $\mcF^{[a]}$ imply  that 
 the $f^{-1}(\mcO_S)$-linear endomorphism $\DMO_\partial^{[a]} := \DMO^{(\M)\Rightarrow {(a)}} (\widetilde{\partial})$ of $\mcF^{[a]}$ does not depend on the choice of  $\widetilde{\partial}$ (i.e., depends only on $\partial$).
 Hence, the morphism 
 \begin{align} \label{dE60}
 \nabla^{[a]}  \left( \text{or} \ (\DMO^{(\M)})^{[a]} \right) : \mcF^{[a]} \migi \Omega^{(a)} \otimes \mcF^{[a]}
 \end{align}
 determined  by assigning $\partial \mapsto \DMO_\partial^{[a]}$ is well-defined, and 
 this  is verified to form an $S^\mr{log}$-connection.

Thus, we obtain a flat module
\begin{align} \label{J16}
(\mcF^{[a]}, \DMO^{[a]})
\end{align}
on the log curve $X^{(a)\mr{log}}/S^\mr{log}$.
 The collection $\{ \mcF^{[a]} \}_{0 \leq a \leq \M +1}$ defines a decreasing filtration on $\mcF$ such that $\mr{Ker}(\nabla^{[a]}) = \mcF^{[a+1]}$ for every $a = 0,1, \cdots, \M$.

\SSP
\bpr \label{UU577}
Let us keep the above notation.
\begin{itemize}
\item[(i)]
  If $\DMO^{(\M)}$ has vanishing $p^{\M +1}$-curvature,
  then $\nabla^{[\M]}$ has vanishing $p$-curvature.
    \item[(ii)]
  For  every $a = 0, \cdots,\M -1$,
the $S^\mr{log}$-connection $\nabla^{[a]}$ 
  has vanishing $p$-curvature.
\end{itemize}
\epr
\begin{proof}
First, we shall prove assertion (i).
Denote by $\mcE nd (\mcF, \DMO^{(\M)})$ (resp., $\mcE nd (\mcF^{[\M]}, \DMO^{[\M]})$) the subsheaf of $\mcE nd_{\mcO_S} (\mcF)$  (resp., $\mcE nd_{\mcO_S} (\mcF^{[\M]})$)
consisting of endomorphisms  preserving the $\mcD^{(\M)}$-module structure $\DMO^{(\M)}$ (resp., the $S^\mr{log}$-connection   $\DMO^{[\M]}$).
The sheaf   $\mcE nd (\mcF, \DMO^{(\M)})$  (resp., $\mcE nd (\mcF^{[\M]}, \DMO^{[\M]})$)  is equipped with an  $\mcO_{X^{(\M +1)}}$-module structure in a natural manner.
The $p^{\M +1}$-curvature $\psi (\DMO^{(\M)}) : F_{X/S}^{(\M +1)*} (\mcT^{(\M +1)}) \migi \mcE nd_{\mcO_S} (\mcF)$ of $\DMO^{(\M)}$  (resp., the $p$-curvature $\psi (\DMO^{[\M]}) : F_{X^{(\M)}/S}^* (\mcT^{(\M +1)}) \migi \mcE nd_{\mcO_S} (\mcF^{[\M]})$ of $\DMO^{[\M]}$) restricts to an $\mcO_{X^{(\M +1)}}$-linear  morphism
\begin{align} \label{UU580}
\psi (\DMO^{(\M)})^\nabla : \mcT^{(\M +1)}\migi  \mcE nd (\mcF, \DMO^{(\M)}) \ \left(\text{resp.,} \ \psi (\DMO^{[\M]})^\nabla : \mcT^{(\M +1)} \migi \mcE nd (\mcF^{[\M]}, \DMO^{[\M]})\right).
\end{align}
We obtain an equivalence
 \begin{align} \label{UU814}
 \psi (\DMO^{(\M)}) = 0 \ 
 \Longleftrightarrow \ 
  \psi (\DMO^{(\M)})^\nabla = 0 \ \left(\text{resp.,} \  \psi (\DMO^{[\M]}) = 0
\ \Longleftrightarrow \ 
 \psi (\DMO^{[\M]})^\nabla = 0\right).
  \end{align}
Hence, the assertion follows from \eqref{UU814} and  the commutativity of the following diagram:
\begin{align} \label{UU812}
\vcenter{\xymatrix@C=46pt@R=36pt{
&\mcT^{(\M +1)} \ar[rd]^-{\psi (\DMO^{[\M]})^\nabla} \ar[ld]_-{\psi (\DMO^{(\M)})^\nabla}&
\\
\mcE nd(\mcF, \DMO^{(\M)}) \ar[rr]_-{\mr{Res}}&& \mcE nd (\mcF^{[\M]}, \DMO^{[\M]}),
}}
\end{align}
where the lower horizontal arrow $\mr{Res}$ denotes the morphism given by $h \mapsto h |_{\mcF^{[\M]}}$ for any $h \in \mcE nd (\mcF, \DMO^{(\M)})$.

To prove assertion (ii),
we note that  the $\mcD_{}^{(a)}$-module structure $\DMO^{(\M)\Rightarrow (a)}$ has vanishing $p^{(a+1)}$-curvature.
Hence, the assertion can be proved by the same  argument  as the proof of  (i),   where $\DMO^{(\M)}$ and $\DMO^{[\M]}$ are  replaced by $\DMO^{(\M)\Rightarrow (a)}$ and $\DMO^{[a]}$, respectively.
\end{proof}
\SSP

 The following assertion is a slight generalization of ~\cite[Proposition 3.2]{Wak8}.

\SSP
\bpr \label{PPer4}
Let us keep the above notation.
Suppose further  that $\psi (\DMO^{(\M)}) = 0$  and that $\mcF$ is  a relatively torsion-free sheaf  of rank $n \in \mbZ_{> 0}$.
Also, let us fix an integer $a$ with $0 \leq a \leq \M+1$.
\begin{itemize}
\item[(i)]
The $\mcO_{X^{(a)}}$-module $\mcF^{[a]}$ is relatively torsion-free of rank $n$. 
In particular, the $\mcO_{X^{(\M +1)}}$-module $\mcS ol (\DMO^{(\M)})$ is a relatively torsion-free sheaf of rank $n$.
\item[(ii)]
For $a \in \{0, \cdots, \M \}$,
the formation of $(\mcF^{[a]}, \DMO^{[a]})$ commutes with base-change over $S$-schemes.
To be precise, let $s : S' \migi S$ be a morphism of $\mbF_p$-schemes, and  use the notation ``$s^*(-)$" to denote the result of base-changing along $s$.
(In particular, we obtain a $\mcD_{X'^{\mr{log}}/S'^{\mr{log}}}^{(\M)}$-module $(s^*(\mcF), s^*(\DMO^{(\M)}))$, where $S'^{\mr{log}} := S' \times_S S^\mr{log}$ and $X'^{\mr{log}} := S'^{\mr{log}} \times_{S^\mr{log}} X^\mr{log}$.)
Then, the natural   morphism  of flat modules
\begin{align}
(s^*(\mcF^{[a]}), s^* (\DMO^{[a]}))\migi ((s^*(\mcF))^{[a]}, (s^*(\DMO^{(\M)}))^{[a]})
\end{align}
constructed inductively on $a$ is an isomorphism.
In particular, the natural morphism $s^*(\mcS ol (\DMO^{(\M)})) \migi \mcS ol (s^*(\DMO^{(\M)}))$ is an isomorphism.
\end{itemize}
\epr
\begin{proof}
There is nothing to prove when  $a = 0$.
Assertions (i) and  (ii) for $a =1$ can be proved by  arguments entirely similar to the proof of ~\cite[Proposition  6.13]{Wak8} together with Proposition \ref{Prr4e} for $\M = 0$.
Moreover,  the remaining cases, i.e.,  the assertions for $a > 1$, can be proved by successively applying  the  assertions for $a = 1$ and Proposition \ref{UU577}.
\end{proof}
\SSP

\begin{rema}[Local case] \label{WW9}
Assertion (ii) of Proposition \ref{PPer4} remains true (for the same reason) even  if we replace $X^\mr{log}/S^\mr{log}$ with  $U_{\oslash}^{\mr{log}}/S$ or $U_{\otimes}^\mr{log}/S^\mr{log}$ introduced in Section \ref{SS1071} later.
The assertion for these situations  will be used in the proof of Proposition-Definition \ref{P022}.
\end{rema}
\SSP

\bco \label{KK1}
Let us keep the above notation.
Suppose that  either $\mcF$ is relatively torsion-free or the relative characteristic  of $X^\mr{log}/S^\mr{log}$ is trivial.
Then, the  converse of Proposition \ref{UU577}, (i),  is true.
To be precise, $\DMO^{(\M)}$ has vanishing $p^{\M +1}$-curvature if $\DMO^{[\M]}$ has vanishing $p$-curvature.
\eco
\begin{proof}
The assertion is immediate when $\M =0$.
Hence, suppose that $\M \geq 1$.
Denote by $\iota : U \migiincl X^{(\M +1)}$  the open subscheme of $X^{(\M +1)}$ where 
the relative characteristic of $X^{(\M +1)\mr{log}}/S^\mr{log}$ is trivial.
Since $U$  is scheme-theoretically dense (cf. ~\cite[Lemma 1.4]{KaFu}),
 the  assumption imposed  above and Propositions \ref{Prr4e} and \ref{PPer4} together  imply  that
the natural morphisms  
\begin{align} \label{UU813}
\mcE nd(\mcF, \DMO^{(\M)})  &\migi \iota_*(\iota^*(\mcE nd (\mcF, \DMO^{(\M)\Rightarrow (\M -1)}))),  \\
\mcE nd (\mcF^{[\M]}, \DMO^{[\M]}) &\migi \iota_* (\iota^* (\mcE nd_{\mcO_{X^{(\M)}}} (\mcF^{[\M]})))  \notag
\end{align}
 are injective.
Since $\nabla^{(\M)\Rightarrow (\M -1)}$ has vanishing $p^\M$-curvature,
 it follows from  the equivalence of categories \eqref{YY6} with $\M$ replaced by $\M -1$ that
 there exists a canonical isomorphism
 \begin{align}
 \iota_*(\iota^*(\mcE nd (\mcF, \DMO^{(\M)\Rightarrow (\M-1)})))\isom \iota_* (\iota^* (\mcE nd_{\mcO_{X^{(\M)}}} (\mcF^{[\M]}))),
 \end{align}
 which is compatible with 
 $\mr{Res}$ (in \eqref{UU812})
  via the morphisms  \eqref{UU813}.
Hence, 
the morphism $\mr{Res}$ turns out to be   injective.
Hence, the assertion  follows from  this fact together with \eqref{UU814} and the commutativity of \eqref{UU812}.
\end{proof}
\SSP

Next, suppose further  that $(\mcF, \DMO^{(\M)})$ is $p^{\M +1}$-flat.
Let us take  an integer $a$ with  $0 \leq a \leq \M$.
Recall from ~\cite[Proposition 1.2.4]{Og} that the {\it Cartier operator} associated to $(\mcF^{[a]}, \DMO^{[a]})$ is, by definition, an $\mcO_{X^{(a+1)}}$-linear morphism 
\begin{align} \label{dGGfh3}
\Omega^{(a)} \otimes \mcF^{[a]} \migi \Omega^{(a+1)} \otimes F_{X^{(a)}/S*}(\mcF^{[a]})
\end{align}
 satisfying a certain condition, where the domain of this morphism is regarded as an $\mcO_{X^{(a+ 1)}}$-module via $F_{X^{(a)}/S}$. 
Since  $(\mcF^{[a]}, \DMO^{[a]})$ has vanishing $p$-curvature,
the image of this morphism lies in $\Omega^{(a+1)} \otimes \mcF^{[a +1]}\left( =\Omega^{(a+1)} \otimes \mr{Ker}(\nabla^{[a]})\right)$ (cf. the comment following ~\cite[Proposition 1.2.4]{Og}).
Hence, by restricting the codomain of  \eqref{dGGfh3},
we obtain an $\mcO_{X^{(a+1)}}$-linear morphism
\begin{align} \label{dEEw2}
C_{(\mcF^{[a]}, \nabla^{[a]})} : \Omega^{(a)} \otimes \mcF^{[a]} \migi  \Omega^{(a+1)} \otimes \mcF^{[a +1]}.
\end{align}

Moreover,
the composite $C_{(\mcF^{[\M]}, \nabla^{[\M]})} \circ \cdots \circ C_{(\mcF^{[1]}, \nabla^{[1]})} \circ C_{(\mcF^{[0]}, \nabla^{[0]})}$
determines an $\mcO_{X^{(\M +1)}}$-linear morphism
\begin{align} \label{dE3}
C_{(\mcF, \DMO^{(\M)})} : \Omega \otimes \mcF \migi \Omega^{(\M +1)} \otimes \mcS ol (\DMO^{(\M)}).
\end{align}

\SSP
\bde \label{DDDf}
We shall refer to $C_{(\mcF, \DMO^{(\M)})}$ as the {\bf  Cartier operator} of $(\mcF, \DMO^{(\M)})$.
\ede
\SSP

\begin{rema}[Functoriality of the Cartier operator] \label{Rem4451}
The formation of the Cartier operators is functorial in the following sense:
Let $h : (\mcF_\circ, \DMO_\circ) \migi (\mcF_\bullet, \DMO_\bullet)$ be a morphism of $p^{\M +1}$-flat $\mcD^{(\M)}$-modules.
Then, this morphism restricts to an $\mcO_{X^{(\M +1)}}$-linear morphism $\mcS ol (h) : \mcS ol (\DMO_\circ) \migi \mcS ol (\DMO_\bullet)$, and the following square diagram is commutative:
\begin{align}  \label{44f01}
\vcenter{\xymatrix@C=56pt@R=36pt{
 \Omega \otimes \mcF_\circ\ar[r]^-{C_{(\mcF_\circ, \DMO_\circ)}} \ar[d]_-{\mr{id} \otimes h} & \Omega^{(\M +1)} \otimes \mcS ol (\DMO_\circ)  \ar[d]^-{\mr{id} \otimes \mcS ol (h)}\\
 \Omega \otimes \mcF_\bullet\ar[r]_-{C_{(\mcF_\bullet, \DMO_\bullet)}}  &\Omega^{(\M +1)} \otimes \mcS ol (\DMO_\bullet).
}}
\end{align}
\end{rema}

\LSP
\subsection{Dual of the Cartier operator} \label{SS44534}

Suppose that  the relative characteristic   of $X^\mr{log}/S^\mr{log}$ is trivial (which implies that $X/S$ is smooth), and that $f : X \rightarrow S$ is proper.
Let $(\mcF, \DMO)$ be a $p^{\M+1}$-flat $\mcD^{(\M)}$-module with $\mcF$ locally free.
By   Grothendieck-Serre duality,  there exists a canonical isomorphism of $\mcO_S$-modules
\begin{align} \label{dEr3}
\int_{(\mcF, \DMO)}^{(\IN)} : f_*^{(\IN)} (\Omega^{(\IN)} \otimes \mcF^{[\IN]\vee})
\isom \mbR^1 f_*^{(\IN)} (\mcF^{[\IN]})^\vee
\end{align}
for each $\IN =0, \cdots, \M +1$.
Also, the inclusion $\mcF^{[\IN +1]} \migiincl \mcF^{[\IN]}$ ($\IN = 0, \cdots \M$) induces an $\mcO_S$-linear morphism
\begin{align}
\Dual_{\mr{Ker}}^{\, (\IN)} : \mbR^1 f_*^{(\IN)} (\mcF^{[\IN]})^\vee \migi
\mbR^1 f_*^{(\IN +1)} (\mcF^{[\IN+1]})^\vee.
\end{align}

On the other hand, 
we obtain the composite isomorphism
\begin{align} \label{dEE90}
\mr{Ker}(\DMO^{[\IN]\vee}) &\isom  \mcH om (\mr{Coker}(\DMO^{[\IN]}), \Omega^{(\IN+1)}) \\
& \isom \mcH om (\Omega^{(\IN+1)} \otimes\mr{Ker} (\DMO^{[\IN]}), \Omega^{(\IN+1)})  \notag \\
& \isom \mr{Ker}(\DMO^{[\IN]})^\vee \left(= \mcF^{[\IN +1]\vee} \right), \notag
\end{align}
where the first arrow follows from ~\cite[Proposition 6.15]{Wak8} 
 and the second arrow denotes  the isomorphism induced by $C_{(\mcF^{[\IN]}, \DMO^{[\IN]})}$.
Under the identification $\mr{Ker}(\DMO^{[\IN]\vee}) =\mcF^{[\IN+1]\vee}$ given by \eqref{dEE90},
 the Cartier operator $C_{(\mcF^{[\IN]\vee}, \DMO^{[\IN]\vee})}$
determines a morphism
\begin{align} \label{ddE10}
\Dual_{\mr{Coker}}^{\, (\IN)} : f_*^{(\IN)} (\Omega^{(\IN)} \otimes \mcF^{[\IN]\vee}) \migi 
 f_*^{(\IN+1)} (\Omega^{(\IN+1)} \otimes \mcF^{[\IN +1]\vee}).
\end{align}

The following square diagram is verified to be commutative:
\begin{align} \label{dEEr}
\vcenter{\xymatrix@C=46pt@R=36pt{
f_*^{(\IN)} (\Omega^{(\IN)}\otimes \mcF^{[\IN]\vee})
\ar[r]^-{\tiny{\Dual_{\mr{Coker}}^{\, (\IN)}}} \ar[d]^-{\wr}_-{\int_{(\mcF, \DMO)}^{(\IN)}} &
 f_*^{(\IN+1)} (\Omega^{(\IN+1)}\otimes \mcF^{[\IN+1]\vee})
 \ar[d]_-{\wr}^-{\int_{(\mcF, \DMO)}^{(\IN+1)}}
\\
\mbR^1 f_*^{(\IN)} (\mcF^{[\IN]})^\vee
\ar[r]_-{\tiny{\Dual_{\mr{Ker}}^{\, (\IN)}}} &
\mbR^1 f_*^{(\IN+1)} (\mcF^{[\IN+1]})^\vee.
}}
\end{align}
By composing the diagrams \eqref{dEEr} for various $\IN$'s,
we obtain a commutative square diagram
\begin{align} \label{dEEr2}
\vcenter{\xymatrix@C=46pt@R=36pt{
f_* (\Omega \otimes \mcF^\vee)
\ar[r]^-{\tiny{{\Dual}_{\mr{Coker}}}} \ar[d]^-{\wr}_-{\int_{(\mcF, \DMO)}^{(0)}} &
f_*^{(\M +1)} (\Omega^{(\M +1)} \otimes \mcS ol (\DMO)^\vee)
 \ar[d]_-{\wr}^-{\int_{(\mcF, \DMO)}^{(\M +1)}}
\\
\mbR^1 f_{*}(\mcF)^\vee
\ar[r]_-{\tiny{{\Dual}_{\mr{Ker}}}} &
\mbR^1 f_*^{(\M +1)}(\mcS ol (\DMO))^\vee.
}}
\end{align}

\vspace{10mm}
\section{Diagonal reductions/liftings of flat modules}  \label{S200}\SSP

This section discusses the constructive definition  of a dormant flat module in prime-power characteristic, partly on the basis of  the argument in ~\cite[Section 2.1, Chapter II]{Mzk2}.
  At the same time, both the diagonal reduction of a dormant flat bundle and a diagonal lifting of a $p^{\M +1}$-flat $\mcD^{(\M)}$-module are  defined.
Also, we describe a map between  certain  deformation spaces  induced by the operation of taking the diagonal reductions in terms of  Cartier operator (cf. Proposition \ref{Pr459}).

Throughout this section, we fix  a nonnegative integer $\CH$.

\LSP
\subsection{Dormant flat modules} \label{SS010}
Let 
$S^\mr{log}$ be an fs log scheme whose underlying scheme $S$ is   flat over $\mbZ /p^{\CH+1}\mbZ$.
Also, let  $f^\mr{log} : X^\mr{log} \migi S^\mr{log}$
be a log curve over $S^\mr{log}$.
For each integer $\M$ with 
$0 \leq \M\leq \CH$, we will denote by a subscript $\M$ the result of reducing an object over $\mbZ/p^{\CH+1}\mbZ$ modulo $p^{\M+1}$.
In particular, we obtain a log curve $X_\M^\mr{log}/S_\M^\mr{log}$.

Let us fix an integer $\M$ with $0 \leq \M \leq  \CH$. 
For simplicity, we write $\mcD^{(\M)}_0 := \mcD^{(\M)}_{X_0^\mr{log}/S^\mr{log}_0}$, $\mcD^{(0)} := \mcD^{(0)}_{X^\mr{log}/S^\mr{log}}$ (cf. the comment at the end of Section \ref{SS040}), and  $\Omega^{(\M)}_0 := \Omega_{X_0^{(\M)\mr{log}}/S_0^\mr{log}}$.

Denote by $\mr{Diag}'_\M$ the set of pairs 
\begin{align} \label{KK5}
(\msF, \DMO_0^{(\M)})
\end{align}
where
\begin{itemize}
\item
$\msF := (\mcF, \DMO)$  is a flat module 
on $X_\M^\mr{log}/S_\M^\mr{log}$  such that $\mcF$ is relatively torsion-free;
\item
$\DMO_0^{(\M)}$ denotes   a $\mcD_0^{(\M)}$-module structure  on $\mcF_0$ (= the reduction modulo $p$ of $\mcF$).
\end{itemize}
 In what follows, we shall define  a subset  $\mr{Diag}_\M$ of $\mr{Diag}'_\M$  inductively  on $\M$.
 
First, we set $\mr{Diag}_0$ to be the subset of $\mr{Diag}'_0$
consisting of pairs $(\msF, \DMO_0^{(0)})$, where $\msF := (\mcF, \DMO)$,
such that $\DMO_0^{(0)}$ coincides with $\DMO$ and has vanishing $p$-curvature.

Next, suppose that we have defined a subset $\mr{Diag}_{\M -1}$ of $\mr{Diag}'_{\M-1}$ for $\M \geq 1$.
Let us take an element 
$(\msF, \DMO_0^{(\M )})$
of $\mr{Diag}'_\M$, where $\msF := (\mcF, \DMO)$,
such that 
  $(\msF_{\M -1}, \DMO_0^{(\M)\Rightarrow (\M -1)})\in \mr{Diag}_{\M -1}$.
 We shall set 
 \begin{align} \label{dFFe3}
 \mcV_{\msF} := \mcS ol ( \DMO_0^{(\M)\Rightarrow (\M -1)}) \left(= \mcF_0^{[\M]} \right),
 \end{align}
 which may be regarded as   an $\mcO_{X_0^{(\M)}}$-module.
 As we will prove below (cf. Proposition  \ref{L0ddd10}, (i)),
the pair $(\msF, \DMO_0^{(\M)})$ and  $\mcV_{\msF}$
 satisfy the following properties:
\begin{itemize}
\item[]
\begin{itemize}
\item[$(\alpha)_{\M-1}$] : 
$\DMO_0^{(\M)\Rightarrow (\M -1)}$ has vanishing $p^\M$-curvature;
\item[$(\beta)_{\M-1}$] : 
Every local section of $\mcV_\msF \left(\subseteq \mcF_0 \right)$
can be lifted, locally on $X$, to a horizontal section in $\mcF_{\M-1}$ with respect to the $S_{\M -1}^\mr{log}$-connection $\DMO_{\M-1}$.
\end{itemize}
\end{itemize}

 Now, let us choose 
 a local section $v$ of $\mcV_{\msF}$ defined over an open subscheme $U$ of $X$.
By the property $(\beta)_{\M -1}$ described above, 
there exists, after possibly  shrinking  $U$,
a section $\widetilde{v}$ of $\mcF$ which is a lifting of $v$
  and horizontal modulo $p^\M$ with respect to $\nabla$.
Then, $\nabla (\widetilde{v})$ lies in $p^\M \cdot \Omega_\M \otimes \mcF$.
By the flatness of  $X_\M$ over $\mbZ/p^{m+1} \mbZ$,
we can divide it  by $p^\M$ to obtain a local section
 of $\Omega_0\otimes \mcF_0$;  we denote this section  by  $\frac{1}{p^\M} \cdot  \nabla (\widetilde{v})$.
 Because of  the property $(\alpha)_{\M -1}$, the $\mcD_0^{(\M -1)}$-module $(\mcF_0, \DMO_0^{(\M) \Rightarrow (\M -1)})$ associates the  Cartier operator
\begin{align}
C_{(\mcF_0, \DMO_0^{(\M) \Rightarrow (\M -1)})}
 : \Omega_0 \otimes\mcF_0 \migi \Omega^{(\M)}_0 \otimes \mcV_\msF
\end{align}
(cf. \eqref{dE3}).
The image of $\frac{1}{p^\M} \cdot  \nabla (\widetilde{v})$  via 
this morphism
 specifies  a local section $w$ of $\Omega^{(\M)}_0 \otimes \mcV_\msF$.
The section $w$ does not depend on the choice of the lifting $\widetilde{v}$, 
and  the resulting assignment $v \mapsto w$  determines a well-defined  $S_0^\mr{log}$-connection 
\begin{align} \label{J2}
\DMO_{\msF} :  \mcV_\msF \migi \Omega^{(\M)}_0 \otimes \mcV_\msF
\end{align}
on the $\mcO_{X_0^{(\M)}}$-module 
$\mcV_\msF$.
In particular, we obtain a flat module 
\begin{align} \label{J13}
(\mcV_\msF, \DMO_\msF)
\end{align}
on the log curve $X_0^{(\M)\mr{log}}/S_0^\mr{log}$.

Then, we shall define 
\begin{align} \label{ee1}
\mr{Diag}_\M
\end{align}
 to be the set of pairs $(\msF, \DMO_0^{(\M)})$ as above satisfying that $\DMO^{[\M]}_0 = \DMO_{\msF}$
   and $\psi (\DMO_{\msF}) = 0$.
The assignments $(\msF, \DMO_0^{(\M)}) \mapsto (\msF_{\M -1}, \DMO_0^{(\M)\Rightarrow (\M -1)})$  ($\M = 1, 2, \cdots, \CH$) yields a sequence of maps
\begin{align} \label{J11}
\mr{Diag}_\CH \migi \cdots \migi \mr{Diag}_\M \migi \cdots \migi \mr{Diag}_1 \migi \mr{Diag}_0.
\end{align}

The following assertion was used in the inductive construction of $\mr{Diag}_{\M}$ just discussed.

\SSP
\bpr \label{L0ddd10} 
Let  $\M$ be an integer with $0 \leq \M \leq \CH$.
\begin{itemize}
\item[(i)]
Each  element $(\msF, \DMO_0^{(\M)})$ (where $\msF:= (\mcF, \DMO)$)  of $\mr{Diag}_{\M}$ 
 satisfies  the following three properties:
\begin{itemize}
\item[$(\alpha)_\M$] : 
$\DMO_0^{(\M)}$ has vanishing $p^{\M+1}$-curvature;
\item[$(\beta)_\M$] : 
 Every local section of $\mcS ol (\DMO_0^{(\M)}) \left(= \mcF_0^{[\M +1]} \right)$ can be lifted, locally on $X$, to a horizontal section in $\mcF$ with respect to the $S_\M^\mr{log}$-connection $\DMO$;
 \item[$(\gamma)_\M$] : 
 For every $a = 0, \cdots, \M$, we have $(\mcV_{\msF_a}, \nabla_{\msF_a}) = (\mcF_0^{[a]}, (\DMO_0^{(\M)})^{[a]})$, where $(\mcV_{\msF_0}, \nabla_{\msF_0})$ $:= (\mcF_0, \nabla_0)$;
\end{itemize}
\item[(ii)]
Let $(\msF_\circ, \DMO_{\circ, 0}^{(\M)})$  and $(\msF_\bullet, \DMO_{\bullet, 0}^{(\M)})$ (where $\msF_\circ := (\mcF_\circ, \nabla_\circ)$ and $\msF_\bullet := (\mcF_\bullet, \nabla_\bullet)$) be  elements of $\mr{Diag}_{\M}$
 such that $\msF_\circ = \msF_\bullet$.
Then,  the equality  $\DMO_{\circ, 0}^{(\M)} = \DMO_{\bullet, 0}^{(\M)}$ holds.
\end{itemize}
 \epr
\begin{proof}
First, 
we shall prove assertion (i)  inductively on $\M$.
There is nothing to prove for the base step, i.e., the case of $\M =0$.
In order to discuss the induction step, suppose that we have proved the three properties $(\alpha)_{\M'}$, $(\beta)_{\M'}$, and $(\gamma)_{\M'}$ ($0 \leq \M' < \M$) for any element of $\mr{Diag}_{\M '}$, in particular, for $(\msF_{\M'}, \DMO_0^{(\M)\Rightarrow (\M ')})$.
Since the equality 
\begin{align} \label{KK2}
(\mcV_{\msF}, \nabla_{\msF}) = (\mcF_0^{[\M]}, (\DMO_0^{(\M)})^{[\M]})
\end{align}
holds, 
  the  property $(\gamma)_{\M-1}$ for $(\msF_{\M-1}, \DMO_0^{(\M)\Rightarrow (\M -1)})$  implies   that
 $(\gamma)_{\M}$ is satisfied for $(\msF, \DMO_0^{(\M)})$.
Moreover, by  \eqref{KK2},   the assertion $(\alpha)_{\M}$ follows from  Corollary \ref{KK1} and  the assumption $\psi (\DMO_\msF) = 0$.
To consider the property $(\beta)_{\M}$,
let us take a local section $v$ of $\mcS ol (\DMO_0^{(\M)}) \left(\subseteq \mcV_\msF \right)$.
Just as in the discussion at the beginning of this subsection,
we can choose  (locally on $X$) local sections $\widetilde{v}$ and  $\frac{1}{p^\M} \cdot \DMO (\widetilde{v})$ associated to $v$.
The following sequence of equalities holds:
\begin{align}
C_{(\mcF_0, \DMO_0^{(\M)})} \left(\frac{1}{p^\M} \cdot \DMO (\widetilde{v})\right)
&= C_{(\mcF_0^{[\M]}, (\DMO_0^{(\M)})^{[\M]})} \left( C_{(\mcF_0, \DMO_0^{(\M)\Rightarrow (\M -1)})} \left(\frac{1}{p^\M} \cdot \DMO (\widetilde{v})\right)\right)\\
& = C_{(\mcF_0^{[\M]}, (\DMO_0^{(\M)})^{[\M]})}(\nabla_\msF (v) )  \notag \\
& = C_{(\mcV_\msF, \nabla_\msF)}(\nabla_\msF (v) ) \notag \\
&  = 0. \notag
\end{align}
It follows that 
the integer 
\begin{align}
a (\widetilde{v}) := \mr{min} \left\{ a \, \Biggl| \,  -1 \leq a \leq \M, \   C_{(\mcF_0, \DMO_0^{(\M)\Rightarrow (a)})} \left(\frac{1}{p^\M} \cdot \DMO (\widetilde{v})\right)  = 0 \right\}
\end{align}
is well-defined, where  $C_{(\mcF_0, \DMO_0^{(\M)\Rightarrow (-1)})} := \mr{id}$.
Since  $C_{(\mcF_0, \DMO_0^{(\M)\Rightarrow (\M -1)})} \left(\frac{1}{p^\M} \cdot \DMO (\widetilde{v})\right) = \nabla_\msF (v) =0$,
we have  $a (\widetilde{v}) \leq \M -1$.
Now, suppose that $a (\widetilde{v}) \geq 0$.
We shall set $u :=  C_{(\mcF_0, \DMO_0^{(\M)\Rightarrow (a (\widetilde{v})-1)})} \left(\frac{1}{p^\M} \cdot \DMO (\widetilde{v})\right)$, which is nonzero by the definition of $a (\widetilde{v})$.
Observe that
\begin{align}
C_{(\mcF_0^{[a (\widetilde{v})]}, (\nabla_0^{(\M)})^{[a (\widetilde{v}) ]})} (u) = 
 C_{(\mcF_0, \DMO_0^{(\M)\Rightarrow (a (\widetilde{v}))})} \left(\frac{1}{p^\M} \cdot \DMO (\widetilde{v})\right) = 0.
\end{align}
Hence, by  the comment following  ~\cite[Proposition 1.2.4]{Og}, there exists (locally on $X$) a local section $u'$ of $\mcF_0^{[a (\widetilde{v})]}$
  satisfying  $(\DMO_0^{(\M)})^{[a (\widetilde{v})]}(u') = u$.
Let us choose 
a section $\widetilde{u}'$   of $\mcF_{a (\widetilde{v})}$ which is 
a lifting  of $u'$ and    horizontal modulo $p^{a (\widetilde{v})}$
 with respect to $\nabla_{a (\widetilde{v})}$.
(Even if $a (\widetilde{v}) > 0$,  such a section always exists locally on $X$ because of 
 the equality $\mcF_0^{[a (\widetilde{v})]} =  \mcS ol (\DMO_0^{(\M)\Rightarrow (a (\widetilde{v})-1)})$ and the induction hypothesis, i.e., the property $(\beta)_{a (\widetilde{v})-1}$ for $(\msF_{a (\widetilde{v})-1}, \DMO_0^{(\M)\Rightarrow (a (\widetilde{v})-1)})$.) 
By the flatness of $X_\M$ over $\mbZ/p^{\M +1}\mbZ$,
we obtain a well-defined element $p^{\M -a (\widetilde{v})} \cdot \widetilde{u}'$ of $\mcF$.
If we  write $\widetilde{v}' := \widetilde{v} - p^{\M -a (\widetilde{v})} \cdot \widetilde{u}'$, then 
the following sequence of equalities holds:
\begin{align}
&  \ \ \  \  C_{(\mcF_0, \DMO_0^{(\M)\Rightarrow (a (\widetilde{v}) -1)})} \left(\frac{1}{p^\M} \cdot \DMO (\widetilde{v}')\right)  \\
 & = 
  C_{(\mcF_0, \DMO_0^{(\M)\Rightarrow (a (\widetilde{v}) -1)})} \left(\frac{1}{p^\M} \cdot \DMO (\widetilde{v})\right) -
   C_{(\mcF_0, \DMO_0^{(\M)\Rightarrow (a (\widetilde{v}) -1)})} \left(\frac{1}{p^\M} \cdot \DMO (p^{\M -a (\widetilde{v})} \cdot \widetilde{u}')\right) \notag \\
   & = u -  C_{(\mcF_0, \DMO_0^{(\M)\Rightarrow (a (\widetilde{v})-1)})} \left(\frac{1}{p^{a (\widetilde{v})}} \cdot \DMO (\widetilde{u}')\right) \notag \\
   & = u - \nabla_{\msF_{a (\widetilde{v})}}(u')\notag \\
   & = u - (\DMO_0^{(\M)})^{[a (\widetilde{v})]} (u') \notag \\
   & = u - u\notag \\
    & = 0. \notag
\end{align}
This implies $a (\widetilde{v})> a (\widetilde{v}')$.
By repeating the procedure for constructing $\widetilde{v}'$ from $\widetilde{v}$ just discussed, we can find  a lifting $\widetilde{v}^\circledcirc \in \mcF$ of $v$ with $a (\widetilde{v}^\circledcirc) =-1$, i.e., $\nabla (\widetilde{v}^\circledcirc) = 0$. 
This  proves the property  $(\beta)_{\M}$, thus completing the proof of assertion (i).

Moreover,  we again use induction on $\M$ to  prove assertion (ii).
Since the base  step is trivial, it suffices to consider
the induction step.
Suppose that the assertion with  $\M$ replaced by $\M -1$ ($\M \geq 1$) has been proved.
Now, let 
$(\msF_\circ, \DMO_{\circ, 0}^{(\M)})$ and  $(\msF_\bullet, \DMO_{\bullet, 0}^{(\M)})$
(where $\msF_\circ := (\mcF_\circ, \nabla_{\circ})$ and $\msF_\bullet := (\mcF_\bullet, \nabla_\bullet)$)
 be elements of $\mr{Diag}_\M$  with $\msF_\circ = \msF_\bullet$.
By the induction hypothesis,
the equality $\DMO_{\circ, 0}^{(\M)\Rightarrow (\M-1)} = \DMO_{\bullet, 0}^{(\M)\Rightarrow (\M-1)}$ holds via the equality  $\mcF_{\circ, 0} = \mcF_{\bullet, 0}$.
This implies  $\mcV_{\msF_\circ} = \mcV_{\msF_\bullet}$.
Under this equality, 
 the $S^\mr{log}$-connection  $\nabla_{\msF_\circ}$ coincides with  $\nabla_{\msF_\bullet}$ because of  the assumption $\msF_\circ = \msF_\bullet$.
 Hence,  it  follows from Lemma \ref{NN1} proved below that $\DMO_{\circ, 0}^{(\M)} = \DMO_{\bullet, 0}^{(\M)}$.
 We have finished   the proof of this proposition.
\end{proof}
\SSP

The following lemma was applied in the proof of the above proposition.

\SSP
\ble \label{NN1}
Suppose that 
$S^\mr{log}$ is an fs log scheme over $\mbF_p$ and that
$X^\mr{log}$ is  a log curve over $S^\mr{log}$.
Let $\M$ be a positive integer, and 
 let $(\mcF_{\circ}, \DMO_{\circ}^{(\M)})$, 
$(\mcF_{\bullet}, \DMO_{\bullet}^{(\M)})$ be $p^{\M +1}$-flat $\mcD^{(\M)}_0$-modules.
We shall assume the following conditions:
\begin{itemize}
\item[(a)]
The equality  $(\mcF_{\circ}, \DMO_{\circ}^{(\M)\Rightarrow (\M -1)}) = (\mcF_{\bullet}, \DMO_{\bullet}^{(\M)\Rightarrow (\M -1)})$ holds;
\item[(b)]   
The $S^\mr{log}$-connection  $(\DMO_{\circ}^{(\M)})^{[\M]}$ corresponds to  $(\DMO_{\bullet}^{(\M)})^{[\M]}$ via the equality $\mcF_{\circ}^{[\M]} = \mcF_{\bullet}^{[\M]}$ resulting from  the condition (a);
\item[(c)]
$\mcF_\circ \left(= \mcF_{\bullet} \right)$ is relatively torsion-free.
\end{itemize}
Then, we have $\DMO_{\circ}^{(\M)} = \DMO_{\bullet}^{(\M)}$.
\ele
\begin{proof}
We shall set $\mcW := \mr{Ker}((\DMO_{\circ}^{(\M)})^{[\M]}) = \mr{Ker}((\DMO_{\bullet}^{(\M)})^{[\M]})$.
Denote by $U$ the open subscheme of $X_0$ where the relative characteristic  of $X^\mr{log}/S^\mr{log}$ is trivial.
The $\mcO_{X^{(\M +1)}}$-linear inclusion $\left(\mcW = \right) \mr{Ker}((\DMO_{\circ}^{(\M)})^{[\M]}) \migiincl \mcF_{\circ}$ (resp., $\left(\mcW = \right) \mr{Ker}((\DMO_{\bullet}^{(\M)})^{[\M]})  \migiincl \mcF_{\bullet}$) extends to a morphism 
 of $\mcD^{(\M)}$-modules
\begin{align} \label{dE1}
(F_{X/S}^{(\M +1)*}(\mcW), \nabla_{\mcW, \mr{can}}^{(\M)}) \migi (\mcF_{\circ}, \DMO_{\circ}^{(\M)}) \ \left(\text{resp.,} \ (F_{X/S}^{(\M +1)*}(\mcW), \nabla_{\mcW, \mr{can}}^{(\M)}) \migi (\mcF_{\bullet}, \DMO_{\bullet}^{(\M)}) \right).
\end{align}
By the equivalence of categories \eqref{YY6},
this morphism becomes an isomorphism when restricted to  $U$.
This means that
 $\DMO_{\circ}^{(\M)}$ coincides with $\DMO_{\bullet}^{(\M)}$ over   $U$.
Since $U$ is a scheme-theoretically dense in  $X_0$ (cf. ~\cite[Lemma 1.4]{KaFu}) and $\mcF_{\circ}\left(=\mcF_{\bullet}\right)$ is relatively torsion-free by assumption,
we obtain the equality  $\DMO_{\circ}^{(\M)} = \DMO_{\bullet}^{(\M)}$, as desired.
\end{proof}
\SSP

\bco \label{YY79}
Let $(\msF, \DMO_0^{(\M)})$ (where $\msF := (\mcF, \nabla)$) be an element of $\mr{Diag}_\M$.
Then, each horizontal local section of $\mcF$ specifies a section of $\mcS ol (\DMO_0^{(\M)})$ via reduction  modulo $p$.
Moreover, the resulting morphism of sheaves
\begin{align} \label{YY82}
\left(\mr{Ker}(\nabla)_0 =\right)\mcS ol (\nabla)_0 \migi \mcS ol (\DMO_0^{(\M)})
\end{align}
is surjective.
\eco
\begin{proof}
We shall prove the first assertion by induction on $\M$.
The base step, i.e., the case of $\M = 0$, is trivial because
$\mcS ol (\DMO) = \mcS ol (\DMO_0^{(\M)})$.
To consider the induction step, suppose that we have proved the assertion with $\M$ replaced by $\M -1$ ($\M \geq 1$).
Let us take an arbitrary local section $v$ in $\mcS ol (\DMO)$.
By the induction hypothesis, the reduction modulo $p$ of $v$, i.e., $v_0 \left(= (v_{\M-1})_0\right)$, belongs to $\mcV_\msF \left(= \mcS ol (\DMO_0^{(\M)\Rightarrow (\M -1)})\right)$.
Since the reduction modulo $p^\M$ of $v$ is horizontal with respect to $\nabla_{\M-1}$,
it follows from the definition of $\nabla_\msF$ that
$\nabla_\msF (v_0) = C_{(\mcF_0, \nabla_0^{(\M)\Rightarrow (\M -1)})} \left( \frac{1}{p^\M} \cdot \nabla (v) \right) = 0$.
Thus, we have $v_0 \in \mcS ol (\nabla_\msF) = \mcS ol (\DMO_0^{(\M)})$.
This completes the proof of the first assertion.
The second assertion, i.e., the surjectivity of \eqref{YY82}, follows from the property $(\beta)_\M$ asserted in  Proposition \ref{L0ddd10}, (i).
\end{proof}

\SSP
\bde \label{D019} 
\begin{itemize}
\item[(i)]
Let $\msF := (\mcF, \DMO)$ be a flat module  on $X^\mr{log}_\M/S^\mr{log}_\M$.
  We shall say  that $\msF$ (or $\DMO$) is {\bf dormant}
  if there exists a $\mcD_0^{(\M)}$-module structure  
  $\DMO_0^{(\M)}$ on $\mcF_0$
    satisfying  $(\msF, \DMO_0^{(\M)}) \in \mr{Diag}_{\M}$.
  \item[(ii)]
   Let  $(\mcE, \DMO)$ be a flat $\mr{PGL}_n$-bundle (where $n >1$) on $X^\mr{log}_\M/S^\mr{log}_\M$ in the sense of ~\cite[Definition 1.28]{Wak8}.
  We shall say that $(\mcE, \DMO)$ is {\bf dormant} if, after possibly replacing $X$ with its \'{e}tale covering,  it may be described as  the projectivization of a rank $n$ dormant flat bundle on $X^\mr{log}_\M/S^\mr{log}_\M$. 
  \end{itemize}
    \ede
\SSP

\begin{rema}[Some properties on dormant flat bundles] \label{erd2}
Whether a flat module  $(\mcF, \DMO)$ on $X^\mr{log}_\M/S^\mr{log}_\M$ with  $\mcF$ relatively torsion-free (or a flat $\mr{PGL}_n$-bundle $(\mcE, \DMO)$)  is dormant or not 
      does not depend on the integer  ``$\CH$".
 Also,  the property of being dormant  is of local nature with respect to the \'{e}tale topology on $X$.
  In the case of $\M =0$, a flat module  (or a flat $\mr{PGL}_n$-bundle) on $X_0^\mr{log}/S_0^\mr{log}$ is dormant if and only if it has vanishing $p$-curvature.
\end{rema}
\SSP

\begin{rema}[Well-definedness of $(\mcV_{\msF}, \nabla_{\msF})$] \label{erd}
Let $\msF := (\mcF, \nabla)$ be a flat module  on $X_\M^\mr{log}/S^\mr{log}_{\M}$ ($\M \geq 1$) such that $\mcF$ is relatively torsion-free and  its  reduction modulo $p^{\M}$ is dormant.
Then, by Proposition \ref{L0ddd10}, (ii),
 the flat module  $(\mcV_{\msF}, \nabla_{\msF})$ is well-defined in the sense that it  does not depend on the choice of a $\mcD_0^{(\M -1)}$-module structure  $\DMO_0^{(\M -1)}$ with $(\msF_{\M -1}, \DMO_0^{(\M -1)}) \in \mr{Diag}_{\M -1}$ (because such a choice $\DMO_0^{(\M -1)}$ is, if it exists, uniquely determined).
\end{rema}

\LSP
\subsection{Diagonal reductions/liftings} \label{SS0DDD10}

Let $\msF := (\mcF, \nabla)$ be a dormant flat module on $X^\mr{log}_\M/S^\mr{log}_\M$.
According to Proposition \ref{L0ddd10}, (ii),
there exists a unique $\mcD_0^{(\M)}$-module structure
 \begin{align} \label{YY223}
 {^{\Diag}}\!\!\DMO : {^L}\mcD_0^{(\M)} \migi \mcE nd_{\mcO_{S_0}} (\mcF_0)
 \end{align}
  on $\mcF_0$ with $(\msF, {^{\Diag}}\!\!\DMO) \in \mr{Diag}_\M$.
In particular, we obtain
a $\mcD_0^{(\M)}$-module
 \begin{align} \label{ee3}
{^{\Diag}}\!\!\msF := (\mcF_0, {^{\Diag}}\!\!\DMO),
 \end{align}
which  has 
 vanishing $p^{\M +1}$-curvature (cf.  Corollary \ref{KK1}).

Also, let $\msE := (\mcE, \nabla)$ be a dormant  flat $\mr{PGL}_n$-bundle on $X_\M^\mr{log}/S_\M^\mr{log}$.
Then, by choosing  \'{e}tale locally on $X$ a dormant flat bundle $\msF := (\mcF, \nabla)$ inducing $\msE$ and successively taking the projectivization of ${^{\Diag}}\!\!\nabla$, we obtain a well-defined  $\M$-PD stratification
\begin{align} \label{eeYY3}
{^{\Diag}}\!\!\STR :=  \{ {^{\Diag}}\!\!\STR_{\ell} : P_{(\M)}^\ell \times_{X_0} \mcE_0 \isom \mcE_0 \times_{X_0} P_{(\M)}^{\ell} \}_{\ell \in \mbZ_{\geq 0}}
 \end{align}
 on $\mcE_0$.
 We shall set
 \begin{align} \label{YY241}
 {^{\Diag}}\!\!\msE :=  (\mcE_0, {^{\Diag}}\!\!\STR),
 \end{align}
 which forms an $\M$-PD stratified $\mr{PGL}_n$-bundle on $X^\mr{log}_0/S_0^\mr{log}$.

\SSP
\bde \label{D010} 
\begin{itemize}
\item[(i)]
With the above notation, 
we shall refer to  ${^{\Diag}}\!\!\msF$ (resp., ${^{\Diag}}\!\!\msE$) as   the {\bf diagonal reduction} of $\msF$ (resp., $\msE$).
\item[(ii)]
Given a $\mcD_0^{(\M)}$-module  (resp., an $\M$-PD stratified $\mr{PGL}_n$-bundle) $\msG$,
we shall refer to any dormant flat module $\msF$  (resp.,  any dormant flat $\mr{PGL}_n$-bundle $\msE$) on $X_\M^\mr{log}/S_\M^\mr{log}$ satisfying  ${^{\Diag}}\!\!\msF = \msG$ (resp., ${^{\Diag}}\!\!\msE = \msG$) as 
  a {\bf diagonal lifting} of $\msG$.
  \end{itemize}
 \ede
\SSP

\begin{rema}[Compatibility with level reduction] \label{YY250}
Let $\msF$ (resp., $\msE$) be as above.
If $\M'$ is a nonnegative integer $\leq \M$, then
the reduction $\msF_{\M'}$  (resp., $\msE_{\M'}$) modulo $p^{\M' +1}$ of $\msF$  (resp., $\msE$) is dormant 
 and  
 the diagonal reduction
   ${^{\Diag}}\!\!(\msF_{\M'})$ (resp., 
   ${^{\Diag}}\!\!(\msE_{\M'})$) is obtained   by  reducing   the level of ${^{\Diag}}\!\!\msF$  (resp., ${^{\Diag}}\!\!\msE$)  to $\M'$.
\end{rema}
\SSP

\begin{rema}[Diagonal reduction of non-dormant flat modules] \label{YY456}
Suppose that the relative characteristic of $X^\mr{log}/S^\mr{log}$ is trivial.
Let $\msF := (\mcF, \nabla)$ be a flat bundle on $X_\M^\mr{log}/S_\M^\mr{log}$ ($\M \geq 1$)
such that $\msF_{\M -1}$ is dormant.  
Write ${^{\Diag}}\!\!\msF_{\M -1} = (\mcF_0, \nabla_0^{(\M -1)})$.
Just as in the discussion at the beginning of the previous subsection,
$(\msF, \nabla_0^{(\M -1)})$ associates a flat bundle $(\mcV_\msF, \nabla_\msF)$ (cf. \eqref{J13}).
According to the equivalence of categories \eqref{YY6} and ~\cite[Proposition 2.2.3, (i)]{PBer2} (or, ~\cite[Corollaire 3.3.1]{Mon}),
 $\nabla_\msF$  induces a $\mcD^{(\M)}_0$-module structure $\nabla_0^{(\M)}$ on 
$F^{(\M)*}_{X_0/S_0} (\mcV_\msF)$ compatible with $\nabla_0^{(\M-1)}$ via the isomorphism  $F^{(\M)*}_{X_0/S_0} (\mcV_\msF) \isom \mcF_0$ induced by the inclusion $\mcV_\msF \migiincl \mcF_0$.
Thus, 
 $\nabla_0^{(\M)}$ may be regarded  as a $\mcD^{(\M)}_0$-module structure on $\mcF_0$ via this isomorphism and satisfies $\nabla^{(\M)\Rightarrow (\M -1)}_0 = \nabla_0^{(\M -1)}$.

Similarly to 
 the case of dormant flat modules, we set
\begin{align} \label{eeQQ3101}
{^{\Diag}}\!\!\msF := (\mcF_0, \nabla_0^{(\M)}),
\end{align}
which we refer to as the {\bf diagonal reduction} of $\msF$.
If $\msF$ is dormant, then this coincides with the diagonal reduction in the sense of Definition \ref{D010}. 
\end{rema}
\SSP

\begin{exa}[Trivial flat bundle] \label{E14}
One may verify that the diagonal reduction of the trivial flat bundle $(\mcO_{X_\M}, \DMO^{(0)}_{X_\M, \mr{triv}})$ (cf. \eqref{dGG6}) is the trivial $\mcD_0^{(\M)}$-module $(\mcO_{X_0}, \DMO^{(\M)}_{X_0, \mr{triv}})$, i.e., 
\begin{align} \label{J10}
{^{\Diag}}\!\!(\mcO_{X_\M}, \DMO^{(0)}_{X_\M, \mr{triv}}) = (\mcO_{X_0}, \DMO^{(\M)}_{X_0, \mr{triv}}).
\end{align}
\end{exa}
\SSP

\bpr \label{P016dd} 
Suppose that the relative characteristic of $X^\mr{log}/S^\mr{log}$ is trivial.
\begin{itemize}
\item[(i)]
Let $\msF := (\mcF, \nabla)$ be a flat module on $X_\M^\mr{log}/S_\M^\mr{log}$ ($\M \geq 1$) such that $\mcF$ is relatively torsion-free and its  reduction modulo $p^\M$ (i.e., $\msF_{\M -1}$) is dormant.
Then, 
$\msF$ is dormant if and only if $\psi (\nabla_\msF) =0$ (cf. Remark \ref{erd}).
In particular, if $\mcF$ is assumed to be a vector bundle and we fix a scheme-theoretically dense open subscheme $U$ of  $X_\M$ (equipped with a log structure pulled-back from $X^\mr{log}$, by which we obtain a log curve $U^\mr{log}/S^\mr{log}_\M$),
then $\msF$ is dormant if and only if its restriction $\msF |_U$ to $U$ is dormant.
\item[(ii)]
Let  $n$ be a positive integer and $\msF := (\mcF, \DMO)$  a rank $n$ flat bundle on $X_\M^\mr{log}/S_\M^\mr{log}$.
Note that, since $\DMO$ is $\mcS ol (\DMO_{X_\M, \mr{triv}}^{(0)})$-linear,
the sheaf of horizontal sections $\mr{Ker}(\DMO)$ forms an $\mcS ol (\DMO_{X_\M, \mr{triv}}^{(0)})$-submodule of $\mcF$.
Then, the following three conditions are equivalent to each other:
\begin{itemize}
\item[(a)]
$\msF$ is dormant;
\item[(b)]
$\msF$ is locally trivial, i.e., isomorphic, Zariski locally on $X$, to $(\mcO_{X_\M}, \DMO_{X_\M, \mr{triv}}^{(0)})^{\oplus n}$;
\item[(c)]
 The $\mcS ol (\DMO_{X_\M, \mr{triv}}^{(0)})$-module $\mr{Ker}(\DMO)$ is locally free of rank $n$ (which implies that the natural morphism $\mcO_{X_\M} \otimes_{\mcS ol (\DMO_{X_\M, \mr{triv}}^{(0)})} \mr{Ker}(\nabla) \rightarrow \mcF$ is an isomorphism).
\end{itemize}
\end{itemize}
 \epr
\begin{proof}
First, we shall prove the first assertion of  (i).
Since the ``only if" part of the required equivalence is immediate from Proposition \ref{UU577}, (i), we only consider the ``if" part.
Suppose that $\psi (\nabla_{\msF}) =0$.
Then, 
the morphism 
\begin{align} \label{YY2}
F_{X_0^{(\M)}/S_0}^* (\mcS ol (\nabla_\msF)) \migi \mcV_\msF
\end{align}
 induced naturally by the inclusion $\mcS ol (\nabla_\msF) \migiincl \mcV_\msF$  is an isomorphism (cf. ~\cite[Theorem (5.1)]{Kal}).
On the other hand,
if we set   $(\mcF_0, \DMO_0^{(\M-1)}) := {^{\Diag}}\!\!\msF_{\M-1}$, 
then  it follows from
the equivalence of categories \eqref{YY6}
 that 
the morphism 
\begin{align} \label{YY1}
\left(  F_{X_0/S_0}^{(\M)*}(\mcS ol (\DMO_0^{(\M-1)})) = \right) F_{X_0/S_0}^{(\M)*}(\mcV_\msF) \migi \mcF_0
\end{align}
 induced by 
the inclusion $\mcS ol (\DMO_0^{(\M-1)}) \migiincl  \mcF_0$ is an isomorphism.
Hence, we obtain the composite isomorphism
\begin{align}
F^{(\M +1)*}_{X_0/S_0}(\mcS ol(\nabla_{\msF})) \left( =  
F^{(\M)*}_{X_0/S_0} (F^*_{X_0^{(\M)}/S_0}(\mcS ol(\nabla_{\msF})))
\right)
\isom F^{(\M)*}_{X_0/S_0} (\mcV_\msF) \xrightarrow{\eqref{YY1}} \mcF_0,
\end{align}
where the first arrow denotes the pull-back of \eqref{YY2} by $F_{X_0/S_0}^{(\M)}$.
Denote by $\DMO_0^{(\M)}$ the $\mcD_0^{(\M)}$-module structure  on $\mcF_0$ corresponding to $\DMO_{\mcS ol(\nabla_\msF), \mr{can}}^{(\M)}$ (cf. Definition \ref{dGGe11}) via this isomorphism.
One may verify that $(\msF, \DMO_0^{(\M)})$ belongs to $\mr{Diag}_\M$,  meaning that  $\msF$ is dormant.
This proves the ``if" part of the required equivalence.

The second assertion of (i)  follows from the first assertion together with the fact that the $p$-curvature of a flat bundle in characteristic $p$ can be identified with a global section of  a certain associated vector bundle.

Next, we shall consider  assertion (ii).
The implication (b) $\Rightarrow$ (c)  is clear.
Also, since the dormancy condition is closed under taking direct sums,
 the 
 implication (b) $\Rightarrow$ (a)
 follows  from the fact that the trivial flat bundle is dormant (cf. Example \ref{E14}).

Let us  prove  the implication (a) $\Rightarrow$ (b).
Suppose that $\msF$ is dormant, i.e., 
there exists a $\mcD_0^{(\M)}$-module structure   $\DMO_0^{(\M)}$ on $\mcF_0$ with $(\msF, \DMO_0^{(\M)}) \in \mr{Diag}_{\M}$.
By the equivalence of categories \eqref{YY6},
the morphism $F^{(\M+1)*}_{X_0/S_0}(\mcS ol (\DMO_0^{(\M)})) \migi \mcF_0$
extending the inclusion $\mcS ol (\DMO_0^{(\M)}) \migiincl \mcF_0$ is an isomorphism.
Hence, the faithful flatness of  $F^{(\M+1)}_{X_0/S_0}$ (cf. ~\cite[Theorem 1.1]{KaFu})
implies that 
$\mcS ol (\DMO_0^{(\M)})$ forms   a rank $n$ vector bundle on $X_0^{(\M+1)}$.
By this fact and  the property $(\beta)_\M$ for $(\msF, \DMO_0^{(\M)})$ (cf. Proposition \ref{L0ddd10}, (i)),
we can  find, for each point $q$ of $X_\M$, a collection of data
\begin{align}
(U; v_1, \cdots, v_n),
\end{align} 
where  $U$  denotes an open neighborhood of $q$ in $X_\M$ and 
$v_1, \cdots, v_n$ are  sections  of $\mcF$  defined on $U$ which are horizontal  with respect to  $\DMO$ and whose reductions $v_{1, 0}, \cdots, v_{n, 0}$  modulo $p$ form a local  basis of $\mcS ol (\DMO_0^{(\M)})$.
The   morphism 
$\mcO_{U}^{\oplus n} \migi \mcF |_U$
 given by $(a_i)_{i=1}^n \mapsto \sum_{i=1}^n a_i \cdot v_i$ (for any $a_1, \cdots, a_n \in \mcO_{U}$) 
 is compatible with the respective 
  $S^\mr{log}$-connections $(\DMO_{U, \mr{triv}}^{(0)})^{\oplus n}$ and  $\nabla$.
 The resulting  morphism  of 
  flat modules
\begin{align} \label{ee9}
(\mcO_{U}, \DMO_{U, \mr{triv}}^{(0)})^{\oplus n}\migi   \msF |_U
\end{align}
is an isomorphism by Nakayama's lemma.
In particular,  $\msF$ is locally trivial, and 
this completes the proof of  (a) $\Rightarrow$ (b).

Finally, we shall prove  (c) $\Rightarrow$ (b).
To this end, we may assume that there exists an $\mcS ol (\DMO_{X_\M, \mr{triv}}^{(0)})$-linear isomorphism $h^\nabla : \mcS ol (\DMO_{X_\M, \mr{triv}}^{(0)})^{\oplus n} \isom \mr{Ker}(\DMO)$. 
For each $j =1, \cdots, n$, denote by $e_j$ the image of $1$ via the inclusion into the $j$-th factor $\mcS ol (\DMO_{X_\M, \mr{triv}}^{(0)}) \migiincl \mcS ol (\DMO_{X_\M, \mr{triv}}^{(0)})^{\oplus n}$.
Then, we obtain the morphism  of flat bundles
\begin{align} \label{TTT10}
h : (\mcO_{X_\M}, \DMO_{X_\M, \mr{triv}}^{(0)})^{\oplus n} \migi \msF
\end{align}
 given by $(a_j)_{j=1}^n \mapsto \sum_{j=1}^n a_j \cdot h^\nabla (e_j)$.
On the other hand, $h^\nabla$ induces 
an isomorphism 
\begin{align} \label{EEr34}
\mcS ol (\DMO_{X_\M, \mr{triv}}^{(0)}) [p]^{\oplus n} \isom \mr{Ker}(\DMO)[p],
\end{align}
 where for each sheaf $\mcH$ we denote by $\mcH [p]$ the kernel of the multiplication-by-$p$ map $\mcH \rightarrow \mcH$.
The assignment $v \mapsto p^\M v$ determines isomorphisms $\mcS ol (\DMO_{X_0, \mr{triv}}^{(0)}) \xrightarrow{\sim} \mcS ol (\DMO_{X_\M, \mr{triv}}^{(0)}) [p]$ and $\mr{Ker}(\DMO_0) \xrightarrow{\sim} \mr{Ker}(\DMO)[p]$.
Using these isomorphisms, we can describe \eqref{EEr34} as 
$\breve{h}_0^\nabla : \mcS ol (\DMO_{X_0, \mr{triv}}^{(0)})^{\oplus n}$ $\xrightarrow{\sim}\mr{Ker}(\DMO_0)$.
By  ~\cite[Proposition 4.60]{Wak8},  $\nabla_0$ has vanishing $p$-curvature, and 
the morphism 
 \begin{align}
 \breve{h}_0 : (\mcO_{X_0}, \DMO_{X_0, \mr{triv}}^{(0)})^{\oplus n} \migi \msF_0.
 \end{align}
 given by $(\overline{a}_j)_{j=1}^n \mapsto \sum_{j=1}^n \overline{a}_j \cdot \breve{h}_0^\nabla (\overline{e}_j)$ is an isomorphism.
 Since 
  the reduction of $h$ modulo $p$ coincides with $\breve{h}_0$,
 it follows from Nakayama's lemma  that 
  $h$ turns out to be an isomorphism.
 This proves  (c) $\Rightarrow$ (b).
\end{proof}
\SSP

\bco \label{YY102}
Suppose that the relative characteristic of $X^\mr{log}/S^\mr{log}$ is trivial.
Let $(\mcE, \DMO)$ be a flat   $\mr{PGL}_n$-bundle  on $X^\mr{log}_\M/S^\mr{log}_\M$.
Then, $(\mcE, \DMO)$ is dormant if and only if  it is, \'{e}tale locally  on $X$,  isomorphic to the trivial flat   $\mr{PGL}_n$-bundle (cf. ~\cite[Section 1.3.2]{Wak8} or Example \ref{NN601}).
\eco
\begin{proof}
The assertion follows from the equivalence (a) $\Leftrightarrow$ (b) obtained  in Proposition \ref{P016dd}, (ii).
\end{proof}
\SSP

Let   $h : \msF \migi \msG$ be a morphism between  dormant  flat modules on $X_\M^\mr{log}/S_\M^\mr{log}$, and write $(\mcF_0, \DMO_{\msF, 0}) :=  {^{\Diag}}\!\!\msF$ and $(\mcG_0, \DMO_{\msG, 0}) :=  {^{\Diag}}\!\!\msG$.
By the functoriality of Cartier operator (cf. Remark \ref{Rem4451}),
we see that $h$ induces a morphism of flat bundles $h_a : (\mcV_{\msF_a}, \nabla_{\msF_a}) \migi (\mcV_{\msG_a}, \nabla_{\msG_a})$ for every $a = 0, \cdots, \M$.
It follows from the property $(\gamma_\M)$ asserted in Proposition \ref{L0ddd10}, (i),  that $h_a$ may be regarded as a morphism
$(\mcF_0^{[a]}, \DMO_{\msF, 0}^{[a]}) \migi (\mcG_0^{[a]}, \DMO_{\msG, 0}^{[a]})$.
Since the  $h_a$'s are compatible with each other via restriction,
the morphism $h_0 : \mcF_0 \migi \mcG_0$ preserves the $\mcD_0^{(\M)}$-module structure.
In other words, it determines a morphism    of $\mcD_0^{(\M)}$-modules
  \begin{align} \label{J22}
  {^{\Diag}}\!\!h : {^{\Diag}}\!\!\msF \migi {^{\Diag}}\!\!\msG.
  \end{align}
  The assignments $\msF \mapsto {^{\Diag}}\!\!\msF$ and  $h \mapsto {^{\Diag}}\!\!h$ define a functor
 \begin{align} \label{ee18}
 {^{\Diag}}\!\!(-) : 
 \left(\begin{matrix} \text{the category of dormant} \\  \text{flat modules
 on $X_\M^\mr{log}/S_\M^\mr{log}$}\end{matrix} \right)
 \rightarrow \left(\begin{matrix} \text{the category of dormant} \\ \text{ (i.e., $p^{\M +1}$-flat) $\mcD_0^{(\M)}$-modules}\end{matrix} \right).
  \end{align}
This functor commutes  with the formations of direct sums, tensor products,  and determinants (in the case where the underlying $\mcO_{X_\M}$-modules of dormant flat modules  are locally free of finite rank).

\LSP
\subsection{Change of structure group} \label{S21Y}

We shall prove  the following two assertions, which are generalizations of 
~\cite[Proposition 4.22]{Wak8}.
Let $\M$ be a nonnegative integer.

\SSP
\bpr \label{P129}
Let $n$ be a positive integer with
 $p \nmid n$ and 
 $\mcL$  a line bundle on $X$.
 Suppose that
either $\CH = 0$ or $\M =0$ is satisfied.
Also, suppose that we are given 
a $\mcD^{(\M)}$-module structure  $\DMO_{\mcL^{\otimes n}}$
 on the $n$-fold tensor product $\mcL^{\otimes n}$ of $\mcL$.
\begin{itemize}
\item[(i)]
There exists a unique 
$\mcD^{(\M)}$-module structure  $\DMO_{\mcL}$
 on $\mcL$ satisfying $\DMO^{\otimes n}_\mcL = \DMO_{\mcL^{\otimes n}}$ (cf. ~\cite[Corollaire 2.6.1]{Mon} for the definition of the $n$-th fold tensor product of a $\mcD^{(\M)}$-module structure).
\item[(ii)]
 If, moreover, $(\mcL^{\otimes n}, \DMO_{\mcL^{\otimes n}})$ 
is dormant, then 
  the resulting $\mcD^{(\M)}$-module  $(\mcL, \DMO_\mcL)$ is dormant.
\end{itemize}
\epr
\begin{proof}
First, we shall consider assertion (i).
Denote by
 $\STR_{\mcL^{\otimes n}}$ the $\M$-PD stratification on $\mcL^{\otimes n}$ corresponding to $\DMO_{\mcL^{\otimes n}}$ (cf. Remark \ref{dGGg5}).
Let us take an open covering $\msU := \{ \mcU_\alpha \}_{\alpha \in I}$ of $X$ such that, for each $\alpha \in I$, there exists a trivialization $\tau_\alpha : \mcL |_{U_\alpha} \isom \mcO_{U_\alpha}$.
For each $\ell \in \mbZ_{\geq 0}$ and $d \in \mbZ_{> 0}$,
$\tau_\alpha$ induces  isomorphisms 
\begin{align}
{^L}\tau_{\alpha, \ell, d} : \mcP^\ell_{(\M)} \otimes \mcL^{\otimes d} |_{U_\alpha} \isom \mcO_{P^\ell_{(\M)}} |_{U_\alpha} \ \ \text{and} \ \   {^R}\tau_{\alpha, \ell, d} : \mcL^{\otimes d} |_{U_\alpha} \otimes \mcP^\ell_{(\M)} \isom \mcO_{P^\ell_{(\M)}} |_{U_\alpha}.
\end{align}
The isomorphisms  ${^L}\tau_{\alpha, \ell, n}$ and ${^R}\tau_{\alpha, \ell, n}$ allow us to regard 
$\phi_{\mcL^{\otimes n}} |_{U_\alpha}$  as an automorphism of $\mcO_{P^\ell_{(\M)}} |_{U_\alpha}$; it  may be described as the multiplication by some  element $\eta_{\alpha, \ell}$ of $\mcO_{P^\ell_{(\M)} |_{U_\alpha}}^\times$.
  Since $\eta_{\alpha, 0} = 1$ and $p \nmid n$, 
   we can find  a unique compatible collection $\{ \lambda_{\alpha, \ell}\}_{\ell \in \mbZ_{\geq 0}}$, where each $\lambda_{\alpha, \ell}$ is an element of $\mcO_{P^\ell_{(\M)} |_{U_\alpha}}^\times$  with $\lambda_{\alpha, \ell}^{\otimes n} = \eta_{\alpha, \ell}$ and $\lambda_{\alpha, 0} =1$.
 By  ${^L}\tau_{\alpha, \ell, 1}$ and ${^R}\tau_{\alpha, \ell, 1}$,
the automorphisms given by
multiplication by $ \lambda_{\alpha, \ell}$ for various $\ell$'s together 
 determine an $\M$-PD stratification $\phi_{\mcL, \alpha}$ on $\mcL |_{U_\alpha}$. 
 The uniqueness of the collection $\{ \lambda_{\alpha, \ell}\}_{\ell}$ implies $\STR_{\mcL, \alpha} |_{U_\alpha \cap U_\beta} = \STR_{\mcL, \beta} |_{U_\alpha \cap U_\beta}$  for any pair  $(\alpha, \beta) \in I \times I$ with $U_\alpha \cap U_\beta \neq \emptyset$.
 Hence, $\STR_{\mcL, \alpha}$ may be glued together to obtain an $\M$-PD stratification $\STR_\mcL$ on $\mcL$.
 Let  $\DMO_\mcL$ denote the $\mcD^{(\M)}$-module structure on $\mcL$ corresponding to $\STR_\mcL$.
 Then, it follows from the definitions of $\STR_{\mcL, \alpha}$'s that
 the equality $\DMO_{\mcL}^{\otimes n} = \DMO_{\mcL^{\otimes n}}$ holds, and that an $\M$-PD stratification satisfying this equality is uniquely determined.
 This completes the proof of assertion (i).
 
 Next, we shall consider assertion (ii).
  The case of $\CH =0$ follows directly from the equality \eqref{EEE499}.
 
Finally, let us  prove the case of $\M = 0$ by induction on $\CH$.
The base step, i.e., the case of $\CH =0$, was already proved.
To consider the induction step, we suppose that the assertion with $\CH$ replaced by $\CH -1$ ($\CH \geq 1$) has been proved.
Denote by  $\DMO_{\mcL^{\otimes n}, 0}^{(\CH)}$ the $\mcD_0^{(\CH)}$-module structure on  $\mcL_0^{\otimes n}$ obtained as the diagonal reduction of 
$(\mcL^{\otimes n}, \nabla_{\mcL^{\otimes n}})$.
By applying assertion (i) to $\DMO_{\mcL^{\otimes n}, 0}^{(\CH)}$,
we obtain uniquely a $\mcD_0^{(\CH)}$-module structure $\DMO_{\mcL, 0}^{(\CH)}$ on $\mcL_0$ with  $(\DMO_{\mcL, 0}^{(\CH)})^{\otimes n} = \DMO_{\mcL^{\otimes n}, 0}^{(\CH)}$.
By the induction hypothesis, $\nabla_{\mcL, \CH -1}$ is dormant and the uniqueness of $\DMO_{\mcL, 0}^{(\CH)}$ implies ${^{\Diag}}\!\!\msL_{\CH -1}
= (\mcL_0, \DMO_{\mcL, 0}^{(\CH) \Rightarrow (\CH -1)})$, where $\msL := (\mcL, \nabla_\mcL)$. 
If $U$ is
the open subscheme of  $X$ where the relative characteristic of $X^\mr{log}/S^\mr{log}$ is trivial, 
then
the restriction $\mcV_\msL |_{U_0^{(\CH)}}$ of $\mcV_\msL$ is a line bundle on $U^{(\CH)}_0$.
The diagonal reduction of $\msL^{\otimes n}$ is $(\DMO_{\mcL, 0}^{(\CH)})^{\otimes n}$, so
the uniqueness portion of assertion (i) (together with  the equivalence of categories \eqref{YY6})
 implies  that 
 $\nabla_\msL$ and $(\nabla_{\mcL, 0}^{(\CH)})^{[\CH]}$ are equal when restricted to
 $U^{(\CH)}_0$.
But,   since $U^{(\CH)}_0$ is scheme-theoretically dense in $X^{(\CH)}_0$ and  $\mcV_\msL$ is relatively torsion-free  (cf. Proposition \ref{PPer4}, (i)),  the equality $\nabla_\msL = (\nabla_{\mcL, 0}^{(\CH)})^{[\CH]}$ holds.
 This implies $(\msL, \nabla_{\mcL, 0}^{(\ell)}) \in \mr{Diag}_\CH$, and hence
 completes the proof of assertion (ii).
\end{proof}
\SSP

\bpr \label{PPYY1}
Let $n$ be a positive integer with $p\nmid n$ and 
$\mcE_{\mr{GL}}$  a $\mr{GL}_n$-bundle on $X$.
Suppose that
either $\CH = 0$ or $\M =0$ is satisfied.
Also, suppose  that we are given a pair $(\phi_{\mr{PGL}}, \phi_{\mbG_m})$, where
\begin{itemize}
\item
$\phi_{\mr{PGL}}$ denotes an $\M$-PD stratification on the $\mr{PGL}_n$-bundle $\mcE_{\mr{PGL}}:= \mcE_{\mr{GL}} \times^{\mr{GL}_n} \mr{PGL}_n$ induced from  $\mcE_{\mr{GL}}$ via the natural quotient $\mu_{\mr{PGL}} : \mr{GL}_n \migisurj \mr{PGL}_n$;
\item
$\phi_{\mbG_m}$ denotes an $\M$-PD stratification on the $\mbG_m$-bundle $\mcE_{\mbG_m} := \mcE_{\mr{GL}} \times^{\mr{GL}_n} \mbG_m$ induced from $\mcE_{\mr{GL}}$ via the determinant map $\mu_{\mbG_m} : \mr{GL}_n \migisurj \mbG_m$.
\end{itemize}
Then, the following assertions hold:
\begin{itemize}
\item[(i)]
There exists a unique $\M$-PD stratification $\STR_{\mr{GL}}$ on $\mcE_{\mr{GL}}$ satisfying the following equalities:
\begin{align} \label{dEEEr}
\STR_{\mr{GL}} \times^{\mr{GL}_n, \mu_{\mr{PGL}}} \mr{PGL}_n = \STR_{\mr{PGL}}, \hspace{10mm} \STR_{\mr{GL}} \times^{\mr{GL}_n, \mu_{\mbG_m}} \mbG_m = \STR_{\mbG_m}
\end{align}
(cf. \eqref{dFFgy6} for the  change of structure group of an $\M$-PD stratification).
\item[(ii)]
 If, moreover,  both $\STR_{\mr{PGL}}$ and $\STR_{\mbG_m}$ are dormant, then
 the resulting pair  $(\mcE_{\mr{GL}}, \STR_{\mr{GL}})$
   (or equivalently, the corresponding $\mcD^{(\M)}$-module)
     is dormant.
   \end{itemize}
\epr
\begin{proof}
First, we shall consider assertion (i).
Let us take  an open covering $\msU := \{ U_\alpha \}_{\alpha \in I}$ of $X$ such that, for each $\alpha \in I$, there exists an $\M$-PD stratification $\phi'_\alpha$ on $\mcE_{\mr{GL}}|_{U_\alpha}$ with $\phi'_\alpha \times^{\mr{GL}_n}\mr{PGL}_n = \phi_{\mr{PGL}}|_{U_\alpha}$.
(Such a covering always exists, since the kernel of the morphism  $\mr{GL}_n \rightarrow \mr{PGL}_n \times \mbG_m$ given by $h \mapsto (\overline{h}, \mr{det}(h))$  is $\mu_n$, which is \'{e}tale under the assumption that $p\nmid n$.)
By  Proposition \ref{P129},  (i), 
there exists an $\M$-PD stratification $\phi_{\mbG_m, \alpha}$ on the trivial $\mbG_m$-bundle  on $U_\alpha$ such that $\phi_{\mbG_m, \alpha}^{\otimes n} \otimes (\phi'_\alpha \times^{\mr{GL}_n} \mbG_m) = \phi_{\mbG_m}|_{U_\alpha}$, where 
the tensor product ``$\otimes$" is defined 
 in a natural manner.
Write $\phi_{\alpha} := \phi'_\alpha \otimes \phi_{\mbG_m, \alpha}$.
Then, it is verified that
 $\phi_\alpha \times^{\mr{GL}_n} \mbG_m = \phi_{\mbG_m}|_{U_\alpha}$ and 
  $\phi_{\alpha}\times^{\mr{GL}_n} \mr{PGL}_n = \phi_{\mr{PGL}} |_{U_\alpha}$.
By the uniqueness assertion of Proposition \ref{P129},  (i), 
an  $\M$-PD stratification on $\mcE_{\mr{GL}}|_{U_\alpha}$ satisfying these properties  is uniquely determined.
In particular, 
for any pair $(\alpha, \beta) \in I \times I$ with $U_\alpha\cap U_\beta \neq\emptyset$, 
we have $\phi_{\alpha} |_{U_\alpha\cap U_\beta} = \phi_\beta |_{U_\alpha\cap U_\beta}$.
This implies that
$\{ \phi_\alpha \}_{\alpha \in I}$  may be glued together to obtain an $\M$-PD  stratification $\STR_{\mr{GL}}$ on $\mcE_{\mr{GL}}$ satisfying \eqref{dEEEr}.
This completes the proof of assertion (i).

Next,  we shall prove  assertion (ii) under the assumption that $\CH = 0$.
The vector bundle $\mcV$ corresponding to $\mcE_{\mr{GL}}$ carries  the $\mcD^{(\M)}$-module structure $\nabla_{\mr{GL}}$ 
corresponding to  $\phi_{\mr{GL}}$.
Since $\phi_{\mr{PGL}}$ is dormant,
the $p^{\M +1}$-curvature of $\nabla_\mr{GL}$ is given by  $\psi (\nabla_\mr{GL}) = \phi' \cdot \mr{id}_\mcV$ for some
$\mcO_X$-linear morphism $\phi' : \mcT^{\otimes p^{\M +1}} \rightarrow  \mcO_X$.
By \eqref{EEE501}, the $p^{\M +1}$-curvature of  the determinant $\mr{det} (\nabla_\mr{GL})$  satisfies $\psi (\mr{det}(\nabla_\mr{GL})) = n \cdot \psi'$.
Since $\nabla_{\mbG_m}$ (which corresponds to $\mr{det}(\nabla_\mr{GL})$) is dormant, we  have $n \cdot \psi' = 0$.
Hence, the assumption $p \nmid n$ implies $\psi' =0$, i.e., $\phi_\mr{GL}$ is dormant. 


Finally,  we shall prove  assertion (ii) under the assumption that $\M  = 0$.
 The problem is of local nature with respect to the \'{e}tale topology on $X$,
 we may assume that there exists a rank $n$ dormant flat  bundle $(\mcF, \DMO_\mcF)$ on $X^\mr{log}/S^\mr{log}$ whose projectivization corresponds to $\STR_{\mr{PGL}}$.
 Denote by $(\mcV, \DMO_\mcV)$ the flat bundle corresponding to $(\mcE_{\mr{GL}}, \STR_{\mr{GL}})$.
 Then, we obtain an isomorphism $h : \mcV \isom \mcL \otimes \mcF$ for some line bundle $\mcL$.
 It follows from Proposition \ref{P129}, (i), that there exists $S^\mr{log}$-connection $\nabla_\mcL$ on $\mcL$ such that 
 the determinant of $h$ defines an isomorphism of flat line bundles
 \begin{align}
 (\mr{det}(\mcV), \mr{det}(\DMO_\mcV))\cong (\mcL^{\otimes n} \otimes \mr{det}(\mcF),  \nabla_\mcL^{\otimes n} \otimes \mr{det}(\nabla_\mcF)).
 \end{align}
  But, since both   $(\mr{det} (\mcV), \mr{det}(\DMO_\mcV))$ (which corresponds to $(\mcE_{\mbG_m}, \STR_{\mbG_m})$)
and $(\mr{det}(\mcF), \mr{det} (\nabla_\mcF))$ are dormant,
Proposition \ref{P129}, (ii), implies that $(\mcL, \nabla_\mcL)$ is dormant.
By the uniqueness portion of assertion (i),
$h$ defines an isomorphism $(\mcV, \nabla_\mcV) \isom (\mcL \otimes \mcF, \nabla_\mcL \otimes \nabla_\mcF)$.
 In particular, $(\mcV, \nabla_\mcV)$ turns out to be  dormant.
 This completes the proof of assertion (ii).
 \end{proof}

\LSP
\subsection{Deformation space of a dormant flat bundle} \label{S21}

In the rest of this section, suppose that the relative characteristic of $X^\mr{log}/S^\mr{log}$ is trivial, and  that $X$ is a geometrically connected, proper, and smooth  curve of genus $> 1$ over $S := \mr{Spec}(R)$ for 
a flat $\mbZ/p^{\CH+1}\mbZ$-algebra $R$.
Let us fix an element $(\msL, \DMO_{\mcL, 0}^{(\CH)})$ (where $\msL := (\mcL, \nabla_\mcL)$) of $\mr{Diag}_{\CH}$ such that $\mcL$ is a line bundle on $X$.

Let $\M$ be an integer with $0 \leq \M < \CH$,  and let
$(\msF, \DMO_0^{(\M)})$ (where $\msF := (\mcF, \DMO)$) be an element of $\mr{Diag}_{\M }$ such that  $\mcF$ is  a vector bundle on $X_\M$ and 
  a pair of identifications $(\mr{det}(\msF), \mr{det} (\DMO_0^{(\M)})) = (\msL_{\M}, \DMO_{\mcL, 0}^{(l) \Rightarrow (\M )})$ is fixed.
We shall write
\begin{align} \label{NN288}
\mr{Lift} (\DMO_0^{(\M )}) \ \left(\text{resp.,} \  \mr{Lift}(\DMO^{(\M)}_0)_\msL\right)
\end{align}
for the set of $\mcD_0^{(\M +1)}$-module structures  $\DMO^{(\M +1)}_0$ on $\mcF_0$ with $\DMO_0^{(\M +1)\Rightarrow (\M)} = \DMO_0^{(\M )}$
 (resp., $\DMO_0^{(\M +1)\Rightarrow (\M)} = \DMO^{(\M )}_0$ and $\mr{det}(\DMO^{(\M +1)}_0)= \DMO_{\mcL, 0}^{(l) \Rightarrow (\M +1)}$). 

According to ~\cite[Th\'{e}or\`{e}me 2.3.6]{PBer2}, 
 giving a $\mcD_0^{(\M +1)}$-module structure  classified by  this set  amounts  to giving   an $S$-connection on the $\mcO_{X_0^{(\M+1)}}$-module $\mcS ol (\DMO_0^{(\M)})$
  (resp., an $S$-connection on the $\mcO_{X_0^{(\M +1)}}$-module $\mcS ol (\DMO_0^{(\M )})$ whose determinant  coincides with $\mr{det} ((\DMO_{\mcL, 0}^{(l)})^{[\M +1]}$). 
Hence, the set  $\mr{Lift} (\DMO_0^{(\M )})$ (resp.,  $\mr{Lift} (\DMO_0^{(\M )})_\msL$)  is, if it is nonempty, equipped with  a canonical structure of torsor modeled on  the $R$-module
\begin{align} \label{fE5}
H^0 (X_0^{(\M +1)}, \Omega^{(\M +1)}_0 \otimes \mcE nd (\mcS ol (\DMO_0^{(\M )})))  \ \ \ \ \\
\left(\text{resp.,} \  H^0 (X_0^{(\M +1)}, \Omega_0^{(\M +1)} \otimes \mcE nd^0   (\mcS ol (\DMO_0^{(\M )}))) \right),  \notag
\end{align}
where $\mcE nd (\mcS ol (\DMO_0^{(\M )})))$ (resp., $\mcE nd^0 (\mcS ol (\DMO_0^{(\M )}))$) denotes the  sheaf of  $\mcO_{X_0^{(\M +1)}}$-linear endomorphisms of $\mcS ol (\DMO_0^{(\M )}))$ (resp., $\mcO_{X_0^{(\M +1)}}$-linear endomorphisms of $\mcS ol (\DMO_0^{(\M )})$ with vanishing trace).
This torsor structure is defined in such a way that
if $\DMO'$ is the $S$-connection corresponding to an element of $\mr{Lift}(\DMO_0^{(\M )})$ (resp., $\mr{Lift}(\DMO_0^{(\M )})_\msL$) and $A$ is an element of the $R$-module displayed in  \eqref{fE5},
then the result of  the action by $A$ on $\DMO'$ is the $S$-connection $\DMO' + A$.

\SSP
\bpr \label{PR345}
The natural inclusion $\mr{Lift} (\DMO_0^{(\M )})_\msL \migiincl \mr{Lift} (\DMO_0^{(\M )})$ commutes  with the respective torsor structures via
the injection
\begin{align} \label{fE10}
& H^0 (X_0^{(\M +1)}, \Omega_0^{(\M +1)} \otimes \mcE nd^0  (\mcS ol (\DMO_0^{(\M )})))  
\migiincl  H^0 (X_0^{(\M +1)}, \Omega_0^{(\M +1)} \otimes \mcE nd (\mcS ol (\DMO_0^{(\M )}))) 
\end{align}
induced by the inclusion $\mcE nd^0 (\mcS ol (\DMO_0^{(\M )})) \migiincl \mcE nd (\mcS ol (\DMO^{(\M )}_0))$.
\epr
\begin{proof}
The assertion follows from the torsor structures  of  $\mr{Lift} (\DMO_0^{(\M )})$ and $\mr{Lift} (\DMO_0^{(\M )})_\msL$.
\end{proof}
\SSP

We shall  denote by 
\begin{align} \label{fE11}
\nabla^{(\M )}_{0, \mcE nd} \ \left(\text{or} \  (\nabla^{(\M )}_{0})_{\mcE nd} \right) : {^L}\mcD_0^{(\M)} \migi \mcE nd_{\mcO_{S_0}} (\mcE nd (\mcF_0))
\end{align}
 the 
$\mcD_0^{(\M )}$-module structure on 
 $\mcE nd (\mcF_0) \left(:= \mcE nd_{\mcO_{X_0}} (\mcF_0) \right)$ induced naturally by $\nabla_0^{(\M )}$.
It  restricts to
a $\mcD_0^{(\M )}$-module structure
\begin{align} \label{fE12}
\nabla^{(\M )}_{0, \mcE nd^0}  \ \left(\text{or} \  (\nabla^{(\M )}_{0})_{\mcE nd^0} \right) : {^L}\mcD_0^{(\M)} \migi \mcE nd_{\mcO_S} (\mcE nd^0 (\mcF_0))
\end{align}
  on the subsheaf $\mcE nd^0 (\mcF_0)$ consisting of $\mcO_{X_0}$-linear endomorphisms with vanishing trace.
The $S$-connection $\nabla_{0, \mcE nd}^{(\M )\Rightarrow (0)}$ (resp., $\nabla_{0, \mcE nd^0}^{(\M )\Rightarrow (0)}$) may be regarded as 
 a complex of sheaves  $\mcK^\bullet [\nabla_{0, \mcE nd}^{(\M )\Rightarrow (0)}]$ (resp., $\mcK^\bullet [\nabla_{0, \mcE nd^0}^{(\M )\Rightarrow (0)}]$)
 concentrated at degrees $0$ and $1$.

Next, we shall write
\begin{align} \label{NN203}
\mr{Lift} (\msF) \  \left(\text{resp.,}  \ \mr{Lift} (\msF)_\msL \right)
\end{align}
for  the set of isomorphism classes of  flat bundles $\msG$ on $X_{\M+1}/S_{\M+1}$ with $\msG_{\M} = \msF$ (resp., flat bundles $\msG$ on $X_{\M+1}/S_{\M+1}$  with $\msG_{\M} = \msF$  together with an identification $\mr{det} (\msG) = \msL_{\M+1}$ compatible with the fixed identification $\mr{det}(\msF) = \msL_{\M}$). 
It follows from  well-known generalities of deformation theory that
the set $\mr{Lift} (\msF)$ (resp., $\mr{Lift} (\msF)_\msL$) is, if it is nonempty, equipped with a  structure of torsor modeled on
$\mbH^1 (X_0, \mcK^\bullet [\DMO_{0, \mcE nd}^{(\M )\Rightarrow (0)}])$
(resp., $\mbH^1 (X_0, \mcK^\bullet [\DMO_{0, \mcE nd^0}^{(\M )\Rightarrow{(0)}}])$) in the manner of ~\cite[Section 6.1.4]{Wak8}.

\SSP
\bpr \label{WWW1}
The map  $\mr{Lift} (\msF)_\msL \rightarrow  \mr{Lift} (\msF)$ obtained by forgetting  identifications of the determinants   commutes  with the respective torsor structures via the morphism 
\begin{align}
\mbH^1 (X_0, \mcK^\bullet [\DMO_{0, \mcE nd^0}^{(\M )\Rightarrow (0)}]) \migi \mbH^1 (X_0, \mcK^\bullet [\DMO_{0, \mcE nd}^{(\M )\Rightarrow (0)}])
\end{align}
induced by the inclusion  $\mcK^\bullet [\DMO_{0, \mcE nd^0}^{(\M )\Rightarrow (0)}] \migiincl \mcK^\bullet [\DMO_{0, \mcE nd}^{(\M )\Rightarrow (0)}]$.
\epr
\begin{proof}
The assertion follows from the torsor structures of $\mr{Lift} (\msF)$ and $\mr{Lift} (\msF)_\msL$.
\end{proof}
\SSP

Since the equivalence of categories \eqref{YY6} commutes  with the formation of internal  Homs,
there exists a natural identification
\begin{align} \label{fE1}
\mcE nd (\mcS ol (\DMO_0^{(\M )})) = \mcS ol (\DMO_{0, \mcE nd}^{(\M )})
 \  \left(\text{resp.,} \ 
 \mcE nd^0 (\mcS ol (\DMO_0^{(\M )})) = \mcS ol (\DMO_{0, \mcE nd^0}^{(\M )})
  \right).
\end{align}
Under this  identification,  the  Cartier operator 
associated to 
$
(\mcE nd (\mcF_0), \DMO_{0, \mcE nd}^{(\M )})$ (resp., $(\mcE nd^0 (\mcF_0), \DMO_{0, \mcE nd^0}^{(\M )})$) yields a morphism of complex 
\begin{align} \label{NN204}
\mcK^\bullet [\DMO_{0, \mcE nd}^{(\M )\Rightarrow (0)}] \migi (\Omega_0^{(\M +1)} \otimes \mcE nd (\mcS ol (\DMO_0^{(\M )}))) [-1] \\
\left(\text{resp.,} \  \mcK^\bullet [\DMO_{0, \mcE nd^0}^{(\M )\Rightarrow (0)}] \migi (\Omega_0^{(\M +1)} \otimes \mcE nd^0 (\mcS ol (\DMO_0^{(\M )}))) [-1]  \right). \notag
\end{align}
By applying the $1$-st hypercohomology functor  to this morphism, we obtain an $R$-linear morphism
\begin{align} \label{fE8}
\mbH^1 (X_0, \mcK^\bullet [\DMO_{0, \mcE nd}^{(\M )\Rightarrow (0)}]) \migi H^0 (X_0^{(\M +1)}, \Omega_0^{(\M +1)}\otimes \mcE nd (\mcS ol (\DMO_0^{(\M )}))) \ \ \  \ \   \, \\
\left(\text{resp.,} \  \mbH^1 (X_0, \mcK^\bullet [\DMO_{0, \mcE nd^0}^{(\M )\Rightarrow (0)}]) \migi H^0 (X_0^{(\M +1)}, \Omega_0^{(\M +1)} \otimes \mcE nd^0 (\mcS ol (\DMO_0^{(\M )})))\right). \notag
\end{align}

\SSP
\bpr \label{P016} 
The assignment  $\msG \mapsto {^{\Diag}}\!\!\msG$ (cf. Remark \ref{YY456}) for each $\msG \in \mr{Lift}(\msF)$ (resp., $\msG \in \mr{Lift}(\msF)_\msL$)
  defines 
  a map of sets
\begin{align} \label{fE6}
\mr{Lift} (\msF) \migi \mr{Lift}(\DMO_0^{(\M )})  \ \left(\text{resp.,} \ \mr{Lift} (\msF)_\msL \migi \mr{Lift}(\DMO_0^{(\M )})_\msL  \right).
\end{align}
Moreover, this map commutes  with
the respective  torsor structures via the morphism \eqref{fE8}.
 \epr
\begin{proof}
The assertion follows immediately from 
 the fact mentioned in Remark \ref{YY250} and the definition of the functor ${^{\Diag}}\!\!(-)$.
\end{proof}
\SSP

Finally, we shall  denote by 
\begin{align}
\mr{Lift} (\DMO_0^{(\M )})^{\psi} \ \left(\text{resp.,} \ \mr{Lift} (\DMO_0^{(\M )})^{\psi}_\msL \right)
\end{align}
the subset of  $\mr{Lift} (\DMO_0^{(\M )})$ (resp., $\mr{Lift} (\DMO_0^{(\M )})_\msL$) consisting of $\mcD_0^{(\M +1)}$-module structures  with vanishing $p^{\M +2}$-curvature.
Also, 
denote by
\begin{align}
\mr{Lift} (\msF)^{\psi} \ \left(\text{resp.,} \ \mr{Lift} (\msF)_\msL^{\psi}\right)
\end{align}
the subset of $\mr{Lift} (\msF)$ (resp., $\mr{Lift} (\msF)_\msL$)
 consisting of dormant flat bundles.
 The following assertion is immediately verified.

\SSP
\bpr \label{Pr459}
The assignment  $\msG \mapsto {^{\Diag}}\!\!\msG$ for each $\msG \in \mr{Lift}(\msF)^\psi$ (resp., $\msG \in \mr{Lift}(\msF)_\msL^\psi$)
 yields a map of sets
\begin{align} \label{fE17}
 \mr{Lift} (\msF)^{\psi} \migi  \mr{Lift} (\DMO_0^{(\M )})^{\psi}
 \ \left(\text{resp.,} \ \mr{Lift} (\msF)^{\psi}_\msL \migi  \mr{Lift} (\DMO_0^{(\M )})^{\psi}_\msL \right)
\end{align}
Moreover,   the following commutative square diagram   is Cartesian:
\begin{align}  \label{fdd01}
\vcenter{\xymatrix@C=36pt@R=36pt{
 \mr{Lift} (\msF)^{\psi}\ar[r]^-{\eqref{fE17}} \ar[d]_-{\mr{inclusion}} &  \mr{Lift} (\DMO_0^{(\M )})^{\psi}
 \ar[d]^-{\mr{inclusion}}
 \\
 \mr{Lift} (\msF) \ar[r]_-{\eqref{fE6}} & \mr{Lift} (\DMO_0^{(\M )})
 }}  \ 
 \left(\text{resp.,} \ \vcenter{\xymatrix@C=36pt@R=36pt{
 \mr{Lift} (\msF)^{\psi}_\msL\ar[r]^-{\eqref{fE17}} \ar[d]_{\mr{inclusion}} &  \mr{Lift} (\DMO_0^{(\M )})^{\psi}_\msL
 \ar[d]^-{\mr{inclusion}}
 \\
 \mr{Lift} (\msF)_\msL \ar[r]_-{\eqref{fE6}} & \mr{Lift} (\DMO_0^{(\M )})_\msL
 }} \right). \hspace{-3mm}
\end{align}
\epr
\begin{proof}
The assertion follows from Propositions \ref{UU577}, (i),   \ref{P016dd}, (i), and Corollary \ref{KK1}.
\end{proof}

\LSP
\subsection{Canonical diagonal liftings of $p^{\M +1}$-flat line bundles} \label{SSrrfe}

In what follows, we will show that any $\mcD^{(\M)}_0$-module structure on a line bundle with vanishing $p^{\M +1}$-curvature admits a unique diagonal lifting when the underlying curve is ordinary (cf. Proposition \ref{Prop7872}).

Denote by $d \mr{log} : \mcO_X^\times \rightarrow \Omega$ the morphism given by assigning $v \mapsto d\mr{log} (v) := \frac{dv}{v}$.
The natural short exact sequence $0 \migi \Omega [-1] \migi \mcK^\bullet [d \mr{log}] \rightarrow \mcO^\times_X [0] \migi 0$ yields an exact sequence of (hyper)cohomology groups
\begin{align}
\mbH^1 (X, \mcK^\bullet [d\mr{log}]) \xrightarrow{\mr{lin}} H^1 (X, \mcO^\times_X) \xrightarrow{\mr{obs} \,  (:=H^1(d\mr{log}))} H^1 (X, \Omega).
\end{align}
Recall from ~\cite[Proposition (7.2.1)]{Katz2} that
the group $\mbH^1 (X, \mcK^\bullet [d\mr{log}])$ (resp., $H^1 (X, \mcO^\times_X)$) parametrizes 
 isomorphism classes of flat line bundles on $X/S$ (resp., line bundles on $X$), and
  the map ``$\mr{lin}$"  is given by forgetting the data of $S$-connections.

For each line bundle $\mcN$ on $X$,
we denote by 
\begin{align} \label{eeQQ468}
\mr{obs}(\mcN) \in H^1 (X, \Omega)
\end{align}
 the image via the map ``$\mr{obs}$" of the element in  $H^1 (X, \mcO^\times_X)$ represented by $\mcN$.
 The obstruction to the surjectivity of $\mr{lin}$ lies in $H^1 (X, \Omega)$, which means that
$\mr{obs}(\mcN) = 0$ if and only if $\mcN$ admits an $S$-connection.

Now, let us fix a line bundle $\mcL$ on $X$ with $\mr{obs}(\mcL) = 0$.
Also, 
 let $\M$ be an integer with $0 \leq \M < \CH$ and $\nabla_{\M}$ an $S_\M$-connection on $\mcL_\M$ such that $\msL_\M :=(\mcL_\M, \nabla_\M)$ is dormant.
 Write ${^{\Diag}}\!\!\msL_\M := (\mcL_0, \nabla_0^{(\M)})$.
As discussed in \eqref{NN288}, the set $\mr{Lift} (\nabla_0^{(\M)})$ of $\mcD_0^{(\M +1)}$-module structures $\nabla_0^{(\M +1)}$ on $\mcL_0$ with $\nabla_0^{(\M +1) \Rightarrow (\M)} = \nabla_0^{(\M)}$ forms, if it is nonempty, a torsor modeled on $H^0 (X_0^{(\M +1)}, \Omega_0^{(\M +1)})$; it contains a subset  $\mr{Lift} (\nabla_0^{(\M)})^\psi$ consisting of $\mcD_0^{(\M +1)}$-module structures with vanishing $p^{\M +2}$-curvature.

Moreover, similarly to \eqref{NN203},
we  have the set 
\begin{align}
\mr{Lift} (\msL_\M)_\mcL \ \left(\text{resp.,} \ \mr{Lift} (\msL_\M)_\mcL^\psi  \right)
\end{align}
  of $S$-connections (resp., dormant $S$-connections) on $\mcL_{\M +1}$  whose reduction modulo $p^{\M+1}$ coincides with $\nabla_\M$.  
Since $\mr{obs}(\mcL) = 0$ and $S$ is affine,
$\mr{Lift} (\msL_\M)_\mcL$ is verified to be nonempty and forms a torsor modeled on $H^0 (X_0, \Omega_0)$.

The operation of taking diagonal reductions ${^{\Diag}}\!\!(-)$ gives a map of sets
\begin{align} \label{eeQQ511}
\mr{Lift}(\msL_\M)_\mcL \rightarrow   \mr{Lift} (\nabla_0^{(\M)})
\end{align}
(cf. Remark \ref{YY456}), by which we see  that $\mr{Lift} (\nabla_0^{(\M)})$ is nonempty.
By the dormancy condition, it can be restricted  to a map
\begin{align} \label{fewidw}
\mr{Lift}(\msL_\M)_\mcL^\psi \rightarrow \mr{Lift} (\nabla_0^{(\M)})^\psi.
\end{align}

Note that the map \eqref{eeQQ511} commutes  with the respective torsor structures via the  Cartier operator
\begin{align} \label{eeQQ801}
H^0 (C_{(\mcO_{X_0}, \nabla_{X_0, \mr{triv}}^{(\M)})}) : H^0 (X_0, \Omega_0) \rightarrow H^0 (X_0^{(\M +1)}, \Omega_0^{(\M +1)})
\end{align}
of $(\mcO_{X_0}, \nabla_{X_0, \mr{triv}}^{(\M)})$ (cf. \eqref{dGG6}, \eqref{dE3}).
Hence, if the curve $X_0/S_0$ is ordinary in the usual  sense, then  (since \eqref{eeQQ801} is bijective) the map \eqref{eeQQ511}, as well as \eqref{fewidw},  becomes a bijection. 
Applying inductively the bijectivity of \eqref{fewidw}  for various $\M$'s,
we obtain the following proposition, asserting the existence of canonical diagonal liftings for the rank one case.

\SSP
\bpr \label{Prop7872}
Let $\mcL$ be as above and  $\nabla^{(\M)}_{0}$  a $\mcD^{(\M)}_0$-module structure on $\mcL_0$ with vanishing $p^{\M +1}$-curvature.
Suppose  that the curve $X_0/S_0$ is ordinary.
Then, there exists a unique  diagonal lifting of $(\mcL_0, \nabla^{(\M)}_{0})$
whose underlying $\mcO_{X_\M}$-module coincides with $\mcL_\M$.
\epr

\vspace{10mm}
\section{Local study  of  $\mcD^{(\M)}$-modules}\label{S0111}
\SSP

In this section, we study the local description of $p^{\M +1}$-flat $\mcD^{(\M)}$-modules   in characteristic $p$ around a marked/nodal point of the underlying pointed stable curve.
By applying the resulting  description, we show  that
 each  dormant flat bundle in characteristic $p^{\M +1}$ can be decomposed into    the direct sum of rank $1$ flat bundles of a certain type (cf. Proposition-Definition \ref{P724}).

Throughout this section, we fix an affine scheme $S = \mr{Spec} (R)$ for a commutative ring $R$.

\LSP
\subsection{Formal neighborhoods of  a marked/nodal point} \label{SS1071}
We shall set
\begin{align} \label{YY124}
U_\oslash := \mr{Spec} (R [\![t ]\!]), \hspace{5mm}
 U_\otimes := \mr{Spec} (R [\![t_1, t_2 ]\!]/(t_1 t_2)),
\end{align}
where $t$, $t_1$, and $t_2$ are formal parameters.
Let $\pi_\oslash : U_\oslash \rightarrow S$ and $\pi_\otimes : U_\otimes \rightarrow S$ denote the natural projections.
For simplicity, we write $\mcO_{\oslash} := \mcO_{U_\oslash}$ and $\mcO_{\otimes} := \mcO_{U_\otimes}$.
Throughout this section, logarithmic differentials on the completed local models are understood in the continuous sense, and connections take values in the corresponding sheaves of completed logarithmic differentials.

Hereinafter, let us fix a pair of nonnegative integers  $(\CH, \M)$ with $\CH = 0$ or $\M =0$, and suppose that $S$ is  flat  over $\mbZ/p^{\CH +1}\mbZ$.
Denote by  $S^\mr{log}$ the log scheme defined as $S$ equipped with the trivial log structure.
Also,
 we equip $U_\oslash$ with the log structure associated to the monoid morphism $\mbN \migi \mcO_\oslash$  given by $n \mapsto t^n$; if $U_\oslash^{\mr{log}}$ denotes the resulting log scheme, then
 we obtain the sheaf  of  noncommutative rings $ \mcD_{\oslash}^{(\M)}$ on $U_\oslash$    corresponding to the $R$-algebra
 $\varinjlim_{j \geq 0} \varprojlim_{n \geq 1} H^0 (U_{\oslash, n}, \mcD_{U_{\oslash, n}^\mr{log}/S^\mr{log}, \leq j}^{(\M)})$,
where  $U_{\oslash, n}^\mr{log}$ ($n \geq 1$) denotes the strict closed subscheme of $U_{\oslash}^\mr{log}$ (over $S$) defined by the ideal  $(t^n)\subseteq R [\![t]\!]$.
This sheaf has two $\mcO_\oslash$-module structures ${^L}\mcD_{\oslash}^{(\M)}$ and ${^R}\mcD_{\oslash}^{(\M)}$ arising from  those of ${^L}\mcD_{U_{\oslash, n}^\mr{log}/S^\mr{log}}^{(\M)}$'s and ${^R}\mcD_{U_{\oslash, n}^\mr{log}/S^\mr{log}}^{(\M)}$'s, respectively.
The $\mcO_{\oslash}$-module  ${^L}\mcD_{\oslash}^{(\M)}$ can be decomposed as the direct sum 
$\bigoplus_{j \in \mbZ_{\geq 0}} \mcO_\oslash \cdot \partial^{\langle j \rangle}$, where
$\partial^{\langle j \rangle}$'s 
denote the sections associated to the logarithmic coordinate $t$ as defined in Section \ref{SS04f4}.
In particular, the multiplication in   $ \mcD_{\oslash}^{(\M)}$ is given by \eqref{e364}.

On the other hand, we equip $U_\otimes$ (resp., $S$) with the log structure associated to the monoid morphism $\mbN \oplus \mbN \migi \mcO_{\otimes}$ (resp., $\mbN \migi \mcO_S$) given by $(n_1, n_2) \mapsto t^{n_1}_1 \cdot t_2^{n_2}$ (resp., $n \mapsto 0^n$); denote by  $U_\otimes^\mr{log}$ (resp., $S^\mr{log}$) the resulting log scheme.
By the diagonal embedding $\mbN \migi \mbN \oplus \mbN$, 
the projection  $U_\otimes \migi S$ extends  to a morphism of log schemes $U_\otimes^\mr{log} \migi S^\mr{log}$.
We obtain the sheaf $\mcD^{(\M)}_{\otimes}$ on $U_\otimes$ corresponding to the $R$-algebra
$\varinjlim_{j \geq 0} \varprojlim_{n \geq 1} H^0 (U_{\otimes, n}, \mcD_{U_{\otimes, n}^\mr{log}/S^\mr{log}, \leq j}^{(\M)})$,
where  $U_{\otimes, n}^\mr{log}$ ($n \geq 1$) denotes the strict closed subscheme of $U_{\otimes}^\mr{log}$ (over $S$) defined by the ideal $(t_1, t_2)^n \subseteq R [\![t_1, t_2]\!]/(t_1t_2)$.
Just as in the case of  $\mcD^{(\M)}_{\oslash}$, this sheaf has two $\mcO_\otimes$-module structures  ${^L}\mcD^{(\M)}_{\otimes}$ and ${^R}\mcD^{(\M)}_{\otimes}$.

For each $i=1, 2$, denote by $\{ \partial_i^{\langle j \rangle}\}_{j \in \mbZ_{\geq 0}}$ the basis of $\mcD^{(\M)}_{\otimes}$ associated to the logarithmic coordinate $t_i$.
In particular, we have ${^L}\mcD^{(\M)}_{\otimes} = \bigoplus_{j \in \mbZ_{\geq 0}} \mcO_\otimes \cdot  \partial_i^{\langle j \rangle}$.

\SSP
\ble \label{P238}
For each positive  integer $j \leq p^\M$, the following equality holds:
\begin{align} \label{YY8}
\partial_{2}^{\langle j \rangle} = (-1)^j \cdot \sum_{j'=1}^j \binom{j-1}{j'-1} \cdot \partial_1^{\langle j'\rangle}.
\end{align}
In particular, for each nonnegative  integer $a\leq \M$,
 the following equalities hold:
\begin{align} \label{e444}
\partial_{2}^{\langle p^a \rangle} =   \sum_{j'=1}^{p^a} (-1)^{j'} \cdot \partial_1^{\langle j'\rangle} = - \partial_1^{\langle p^a \rangle} -1 + \prod_{b = 0}^{a-1} (1 -(\partial_1^{\langle p^b \rangle})^{p-1}),
\end{align}
where we set $\prod_{b =0}^{-1} (-) := 1$.
\ele
\begin{proof}
Note that the morphism $\mu_{(\M)}$ and the sheaves $\mcP_{(\M)}^\ell$ ($\ell \in \mbZ_{\geq 0}$) discussed  in Section \ref{SS04f4} can be defined
even in the case where   $X^\mr{log}/S^\mr{log}$ is replaced by  $U_\otimes^\mr{log}/S^\mr{log}$.
 In particular, we obtain  the sections 
 $\eta_i := \mu_{(\M)}(t_i)-1$  ($i =1,2$) of $\mcP_{(\M)}^{\ell}$.
Since $\mu_{(\M)}$ preserves the monoid structure, we have
\begin{align}
(1 + \eta_1)(1 + \eta_2) = \mu_{(\M)}(t_1) \cdot \mu_{(\M)}(t_2) = \mu_{(\M)}(t_1 \cdot t_2) 
=1,
\end{align}
where the last equality follows from the fact that the section $t_1 t_2$ comes from the log structure of the base space $S^\mr{log}$.
This implies 
\begin{align}
\eta_1 =  -1 +(1 + \eta_2)^{-1}
 =  - \eta_2  +  \eta_2^2  - \eta_2^3 + \eta_2^4 - \cdots \left(= \sum_{j = 1}^\infty (-\eta_2)^j \right). \notag
\end{align}
Hence, for each  $j' \leq  p^\M$, 
the following equalities hold:
\begin{align} \label{e342}
\eta_1^{\{ j'\}} = \eta_1^{j'}   = ( \sum_{j = 1}^\infty (-\eta_2)^j )^{j'}
= \sum_{i=0}^\infty  \binom{j'+ i -1}{j'-1}\cdot (-\eta_2)^{j' + i}.
\end{align}
By taking duals, we obtain \eqref{YY8}.
This completes  the proof of the first assertion.
Moreover,
the second  assertion follows from \eqref{YY8} together with Lemma \ref{L99}.
\end{proof}
\SSP

Next, we shall set 
\begin{align} \label{MM3}
\iota_1 :U_\oslash \migiincl U_\otimes   \ \left(\text{resp.,} \ \iota_2 : U_\oslash \migiincl U_\otimes  \right)
\end{align}
to be  the closed  immersion  corresponding to the surjection  $R [\![t_1, t_2 ]\!]/(t_1t_2) \migisurj R [\![t]\!]$ given by $t_1 \mapsto t$ and $t_2 \mapsto 0$ (resp., $t_1 \mapsto 0$ and $t_2 \mapsto t$).
The inclusion into the $i$-th factor $\mbN \migiincl \mbN \oplus \mbN$ determines a morphism of log schemes
\begin{align} \label{MM2}
U_\oslash \times_{\iota_i, U_\otimes} U_\otimes^\mr{log} \migi U_\oslash^\mr{log}
\end{align}
whose underlying morphism of $S$-schemes  coincides with the identity morphism of $U_\oslash$.
This morphism induces an identification $\iota_i^* ({^L}\mcD^{(\M)}_\otimes) = {^L}\mcD_\oslash^{(\M)}$.
 The $\mcO_\otimes$-linear  surjection
\begin{align} \label{MM1}
\omega_i : {^L}\mcD_\otimes^{(\M)} \migisurj \left(\iota_{i*}(\iota_i^* ({^L}\mcD^{(\M)}_\otimes)) =  \right) \iota_{i*}({^L}\mcD_\oslash^{(\M)})
\end{align}
corresponding to this identification via the adjunction relation ``$\iota^*_i(-) \dashv \iota_{i*}(-)$
 maps  $\partial_i^{\langle j \rangle}$ to $\partial^{\langle j \rangle}$ for every $j \geq 0$.

For a $\mcD_\otimes^{(\M)}$-module  $(\mcF, \DMO)$,
we shall write
\begin{align} \label{dE106}
\iota^*_i (\DMO) : {^L}\mcD_\oslash^{(\M)} \migi \mcE nd_{\mcO_S} (\iota_i^* (\mcF))
\end{align} 
for the  $\mcD_\oslash^{(\M)}$-module structure on $\iota_i^* (\mcF)$ induced naturally by $\DMO$ under the identification $\iota_i^* ({^L}\mcD^{(\M)}_\otimes) = {^L}\mcD_\oslash^{(\M)}$.

\LSP
\subsection{Monodromy operator} \label{SS1076}

We shall set
\begin{align} \label{YY11}
\mcB_S := \bigoplus_{j \in \mbZ_{\geq 0}} \mcO_S \cdot \partial_\mcB^{\langle j \rangle},
\end{align}
where $\partial_\mcB^{\langle i \rangle}$'s are   abstract symbols.
We equip $\mcB_S$ with a structure of $\mcO_S$-algebra given by 
\begin{align} \label{YY12}
\partial_\mcB^{\langle j' \rangle} \cdot \partial_\mcB^{\langle j''\rangle} = \sum_{j = \mr{max}\{j', j'' \}}^{j' + j''} \frac{j!}{(j' + j''-j)! \cdot (j-j')! \cdot (j-j'')!}\cdot \frac{q_{j'}! \cdot q_{j''}!}{q_j !} \cdot \partial_\mcB^{\langle j\rangle}.
\end{align}
In particular,  the  $\mcO_S$-algebra $\mcB_S$ is commutative and  generated by the sections $\partial_\mcB^{\langle 1 \rangle}, \partial_\mcB^{\langle p \rangle}, \cdots, \partial_\mcB^{\langle p^{\M}\rangle}$.

Let $\mcG$ be an $\mcO_S$-module and $\mu$ a morphism of $\mcO_S$-algebras 
$\mcB_S \migi \mcE nd_{\mcO_S} (\mcG)$.
For each $\IN \in \mbZ_{\geq 0}$,
we set 
$\mu^{\langle p^\IN \rangle} := \mu  (\partial_\mcB^{\langle p^\IN \rangle})$, which is an element of $\mr{End}_{\mcO_S}(\mcG)$.
The morphism $\mu$ is uniquely determined by  
the $(\M+1)$-tuple
\begin{align} \label{e187}
\mu^{\langle \bullet \rangle} := (\mu^{\langle 1 \rangle}, \mu^{\langle p \rangle} \cdots,  \mu^{\langle p^{\M} \rangle}) \in \mr{End}_{\mcO_S}(\mcG)^{\oplus (\M +1)}.
\end{align}

Denote by $\sigma_\oslash : S \migi U_\oslash$ (resp., $\sigma_\otimes : S \migi U_\otimes$) the closed immersion corresponding to the surjection $R [\![t]\!] \migisurj R$ given by $t \mapsto 0$ (resp., the surjection $R [\![t_1, t_2]\!] /(t_1t_2)\migisurj R$ given by $t_1 \mapsto 0$ and $t_2 \mapsto 0$).
It follows from \eqref{e364} and  \eqref{YY12} that the assignment $\partial^{\langle j \rangle} \mapsto \partial_\mcB^{\langle j \rangle}$ ($j \in \mbZ_{\geq 0}$) determines an isomorphism of $\mcO_S$-algebras 
$\sigma_\oslash^*({^L}\mcD_\oslash^{(\M)}) \isom \mcB_S$.
This morphism induces, via the adjunction relation ``$\sigma_\oslash^* (-) \dashv \sigma_{\oslash*} (-)$",
an $\mcO_{\oslash}$-linear surjection
\begin{align} \label{YY16}
\omega : {^L}\mcD_\oslash^{(\M)} \migisurj \sigma_{\oslash *} (\mcB_S).
\end{align}

Now, let $(\mcF, \DMO)$ be a $\mcD_\oslash^{(\M)}$-module.
The  $\mcD_\oslash^{(\M)}$-module structure  $\DMO$ induces  an $\sigma_\oslash^*({^L}\mcD_\oslash^{(\M)})$-action $\sigma_\oslash^*(\DMO)$  on $\sigma_\oslash^*(\mcF)$; it gives 
the composite
\begin{align}
\mu (\DMO) : \mcB_S \isom \sigma_\oslash^*({^L}\mcD_\oslash^{(\M)}) \xrightarrow{\sigma^*_\oslash (\DMO)} \mcE nd_{\mcO_S} (\sigma_\oslash^* (\mcF)).
\end{align}
In particular, we obtain 
\begin{align} \label{e18d7}
\mu (\DMO)^{\langle \bullet \rangle} := (\mu (\DMO)^{\langle 1 \rangle}, \mu (\DMO)^{\langle p \rangle} \cdots,  \mu (\DMO)^{\langle p^{\M} \rangle}) \in \mr{End}_{\mcO_S}(\sigma_\oslash^*(\mcF))^{\oplus (\M +1)}.
\end{align} 
Since $\mcB_S$ is commutative,
the elements $\mu (\DMO)^{\langle 1 \rangle}, \cdots, \mu (\DMO)^{\langle p^\M \rangle}$ commute with each other.

\SSP
\bde \label{D98}
We shall refer to $\mu (\DMO)$ and $\mu (\DMO)^{\langle \bullet \rangle}$ as 
the {\bf monodromy operator} of $\DMO$. 
Also, for each $a = 0, \cdots, \M$,
we shall refer to $\mu (\DMO)^{\langle p^a \rangle}$ as the {\bf $a$-th monodromy operator} of $\DMO$.
\ede
\SSP

Note  that there exists uniquely an automorphism
\begin{align} \label{e445eee}
\mr{sw} : \mcB_S \isom \mcB_S
\end{align}
 of the $\mcO_S$-algebra $\mcB_S$ 
 determined by 
 $\mr{sw} (\partial_\mcB^{\langle j \rangle}) = (-1)^j \cdot \sum_{j'=1}^j \binom{j-1}{j'-1}\cdot \partial_\mcB^{\langle j'\rangle}$
  for every positive integer $j \leq p^\M$ (cf. Lemma \ref{P238}).
 In particular,  the equality   $\mr{sw} (\partial_\mcB^{\langle p^a \rangle}) = \sum_{j'=1}^{p^a} (-1)^{j'} \cdot \partial_\mcB^{\langle j'\rangle}$ holds for $a \leq \M$.
 This automorphism is   involutive, i.e., $\mr{sw} \circ \mr{sw} = \mr{id}_{\mcB_S}$.

 For an $\mcO_S$-module $\mcG$,
 we shall set
\begin{align} \label{dE102}
\mr{sw}^\bullet_\mcG : \mr{End}_{\mcO_S} (\mcG)^{\oplus (\M +1)} \isom \mr{End}_{\mcO_S} (\mcG)^{\oplus (\M +1)}
\end{align}
to be the  endomap of  $\mr{End}_{\mcO_S} (\mcG)^{\oplus (\M +1)}$ given by
\begin{align} \label{dE10122}
\mr{sw}^\bullet_\mcG ((A_\IN)_{\IN=0}^{\M}) := ((-A_\IN -\mr{id}_\mcG + \prod_{b=0}^{\IN-1}(\mr{id}_\mcG - A_b^{p-1}))_{\IN=0}^\M),
\end{align}
where $\prod_{b=0}^{-1} (-) := \mr{id}_\mcG$ and $\prod_{b = 0}^{\IN-1} B_b := B_0 \circ  B_1 \circ  \cdots \circ B_{\IN-1}$ ($\IN  \geq 1$).
If $(\mcF, \DMO)$ is as above, then 
 the second equality in \eqref{e444} implies  
 \begin{align} \label{dE110}
 (\mu (\DMO) \circ \mr{sw})^{\langle \bullet \rangle} = \mr{sw}_{\sigma_\oslash^*(\mcF)}^\bullet (\mu (\DMO)^{\langle \bullet \rangle}).
 \end{align}  

 \SSP
\bpr \label{P27}
Let $\mcF$ be a locally free  $\mcO_{\otimes}$-module.
\begin{itemize}
\item[(i)]
Let $\DMO$ be a $\mcD_\otimes^{(\M)}$-module structure on $\mcF$.
Then, the equality 
\begin{align}
\mu (\iota_1^*(\DMO)) \circ \mr{sw} = \mu (\iota_2^*(\DMO))  
\left(\Longleftrightarrow  \mr{sw}^\bullet_{\sigma_\otimes^*(\mcF)}(\mu (\iota_1^*(\DMO))^{\langle \bullet \rangle}) = \mu (\iota_2^*(\DMO))^{\langle \bullet \rangle} \ \text{by \eqref{dE110}} \right)
\end{align}
 holds.
\item[(ii)]
Conversely, 
let $\DMO_i$ (for each   $i =1,2$) be   a $\mcD_\oslash^{(\M)}$-module structure on $\iota_i^*(\mcF)$, 
 and suppose that
the equality $\mu (\DMO_1) \circ \mr{sw} = \mu (\DMO_2)$ (or equivalently, $\mr{sw}^\bullet_{\sigma_\otimes^*(\mcF)}(\mu (\DMO_1)^{\langle \bullet \rangle}) = \mu (\DMO_2)^{\langle \bullet \rangle}$) holds.
Then, there exists a unique $\mcD_\otimes^{(\M)}$-module structure $\DMO$ on $\mcF$
 satisfying the equalities  $\iota_1^*(\DMO) = \DMO_1$ and $\iota_2^*(\DMO) = \DMO_2$ under the natural  identifications $\iota_1^* ({^L}\mcD_{\otimes}^{(\M)}) = {^L}\mcD_\oslash^{(\M)}$ and $\iota_2^* ({^L}\mcD_{\otimes}^{(\M)}) = {^L}\mcD_\oslash^{(\M)}$, respectively.
\end{itemize}
\epr
\begin{proof}
We shall set  ${^L}\breve{\mcD}_\otimes^{(\M)}$ to be the sheaf defined as
\begin{align} \label{MM6}
{^L}\breve{\mcD}_\otimes^{(\M)}&
:= \iota_{1*}({^L}\mcD_\oslash^{(\M)})  \times_{\omega, \mcB_S, \mr{sw}\circ \omega} \iota_{2*}({^L}\mcD_\oslash^{(\M)})
\\
& \left(=\left\{(s_1, s_2) \in \iota_{1*}({^L}\mcD_\oslash^{(\M)})  \times \iota_{2*}({^L}\mcD_\oslash^{(\M)})  \, \Big| \, \iota_{1*}(\omega)  (s_1) = (\sigma_{\otimes *}(\mr{sw})\circ \iota_{2*}(\omega)) (s_2)  \right\} \right).\notag
\end{align}
By using the  isomorphism $\mcO_\otimes \isom  \iota_{1*}(\mcO_\oslash) \times_{\sigma_{\otimes*}(\mcO_S)} \iota_{2*}(\mcO_\oslash)$ given by $t_1 \mapsto (t, 0)$ and $t_2 \mapsto (0, t)$,
we equip ${^L}\breve{\mcD}_\otimes^{(\M)}$ with an $\mcO_\otimes$-module structure.
It follows from Lemma \ref{P238} that the assignment
$s \mapsto (\omega_1(s), \omega_2(s))$ for each local section $s \in {^L}\mcD_\otimes^{(\M)}$ defines an $\mcO_\otimes$-linear isomorphism 
\begin{align} \label{e459}
{^L}\mcD_\otimes^{(\M)} \isom {^L}\breve{\mcD}_\otimes^{(\M)}.
\end{align}
Hence,  both assertions (i) and (ii)  are  direct consequences of this isomorphism.
\end{proof}
\SSP

\LSP
\subsection{$\mcD^{(\M)}$-module structures  $\nabla_{\oslash, \EX}^{(\M)}$ and $\nabla_{\otimes, \EX}^{(\M)}$} \label{SS182}

In Sections \ref{SS182}-\ref{SS203},
we suppose that $S = \mr{Spec} (R)$ for a commutative ring $R$ over $\mbF_p$ (equipped with the trivial $\M$-PD structure).
Note that the discussions in Sections \ref{SS0dd58}-\ref{SS0ddd58} can be applied 
 even when ``$X^\mr{log}/S^\mr{log}$" is replaced by $U_\oslash^\mr{log}/S^\mr{log}$ or $U_\otimes^\mr{log}/S^\mr{log}$.
 In particular, 
we can define the $p^{\M +1}$-curvature of a $\mcD_\oslash^{(\M)}$-module, as well as of a $\mcD_\otimes^{(\M)}$-module.

The completed $(\M+1)$-st  Frobenius twist $U_{\oslash}^{(\M +1)}$ (resp., $U_{\otimes}^{(\M +1)}$) of $U_\oslash$ (resp., $U_\otimes$) over $S$ may be identified with  
the affine scheme over $S$ associated to the $R$-subalgebra  
\begin{align}
R [\![t^{p^{\M+1}}]\!] \ \left(\text{resp.},  \ R [\![t^{p^{\M+1}}_1, t_2^{p^{\M+1}}]\!]/(t_1^{p^{\M+1}}t_2^{p^{\M+1}})\right)
\end{align}
 of $R  [\![t ]\!]$ (resp., $R  [\![t_1, t_2 ]\!]/(t_1 t_2)$).
For simplicity, we write $\mcO_{\oslash}^{(\M +1)} := \mcO_{U_{\oslash}^{(\M +1)}}$ and
$\mcO_{\otimes}^{(\M +1)} := \mcO_{U_{\otimes}^{(\M +1)}}$.

Given an element   $\EX$ of $\mbZ/p^{\M +1}\mbZ$, 
we denote by $\widetilde{\EX}$ the  integer defined as the unique lifting of $d$ via the natural surjection $\mbZ \migisurj \mbZ/p^{\M +1}\mbZ$ satisfying $0 \leq \widetilde{\EX} <p^{\M +1}$.
Let $\widetilde{\EX}_{[0]}, \cdots, \widetilde{\EX}_{[\M]}$
be 
the  collection of  integers uniquely determined by the condition that    $\widetilde{\EX}= \sum_{\IN=0}^{\M}p^\IN \cdot \widetilde{d}_{[\IN]}$ and $0 \leq \widetilde{\EX}_{[\IN]} <p$ ($\IN = 0, \cdots, \M$).
Also, 
for  each  $\IN = 0, \cdots, \M$, we write
 $\widetilde{\EX}_{[0, \IN]} := \sum_{b=0}^{\IN} p^b \cdot \widetilde{d}_{[b]}$, i.e., the remainder obtained by dividing $\widetilde{d}$ by $p^{\IN +1}$,  and write
$d_{[\IN]}$ (resp., $d_{[0, \IN]}$) for  the image of $\widetilde{d}_{[\IN]}$ (resp., $\widetilde{d}_{[0, \IN]}$) via the natural projection $\mbZ \migisurj \mbF_p$ (resp., $\mbZ \migisurj \mbZ/p^{\IN+1}\mbZ$).

Now, let us fix  an element $d \in \mbZ/p^{\M +1}\mbZ$.
Then, there exists a unique $\mcD_\oslash^{(\M)}$-module structure 
\begin{align} \label{YY52}
\DMO_{\oslash, \EX}^{(\M)} : {^L}\mcD_{\oslash}^{(\M)} \migi \mcE nd_{\mcO_S} (\mcO_{\oslash})
\end{align}
on $\mcO_{\oslash}$  determined by the condition that
$\DMO_{\oslash, \EX}^{(\M)} (\partial^{\langle j \rangle}) (t^n) = q_j ! \cdot \binom{n -\widetilde{\EX}}{j}\cdot t^n$
   for every $j$, $n \in \mbZ_{\geq 0}$.
The resulting $\mcD_{\oslash}^{(\M)}$-module 
\begin{align} \label{UU3}
\msO_{\oslash, \EX}^{(\M)}
   := (\mcO_{\oslash}, \DMO_{\oslash, \EX}^{(\M)})
\end{align}
 is isomorphic  to the unique extension of 
 $\msO_{\oslash, 0}^{(\M)}$
   to $t^{-\widetilde{\EX}}\cdot \mcO_{\oslash} \left(\supseteq \mcO_\oslash \right)$.
In particular  $\DMO^{(\M)}_{\oslash, \EX}$ has vanishing $p^{\M +1}$-curvature.

\SSP
\bpr \label{L093}
For each $\IN = 0, \cdots, \M$, we have 
\begin{align}
\mcS ol (\DMO_{\oslash, \EX}^{(\M)\Rightarrow (\IN)}) = t^{\widetilde{\EX}_{[0, \IN]}}\cdot \mcO_{\oslash}^{(\IN+1)} \left(\subseteq \mcO_\oslash \right).
\end{align}
Moreover, under the identification $\left(\mcS ol (\DMO_{\oslash, \EX}^{(\M)\Rightarrow (\IN-1)}) =\right) t^{\widetilde{\EX}_{[0, \IN-1]}}\cdot \mcO_{\oslash}^{(\IN)} = \mcO_{\oslash}$ given by $t^{\widetilde{\EX}_{[0, \IN-1]}} \cdot h (t^{p^{\IN}})\leftrightarrow h (t)$ for each $h (t) \in R[\![t]\!]$ (where $t^{\widetilde{\EX}_{[0, \IN-1]}}\cdot \mcO_{\oslash}^{(\IN)} := \mcO_\oslash$ if $a = 0$),
the following equality of $S$-connections holds:
\begin{align}
(\DMO_{\oslash, \EX}^{(\M)})^{[\IN]} = \DMO_{\oslash, \EX_{[\IN ]}}^{(0)}.
\end{align}
\epr
\begin{proof}
The assertion follows from the definition of $\DMO_{\oslash, \EX}^{(\M)}$.
\end{proof}

\SSP
\bpr \label{L090}
Under the natural identification $\mr{End}_{\mcO_S} (\sigma_\oslash^* (\mcO_\oslash)) = H^0 (S, \mcO_S) \left(\supseteq \mbF_p \right)$,
the following equality holds:
\begin{align}
\mu (\DMO_{\oslash, \EX}^{(\M)})^{\langle \bullet \rangle} = ((-\EX)_{[0]}, \cdots, (-\EX)_{[\M]}).
\end{align}
Moreover, 
 the equality  
\begin{align} \label{e200}
\mu (\DMO_{\oslash, \EX}^{(\M)})  \circ \mr{sw} = \mu (\DMO_{\oslash, -\EX}^{(\M)}) \ \left(\Longleftrightarrow \mr{sw}^\bullet_{\mcO_S} (\mu (\DMO_{\oslash, \EX}^{(\M)})^{\langle \bullet \rangle}) 
= \mu (\DMO_{\oslash, -\EX}^{(\M)})^{\langle \bullet \rangle} \right)
\end{align}
holds.
\epr
\begin{proof}
Since the first assertion follows from the definition of $\DMO_{\oslash, \EX}^{(\M)}$, we only prove the second assertion.
To this end, it suffices to consider the case of $\EX \neq 0$.
Let us write $c := - \EX$. 
Since $c \neq 0$,  the nonnegative integer $\IN_0 := \mr{min}\left\{ \IN \, | \,  c_{[\IN]} \neq 0\right\}$ is well-defined.
For each  $\IN \leq  \IN_0$, the following equalities  hold:
\begin{align} \label{TTT1234}
-c_{[\IN]} - 1 + \prod_{b =0}^{\IN -1} (1-c_{[b]}^{p-1}) = -c_{[\IN]} - 1 + \prod_{b =0}^{\IN -1} (1-0) = -c_{[\IN]}  -1 + 1 = d_{[\IN]}.
\end{align}
On the other hand,  if $\IN > \IN_0$, then we have 
\begin{align} \label{TTT1235}
-c_{[\IN]} - 1 + \prod_{b =0}^{\IN -1} (1-c_{[b]}^{p-1})   
 =  -c_{[\IN]} - 1 = d_{[\IN]},
\end{align}
where the first equality follows from  $1- c_{[\IN_0]}^{p-1} = 0$.
By \eqref{TTT1234}, \eqref{TTT1235}, and the first assertion,
the following sequence of equalities holds:
\begin{align}
\mr{sw}^\bullet_{\mcO_S} \left(\mu (\DMO_{\oslash, \EX}^{(\M)})^{\langle \bullet \rangle}\right) 
& = \mr{sw}^\bullet_{\mcO_S} \left(c_{[0]}, \cdots, c_{[\M]}\right) \\
& = (-c_{[\IN]} - 1 + \prod_{b =0}^{\IN-1} (1-c_{[b]}^{p-1}))_{\IN =0}^\M \notag \\
&= \left(d_{[\IN]}\right)_{\IN=0}^\M \notag \\
& = \mu (\DMO_{\oslash, -\EX}^{(\M)})^{\langle \bullet \rangle}. \notag
\end{align}
This completes the proof of the assertion.
\end{proof}
\SSP

By Propositions \ref{P27}, (ii), and  \ref{L090},
there exists a unique $\mcD_{\otimes}^{(\M)}$-module structure 
\begin{align}
\DMO_{\otimes, \EX}^{(\M)} :  {^L}\mcD_{\otimes}^{(\M)} \migi \mcE nd_{\mcO_S} (\mcO_{\otimes})
\end{align}
on $\mcO_\otimes$ 
with $\iota_1^*(\DMO_{\otimes, \EX}^{(\M)}) = \DMO_{\oslash, \EX}^{(\M)}$ and  $\iota_2^*(\DMO_{\otimes , \EX}^{(\M)}) = \DMO_{\oslash, -\EX}^{(\M)}$.
It is verified that $\DMO_{\otimes, \EX}^{(\M)}$ has vanishing $p^{\M +1}$-curvature.
We shall write
\begin{align} \label{WW2}
\msO_{\otimes, \EX}^{(\M)}
  := (\msO_\otimes, \DMO_{\otimes, \EX}^{(\M)}).
\end{align}

\SSP
\bpr \label{Ldd093}
We shall write $c := -\EX$.
Also, let us fix  $\IN \in \{ 0, \cdots, \M\}$.
\begin{itemize}
\item[(i)]
We have 
\begin{align}
\mcS ol  (\DMO_{\otimes, \EX}^{(\M)\Rightarrow (\IN)})
  = \begin{cases}\mcO_\otimes^{(\IN +1)} & \text{if $\widetilde{\EX}_{[0, \IN]} =0$;} \\  t_1^{\widetilde{\EX}_{[0, \IN]}} \cdot  \iota_{1*}(\mcO_{\oslash}^{(\IN+1)}) \oplus 
 t_2^{\widetilde{c}_{[0, \IN]}} \cdot \iota_{2*}(\mcO_{\oslash}^{(\IN+1)}) 
 & \text{if otherwise.} \end{cases}
\end{align}
\item[(ii)]
Suppose that $a = 0$ or $\widetilde{\EX}_{[0, \IN-1]} = 0$.
 Then, 
 under the identification $\mcO_\otimes^{(a)} = \mcO_\otimes$ given by $h (t^{p^a}_1, t^{p^a}_2) \leftrightarrow h (t_1, t_2)$  for each $h (t_1, t_2) \in R [\![t_1, t_2]\!]/(t_1 t_2)$,
 the following equality of $S^\mr{log}$-connections  holds:
 \begin{align}
 (\DMO_{\otimes, \EX}^{(\M)})^{[\IN]} = \DMO^{(0)}_{\otimes, d_{[\IN]}}.
 \end{align}
\item[(iii)]
Suppose that $\IN \geq 1$ and 
 $\widetilde{\EX}_{[0, \IN-1]} \neq 0$.
Then, 
 under the identification 
 \begin{align}
\left(\mcS ol  (\DMO_{\otimes, \EX}^{(\M)\Rightarrow (\IN -1)}) = \right) t_1^{\widetilde{\EX}_{[0, \IN-1]}} \cdot  \iota_{1*}(\mcO_{\oslash}^{(\IN)}) \oplus 
 t_2^{\widetilde{c}_{[0, \IN-1]}} \cdot \iota_{2*}(\mcO_{\oslash}^{(\IN)}) = \iota_{1 *}(\mcO_\oslash) \oplus \iota_{2*} (\mcO_\oslash)
 \end{align}
  given by 
 $(t_1^{\widetilde{\EX}_{[0, \IN-1]}} \cdot h_1 (t_1^{p^{\IN}}), t_2^{\widetilde{c}_{[0, \IN-1]}} \cdot h_2 (t_2^{p^{\IN}})) \leftrightarrow (h_1 (t), h_2 (t))$  for  each pair $(h_1 (t), h_2 (t)) \in \mcO_S[\![t]\!]^{\times 2}$,
 the following equality of $S^\mr{log}$-connections holds:
\begin{align}
(\DMO_{\otimes, \EX}^{(\M)})^{[\IN]} = \iota_{1*}(\DMO_{\oslash, \EX_{[\IN]}}^{(0)}) \oplus \iota_{2*}(\DMO_{\oslash, c_{[\IN]}}^{(0)}).
\end{align}
\end{itemize}
\epr
\begin{proof}
The assertions follow from Proposition \ref{L093}.
\end{proof}
\SSP

Also, the following assertion can be proved immediately.

\SSP
\bpr \label{YY30}
Let $d$ and  $c$ be elements of $\mbZ/p^{\M +1} \mbZ$.
Also, let $\odot \in \{ \oslash, \otimes \}$.
\begin{itemize}
\item[(i)]
The canonical isomorphism $\mcO_{\odot} \otimes_{\mcO_\odot} \mcO_\odot \isom \mcO_\odot$ defines  an isomorphism of $\mcD_\odot^{(\M)}$-modules 
\begin{align} \label{YY31}
\msO_{\odot, d}^{(\M)} \otimes \msO_{\odot, c}^{(\M)} \isom \msO_{\odot, d+c}^{(\M)}.
\end{align}
In particular, we have
$\msO_{\odot, d}^{(\M)\vee} = \msO_{\odot, -d}^{(\M)}$ (cf.  ~\cite[Corollaire 2.6.1, (ii)]{Mon} for the definition of dual $(-)^\vee$).
\item[(ii)]
The  following equality holds:
\begin{align} \label{YY32}
\mr{Hom}(\msO_{\odot, d}^{(\M)}, \msO_{\odot, c}^{(\M)})
= \left\{ \mr{mult}_{s} \, \Big| \, s \in H^0 (U_\odot, \mcS ol (\DMO_{\odot, c-\EX}^{(\M)}))\right\},
\end{align}
where  $\mr{mult}_{s}$ denotes the endomorphism of $\mcO_\odot$ given by multiplication by $s$.
In particular, there exists a surjective morphism 
$\msO_{\odot, d}^{(\M)} \migisurj \msO_{\odot, c}^{(\M)}$
 if and only if the equality $d = c$ holds.
\item[(iii)]
Let $\M'$ be a nonnegative integer $\leq \M$.
Then, the following equality holds:
\begin{align} \label{YY69}
\DMO_{\odot, d}^{(\M)\Rightarrow (\M')} = \DMO_{\odot, d_{[0, \M']}}^{(\M')}.
\end{align}
\end{itemize}
\epr

\LSP
\subsection{Local description of  $p^{\M +1}$-flat $\mcD^{(\M)}$-modules} \label{SS203}

In this subsection, we show that a $p^{\M +1}$-flat  $\mcD_{\odot}^{(\M)}$-module (where $\odot \in \{ \oslash, \otimes \}$) can be described as the direct sum of 
$\msO_{\odot, d}^{(\M)}$'s
 for  some elements $d \in \mbZ/p^{\M +1}\mbZ$.
This assertion generalizes ~\cite[Proposition 1.1.12]{Kin1}, in which $\mcD^{(\M)}$-modules  on the formal disc $\mr{Spec}(k[\![t]\!])$ over an algebraically closed field  $k$ were discussed.
See also ~\cite[Corollary 2.10]{O2} for the case of $\M = 0$ and ~\cite[Theorem 3.3]{Gie} for the case of so-called  $F$-divided (or, stratified) bundles, i.e., the case of $\M = \infty$.

\SSP
\begin{prdef}\label{P022}
Suppose that $S = \mr{Spec}(R)$ for
a local ring  $(R, \mfm)$ over $\mbF_p$ such that the residue field $k := R/\mfm$ is algebraically closed.
Let $\odot \in \{ \oslash, \otimes \}$, and let
 $(\mcF, \DMO^{(\M)})$ be a
  $p^{\M +1}$-flat  $\mcD_{\odot}^{(\M)}$-module  such that $\mcF$ is a (locally) free $\mcO_\odot$-module of 
  rank $n >0$.
Then, there exists an isomorphism of $\mcD_{\odot}^{(\M)}$-modules
\begin{align} \label{YY19}
 \bigoplus_{i=1}^n  \msO_{\odot, d_i}^{(\M)}
 \isom  (\mcF, \DMO^{(\M)})
\end{align}
for some $d_1, \cdots, d_n \in \mbZ/p^{\M +1}\mbZ$.
(This implies that, for each $a=0, \cdots, \M$, the $a$-th monodromy operator $\mu (\DMO^{(\M)})^{\langle p^a \rangle}$ of $\DMO^{(\M)}$
can be transformed, after choosing a suitable trivialization $\sigma_\odot^*(\mcF) \isom \mcO_S^{\oplus n}$ of $\sigma_\odot^*(\mcF)$,  into 
 the diagonal matrix with diagonal entries $(-d)_{1[a]}, \cdots, (-d)_{n[a]}$.)
Moreover, the resulting  multiset 
\begin{align} \label{YY23}
e (\DMO^{(\M)}) := [d_1, \cdots, d_n]
\end{align}
 depends only on the isomorphism class of $(\mcF, \nabla^{(\M)})$.
We shall refer to 
$e (\DMO^{(\M)})$
 as the {\bf exponent} of $(\mcF, \DMO^{(\M)})$.
\end{prdef}
\begin{proof}
First, let us consider the case where $\odot = \oslash$.
Hereinafter, we shall use subscripted ``$k$" to denote the result of 
taking a completed tensor product $(-) \widehat{\otimes}_R k$.
Since the $\mcD_{\oslash, k}^{(\M)}$-module $(\mcF_k, \nabla_k^{(\M)})$ has vanishing $p^{\M +1}$-curvature, 
it follows from ~\cite[Proposition 1.1.12]{Kin1} that
there exist a multiset $[\EX_1, \cdots, \EX_n]$ of elements in $\mbZ/p^{\M +1}\mbZ$  and 
 an isomorphism  of $\mcD_{\oslash, k}^{(\M)}$-modules
\begin{align}
\xi : \left(\bigoplus_{i=1}^n (\mcO_{\oslash, k}, \DMO_{\oslash, d_i, k}^{(\M)}) \cong\right)\bigoplus_{i=1}^n (t^{-\widetilde{\EX}_i} \cdot \mcO_{\oslash, k}, \nabla_{\oslash, 0, k}^{(\M)}) \isom (\mcF_k, \DMO_k^{(\M)}),
\end{align}
where, for each $\widetilde{\EX} \in \mbZ_{\geq 0}$, we abuse notation by writing  $\nabla_{\oslash, 0, k}^{(\M)}$ for  the $\mcD_{\oslash, k}^{(\M)}$-module structure 
 on the  line bundle $t^{-\widetilde{\EX}} \cdot \mcO_{\oslash, k} \left(\supseteq  \mcO_{\oslash, k}\right)$ extending
 $(\mcO_{\oslash, k},\nabla_{\oslash, 0, k}^{(\M)})$.

  Now, let us choose $j \in \{1, \cdots, n \}$.
  Denote by $\DMO^{(\M)}_j$ the $\mcD_{\oslash}^{(\M)}$-module structure on $t^{\widetilde{\EX}_j} \cdot \mcF \left(\subseteq \mcF \right)$ obtained by restricting $\DMO^{(\M)}$.
Then,  $\xi$ restricts to an isomorphism
\begin{align}
\xi_j : \bigoplus_{i=1}^n (t^{-\widetilde{\EX}_i+ \widetilde{d}_j} \cdot \mcO_{\oslash, k}, \nabla_{\oslash, 0, k}^{(\M)}) \isom (t^{\widetilde{\EX}_j} \cdot \mcF_k, \DMO_{j, k}^{(\M)}).
\end{align}
  If  $e_j$ denotes the image of $1 \in \mcO_{\oslash, k}$ via the inclusion into the  $j$-th  factor $\mcO_{\oslash, k} \migiincl \mcO_{\oslash, k}^{\oplus n}$,
  then  it  is  a horizontal section in  the domain of $\xi_j$.
  In particular, we have $\xi_j (e_j) \in \mcS ol (\DMO_{j, k}^{(\M)})$.
  Since the natural morphism $\mcS ol (\DMO_{j}^{(\M)}) \migi \mcS ol (\DMO_{j, k}^{(\M)})$ is surjective  by Proposition \ref{PPer4}, (ii), 
  we can find a section $v_j \in \mcS ol (\DMO_{j}^{(\M)})$ mapped to  $\xi_j (e_j)$ via this surjection.
  The section $v_j$ determines a morphism  of
$\mcD_{\oslash}^{(\M)}$-modules $(\mcO_{\oslash}, \nabla_{\oslash, 0}^{(\M)}) \migi (t^{\widetilde{\EX}_j} \cdot\mcF, \nabla_j^{(\M)})$;
it  can be extended uniquely  to a morphism $\zeta_j : (t^{-\widetilde{\EX}_j}\cdot \mcO_{\oslash}, \nabla_{\oslash, 0}^{(\M)}) \migi (\mcF, \nabla^{(\M)})$.
Thus, we obtain the composite  
\begin{align}
\zeta : \bigoplus_{i=1}^n (\mcO_{\oslash}, \nabla_{\oslash, \EX_i}^{(\M)}) \xrightarrow{\sim} \bigoplus_{i=1}^n (t^{-\widetilde{\EX}_i} \cdot \mcO_{\oslash}, \nabla_{\oslash, 0}^{(\M)}) \xrightarrow{\bigoplus_i \zeta_i}  (\mcF, \nabla^{(\M)}),
\end{align}
where the first arrow is  the direct sum of the natural isomorphisms
$(\mcO_{\oslash}, \nabla_{\oslash, \EX_i}^{(\M)}) \isom (t^{-\widetilde{\EX}_i} \cdot \mcO_{\oslash}, \nabla_{\oslash, 0}^{(\M)})$.
Note that the reduction modulo $\mfm$ of $\bigoplus_i \zeta_i$ coincides with the isomorphism $\xi$.
Hence, $\zeta$ turns out to be an isomorphism by Nakayama's lemma.
This  proves  the first assertion for $\odot = \oslash$.
Moreover, the second assertion  follows immediately from Proposition \ref{YY30}, (ii), in the case of $\odot = \oslash$.

Next, we shall consider the case of $\odot = \otimes$.
It follows from the above discussion that
there exists an isomorphism of $\mcD_\oslash^{(\M)}$-modules
\begin{align}
\zeta_{\oslash, 1} : \left(\bigoplus_{i=1}^n (\iota_1^*(\mcO_\otimes), \iota_1^*(\DMO_{\otimes, \EX_i}^{(\M)})) =\right) \bigoplus_{i=1}^n (\mcO_{\oslash},  \nabla^{(\M)}_{\oslash, \EX_{i}}) \isom (\iota_1^*(\mcF), \iota_1^*(\DMO^{(\M)}))
\end{align}
for some $\EX_{1}, \cdots, \EX_{n} \in \mbZ/p^{\M +1}\mbZ$.
By  Propositions \ref{P27}, (i), and \ref{L090}, 
the exponent of $(\iota_2^*(\mcF), \iota_2^*(\nabla^{(\M)}))$ coincides  with $[-\EX_1, \cdots, -\EX_n]$.
Hence,  we can find an isomorphism of $\mcD_{\oslash}^{(\M)}$-modules
\begin{align}
\zeta_{\oslash, 2} : \bigoplus_{i=1}^n (\iota_2^*(\mcO_\otimes),  \iota_2^*(\DMO_{\otimes, \EX_i}^{(\M)})) \isom (\iota_2^*(\mcF), \iota_2^*(\nabla^{(\M)})).
\end{align}
The  automorphism $\sigma_\oslash^* (\zeta_{\oslash, 2})^{-1} \circ  \sigma_\oslash^*(\zeta_{\oslash, 1})$ defines  an 
$n \times n$ matrix $A := (a_{ij})_{1 \leq i, j \leq n} \in \mr{GL}_n(R)$ (under the natural identification $H^0 (S, \sigma_\otimes^* (\mcO_\otimes)) = R$).
We shall set $D := (D_{[0]}, \cdots, D_{[\M]})$, where each $D_{[\IN]}$ ($\IN=0, \cdots, \M$) denotes   the diagonal matrix with diagonal entries  $(-\EX)_{1[\IN]}, \cdots, (-\EX)_{n [\IN]}$.
 By regarding this matrix as an element of $\mr{End}_{\mcO_S}(\mcO_S^{\oplus n})^{\oplus (\M +1)}$, 
 we can identify $D$ (resp., $\mr{sw}^\bullet_{\mcO_S} (D)$) with the monodromy operator $\mu (\bigoplus_{i=1}^n \iota_1^*(\DMO_{\otimes, \EX_i}^{(\M)}))^{\langle \bullet \rangle}$ (resp., $\mu (\bigoplus_{i=1}^n \iota_2^*(\DMO_{\otimes, \EX_i}^{(\M)}))^{\langle \bullet \rangle}$) of $\bigoplus_{i=1}^n \iota_1^*(\DMO_{\otimes, \EX_i}^{(\M)})$ (resp., $\bigoplus_{i=1}^n \iota_2^*(\DMO_{\otimes, \EX_i}^{(\M)})$).
Since both $\zeta_{\oslash, 1}$ and $\zeta_{\oslash, 2}$ preserve the $\mcD_\oslash^{(\M)}$-module structure, we obtain the following equalities of elements in $\mr{End}_{\mcO_S}(\mcO_S^{\oplus n})^{\oplus (\M +1)}$:
\begin{align}
\mu \circ \sigma_\oslash^* (\zeta_{\oslash, 1}) = \sigma_\oslash^* (\zeta_{\oslash, 1}) \circ D,
 \hspace{5mm}
 \mr{sw}_{\mcO_S}^\bullet (\mu) \circ \sigma_\oslash^* (\zeta_{\oslash, 2}) = \sigma_\oslash^* (\zeta_{\oslash, 2}) \circ \mr{sw}_{\mcO_S}^\bullet (D),
\end{align}
where $\mu := \mu (\iota_1^*(\DMO^{(\M)}))^{\langle \bullet \rangle}$.
These equalities imply
\begin{align}
  \sigma_\oslash^* (\zeta_{\oslash, 2})  \circ D \circ   \sigma_\oslash^* (\zeta_{\oslash, 2})^{-1} & =  \sigma_\oslash^* (\zeta_{\oslash, 2})  \circ  \mr{sw}_{\mcO_S}^\bullet ( \mr{sw}_{\mcO_S}^\bullet (D)) \circ   \sigma_\oslash^* (\zeta_{\oslash, 2})^{-1} \\
&
=  \mr{sw}_{\mcO_S}^\bullet( \sigma_\oslash^* (\zeta_{\oslash, 2}) \circ \mr{sw}_{\mcO_S}^\bullet (D) \circ \sigma_\oslash^* (\zeta_{\oslash, 2})^{-1})
\notag \\
& =   \mr{sw}_{\mcO_S}^\bullet(\mr{sw}_{\mcO_S}^\bullet (\mu)) \notag \\
& = \mu \notag \\
& = \sigma_\oslash^* (\zeta_{\oslash, 1}) \circ D \circ\sigma_\oslash^* (\zeta_{\oslash, 1})^{-1}. \notag
\end{align}
 It follows that $D A = A D$, i.e., $(D_{[0]}A, \cdots, D_{[\M]}A) = (A D_{[0]}, \cdots, A D_{[\M]})$.
Hence, for each pair $(i, j)$ with $d_i \neq d_j$,
we have $a_{ij} = a_{ji} =  0$.
By Proposition \ref{YY30}, (ii),  
 the matrix $A$  defines an automorphism $\alpha$ of the $\mcD_{\oslash}^{(\M)}$-module 
$\bigoplus_{i=1}^n (\iota_2^*(\mcO_\otimes),  \iota_2^*(\DMO_{\otimes, \EX_i}^{(\M)}))$ when regarded  as an element of $\mr{GL}_n (R[\![t]\!])$ via the natural inclusion $R \migiincl R[\![t]\!]$.
The equality $\sigma_{\oslash}^* (\zeta_{\oslash, 1}) = \sigma_{\oslash}^* (\zeta_{\oslash, 2}\circ\alpha)$ holds, so the two morphisms of $\mcD_{\oslash}^{(\M)}$-modules 
$\zeta_{\oslash, 1}$,  $\zeta_{\oslash, 2} \circ \alpha$ may be glued together (by using \eqref{e459}) to obtain an isomorphism 
\begin{align}
\bigoplus_{i=1}^n (\mcO_\otimes, \nabla_{\otimes, d_i}^{(\M)}) \isom (\mcF, \DMO^{(\M)}).
\end{align}
This completes the proof of  the first assertion for $\odot = \otimes$.
Moreover, the second assertion can be proved by applying Proposition \ref{YY30}, (ii),  in the case of  $\odot = \otimes$.
Thus, we have finished the proof of this proposition.
 \end{proof}
\SSP

The following assertion is  
 a corollary of the above proposition.
 
\SSP
\bco \label{PP4}
Let  $\IN \in \{ 0, \cdots, \M \}$, and let $(\mcF, \DMO)$ be a $p^{\M +1}$-flat  $\mcD_{\otimes}^{(\M)}$-module   such that $\mcF$ is locally free   of rank $n > 0$.
Then, the  natural $\mcO_{\otimes}^{(\IN)}$-linear morphism 
$(\mcF^{\vee})^{[\IN]} \migi (\mcF^{[\IN]})^\vee$
 restricts to an $\mcO_{\otimes}^{(\IN+1)}$-linear isomorphism
\begin{align}
\left((\mcF^{\vee})^{[\IN+1]} =  \right)\mr{Ker}((\DMO^\vee)^{[\IN]}) \isom \mr{Ker} ((\DMO^{[\IN]})^\vee).
\end{align}
\eco
\begin{proof}
It suffices to prove the assertion under the assumption  in Proposition-Definition \ref{P022}.
 Then,  by that proposition, the problem is reduced to the case where $(\mcF, \DMO)  = (\mcO_\otimes, \DMO_{\otimes, d}^{(\M)})$ for an element  $d$ of $\mbZ /p^{\M +1} \mbZ$.
 
 If $\IN=0$ or $\widetilde{d}_{[0, \IN -1]} = 0$ ($\IN \geq 1$), then the assertion follows immediately from Proposition \ref{Ldd093}, (ii).
 Before  proving  the remaining case,
 let us make the following observation.  
 For each $u$, $v \in \mbF_p$, we shall denote by $\DMO_{\oslash, (u, v)}^{(0)}$ the 
 $\mcD_\otimes^{(0)}$-module structure
  on the $\mcO_\otimes$-module  $\iota_{1*}(\mcO_\oslash)\oplus \iota_{2*}(\mcO_\oslash)$ expressed as 
 $\DMO_{\oslash, (u, v)}^{(0)} := \iota_{1*}(\DMO_{\oslash, u}^{(0)}) \oplus  \iota_{2*}(\DMO_{\oslash, v}^{(0)})$.
Observe that
\begin{align}
  \mr{Hom}_{\mcO_\otimes}\left(\iota_{1*}(\mcO_\oslash)\oplus \iota_{2*}(\mcO_\oslash), \mcO_\otimes\right) 
& \left(=  \mr{Hom}_{\mcO_\otimes}\left(\iota_{1*}(\mcO_\oslash), \mcO_\otimes\right) 
 \oplus  \mr{Hom}_{\mcO_\otimes}\left(\iota_{2*}(\mcO_\oslash), \mcO_\otimes\right) \right) \\
& =   \left\{ (\mr{mult}_{\alpha}, \mr{mult}_{\beta}) \, \Big| \, \alpha \in t_1 \cdot R[\![t_1]\!],
\beta \in t_2 \cdot R[\![t_2]\!]
\right\},  \notag
\end{align}
where 
  $\mr{mult}_{(-)}$ denotes the morphism given by multiplication by $(-)$.
This equality restricts to an equality
 \begin{align} \label{YY40}
 &  \ \ \ \  \mr{Hom}\left((\iota_{1*}(\mcO_\oslash)\oplus \iota_{2*}(\mcO_\oslash), \DMO_{\oslash, (u, v)}^{(0)}), (\mcO_\otimes, \nabla_{\otimes, 0}^{(0)})\right)  \left(=H^0 (U_\otimes, \mcS ol ((\DMO_{\oslash, (u, v)}^{(0)})^\vee))  \right) \\
& =      \left\{ (\mr{mult}_{\alpha}, \mr{mult}_{\beta}) \, \Big| \, \alpha \in t_1^{p-\widetilde{u}} \cdot R[\![t_1^p]\!],
\beta \in t_2^{p-\widetilde{v}} \cdot R[\![t_2^p]\!]
\right\}.
\notag
 \end{align}
 
 Now, let us go back to the proof of the remaining case, i.e., the case where $\IN \geq 1$ and $\widetilde{d}_{[0, \IN-1]} \neq 0$.
 Write $c := - d$.
 By taking account of \eqref{YY40} and Proposition \ref{Ldd093}, (iii),
 we obtain the following sequence of equalities:
 \begin{align}
 &  \ \ \ \ H^0 (U_\otimes, \mr{Ker} (((\DMO_{\otimes, d}^{(\M)})^{[\IN]})^\vee)) \\
&  = 
 \mr{Hom}\left(
 (t_1^{\widetilde{d}_{[0, \IN-1]}} \cdot \iota_{1*}(\mcO_\oslash^{(\IN)}) \oplus t_2^{\widetilde{c}_{[0, \IN-1]}} \cdot \iota_{2*}(\mcO_\oslash^{(\IN)}), (\DMO_{\otimes, d}^{(\M)})^{[\IN]}), (\mcO_{\otimes}^{(\IN)}, \DMO_{U_\otimes^{(\IN)}, \mr{triv}}^{(0)})\right) \notag\\
  & = \left\{ (\mr{mult}_\alpha, \mr{mult}_\beta) \, \bigg| \, \alpha \in  t_1^{-\widetilde{d}_{[0, \IN-1]}+ p^{\IN}(p-\widetilde{d}_{[\IN]})} \cdot R[\![t_1^{p^{\IN+1}}]\!], \beta \in   t_2^{-\widetilde{c}_{[0, \IN-1]}+ p^{\IN}(p-\widetilde{c}_{[\IN]})} \cdot R[\![t_2^{p^{\IN+1}}]\!] \right\} \notag \\
 & = \left\{ (\mr{mult}_\alpha, \mr{mult}_\beta) \, \bigg| \, \alpha \in t_1^{\widetilde{c}_{[0, \IN]}} \cdot R[\![t_1^{p^{\IN+1}}]\!], \beta \in   t_2^{\widetilde{d}_{[0, \IN]}} \cdot R[\![t_2^{p^{\IN+1}}]\!] \right\} \notag \\
 & = H^0 (U_\otimes, \mr{Ker}((\DMO_{\otimes, -d}^{(\M)})^{[\IN]})) \notag \\
 & = H^0 (U_\otimes, \mr{Ker}  (((\DMO_{\otimes, d}^{(\M)})^\vee)^{[\IN]})). \notag
 \end{align}
 This completes the proof of the assertion.
\end{proof}

\LSP
\subsection{Local description of dormant flat bundles} \label{SS034}

In the rest of this section,  we suppose that   $S$ is  flat  over $\mbZ/p^{\M +1}\mbZ$.
Note that  the notion of a {\it dormant} flat bundle on $U_\oslash^\mr{log}/S^\mr{log}$, as well as the functor ${^{\Diag}}\!\!(-)$,  can be defined in the same  manner as the case where the underlying space is a log curve. 
Just as in the discussion of Section \ref{S200}, 
we will denote by a subscript  $\M' $ (for each $\M' \leq \M$)  the result of reducing an object over $\mbZ/p^{\M +1}\mbZ$ modulo $p^{\M' +1}$.
 Also, we use the notation ``$\DMO_{\oslash, d, 0}^{(\M)}$" to denote the $\mcD_{\oslash, 0}^{(\M)}$-module structure $\DMO_{\oslash, \EX}^{(\M)}$  on $\mcO_{\oslash, 0}$ introduced in \eqref{YY52}.
 In particular, we obtain  $\msO_{\oslash, \EX, 0}^{(\M)}: = (\mcO_{\oslash, 0}, \DMO_{\oslash, d, 0}^{(\M)})$.

Let $(\mcF, \nabla)$ be a flat module  on $U_\oslash^\mr{log}/S$.
Since $\nabla$ induces an $\mcO_S$-linear morphism $\sigma_\oslash^*(\nabla) : \sigma_\oslash^* (\mcF) \migi \sigma_{\oslash}^* (\Omega_{U_\oslash^\mr{log}/S^\mr{log}} \otimes \mcF)$,
we obtain the $\mcO_{S}$-linear composite
\begin{align} \label{YY89}
\mu (\nabla) : \sigma_{\oslash}^* (\mcF) & \xrightarrow{\sigma_\oslash^*(\nabla)} \sigma_{\oslash}^* (\Omega_{U_\oslash^\mr{log}/S^\mr{log}} \otimes \mcF)  \\
& \isom \sigma_\oslash^* (\Omega_{U_\oslash^\mr{log}/S^\mr{log}}) \otimes \sigma_\oslash^* (\mcF) \notag \\
& \xrightarrow{\mr{Res}_{\sigma_\oslash} \otimes \mr{id}} \left(\mcO_S \otimes \sigma_\oslash^*(\mcF) = \right) \sigma_\oslash^*(\mcF), \notag
\end{align} 
where $\mr{Res}_{\sigma_\oslash}$ denotes the residue map $\sigma_\oslash^* (\Omega_{U_\oslash^\mr{log}/S^\mr{log}}) \isom \mcO_S$.

\SSP
\bde \label{YY88}
We shall refer to $\mu (\nabla) \left(\in \mr{End}_{\mcO_S} (\sigma_\oslash^* (\mcF)) \right)$ as the {\bf monodromy operator} of $\nabla$.
(This is essentially the same as  the notion of ($0$-th) monodromy operator of the corresponding  $\mcD_\oslash^{(0)}$-module, in the sense of Definition \ref{D98}.)
\ede
\SSP

For each  element $\EX$ of $\mbZ/p^{\M+1}\mbZ$,
we denote by $\DMO_{\oslash, \EX} : \mcO_\oslash \migi \Omega_{U_\oslash^\mr{log}/S^\mr{log}}$ the $S^\mr{log}$-connection on $\mcO_\oslash$ given by 
 \begin{align} \label{WW101}
\DMO_{\oslash, \EX} (v) := D (v) - \EX \cdot  v \cdot \frac{dt}{t}
\end{align}
for any local section $v \in \mcO_\oslash$, where $D$ denotes the universal continuous logarithmic derivation $\mcO_\oslash \migi \Omega_{U_\oslash^\mr{log}/S^\mr{log}}$.
Thus, we obtain a  flat line bundle
\begin{align} \label{WW102}
\msO_{\oslash, \EX} := (\mcO_{\oslash}, \DMO_{\oslash, \EX})
\end{align}
on $U^\mr{log}_\oslash/S^\mr{log}$.
Note that $\msO_{\oslash, \EX}$ is 
  isomorphic to the unique extension of 
 $\msO_{\oslash, 0}$
  to $t^{-\widetilde{\EX}} \cdot \mcO_{\oslash} \left(\supseteq \mcO_\oslash \right)$.
Also, the equality  $\mu (\nabla_{\oslash, \EX}) = -\EX$ holds.

\SSP
\bpr \label{P72e}
Let $\EX$ be an element of $\mbZ/p^{\M+1}\mbZ$.
\begin{itemize}
\item[(i)]
The flat bundle $\msO_{\oslash, \EX}$ 
is dormant
 and satisfies 
\begin{align}\label{YY67}
{^{\Diag}}\!\!\msO_{\oslash, \EX} = \msO_{\oslash, \EX, 0}^{(\M)}.
\end{align}
Moreover,
  $\msO_{\oslash, \EX}$ is the unique (up to isomorphism)
diagonal lifting of $\msO_{\oslash, \EX, 0}^{(\M)}$.
\item[(ii)]
The following equality holds:
\begin{align} \label{YY75}
H^0 (U_\oslash, \mr{Ker} (\nabla_{\oslash, \EX})) =  \sum_{\IN=0}^{\M} p^{\M  -\IN} t^{\widetilde{d}_{[0, a]}}  R [\![t^{p^{\IN+1}}]\!]  \left(\subseteq R [\![t]\!]  \right). 
\end{align}
\end{itemize}
\epr
\begin{proof}
We shall prove the first  assertion of (i) by induction on $\M$.
There is  nothing to prove for the base step, i.e., the case of $\M = 0$.
To consider the induction step, we suppose  that the assertion with $\M$ replaced by $\M -1$ ($\M \geq 1$) has been proved.
We shall set $\msF := (\mcO_\oslash, \DMO_{\oslash, \EX})$.
By the induction hypothesis, 
$(\mcO_{\oslash, \M-1}, \DMO_{\oslash,  \EX_{[0, \M-1]}})$ is a diagonal lifting of $(\mcO_{\oslash, 0}, \DMO_{\oslash, \EX, 0}^{(\M)\Rightarrow (\M-1)})$  (cf. Proposition \ref{YY30}, (iii)).
It follows from Proposition \ref{L093} that
\begin{align}
H^0 (U_{\oslash, 0}^{(\M)}, \mcV_\msF) \left(= H^0 (U_{\oslash, 0}^{(\M)}, \mcS ol (\DMO_{\oslash, \EX, 0}^{(\M) \Rightarrow (\M -1)})) \right)= t^{\widetilde{d}_{[0, \M -1]}} \cdot R_0 [\![t^{p^\M}]\!].
\end{align}
The global section   $t^{\widetilde{d}_{[0, \M -1]}} \in R [\![t]\!]$ is horizontal modulo $p^\M$ with respect to $\DMO_{\oslash, \EX}$ and  satisfies
\begin{align}
\DMO_{\oslash, \EX} (t^{\widetilde{d}_{[0, \M -1]}}) =  - p^\M \cdot (d_{[\M]} \cdot t^{\widetilde{d}_{[0, \M -1]}}) \cdot \frac{dt}{t}.
\end{align}
Hence, under the identification 
\begin{align} \label{YY71}
t^{\widetilde{d}_{[0, \M -1]}} \cdot R_0  [\![t^{p^\M}]\!] = R_0  [\![t]\!]
\end{align}
 given by $t^{\widetilde{d}_{[0, \M -1]}} \cdot h(t^{p^\M})\leftrightarrow  h(t)$ (for any $h(t) \in R_0  [\![t]\!]$), 
the flat bundle  $(\mcV_\msF, \DMO_{\msF})$ coincides with  $(\mcO_{\oslash, 0}, \DMO_{\oslash, d_{[\M]}})$.
 By ~\cite[Proposition 4.2.1, (i) and (iv)]{Wak9},
  $\DMO_{\oslash, \EX, 0}^{(\M)}$ is the unique  $\mcD_{\oslash, 0}^{(\M)}$-module structure $\DMO^{(\M)}$ on $\mcO_{\oslash, 0}$ satisfying  $\DMO^{(\M)\Rightarrow (\M-1)} = \DMO_{\oslash, d, 0}^{(\M) \Rightarrow (\M-1)}$ and $(\mcV_\msF, \DMO_{\msF}) =  (\mcO_{\oslash, 0}^{[\M]}, (\DMO^{(\M)})^{[\M]})$ (under the identification \eqref{YY71}).
  In particular, 
  $(\mcO_\oslash, \DMO_{\oslash, \EX})$
  turns out to be a diagonal lifting of  $(\mcO_{\oslash, 0}, \DMO_{\oslash, \EX, 0}^{(\M)})$.
  This completes the proof of the first assertion of (i).
  
  To consider the second assertion,  we suppose that there exists a diagonal lifting $\msN := (\mcN, \nabla_\mcN)$ of $\msO_{\oslash, d, 0}^{(\M)}$.
  After tensoring with $\msO_{\oslash, -d}$,  we may assume that $d = 0$.
  By Corollary \ref{YY79} and the affineness of $S$, the unit section $1 \in H^0 (U_{\oslash, 0}^{(\M +1)}, \mcS ol (\nabla_{\oslash, 0, 0}^{(\M)})) =R_0 [\![t^{p^{\M +1}}]\!]$ lifts to a global section  $\widetilde{1}$ of $\mcS ol (\nabla_\mcN)$.
  The assignment $a \mapsto a \cdot \widetilde{1}$ determines a morphism $\msO_{\oslash, 0} \rightarrow \msN$, whose reduction modulo $p$ is an isomorphism.
 It follows from Nakayama's lemma that this morphism  turns out to be an isomorphism.
 Thus, we finish the proof of the required uniqueness. 
  
  Also, assertion (ii) follows immediately from the definition of $\DMO_{\oslash, \EX}$.
 \end{proof}
\SSP

The following assertion can be proved immediately, so we omit the proof.

\SSP
\bpr \label{YY44}
Let $d$ and $c$ be elements of $\mbZ/p^{\M +1}\mbZ$.
\begin{itemize}
\item[(i)]
The canonical isomorphism $\mcO_\oslash \otimes_{\mcO_\oslash} \mcO_\oslash \isom \mcO_\oslash$ defines  an isomorphism of flat bundles
\begin{align} \label{YY45}
\msO_{\oslash, d} \otimes \msO_{\oslash, c} \isom \msO_{\oslash, d +c}.
\end{align}
In particular, we have 
$\msO_{\oslash, \EX}^\vee = \msO_{\oslash, - \EX}$.
Moreover, 
the isomorphism \eqref{YY45} is compatible with \eqref{YY31}
    via \eqref{YY67}.
\item[(ii)]
We have
\begin{align} \label{YY84}
\mr{Hom}(\msO_{\oslash, \EX}, \msO_{\oslash, c})
= \left\{ \mr{mult}_{s} \, | \, s \in  H^0 (U_\oslash, \mr{Ker}(\DMO_{\oslash, c-d})) \right\},
\end{align}
where $\mr{mult}_{s}$  denotes the endomorphism of $\mcO_\oslash$ given by multiplication by $s$.
In particular, 
there exists a surjection morphism $\msO_{\oslash, \EX} \migisurj \msO_{\oslash, c}$ if and only if the equality $\EX = c$ holds.
Moreover, 
the equality \eqref{YY84} is compatible,
in a natural sense, with    \eqref{YY32}  via the functor ${^{\Diag}}\!\!(-)$ and the composite
\begin{align}
H^0 (U_\oslash, \mr{Ker}(\DMO_{\oslash, c-d})) \xrightarrow{\mr{mod} \, p}
H^0 (U_{\oslash, 0}, \mr{Ker} (\DMO_{\oslash, c-d})_0)
\xrightarrow{\eqref{YY82}} H^0 (U_{\oslash, 0}, \mcS ol (\DMO_{\oslash, c-d, 0}^{(\M)})).
\end{align}
\end{itemize}
\epr
\SSP

Also, by applying Proposition-Definition \ref{P022} (for  $\odot = \oslash$),
we obtain the following assertion.

\SSP
\begin{prdef} \label{P724}
Suppose that  $S_0 = \mr{Spec}(R_0)$ for a local ring $(R_0, \mfm_0)$ over $\mbF_p$ such that the residue field $k := R_0/\mfm_0$ is algebraically closed.
Let $\msF := (\mcF, \DMO)$ be a dormant flat bundle on $U_\oslash^\mr{log}/S^\mr{log}$ of rank $n >0$.
Then,  there exists an isomorphism  of flat bundles
\begin{align} \label{ee14}
\bigoplus_{i=1}^n 
\msO_{\oslash, \EX_i}
\isom \msF 
\end{align}
  for some $\EX_1, \cdots, \EX_n \in \mbZ/p^{\M +1}\mbZ$.
  (This implies that the monodromy operator $\mu (\nabla)$ of $\nabla$ can be transformed, after choosing a suitable trivialization $\sigma_\oslash^*(\mcF) \isom \mcO_S^{\oplus n}$ of $\sigma_\oslash^*(\mcF)$, into the diagonal matrix with diagonal entries $-\EX_1, \cdots, -\EX_n$.)
Moreover,  the resulting multiset 
\begin{align}
e (\nabla) := [\EX_1, \cdots, \EX_n]
\end{align}
  depends only on the isomorphism class of  $\msF$.
We shall refer  to $e (\nabla)$ as the {\bf exponent} of $\msF$.
\end{prdef}
\begin{proof}
According to Proposition-Definition \ref{P022},
there exists an isomorphism
\begin{align}
\zeta_0  : \bigoplus_{i=1}^n 
\msO_{\oslash, \EX_i, 0}^{(\M)} 
\isom {^{\Diag}}\!\!\msF
\end{align}
for some $\EX_1, \cdots, \EX_n \in \mbZ/p^{\M +1}\mbZ$.
Let us choose  $j \in \{ 1, \cdots, n \}$.
Since the functor  ${^{\Diag}}\!\!(-)$ commutes  with the formation of  tensor products, 
it follows from Propositions \ref{YY30}, (i),  and  \ref{P72e}, (i), that
$\zeta_0$ induces an isomorphism 
\begin{align}
\zeta_{0, j} : \bigoplus_{i=1}^n 
\msO_{\oslash, \EX_i-\EX_j, 0}^{(\M)}
 \isom {^{\Diag}}\!\! (\msF \otimes  \msO_{\oslash, -\EX_j})
  \left(=  {^{\Diag}}\!\!(\mcF, \nabla \otimes \nabla_{\oslash, -\EX_j})\right).
\end{align}
If $e_{j}$ denotes the image of $1 \in \mcO_{\oslash, 0}$ via  the inclusion into the $j$-th factor $\mcO_{\oslash, 0} \migiincl \mcO_{\oslash, 0}^{\oplus n}$,
then its image $\zeta_{0, j}(e_j)$  is a horizontal section in the codomain of $\zeta_{0, j}$.
By the property $(\beta)_\M$ asserted in Proposition \ref{L0ddd10}, (i) (or Corollary \ref{YY79}), we can find a lifting of  $\zeta_{0, j}(e_j)$ in $\mcF$ that is   horizontal  with respect to $\nabla \otimes \nabla_{\oslash, -\EX_j}$; this section
corresponds to  a morphism of flat bundles 
$\msO_{\oslash, 0}
 \migi (\mcF, \nabla \otimes \nabla_{\oslash, -\EX_j})$.
 The tensor  product of this morphism and the identity morphism of 
 $\msO_{\oslash, \EX_j}$ determines 
a morphism  
$\zeta_{j} : \msO_{\oslash, \EX_j} \migi \msF$.
Thus, we obtain a morphism of flat bundles
\begin{align}
\zeta := \bigoplus_{j=1}^n \zeta_j : \bigoplus_{j=1}^n\msO_{\oslash, \EX_j} \migi \msF.
\end{align}
Since the reduction of $\zeta$ modulo $p$ coincides with the isomorphism $\zeta_0$ (after reducing the level to $0$),
it follows from Nakayama's lemma that $\zeta$ is an isomorphism.
This completes the proof of the first assertion.
The second assertion can be proved immediately from Proposition \ref{YY44}, (ii). \end{proof}

\vspace{10mm}
\section{Dormant $\mr{PGL}_n^{(\N)}$-opers on pointed stable curves} \label{DDe345}
\SSP

This section deals with higher-level generalizations of (dormant) $\mr{PGL}_n$-opers, which we call (dormant) $\mr{PGL}_n^{(\N)}$-opers.
The goal of this section is to show that the moduli category of dormant $\mr{PGL}_n^{(\N)}$-opers may be represented by  a  Deligne-Mumford stack (cf. Theorems \ref{P99} and  \ref{T40}). 

Throughout  this section, we shall fix 
 an element  $(\CH, \N)$  of  $\mbZ_{\geq 0} \times \mbZ_{>0}$ with $\CH =0$ or $\N =1$.
Also, denote  by $B$ the Borel subgroup of $\mr{PGL}_n$  ($n \in \mbZ_{>0}$)
defined as the image,  via the natural projection $\mr{GL}_n \migisurj \mr{PGL}_n$,
of the group of invertible upper-triangular $n \times n$ matrices.

\LSP
\subsection{Moduli space of pointed stable curves} \label{QR780}

Let  $(g,r)$ be  a pair of nonnegative integers with $2g-2+r >0$.
For each commutative ring $R$, 
denote by 
\begin{align} \label{YY90}
\overline{\mcM}_{g,r, R} \ \left(\text{or simply}, \ \overline{\mcM}_{g,r}\right)
\end{align}
 the moduli  stack classifying $r$-pointed stable curves of genus $g$ over $R$.
Also, denote by $\mcM_{g,r, R}$ (or simply, $\mcM_{g, r}$) the dense open substack of $\overline{\mcM}_{g,r, R}$ classifying nonsingular curves.
The normal crossing divisor defined as the boundary $\overline{\mcM}_{g,r, R} \setminus \mcM_{g,r, R}$ on $\overline{\mcM}_{g,r, R}$ determines a log structure;
 we shall denote the resulting fs log stack  by 
 \begin{align} \label{QQ013}
 \overline{\mcM}_{g,r, R}^{\mr{log}} \ \left(\text{or simply}, \ \overline{\mcM}_{g,r}^\mr{log}\right).
 \end{align}

 Next, let  $\msX := (f:X \migi S, \{ \sigma_i \}_{i=1}^r)$ be an $r$-pointed stable curve of genus $g$ over an $R$-scheme $S$, where $\sigma_i$ denotes the $i$-th marked point $S \migi X$.
Recall from ~\cite[Theorem 2.6]{KaFu}  that
 there exist canonical log structures on  $X$ and $S$;  we shall denote the resulting log schemes by $X^\mr{log}$ and $S^\mr{log}$, respectively.
 (The log structure of $S^\mr{log}$ is obtained as the pull-back  from $\overline{\mcM}_{g, r}^\mr{log}$ via the classifying morphism $S \migi \overline{\mcM}_{g,r}$ of $\msX$.)
 
 The following assertion will be applied in the proof of Lemma \ref{LemLem}.

\SSP
\bpr \label{Prop277}
Let $X^\mr{log}/S^\mr{log}$ be as above.
Suppose that $S$ is flat over $\mbZ/p^{\CH +1} \mbZ$, and that $f$ is smooth (i.e., the $S$-rational point of $\overline{\mcM}_{g, r}$ classifying $\msX$ lies in $\mcM_{g, r}$).
Also, let $U$ be a scheme-theoretically dense  open subscheme of $X \setminus \bigcup_{i=1}^r \mr{Im}(\sigma_i)$, and 
 $\msL := (\mcL, \nabla)$  a flat line bundle on  $X^\mr{log}/S^\mr{log}$ whose restriction $\msL |_U$ to $U$ is dormant.
 Then, $\msL$ is dormant.
\epr
\begin{proof}
For simplicity, we shall write $\mcD^{(m)}$ (for each $m \geq 0$) instead of $\mcD^{(m)}_{X_0^\mr{log}/S_0^\mr{log}}$.
By Proposition \ref{P016dd}, (i),
we may assume, without loss of generality, that $U = X \setminus \bigcup_{i=1}^r \mr{Im}(\sigma_i)$.
We shall prove the assertion by induction on $\CH$.

The base step, i.e., $\CH = 0$, is clear because $U$ is scheme-theoretically dense in $X$ and the $p$-curvature of a flat bundle in characteristic $p$ can be regarded as a global section of a certain associated vector bundle.

Next, to discuss the induction step,
suppose that we have proved the required assertion with $\CH$ replaced with $\CH -1$ ($\CH \in \mbZ_{> 0}$).
In particular, the reduction $(\mcL_{\CH -1}, \nabla_{\CH -1})$ modulo $p^{\CH}$ is dormant.
The  diagonal reduction of $\nabla_{\CH -1}$ determines 
 a $\mcD^{(\CH -1)}$-module structure $\nabla_0^{(\CH -1)}$ on $\mcL_0$.
Just as in \eqref{J13}, 
we have an $S^\mr{log}$-connection $\nabla_\msL$ on $\mcV_\msL := \mcS ol (\nabla_0^{(\CH -1)})$ associated to $\msL$; it  has vanishing $p$-curvature because $\msL |_U$ is dormant.
According to ~\cite[Corollaire 3.3.1]{Mon}, $\nabla_\msL$ induces a $\mcD^{(\CH )}$-module structure $\nabla^{(\CH)}_{\mcL^\flat_0}$ on $\mcL^\flat_0 := F_{X_0/S_0}^{(\CH)*}(\mcS ol (\nabla_0^{(\CH -1)}))$ with vanishing $p^{\CH +1}$-curvature that are compatible with $\nabla_0^{(\CH -1)}$ via the inclusion $\tau : \mcL^\flat_0 \hookrightarrow \mcL_0$ and the natural morphism $\mcD^{(\CH -1)}\rightarrow \mcD^{(\CH )}$.
Denote by $\nabla_{u_*(\mcL^\flat_0|_U)}^{(\CH)}$ the
 $\mcD^{(\CH )}$-module structure on $u_*(\mcL^\flat_0|_U) \left( = u_*(\mcL_0|_U) \right)$ extending $\nabla^{(\CH)}_{\mcL^\flat_0}$ via the open immersion $u : U \hookrightarrow X$.

 In what follows, we prove the claim that
 $\mcL_0 \left(\subseteq u_*(\mcL^\flat_0|_U) \right)$ is closed under 
  $\nabla_{u_*(\mcL^\flat_0|_U)}^{(\CH)}$.
  To  this end, we may assume that $S$ is connected.
Then,  the equality  $\mcL_0^\flat = \mcL_0 (-\sum_{i=1}^r a_i \sigma_i)$ holds for some  integers $a_i$ with $0 \leq a_i < p^{\CH}$ (cf. Proposition-Definition \ref{P022}).
  We choose  $i \in \{ 1, \cdots, r \}$.
 Each section $s$ of $\mcL_0$ over a sufficiently small open neighborhood of a point in $\mr{Im}(\sigma_i)$ may be described as $s = t^{-a_i} \cdot s'$ for some $s' \in \mcL^\flat_0$, where $t$ denotes a local function defining $\mr{Im}(\sigma_i)$. 
Denote by  $\{ \partial^{\langle j \rangle} \}_{j \geq 0}$ 
 the local basis   of $\mcD^{(\CH )}$ associated to $t$ in  the manner of Section \ref{SS04f4}.
Since the $\mcO_X$-algebra $\mcD^{(\ell)}$ is locally  generated by $\{ \partial^{\langle p^{a} \rangle} \}_{a=0}^{\ell}$,
the problem is reduced to show that $ \nabla_{u_*(\mcL^\flat_0|_U)}^{(\CH)} (\partial^{\langle p^{a} \rangle} ) (s) \in \mcL_0$ for $a =0, \cdots, \ell$.
If $a = 0, \cdots, \ell -1$, this  is true because $\nabla_{u_*(\mcL^\flat_0|_U)}^{(\CH)} (\partial^{\langle p^{a} \rangle} ) (s)$ coincides with $\nabla_0^{(\ell -1)} (\partial^{\langle p^a \rangle})(s)$ via $\mcD^{(\ell -1)} \rightarrow \mcD^{(\ell)}$.
As for the case of $a = \ell$,
 we have 
 \begin{align}
 \nabla_{u_*(\mcL^\flat_0|_U)}^{(\CH)} (\partial^{\langle p^{\CH} \rangle} )(t^{-a_i} \cdot s') &= 
 \sum_{j=0}^{p^{\CH}} 
     \partial^{\langle j \rangle} (t^{-a_i}) \cdot 
 \nabla_{u_*(\mcL^\flat_0|_U)}^{(\CH)} (\partial^{\langle p^{\CH} -j \rangle} )
 (s') \\
 &= \sum_{j=0}^{p^{\CH}} 
 \partial^{\langle j \rangle} (t^{-a_i}) \cdot \nabla^{(\ell)}_{\mcL_0^\flat} (\partial^{\langle p^{\CH} -j \rangle}) (s') \notag \\
 \end{align}
 (cf. ~\cite[Corollaire 2.6.1]{Mon}).
 Since 
 $ \partial^{\langle j \rangle} (t^{-a_i}) \in t^{-a_i} \cdot \mcO_X$ for any $j$ 
 (cf. ~\cite[Lemme 2.3.3]{Mon}) and $\nabla^{(\ell)}_{\mcL_0^\flat} (\partial^{\langle p^{\CH} -j \rangle}) (s') \in \mcL_0^\flat$,
 we have  $\nabla_{u_*(\mcL^\flat_0|_U)}^{(\CH)} (\partial^{\langle p^{\CH} \rangle} ) (s) \in \mcL_0$.
 This proves the claim, as desired. 
 
Now,  let $\nabla^{(\ell)}_{\mcL_0}$ denote the resulting $\mcD^{(\ell)}$-module structure on $\mcL_0$.
One may verify (from the inequalities $a_i < p^\ell$ for $i$'s) that $(\msL, \nabla_{\mcL_0}^{(\ell)}) \in \mr{Diag}_{\ell}$ (cf. \eqref{ee1}), i.e., $\msL$ is dormant.
 This completes the proof of the assertion.
\end{proof}

\LSP
\subsection{$\mr{PGL}_n^{(\N)}$-opers on log curves} \label{SS041}

Let $S^\mr{log}$ be an fs log scheme  whose underlying scheme $S$ is 
flat over $\mbZ/p^{\CH +1}\mbZ$, and let
 $f^\mr{log} : X^\mr{log} \migi S^\mr{log}$ be a log curve over $S^\mr{log}$.
 Suppose that $S$ is equipped with an $(\N-1)$-PD structure that extends to  $X$ via $f$.
For simplicity, we write $\Omega := \Omega_{X^\mr{log}/S^\mr{log}}$, $\mcT := \mcT_{X^\mr{log}/S^\mr{log}}$,  $\mcD^{(\N-1)} := \mcD^{(\N-1)}_{X^\mr{log}/S^\mr{log}}$, and $\mcD^{(\N-1)}_{\leq j} := \mcD^{(\N-1)}_{X^\mr{log}/S^\mr{log}, \leq j}$ ($j\in \mbZ_{\geq 0}$).
 In this subsection, we shall 
 fix an integer $n$ with 
 $1 < n < p$.

Let  us consider a pair
\begin{align} \label{e100}
\msE^\spadesuit := (\mcE_B, \STR)
\end{align}
consisting of a $B$-bundle $\mcE_B$ on $X$ and an $(\N-1)$-PD stratification on 
$\mcE$, where  $\mcE := \mcE_B \times^B \mr{PGL}_n$, i.e., $\mcE$ is 
 the $\mr{PGL}_n$-bundle  induced from  $\mcE_B$ via change of structure group by the natural inclusion $B \migiincl \mr{PGL}_n$.
 Denote by $\DMO_\STR$ the $S^\mr{log}$-connection on $\mcE$ corresponding,  via 
 the equivalence of categories \eqref{Efjj2},  to  the $0$-PD stratification induced from $\STR$.

\SSP
\bde \label{D39}
\begin{itemize}
\item[(i)]
We shall say that $\msE^\spadesuit$ is a {\bf  $\mr{PGL}_n^{(\N)}$-oper}  (or a {\bf $\mr{PGL}_n$-oper of level $\N$}) on $X^\mr{log}/S^\mr{log}$ if the pair $(\mcE_B, \nabla_{\STR})$ forms an $\mfs \mfl_n$-oper on $X^\mr{log}/S^\mr{log}$ in the sense of ~\cite[Definition 2.1]{Wak8}.
(Note that the notion of an $\mfs \mfl_n$-oper  can be defined in exactly the same way as  ~\cite{Wak8}, even though  we are  working over  $\mbZ/p^{\CH +1}\mbZ$.
In particular, a $\mr{PGL}_n^{(1)}$-oper
can be formulated in terms of flat connections,
i.e., without fixing a ($0$-)PD structure on $S$.)
If the log curve $X^\mr{log}/S^\mr{log}$ arises from a pointed stable curve $\msX$,
then  we will refer to any $\mr{PGL}_n^{(\N)}$-oper on $X^\mr{log}/S^\mr{log}$ as a {\bf $\mr{PGL}_n^{(\N)}$-oper  on $\msX$}.
Also, for simplicity, we shall refer to any $\mr{PGL}_n^{(1)}$-oper as a {\bf $\mr{PGL}_n$-oper}.
\item[(ii)]
Let $\msE^\spadesuit_\circ := (\mcE_{\circ, B}, \STR_{\circ})$
and  $\msE^\spadesuit_\bullet := (\mcE_{\bullet, B}, \STR_{\bullet})$
 be $\mr{PGL}_n^{(\N)}$-opers  on $X^\mr{log}/S^\mr{log}$.
 Write $\mcE_\circ := \mcE_{\circ, B} \times^B \mr{PGL}_n$ and 
 $\mcE_\bullet := \mcE_{\bullet, B} \times^B \mr{PGL}_n$.
Then, an {\bf isomorphism of $\mr{PGL}_n^{(\N)}$-opers} from $\msE^\spadesuit_\circ$ to $\msE^\spadesuit_\bullet$ is defined as an isomorphism of $B$-bundles  $\mcE_{\circ, B} \isom \mcE_{\bullet, B}$     that induces, via change of structure group by $B \migiincl \mr{PGL}_n$, an isomorphism of $(\N-1)$-PD stratified $\mr{PGL}_n$-bundles $(\mcE_\circ, \STR_\circ) \isom (\mcE_\bullet, \STR_\bullet)$.
\end{itemize}
\ede
\SSP

Moreover, we shall define the notion of a {\it dormant} $\mr{PGL}_n^{(\N)}$-oper, generalizing ~\cite[Chapter II, Definition 1.1]{Mzk2}, ~\cite[Definition 3.15]{Wak8}, and ~\cite[Definition 5.1.1]{Wak9}.

\SSP
\bde \label{DD45YY}
Let $\msE^\spadesuit := (\mcE_B, \STR)$ be a $\mr{PGL}_n^{(\N)}$-oper on $X^\mr{log}/S^\mr{log}$.
\begin{itemize}
\item[(i)]
Suppose that $\CH = 0$, i.e., $S$ is a scheme over $\mbF_p$.
Then, we shall say that $\msE^\spadesuit$ is {\bf dormant} if $\STR$ has vanishing $p^\N$-curvature.
\item[(ii)]
Suppose that $\N =1$.
Then, we shall say that $\msE^\spadesuit$ is {\bf dormant} if the
induced   flat $\mr{PGL}_n$-bundle
  is dormant in the sense of Definition \ref{D019}, (ii).
\end{itemize}
\ede

\SSP
\begin{rema}[$G$-oper of higher level] \label{RRe3}
In a similar manner to the above definition, 
we can extend  the definition  of a $G$-oper (discussed in  ~\cite[Definition 2.1]{Wak8}) to  higher level  at least when $G$ is a semisimple algebraic group of adjoint type. 
\end{rema}
\SSP

Let $\msE^\spadesuit := (\mcE_B, \STR)$ be a $\mr{PGL}_n^{(\N)}$-oper on $X^\mr{log}/S^\mr{log}$.
Suppose that
we are given  an fs log scheme $S'^{\mr{log}}$ flat over $\mbZ/p^{\CH +1} \mbZ$ and a morphism of log schemes $s^\mr{log} : S'^{\mr{log}} \migi S^\mr{log}$  over $\mbZ/p^{\CH +1} \mbZ$.
Then, the pair of base-changes
\begin{align} \label{e408}
s^* (\msE^\spadesuit) := (s^* (\mcE_B), s^* (\STR))
\end{align}
forms a $\mr{PGL}_n^{(\N)}$-oper on the log curve $(S'^{\mr{log}} \times_{S^\mr{log}} X^\mr{log})/S'^{\mr{log}}$.
Conversely, the formation of a $\mr{PGL}_n^{(\N)}$-oper has descent with respect to, e.g.,  the \'{e}tale topology on $S$.

Let us prove the following fundamental property of  $\mr{PGL}_n^{(\N)}$-opers, which is a generalization of ~\cite[Section 1.3, Proposition]{BD2} and  ~\cite[Proposition 2.9]{Wak8} (in the case of  $\mr{PGL}_n$-opers).

\SSP
\bpr \label{P44}
Any $\mr{PGL}_n^{(\N)}$-oper  on $X^\mr{log}/S^\mr{log}$ does not have nontrivial automorphisms.
\epr
\begin{proof}
We prove the assertion by  induction  on $\CH$.
The base step, i.e., the case of $\CH = 0$,  follows from  
   ~\cite[Proposition 2.9]{Wak8}.

Next,  to prove the induction step,
we suppose that the assertion with $\CH$ replaced by $\CH -1$ ($\CH \geq 1$)
has been proved.
We may assume, without loss of generality, that $S$ is affine.
Let $\msE^\spadesuit := (\mcE_B, \STR)$ be a $\mr{PGL}_n^{(\N)}$-oper on $X^\mr{log}/S^\mr{log}$ and $h : \msE^\spadesuit \isom \msE^\spadesuit$ an automorphism of $\msE^\spadesuit$.
By the induction hypothesis, the reduction  of $h$ modulo $p^{\CH}$ is equal to  the identity morphism.
Denote by $h^{(1)}$ the automorphism of the $\mr{PGL}_n$-oper $(\msE^\spadesuit)^{(1)} := (\mcE_B, \nabla_\phi)$ induced by $\msE^\spadesuit$.
According to   an argument similar to the argument in ~\cite[Section 6.3]{Wak8} (which deals with the case where the base space $S^\mr{log}$ is a log scheme over a field),
$h^{(1)}$ coincides with the identity morphism.
Indeed,   the space of automorphisms of $(\msE^\spadesuit)^{(1)}$ inducing the identity morphism via reduction modulo $p^{\CH}$ has a structure of torsor modeled on  
 $H^0 (X, \mr{Ker}(\nabla^{\mr{ad} (0)}))$, where ``$\nabla^{\mr{ad} (0)}$" denotes the morphism defined  in ~\cite[Eq.\,(759)]{Wak8} (for $\mfg = \mfs \mfl_n$) associated to
 the reduction modulo $p$ of $(\msE^\spadesuit)^{(1)}$;
but the equality $H^0 (X, \mr{Ker}(\nabla^{\mr{ad} (0)})) = 0$ holds by ~\cite[Proposition 6.5]{Wak8}, which implies 
 $h= \mr{id}_{\mcE_B}$.
 This proves  the induction step, and hence we have finished the proof of  the assertion.
\end{proof}
\SSP

We shall set $R := \mbZ/p^{\CH +1}\mbZ$ ($\ell \in \mbZ_{\geq 0}$).
Denote by $\mcS et$  the category of (small) sets and by $\mcS c h_{R}^{\mr{flat}}/S$ the category of flat $R$-schemes $S'$ equipped with an $R$-morphism $S' \migi S$.
 We shall 
 denote by
\begin{align} \label{e401}
\mcO p_\spadesuit 
\ \left(\text{resp.,} \  \mcO p_\spadesuit^\ZZZ \right) : 
\mcS ch_{R}^{\mr{flat}}/S \migi \mcS et
\end{align}
the contravariant functor  on  $\mcS ch_{R}^{\mr{flat}}/S$, which to any $S$-scheme $S'$ in $\mr{Ob}(\mcS ch_{R}^{\mr{flat}}/S)$, assigns the set of isomorphism classes of $\mr{PGL}_n^{(\N)}$-opers (resp., dormant $\mr{PGL}_n^{(\N)}$-opers) on the log curve $(S' \times_S X^\mr{log})/(S' \times_S S^\mr{log})$.
By   Proposition \ref{P44} above, 
$\mcO p_\spadesuit$  (resp., $\mcO p_\spadesuit^\ZZZ$)
 turns out to be  a sheaf with respect to the big \'{e}tale topology on $\mcS ch_{R}^{\mr{flat}}/S$.

\LSP
\subsection{$\mr{GL}_n^{(\N)}$-opers on log curves} \label{SS042}

Next, we shall define the notion of a $\mr{GL}_n^{(\N)}$-oper and define a certain equivalence relation in the set of $\mr{GL}_n^{(\N)}$-opers.
Fix an integer  $n$ with $n > 1$.

Let us consider a collection of data
\begin{align} \label{TTT1}
\msF^\heartsuit := (\mcF, \DMO, \{ \mcF^j \}_{j =0}^n),
\end{align}
where
\begin{itemize}
\item
$\mcF$ is a vector bundle on $X$ of rank $n$;
\item 
$\nabla$ is a $\mcD^{(\N-1)}$-module structure  on $\mcF$;
\item
$\{ \mcF^j \}_{j=0}^n$ is an $n$-step decreasing filtration
\begin{align} \label{TTT2}
0 = \mcF^n \subseteq \mcF^{n-1} \subseteq \cdots \subseteq \mcF^0 = \mcF
\end{align}
on $\mcF$ consisting of subbundles such that the subquotients  $\mcF^j/\mcF^{j+1}$ are line bundles.
\end{itemize}

\SSP
\bde  \label{D01234}
 \begin{itemize}
 \item[(i)]
 We shall say that $\msF^\heartsuit$ is a {\bf $\mr{GL}_n^{(\N)}$-oper} (or a {\bf $\mr{GL}_n$-oper of level $\N$}) on $X^\mr{log}/S^\mr{log}$
if, for every $j=0, \cdots, n-1$,
the $\mcO_X$-linear morphism $\mcD^{(\N -1)} \otimes \mcF \migi \mcF$ induced by $\nabla$ restricts to   an isomorphism
\begin{align} \label{YY92}
\mcD_{\leq n-j-1}^{(\N -1)} \otimes \mcF^{n-1} \isom \mcF^{j}.
\end{align}
When $X^\mr{log}/S^\mr{log}$ arises from a pointed stable curve $\msX$, we shall  refer to any $\mr{GL}_n^{(\N)}$-oper on $X^\mr{log}/S^\mr{log}$ as a {\bf $\mr{GL}_n^{(\N)}$-oper on $\msX$}.
(When $\N =1$,  this definition is the same as the usual definition of a $\mr{GL}_n$-oper, as defined in, e.g., ~\cite[Definition 4.17]{Wak8}.)
\item[(ii)]
Let $\msF_\circ^\heartsuit := (\mcF_\circ, \DMO_\circ, \{ \mcF_\circ^j \}_j)$ and  $\msF_\bullet^\heartsuit := (\mcF_\bullet, \DMO_\bullet, \{ \mcF_\bullet^j \}_j)$
 be $\mr{GL}_n^{(\N)}$-opers on $X^\mr{log}/S^\mr{log}$.
An {\bf isomorphism of $\mr{GL}_n^{(\N)}$-opers}
from $\msF_\circ^\heartsuit$ to $\msF_\bullet^\heartsuit$
is an isomorphism $(\mcF_\circ, \DMO_\circ) \isom (\mcF_\bullet, \DMO_\bullet)$ of $\mcD^{(\N-1)}$-modules preserving the filtration.
\end{itemize}
  \ede
\SSP

\bde \label{D50}
Let $\msF^\heartsuit := (\mcF, \nabla, \{ \mcF^j \}_j)$ be a $\mr{GL}_n^{(\N)}$-oper on $X^\mr{log}/S^\mr{log}$.
\begin{itemize}
\item[(i)]
Suppose that $\CH =0$, i.e., $S$ is a scheme over $\mbF_p$.
Then, we shall say that $\msF^\heartsuit$ is {\bf dormant} if $\DMO$ has  vanishing $p^{\N}$-curvature.
\item[(ii)]
Suppose that $\N =1$.
Then, we shall say that $\msF^\heartsuit$ is {\bf dormant}
if   the flat bundle $(\mcF, \DMO)$ is dormant in the sense of Definition \ref{D019}, (i).
\end{itemize}
\ede

\SSP
\begin{exa}[$\mr{GL}_2^{(\N)}$-opers for   large $\N$'s] \label{EEdf34}
Despite the fact that
dormant $\mr{GL}_2^{(1)}$-opers have been substantially investigated in many references,  there are few examples 
for higher level 
at the time of writing this manuscript.
We here  give an example of a dormant $\mr{GL}_2^{(\N)}$-oper (for a  general $\N$) constructed  recently in ~\cite{Wak10}.

Let $k$ be an algebraically closed field over $\mbF_p$ and $X$ a  curve embedded in the projective plane $\mbP^2 := \{ [t_0: t_1: t_2] \, | \, (t_0, t_1, t_2) \neq (0, 0, 0) \}$ over $k$.
Denote by $\mr{Grass} (2, 3)$ the Grassmann variety classifying $2$-dimensional quotient spaces of the $k$-vector space $k^3$; it may be identified with the space of $1$-planes in $\mbP^2$.
Recall that the {\it Gauss map} on $X$ is the rational morphism $\gamma : X \dashrightarrow \mr{Grass}(2,3)$ 
that assigns to each smooth point $x$  the embedded tangent space to $X$ at $x$ in $\mbP^2$.

Now, suppose that $p>2$ and that
 $X$ is  the Fermat hypersurface of degree $p^\N +1$ in 
 $\mbP^2$,
 i.e., the smooth hypersurface defined by the homogeneous polynomial $t_0^{p^{\N}+1} + t_1^{p^\N +1}+t_2^{p^\N +1}$.
It is well-known  that the Gauss map on $X$ factors through the $N$-th relative Frobenius morphism $F_{X/k}^{(\N)}$.
In particular, we obtain a rank $2$ vector bundle $\mcF^\nabla$   on  $X^{(\N)}$ by 
 pulling-back  the universal quotient bundle on $\mr{Grass}(2,3)$ via the resulting morphism $X^{(\N)} \migi \mr{Grass}(2,3)$.
 If 
 $\{ F^{(\N)*}_{X/k}(\mcF^\nabla)^j\}_j$ denotes the Harder-Narasimhan filtration on 
  $F^{(\N)*}_{X/k}(\mcF^\nabla)$, then it forms a $2$-step (decreasing) filtration.
Moreover, it follows from ~\cite[Theorem C]{Wak10} that the collection of data
\begin{align}
(F^{(\N)*}_{X/k}(\mcF^\nabla), \nabla^{(\N -1)}_{\mcF^\nabla, \mr{can}}, \{ F^{(\N)*}_{X/k}(\mcF^\nabla)^j\}_j)
\end{align}
(cf. \eqref{QQwkko} for the definition of $\DMO_{(-), \mr{can}}^{(\M)}$) defines  a dormant $\mr{GL}_2^{(\N)}$-oper on $X/k$.
\end{exa}
\SSP

\bpr \label{YY439}
Let $\mcF^\heartsuit := (\mcF, \DMO^{(\N -1)}, \{ \mcF^j \}_j)$ be a (dormant) $\mr{GL}_n^{(\N)}$-oper on $X^\mr{log}/S^\mr{log}$.
Also, let $\N'$ be a positive integer with $n< p^{\N'}$ and $\N' \leq \N$.
Then, the  collection 
\begin{align} \label{YY467}
(\mcF, \DMO^{(\N -1) \Rightarrow (\N' -1)}, \{\mcF_j \}_j)
\end{align}
 forms a (dormant) $\mr{GL}_n^{(\N')}$-oper on $X^\mr{log}/S^\mr{log}$.
\epr
\begin{proof}
The assertion follows from the fact that the natural morphism $\mcD_{\leq n - j-1}^{(\N'-1)} \otimes \mcF \migi \mcD_{\leq n - j-1}^{(\N -1)} \otimes \mcF$ is an isomorphism  for every $j =0, \cdots, n-1$.
\end{proof}
\SSP

Next, we shall define an equivalence relation in the set of $\mr{GL}_n^{(\N)}$-opers.
Let $\msF^\heartsuit := (\mcF, \DMO, \{ \mcF^j\}_j)$ be 
 a $\mr{GL}_n^{(\N)}$-oper on $X^\mr{log}/S^\mr{log}$ and 
    $\msL := (\mcL, \DMO_\mcL)$ an invertible $\mcD^{(\N-1)}$-module.
In particular, we obtain a $\mcD^{(\N -1)}$-module structure $\DMO \otimes \DMO_\mcL$ on the tensor product $\mcF \otimes \mcL$
arising from $\DMO$ and $\DMO_\mcL$.
One may verify that the collection
\begin{align} \label{Ep2}
\msF^\heartsuit_{\otimes \msL} := (\mcF \otimes \mcL, \DMO \otimes \DMO_\mcL, \{ \mcF^j  \otimes \mcL \}_{j=0}^n)
 \end{align}
forms 
a $\mr{GL}_n^{(\N)}$-oper  on $X^\mr{log}/S^\mr{log}$.
If  $\msF^\heartsuit$ and $\msL$ are dormant, then $\msF^\heartsuit_{\otimes \msL}$ is verified to be dormant.

\SSP
\bde  \label{D0f30}
 Let $\msF^\heartsuit_\circ := (\mcF_\circ, \DMO_\circ, \{ \mcF_\circ^j \}_j)$ 
 and $\msF^\heartsuit_\bullet := (\mcF_\bullet, \DMO_\bullet, \{ \mcF_\bullet^j \}_j)$
   be $\mr{GL}_n^{(\N)}$-opers (resp., dormant $\mr{GL}_n^{(\N)}$-opers) on $X^\mr{log}/S^\mr{log}$.
 We shall say that 
 $\msF^\heartsuit_\circ$ is {\bf equivalent to $\msF^\heartsuit_\bullet$}
 if 
 there exists an invertible (resp., a dormant invertible) $\mcD^{(\N-1)}$-module $\msL := (\mcL, \DMO_\mcL)$
   such that 
the $\mr{GL}_n^{(\N)}$-oper $\msF^\heartsuit_{\circ, \otimes \msL}$
 is isomorphic to 
 $\msF_{\bullet}^\heartsuit$. 
 We use the notion ``$\msF^\heartsuit_\circ \sim \msF^\heartsuit_\bullet$"
 to  indicate the situation that $\msF^\heartsuit_\circ$ is equivalent to $\msF_\bullet^\heartsuit$.
 (The binary relation ``$\sim$" in the set of (dormant) $\mr{GL}_n^{(\N)}$-opers on $X^\mr{log}/S^\mr{log}$ actually  defines an equivalence relation.)
 Moreover, for a (dormant) $\mr{GL}_n^{(\N)}$-oper $\msF^\heartsuit$, we shall write $[\msF^\heartsuit]$ for the equivalence class represented by $\msF^\heartsuit$.
 \ede
\SSP

Let $s : S' \migi S$ be  an $S$-scheme classified by $\mcS ch_{R}^\mr{flat}/S$.
The base-change $s^*(\msF^\heartsuit)$ of $\msF^\heartsuit$ along $s$
 can be constructed,   and the  formulation of base-changes  preserves the equivalence relation ``$\sim$".
Hence, we obtain the 
contravariant functor
\begin{align} \label{e745}
\mcO p_\heartsuit
\ \left(\text{resp.,} \  \mcO p_\heartsuit^\ZZZ\right) : \mcS ch_{R}^\mr{flat}/S \migi \mcS et
\end{align}
on $\mcS ch_{R}^\mr{flat}/S$  defined as the sheaf (with respect to the \'{e}tale topology) associated to 
the functor  
which, to any $S$-scheme $S'$ in $\mr{Ob}(\mcS ch_{R}^\mr{flat}/S)$,
assigns the set of equivalence classes of $\mr{GL}_n^{(\N)}$-opers (resp., dormant $\mr{GL}_n^{(\N)}$-opers) on $(S' \times_S X^\mr{log})/(S' \times_S S^\mr{log})$.

\SSP
\bpr \label{YY104}
Suppose that $p \nmid n$.
Then, the natural morphism of functors 
 $\mcO p_\heartsuit^\ZZZ \!\migi \mcO p_\heartsuit$ is injective.
 (In the subsequent discussion, we shall regard $\mcO p_\heartsuit^\ZZZ\!$ as a subfunctor of  $\mcO p_\heartsuit$ by using this injection.)
\epr
\begin{proof}
Suppose that  two dormant $\mr{GL}_n^{(\N)}$-opers $\msF^\heartsuit_\circ := (\mcF_\circ, \DMO_\circ, \{ \mcF_\circ^j \}_j)$ and  $\msF^\heartsuit_\bullet := (\mcF_\bullet, \DMO_\bullet, \{ \mcF_\bullet^j \}_j)$
  are  equivalent  in the set of ({\it non-dormant}) $\mr{GL}_n^{(\N)}$-opers on $X^\mr{log}/S^\mr{log}$.
  This means that there exists an invertible $\mcD^{(\N-1)}$-module $\msL := (\mcL, \DMO_\mcL)$ with $\msF^\heartsuit_{\circ, \otimes \msL}\cong \msF^\heartsuit_\bullet$.
By taking determinants, we obtain 
\begin{align} \label{YY98}
(\mr{det}(\mcF_\circ), \mr{det}(\DMO_\circ)) \otimes (\mcL^{\otimes n}, \DMO_\mcL^{\otimes n}) \cong (\mr{det}(\mcF_\bullet), \mr{det}(\DMO_\bullet)).
\end{align}
 Since both $\mr{det}(\DMO_\circ)$ and $\mr{det}(\DMO_\bullet)$ are dormant, $\nabla_\mcL^{\otimes n}$ is dormant by \eqref{YY98}.
It follows from Proposition \ref{P129}, (ii), that $\DMO_\mcL$ turns out to be  dormant.
 That is to say, $\msF_\circ^\heartsuit$ is equivalent to $\msF_\bullet^\heartsuit$   in the set of {\it dormant} $\mr{GL}_n^{(\N)}$-opers.
 \end{proof}
 \SSP

\SSP
\bpr \label{Prop456}
Suppose that $n > p^\N$.
Then, there is no dormant $\mr{GL}_n^{(\N)}$-oper on $X^\mr{log}/S^\mr{log}$.
\epr
\begin{proof}
If $\ell >0$, then $\N =1$, and reduction modulo $p$ reduces the assertion to the case $\ell =0$.
In this case,
suppose, on the contrary, that there exists a dormant $\mr{GL}_n^{(\N)}$-oper $\msF^\heartsuit := (\mcF, \nabla, \{ \mcF^j \}_j)$ on $X^\mr{log}/S^\mr{log}$.
 We equip $\mcD_{}^{(\N-1)} \otimes \mcF^{n-1}$ with  the  left  $\mcD_{}^{(\N-1)}$-module structure   given by left multiplication.
Write
$\mcP$
 for the quotient of 
 the $\mcD_{}^{(\N-1)}$-module
  $\mcD_{}^{(\N-1)} \otimes \mcF^{n-1}$ by the $\mcD_{}^{(\N-1)}$-submodule generated by the image of $\psi_{X^\mr{log}/S^\mr{log}} \otimes \mr{id}_{\mcF^{n-1}} : \mcT^{\otimes p^\N} \otimes \mcF^{n-1} \migi \mcD_{}^{(\N-1)} \otimes \mcF^{n-1}$;
 this   $\mcD^{(\N-1)}$-module structure on  $\mcP$ will be denoted by $\DMO_\mcP$.
  By construction,  $\nabla_{\mcP}$ has  vanishing  $p^\N$-curvature.
By the definition of a $\mr{GL}_n^{(\N)}$-oper,
the composite
\begin{align} \label{e4002}
 \mcD_{}^{(\N-1)} \otimes  \mcF^{n-1} 
\xrightarrow{\mr{inclusion}}
 \mcD_{}^{(\N-1)} \otimes \mcF \xrightarrow{\nabla}
 \mcF
 \end{align}
 is surjective.
 Also, this composite 
 factors through the quotient $\mcD_{}^{(\N-1)} \otimes \mcF^{n-1} \migisurj \mcP$ because
it preserves the $\mcD_{}^{(\N-1)}$-module structure and 
$\DMO$ has vanishing $p^\N$-curvature.
Thus, it induces  a surjective morphism  $\mcP \twoheadrightarrow \mcF$.
Since $\mcP$ has rank $p^\N$, this contradicts the assumption $n > p^\N$.
This completes the proof.
\end{proof}

Suppose further that $n <p$.
Let  $\msF^\heartsuit := (\mcF, \DMO, \{ \mcF^j\}_j)$ be a 
 $\mr{GL}_n^{(\N)}$-oper (resp., a dormant $\mr{GL}_n^{(\N)}$-oper) on $X^\mr{log}/S^\mr{log}$.
Denote by $\mcE_\mcF$
   the $\mr{PGL}_n$-bundle associated to the vector bundle $\mcF$ via a change of structure group by the projection $\mr{GL}_n \migisurj \mr{PGL}_n$.
The filtration  $\{ \mcF^j \}_j$ determines a $B$-reduction $\mcE_{\{ \mcF^j \}_j}$ of  $\mcE_\mcF$.
Also, $\DMO$ induces an $(\N -1)$-PD stratification $\STR_\DMO$ on $\mcE_\mcF$ (cf. Remark \ref{dGGg5}).
The resulting pair
\begin{align} \label{e448}
\msF^{\heartsuit \Rightarrow \spadesuit} := (\mcE_{\{ \mcF^j \}_j}, \STR_\DMO)
\end{align}
forms a 
 $\mr{PGL}_n^{(\N)}$-oper  (resp., a dormant $\mr{PGL}_n^{(\N)}$-oper) on $X^\mr{log}/S^\mr{log}$ (cf. ~\cite[Section 4.4.6]{Wak8}).
The isomorphism class of $\msF^{\heartsuit \Rightarrow \spadesuit}$ depends only on the equivalence class $[\msF^\heartsuit]$ of $\msF^\heartsuit$.
Hence, 
the assignment 
 $[\msF^\heartsuit] \mapsto \msF^{\heartsuit \Rightarrow \spadesuit}$ determines a well-defined morphism of functors
\begin{align} \label{e777}
\Lambda_{\heartsuit \Rightarrow \spadesuit} : \mcO p_\heartsuit \migi \mcO p_\spadesuit. 
\ \left(\text{resp.,} \ \Lambda_{\heartsuit \Rightarrow \spadesuit}^\ZZZ : 
  \mcO p_\heartsuit^\ZZZ \migi \mcO p_\spadesuit^\ZZZ\right).
\end{align}

\LSP
\subsection{$n^{(\N)}$-theta characteristics} \label{SS055}

We introduce  $n^{(\N)}$-theta characteristics, generalizing $n$-theta characteristic defined in ~\cite[Definition 4.31]{Wak8}.
Let $n$ be an integer $>1$.

\SSP
\bde \label{D41}
An {\bf $n^{(\N)}$-theta characteristic} (or an {\bf $n$-theta characteristic of level $\N$}) of $X^\mr{log}/S^\mr{log}$ is a pair
\begin{align} \label{TTT4}
\vartheta := (\varTheta, \DMO_\vartheta),
\end{align}
consisting of a line bundle $\varTheta$ on $X$ and a $\mcD^{(\N -1)}$-module structure $\DMO_\vartheta$ on the line bundle $\mcT^{\otimes \frac{n (n-1)}{2}} \otimes \varTheta^{\otimes n}$.
Also, we say that an $n^{(\N)}$-theta characteristic $\vartheta := (\varTheta, \DMO_\vartheta)$ is {\bf dormant} if  $(\mcT^{\otimes \frac{n (n-1)}{2}} \otimes \varTheta^{\otimes n}, \DMO_\vartheta)$ is dormant.
\ede
\SSP

\begin{exa}[$n^{(\N)}$-theta characteristics arising from theta characteristics] \label{Ex20}
Recall that
 a {\it theta characteristic}  of $X^\mr{log}/S^\mr{log}$ is  a line bundle $\varTheta_0$ on $X$ together with an isomorphism $\mcO_X \isom \mcT \otimes \varTheta_0^{\otimes 2}$.

Suppose that we are given   a theta characteristic
$\varTheta_0$  of $X^\mr{log}/S^\mr{log}$, and 
fix an isomorphism $\tau_{\varTheta_0} : \mcO_X \isom \mcT \otimes \varTheta_0^{\otimes 2}$.
Then, the line bundle $\varTheta_0$ together with  the $\mcD^{(\N -1)}$-module structure on $\mcT \otimes \varTheta_0^{\otimes 2}$ corresponding to $\DMO_{X, \mr{triv}}^{(\N -1)}$ via  $\tau_{\varTheta_0}$ specifies a  $2^{(\N)}$-theta characteristic.

More generally, by setting  $\varTheta := \varTheta_0^{\otimes (n-1)}$, the line bundle $\mcT^{\otimes \frac{n (n-1)}{2}} \otimes \varTheta^{\otimes n}$ admits a $\mcD^{(\N -1)}$-module structure $\DMO_{\vartheta_0}$ corresponding to $\DMO^{(\N -1)}_{X, \mr{triv}}$ via the composite isomorphism
\begin{align} \label{e466}
\mcO_X \xrightarrow{\tau_{\varTheta_0}^{\otimes \frac{n (n-1)}{2}}} (\mcT \otimes \varTheta_0^{\otimes 2})^{\otimes \frac{n (n-1)}{2}}\isom \mcT^{\otimes \frac{n (n-1)}{2}} \otimes (\varTheta^{\otimes (n-1)}_0)^{\otimes n} \left(=  \mcT^{\otimes \frac{n (n-1)}{2}} \otimes \varTheta^{\otimes n}\right).
\end{align}
Thus, the  resulting pair
\begin{align} \label{TTT56}
\vartheta_0 := (\varTheta, \DMO_{\vartheta_0})
\end{align}
forms an $n^{(\N)}$-theta characteristic of $X^\mr{log}/S^\mr{log}$ (cf. ~\cite[Example 4.34]{Wak8} for the case of $(\CH, \N) =(0, 1)$).
Moreover, this $n^{(\N)}$-theta characteristic is verified to be  dormant.

Note that if, for example,  $X^\mr{log}/S^\mr{log}$ arises from a pointed stable curve, then the resulting $n^{(\N)}$-theta characteristic does not depend on the choice of the isomorphism $\tau_\varTheta$; this is because
 any automorphism of a line bundle on $X$ is given by multiplication by an element of $H^0 (S, \mcO_S^\times) \left(= H^0 (X, \mcO_X^{\times })\right)$, which is therefore compatible with any  $\mcD^{(\N-1)}$-action.
\end{exa}
\SSP

The following assertion is a higher-level generalization of the fact proved in ~\cite[Section 4.6.4]{Wak8}.

\SSP
\bpr \label{P204}
Suppose that $p \nmid n$.
Then, there always  exists a dormant $n^{(\N)}$-theta characteristic on $X^\mr{log}/S^\mr{log}$.
\epr
\begin{proof}
Since $(p^{\N + \CH}, n) = 1$,
one may find a pair of integers $(a, b)$ with $a \cdot n + b \cdot p^{\N + \CH} = \frac{n (n-1)}{2}$.
By letting $\varTheta_{a, b} := \Omega^{\otimes a}$, we have a composite isomorphism
\begin{align} \label{e550}
\mcT^{\otimes \frac{n(n-1)}{2}} \otimes \varTheta_{a, b}^{\otimes n} \isom \mcT^{\otimes \left(\frac{n(n-1)}{2} - a \cdot n \right)} \isom \mcT^{\otimes b \cdot p^{\N + \CH}} \isom (\mcT^{\otimes b})^{\otimes p^{\N + \CH}}.
\end{align}
Let us take  an open covering $\{ U_\alpha \}_{\alpha \in I}$ of $X$ together with a collection $\{ \tau_\alpha \}_{\alpha \in I}$, where each $\tau_\alpha$ denotes a trivialization $\mcT^{\otimes b}|_{U_\alpha} \isom \mcO_{U_\alpha}$ of $\mcT^{\otimes b}$ over $U_\alpha$.
For each $\alpha \in I$, denote by $\DMO_\alpha$ the $\mcD^{(\N-1)}|_{U_\alpha}$-module structure on $(\mcT^{\otimes b})^{\otimes p^{\N + \CH}} |_{U_\alpha}$ corresponding to 
$\DMO_{U_\alpha, \mr{triv}}^{(\N -1)}$
via 
the isomorphism
\begin{align}
\breve{\tau}_\alpha := \tau_\alpha^{\otimes p^{\N + \CH}} : (\mcT^{\otimes b})^{\otimes p^{\N + \CH}}|_{U_\alpha} \isom \left(\mcO_{U_\alpha}^{\otimes p^{\N + \CH}} =  \right)\mcO_{U_\alpha}.
\end{align}
For each pair $(\alpha, \beta) \in I \times I$ with $U_\alpha \cap U_\beta \neq \emptyset$,
the automorphism $\breve{\tau}_{\alpha} \circ \breve{\tau}_{\beta}^{-1}$  of $\mcO_{U_{\alpha}\cap U_\beta}$ is given by  multiplication by a section of the form $u^{p^{\N + \CH}}$ for some $u \in H^0 (U_{\alpha}\cap U_\beta, \mcO_X^\times)$.
By the local description of $\DMO_{(-), \mr{triv}}^{(\N -1)}$ displayed in \eqref{dd123},
 this automorphism is verified  to preserve the $\mcD^{(\N-1)}|_{U_{\alpha}\cap U_\beta}$-module structure.
It follows that the collection $\{ \DMO_\alpha \}_{\alpha \in I}$ may be glued together to obtain a $\mcD^{(\N -1)}$-module structure on $(\mcT^{\otimes b})^{\otimes p^{\N +\CH}}$;
if $\DMO_{a, b}$ denotes the corresponding $\mcD^{(\N -1)}$-module structure on $\mcT^{\otimes \frac{n(n-1)}{2}} \otimes \varTheta_{a, b}^{\otimes n}$ via \eqref{e550},
then the resulting pair $(\varTheta_{a, b}, \nabla_{a, b})$ specifies a required $n^{(\N)}$-theta characteristic.
\end{proof}
\SSP

Let us  fix an $n^{(\N)}$-theta characteristic $\vartheta := (\varTheta,  \DMO_\vartheta)$ of $X^\mr{log}/S^\mr{log}$, and set
\begin{align} \label{e607}
\mcF_\varTheta := \mcD_{\leq  n-1}^{(\N-1)} \otimes \varTheta, \hspace{5mm} \mcF_\varTheta^j := \mcD_{\leq n-j-1}^{(\N-1)} \otimes \varTheta \ \ (j =0, \cdots, n).
\end{align}
Since $\mcF^j_\varTheta / \mcF^{j+1}_\varTheta$ ($j=0, \cdots, n-1$) may be identified with $\mcT^{\otimes (n-j-1)} \otimes \varTheta$,
we obtain the composite isomorphism between line bundles
\begin{align} \label{e703}
\mr{det} (\mcF_\varTheta) \isom \bigotimes_{j=0}^{n-1} \mcF^j_\varTheta / \mcF^{j+1}_\varTheta
\isom \bigotimes_{j=0}^{n-1} \mcT^{\otimes (n-j-1)} \otimes \varTheta \isom 
\mcT^{\otimes \frac{n (n-1)}{2}} \otimes \varTheta^{\otimes n}.
\end{align}

\SSP
\bde \label{D42}
\begin{itemize}
\item[(i)]
A {\bf (dormant) $(\mr{GL}_n^{(\N)}, \vartheta)$-oper} on $X^\mr{log}/S^\mr{log}$ is 
a $\mcD^{(\N -1)}$-module structure $\DMO^\diamondsuit$ on $\mcF_\varTheta$
such that the collection of data
\begin{align} \label{e700}
\DMO^{\diamondsuit \Rightarrow \heartsuit} := (\mcF_\varTheta, \DMO^\diamondsuit, \{ \mcF_\varTheta^j \}_{j=0}^n)
\end{align}
forms a (dormant) $\mr{GL}_n^{(\N)}$-oper on $X^\mr{log}/S^\mr{log}$ and
the equality $\mr{det}(\DMO^\diamondsuit) = \DMO_\vartheta$ holds 
under the identification $\mr{det} (\mcF_\varTheta) = \mcT^{\otimes \frac{n (n-1)}{2}} \otimes \varTheta^{\otimes n}$ given by \eqref{e703}.
When $X^\mr{log}/S^\mr{log}$ arises from a pointed stable curve $\msX$, we will refer to any (dormant) $(\mr{GL}_n^{(\N)}, \vartheta)$-oper on $X^\mr{log}/S^\mr{log}$ as a {\bf (dormant) $(\mr{GL}_n^{(\N)}, \vartheta)$-oper on $\msX$}.
 \item[(ii)]
 Let $\DMO_\circ^\diamondsuit$ 
 and $\DMO_\bullet^\diamondsuit$ 
  be (dormant) $(\mr{GL}_n^{(\N)}, \vartheta)$-opers on $X^\mr{log}/S^\mr{log}$.
 We say that $\DMO_\circ^\diamondsuit$ is {\bf isomorphic to $\DMO_\bullet^\diamondsuit$}
 if the associated $\mr{GL}_n^{(\N)}$-opers $\DMO_\circ^{\diamondsuit \Rightarrow \heartsuit}$ and $\DMO_\bullet^{\diamondsuit \Rightarrow \heartsuit}$ are isomorphic.
 \end{itemize}
\ede
\SSP

We shall prove the following two propositions  concerning $(\mr{GL}_n^{(\N)}, \vartheta)$-opers.

\SSP
\bpr \label{P114}
Let $\msF^\heartsuit:= (\mcF, \DMO, \{ \mcF^j \}_j)$ be a $\mr{GL}_n^{(\N)}$-oper on $X^\mr{log}/S^\mr{log}$.
\begin{itemize}
\item[(i)]
There exists a triple
\begin{align} \label{WW301}
(\vartheta_0, \DMO_0^\diamondsuit, h_0)
\end{align}
consisting of 
an $n^{(\N)}$-theta characteristic $\vartheta_0 := (\varTheta_0, \DMO_{\vartheta_0})$ of $X^\mr{log}/S^\mr{log}$,
a $(\mr{GL}_n^{(\N)}, \vartheta_0)$-oper $\DMO_0^\diamondsuit$, and
an isomorphism of $\mr{GL}_n^{(\N)}$-opers $h_0 : \DMO_0^{\diamondsuit \Rightarrow \heartsuit} \isom \msF^\heartsuit$.
(Such an $n^{(\N)}$-theta characteristic is uniquely determined up to isomorphism in a certain sense.)
\item[(ii)]
(Recall that we have fixed a choice of an $n^{(\N)}$-theta characteristic $\vartheta$.)
Suppose that $n <p$.
Then, there exists a pair
\begin{align}
(\DMO^\diamondsuit, \msL)
\end{align}  
consisting of 
a $(\mr{GL}_n^{(\N)}, \vartheta)$-oper  $\DMO^\diamondsuit$ on $X^\mr{log}/S^\mr{log}$ and an invertible $\mcD^{(\N-1)}$-module $\msL$ with 
$\msF^\heartsuit_{\otimes \msL} \cong \DMO^{\diamondsuit \Rightarrow \heartsuit}$.
If, moreover, both $\vartheta$ and $\msF^\heartsuit$ are dormant,
we can choose such $\DMO^\diamondsuit$ and $\msL$ as being dormant.
\end{itemize}
\epr
\begin{proof}
First, we shall prove assertion (i).
Let us set $\varTheta_0 := \mcF^{n-1}$.
Consider the composite
\begin{align} \label{WW340}
h_0 : \mcF_{\varTheta_0} \left(= \mcD_{\leq n-1}^{(\N -1)} \otimes \varTheta_0 \right) \migi  \mcD_{}^{(\N -1)} \otimes \mcF \xrightarrow{\DMO} \mcF, 
\end{align}
where the first arrow arises from the inclusions $\mcD_{\leq n-1}^{(\N -1)} \migiincl  \mcD^{(\N -1)}$ and $\varTheta_0 \migiincl \mcF$.
Since $\msF^\heartsuit$ is a $\mr{GL}_n^{(\N)}$-oper, this composite turns out to be an isomorphism.
Denote by $\DMO_0^\diamondsuit$ (resp., $\DMO_{\vartheta_0}$) the $\mcD^{(\N -1)}$-module structure on  $\mcF_{\varTheta_0}$ (resp., $\mcT^{\otimes \frac{n (n-1)}{2}} \otimes \varTheta_0^{\otimes n} \left(= \mr{det} (\mcF_{\varTheta_0}) \right)$)  corresponding to $\DMO$ (resp., $\mr{det} (\DMO)$) via the isomorphism $h_0$  (resp., the isomorphism $\mr{det} (\mcF_{\varTheta_0}) \isom \mr{det} (\mcF)$ induced by $h_0$).
Then,  the resulting collection  $(\vartheta_0, \DMO_0^\diamondsuit, h_0)$, where $\vartheta_0 := (\varTheta_0, \DMO_{\vartheta_0})$,   determines the desired triple.

Next, we shall prove assertion (ii).
Write $\mcL := \varTheta \otimes \mcF^{n-1 \vee}$, and consider 
 the  composite of isomorphisms
\begin{align} \label{e710}
(\mcT^{\otimes \frac{n(n-1)}{2}} \otimes \varTheta^{\otimes n})   \otimes \mr{det}(\mcF)^\vee 
&\isom
(\mcT^{\otimes \frac{n(n-1)}{2}} \otimes \varTheta^{\otimes n}) \otimes (\mcT^{\otimes \frac{n(n-1)}{2}} \otimes (\mcF^{n-1})^{\otimes n})^\vee \\
& \isom \mcT^{\otimes \frac{n(n-1)}{2}} \otimes (\mcT^{\otimes \frac{n(n-1)}{2}})^\vee \otimes \varTheta^{\otimes n} \otimes (\mcF^{n-1 \vee})^{\otimes n} \notag \\
& \isom \left( (\varTheta \otimes \mcF^{n-1 \vee})^{\otimes n}=\right) \mcL^{\otimes n}, \notag
\end{align}
where the first arrow is the morphism induced by 
the isomorphism
$\mr{det}(\mcF) \isom \mcT^{\otimes \frac{n(n-1)}{2}} \otimes (\mcF^{n-1})^{\otimes n}$ in ~\cite[Eq.\,(501)]{Wak8} (applied to the $\mr{GL}_n$-oper determined by $\msF^\heartsuit$).
It follows from Proposition \ref{P129}, (i), that there exists a unique $\mcD^{(\N-1)}$-module structure $\DMO_\mcL$ on $\mcL$ whose $n$-th tensor product corresponds to $\DMO_\vartheta \otimes \mr{det}(\DMO)^\vee$ via \eqref{e710}.
Let us consider the composite
\begin{align} \label{e780}
\mcF_\varTheta \left(=\mcD_{\leq n-1}^{(\N-1)} \otimes \varTheta\right) &\isom \mcD_{\leq n-1}^{(\N-1)}  \otimes (\mcL\otimes \mcF^{n-1})  \\
& \xrightarrow{\mr{inclusion}}  \mcD^{(\N-1)} \otimes  (\mcL \otimes \mcF) \notag\\
&\xrightarrow{\DMO_\mcL \otimes \DMO} \mcL \otimes \mcF. \notag
\end{align}
Since $\msF^\heartsuit$ is a $\mr{GL}_n^{(\N)}$-oper, 
this composite turns out  to be an isomorphism.
Moreover,  the equality $\nabla_\vartheta = \mr{det}(\nabla_\mcL \otimes \nabla)$ holds under the identification
$\mcT^{\otimes \frac{n(n-1)}{2}} \otimes \varTheta^{\otimes n} = \mr{det}(\mcL \otimes \mcF)$ given by \eqref{e703} and 
 the determinant  of \eqref{e780}.
 Hence, if $\DMO^\diamondsuit$ denotes the $\mcD^{(\N-1)}$-module structure on $\mcF_\varTheta$ corresponding to $\DMO_\mcL \otimes \DMO$ via \eqref{e780},
 then the pair $(\DMO^\diamondsuit, \msL)$, where $\msL := (\mcL, \DMO_\mcL)$, 
 defines  the desired pair. 
  Also, the second  assertion follows from Proposition \ref{PPYY1}, (ii).
 This completes the proof  of the assertion.
\end{proof}
\SSP

\bpr \label{L30}
Suppose that $n < p$ and $\CH = 0$.
Let $\DMO^\diamondsuit_\circ$ and  $\DMO^\diamondsuit_\bullet$
 be $(\mr{GL}_n^{(\N)}, \vartheta)$-opers on $X^\mr{log}/S^\mr{log}$.
Write $\vartheta^{(1)}$ for  the $n^{(1)}$-theta characteristic obtained by reducing the level of $\vartheta$ to $1$.
Also, for each $\star =\circ, \bullet$,
denote by $\DMO^{(1)\diamondsuit}_\star$ the $(\mr{GL}_n^{(1)}, \vartheta^{(1)})$-oper induced by $\DMO_\star^\diamondsuit$.
Suppose that 
$\DMO^\diamondsuit_\circ$ is isomorphic to $\DMO^\diamondsuit_\bullet$ and
the equality $\DMO^{(1)\diamondsuit}_\circ = \DMO^{(1)\diamondsuit}_\bullet$ holds.
Then, we have  $\DMO^\diamondsuit_\circ = \DMO^\diamondsuit_\bullet$.
\epr
\begin{proof}
Let us take an isomorphism  of $\mr{GL}_n^{(\N)}$-opers
$h : \DMO^{\diamondsuit \Rightarrow \heartsuit}_\circ \isom \DMO_\bullet^{\diamondsuit \Rightarrow \heartsuit}$. 
By the assumption  $\DMO^{(1)\diamondsuit}_\circ = \DMO^{(1)\diamondsuit}_\bullet$,
 the isomorphism of $\mr{GL}_n^{(1)}$-opers
$h^{(1)} : \DMO^{(1)\diamondsuit \Rightarrow \heartsuit}_\circ \isom \DMO_\bullet^{(1)\diamondsuit \Rightarrow \heartsuit}$
induced by $h$ defines an automorphism of $\DMO_\circ^{(1)\diamondsuit\Rightarrow \heartsuit}$.
According to  ~\cite[Proposition 4.25]{Wak8},
$h\left( = h^{(1)}\right)$ coincides with the automorphism $\mr{mult}_a : \mcF_\varTheta \isom \mcF_\varTheta$  given by  multiplication by an element $a$ of $H^0 (X, \mcO_X^{\times})$.
The multiplication by  $a^n$ defines an isomorphism
$(\mr{det}(\mcF_\varTheta), \mr{det}(\DMO_\circ^\diamondsuit)) \isom (\mr{det}(\mcF_\varTheta), \mr{det}(\DMO_\bullet^\diamondsuit))$; this  
may be regarded as  an automorphism  of $(\mcT^{\otimes \frac{n(n-1)}{2}} \otimes \varTheta^{\otimes n}, \DMO_\vartheta)$ via \eqref{e703}.
Hence, we have $a^n \in H^0 (X, (F_{X/S}^{(\N)})^{-1}(\mcO_{X^{(\N)}}))$.
This implies $a \in H^0 (X, (F_{X/S}^{(\N)})^{-1}(\mcO_{X^{(\N)}}))$ because of  the assumption  $n < p$.
Since 
$\DMO_\circ^{\diamondsuit}$ is the unique $\mcD^{(\N -1)}$-module structure  on $\mcF_\varTheta$ compatible with $\DMO_\circ^{\diamondsuit}$ itself via $\mr{mult}_a \left(=h \right)$,
the equality $\DMO_\circ^{\diamondsuit} = \DMO_\bullet^{\diamondsuit}$  must be satisfied.
This completes the proof of the assertion.
\end{proof}

\LSP
\subsection{Representability of the moduli space} \label{SS100}

For  any morphism  $s : S' \migi S$ between flat $R$-schemes,  the pair of base-changes
\begin{align}
s^*(\vartheta)  \ \left(\text{or} \ \vartheta_{S'} \right):= (s^*(\varTheta), s^*(\DMO_\vartheta))
\end{align}
specifies  an $n^{(\N)}$-theta characteristic  of  the log curve $(S' \times_{S}X^\mr{log})/(S' \times_S S^\mr{log})$.
Hence, we obtain the $\mcS et$-valued contravariant functor
\begin{align} \label{e733}
\mcO p_{\diamondsuit, \vartheta} \ \left(\text{resp.,} \ \mcO p_{\diamondsuit, \vartheta}^\ZZZ  \right) : \mcS ch_{R}^\mr{flat}/S \migi \mcS et
\end{align}
on $\mcS ch_{R}^\mr{flat}/S$ which, to any $S$-scheme $s : S' \migi S$ in $\mr{Ob}(\mcS ch_{R}^\mr{flat}/S)$, assigns the set of isomorphism classes of $(\mr{GL}_n^{(\N)}, s^*(\vartheta))$-opers  (resp., dormant $(\mr{GL}_n^{(\N)}, s^*(\vartheta))$-opers) on $(S' \times_{S}X^\mr{log})/(S' \times_S S^\mr{log})$.

Suppose that $(1 <) n < p$, and 
let $\star$ denote either the absence or presence  of ``\,$\mr{Zzz...}$".
Then, the assignments $\DMO^\diamondsuit \mapsto (\DMO^{\diamondsuit \Rightarrow \heartsuit})^{\Rightarrow \spadesuit}$ and $\DMO^\diamondsuit \mapsto \DMO^{\diamondsuit \Rightarrow \heartsuit}$ define morphisms  of  functors
\begin{align} \label{e781}
\Lambda_{\diamondsuit \Rightarrow \spadesuit, \vartheta}^\star 
:
 \mcO p_{\diamondsuit, \vartheta}^\star \migi \mcO p_{\spadesuit}^\star \hspace{5mm} \text{and} \hspace{5mm}
 \Lambda_{\diamondsuit \Rightarrow \heartsuit, \vartheta}^\star : 
  \mcO p_{\diamondsuit, \vartheta}^\star \migi \mcO p_{\heartsuit}^\star,
\end{align}
respectively.
These morphisms  make the following diagram commute:
\begin{align} \label{E783}
\vcenter{\xymatrix@C=46pt@R=36pt{
\mcO p_{\diamondsuit,  \vartheta}^\star \ar[rr]^-{\Lambda_{\diamondsuit \Rightarrow \heartsuit, \vartheta}^\star}  \ar[rd]_-{\Lambda_{\diamondsuit \Rightarrow \spadesuit, \vartheta}^\star}&& \mcO p_{\heartsuit}^\star  \ar[ld]^-{\Lambda_{\heartsuit \Rightarrow \spadesuit}^\star}
\\
& \mcO p_{\spadesuit}^\star. &
}}
\end{align}
Moreover, we  can obtain the following assertion.

\SSP
\bt \label{P14}
Suppose that $n < p$, and 
let $\star$ denote either the absence or presence  of ``\,$\mr{Zzz...}$".
In the case of $\star = \mr{Zzz...}$, 
we assume that the fixed $n^{(\N)}$-theta characteristic $\vartheta$ is dormant.
Then, the morphisms of functors 
$\Lambda_{\heartsuit \Rightarrow \spadesuit}^\star$,
$\Lambda_{\diamondsuit \Rightarrow \spadesuit, \vartheta}^\star$, and 
$\Lambda_{\diamondsuit \Rightarrow \heartsuit, \vartheta}^\star$ are all isomorphisms.
\et
\begin{proof}
The assertion for $\Lambda_{\diamondsuit \Rightarrow \spadesuit, \vartheta}^\star$
follows immediately  from Proposition \ref{PPYY1}, (i) and (ii).
The surjectivity  of $\Lambda^\star_{\diamondsuit \Rightarrow \heartsuit, \vartheta}$ is a direct consequence of Proposition  \ref{P114}, (ii).
The injectivity of $\Lambda_{\diamondsuit \Rightarrow \heartsuit, \vartheta}^\star$ can be proved by an argument entirely similar to the argument in the injectivity of ``$\Lambda_\vartheta^{\diamondsuit \Rightarrow \heartsuit}$" discussed in ~\cite[Proposition 4.40]{Wak8}.
Moreover, these results together with the commutativity of \eqref{E783} deduce the remaining portion, i.e., $\Lambda_{\heartsuit \Rightarrow \spadesuit}^\star$ turns out to be an isomorphism.
\end{proof}
\SSP

\bco \label{UU39}
Suppose that $n < p$.
\begin{itemize}
\item[(i)]
For any (dormant) $\mr{PGL}_n^{(\N)}$-oper $\msE^\spadesuit$ on $X^\mr{log}/S^\mr{log}$,
there always exists  a (dormant) $\mr{GL}_n^{(\N)}$-oper $\msF^\heartsuit$ on $X^\mr{log}/S^\mr{log}$ with $\msF^{\heartsuit \Rightarrow \spadesuit} \cong \msE^\spadesuit$.
\item[(ii)]
Suppose further that the relative characteristic of $X^\mr{log}/S^\mr{log}$ is trivial.
Let $\msE^\spadesuit := (\mcE_B, \STR)$ be a dormant $\mr{PGL}_n^{(\N)}$-oper on $X^\mr{log}/S^\mr{log}$, and write
 $\mcE := \mcE_B \times^B \mr{PGL}_n$.
Then,  $(\mcE, \STR)$ is, Zariski locally on $X$, isomorphic to the trivial $(\N-1)$-PD stratified  $\mr{PGL}_n$-bundle $(X \times \mr{PGL}_n, \STR_\mr{triv})$ (cf. Example \ref{NN601}). 
\item[(iii)]
Suppose further that  $\ell =0$.
Let $\msE^\spadesuit_\circ$ and $\msE^\spadesuit_\bullet$ be $\mr{PGL}_n^{(\N)}$-opers on $X^\mr{log}/S^\mr{log}$ with $\msE^\spadesuit_\circ\cong \msE^\spadesuit_\bullet$, and
for each $\star = \circ, \bullet$, we denote by $\msE^{\spadesuit (1)}_\star$ the $\mr{PGL}_n$-oper induced by $\msE^\spadesuit_\star$.
Then, each isomorphism of $\mr{PGL}_n$-opers $\msE^{\spadesuit (1)}_{\circ} \xrightarrow{\sim} \msE^{\spadesuit (1)}_{\bullet}$ defines an isomorphism of $\mr{PGL}_n^{(\N)}$-opers $\msE^{\spadesuit}_{\circ} \xrightarrow{\sim} \msE^{\spadesuit}_{\bullet}$.
\end{itemize}
\eco
\begin{proof}
Assertion (i) follows from Proposition \ref{P204} and the surjectivity of $\Lambda^{^{(\mr{Zzz...})}}_{\diamondsuit \Rightarrow \spadesuit, \vartheta}$.
Assertion (ii) follows from  assertion (i) and the equivalence (a) $\Leftrightarrow$ (b)  obtained  in Proposition \ref{P016dd}, (ii).
Assertion (iii) follows from Proposition \ref{L30} together with the fact that  $\Lambda_{\diamondsuit \Rightarrow \spadesuit, \vartheta}$ is an isomorphism.
\end{proof} 
\SSP

By applying Theorem  \ref{P14} above, we can prove the representability of the functors $\mcO p_{\spadesuit}$ and $\mcO p_{\spadesuit}^\ZZZ$ for $\CH = 0$, as follows.

\SSP
\bt \label{P99}
Let $(g, r)$ be a pair of nonnegative integers with $2g-2+r >0$.
Suppose that $n< p$ and that  $X^\mr{log}/S^\mr{log}$ arises from an $r$-pointed stable curve 
 of genus $g$ over an $\mbF_p$-scheme $S$.
Then, both $\mcO p_{\spadesuit}$ and $\mcO p_{\spadesuit}^\ZZZ$ may be represented by (possibly empty)   affine $S$-schemes of finite type, and the natural  inclusion
$\mcO p_{\spadesuit}^\ZZZ \migiincl \mcO p_{\spadesuit}$ defines a closed immersion between $S$-schemes.
(In the rest of this manuscript, we will use the notations $\mcO p_{\spadesuit}$ and $\mcO p_{\spadesuit}^\ZZZ$
  for writing the schemes representing  the functors $\mcO p_{\spadesuit}$ and $\mcO p_{\spadesuit}^\ZZZ$, respectively.)
\et
\begin{proof}
Let us choose a dormant $n^{(\N)}$-theta characteristic $\vartheta := (\varTheta, \DMO_\vartheta)$ of $X^\mr{log}/S^\mr{log}$ (cf. Proposition \ref{P204}).
It
 induces a dormant  $n^{(1)}$-theta characteristic  $\vartheta^{(1)}$  via  reduction to level $1$.
 For each $M \in \{ 0, \N -1 \}$, 
 we consider the set-valued functor $H_M$ on the category of $S$-schemes given by 
 \begin{align}
 H_M (S') := \mr{Hom}_{\mcO_{X'}} \left( \mcD_{X'^{\mr{log}}/S'^{\mr{log}}, \leq p^{\N -1}}^{(M)} \otimes \mcF_{\varTheta_{S'}}, \mcF_{\varTheta_{S'}}\right)
 \end{align}
 for each $S$-scheme $S'$, 
 where $\varTheta_{S'}$ denotes the base-change of $\varTheta$ over $S'$, and $S'^{\mr{log}} := S' \times_S S^\mr{log}$,  $X'^{\mr{log}} := S' \times_S X^\mr{log}$.
Since $f$ is proper, flat, and of finite presentation, and the sheaves occurring here are finite locally free, the representability theorem for relative Hom functors shows that 
 $H_M$ can be represented by an affine $S$-scheme of finite presentation.
 The natural morphism $\mcD_{X^\mr{log}/S^\mr{log}, \leq p^{\N -1}}^{(0)} \otimes \mcF_{\varTheta} \migi \mcD_{X^\mr{log}/S^\mr{log}, \leq p^{\N -1}}^{(\N -1)} \otimes \mcF_{\varTheta}$  induces an $S$-morphism 
 \begin{align}
 \nu : H_{\N -1} \migi H_0.
 \end{align}
 For simplicity, we write $H := H_{\N -1}$.
 
The $\mcO_X$-algebra  $\mcD^{(\N -1)}_{X^\mr{log}/S^\mr{log}}$ is generated by 
the sections of  $\mcD^{(\N-1)}_{X^{\mr{log}}/S^{\mr{log}}, \leq p^{\N-1}}$, and
the multiplication in $\mcD^{(\N-1)}_{X^{\mr{log}}/S^{\mr{log}}}$ induces a surjection $T (\mcD^{(\N-1)}_{X^{\mr{log}}/S^{\mr{log}}, \leq p^{\N-1}}) \twoheadrightarrow \mcD^{(\N-1)}_{X^{\mr{log}}/S^{\mr{log}}}$, where $T(\mcD^{(\N-1)}_{X^{\mr{log}}/S^{\mr{log}}, \leq p^{\N-1}})$ denotes the tensor algebra of $\mcD^{(\N-1)}_{X^{\mr{log}}/S^{\mr{log}}, \leq p^{\N-1}}$ obtained by using its $\mcO_X$-bimodule structure.
Let $\mcK$ denote the kernel of this surjection.
Since $\mcD_{X^\mr{log}/S^\mr{log}}^{(\N -1)}$ is flat over $S$, the formation of $\mcK$ commutes with base-change over $S$-schemes.
 The base-change of $\mcK$ over $H$ is denoted by $\mcK_H$. 
 
 The universal family of morphisms
 $\mcD^{(\N -1)}_{X_H^\mr{log}/H^\mr{log}, \leq p^{\N -1}} \otimes \mcF_{\varTheta_H} \rightarrow \mcF_{\varTheta_H}$ parametrized by $H$ extends, in a natural manner, to a morphism $T (\mcD^{(\N-1)}_{X_H^{\mr{log}}/H^{\mr{log}}, \leq p^{\N-1}}) \otimes \mcF_{\varTheta_H} \rightarrow \mcF_{\varTheta_H}$.
 Then, there exists a closed subscheme $H'$ of $H$ where the restriction $\mcK_H \otimes \mcF_{\varTheta} \rightarrow \mcF_{\varTheta_H}$ of this morphism vanishes. 
Observe  that for  $S$-schemes $S'$, $\mcD^{(\N-1)}_{X'^{\mr{log}}/S'^{\mr{log}}}$-module structures   
  on $\mcF_{\varTheta_{S'}}$
 are 
  completely  
determined by its restrictions to $\mcD^{(\N-1)}_{X'^{\mr{log}}/S'^{\mr{log}}, \leq p^{\N-1}}$, and that
 an $\mcO_{X'}$-linear morphism $\mcD_{X'^{\mr{log}}/S'^{\mr{log}}, \leq p^{\N -1}}^{(N -1)} \otimes \mcF_{\varTheta_{S'}} \rightarrow \mcF_{\varTheta_{S'}}$ defines  a $\mcD_{X'^{\mr{log}}/S'^{\mr{log}}}^{(\N -1)}$-module structure on 
$\mcF_{\varTheta_{S'}}$ if and only if its classifying morphism $S' \rightarrow H$ factors through the closed immersion $H' \hookrightarrow H$. 
 Also,  denote by $H''$ the closed subscheme of $H'$ classifying $\mcD_{X'^{\mr{log}}/S'^{\mr{log}}}^{(\N -1)}$-module structures (for each $S$-scheme $S'$) whose determinants coincide with $\nabla_\vartheta$ via  \eqref{e703}.
 
Recall from ~\cite[Theorem 4.66]{Wak8}
that  $T := \mcO p_{\diamondsuit, \vartheta^{(1)}}$ (with  ``$\N$" taken to be $1$) may be represented by an affine $S$-scheme of finite  type.
Denote by 
$\msY:= (h : Y \migi T, \{ \sigma_{Y, i}\}_i)$  the base-change over $T$ of the pointed stable curve defining $X^\mr{log}/S^\mr{log}$.
 The  $(\mr{GL}_n^{(1)}, \vartheta^{(1)}_T)$-oper  on $\msY$ corresponding the universal family of $(n, \vartheta_T^{(1)})$-projective connections (cf. ~\cite[Definition 4.37, (ii)]{Wak8}) via the assignment ``$D^\clubsuit \mapsto D^{\clubsuit \Rightarrow \diamondsuit}$"  in ~\cite[Section 4.5.3]{Wak8}
 determines  an $\mcO_Y$-linear morphism
 \begin{align}
 \DMO^\diamondsuit : \mcD_{Y^\mr{log}/T^\mr{log}}^{(0)} 
 \otimes \mcF_{\varTheta_T} \migi \mcF_{\varTheta_T}.
 \end{align}
The restriction  of $\DMO^\diamondsuit$ to $\mcD_{Y^\mr{log}/T^\mr{log}, \leq p^{\N-1}}^{(0)} 
 \otimes \mcF_{\varTheta_T} \left(\subseteq \mcD_{Y^\mr{log}/T^\mr{log}}^{(0)} 
 \otimes \mcF_{\varTheta_T}\right)$ specifies a $T$-rational point  $\sigma$  of
 $H_0$.

By   Proposition  \ref{L30},
 the set of isomorphism classes of 
  $(\mr{GL}_n^{(\N)}, \vartheta_{S'})$-opers 
 on $X'^{\mr{log}}/S'^{\mr{log}}$ 
 for each $S$-scheme $S'$ corresponds bijectively to  the $S'$-rational points of the $S$-scheme $H'' \times_{H_0, \sigma}T$.
 In particular, $\mcO p_{\spadesuit}$ ($\cong \mcO p_{\diamondsuit, \vartheta}$  by Theorem  \ref{P14}) may  be represented by the  affine $S$-scheme  $H'' \times_{H_0, \sigma}T$ of finite type.
 This completes the proof  of the assertion for $\mcO p_{\spadesuit}$.
 
 The assertions for $\mcO p_{\spadesuit}^\ZZZ$ follow immediately from the fact just proved and  the definition of $p^\N$-curvature. 
\end{proof}
\SSP

\begin{rema}[Representability for higher rank cases] \label{WW14}
If the notion of a dormant $\mr{PGL}_n^{(\N)}$-oper can be defined  for a large $n$, 
 then 
 it will be   expected that   Theorem \ref{P99}  is   true even when $n \geq  p$.
Indeed, as proved in ~\cite[Theorem B]{Wak9},
there exists a canonical duality
between dormant $\mr{PGL}_n^{(\N)}$-opers and dormant
$\mr{PGL}_{p^\N -n}^{(\N)}$-opers.
 (Although  the discussion in ~\cite{Wak9} only deals with the case where $X$ is smooth, we can immediately extend  the duality  to pointed stable curves in the same manner.)
This result  and Theorem \ref{P99} together imply the representability of $\mcO p_{\spadesuit}^\ZZZ$ for $p^\N - p  < n < p^\N -1$.
\end{rema}
\SSP

Let $\vartheta' := (\varTheta', \nabla_{\vartheta'})$ be another $n^{(\N)}$-theta characteristic of $X^\mr{log}/S^\mr{log}$.
Then, there exists an invertible $\mcD^{(\N -1)}$-module $(\mcL, \nabla_\mcL)$ equipped with an isomorphism $\alpha_{\vartheta, \vartheta'} : \varTheta  \otimes \mcL \xrightarrow{\sim} \varTheta'$ such that
$\nabla_\vartheta \otimes \nabla_\mcL^{\otimes n}$ commutes with $\nabla_{\vartheta'}$ via the composite of  isomorphisms
\begin{align} \label{EE343}
\left(\mcT^{\otimes \frac{n (n-1)}{2}} \otimes \varTheta^{\otimes n}\right) \otimes \mcL^{\otimes n} \xrightarrow{\sim} \mcT^{\otimes \frac{n (n-1)}{2}} \otimes (\varTheta \otimes \mcL)^{\otimes n} \xrightarrow{\mr{id} \otimes \alpha_{\vartheta, \vartheta'}} \mcT^{\otimes \frac{n (n-1)}{2}} \otimes \varTheta'^{\otimes n}.
\end{align}

Now, let $\nabla^\diamondsuit$ be a $(\mr{GL}_n^{(\N)}, \vartheta)$-oper on $X^\mr{log}/S^\mr{log}$.
Consider the composite isomorphism
\begin{align} \label{EE344}
\mcF_{\varTheta'} \left(= \mcD_{\leq n-1}^{(\N -1)} \otimes \varTheta \otimes \mcL\right) &\xrightarrow{a \otimes b \otimes c \mapsto a \otimes c \otimes b} \mcD_{\leq n-1}^{(\N -1)} \otimes \mcL \otimes \varTheta \\
& \xrightarrow{\mr{sw}_{(\mcL, \nabla_\mcL)}^{n-1}\otimes \mr{id}_\varTheta}  \left( \mcL \otimes \mcD_{\leq n-1}^{(\N -1)} \otimes \varTheta  = \right) \mcL \otimes \mcF_\varTheta,
\end{align}
where  for each invertible $\mcD^{(\N -1)}$-module $(\mcN, \nabla_\mcN)$ and a nonnegative integer $j \geq 0$, let $\mr{sw}_{(\mcN, \nabla_\mcN)}^j$  denote the isomorphism $\mcD_{\leq j}^{(\N -1)} \otimes \mcN \left(= \mcD_{\leq j}^{(\N -1)}\otimes (\mcN \otimes \mcD_{\leq 0}^{(\N -1)}) \right) \xrightarrow{\sim} \mcN \otimes \mcD_{\leq j}$ obtained by restricting the $\mcD^{(\N -1)}$-action $\mcD^{(\N -1)} \otimes  (\mcN \otimes \mcD^{(\N -1)}) \rightarrow \mcN \otimes \mcD^{(\N -1)}$ arising from $\nabla_\mcN$ and the left $\mcD^{(\N -1)}$-action on $\mcD^{(\N -1)}$ itself.
The $\mcD^{(\N -1)}$-module structure $\nabla_\mcL \otimes \nabla^\diamondsuit$ corresponds to a $\mcD^{(\N -1)}$-module structure $\nabla^\diamondsuit_{\vartheta \Rightarrow \theta'}$ on $\mcF_{\varTheta'}$ via this composite.
Since  the filtration 
$\{ \mcF^j_{\varTheta'} \}_j$ commutes with $\{ \mcL \otimes\mcF^j_\varTheta\}_j$ via this composite,
we see that $\nabla^\diamondsuit_{\vartheta \Rightarrow \theta'}$ defines a $(\mr{GL}_n^{(\N)}, \vartheta')$-oper.
If both $\vartheta$ and $\vartheta'$ are dormant (which implies that $(\mcL, \nabla_\mcL)$ is dormant),
then $\nabla^\diamondsuit$ is dormant precisely when $\nabla^\diamondsuit_{\vartheta \Rightarrow \theta'}$ is dormant.

Let $\star$ be either the absense or presence of ``$\mr{Zzz}...$".
Then, the  assignment $\nabla^\diamondsuit \mapsto \nabla^\diamondsuit_{\vartheta \Rightarrow \theta'}$ commutes with base-change over $S$-schemes, and hence we obtain a morphism
\begin{align} \label{EE345}
\diamondsuit^\star_{\vartheta \Rightarrow \theta'} : \mcO p^\star_{\diamondsuit, \vartheta} \rightarrow \mcO p^\star_{\diamondsuit, \vartheta'}
\end{align}
over $S$.
Since  $\diamondsuit_{\vartheta' \Rightarrow \theta}^\star \circ \diamondsuit^\star_{\vartheta \Rightarrow \theta'}$ coincides with the identity morphism,
$\diamondsuit_{\vartheta \Rightarrow \theta'}$ turns out to be an isomorphism.
By construction, the following diagram is commutative:
\begin{align} \label{EE99}
\vcenter{\xymatrix@C=46pt@R=36pt{
\mcO p_{\diamondsuit,  \vartheta}^\star \ar[rr]^-{\diamondsuit_{\vartheta \Rightarrow \theta'}^\star}  \ar[rd]_-{\Lambda_{\diamondsuit \Rightarrow \spadesuit, \vartheta}^\star}&& \mcO p_{\diamondsuit, \theta'}^\star  \ar[ld]^-{\Lambda_{\diamondsuit \Rightarrow \spadesuit, \vartheta'}^\star}
\\
& \mcO p_{\spadesuit}^\star. &
}}
\end{align}

\LSP
\subsection{Universal moduli space over $\overline{\mcM}_{g, r}$} \label{SSuniv}

Denote by $\mcS ch_{R}^\mr{flat}$ the
category of flat $R$-schemes, where $R := \mbZ/p^{\CH +1}\mbZ$ .
Given a pair of nonnegative integers $(g, r)$  with $2g-2+r>0$, we shall 
 denote by
\begin{align} \label{eY401}
\mcO p_{n, \N, g,r, R},  \ \text{or simply} \ \mcO p_{g,r} 
\ \left(\text{resp.,} \  \mcO p_{n, \N, g,r, R}^\ZZZ,  \ \text{or simply} \ \mcO p^\ZZZ_{g,r}  \right)
\end{align}
the category over $\mcS ch_{R}^\mr{flat}$ defined as follows:
\begin{itemize}
\item
The objects are the pairs $(\msX, \msE^\spadesuit)$, where $\msX$ denotes an $r$-pointed stable curve of genus $g$ over a flat $R$-scheme $S$
  and $\msE^\spadesuit$ denotes
a $\mr{PGL}_n^{(\N)}$-oper 
(resp., a dormant $\mr{PGL}_n^{(\N)}$-oper)
 on $\msX$;
\item
The morphisms from $(\msX_\circ, \msE_\circ^\spadesuit)$ to $(\msX_\bullet, \msE_\bullet^\spadesuit)$ are morphisms $(\nu_S, \nu_X) : \msX_\circ \migi \msX_\bullet$ of $r$-pointed stable curves (cf. ~\cite[Definition 1.36, (ii)]{Wak8}) with $\msE^\spadesuit_\circ \cong \nu_S^*(\msE_\bullet^\spadesuit)$;
\item
The projection $\mcO p_{n, \N, g, r, R} \migi \mcS ch_{R}^\mr{flat}$ 
(resp., $\mcO p_{n, \N, g, r, R}^\ZZZ \migi \mcS ch_{R}^\mr{flat}$)
   is given by assigning, to each pair $(\msX, \msE^\spadesuit)$ as above,
the base scheme $S$ of $\msX$.
\end{itemize}

 Forgetting the data of $\mr{PGL}_n^{(\N)}$-opers yields the projection
 $\mcO p_{g, r} \rightarrow \overline{\mcM}_{g, r}$, which restricts to a morphism
 \begin{align} \label{YY144}
 \Pi_{n, \N, g,r, R} \ \left(\text{or simply}, \  \Pi_{g,r} \right) : \mcO p^\ZZZ_{g,r} \migi \overline{\mcM}_{g,r}.
 \end{align}
 By  Proposition \ref{P44}, 
$\mcO p_{g, r}$  (resp., $\mcO p_{g, r}^\ZZZ$)
 turns out to be  fibered in equivalence relations over $\overline{\mcM}_{g, r}$, i.e., specifies a set-valued  \'{e}tale sheaf on $\overline{\mcM}_{g, r}$.

Also, we obtain 
a projective system
\begin{align} \label{NN8123}
\cdots \migi \mcO p^\ZZZ_{n, \N, g,r, R} \migi \cdots \migi \mcO p^\ZZZ_{n, 2, g,r, R} \migi \mcO p^\ZZZ_{n, 1, g,r, R}
\end{align}
each of whose morphisms is given by reducing the level of dormant $\mr{PGL}_n$-opers.

By applying Theorem \ref{P99} to various pointed stable curves,
we obtain the following assertion.

\SSP
\bt \label{T40}
Suppose that $n < p$ and $R = \mbF_p$.
Then, the  fibered category $\mcO p_{g,r}$ may be  represented by a (possibly empty) Deligne-Mumford stack of finite type over $\mbF_p$ and the natural projection $\mcO p_{g,r} \migi \overline{\mcM}_{g,r}$ is (represented by schemes  and) affine.
Moreover, 
$\mcO p_{g,r}^\ZZZ\!$ may be represented by a closed substack of $\mcO p_{g,r}$.
In particular, the projection $\Pi_{g, r}$ and the morphisms in \eqref{NN8123} are (represented by schemes and) affine.
\et
\SSP

Let us take a 
 pair  $(\msX, \msE^\spadesuit)$  classified by  $\mcO p^\ZZZ_{n, 1, g,r, \mbZ/p^\N \mbZ}$,  where $\msE^\spadesuit := (\mcE_B, \STR)$.
 Then, the  pair 
 \begin{align} \label{YY259}
 {^{\Diag}}\!\!\msE^\spadesuit := (\mcE_{B, 0}, {^{\Diag}}\!\!\STR)
 \end{align}
 consisting of 
the reduction $\mcE_{B, 0}$  modulo $p$  of $\mcE_B$ and the diagonal reduction ${^{\Diag}}\!\!\STR$   of $\STR$ (cf. \eqref{eeYY3}) specifies   a dormant $\mr{PGL}_n^{(\N)}$-oper on $\msX_0$ (= the reduction modulo $p$ of $\msX$).
 
 \SSP
\bde \label{YY300}
We shall refer to ${^{\Diag}}\!\!\msE^\spadesuit$ as the {\bf diagonal reduction} of $\msE^\spadesuit$.
\ede
\SSP

 The resulting assignment $(\msX, \msE^\spadesuit) \mapsto (\msX_0,  {^{\Diag}}\!\!\msE^\spadesuit)$ defines 
a functor
\begin{align}  \label{YY260}
{^{\Diag}}\!\!(-) : \mcO p^\ZZZ_{n, 1, g,r, \mbZ/p^\N \mbZ} \migi  \mcO p^\ZZZ_{n, \N, g,r, \mbF_p}.
\end{align}
Moreover,  the square diagram of categories 
\begin{align} \label{YY203}
\vcenter{\xymatrix@C=46pt@R=36pt{
\mcO p^\ZZZ_{n, 1, g,r, \mbZ/p^\N \mbZ} \ar[r]^-{{^{\Diag}}\!\!(-)} \ar[d]_-{\mr{projection}} & \mcO p^\ZZZ_{n, \N, g,r, \mbF_p} \ar[d]^-{\mr{projection}}
\\
\mcS ch_{\mbZ/p^\N \mbZ}^{\mr{flat}}\ar[r] & \mcS ch_{\mbF_p}^{\mr{flat}}
}}
\end{align}
is $1$-commutative, 
where the lower horizontal arrow denotes the functor given by base-change along the closed immersion $\mr{Spec}(\mbF_p) \migiincl \mr{Spec}(\mbZ/p^\N \mbZ)$.

\LSP
\subsection{Correspondence with projective connections} \label{SS0454}

We here describe the bijective correspondence between $\mr{GL}_n$-opers and ordinary linear differential operators.
The case in characteristic $p$ was already discussed in the previous study.
As most  of the arguments in this section are exactly the same as those made in  ~\cite[Section 4]{Wak8}, 
 the details of the proofs are omitted.

Let  $S^\mr{log}$ be an fs log scheme whose underlying scheme $S$ is flat over $R$ and  $f^\mr{log} : X^\mr{log} \migi S^\mr{log}$ a log curve over $S^\mr{log}$.
Also, let $\mcL$ and  $\mcN$ be 
 line bundles  on $X$, and fix an integer $n$ with $1 < n <p$.
 In Sections \ref{SS0454}-\ref{Seer54}, we suppose that $\N =1$ and  $R := \mbZ/p^{\CH +1}\mbZ$.

\SSP
\bde \label{Def5892}
By  an {\bf $n$-th order  (linear) differential operator} (over $S^\mr{log}$) from $\mcL$ to $\mcN$, we mean 
an $f^{-1}(\mcO_S)$-linear morphism $\mpD : \mcL \rightarrow \mcN$ locally expressed, by using some local identifications $\mcL = \mcN = \mcO_X$ and  the local basis $\{ \partial^j \left(:= \partial^{\langle j \rangle}\right)\}_{j}$ of $\mcD^{(0)}$ associated to some logarithmic coordinate of $X^\mr{log}/S^\mr{log}$, as 
\begin{align} \label{eQ1}
\mpD  : v \mapsto \mpD (v) = \sum_{j = 0}^n a_j \cdot \partial^{j} (v),
\end{align}
where $a_0, \cdots, a_n$'s are local sections of $\mcO_X$ with $a_n \neq 0$.
\ede
\SSP

Denote by 
\begin{align} \label{eQ2}
\mcD {\it iff}_{\!\leq n} (\mcL, \mcN)
\end{align} 
the Zariski sheaf on $X$ consisting of locally defined $m$-th order differential operators from $\mcL$ to $\mcN$ with $m \leq n$.
It is verified that $\mcD {\it iff}_{\!\leq n} (\mcL, \mcN)$ forms  a  subsheaf of $\mcH om_{f^{-1}(\mcO_S)} (\mcL, \mcN)$.
Moreover, the composition  with the $f^{-1}(\mcO_S)$-linear morphism $\mcN \otimes \mcD^{(0)}_{\leq n} \rightarrow \mcN$ given by $v \otimes D \mapsto v \otimes D (1)$ yields an isomorphism
\begin{align} \label{eQ3}
\mcH om_{\mcO_X} (\mcL, \mcN \otimes \mcD^{(0)}_{\leq n}) \isom \mcD {\it iff}_{\leq n} (\mcL, \mcN) \left(\subseteq \mcH om_{f^{-1}(\mcO_S)} (\mcL, \mcN) \right).
\end{align}
This isomorphism for $\mcL = \mcN = \mcO_X$ is nothing but (the restriction to $\mcD_{\leq n}^{(0)}$ of) the $\mcD^{(0)}$-module structure $\nabla^{(0)}_{X, \mr{triv}}$ on $\mcO_X$ defined in \eqref{dGG6}.

\SSP
\bde \label{Def898}
Let $\mpD$ be 
  an $n$-th order differential operator $\mcL \rightarrow \mcN$.
The  composite
\begin{align} \label{eQ7}
\Sigma_1 (\mpD) : \mcL \xrightarrow{D} \mcN  \otimes \mcD^{(0)}_{\leq n} \migisurj \left(\mcN \otimes (\mcD^{(0)}_{\leq n}/\mcD^{(0)}_{\leq n-1}) = \right) \mcN \otimes \mcT^{\otimes n},
\end{align}
where $D$ denotes the morphism 
corresponding to $\mpD$ via \eqref{eQ3},
is called the {\bf principal symbol} of $\mpD$.
Also, we shall say that $\mpD$ {\bf has unit principal symbol}
if $\mcL = \mcN \otimes \mcT^{\otimes n}$ and $\Sigma_1 (\mpD)$ coincide with the identity morphism $\mr{id}_\mcL$ of $\mcL$.
\ede
\SSP

Let us fix a line bundle $\varTheta$ on $X$,
and 
consider an $n$-th order differential operator  $\mpD^\clubsuit : \varTheta^\vee \rightarrow \Omega^{\otimes n}\otimes \varTheta^\vee$ from $\varTheta^\vee$ to $\Omega^{\otimes n} \otimes \varTheta^\vee$ with unit principal symbol.
 It defines an $\mcO_X$-linear morphism $\varTheta^\vee \rightarrow (\Omega^{\otimes n} \otimes \varTheta^\vee) \otimes \mcD_{\leq n}^{(0)}$ by  \eqref{eQ3} and moreover determines an $\mcO_X$-linear morphism
 $D^{\CC} : \mcT^{\otimes n} \otimes \BB \migi \mcD_{\leq n}^{(0)} \otimes \BB$
   via  the composite  of natural isomorphisms
\begin{align} \label{WW248}
\mcH om_{\mcO_X} (\BB^\vee, (\Omega^{\otimes n} \otimes \BB^\vee) \otimes \mcD^{(0)}_{\leq n}) &\isom \Omega^{\otimes n} \otimes \BB^\vee \otimes \mcD_{\leq n}^{(0)} \otimes \BB  \\
&\isom \mcH om_{\mcO_X} (\mcT^{\otimes n} \otimes \BB, \mcD_{\leq n}^{(0)}\otimes \BB). \notag
\end{align}
The quotient  $(\mcD^{(0)}\otimes \BB)/\langle \mr{Im}(D^{\CC})\rangle$
of $\mcD^{(0)}\otimes \BB$ by the $\mcD^{(0)}$-submodule generated by the image 
 of  $D^{\CC}$ specifies a left $\mcD^{(\N -1)}$-module.
  Since  $\Sigma_1 (\mpD^\clubsuit)=1$,  the composite
\begin{align}  \label{GL13}
 \mcF_\BB
 \xrightarrow{\mr{inclusion}}
  \mcD^{(0)} \otimes \BB
   \xrightarrow{\mr{quotient}}
   \mcD^{(0)}\otimes \BB/\langle \mr{Im}(D^{\! \CC})
   \rangle\end{align}
(cf. \eqref{e607} for the definition of $\mcF_\BB$) turns out to be  an isomorphism of  $\mcO_X$-modules.
The  $\mcD^{(0)}$-action on  $(\mcD^{(0)}\otimes \BB)/\langle \mr{Im}(D^{ \CC})\rangle$ is transposed into  
an $S^\mr{log}$-connection
\begin{align}
\label{GL14}  \mpD^{\clubsuit \Rightarrow\diamondsuit} :  \mcF_\varTheta \rightarrow \Omega \otimes \mcF_\varTheta
    \end{align}
on  $\mcF_\BB$ via this composite  isomorphism.
The collection of data
\begin{align} 
\label{GL61}
 \mpD^{\clubsuit \Rightarrow\heartsuit} := (\mcF_\BB,
  \mpD^{\clubsuit \Rightarrow\diamondsuit}, \{ \mcF^j_\BB
  \}_{j=0}^n) 
    \end{align}
 forms a $\mr{GL}_n$-oper on $X^\mr{log}/S^\mr{log}$, and 
the 
determinant
\begin{align} \label{eQ8}
\Sigma_2 (\mpD^\clubsuit) := \mr{det}(\mpD^{\clubsuit \Rightarrow \diamondsuit})
\end{align}
may be regarded as an $S^\mr{log}$-connection 
 on $\mcT^{\otimes \frac{n(n-1)}{2}} \otimes \varTheta^{\otimes n}$ via \eqref{e703}.

\SSP
\bde \label{WW288}
\begin{itemize}
\item[(i)]
We shall refer to $\Sigma_2 (\mpD^\clubsuit)$ as the {\bf subprincipal symbol} of $\mpD^{\clubsuit}$.
Also,  under the assumption that $\varTheta = \varTheta_0^{\otimes (n-1)}$ for some theta characteristic $\varTheta_0$ of $X^\mr{log}/S^\mr{log}$ (cf. Example \ref{Ex20}), we say that $\mpD^\clubsuit$ {\bf has vanishing subprincipal symbol} if  $\Sigma_2 (\mpD^\clubsuit)$ coincides with the $\nabla_{\vartheta_0}$ defined in \eqref{TTT56}.
\item[(ii)]
Suppose further that we are given an $n^{(1)}$-theta characteristic  $\vartheta := (\varTheta, \nabla_\vartheta)$  of $X^\mr{log}/S^\mr{log}$.
We say that $\mpD^\clubsuit$ is an {\bf $(n, \vartheta)$-projective connection} on $X^\mr{log}/S^\mr{log}$ if the equality $\Sigma_2 (\mpD^\clubsuit) = \nabla_\vartheta$ holds, or equivalently,  $\mpD^{\clubsuit \Rightarrow\diamondsuit}$ specifies a $(\mr{GL}_n, \vartheta)$-oper on $X^\mr{log}/S^\mr{log}$.
(In particular, if $\varTheta = \varTheta_0^{\otimes (n-1)}$ for $\varTheta_0$ as above, then $\mpD^\clubsuit$ has vanishing subprincipal symbol if and only if it is an $(n, \vartheta_0)$-projective connection, where $\vartheta_0$ is as defined in \eqref{TTT56}.)
\end{itemize}
\ede
\SSP

 In what follows, we  fix an $n^{(1)}$-theta characteristic $\vartheta := (\BB, \DMO_\vartheta)$ of $X^\mr{log}/S^\mr{log}$.
Denote by
\begin{align} \label{WW249}
\mcD i\!f\!f_{\clubsuit, \bb} : \mcS ch_{R}^{\mr{flat}}/S \migi \mcS et
\end{align}
 the contravariant functor on $\mcS ch_{R}^{\mr{flat}}/S$ which, to any $S$-scheme $s : S' \migi S$ in $\mr{Ob}(\mcS ch_{R}^{\mr{flat}}/S)$, 
assigns the set of 
$(n, s^*(\vartheta))$-projective connections on $(S' \times_S X^\mr{log})/(S' \times_S S^\mr{log})$ 
(where $s^*(-)$ denotes the result of base-changing  along $s$).

Then, the  assignment $\mpD^\clubsuit \mapsto \mpD^{\clubsuit \Rightarrow \heartsuit}$
 determines a morphism of functors
\begin{align} 
\label{GL62}  \Lambda_{\clubsuit \Rightarrow \diamondsuit, \bb} :
\mcD i\!f\!f_{\clubsuit, \bb}
\migi \mcO p_{\diamondsuit, \vartheta}. 
\end{align}
Just as in the proof  of ~\cite[Proposition 4.28]{Wak8}, this morphism is verified to be an isomorphism.
Thus, by an argument similar to the proof of ~\cite[Theorem 4.66]{Wak8},
we obtain the following assertion. 

\SSP
\bpr \label{Prop7883}
Suppose that $X^\mr{log}/S^\mr{log}$ arises from an $r$-pointed stable curve of genus $g$ over a flat $R$-scheme $S$ with $2g-2+r >0$.
We shall set 
\begin{align} \label{eQ44}
\mcV_\varTheta := \mcH om_{\mcO_X} (\varTheta^\vee, (\Omega^{\otimes n} \otimes \varTheta^\vee) \otimes \mcD^{(0)}_{\leq (n-2)}).
\end{align}
(Similarly to ~\cite[Lemma 4.67, (ii)]{Wak8}, the direct image  $f_* (\mcV_\varTheta)$ is verified to be a vector bundle on $S$ of rank $(n^2 -1) (g-1) + \frac{(n^2 + n -2)r}{2}$.)
Then, the  functor $\mcO p_{\diamondsuit, \vartheta}$ (which is isomorphic to both $\mcO p_\heartsuit$ and $\mcO p_{\spadesuit}$ via $\Lambda_{\diamondsuit \Rightarrow \heartsuit, \vartheta}$ and $\Lambda_{\diamondsuit \Rightarrow \spadesuit, \vartheta}$, respectively) forms a $f_* (\mcV_\varTheta)$-torsor. 
\epr

Let us keep the notation in the discussion following Remark \ref{WW14} with $\N =1$.
Let $\mpD^{\clubsuit}$ an $(n, \vartheta)$-projective connection  on  $X^\mr{log}/S^\mr{log}$, regarded as an $\mcO_X$-linear morphism  $\varTheta^\vee \rightarrow (\Omega^{\otimes n} \otimes \varTheta^\vee) \otimes \mcD_{\leq n}^{(0)}$.
Consider the composite
\begin{align}
\mpD_{\vartheta \Rightarrow \theta'}^\clubsuit : \varTheta'^\vee &\xrightarrow{\sim} \varTheta^\vee \otimes \mcL^\vee \\
& \xrightarrow{\mpD^\clubsuit \otimes \mr{id}_{\mcL^\vee}} (\Omega^{\otimes n} \otimes \varTheta^\vee) \otimes \mcD_{\leq n}^{(0)} \otimes \mcL^\vee \\
& \xrightarrow{\mr{id}_{\Omega^{\otimes n} \otimes \varTheta^\vee} \otimes \mr{sw}^n_{(\mcL^\vee, \nabla_\mcL^\vee)}}
(\Omega^{\otimes n} \otimes \varTheta^\vee) \otimes \mcL^\vee \otimes \mcD_{\leq n}^{(0)} \\
& \xrightarrow{\sim} (\Omega^{\otimes n} \otimes \varTheta'^\vee) \otimes \mcD_{\leq n}^{(0)}.
\end{align}
It is immediately verified that this forms an $(n, \vartheta')$-projective connection.

The assignment $\mpD^\clubsuit \mapsto \mpD^\clubsuit_{\vartheta \Rightarrow \theta'}$ commutes with base-change over $S$-schemes, so  we obtain a morphism 
\begin{align} \label{EE1001}
\clubsuit_{\vartheta \Rightarrow \vartheta'} :  \mcD i\!f\!f_{\clubsuit, \bb} \rightarrow \mcD i\!f\!f_{\clubsuit, \vartheta'}
\end{align}
over $S$.
Since $\clubsuit_{\vartheta' \Rightarrow \vartheta} \circ\clubsuit_{\vartheta \Rightarrow \vartheta'}$ coincides with identity morphism,
the morphism $\clubsuit_{\vartheta \Rightarrow \vartheta'}$ turns out to be an isomorphism.
Moreover, by construction, the following square diagram is commutative:
\begin{align} \label{EE99e}
\vcenter{\xymatrix@C=46pt@R=36pt{
 \mcD i\!f\!f_{\clubsuit, \bb} \ar[d]^-{\wr}_-{\Lambda_{\clubsuit \Rightarrow \diamondsuit, \vartheta}}  \ar[r]_-{\sim}^-{\clubsuit_{\vartheta \Rightarrow \vartheta'}}&  \mcD i\!f\!f_{\clubsuit, \bb'} \ar[d]_-{\wr}^-{\Lambda_{\clubsuit \Rightarrow \diamondsuit, \vartheta'}}
 \\
 \mcO p_{\diamondsuit, \vartheta} \ar[r]^-{\sim}_-{\diamondsuit_{\vartheta \Rightarrow \theta'}}& \mcO p_{\diamondsuit, \vartheta'}.
}}
\end{align}

\LSP
\subsection{Projective connections having a full set of root functions} \label{Seer54}

In this subsection,
 we suppose that the log structures of $X^\mr{log}$ and $S^\mr{log}$ are trivial, or more generally the relative characteristic of $X^\mr{log}/S^\mr{log}$ is trivial  (which implies that $X/S$ is smooth and $\mcD^{(0)} = \mcD_{X/S}^{(0)}$).
Also, suppose that $\vartheta$ is dormant.
Let $\mcL$, $\mcN$ be line bundles on $X$ and $\mpD$ an $n$-th order differential operator from $\mcL$ to $\mcN$.
Since 
$\mpD$ is $\mcS ol (\DMO_{X, \mr{triv}}^{(0)})$-linear, 
the kernel $\mr{Ker}(\mpD)$ forms an $\mcS ol (\DMO_{X, \mr{triv}}^{(0)})$-submodule of $\mcL$.

\SSP
\bde \label{D01fgh}
We shall say that $\mpD$   {\bf has  a full set of  root functions}
(or the differential equation associated to $\mpD$  {\bf  has a full set of solutions})  if  the $\mcS ol (\DMO_{X, \mr{triv}}^{(0)})$-module  $\mr{Ker} (\mpD)$  is locally free of rank $n$.
\ede
\SSP

The condition of having 
 a full set of root functions  is closed under base-change, so
 we obtain the subfunctor
\begin{align} \label{QQ392}
\mcD i\!f\!f_{\clubsuit, \bb}^\mr{full}
\end{align}
 of $\mcD i\!f\!f_{\clubsuit, \bb}$ consisting of
 $(n, \bb)$-projective connections
 having a full set of root functions.
 
 By an argument similar to the proof of ~\cite[Proposition 4.65]{Wak8}, we see that
  $\Lambda_{\clubsuit \Rightarrow \diamondsuit, \bb}$ restricts to an isomorphism 
 \begin{align} \label{QQ379}
 \Lambda_{\clubsuit \Rightarrow \diamondsuit, \vartheta}^\mr{full} : \mcD i\!f\!f_{\clubsuit, \bb}^\mr{full} \isom \mcO p^\ZZZ_{\diamondsuit, \bb}.
 \end{align}
(In fact, the argument in ~\cite{Wak8} can be applied by using  the equivalence (a) $\Leftrightarrow$ (c) obtained in 
Proposition \ref{P016dd}, (ii), instead of ~\cite[Proposition 4.60]{Wak8}.)
In particular, by combining with Theorem \ref{P14}, 
we obtain the following sequence consisting of isomorphisms between functors:
\begin{align}  \label{GGGjjG}
\vcenter{\xymatrix@C=46pt@R=36pt{
\mcD i\!f\!f_{\clubsuit, \bb}^\mr{full}
\ar[rd]^-{\sim}_-{\Lambda_{\clubsuit \Rightarrow \diamondsuit, \vartheta}^\mr{full}} && \mcO p^\ZZZ_\heartsuit \ar[rd]^-{\sim}_-{\Lambda_{\heartsuit \Rightarrow \spadesuit}^\ZZZ}&
\\
& \mcO p^\ZZZ_{\diamondsuit, \bb}  \ar[ru]_-{\sim}^-{\Lambda_{\diamondsuit \Rightarrow \heartsuit, \vartheta}^\ZZZ}&& \mcO p^\ZZZ_{\spadesuit}
}}
\end{align}
(cf. ~\cite[Theorem D]{Wak8}).
Also, 
the following assertion holds.

\SSP 
 \bco \label{WW456}
Let $\varTheta_0$ be a theta characteristic of $X/S$.
 Then, there exists a bijective correspondence between the following two sets:
 \begin{itemize}
 \item
  The set of isomorphism classes of dormant $\mr{PGL}_n$-opers on $X/S$;
  \item
 The set of   $n$-th order differential  operators on $\varTheta_0^{\otimes (1-n)}$ having a full set of root functions with unit principal symbol and vanishing subprincipal symbol.
   \end{itemize}
 \eco
\begin{proof}
The  assertion follows from 
the isomorphisms  $\Lambda_{\diamondsuit \Rightarrow \spadesuit, \vartheta}^\ZZZ$ (obtained in Theorem \ref{P14}) and  $\Lambda_{\clubsuit \Rightarrow \diamondsuit, \vartheta}^\mr{full}$ (constructed above), where ``$\vartheta$" is taken to be $\vartheta_0$.
\end{proof}

\LSP
\subsection{Normalization of $\mr{PGL}_n^{(\N)}$-opers} \label{SSnormal}

Let  $X^\mr{log}/S^\mr{log}$ be as in Section \ref{SS0454} (but $\N$ is taken to be an arbitrary positive integer) and let $\vartheta := (\varTheta, \nabla_{\vartheta})$ an $n^{(\N)}$-theta characteristic of $X^\mr{log}/S^\mr{log}$.
Denote by $\vartheta^{(1)}$ the level-$1$ reduction of $\vartheta$.
Also, denote by  ${^\dagger}\mcE_B$ the canonical  $B$-bundle on $X$ defined in ~\cite[Eq.\,(605)]{Wak8}, and set ${^\dagger}\mcE := {^\dagger}\mcE_B \times^B \mr{PGL}_n$.
By definition,
the $\mr{PGL}_n$-bundle ${^\dagger}\mcE$ is obtained as the projectivization of  the vector bundle $\mcF_\varTheta$, and the complete flag $\{ \mcF_\varTheta^j \}_j$ corresponds to ${^\dagger}\mcE_B$.

\bde[cf. ~\cite{Wak8}, Definition 4.53, for the case $\N =1$] \label{Defg4}
\begin{itemize}
\item[(i)]
Let $\nabla^\diamondsuit$ be a $(\mr{GL}_n^{(\N)}, \vartheta)$-oper on $X^\mr{log}/S^\mr{log}$.
We say that $\nabla^\diamondsuit$ is {\bf normal} if 
its level-$1$ reduction $\nabla^{\diamondsuit (1)}$ satisfies $\nabla^{\diamondsuit (1)} = \mpD^{\clubsuit \Rightarrow \diamondsuit}$
for some  $(n, \vartheta^{(1)})$-projective connection $\mpD^\clubsuit$.
\item[(ii)]
Let $\msE^\spadesuit := (\mcE_B, \phi)$ be a $\mr{PGL}_n^{(\N)}$-oper on $X^\mr{log}/S^\mr{log}$.
We say that $\msE^\spadesuit$ is {\bf normal}
if 
 $\msE^\spadesuit =\nabla^{\diamondsuit \Rightarrow \spadesuit}$ for some normal $(\mr{GL}_n^{(\N)}, \vartheta)$-oper $\nabla^\diamondsuit$.
(In particular, the underlying $B$-bundle of a normal $\mr{PGL}_n^{(\N)}$-oper is ${^\dagger}\mcE_B$. Also,  by the commutativity of \eqref{EE99}, \eqref{EE99e}, and the isomorphism \eqref{EE344}, the notion of normality does not depend on the choice of $\vartheta$.)
\end{itemize}
\ede

The following assertion generalizes ~\cite[Proposition 4.55]{Wak8}.

\begin{prdef} \label{Prop933}
For each $\mr{PGL}_n^{(\N)}$-oper $\msE^\spadesuit := (\mcE_B, \phi)$ on $X^\mr{log}/S^\mr{log}$,
there exists uniquely a pair
\begin{align} \label{EE348}
({^\dagger}\msE^\spadesuit, \mr{nor}_{\msE^\spadesuit})
\end{align}
consisting of a normal  $\mr{PGL}_n^{(\N)}$-oper ${^\dagger}\msE^\spadesuit$ on $X^\mr{log}/S^\mr{log}$ and an isomorphism of $\mr{PGL}_n^{(\N)}$-opers ${^\dagger}\msE^\spadesuit \xrightarrow{\sim} \msE^\spadesuit$.
We refer to  the pair  $({^\dagger}\msE^\spadesuit, \mr{nor}_{\msE^\spadesuit})$ as the {\bf normalization} of $\msE^\spadesuit$.
\end{prdef}
\begin{proof}
Denote by $\msE^{\spadesuit (1)}$ the $\mr{PGL}_n$-oper induced from $\msE^\spadesuit$.
Since  the morphism $\Lambda_{\clubsuit \Rightarrow \diamondsuit, \vartheta}$ and $\Lambda_{\diamondsuit \Rightarrow \spadesuit, \vartheta}$ are isomorphisms, 
there exists a normal $\mr{PGL}_n$-oper ${^\dagger}\msE^{\spadesuit (1)}$ on $X^\mr{log}/S^\mr{log}$ and an isomorphism of $\mr{PGL}_n$-opers $\mr{nor}_{\msE^\spadesuit} : {^\dagger}\msE^{\spadesuit (1)} \xrightarrow{\sim} \msE^{\spadesuit (1)}$.
By Proposition \ref{P44}, such a pair is uniquely determined.
Let  ${^\dagger}\phi$ denote the $(\N -1)$-PD stratification on ${^\dagger}\mcE$ corresponding to $\phi$ via $\mr{nor}_{\msE^\spadesuit}$.
Then,  ${^\dagger}\msE^\spadesuit := ({^\dagger}\mcE_B, {^\dagger}\phi)$ forms a $\mr{PGL}_n^{(\N)}$-oper, and  the resulting pair $({^\dagger}\msE^\spadesuit, \mr{nor}_{\msE^\spadesuit})$ satisfies  required condition.
The uniqueness follows from Proposition \ref{P44} and Corollary \ref{UU39}, (iii).
This completes the proof.
\end{proof}

\begin{rema}[Compatibility with diagonal reduction] \label{Rem2222}
Let $(\msX, \msE^\spadesuit)$ be a pair classified by $\mcO p_{n, 1, g, r, \mbZ/p^\N}^{^\mr{Zzz...}}$.
By construction, we see that if $\msE^\spadesuit$ is normal, its diagonal reduction  ${^{\Diag}}\!\!\msE^\spadesuit$ is normal.
Also, let $({^\dagger}\msE^\spadesuit, \mr{nor}_{\msE^\spadesuit})$ be the normalization of $\msE^\spadesuit$.
Then, the  isomorphism $\mr{nor}_{\msE^\spadesuit}$ yields an isomorphism of dormant $\mr{PGL}_n^{(\N)}$-opers ${^{\Diag}}\!\!\mr{nor}_{\msE^\spadesuit} : {^{\Diag}}\!\!{^\dagger}\msE^\spadesuit \xrightarrow{\sim} {^{\Diag}}\!\!\msE^\spadesuit$, and the resulting pair $({^{\Diag}}\!\!{^\dagger}\msE^\spadesuit, {^{\Diag}}\!\!\mr{nor}_{\msE^\spadesuit})$ defines the normalization of ${^{\Diag}}\!\!\msE^\spadesuit$, i.e., ${^{\Diag}}\!\!\mr{nor}_{\msE^\spadesuit} = \mr{nor}_{{\tiny{^{\Diag}}}\!\!\msE^\spadesuit}$.
\end{rema}

\vspace{10mm}
\section{Radii of dormant $\mr{PGL}_n^{(\N)}$-opers}\label{S0133}

This section discusses  the radii of  dormant $\mr{PGL}_n^{(\N)}$-opers, which  determines a sort of boundary condition to glue together  dormant $\mr{PGL}_n^{(\N)}$-opers on pointed curves along the fibers over the points of attachment.
After proving a certain factorization property of the moduli space
$\mcO p^\ZZZ_{g, r}$ 
 in accordance with  the data of radii,
we obtain the nonemptiness of that space (cf. Corollary \ref{C56}).

Let us fix 
  a pair of nonnegative integers $(g,r)$ with $2g-2 +r >0$. 
  Also, fix  a pair of nonnegative integers $(\CH, \M)$  with  $\CH=0$ or $\M =0$.

\LSP
\subsection{Monodromy/Exponent at a marked point} \label{SS045}

Let $S$ be a flat scheme over $\mbZ/p^{\CH +1}\mbZ$, and  let $\msX := (f: X \migi S, \{ \sigma_i \}_i)$ be  an $r$-pointed stable curve of genus $g$ over $S$.
In particular, we obtain a log curve $X^\mr{log}/S^\mr{log}$.
For simplicity, we write  $\Omega := \Omega_{X^\mr{log}/S^\mr{log}}$, $\mcT_{} := \mcT_{X^\mr{log}/S^\mr{log}}$, and $\mcD^{(\M)} := \mcD_{X^\mr{log}/S^\mr{log}}^{(\M)}$.

Suppose that $r >0$, and  
  fix $i \in \{ 1, \cdots, r \}$.
Let   $\mcB_S$ be the $\mcO_S$-algebra defined by  \eqref{YY11}, with multiplication given by \eqref{YY12}.
  Although introduced for affine $S$,  this definition applies unchanged to arbitrary schemes over $\mbZ/p^{\CH +1}\mbZ$.
Note that there exists a canonical isomorphism of $\mcO_S$-algebras
 \begin{align} \label{e66}
\sigma^*_{i} ({^L}\mcD^{(\M)}) \isom \mcB_S 
\end{align}
determined by the following condition:
if $\{ \partial^{\langle j \rangle} \}_{j \in \mbZ_{\geq 0}}$ is the  local basis  of  $\mcD^{(\M)}$ associated, in the manner of  Section \ref{SS04f4},  to any local section $t$ defining the closed subscheme $\mr{Im}(\sigma_i)$ (which determines a logarithmic coordinate on $X^\mr{log}/S^\mr{log}$ around $\mr{Im}(\sigma_i)$),
then $\sigma_i^{-1}(\partial^{\langle j \rangle})$ is mapped to $\partial_\mcB^{\langle j \rangle}$ via \eqref{e66} for every $j$.

Let us take 
a $\mcD^{(\M)}$-module $(\mcF, \DMO)$.
 The $\mcD^{(\M)}$-module structure  $\DMO$ induces a $\sigma_i^* (\mcD^{(\M)})$-action $\sigma_i^* (\DMO)$ on
 $\sigma^*_i (\mcF)$; this action gives 
 the composite
\begin{align} \label{e445}
\mu_i (\DMO) : \mcB_S \xrightarrow{\eqref{e66}^{-1}}  \sigma^*_{i} ({^L}\mcD^{(\M)}) \xrightarrow{\sigma^*_i (\DMO)} \mcE nd_{\mcO_S} (\sigma^*_i (\mcF)).
\end{align}
In particular, we obtain 
\begin{align} \label{YY121}
\mu_i (\DMO)^{\langle \bullet \rangle} := (\mu_i (\DMO)^{\langle 1 \rangle}, \mu_i (\DMO)^{\langle p \rangle}, \cdots, \mu_i (\DMO)^{\langle p^\M \rangle}) \in \mr{End}_{\mcO_S} (\sigma_i^* (\mcF))^{\oplus (\M +1)}.
\end{align}
(cf. \eqref{e187}).

\SSP
\bde \label{D19}
We shall refer to $\mu_i (\DMO)$ and  $\mu_i (\DMO)^{\langle \bullet \rangle}$ as  the {\bf monodromy operator} of $\DMO$ at $\sigma_i$.
Also, for each $a = 0, \cdots, \M$, we shall refer to $\mu_i (\DMO)^{\langle p^a \rangle}$ as the {\bf $a$-th monodromy operator} of $\DMO$ at $\sigma_i$.
\ede
\SSP

Let us take a sufficiently small affine open subscheme $U$ of $X$ meeting   $\mr{Im}(\sigma_i)$,  and take 
a section $t \in \mcO_X$  on $U$  defining   the closed subscheme $\mr{Im}(\sigma_i) \cap U$ of $U$.
The affine scheme $\widehat{U}_t$  associated to the $t$-adic formal completion  of $U$ may be identified with the scheme $U_\oslash$ (cf. \eqref{YY124}) over some affine scheme $\mr{Spec}(R)$.
Under this  identification $\widehat{U}_t = U_\oslash$,
the restriction $\mcD^{(\M)} |_{\widehat{U}_t}$ of $\mcD^{(\M)}$  to $\widehat{U}_t$  may be considered as   $\mcD_\oslash^{(\M)}$ (cf.  Section \ref{SS1071}).
The monodromy operator $\mu_i (\DMO)$ of $\DMO$ at $\sigma_i$ coincides with
that of 
the $\mcD_\oslash^{(\M)}$-module structure  on $\mcF |_{\widehat{U}_t}$ (in the sense of Definition \ref{D98})
obtained by restricting $\DMO$.

Moreover, we suppose that $(\mcF, \DMO)$ is dormant, $\mcF$ is a vector bundle of rank $n >0$, and 
  $S$ is connected.
According to Proposition-Definitions \ref{P022}  and \ref{P724},
there exists a well-defined multiset
\begin{align} \label{YY130}
e_i (\DMO) := [\EX_1, \cdots, \EX_n]
\end{align}
of elements in $\mbZ/p^{\ell + \M +1}\mbZ$ 
satisfying the following condition:
for 
any local section $t$ as above,
the $\mcD_\oslash^{(\M)}$-module $(\mcF, \DMO) |_{\widehat{U}_t}$ obtained by restricting $(\mcF, \DMO)$ to $\widehat{U}_t \left( = U_\oslash \right)$ has  exponent $e_i (\DMO)$.
The following definition generalizes  ~\cite[Definition 4.3.1]{Wak9}.

\SSP
\bde \label{D30}
\begin{itemize}
\item[(i)]
With  the above notation,
we shall refer to $e_i (\DMO)$ as  
the {\bf (local) exponent} of $\DMO$ at the marked point $\sigma_i$.
\item[(ii)]
We shall set $\N := \M +1$.
Also, let $\msF^\heartsuit := (\mcF, \DMO, \{ \mcF^j \}_j)$ be a dormant $\mr{GL}_n^{(\N)}$-oper   on $\msX$.
Then,   the multiset 
\begin{align} \label{YY567}
e_i (\msF^\heartsuit) := e_i (\DMO)
\end{align}
  is called the {\bf (local) exponent} of $\msF^\heartsuit$ at $\sigma_i$.
\end{itemize}
\ede
\SSP

The following assertion will be applied in the proofs of Propositions \ref{P39} and \ref{P00245}.

\SSP
\bpr \label{NN77}
Let us keep the above notation, and suppose that $\ell =0$.
Also, let $n$ be an integer with $1 < n \leq p^\N$ and $\msF^\heartsuit := (\mcF, \DMO, \{ \mcF^j \}_j)$ a dormant $\mr{GL}_n^{(\N)}$-oper on $\msX$.
Suppose further  that $S =\mr{Spec}(R)$,  and that we are given  an isomorphism of $\mcD_\oslash^{(\N -1)}$-modules
\begin{align}
h : (\mcF, \DMO) |_{\widehat{U}_t} \isom \bigoplus_{j =1}^{n} \msO_{\oslash, \EX_j}^{(\N -1)}
\end{align}
 for some $d_1, \cdots, d_n \in \mbZ/p^\N \mbZ$.
(In particular,  the exponent  $e_i (\msF^\heartsuit)$ of $\msF^\heartsuit$ at $\sigma_i$ coincides with $[d_1, \cdots, d_n]$.)
Then, the following assertions hold:
\begin{itemize}
\item[(i)]
Let $v$ be a global section of 
 the line subbundle $\mcF^{n-1}|_{\widehat{U}_t} \left(\subseteq \mcF |_{\widehat{U}_t} \right)$, and  write
\begin{align}
(v_1 (t), \cdots, v_n (t)) := h (v) \in R [\![t]\!]^{\oplus n}.
\end{align}
Suppose that $v$ formally generates   $\mcF^{n-1}|_{\widehat{U}_t}$.
Then, each section $v_j (t)$ ($j =1, \cdots, n$) belongs to $R[\![t]\!]^{\times}$, or equivalently $v_j (0) \in R^\times$.
\item[(ii)]
The elements   $d_1, \cdots, d_n$ are mutually distinct.
Moreover, this fact implies that, if we are given  global sections $v_j (t) \in R [\![t]\!]^\times$ for $j=1, \cdots, n$,
then $h^{-1} (v_1 (t), \cdots, v_n (t))$ formally generates the $\mcD_\oslash^{(\N -1)}$-module $(\mcF, \nabla)$.
\end{itemize}
\epr
\begin{proof}
Both assertions (i) and  (ii) follow from arguments entirely similar to the proof of ~\cite[Proposition 4.4.1]{Wak9}, so we omit the details of the proofs.
(The case of $\N =1$ can also be found in ~\cite[Proposition 8.4]{Wak8}.)
\end{proof}

\LSP
\subsection{Gluing $\mcD^{(\M)}$-modules} \label{SS1070}

We recall the definition of a semi-graph in the sense of  ~\cite[Definition 7.1]{Wak8}. 

\SSP
\bde \label{Defui3}
A {\bf  semi-graph} is a triple 
\begin{align} \label{eQ332f}
\GR := (V, E, \zeta)
\end{align}
 consisting of the following data:
\begin{itemize}
\item
a set $V$, whose elements are called {\it vertices};
\item
a set $E$, whose elements are called {\it edges}, consisting of sets with cardinality $2$ such that $e \neq e' \in E$ implies $e \cap e' = \emptyset$;
\item
a map $\zeta : \bigsqcup_{e \in E} e \migi V\sqcup \{ \circledast \}$ (where ``$\circledast$" denotes an abstract symbol with $\circledast \notin V$), called a {\it coincidence map}, such that $\zeta (e) \neq \{ \circledast \}$ for any $e \in E$.
\end{itemize}
Each edge $e \in E$ with $\circledast \in \zeta (e)$ (resp., $\circledast \notin \zeta (e)$) is called {\bf open} (resp., {\bf closed}), and write $E^{\mr{op}}$ (resp., $E^{\mr{cl}}$) for the set of open (resp., closed) edges in $E$.
(In particular, we have $E = E^{\mr{op}}\sqcup E^{\mr{cl}}$.)
For each edge $e \in E$, we will refer to any element $b$ of $e$ 
  as a {\bf branch}  of  $e$.
Also, for each 
 $v\in V \sqcup \{ \circledast\}$, we write
 \begin{align}
 B_v  :=  \zeta^{-1} (\{ v\}).
 \end{align}
\ede
\SSP

Moreover, we recall some notions related to semi-graphs.

\SSP
\bde[cf. ~\cite{Wak8}, Definition 7.2] \label{Def3320}
Fix a semi-graph  $\GR := (V, E, \zeta)$.
\begin{itemize}
\item[(i)]
We  say that $\GR$ is {\bf finite} if both $V$ and $E$ are finite.
\item[(ii)]
We say that  $\GR$ is  {\bf connected} if for any two distinct vertices $u, v \in V$, there exists a sequence of edges $e_1, \cdots, e_\ell \in E$ ($\ell \geq 1$) such that $u \in \zeta (e_1)$, $v \in \zeta (e_\ell)$, and $\zeta (e_j) \cap \zeta (e_{j+1}) \cap V\neq \emptyset$ for any $j=1, \cdots, \ell -1$. 
 \item[(iii)]
 We  say that  $\GR$ is {\bf trivalent} (or {\bf $3$-regular}) if, for  any vertex $v \in V$, the cardinality of  $B_v$ is precisely $3$.
 \item[(iv)]
We 
shall set
\begin{align} \label{eeQQ47}
g (\GR) := 1  - \sharp (V) + \sharp (E) - \sharp (B_{\circledast}), \hspace{10mm}
r (\GR) := \sharp (B_{\circledast}).
\end{align}
Also, under the assumption that  $\GR$ is trivalent,
 we say that
$\GR$  is {\bf of type $(g',r')$} (for a pair of nonnegative integer $(g', r')$) if the equalities $g (\GR) = g'$ and $r (\GR) = r'$ hold.
\end{itemize}
\ede
\SSP

\bde[cf. ~\cite{Wak8}, Definition 7.4] \label{Def112}
\begin{itemize}
\item[(i)]
By {\bf clutching data}, we mean a collection of data
\begin{align} \label{eQQer}
\mbG := (\GR, \{ (g_j, r_j)\}_{j=1}^J, \{ \lambda_j\}_{j=1}^J)
\end{align}
where
\begin{itemize}
\item
$\GR := (V, E, \zeta)$ is a finite connected semi-graph with $J$ vertices $V := \{ v_1, \cdots, v_J \}$ (where $J \in \mbZ_{>0}$), numbered $1$  through $J$;
\item
$(g_j, r_j)$ ($j=1, \cdots, J$) is a pair of nonnegative integers with $2g_j - 2 + r_j > 0$;
\item
$\lambda_j$ ($j=1, \cdots, J$) denotes a bijection of sets $B_{v_j} \isom \{1, \cdots, r_j \}$.
\end{itemize}
\item[(ii)]
Let $\mbG$ be clutching data as in \eqref{eQQer} and $(g', r')$ a pair of nonnegative integers.
We say that $\mbG$ is {\bf of type $(g', r')$} if the equalities $g' = g (\GR) + \sum_{j=1}^J g_j$ and $r' = r (\GR)$ hold.
\item[(iii)]
Let $\mbG$ be clutching data as in \eqref{eQQer}.
We say that $\mbG$ is {\bf trivalent} (or {\bf $3$-regular})
if $(g_j, r_j)=(0, 3)$ for every $j=1, \cdots, J$.
Since trivalent clutching data is determined by ``$\GR$'' and ``$\{ \lambda_j \}_j$" (and satisfies $J = 2g'- 2 +r'$ when $\mbG$ is of type $(g', r')$),
we often indicate it by $(\GR, \{ \lambda_j \}_{j=1}^J)$ (or $(\GR, \{ \lambda_j \}_{j=1}^{2g' -2 +r'})$).
\end{itemize}
\ede
\SSP

\begin{rema}[The natural ordering of open edges] \label{Rem7878}
Let   $\mbG := (\GR, \{ (g_j, r_j)\}_{j=1}^J, \{ \lambda_j\}_{j=1}^J)$  be clutching data  of type $(g,r)$.
For each edge  $\{ b, b' \} \in E$, we refer to $b'$ as the {\bf opposite branch} of $b$ (hence $b$ is the opposite branch of $b'$).  
Then, we order the elements in the set $B_\circledast$
by using the order of their opposite branches defined 
 {\it lexicographically}, i.e., a branch incident  to $v_j$ comes before a branch incident to $v_{j'}$ if $j < j'$, and among branches incident to a common $v_j$, we take the ordering induced by $\lambda_j$.
According to the resulting ordering, we occasionally  write $B_\circledast := \{ b_{\circledast, 1}, \cdots, b_{\circledast, r} \}$.
\end{rema}
\SSP

Let us fix   clutching data   $\mbG := (\GR, \{ (g_j, r_j)\}_{j=1}^J, \{ \lambda_j\}_{j=1}^J)$   of type $(g,r)$.
For each $j = 1, \cdots, J$, suppose that we are given  an $r_j$-pointed stable curve  ${\msX_j} := (X_j /S_j, \{ \sigma_{j, i} \}^{r_j}_{i=1})$  of genus $g_j$  over $S_j := S$.
Also, suppose that $\msX$ may be obtained by gluing together the $\msX_j$'s 
by means of $\mbG$ in the manner of ~\cite[Section 7.2.1]{Wak8}.
(The discussions in ~\cite{Wak8} only deal with schemes over a {\it field}, but the gluing construction mentioned there  can be applied to a general base space $S$.)
In particular, there exist natural morphisms
\begin{align} \label{YY140}
\mr{Clut}_j : X_j \migi X
\end{align}
 ($j =1, \cdots, J$)  of $S \left(=S_j\right)$-schemes.

Recall  that, for each $j$,   
the scheme $S_j$ (resp., $X_j$) is equipped with  the   log structure  pulled-back from  $\overline{\mcM}_{g_j,r_j}^\mr{log}$ (resp., 
the universal family of log curves over $\overline{\mcM}_{g_j,r_j}^\mr{log}$)
  via the classifying morphism of ${\msX_j}$.
Denote by  $S_j^{\mr{log}}$ (resp., $X^{\mr{log}}_j$) the resulting log scheme.
For simplicity, we shall write 
 $\mcT_j := \mcT_{X_j^\mr{log}/S_j^\mr{log}}$
 and 
 $\mcD_j^{(\M)} := \mcD_{X_j^\mr{log}/S_j^\mr{log}}^{(\M)}$.

 On the other hand, write $X^{\mr{log} |_X}_j$ for the log scheme obtained by equipping $X_j$ with the log structure pulled-back from  $X^\mr{log}$ via $\mr{Clut}_j$.
The structure morphism $X_j \migi S_j \left(=S \right)$ of  $X_j/S_j$ extends to a morphism $X_j^{\mr{log} |_X} \migi S^{\mr{log}}$ of log schemes.
Moreover, the natural morphisms $S^\mr{log} \migi S \left(=S_j \right)$ and  $X_j^{\mr{log} |_X} \migi X_j$
extend to
 morphisms $S^\mr{log} \migi S_j^{\mr{log}}$ and   $X_j^{\mr{log} |_X} \migi X_j^{\mr{log}}$,  respectively,  which make the following square diagram commute:
\begin{align}  \label{fff01}
\vcenter{\xymatrix@C=46pt@R=36pt{
 X_j^{\mr{log} |_X} \ar[r] \ar[d]_-{\mr{projection}} &X_j^{\mr{log}}  \ar[d]^-{\mr{projection}}\\
S^{\mr{log}} \ar[r]&S_j^{\mr{log}}.
}}
\end{align}
This diagram induces  a morphism 
$X_j^{\mr{log} |_X} \migi X_j^{\mr{log}} \times_{S_j^{\mr{log}}} S^{\mr{log}}$
   over $S^{\mr{log}}$ 
   whose
underlying morphism of  $S$-schemes  coincides with the identity morphism of $X_j$.
The differential of this morphism  yields  isomorphisms of $\mcO_{X_j}$-modules
 \begin{align} \label{fff02}
 \delta_{j}^\mcT : \mcT_{X_j^{\mr{log} |_X}/S^\mr{log}} \left(= \mr{Clut}_{j}^*(\mcT)\right)
  \isom  \mcT_j,
  \hspace{5mm}
 \delta_{j}^\mcD : {^L}\mcD^{(\M)}_{X_j^{\mr{log} |_X}/S^\mr{log}} \left(= \mr{Clut}_{j}^*({^L}\mcD^{(\M)})\right)
  \isom  {^L}\mcD_j^{(\M)}.
 \end{align}

 Now,  
 let  $\mcF$ be an $\mcO_X$-module.
  For each $j = 1, \cdots, J$, denote by $\mcF_j$  the pull-back  of $\mcF$ to $X_j$.

\SSP
 \bde \label{QG2}
 A  {\bf $\mbG$-$\mcD^{(\M)}$-module structure}  on $\mcF$ is a collection
\begin{align}
\{ \DMO_j \}_{j=1}^J,
\end{align}
where
    each $\DMO_j$ denotes  a $\mcD_{j}^{(\M)}$-module structure  on $\mcF_j$,
such that, for  any closed edge  $\{ b, b' \} \in  E^{\mr{cl}}$ of $\GR$ with $\zeta (b) = v_{j}$,  $\zeta (b') = v_{j'}$ ($j, j' \in \{1, \cdots, J \}$), 
 the equality 
 \begin{align}
 \mu_{\lambda_j (b)}(\DMO_{j}) \circ  \mr{sw}= \mu_{\lambda_{j'}(b')}(\DMO_{j'})
  \ \left( \Longleftrightarrow  \mr{sw}^\bullet_{\sigma^*(\mcF)} ( \mu_{\lambda_j (b)}(\DMO_{j}))^{\langle \bullet \rangle} =  \mu_{\lambda_{j'}(b')}(\DMO_{j'})^{\langle \bullet \rangle} \ \text{by \eqref{dE110}}\right) 
 \end{align}
  holds under the natural identification $\sigma^*(\mcF) := \sigma_{j, \lambda_j (b)}^*(\mcF_{j}) =  \sigma_{j', \lambda_{j'} (b')}^*(\mcF_{j'})$.
 \ede
\SSP

 Let  
 $\DMO : {^L}\mcD^{(\M)} \migi \mcE nd_{\mcO_S}(\mcF)$ be a $\mcD^{(\M)}$-module structure 
 on $\mcF$.
For each $j =1, \cdots, J$, 
the composite
\begin{align} \label{YY153}
\DMO |_{j} : {^L}\mcD_j^{(\M)} \xrightarrow{(\delta^\mcD_{j})^{-1}} {^L}\mcD^{(\M)}_{X_j^{\mr{log}|_X}/S^\mr{log}} \xrightarrow{\mr{Clut}^*_j (\nabla)} \mcE nd_{\mcO_S}(\mcF_j)
\end{align}
   specifies 
   a $\mcD_{j}^{(\M)}$-module structure 
    on $\mcF_j$. 
    We shall  refer to $\DMO |_j$ (resp., $(\mcF_j, \DMO |_j)$) as the {\bf restriction} of $\DMO$ (resp., $(\mcF, \DMO)$) to $\mcF_j$ (resp., $X_j$).
 
 \SSP
\bpr \label{YY157}
Suppose that $\mcF$ is locally free, and
let $\DMO$ be a 
$\mcD^{(\M)}$-module structure 
 on $\mcF$.
 \begin{itemize}
 \item[(i)]
The collection  of its restrictions $\{ \DMO |_j \}_{j=1}^J$ specifies
a $\mbG$-$\mcD^{(\M)}$-module structure 
on $\mcF$.
 Moreover, the resulting assignment $\DMO \mapsto \{ \DMO  |_j \}_{j=1}^J$ determines a bijection of sets
\begin{align}
 \begin{pmatrix}
\text{the set of } \\
\text{$\mcD^{(\M)}$-module structures on $\mcF$} 
\end{pmatrix}
\isom 
\begin{pmatrix}
\text{the set of } \\
\text{$\mbG$-$\mcD^{(\M)}$-module structures on $\mcF$} 
\end{pmatrix}.
\end{align}
\item[(ii)]
$(\mcF, \DMO)$ is dormant if and only if $(\mcF_j, \DMO |_j)$ is dormant for every $j =1, \cdots , J$.
Moreover,  for each $j$, 
the formation of restrictions $(\mcF, \DMO) \mapsto (\mcF_j, \DMO |_j)$ commutes with diagonal reduction.
\end{itemize}
 \epr
\begin{proof}
Assertion (i) follows from Proposition  \ref{P27}, (i) and (ii) (cf. the discussion in the proof of ~\cite[Proposition 7.6]{Wak8}).

Assertion (ii) for $\CH =0$ follows from the commutativity of  the following square diagram defined  for every $j =1, \cdots, J$:
\begin{align} \label{Eh5007}
\vcenter{\xymatrix@C=76pt@R=36pt{
F_{X_j}^{(\M +1)*}(\mcT_{X_j^{\mr{log}|_X}/S^\mr{log}})\ar[r]^-{\mr{Clut}_j^*(\psi_{X^\mr{log}/S^\mr{log}})} \ar[d]^-{\wr}_-{F_{X_j}^{(\M+1)*}(\delta^\mcT_{j})}& 
\mcD^{(\M)}_{X_j^{\mr{log}|_X}/S^\mr{log}}
 \ar[d]_-{\wr}^-{\delta^\mcD_{j}}\\
F_{X_j}^{(\M +1)*}(\mcT_j) \ar[r]_-{\psi_{X_j^\mr{log}/S^\mr{log}}} &  \mcD_j^{(\M)}.
}}
\end{align}
Also, the assertion for  $\M = 0$ can be proved by applying  the assertion  for $\CH =0$   together with the definition of the functor ${^{\Diag}}\!\!(-)$.
\end{proof}

\LSP
\subsection{Definition of radius} \label{SS044}

Let  $n$ be  an integer with $1 < n < p$.
We shall 
discuss the notion  of radius associated to a dormant $\mr{PGL}_n$-oper  of finite level on a pointed stable curve
(cf. ~\cite[Definition 5.3.2]{Wak9} for the case of pointed smooth curves).
 The level $1$ case  was already  defined  in  ~\cite[Definition 2.32]{Wak8}
 under the identification between   elements in     ``$\mfS_n \backslash \mbF_p^{\times n} /\Delta$" (cf. \eqref{e847} below) and   certain $\mbF_p$-rational points in  the adjoint quotient of the Lie algebra $\mfs \mfl_n$. 

For convenience, we shall set $\N := \M +1$ and 
  $\CL := \CH + \N$. 
Denote by $\Delta$ the image of the diagonal embedding  $\mbZ/p^\CL \mbZ \migiincl (\mbZ/p^\CL \mbZ)^{\times n}$, which is a group homomorphism.
In particular,  we obtain the quotient  set $(\mbZ/p^\CL \mbZ)^{\times n}/\Delta$.
Note that the set $(\mbZ/p^\CL \mbZ)^{\times n}$ is equipped with the action of the symmetric group of $n$ letters $\mfS_n$  by permutation;
this action induces a well-defined $\mfS_n$-action on $(\mbZ/p^\CL \mbZ)^{\times n}/\Delta$.
Hence, we obtain the quotient sets 
\begin{align} \label{e847}
\mfS_n \backslash (\mbZ/p^\CL \mbZ)^{\times n}, \hspace{5mm}
\mfS_n \backslash (\mbZ/p^\CL \mbZ)^{\times n}/\Delta,
\end{align}
 and moreover, obtain the natural projection
\begin{align} \label{ssak}
\pi_\Delta : \mfS_n \backslash (\mbZ/p^\CL \mbZ)^{\times n} \migisurj \mfS_n \backslash (\mbZ/p^\CL \mbZ)^{\times n}/\Delta.
\end{align}

Each element of $\mfS_n \backslash (\mbZ/p^\CL \mbZ)^{\times n}$ may be regarded as a multiset of $\mbZ/p^\CL \mbZ$ whose cardinality equals $n$.
Also, we occasionally identify $\mfS_n \backslash (\mbZ/p^\CL \mbZ)^{\times n}$ (resp., $\mfS_n \backslash (\mbZ/p^\CL \mbZ)^{\times n}/\Delta$)  with the scheme defined as the disjoint union of copies of $\mr{Spec}(\mbZ)$ indexed by  the set $\mfS_n \backslash (\mbZ/p^\CL \mbZ)^{\times n}$ (resp., $\mfS_n \backslash (\mbZ/p^\CL \mbZ)^{\times n}/\Delta$).

\SSP
\begin{rema}[Radii for $n=2$] \label{ewgk}
Denote by $(\mbZ/p^\CL \mbZ)/\{ \pm 1 \}$ the set of equivalence classes of elements $a \in \mbZ/p^\CL \mbZ$, in which 
 $a$ and $-a$ are identified.
 Then, the assignment $a \mapsto [a, -a]$ determines a well-defined bijection
\begin{align} \label{wpaid}
(\mbZ/p^\CL \mbZ)/\{ \pm 1 \} \isom \mfS_2 \backslash (\mbZ/p^\CL \mbZ)^{\times 2}/\Delta.
\end{align}
We occasionally identify $\mfS_2 \backslash (\mbZ/p^\CL \mbZ)^{\times 2}/\Delta$ with  $(\mbZ/p^\CL \mbZ)/\{ \pm 1 \}$
by using this bijection.
Under this identification, the notion of radius introduced below coincides with the classical notion of radius for  torally indigenous bundles
 in the sense of
 ~\cite[Chapter I, Definition 4.1]{Mzk2} (see also ~\cite[Remark 5.3.3]{Wak9}).
\end{rema}
\SSP

Let  us take   a dormant $\mr{PGL}_n^{(\N)}$-oper $\msE^\spadesuit$ on $\msX$.
According to Corollary \ref{UU39}, (i), we can find  a dormant $\mr{GL}_n^{(\N)}$-oper $\msF^\heartsuit := (\mcF, \DMO, \{ \mcF^j \}_j)$ on $\msX$ with $\msF^{\heartsuit \Rightarrow \spadesuit} \cong \msE^\spadesuit$.
For each $i=1, \cdots, r$,
the exponent of  $\msF^\heartsuit$ at $\sigma_i$ determines an element
\begin{align}
e_i (\msF^\heartsuit) := [\EX_{i, 1}, \cdots, \EX_{i, n}] \in  (\mfS_n \backslash (\mbZ/p^\CL \mbZ)^{\times n}) (S),
\end{align}
where  $\EX_{i, 1}, \cdots, \EX_{i, n} \in (\mbZ/p^\CL \mbZ) (S)$.
Its  image  
\begin{align} \label{YY161}
\rho_i (\msE^\spadesuit) := \pi_\Delta (e_i (\msF^\heartsuit)) \in (\mfS_n \backslash (\mbZ/p^\CL \mbZ)^{\times n}/\Delta)(S)
\end{align}
 via $\pi_\Delta$ depends only on the equivalence class $[\msF^\heartsuit]$, i.e., the isomorphism class of $\msE^\spadesuit$.
In fact, if $\msL := (\mcL, \DMO_\mcL)$ is a dormant invertible  $\mcD^{(\N-1)}$-module whose exponent at $\sigma_i$ is $c_i \in (\mbZ/p^\CL \mbZ) (S)$,
then the exponent of 
$\msF_{\otimes \msL}^\heartsuit$
 at $\sigma_i$ coincides with
$[\EX_{i, 1}+c_i, \cdots, \EX_{i, n}+ c_i]$.
Since  $\pi_\Delta ([\EX_{i, 1}, \cdots, \EX_{i, n}]) = \pi_\Delta ([\EX_{i, 1}+c_i, \cdots, \EX_{i, n}+ c_i])$, we see that
$\rho_i (\msE^\spadesuit)$ is a well-defined element  associated to $\msE^\spadesuit$.

\SSP
\bde \label{epaddd}
\begin{itemize}
\item[(i)]
With   the above notation, 
we shall refer to $\rho_{ i}(\msE^\spadesuit)$ as the {\bf radius} of $\msE^\spadesuit$ at $\sigma_i$.
\item[(ii)]
Let $\rho := (\rho_i )_{i=1}^r$ be an element of $(\mfS_n \backslash (\mbZ/p^\CL \mbZ)^{\times n}/\Delta)^{\times r}$.
We shall say that a dormant $\mr{PGL}_n^{(\N)}$-oper  $\msE^\spadesuit$ is {\bf of radii $\rho$} if $\rho_i = \rho_{i} (\msE^\spadesuit)$ for every $i =1, \cdots, r$.
When $r =0$, we will refer to any dormant $\mr{PGL}_n^{(\N)}$-oper as being {\bf of radii $\emptyset$}.
\end{itemize}
\ede
\SSP

The following two propositions describe  some basic properties of the  radius of a dormant $\mr{PGL}_n^{(\N)}$-oper.

\SSP
\bpr \label{YY499}
Let $\msE^\spadesuit$ be a dormant $\mr{PGL}_n^{(\N)}$-oper on $\msX$ of radii $\rho \in (\mfS_n \backslash (\mbZ/p^\CL \mbZ)^{\times n}/\Delta)^{\times r}$.
For each integer $\CL'$  with $1 \leq  \CL' \leq \CL$, we denote by 
\begin{align} \label{NN8121}
q_{\CL \Rightarrow \CL'} : (\mfS_n \backslash (\mbZ/p^\CL \mbZ)^{\times n}/\Delta)^{\times r} \migisurj (\mfS_n \backslash (\mbZ/p^{\CL'} \mbZ)^{\times n}/\Delta)^{\times r}
\end{align}
the surjection induced from the natural quotient $\mbZ/p^\CL \mbZ \migisurj \mbZ/p^{\CL'} \mbZ$.
\begin{itemize}
\item[(i)]
Suppose further that $\CH =0$ (hence $\N = \CL$).
Then, for each positive integer $\N' \leq \N$, the dormant $\mr{PGL}_n^{(\N')}$-oper on $\msX$ obtained  by reducing the level of $\msE^\spadesuit$ to  $\N'$ is of radii $q_{\N \Rightarrow \N'}(\rho)$.
\item[(ii)]
Suppose further that $\N =1$ (hence $\CH +1= \CL$).
Then, for each nonnegative integer $\CH' \leq \CH$,  
the dormant $\mr{PGL}_n$-oper on $\msX_{\CH'}$ (:= the reduction of $\msX$ modulo $p^{\CH' +1}$) obtained  by reducing  $\msE^\spadesuit$  modulo $p^{\CH' +1}$  is of radii $q_{\CH +1 \Rightarrow \CH' +1}(\rho)$.
\end{itemize}
\epr
\begin{proof}
Assertion (i) follows from Proposition \ref{YY30}, (iii).
Also, assertion (ii) follows from the definition of the radius of a dormant $\mr{PGL}_n$-oper.
\end{proof}
\SSP

\bpr \label{WWed3}
We shall  fix $i \in \{1, \cdots, r \}$,  $a_0\in\mbZ/p^\CL \mbZ$, and $\rho_0 \in \mfS_n \backslash (\mbZ/p^\CL \mbZ)^{\times n}/\Delta$.
Note that (since $n <p$) there exists a unique 
multiset $e_0 := [d_1, \cdots, d_n] \in \mfS_n \backslash (\mbZ/p^\CL \mbZ)^{\times n}$ with $\pi_\Delta (e_0) = \rho_0$ and $\sum_{i=1}^n d_i = a_0$.
Also, let $\vartheta := (\varTheta, \DMO_{\vartheta})$ be a dormant $n^{(\N)}$-theta characteristic of $X^\mr{log}/S^\mr{log}$ such that the exponent of $\DMO_{\vartheta}$ at $\sigma_i$ is $a_0$,  and let  $\DMO^\diamondsuit$ be a dormant $(\mr{GL}_n^{(\N)}, \vartheta)$-oper on $\msX$.
Denote by $\msE^\spadesuit$ the dormant $\mr{PGL}_n^{(\N)}$-oper corresponding to $\DMO^\diamondsuit$ via $\Lambda_{\diamondsuit \Rightarrow \spadesuit, \vartheta}^\ZZZ$.
Then, the radius of $\msE^\spadesuit$ at $\sigma_i$ is $\rho_0$ if and only if the exponent of  $\DMO^{\diamondsuit \Rightarrow \heartsuit}$ at $\sigma_i$ is $e_0$.
\epr
\begin{proof}
The assertion follows from the various definitions involved.
\end{proof}
\SSP

Denote by 
$\widetilde{\Xi}_{n, \CL}$
 the subset of $\mfS_n \backslash (\mbZ/p^\CL \mbZ)^{\times n}$ consisting of multisets $[\EX_1, \cdots, \EX_n]$ 
 such that the elements 
 $\EX_{1[0]}, \cdots, \EX_{n[0]}$ of $\mbF_p$ determined by $\EX_1, \cdots, \EX_n$
 are mutually distinct.
 (In particular, $\widetilde{\Xi}_{n, \CL}$ may be identified with a set of subsets of $\mbZ/p^\CL \mbZ$ with cardinality $n$.)
Also,  we denote 
 the image of $\widetilde{\Xi}_{n, \CL}$ via the projection $\pi_\Delta : \mfS_n \backslash (\mbZ/p^\CL \mbZ)^{\times n} \migisurj \mfS_n \backslash (\mbZ/p^\CL \mbZ)^{\times n}/\Delta$ by 
\begin{align} \label{e401YY}
\Xi_{n, \CL}  \ \left(\subseteq \mfS_n \backslash (\mbZ/p^\CL \mbZ)^{\times n}/\Delta  \right).
\end{align}
Under the identification $(\mbZ/p^\CL \mbZ)/\{ \pm 1 \} = \mfS_2 \backslash (\mbZ/p^\CL \mbZ)^{\times 2}/\Delta$ given by \eqref{wpaid},
$\Xi_{2, \CL}$ may be identified with the subset $(\mbZ/p^\CL \mbZ)^\times /\{\pm 1 \}$ of $(\mbZ/p^\CL \mbZ)/\{ \pm 1 \}$.

\SSP
\bpr \label{P39}
Let $\rho$  be 
 an element of  $ \mfS_n \backslash (\mbZ/p^\CL \mbZ)^{\times n}/\Delta$ such that
 there exists a dormant $\mr{PGL}_n^{(\N)}$-oper   on $\msX$ whose radius at $\sigma_i$ (for some $i \in \{1, \cdots, r\}$) coincides with $\rho$.
Then, $\rho$ belongs to $\Xi_{n, \CL}$.
\epr
\begin{proof}
Let us choose a dormant $\mr{PGL}_n^{(\N)}$-oper $\msE^\spadesuit$ on $\msX$ whose radius at $\sigma_i$ coincides with $\rho$.
There exists   a dormant $\mr{GL}_n^{(\N)}$-oper $\msF^\heartsuit := (\mcF, \DMO, \{ \mcF^j\}_j)$ on $\msX$ with $\msF^{\heartsuit \Rightarrow \spadesuit}\cong \msE^\spadesuit$ (cf. Corollary \ref{UU39}, (i)).
The exponent $e_i (\msF^\heartsuit)$ of $\msF^\heartsuit$ at $\sigma_i$ can be described as a multiset $[\EX_1, \cdots, \EX_n]$ ($\EX_1, \cdots, \EX_n \in \mbZ/p^\CL \mbZ$).
Now, let us consider the case where  $\CH = 0$, or equivalently $\CL = \N$ (resp., $\N =1$, or equivalently $\CL = \CH +1$).
By  
Propositions \ref{YY439} and \ref{YY499}, (i) (resp., Proposition \ref{YY499}, (ii)),
the radius at $\sigma_i$ of the dormant $\mr{PGL}_n$-oper obtained by reducing the level of $\msE^\spadesuit$ to $1$ (resp., by reducing $\msE^\spadesuit$ modulo $p$)
coincides with the image of $[\EX_{1[0]}, \cdots, \EX_{n[0]}]$ in 
$\mfS_n \backslash \mbF_p^{\times n}/\Delta$.
Then, it follows from
Propositions \ref{YY439} and  \ref{NN77}, (ii) (or
 ~\cite[Proposition 4.4.1]{Wak9}) that $\EX_{1[0]}, \cdots, \EX_{n[0]}$ are mutually distinct.
This means that $\rho_i (\msE^\spadesuit) \in \Xi_{n, \CL}$, and hence completes the proof of this proposition.
\end{proof}
\SSP

\begin{rema}[Comparison with the classical definition] \label{Rem89e}
Suppose that $\N = 1$.
Then, the notion of radius defined above is essentially the same as 
 ~\cite[Definitions 2.32 and  4.46]{Wak8}, and the
definitions discussed there  can be naturally extended to our setting, i.e., the case where the ground ring has   prime-power characteristic.

To see this,
 let us take an element $\rho$ of   $\Xi_{n, L}$, which can be expressed as $\rho = \pi_\Delta ([d_1, \cdots, d_n])$ for a unique element  $[d_1, \cdots, d_n]$ of $\widetilde{\Xi}_{n, L}$ with $\sum_{j=1}^n d_j = 0$.
It determines  a well-defined element  
\begin{align}
s^\mft (\rho) := (s_2^\mft (d_1, \cdots, d_n), s_3^\mft (d_1, \cdots, d_n), \cdots, s_n^\mft (d_1, \cdots, d_n)) \in H^0 (S, \mcO_S)^{\oplus (n-1)},
\end{align}
where $s^\mft_i$'s are the elementary symmetric polynomials, i.e., 
\begin{align}
s_1^\mft (\lambda_1, \cdots, \lambda_n):= \sum_{j=1}^n \lambda_j, \hspace{3mm}
s_2^\mft (\lambda_1, \cdots, \lambda_n) := \sum_{j < j'} \lambda_j \cdot \lambda_{j'}, \hspace{3mm} \cdots, \hspace{3mm}
s_n^\mft (\lambda_1, \cdots, \lambda_n) := \prod_{j=1}^n \lambda_j.
\end{align}
The assignment $\rho \mapsto s^\mft (\rho)$ defines an injection
\begin{align} \label{eQ2234}
\Xi_{n, L} \hookrightarrow H^0 (S, \mcO_S)^{\oplus (n-1)}.
\end{align}

Now, let $\msE^\spadesuit$ be  a $\mr{PGL}_n$-oper  on $\msX$ (that is not necessarily dormant).
By Corollary \ref{UU39}, (i),  there exists a $\mr{GL}_n$-oper $\msF^\heartsuit := (\mcF, \nabla, \{ \mcF^j \}_{j=0}^n)$ with $\msF^{\heartsuit \Rightarrow \spadesuit} \cong \msE^\spadesuit$, where we use the notation $\nabla$ to denote an $S^\mr{log}$-connection on $\mcF$.
For each $i =1, \cdots, r$,
the monodromy operator of $\nabla$ at $\sigma_i$ (in the sense of ~\cite[Definition 4.42]{Wak8}) is the element $\mu_i^\nabla$ of $\mr{End} (\sigma_i^*(\mcF))$ defined to be the composite
\begin{align}
\mu_i^\nabla : \sigma_i^*(\mcF) \xrightarrow{\sigma_i^*(\nabla)} \sigma_i^* (\Omega \otimes \mcF) \xrightarrow{\sim}
\sigma_i^*(\Omega) \otimes \sigma_i^*(\mcF) \xrightarrow{\sim}\sigma_i^*(\mcF),
\end{align}
where the last arrow arises from the residue isomorphism $\sigma^*_i (\Omega) \isom \mcO_S$.

Then, there exists a tuple 
\begin{align}
\varrho_i (\msE^\spadesuit) := (s_2 (\rho), \cdots, s_n (\rho)) \in H^0 (S, \mcO_S)^{\oplus (n-1)}
\end{align}
of elements of $H^0 (S, \mcO_S)$ uniquely determined by the equality
\begin{align}
\mr{det}\left(t \cdot \mr{id}_{\sigma_i^* (\mcF)} - \breve{\mu}_i^\nabla \right) = t^n + \sum_{j=2}^n  s_j (\rho) \cdot t^{n-j} \in H^0 (S, \mcO_S)[t],
\end{align}
where  $\breve{\mu}_i^\nabla := \mu_i^\nabla - \frac{\mr{tr}(\mu_i^\nabla)}{n} \cdot \mr{id}_{\sigma_i^*(\mcF)}$, and  $\varrho_i (\msE^\spadesuit)$ does not depend on the choice of $\msF^\heartsuit$.
We refer to   $\varrho_i (\msE^\spadesuit)$   as the {\bf radius} of $\nabla$ at $\sigma_i$
(cf.  ~\cite[Definition 4.43]{Wak8}).
It follows from the various definitions involved that if $\msE^\spadesuit$ is dormant, then  the equality $\rho_i (\msE^\spadesuit) = \varrho_i (\msE^\spadesuit)$ holds via the injection \eqref{eQ2234}.

Moreover, for $\rho \in (\mfS_n \backslash (\mbZ/p^\CL \mbZ)^{\times n}/\Delta)^{\times r}$ (where $\rho := \emptyset$ if $r =0$),  we denote by 
$\mcO p_{\spadesuit, \rho}$ the subfunctor of $\mcO p_{\spadesuit}$ classifying $\mr{PGL}_n$-opers of radii $\rho$.
Then, it follows from an argument similar to the arguments in ~\cite[Sections 4.7.5, 4.7.6, and 4.11.1]{Wak8} that $\mcO p_{\spadesuit, \rho}$ admits a structure of  $f_* (\mcV_\varTheta (-\sum_{i=1}^r \sigma_i))$-torsor  (cf. \eqref{eQ44} for the definition of $\mcV_\varTheta$), where  the direct image $f_* (\mcV_\varTheta (-\sum_{i=1}^r \sigma_i))$ forms   a vector bundle on $S$ of rank $(n^2 -1)(g-1) + \frac{(n^2 - n)r}{2}$   (cf. ~\cite[Lemma 4.67]{Wak8}).
\end{rema}
\SSP

Write  $R := \mbZ/p^{\CH +1}\mbZ$.
Given an $r$-tuple   
$\rho:= (\rho_1, \cdots, \rho_r) \in 
(\mfS_n \backslash (\mbZ/p^\CL \mbZ)^{\times n}/\Delta)^{\times r}$ (where $\rho := \emptyset$ if $r =0$), 
 we shall  write 
\begin{align} \label{e4000}
\mcO p_{n, \N, \rho, g,r, R}^\ZZZ, \  \text{or simply}  \  \mcO p_{\rho, g,r}^\ZZZ, 
\end{align}
 for the full subcategory of  $\mcO p_{n, \N, g,r, R}^\ZZZ$ consisting of pairs $(\msX, \msE^\spadesuit)$ such that the dormant $\mr{PGL}_n^{(\N)}$-oper $\msE^\spadesuit$ is  of radii $\rho$.
The projection $\Pi_{g, r}$ restricts to a morphism
\begin{align} \label{TTT15}
\Pi_{n, \N, \rho, g, r, R} \left(\text{or simply} \ \Pi_{\rho, g, r}\right) : \mcO p^\ZZZ_{\rho, g, r} \migi \overline{\mcM}_{g, r}.
\end{align}

It follows from  Proposition \ref{P39} that 
$\mcO p_{\rho, g,r}^\ZZZ$ is empty unless $\rho \in \Xi_{n, \CL}^{\times r}$, and  that 
$\mcO p_{g, r}^\ZZZ\!$ and $\Pi_{g, r}$ decompose as the disjoint unions
\begin{align} \label{e888}
\mcO p_{g,r}^\ZZZ = \coprod_{\rho \in \Xi_{n, \CL}^{\times r}}\mcO p_{\rho, g,r}^\ZZZ \hspace{5mm} \text{and} \hspace{5mm} \Pi_{g, r} =  \coprod_{\rho \in \Xi_{n, \CL}^{\times r}}\Pi_{\rho, g, r},
\end{align}
respectively, where  $\Xi_{n, \CL}^{\times r} := \{ \emptyset \}$ if $r =0$.

\SSP
\begin{rema}[Changing the ordering of marked points] \label{Rem3391i}
Choose an  element $\varsigma$ of the symmetric group  of $r$ letters $\mfS_r$.
 Let $\msX$ be as before, and  write 
 \begin{align}\label{Curve22}
 \msX^\varsigma := (f: X \rightarrow S, \{ \sigma_{\varsigma (i)} \}_{i=1}^r),
 \end{align}
  i.e.,   the pointed stable curve obtained from $\msX$ by changing the order of its marked points  via the action of $\varsigma$.
Then, each dormant $\mr{PGL}_n^{(\N)}$-oper on $\msX$ of radii $\rho$ can be regarded as that on $\msX$ whose radii is given by $\rho^\varsigma := (\rho_{\varsigma (1)}, \cdots, \rho_{\varsigma (r)})$.
Hence, 
the assignment $(\msX, \msE^\spadesuit) \mapsto (\msX^\varsigma, \msE^\spadesuit)$ determines  an isomorphism  $\mcO p_{\rho, g,r}^\ZZZ \xrightarrow{\sim} \mcO p_{\rho^\varsigma, g,r}^\ZZZ$, fitting into the following square diagram
\begin{align} \label{Errfe1}
\vcenter{\xymatrix@C=46pt@R=36pt{
 \mcO p_{\rho, g,r}^\ZZZ \ar[r]^-{\sim} \ar[d]_-{\Pi_{\rho, g, r}} & \mcO p_{\rho^\varsigma, g,r}^\ZZZ \ar[d]^-{\Pi_{\rho^\varsigma, g, r}} \\
 \overline{\mcM}_{g, r} \ar[r]_-{\sim} & \overline{\mcM}_{g, r},
 }}
\end{align}
where the lower horizontal arrow denotes the automorphism of $\overline{\mcM}_{g, r}$ induced by $\msX \mapsto \msX^\varsigma$.
\end{rema}
\SSP

By Proposition \ref{P72e}, (i),  the functor ${^{\Diag}}\!\!(-)$ (cf. \eqref{YY260}) restricts to a functor 
\begin{align} \label{YY329}
{^{\Diag}}\!\!(-) : \mcO p_{n, 1, \rho, g,r, \mbZ/p^\N \mbZ}^\ZZZ \migi\mcO p_{n, \N, \rho, g, r,  \mbF_p}^\ZZZ.
\end{align}
That is to say,  a dormant $\mr{PGL}_n^{(1)}$-oper on a pointed stable curve classified by $\overline{\mcM}_{g, r, \mbZ/p^\N \mbZ}$  is of radii $\rho$ if and only if its diagonal reduction is of radii $\rho$
(cf. Definition \ref{YY300}).

\SSP
\bt \label{P91}
Suppose  that $\CH =0$.
Then, 
for each $\rho \in\Xi_{n, \CL}^{\times r}$,
the category $\mcO p_{\rho, g,r}^\ZZZ$ may be represented by a (possibly empty)  Deligne-Mumford stack  of finite type over $\mbF_p$,  and the projection  $\Pi_{\rho, g, r}$
 is affine.
\et
\begin{proof}
The assertion follows from the decompositions displayed in  \eqref{e888} together with Theorem \ref{T40}.
\end{proof}

\LSP
\subsection{Gluing dormant $\mr{PGL}_n^{(\N)}$-opers} \label{QG1090}

In this subsection, we shall discuss a gluing procedure for  dormant $\mr{PGL}_n^{(\N)}$-opers by means of clutching data.

Given each $\EX := (\EX_1, \cdots, \EX_n) \in (\mbZ/p^\CL \mbZ)^{\times n}$, we write
$\EX^\veebar := (-\EX_1, \cdots, -\EX_n)$.
The assignment $\EX \mapsto \EX^\veebar$ induces a well-defined involution 
\begin{align} \label{e56}
(-)^\veebar : \mfS_n \backslash (\mbZ/p^\CL \mbZ)^{\times n} / \Delta \isom \mfS_n \backslash (\mbZ/p^\CL \mbZ)^{\times n} / \Delta
\end{align}
on the set $\mfS_n \backslash (\mbZ/p^\CL \mbZ)^{\times n} / \Delta$.
Also, for each $\rho := (\rho_1, \cdots, \rho_r) \in (\mfS_n \backslash (\mbZ/p^\CL \mbZ)^{\times n} / \Delta)^{\times r}$,
we set $\rho^\veebar := (\rho_1^\veebar, \cdots, \rho_r^\veebar)$.
Note that $(-)^\veebar$ restricts to an involution on $\Xi_{n, \CL}$ (hence also on $\Xi_{n, \CL}^{\times r}$).

Let 
$\mbG$ and 
$\msX_j$ ($j=1, \cdots, J$) be as in Section \ref{SS1070}.

\SSP
\bde \label{y0175}
A {\bf  set of  $\mbG$-$\Xi_{n, \CL}$-radii}  is an ordered  set 
\begin{align} \label{YY172}
 \rho_\mbG := \{  \rho^j \}_{j=1}^J,
 \end{align}
where each $\rho^j := (\rho^j_i )_{i=1}^{r_j}$  is an element  of 
$\Xi_{n, \CL}^{\times r_j}$
 such   that for 
every closed edge $\{ b, b' \} \in E^{\mr{cl}}$ of $\GR$ with  $\zeta (b) = v_{j}$, $\zeta (b') = v_{j'}$ (for some $j$, $j' \in \{ 1, \cdots, J \}$),
the equality  $\rho^{j}_{\lambda_{j} (b)} = (\rho^{j'}_{\lambda_{j'} (b')})^\veebar$ holds.
 \ede

\bde \label{QG8}
\begin{itemize}
\item[(i)]
A {\bf dormant $\mbG$-$\mr{PGL}_n^{(\N)}$-oper}  on $\msX$ is a collection
\begin{align}
\msE_\mbG^\spadesuit := (\mcE_B, \{ \STR_j \}_{j=1}^J),
\end{align}
where $\mcE_B$ denotes a $B$-bundle on $X$ and each $\STR_j$ ($j=1, \cdots, J$) denotes an $(\N-1)$-PD stratification  
  on
   $(\mcE_B \times^B \mr{PGL}_n)_j \left(:= \mcE_{B, j} \times^B \mr{PGL}_n\right)/X_j^\mr{log}/S_j^\mr{log}$, satisfying the following conditions:
\begin{itemize}
\item
For each $j =1, \cdots, J$,  the pair $\msE^\spadesuit_{\mbG, j} := (\mcE_{B, j}, \STR_j)$ forms a dormant $\mr{PGL}_n^{(\N)}$-oper on $\msX_j$.
\item
For each closed edge $\{ b, b' \} \in E^{\mr{cl}}$ of $\GR$ with $\zeta (b) = v_j$, $\zeta (b') = v_{j'}$ (for some $j, j' \in \{1, \cdots, J \}$), the equality $\rho_{\lambda_j (b)} (\msE^\spadesuit_{\mbG, j})= (\rho_{\lambda_{j'}(b')} (\msE^\spadesuit_{\mbG, j'}))^\veebar$ holds.
\end{itemize}
\item[(ii)]
Let $\msE_{\circ, \mbG}^\spadesuit := (\mcE_{\circ, B}, \{ \STR_{\circ, j} \}_{j=1}^J)$ and 
$\msE_{\bullet, \mbG}^{\spadesuit} := (\mcE_{\bullet, B}, \{ \STR_{\bullet, j} \}_{j=1}^J)$
be dormant $\mbG$-$\mr{PGL}_n^{(\N)}$-opers on $\msX$.
An {\bf isomorphism of dormant $\mbG$-$\mr{PGL}_n^{(\N)}$-opers}  from 
$\msE_{\circ, \mbG}^\spadesuit$ to $\msE_{\bullet, \mbG}^\spadesuit$ is defined as an isomorphism $\mcE_{\circ, B} \isom \mcE_{\bullet, B}$ of $B$-bundles such  that, for every $j=1, \cdots, J$, the induced isomorphism $\mcE_{\circ, B,  j} \isom \mcE_{\bullet, B,  j}$ defines an isomorphism of $\mr{PGL}_n^{(\N)}$-opers $\msE^\spadesuit_{\circ, \mbG,  j} \isom \msE_{\bullet, \mbG,  j}^{\spadesuit}$.
\end{itemize}
\ede

\bde
Let $\rho_\mbG := \{\rho^j \}_{j=1}^J$ be a set of $\mbG$-$\Xi_{n, \CL}$-radii  and 
$\msE_\mbG^\spadesuit := (\mcE_B, \{ \STR_j \}_{j=1}^J)$ a  dormant $\mbG$-$\mr{PGL}_n^{(\N)}$-oper on $\msX$.
Then, we shall say that $\msE_\mbG^\spadesuit$ is {\bf of radii $\rho_\mbG$} if
for any $j=1, \cdots, J$, the dormant  $\mr{PGL}_n^{(\N)}$-oper  $\msE^\spadesuit_{\mbG, j}$ is of radii $\rho^j$.
\ede

\SSP
\bpr  \label{YY277}
\begin{itemize}
\item[(i)]
Any dormant $\mbG$-$\mr{PGL}_n^{(\N)}$-oper on $\msX$ does not have nontrivial automorphisms.
\item[(ii)]
Suppose that $\N =1$, and  we set $\N' := \CH +1$.
Denote by
 $\msX_0$
 the reduction of $\msX$ modulo $p$.
Also, let $\msE_\mbG^\spadesuit := (\mcE_B, \{ \STR_j \}_{j})$ be a dormant $\mbG$-$\mr{PGL}_n^{(1)}$-oper on $\msX$.
Then, the collection of data
\begin{align} \label{YY278}
{^{\Diag}}\!\!\msE_\mbG^\spadesuit := (\mcE_{B, 0}, \{ {^{\Diag}}\!\!\STR_j \}_{j=1}^J)
\end{align}
forms a dormant $\mbG$-$\mr{PGL}_n^{(\N')}$-oper on $\msX_0$.
(We shall refer to ${^{\Diag}}\!\!\msE_\mbG^\spadesuit$ as the {\bf diagonal reduction} of $\msE_\mbG^\spadesuit$.)
If, moreover, $\rho_\mbG$ is a set of $\mbG$-$\Xi_{n, \N}$-radii,
then $\msE^\spadesuit$ is of radii $\rho_\mbG$ if and only if ${^{\Diag}}\!\!\msE_\mbG^\spadesuit$ is of radii $\rho_\mbG$.
\end{itemize}
\epr
\begin{proof}
Assertion (i) follows from Proposition \ref{P44}.
Also, assertion (ii) follows from the various definitions involved and the functor ${^{\Diag}}\!\!(-)$ (cf. \eqref{YY329}).
\end{proof}
 \SSP

 For each $j=1, \cdots, J$, denote by ${^\dagger}\mcE_{B, j}$ the canonical $B$-bundle  on $X_j$ associated to the log curve $X_j^\mr{log}/S^\mr{log}_j$ (cf. Section \ref{SSnormal} or ~\cite[Eq.\,(605)]{Wak8}).
Also, set ${^\dagger}\mcE_j := {^\dagger}\mcE_{B, j} \times^B \mr{PGL}_n$.
Using the isomorphism $\delta_j^\mcD$ (cf.  \eqref{fff02}),  we obtain  isomorphisms
\begin{align}
\delta_{j}^{\mcE_B} :  \mr{Clut}^*_j ({^\dagger}\mcE_{B}) \xrightarrow{\sim} {^\dagger}\mcE_{B, j}, \hspace{5mm}
\delta_{j}^\mcE :  \mr{Clut}^*_j ({^\dagger}\mcE) \xrightarrow{\sim} {^\dagger}\mcE_{j}
\end{align}

\bde[cf. ~\cite{Wak8}, Definition 7.11] \label{Defg67}
A dormant $\mbG$-$\mr{PGL}_n^{(\N)}$-oper on $\msX$ is said to be {\bf normal}
if it is of the form $\msE^\spadesuit_\mbG := ({^\dagger}\mcE_B, \{ \phi_j \}_{j=1}^J)$ such that for each $j=1, \cdots, J$, the $\mr{PGL}_n^{(\N)}$-oper  $({^\dagger}\mcE_{B, j}, \delta_{j*}^\mcE(\phi_j))$ is normal   in the sense of Definition \ref{Defg4}, where $\delta_{j*}^\mcE(\phi_j)$ denotes the $(\N -1)$-PD stratification corresponding to $\phi_j$ via $\delta_{j}^\mcE$.
(By Proposition-Definition \ref{Prop933}, each dormant $\mbG$-$\mr{PGL}_n^{(\N)}$-oper is isomorphic to a normal one.)
\ede
\SSP

Let us choose a dormant $n^{(\N)}$-theta characteristic $\vartheta := (\varTheta, \DMO_{\vartheta})$ of $X^\mr{log}/S^\mr{log}$.
Also, let $\msE^\spadesuit := (\mcE_B, \STR)$ be a dormant $\mr{PGL}_n^{(\N)}$-oper on $\msX$.
By Theorem \ref{P14}, $\msE^\spadesuit$ defines a dormant $(\mr{GL}_n^{(\N)}, \vartheta)$-oper $\DMO^\diamondsuit$ on $\msX$ via the isomorphism $\Lambda_{\diamondsuit \Rightarrow \spadesuit, \vartheta}^\ZZZ$.
It follows from Proposition \ref{YY157} that  the collection $\{ \DMO^\diamondsuit |_{j} \}_{j=1}^J$ (cf. \eqref{YY153})  forms a $\mbG$-$\mcD^{(\N -1)}$-module structure on $\mcF_\varTheta$.
For each $j=1, \cdots, J$,  denote by $\mcE_{B, j}$ (resp., $\mcE_j$)  the restriction of $\mcE_B$ (resp.,  $\mcE := \mcE_B \times^B \mr{PGL}_n$) to $X_j$, and by 
  $\STR_j$  the $(\N-1)$-PD stratification on $\mcE_j/X_j^\mr{log}/S_j^\mr{log}$ induced from $\DMO^\diamondsuit |_{j}$ via projectivization.
  Then,  
the pair $\msE^\spadesuit_{\mbG, j} := (\mcE_{B, j}, \STR_j)$ forms a dormant $\mr{PGL}_n^{(\N)}$-oper on $\msX_j$.
Let    $\{ b, b' \} \in E^{\mr{cl}}$ be  
a closed edge of $\GR$  with $\zeta (b) = v_j$, $\zeta (b') = v_{j'}$ (for some $j, j' \in \{1, \cdots, J \}$),   and 
 write $i := \lambda_j (b)$,   $i' := \lambda_{j'}(b')$.
Proposition \ref{L090} implies the equality 
\begin{align} \label{YY293}
e_i (\DMO^\diamondsuit |_{j}) = e_{i'} (\DMO^\diamondsuit |_{j'})^\veebar \ \left(\Longrightarrow \rho_i (\msE_{\mbG, j}^\spadesuit) = (\rho_{i'} (\msE_{\mbG, j'}^\spadesuit))^\veebar \right).
\end{align}
 Thus,
 the collection 
 \begin{align}
 \msE^\spadesuit_{\emptyset \Rightarrow \mbG} := (\mcE_B, \{ \STR_j \}_{j=1}^J)
 \end{align}
  specifies   a dormant $\mbG$-$\mr{PGL}_n^{(\N)}$-oper on $\msX$.

Since the formations of $\Lambda_{\diamondsuit \Rightarrow \spadesuit, \vartheta}$ and $\Lambda_{\clubsuit \Rightarrow \diamondsuit, \vartheta}$ commute with pull-back via $\mr{Clut}_j$ ($j=1, \cdots, J$), 
$ \msE^\spadesuit_{\emptyset \Rightarrow \mbG}$ is verified to be normal whenever $\msE^\spadesuit$ is normal.

\SSP 
\bpr \label{YY188}
\begin{itemize}
\item[(i)]
The functor
 \begin{align} \label{YY291}
 \begin{pmatrix}
\text{the  groupoid of normal dormant} \\
\text{$\mr{PGL}_n^{(\N)}$-opers on $\msX$} 
\end{pmatrix}
\rightarrow
\begin{pmatrix}
\text{the groupoid of normal dormant} \\
\text{$\mbG$-$\mr{PGL}_n^{(\N)}$-opers on $\msX$} 
\end{pmatrix}
\end{align}
given by assigning  $\msE^\spadesuit \mapsto \msE^\spadesuit_{\emptyset \Rightarrow \mbG}$  is an equivalence of categories.
Moreover,  this equivalence commutes with the formation of diagonal reductions.
\item[(ii)]
Let $\rho_\mbG := \{ \rho^j \}_{j=1}^J$ be a set of $\mbG$-$\Xi_{n, \CL}$-radii.
Then, there exists a canonical equivalence of categories
 \begin{align} \label{YY391}
\begin{pmatrix}
\text{the groupoid of normal dormant} \\
\text{$\mbG$-$\mr{PGL}_n^{(\N)}$-opers of radii $\rho_\mbG$ on $\msX$} 
\end{pmatrix}
\isom 
 \prod_{j=1}^J \begin{pmatrix}
\text{the  groupoid of  normal dormant} \\
\text{$\mr{PGL}_n^{(\N)}$-opers of radii $\rho^j$ on $\msX_j$} 
\end{pmatrix}.
\end{align}
Moreover,  this equivalence commutes with the formation of diagonal reductions.
\end{itemize}
 \epr
\begin{proof}
First, we shall prove assertion (i).
Since  the case $\N =1$ can be established by an argument  similar to the proof  of  ~\cite[Proposition 7.12]{Wak8} (where we deal with the $\qq$-normality of opers), it suffices to prove the case where $\N \geq 2$.
To this end,
we construct an assignment inverse to $\msE^\spadesuit \mapsto \msE^\spadesuit_{\emptyset \Rightarrow \mbG}$. 
Let   $\msE_\mbG^\spadesuit := (\mcE_B, \{ \STR_j \}_{j=1}^J)$ be a normal dormant $\mbG$-$\mr{PGL}_n^{(\N)}$-oper on $\msX$.
Note that for each $j =1, \cdots, J$, the pair of restrictions $\vartheta_j := (\varTheta_j, \DMO_{\vartheta}|_j)$, where $\varTheta_j := \mr{Clut}_j^*(\varTheta)$,  forms a dormant $n^{(\N)}$-theta characteristic of $X_j^\mr{log}/S_j^\mr{log}$. 
Denote by $\DMO_j^\diamondsuit$ ($j=1, \cdots, J$) the  dormant $(\mr{GL}_n^{(\N)}, \vartheta_j)$-oper on $\msX_j$ corresponding to $\msE_j^\spadesuit := (\mcE_{B, j}, \STR_j)$.

Let us take a closed edge $\{b_1, b_2 \} \in E^{\mr{cl}}$
 of $\GR$ with $\zeta (b_1) = v_{j_1}$, $\zeta (b_2) = v_{j_2}$ (for some $j_1, j_2 \in \{ 1, \cdots, J\}$).
 Then, the equalities  $e_{\lambda_{j_1}(b_1)} (\DMO_{\vartheta}|_{j_1}) = - e_{\lambda_{j_2} (b_2)} (\DMO_{\vartheta}|_{j_2})$  and $\rho_{\lambda_{j_1}(b_1)} (\msE^\spadesuit_{j_1}) = \rho_{\lambda_{j_2}(b_2)}(\msE^\spadesuit_{j_2})^\veebar$ hold (cf. Propositions \ref{YY157}, (i), and \ref{YY277}).
We set $\sigma := \mr{Clut}_{j_1} \circ  \sigma_{j_1, \lambda_{j_1} (b_1)}  = \mr{Clut}_{j_2}\circ \sigma_{j_2, \lambda_{j_2} (b_2)}$.
Denote by
$\widehat{U}$ (resp.,  $\widehat{U}_{1}$; resp.,  $\widehat{U}_{2}$)  the formal neighborhood of $\mr{Im} (\sigma)$ (resp., $\mr{Im} (\sigma_{j_1, \lambda_{j_1} (b_1)})$; resp., $\mr{Im} (\sigma_{j_2, \lambda_{j_2}(b_2)})$) in $X$ (resp., $X_{j_1}$; resp., $X_{j_2}$).

Since the formation of a (normal) dormant $\mbG$-$\mr{PGL}_n^{(\N)}$-oper has descent with respect to the \'{e}tale topology,
 {\it we are always free to replace $S$ with its  \'{e}tale covering}.
Hence, we may assume that 
there exist a  generator $w$ of the line bundle  $\varTheta$ over $\widehat{U}$ and 
 identifications
 $U_\otimes \xrightarrow{\sim} \widehat{U}$, $U_\oslash \xrightarrow{\sim} \widehat{U}_{1}$, and $U_\oslash \xrightarrow{\sim} \widehat{U}_{2}$ making the following diagrams ($\ell =1, 2$) commute:
\begin{align} \label{EE900}
\vcenter{\xymatrix@C=46pt@R=36pt{
U_\oslash \ar[r]^-{\sim} \ar[d]_-{\iota_\ell} & \widehat{U}_{\ell}\ar[d]^{\mr{inclusion}}  \\ 
U_\otimes \ar[r]_-{\sim} & \widehat{U}
}}
\end{align}
For $\ell =1, 2$, let $\{ \partial_\ell^{\langle j' \rangle} \}_{j'}$  denote the local basis of $\mcD_{j_\ell}^{(\N -1)}$
 at $\sigma_{j_\ell, \lambda_{j_\ell} (b_\ell)}$
    associated to $t_\ell$
       defined as in  Section \ref{SS04f4}.
By mean of  the identification $\varTheta = \mcO_{\widehat{U}}$  given by $w$ (which induces an identification $\varTheta |_{\widehat{U}_{\ell}} = \mcO_{\widehat{U}_{j_\ell}}$),
the restriction $(\mcF_{\varTheta_{j_\ell}}, \nabla_{j_\ell}^\diamondsuit) |_{\widehat{U}_{j_\ell}}$ of $(\mcF_{\varTheta_{j_\ell}}, \nabla_{j_\ell}^\diamondsuit)$ can be identified with $\bigoplus_{j' = 0}^{n -1} \mcO_{\widehat{U}_{j_\ell}} \partial^{\langle j' \rangle}_{\ell}$. 
The natural isomorphism $\sigma_{1 \Rightarrow 2} : \sigma^*_{j_1, \lambda_{j_1}(b_1)}(\mcF_{\varTheta_{j_1}}) \xrightarrow{\sim} \sigma^*_{j_2, \lambda_{j_2}(b_2)}(\mcF_{\varTheta_{j_2}})$ is given by  $\partial_1^{\langle  0\rangle} \mapsto \partial_2^{\langle 0 \rangle}$ and  $\partial_1^{\langle j' \rangle} \mapsto (-1)^{j'} \cdot \sum_{\ell =1}^{j'} \binom{j' -1}{\ell -1} \cdot \partial_2^{\langle \ell \rangle}$ for $j' \geq 1$
(cf. Lemma \ref{P238}).

Let us  fix an isomorphism of $\mcD_{j_\ell}^{(\N -1)}$-modules 
\begin{align}
h_{\ell} : (\mcF_{\varTheta_{j_\ell}}, \DMO^\diamondsuit_{j_\ell})|_{\widehat{U}_{\ell}}
\left(=  \bigoplus_{j' = 0}^{n -1} \mcO_{\widehat{U}_{\ell}} \partial^{\langle j' \rangle}_{\ell}\right)
 \isom \bigoplus_{s =1}^n \msO_{\oslash, d_{s}^\ell}^{(\N -1)},
\end{align}
 where $[d_{1}^{\ell}, \cdots, d_{n}^{\ell}]$ is  the exponent of $\DMO^{\diamondsuit \Rightarrow \heartsuit}_{j_\ell}$ at $\sigma_{j_\ell, \lambda_{j_\ell}(b_\ell)}$.
 It follows from  Proposition \ref{WWed3} that, after possibly rearranging  the indexes,
we obtain the equality  $d^1_{s} = -d_{s}^{2}$  for every $s =1, \cdots, n$.

 We set $(v_{\ell, 1}(t), \cdots, v_{\ell, n}(t)) := h_\ell (\partial_\ell^{\langle 0 \rangle}) \in H^0 (\widehat{U}_{\ell}, \mcO_S [\![t ]\!]^\times)^{\oplus n}$ (cf. Proposition \ref{NN77}, (i)).
 According to Proposition \ref{NN77}, (i), we may assume, 
after possibly  composing $h_{\ell}$  with an automorphism of $\bigoplus_{s =1}^n \msO_{\oslash, d^\ell_{s}}^{(\N -1)}$ determined  by an element of $\mr{GL}_n (\mcO_S)$,
 that 
 $v_{s} :=v_{1, s} (0) = v_{2, s} (0)$ for every $s =1, \cdots, n$.
 Then, 
 the restriction $\overline{h}_\ell$ of $h_\ell$ to  $\sigma_{j_\ell, \lambda_{j_\ell} (b_\ell)}$ satisfies 
 the following sequence of equalities  for any $j'$ with $1 \leq j' \leq n-1$:
 \begin{align}
  \overline{h}_\ell (\partial^{\langle j' \rangle}_\ell |_{\sigma_{j_\ell, \lambda_{j_\ell} (b_\ell)}}) &= \overline{h}_\ell (\nabla^{\diamondsuit}_{j_\ell} (\partial^{\langle j' \rangle}) (\partial^{\langle 0 \rangle})|_{\sigma_{j_\ell, \lambda_{j_\ell} (b_\ell)}})\\
&  =\nabla_{\vec{d}^\ell} (\partial^{\langle j' \rangle}) (h_\ell (\partial^{\langle 0 \rangle}) |_{\sigma_{j_\ell, \lambda_{j_\ell} (b_\ell)}})\notag \\
& = (\nabla^{(\N -1)}_{\oslash,  d_s^{\ell}} (\partial^{\langle j' \rangle}) (v_{\ell, s} (t)) |_{t= 0})_{s =1}^n \\
& =  (\binom{-\widetilde{d}_s^\ell}{j'} \cdot v_{s})_{s=1}^n,
 \end{align}
 where $\nabla_{\vec{d}^\ell} := \bigoplus_{s=1}^n \nabla_{\oslash, d_s^\ell}^{(\N -1)}$, and  the first equality follows from the assumption that $\msE^\spadesuit_j$ is normal, i.e., $\nabla_j^{\diamondsuit}$ arises from an $(n, \vartheta)$-projective connection.
 Hence, we have
 \begin{align}
 (\overline{h}_2 \circ \sigma_{1 \Rightarrow 2}) (\partial_1^{\langle j' \rangle}|_{\sigma_{j_1, \lambda_{j_1} (b_1)}})
 & =  \overline{h}_2 ((-1)^{j'} \cdot \sum_{\ell =1}^{j'} \binom{j' -1}{\ell -1} \cdot \partial_2^{\langle \ell \rangle}|_{\sigma_{j_2, \lambda_{j_2} (b_2)}}) \\
 & =  (-1)^{j'} \cdot \sum_{\ell =1}^{j'} \binom{j' -1}{\ell -1} \cdot \overline{h}_2 (\partial_2^{\langle \ell \rangle}|_{\sigma_{j_2, \lambda_{j_2} (b_2)}}) \\
 & = \left( (-1)^{j'} \cdot \sum_{\ell =1}^{j'} \binom{j' -1}{\ell -1} \cdot \binom{-\widetilde{d}_s^2}{\ell} \cdot v_s \right)_{s=1}^n \notag \\
 & = \left( \binom{\widetilde{d}^2_s}{j'} \cdot v_s \right)_{s=1}^n \notag \\
 & = (\binom{-\widetilde{d}_s^1}{j'} \cdot v_{s})_{s=1}^n \notag \\
 & =  \overline{h}_1 (\partial^{\langle j' \rangle}_1 |_{\sigma_{j_1, \lambda_{j_1} (b_1)}}).
 \end{align}
It follows that $\overline{h}_2^{-1} \circ \overline{h}_1 = \sigma_{1 \Rightarrow 2}$.
Since  $\mu (\nabla_{\vec{d}^1}) \circ \mr{sw} = \mu (\nabla_{\vec{d}^2})$,
we have $\mu (\nabla^\diamondsuit_{j_1}) \circ \mr{sw} = \mu (\nabla^\diamondsuit_{j_2})$.

Thus,
 the collections  $(\mcF_{\varTheta_j}, \DMO_j^{\diamondsuit}, \{ \mcF_{\varTheta_j}^s \}_{s=0}^n)$ ($j=1, \cdots, J$) may  be glued together to  obtain 
a dormant $\mr{GL}_n^{(\N)}$-oper $\msF^\heartsuit$ (cf. Proposition \ref{P27}, (ii)).
Denote by $\msE^\spadesuit_{\mbG \Rightarrow \emptyset}$ the  dormant $\mr{PGL}_n^{(\N)}$-oper induced from $\msF^\heartsuit$.
Then, it is immediately verified that the assignment $\msE^\spadesuit_\mbG \mapsto \msE^\spadesuit_{\mbG \Rightarrow \emptyset}$ determines an inverse to the functor \eqref{YY291}.
This completes the proof of the first assertion of (i).
The second  assertion of (i) follows from the definitions of \eqref{YY291} and ${^{\Diag}}\!\!(-)$.

Next, we shall prove assertion (ii).
The assignment $\msE^\spadesuit_{\mbG} := ({^\dagger}\mcE_B, \{ \STR_j \}_{j=1}^J) \mapsto (\msE_{\mbG, j}^\spadesuit)_{j=1}^J$ (where $\msE^\spadesuit_{\mbG, j} := ({^\dagger}\mcE_{B, j}, \STR_j)$) defines a functor  of the form \eqref{YY391}.
This functor is immediately verified to commute  with the formation of diagonal reductions.
By assertion (i), we see that this functor defines a required equivalence of categories.
\end{proof}
\SSP

\begin{rema}[The choice of an $n^{(\N)}$-theta characteristic] \label{RR223}
Despite the fact that
 a specific $n^{(\N)}$-theta characteristic $\vartheta$  is fixed in the proof of     assertion (i),
the resulting functor  \eqref{YY291} does not depend on the choice of $\vartheta$.
\end{rema}
\SSP

Just as in   the discussion of ~\cite[Section 7.1.3]{Wak8},
 we obtain the  morphism between moduli stacks
\begin{align} \label{WW81}
\mr{Clut}_\mbG : \prod_{j=1}^{J} \overline{\mcM}_{g_j, r_j} \migi \overline{\mcM}_{g,r},
\end{align}
given by assigning $\{ \msX_j \}_{j=1}^J \mapsto \msX$,
where the product ``$\prod$" in the domain  is taken over $R$.

Let $\rho_\mbG := \{ \rho^j \}_{j=1}^J$ be a set of $\mbG$-$\Xi_{n, \CL}$-radii  and 
$\msE^\spadesuit_\mbG$ be a normal dormant $\mbG$-$\mr{PGL}_n$-oper on $\msX$ of radii $\rho_\mbG$.
By 
the equivalence of categories \eqref{YY291}  resulting from 
Proposition \ref{YY188}, (i),
$\msE^\spadesuit_\mbG$
corresponds to 
a normal dormant   $\mr{PGL}_n$-oper $\msE_{\mbG \Rightarrow \emptyset}^\spadesuit$   on $\msX$.
 The radii  of $\msE_{\mbG \Rightarrow \emptyset}^\spadesuit$ can be determined by   $\rho^j$ ($j = 1, \cdots, J$) in accordance with   the clutching data $\mbG$ (cf. Remark \ref{Rem7878}); we denote the resulting radii of $\msE_{\mbG \Rightarrow \emptyset}^\spadesuit$  by 
\begin{align} \label{fff036}
\rho_{\mbG \Rightarrow \emptyset} \in \Xi_{n, \CL}^{\times r}.
 \end{align}

On the other hand, it follows from Proposition-Definition \ref{Prop933} and Proposition \ref{YY188}, (ii), that
 any point of $\prod_{j =1}^J \mcO p_{\rho^j, g_j,r_j}^{^\mr{Zzz...}}$ classifies  a  unique (up to isomorphism) normal  dormant 
$\mbG$-$\mr{PGL}_n^{(\N)}$-oper of radii $\rho_\mbG$.
Hence, the assignment $(\{\msX_j\}_{j=1}^J, \msE^\spadesuit_{\mbG}) \mapsto (\msX, \msE_{\mbG \Rightarrow \emptyset}^\spadesuit)$  determines
a well-defined morphism
\begin{align} \label{YY337}
\mr{Clut}_{\mbG, \rho_\mbG} : \prod_{j =1}^J \mcO p_{\rho^j, g_j,r_j}^{^\mr{Zzz...}} \migi  
\mcO p_{\rho_{\mbG \Rightarrow \emptyset}, g,r}^{^\mr{Zzz...}}.
\end{align}

\SSP
\bt \label{y0176}
Let us keep the above notation.
Also, let $\rho$ be  an element of $\Xi_{n, \CL}^{\times r}$ (where $\rho := \emptyset$ if $r  =0$).
Then, the following commutative square diagram  is Cartesian:
\begin{align} \label{1050}
\vcenter{\xymatrix@C=76pt@R=36pt{
\coprod_{ \rho_\mbG := \{ \rho^j\}_{j=1}^J}\prod_{j =1}^J \mcO p^\ZZZ_{\rho^j, g_j,r_j} 
\ar[r]^-{\coprod_{\rho_\mbG} \mr{Clut}_{\mbG, \rho_\mbG}} \ar[d]_-{\coprod_{\rho_\mbG} \prod_j \Pi_{\rho^j, g_j, r_j}}& \mcO p^\ZZZ_{\rho, g,r} 
\ar[d]^-{\Pi_{\rho, g, r}}\\
 \prod_{j=1}^J \overline{\mcM}_{g_j, r_j}  \ar[r]_-{\mr{Clut}_\mbG}& \overline{\mcM}_{g,r},
}}
\end{align}
 where the products ``$\prod$'' are taken over $R$ and 
  the disjoint union on the upper-left corner runs  over  the  sets of $\mbG$-$\Xi_{n, \CL}$-radii $\rho_\mbG$  with $\rho_{\mbG \Rightarrow \emptyset} = \rho$. 
  Finally, this diagram is compatible with the functors ${^{\Diag}}\!\!(-)$ (cf.  \eqref{YY329}) in a natural sense.
 \et
\begin{proof}
The assertion follows from Proposition \ref{YY188}, (i) and (ii).
\end{proof}

\LSP
\subsection{Nonemptiness of the moduli space} \label{SS30d2}

Let $S$ be as before.
Denote  by $\mbP$ the projective line over $S$,  i.e., $\mbP := \mcP roj (\mcO_S [x, y])$.
Also, denote by 
$[0]$, $[1]$, and $[\infty]$ the $S$-rational points of   $\mbP$ determined by the values $0$, $1$, and $\infty$, respectively.
After ordering  the three points $ [0], [1], [\infty]$, we obtain 
 a unique (up to isomorphism) $3$-pointed stable   curve 
\begin{align} \label{1051}
\msP := (\mbP/S, \{ [0], [1], [\infty] \})
\end{align}
 of genus $0$ over $S$.
 In particular, we obtain a log curve $\mbP^\mr{log}$ over $S \left(= S^\mr{log} \right)$.

We shall write $\mcL: = \mcO_{\mbP}(-1) \otimes \mcO_{\mbP}([0]+ [1]+ [\infty])$, and write $\kappa$ for the $\mcO_{\mbP}$-linear injection $\mcO_{\mbP}(-1) \migiincl \mcO_\mbP^{\oplus 2}$ given by $w \mapsto (wx, w y)$ for each local section
$w \in \mcO_\mbP (-1)$.
Also, let $\mcF$ be a vector bundle on $\mbP$ which makes the following square diagram cocartesian:
\begin{align} \label{dE61}
\vcenter{\xymatrix@C=46pt@R=36pt{
\mcO_\mbP (-1)\ar[r]^-{\kappa} \ar[d]_-{\mr{inclusion}} & \mcO_\mbP^{\oplus 2}  \ar[d] \\
\mcL \ar[r] & \mcF.
}}
\end{align}
The trivial $\mcD^{(\N -1)}$-module structure $(\DMO_{\mbP, \mr{triv}}^{(\N -1)})^{\oplus 2}$ on $\mcO_\mbP^{\oplus 2}$ extends uniquely to a $\mcD^{(\N -1)}$-module structure  $\DMO_\mcF^{(\N -1)}$ on $\mcF$.

Note that  the $\mcO_{\mbP}$-linear  composite
\begin{align} \label{dE100}
\mcL \migiincl  \mcF \xrightarrow{\DMO_\mcF^{(\N -1) \Rightarrow (0)}} \Omega \otimes \mcF \migisurj \Omega \otimes (\mcF/\mcL)
\end{align} 
is an isomorphism.
In fact, on the open subscheme $\mbP \setminus \{ [\infty] \} = \mcS pec (\mcO_S [z])$, where $z = x/y$, we consider two sections  $v := \frac{1}{z(z-1)}(z, 1)$ and $w := (1, 0)$ of $\mcF$.
These sections form a basis of $\mcF$, and $v$  defines a local generator of $\mcL$.
Since $\nabla_\mcF^{(\N -1)\Rightarrow (0)} (v) = \frac{dz}{z (z-1)} \otimes w - \frac{d(z(z-1))}{z(z-1)} \otimes v$,
the section $v$ is mapped to $w' := \frac{dz}{z(z-1)} \otimes \overline{w}$ under \eqref{dE100}.
The section  $w'$ defines a local generator of $\Omega \otimes (\mcF/\mcL)$,  so  \eqref{dE100} turns out to be an isomorphism over  $\mbP \setminus \{ [\infty] \}$.
Applying an similar argument  to a neighborhood of $[\infty]$,
we conclude that  \eqref{dE100}  is an isomorphism over the entire space.
Thus, the triple 
\begin{align}
\msF^{\heartsuit}_{\N, S} := (\mcF, \DMO_\mcF^{(\N -1)}, \mcL)
\end{align}
 forms a  $\mr{GL}_2^{(\N)}$-oper on $\msP$.
 
 Also,
 we shall set 
 \begin{align} \label{YY400}
 \msE^{\spadesuit}_{\N, S} := \msF^{\heartsuit\Rightarrow \spadesuit}_{\N, S},
 \end{align}
 which defines a $\mr{PGL}_2^{(\N)}$-oper on $\msP$.
  Both $\msF^{\heartsuit}_{1, S}$ and  $\msE^{\spadesuit}_{1, S}$ are  dormant 
 and  
satisfy
\begin{align} \label{dE101}
{^{\Diag}}\!\!\msF^{\heartsuit}_{1, S} = \msF^{\heartsuit}_{\N, S_0}, \hspace{10mm}
{^{\Diag}}\!\!\msE^{\spadesuit}_{1, S} = \msE^{\spadesuit}_{\N, S_0}.
\end{align}

Denote by 
\begin{align} \label{YY334}
\epsilon
  \in 
   (\mbZ/p^\CL \mbZ)^\times/\{ \pm 1 \}
\end{align}
 the element  defined as  the image of $\frac{1}{2} \in (\mbZ/p^\CL \mbZ)^\times$ via the natural  quotient
$\mbZ/p^\CL \mbZ \migisurj (\mbZ/p^\CL \mbZ)/\{ \pm 1 \}$.
(By the bijection \eqref{wpaid}, we occasionally regard it as an element of $\mfS_2 \backslash (\mbZ/p^\CL \mbZ)^{\times 2}/\Delta$.)
 The exponent of $\msF^{\heartsuit}_{\N, S_0}$ at every marked point coincides with  the multiset $[1, 0]$ (cf. Proposition \ref{Prop15} for a more general assertion), so we have $\rho_i (\msE^{\spadesuit}_{\N, S_0}) = \epsilon$ for every $i=1, \cdots, r$.

Here, we shall make the definition of a totally degenerate curve over a general base space $S$ (cf., e.g.,   ~\cite[Definition 7.15]{Wak8} for the case where $S$ is the spectrum of  an algebraically closed field).

\SSP
\bde \label{DD3WW}
 Let  $\msX := (X/S, \{ \sigma_i \}_{i=1}^r)$ be  an $r$-pointed stable curve of genus $g$ over $S$.
We say  that $\msX$ is  {\bf totally degenerate}
if there exists  a pair $(\varsigma, \mbG)$ consisting of an  element $\varsigma$ of the symmetric group  of $r$ letters $\mfS_r$ and   trivalent clutching data $\mbG := (\GR, \{\lambda_j \}_{j=1}^J)$  of type $(g, r)$ (cf. Definition \ref{Def112})
satisfying  the following conditions:
\begin{itemize}
\item
$\GR$ 
coincides with  the dual semi-graph of $\msX$ (cf. ~\cite[Section 7.1.2]{Wak8});
\item
$\msX$ is isomorphic to $\msY^\varsigma$ (cf. \eqref{Curve22} for the definition of $(-)^\varsigma$), where $\msY$ denotes the pointed stable curve  obtained by gluing together $J$ copies of  the $3$-pointed  projective line  $\msP$ by means of $\mbG$ in the manner of ~\cite[Section 7.2.1]{Wak8}.
(Note that the ordering of the marked points in $\msY$
  is compatible with that of  elements in $B_\circledast$ defined in the manner of Remark \ref{Rem7878}.)
\end{itemize}
The isomorphism class of $\msX$ depends only on  
such a pair $(\varsigma, \mbG)$, and 
 we will say 
 that {\bf $(\varsigma, \mbG)$ induces  $\msX$}.
If, moreover,  we can choose the element $\varsigma$ as the identity permutation,  then we say that {\bf $\mbG$  induces $\msX$}
\ede

\SSP
\bt \label{T52}
Let $\msX$ be  an $r$-pointed totally  degenerate curve  of genus $g$ over 
$S$.
We shall  set  
\begin{align} \label{NN8122}
\epsilon^{\times r} := 
(\epsilon, \cdots, \epsilon)
 \in ((\mbZ/p^\CL \mbZ)^\times / \{ \pm 1 \})^{\times r},
\end{align}
 where $\epsilon^{\times r} := \emptyset$ if $r =0$.
Then, 
 there exists a dormant $\mr{PGL}_2^{(\N)}$-oper  on $\msX$ of radii $\epsilon^{\times r}$.
In particular, the category $\mcO p^\ZZZ_{\epsilon^{\times r}, g, r}$ (hence also $\mcO p^\ZZZ_{g,r}$) for $n =2$ is nonempty.
\et
\begin{proof}
We may assume, without loss of generality, that there exists
 trivalent clutching data inducing  $\msX$ (cf. Remark \ref{Rem3391i}).
 In particular, $\msX$ may be obtained by gluing together  $3$-pointed projective lines $\msX_j \left(\cong \msP \right)$ ($j=1, \cdots, J$) by means of $\mbG$.
Since $\epsilon = \epsilon^\veebar$,
the dormant $\mr{PGL}_2^{(\N)}$-opers $\msE^\spadesuit_{\N, S}$ defined on the respective  components $\msX_j$  may be glued together to form a dormant $\mr{PGL}_2^{(\N)}$-oper on $\msX$ (cf. Proposition \ref{YY188}, (i) and (ii)).
The resulting  dormant $\mr{PGL}_2^{(\N)}$-oper   is verified to be of radii $\epsilon^{\times r}$, so this completes the proof of the assertion.
\end{proof}
\SSP

Next, let $S^\mr{log}$ and $f^\mr{log} : X^\mr{log} \migi S^\mr{log}$ be as in Section \ref{SS041}.
Also, let us take a $\mr{GL}_2^{(\N)}$-oper $\msF^\heartsuit := (\mcF, \DMO, \{ \mcF^j \}_{j=0}^2)$ on $X^\mr{log}/S^\mr{log}$.
For an integer $n$ with $1 < n < p$,
the $(n-1)$-st symmetric product $S^{n-1}(\mcF)$  of $\mcF$ over $\mcO_X$
 forms a rank $n$ vector bundle.
 It admits a $\mcD^{(N-1)}$-module structure  $S^{n-1}(\DMO)$  induced  naturally by $\DMO$.
Moreover,  $S^{n-1}(\mcF)$ is equipped  with an $n$-step decreasing filtration $\{ S^{n-1}(\mcF)^j\}_{j=0}^n$ induced from $\{ \mcF^j \}_j$;
to be precise,  we set $S^{n-1}(\mcF)^0 := S^{n-1}(\mcF)$, $S^{n-1}(\mcF)^n := 0$, and $S^{n-1}(\mcF)^j$ (for each $j =1, \cdots , n-1$) is defined as  
the image of $(\mcF^1)^{\otimes j} \otimes \mcF^{\otimes (n-1-j)}$ via the natural quotient $\mcF^{\otimes (n-1)} \migisurj S^{n-1}(\mcF)$.
This filtration satisfies 
\begin{align} \label{E61}
S^{n-1}(\mcF)^j/S^{n-1}(\mcF)^{j+1} \cong (\mcF^1)^{\otimes (n-1)} \otimes \mcT^{\otimes (n-1-j)}
\end{align}
 for every $j=0, \cdots, n-1$.
 By the assumption  $n < p$,  the collection 
\begin{align}\label{e810}
S^{n-1}(\msF^\heartsuit) := (S^{n-1}(\mcF), S^{n-1}(\DMO), \{ S^{n-1}(\mcF)^j \}_{j=0}^n)
\end{align}
is verified to  form a  $\mr{GL}_n^{(\N)}$-oper  on $X^\mr{log}/S^\mr{log}$.
If $\msF^\heartsuit$ is dormant, then so is $S^{n-1}(\msF^\heartsuit)$.

Since the formation of $S^{n-1}(\msF^\heartsuit)$ commutes with base-change over  $S$-schemes and preserves the equivalence class (cf. Definition \ref{D0f30}),
 the assignment $\msF^\heartsuit \mapsto S^{n-1}(\msF^\heartsuit)$ 
defines a morphism of functors
\begin{align} \label{e814}
\mcO p_{\heartsuit, n=2} \migi \mcO p_\heartsuit \ \left(\text{resp.,} \  \mcO p_{\heartsuit, n=2}^\ZZZ  \migi \mcO p_{\heartsuit}^\ZZZ \right),
\end{align}
where  $\mcO p_{\heartsuit, n=2}$ (resp., $\mcO p_{\heartsuit, n=2}^\ZZZ$) denotes the functor  $\mcO p_\heartsuit$ (resp., $\mcO p_{\heartsuit}^\ZZZ$) in the case of $n=2$.

The formation of this morphism commutes with base change over $S^\mr{log}$;
by considering  the case where the underlying log curve arises from the universal family of pointed curves over $\overline{\mcM}_{g, r}$,
we obtain, via the isomorphism $\Lambda_{\heartsuit \Rightarrow \spadesuit}^\ZZZ$ (cf. Theorem \ref{P14}),  a morphism of stacks
\begin{align} \label{EEE77}
\mcO p_{2, \N, g, r, R} \rightarrow \mcO p_{n, \N, g, r, R} \ \left(\text{resp.,} \  \mcO p_{2, \N, g, r, R}^\ZZZ \rightarrow \mcO p_{n, \N, g, r, R}^\ZZZ \right)
\end{align}
over $\overline{\mcM}_{g, r}$.
Using this morphism, one can induce from Theorem \ref{T52} the following assertion (cf. Corollary \ref{Cor333} for another nonemptiness assertion).



\SSP
\bco \label{C56}
Suppose that $R = \mbF_p$.
Also, let $n$ and  $\N$ be positive  integers with $1 < n < p$.
Then, the fibers of the projection $\Pi_{g, r}$ 
over the geometric points of $\overline{\mcM}_{g,r}$ classifying totally degenerate curves are  nonempty.
In particular, $\mcO p^\ZZZ_{g,r}$ is nonempty.
\eco


\LSP
\subsection{Factorization property of generic degrees} \label{SSeew}

Suppose that $\ell = 0$ and $R = \mbF_p$.
As a corollary of \ref{y0176},
we can  prove  that, under an \'{e}taleness assumption,  
the collection of 
the generic degrees $\mr{deg}(\Pi_{\rho, g, r})$ of $\Pi_{\rho, g, r}$
 satisfies  specific nice factorization properties determined in accordance with  
 various clutching morphisms $\mr{Clut}_\mbG$.

Let $\rho$ be an element of $\Xi_{n, \N}^{\times r}$ and $\mbG :=  ( \GR, \{ (g_j, r_j) \}_{j=1}^J, \{ \lambda_j\}_{j=1}^J)$ clutching data   of type $(g, r)$.
Assume   that 
$\Pi_{\rho, g, r}$  and  $\Pi_{\rho^j, g_j, r_j}$'s  ($j=1, \cdots, J$) are  finite and \'{e}tale  over  \'{e}tale  neighborhoods of all the points  in $\overline{\mcM}_{g, r}$ and $\overline{\mcM}_{g_j, r_j}$, respectively,  classifying totally degenerate curves.
Then,   
the  morphisms  $\Pi_{\rho, g, r}$ and  $\Pi_{\rho^j, g_j, r_j}$'s  ($j=1, \cdots, J$) are verified to be   generically \'{e}tale.
In other words, 
 there exists 
a dense open substack $\mcI$ of $\overline{\mcM}_{g, r}$ (resp., $\overline{\mcM}_{g_j, r_j}$) such that
 the projection $\Pi_{\rho, g, r}^{-1}(\mcI) \rightarrow \mcI$ (resp.,  $\Pi_{\rho^j, g_j, r_j}^{-1} (\mcI) \rightarrow \mcI$) is finite and   
 \'{e}tale.
In particular, it makes sense to discuss the generic degrees $\mr{deg} (\Pi_{\rho, g, r})$ and  $\mr{deg} (\Pi_{\rho^j, g_j, r_j})$  of 
$\Pi_{\rho, g, r}$ and  $\Pi_{\rho^j, g_j, r_j}$, respectively.

\SSP
\bt \label{Theorem44}
Let us keep the above assumption.
Then, the generic degrees $\mr{deg} (\Pi_{\rho, g, r})$,  $\mr{deg} (\Pi_{\rho^j, g_j, r_j})$  satisfy the equality
\begin{align} \label{Eqqei4}
\mr{deg} (\Pi_{\rho, g, r}) = \sum_{\rho_\mbG := \{ \rho^j \}_{j=1}^J}\prod_{j=1}^J \mr{deg}(\Pi_{\rho^j, g_j, r_j}),
\end{align}
  where the sum in the right-hand side runs over the  sets of $\mbG$-$\Xi_{n, \N}$-radii $\rho_\mbG$  with $\rho_{\mbG \Rightarrow \emptyset} = \rho$.
\et
\begin{proof}
Let us consider the square diagram \eqref{1050}.
This is Cartesian by Theorem \ref{y0176}, and the image of its  lower horizontal arrow $\mr{Clut}_\mbG$ contains a point classifying a totally degenerate curve, which we denote by $s$.
Hence, by 
our  assumption, 
the left-hand vertical arrow  turns out to be  generically \'{e}tale.
The right-hand side of the desired equality  \eqref{Eqqei4} coincides with  the generic degree of the left-hand vertical arrow in   the  diagram \eqref{1050}.
Our assumption implies that the both sides of \eqref{Eqqei4} are computed by restricting the vertical arrows in that diagram  over 
an \'{e}tale neighborhood of $s$.
Hence,  the assertion  follows from that   \eqref{1050} is Cartesian.
\end{proof}

\SSP
\begin{exa}[Cutting edges]
Let us describe the cartesian diagram \eqref{1050}, as well as the equality \eqref{Eqqei4}, in two particular cases corresponding to clutching morphisms, in the classical sense, between moduli spaces of pointed stable curves.

\begin{itemize}
\item[(i)]
First, let $g_1$, $g_2$, $r_1$, and $r_2$ be nonnegative integers  with $2g_i -1 + r_i > 0$ ($i=1,2$) and $g = g_1 + g_2$, $r = r_1 + r_2$.
These integers associate the gluing morphism 
\begin{align}
\Phi_\mr{tree} : \overline{\mcM}_{g_1, r_1 +1} \times \overline{\mcM}_{g_2, r_2 +1} \migi \overline{\mcM}_{g, r}
\end{align}
 obtained by attaching the respective last marked points of curves classified by $\overline{\mcM}_{g_1, r_1 +1}$ and $\overline{\mcM}_{g_2, r_2 +1}$ to form a node.

For $\rho_1\in \Xi_{n, \N}^{\times r_1}$, $\rho_2\in \Xi_{n, \N}^{\times r_2}$,   and $\rho_0 \in \Xi_{n, \N}$,
there exists a morphism
\begin{align}
{^p}\Phi_{\mr{tree}, \rho_0} : \mcO p^{^\mr{Zzz...}}_{(\rho_1, \rho_0), g_1, r_1+1} \times \mcO p^{^\mr{Zzz...}}_{(\rho_2, \rho^\veebar_0), g_2, r_2 +1} \migi \mcO p^{^\mr{Zzz...}}_{(\rho_1, \rho_2), g,r}
\end{align}
obtained by gluing together two dormant $\mr{PGL}_n^{(\N)}$-opers along the fibers over the respective last marked points of the underlying curves.
These morphisms for various $\rho_0$'s make the following square diagram commute:
\begin{align} \label{Eq4s1}
\vcenter{\xymatrix@C=46pt@R=36pt{
\coprod_{\rho_0 \in \Xi_{n, \N}}\mcO p^{^\mr{Zzz...}}_{(\rho_1, \rho_0), g_1, r_1 +1} \times \mcO p^{^\mr{Zzz...}}_{(\rho_2, \rho^\veebar_0), g_2, r_2 +1} \ar[r]^-{\coprod_{\rho_0}{^p}\Phi_{\mr{tree}}, \rho_0} \ar[d]_-{\coprod_{\rho_0}\Pi_{(\rho_1, \rho_0), g_1, r_1+1} \times \Pi_{(\rho_2, \rho^\veebar_0), g_2, r_2 +1}} & \mcO p^{^\mr{Zzz...}}_{(\rho_1, \rho_2), g, r} \ar[d]^-{\Pi_{(\rho_1, \rho_2), g, r}} \\
\overline{\mcM}_{g_1, r_1 +1} \times \overline{\mcM}_{g_2, r_2 +1} \ar[r]_-{\Phi_{\mr{tree}}} & \overline{\mcM}_{g, r}.
}}
\end{align}
This diagram coincides with  \eqref{1050} in the case where $\mbG$ is taken to be the clutching data ``$\msG_\mr{tree}$" defined in ~\cite[Eq.\,(908)]{Wak8}, i.e.,  the clutching data whose underlying semi-graph is visualized  as in Figure 1 below.

Under the assumption in Theorem \ref{Theorem44},
the equality \eqref{Eqqei4} in our situation here  reads 
\begin{align} \label{Eq101}
\mr{deg}(\Pi_{(\rho_1, \rho_2), g, r}) = \sum_{\rho_0 \in \Xi_{n, \N}} \mr{deg}(\Pi_{(\rho_1, \rho_0), g_1, r_1+1}) \cdot \mr{deg}(\Pi_{(\rho_2, \rho_0^\veebar), g_2, r_2 +1}).
\end{align}

\item[(ii)]
Next, given nonnegative integers $g$, $r$ with $2g + r > 0$, we shall write
\begin{align} \label{eq1}
\Phi_\mr{loop} : \overline{\mcM}_{g, r+2} \migi \overline{\mcM}_{g+1, r}
\end{align}
 for the gluing morphism obtained by attaching the last two marked points of each curve classified by $\overline{\mcM}_{g, r+2}$ to form a node.

For $\rho \in \Xi_{n, \N}^{\times r}$ and $\rho_0 \in \Xi_{n, \N}$,
 there exists a morphism
 \begin{align} \label{eq2}
 {^p}\Phi_{\mr{loop}, \rho_0} : \mcO p^{^\mr{Zzz...}}_{(\rho, \rho_0, \rho_0^\veebar), g, r+2} \migi \mcO p^{^\mr{Zzz...}}_{\rho, g+1, r}
 \end{align}
 obtained by gluing each dormant $\mr{PGL}_n^{(\N)}$-oper along the fibers over the last two marked points of the underlying curve.
 These morphisms for various $\rho_0$'s make the following square diagram commute:
 \begin{align} \label{Eq41}
\vcenter{\xymatrix@C=46pt@R=36pt{
\coprod_{\rho_0 \in \Xi_{n, \N}}\mcO p^{^\mr{Zzz...}}_{(\rho, \rho_0, \rho_0^\veebar), g, r+2}\ar[r]^-{\coprod_{\rho_0}{^p}\Phi_{\mr{loop}, \rho_0}} \ar[d]_-{\coprod_{\rho_0}\Pi_{(\rho, \rho_0, \rho_0^\veebar), g, r+2}} & \mcO p^{^\mr{Zzz...}}_{\rho, g+1, r} \ar[d]^-{\Pi_{\rho, g+1, r}} \\
\overline{\mcM}_{g, r +2} \ar[r]_-{\Phi_{\mr{loop}}} & \overline{\mcM}_{g+1, r}.
}}
\end{align}
This diagram coincides with \eqref{1050} in the case where $\mbG$ is taken to be the clutching data ``$\msG_{\mr{loop}}$" defined in ~\cite[Eq.\,(912)]{Wak8}, i.e., the clutching data whose underlying semi-graph is  visualized as in Figure 2 below.

Under the assumption  in Theorem \ref{Theorem44}, 
the equality \eqref{Eqqei4} in our situation here reads 
\begin{align} \label{Eq100}
\mr{deg}(\Pi_{(\rho, g+1, r)}) = \sum_{\rho_0 \in \Xi_{n, \N}} \mr{deg}(\Pi_{(\rho, \rho_0, \rho_0^\veebar), g, r+2}).
\end{align}
\end{itemize}
\hspace{20mm} 
 \includegraphics[width=14cm,bb=0 0 1212 350,clip]{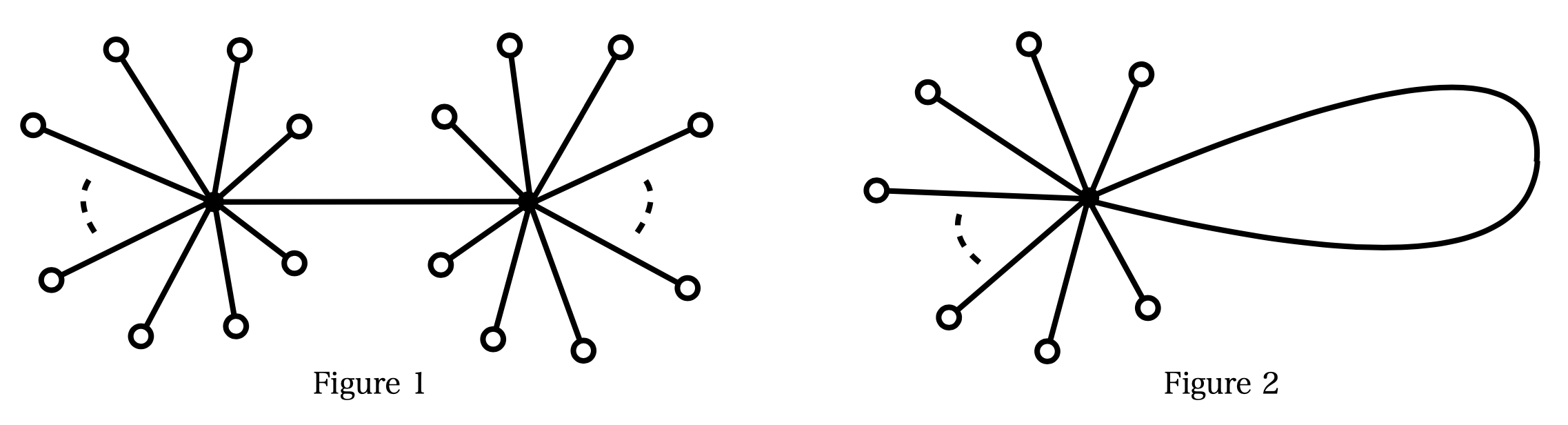}

(Here, ``$\circ\hspace{-1.5mm}-\hspace{-1.5mm}-\hspace{-1.5mm}\bullet$'' represents an open edge and 
``$\bullet\hspace{-1.5mm}-\hspace{-1.5mm}-\hspace{-1.5mm}\bullet$'' represents a closed edge.)
\end{exa}

\vspace{10mm}
\section{$2$d TQFT for dormant $\mr{PGL}_n^{(\N)}$-opers} \label{S18}\SSP

In this section, we prove some assertions concerning  the moduli spaces $\mcO p^{^\mr{Zzz...}}_{\rho, g, r}$, as well as their generic degrees $\mr{deg}(\mr{\Pi}_{\rho, g, r})$.
These facts together with   the factorization property resulting from Theorem \ref{Theorem44}  are collectively explained by the notion of a  $2$d TQFT or  a fusion rule (cf. Theorem \ref{Theorem4f4}).
It gives  an effective way of computing the values $\mr{deg}(\Pi_{\rho, g, r})$ by means of the ring-theoretic structure of the corresponding Frobenius algebra (cf. Proposition  \ref{Prop16}).

Let $n$ be an integer with $1 < n < p$, $\N$ a positive integer, $(g, r)$ a pair of nonnegative integers with $2g-2+r > 0$.
We shall set $\overline{\mcM}_{g, r} := \overline{\mcM}_{g, r, \mbF_p}$,
$\mcO p^{^\mr{Zzz...}}_{\rho, g, r} := \mcO p^{^\mr{Zzz...}}_{n, \N, \rho, g, r, \mbF_p}$, and $\Pi_{\rho, g, r} := \Pi_{n, \N, \rho, g, r, \mbF_p}$ 
for each $\rho \in \Xi_{n, \N}^{\times r}$.

\LSP
\subsection{Forgetting tails} \label{SS060}

We denote by 
\begin{align} \label{eeQwe1}
\Phi_{\mr{tail}} : \overline{\mcM}_{g, r+1} \migi \overline{\mcM}_{g,r}
\end{align}
 the morphism obtained by forgetting the last marked point and successively contracting any resulting unstable components of each curve classified by $\overline{\mcM}_{g, r+1}$. 
 In what follows, we consider the behavior of the moduli stack of dormant $\mr{PGL}_n^{(\N)}$-opers according to pull-back along $\Phi_{\mr{tail}}$.

Let $S$ be a scheme over $\mbF_p$
  and 
 $\msX := (X/S, \{ \sigma_{i} \}_{i=1}^{r+1})$ an $(r+1)$-pointed {\it smooth} curve over $S$ of genus $g$ (hence $S^\mr{log} = S$).
We shall write 
\begin{align} \label{eeQwe13}
\msX_{\mr{tail}} := (X_{\mr{tail}}, \{ \sigma_{\mr{tail}, i} \}_{i=1}^r)
\end{align}
for the $r$-pointed curve obtained from $\msX$ by forgetting the last marked point.
  That is to say, $X_{\mr{tail}} = X$ and $\sigma_{\mr{tail}, i} = \sigma_{i}$ for every $i \in \{1, \cdots, r \}$.

Next, let $\vartheta := (\varTheta, \nabla_\vartheta)$ be a dormant $n^{(\N)}$-theta characteristic of $X^\mr{log}_{\mr{tail}}/S$,
which consists of a line bundle $\varTheta$ on $X$ and a $\mcD^{(\N -1)}_{X^\mr{log}_\mr{tail}/S}$-module structure $\nabla_\vartheta$ on $\mcT^{\otimes \frac{n(n-1)}{2}}_{X^\mr{log}_\mr{tail}/S}\otimes \varTheta^{\otimes n}$ with vanishing $p^\N$-curvature.
Note that there always exists such a pair because of the assumption $p \nmid n$ (cf. Proposition \ref{P204}).

By letting $\varTheta_+ := \varTheta ((n-1) \sigma_{r+1})$,
 we have a sequence of natural isomorphisms
\begin{align} \label{Eq29}
\mcT_{X^\mr{log}/S}^{\otimes \frac{n (n-1)}{2}} \otimes \varTheta_+^{\otimes n} &\isom \mcT_{X^\mr{log}_\mr{tail}/S} (-\sigma_{r+1})^{\otimes \frac{n(n-1)}{2}} \otimes \varTheta ((n-1)\sigma_{r+1})^{\otimes n} \\
&\isom \mcT_{X^\mr{log}_\mr{tail}/S}^{\otimes \frac{n(n-1)}{2}} \otimes \varTheta^{\otimes n} \left(\frac{n(n-1)}{2}\sigma_{r+1}\right). \notag
\end{align}
The $\mcD_{X^\mr{log}/S}^{(\N-1)}$-module structure on $\mcT^{\otimes \frac{n(n-1)}{2}}_{X^\mr{log}_\mr{tail}/S}\otimes \varTheta^{\otimes n}$  associated to $\nabla_\vartheta$ extends, via \eqref{Eq29},  to a  unique $\mcD^{(N-1)}_{X^\mr{log}/S}$-module structure $\nabla_{\vartheta, +}$ on $\mcT_{X^\mr{log}/S}^{\otimes \frac{n(n-1)}{2}} \otimes \varTheta_+^{\otimes n}$.
The resulting pair
\begin{align}
\vartheta_+ := (\varTheta_+, \nabla_{\vartheta, +})
\end{align}
 forms a dormant $n^{(\N)}$-theta characteristic of $X^\mr{log}/S$.

In our discussion here, we shall set
\begin{align}
\mcF_\varTheta := \mcD_{X^\mr{log}_\mr{tail}/S, \leq n-1}^{(\N -1)} \otimes \varTheta  \ \ \  \text{and} \  \ \ 
\mcF_\varTheta^j := \mcD_{X^\mr{log}_\mr{tail}/S, \leq n-j- 1}^{(\N -1)}\otimes \varTheta \hspace{12mm}\\ 
\left(\text{resp.,} \ \mcF_{\varTheta, +} := \mcD_{X^\mr{log}/S, \leq n-1}^{(\N -1)} \otimes \varTheta_+   \  \ \ \text{and} \  \ \ 
\mcF_{\varTheta, +}^j := \mcD_{X^\mr{log}/S, \leq n-j- 1}^{(\N -1)}\otimes \varTheta_+   \right)
\end{align}
for  $j=0, \cdots, n$.
In particular, $\{ \mcF_\varTheta^j \}_j$ and $\{ \mcF_{\varTheta, +}^j \}_j$ define decreasing filtrations of $\mcF_{\varTheta}$ and $\mcF_{\varTheta, +}$, respectively, such that
\begin{align} \label{eeQQ113}
\mcF^j_\varTheta / \mcF^{j+1}_\varTheta \cong \mcT_{X^\mr{log}_\mr{tail}/S}^{\otimes (n-1-j)} \otimes \varTheta, 
\hspace{10mm}
\mcF^j_{\varTheta, +}/\mcF^{j+1}_{\varTheta, +} \cong  \mcT_{X^\mr{log}/S}^{\otimes (n-1-j)} \otimes \varTheta_+.  
\end{align} 
The tensor product of the natural  inclusion
$\mcD_{X^\mr{log}_{\mr{tail}}/S, \leq n-1}^{(\N -1)} \migiincl \mcD_{X^\mr{log}/S, \leq n-1}^{(\N -1)}((n-1) \sigma_{r+1})$ and the identity morphism of $\varTheta$ forms an inclusion $\mcF_\varTheta \migiincl \mcF_{\varTheta, +}$, which  preserves  the filtration.

 \SSP
\bpr \label{Prop15}
Let $\nabla^\diamondsuit$ be a normal $(\mr{GL}_n^{(\N)}, \vartheta)$-oper on $\msX_\mr{tail}$.
Then, there exists a unique $(\mr{GL}_n^{(\N)}, \vartheta_+)$-oper $\nabla^\diamondsuit_+$ on $\msX$ extending $\nabla^\diamondsuit$.
If, moreover, $\nabla^\diamondsuit$ is dormant, then
$\nabla_+^\diamondsuit$ is dormant and 
its  exponent   at $\sigma_{r+1}$ coincides with 
$[\overline{0}, \overline{1}, \cdots, \overline{n-1}]$, where $\overline{m}$ (for each integer $m$) denotes the image of $m$ via $\mbZ \migisurj \mbZ/p^\N \mbZ$.
\epr
\begin{proof}
We shall write  $U := X \setminus \mr{Im}(\sigma_{r+1})$ and write $u$ for the open immersion $U \migiincl X$.
The natural morphism $\mcF_\varTheta \rightarrow u_* (\mcF_\varTheta |_U) \left(=  u_* (\mcF_{\varTheta, +} |_U) \right)$ is injective, and  
 $\nabla^\diamondsuit$ extends, via this injection,  to a $\mcD_{X^\mr{log}/S}^{(\N -1)}$-module structure $\breve{\nabla}^\diamondsuit$   on $u_* (\mcF_\varTheta |_U)$.

In what follows, we prove the claim that $\mcF_{\varTheta, +} \left(\subseteq  u_* (\mcF_\varTheta |_U) \right)$ is  closed under  $\breve{\nabla}^\diamondsuit$.
Choose a local function $t$ defining $\mr{Im}(\sigma_{r +1})$
and a local generator $v$ of the line bundle $\varTheta$ around $\mr{Im} (\sigma_{r +1})$ (hence $\varTheta$ is locally identified with $\mcO_X v$).
The   local basis  $\{ \partial^{\langle j \rangle}_{\mr{non}\text{-}\mr{log}} \}_{j \in \mbZ_{\geq 0}}$  of $\mcD_{X_{\mr{tail}}^\mr{log}/S}^{(\N -1)}$ associated to  $t$ (in the manner of  ~\cite[Section 1.2.3]{PBer1}) gives
locally defined decompositions 
\begin{align} \label{eeQQ1227}
\mcD_{X_{\mr{tail}}^\mr{log}/S}^{(\N -1)} =  \bigoplus_{j \in \mbZ_{\geq 0}}\mcO_X  \cdot \partial^{\langle  j \rangle}_{\mr{non}\text{-}\mr{log}}
 \ \ \ \text{and} \ \ \ 
\mcF_{\varTheta} = \bigoplus_{j= 0}^{n-1} \mcO_X  \cdot \partial^{\langle j \rangle}_{\mr{non}\text{-}\mr{log}} \otimes v.
\end{align} 
These decompositions extend to 
\begin{align} \label{eeQQ120}
\mcD_{X^\mr{log}/S}^{(\N -1)} =  \bigoplus_{j \in \mbZ_{\geq 0}}\mcO_X  \cdot t^{j} \cdot \partial^{\langle  j \rangle}_{\mr{non}\text{-}\mr{log}}
 \ \ \ \text{and} \ \ \ 
\mcF_{\varTheta, +} = \bigoplus_{j= 0}^{n-1} \mcO_X  \cdot t^{j-n+1} \cdot \partial^{\langle j \rangle}_{\mr{non}\text{-}\mr{log}} \otimes v,
\end{align}
respectively. 
Let 
$j$ and $j'$ be nonnegative integers with $j' \leq n-1$,
and regard 
$t^{j} \cdot \partial^{\langle  j\rangle}_{\mr{non}\text{-}\mr{log}}$ and 
$t^{j' - n+1} \cdot  \partial^{\langle  j' \rangle}_{\mr{non}\text{-}\mr{log}} \otimes v$
as local sections of $\mcD_{X^\mr{log}/S}^{(\N -1)}$ and $\mcF_{\varTheta, +}$, respectively, via \eqref{eeQQ120}.
Applying  the discussion following ~\cite[Definition 2.5]{GLQ} to the case where $t$ is invertible in the underlying ring,
we have $\partial^{\langle  \ell \rangle}_{\mr{non}\text{-}\mr{log}} (t^{-m}) = q_\ell ! \cdot \binom{-m}{\ell} \cdot  t^{-m - \ell}$ for $\ell, m \geq 0$. 
Then, it follows from \eqref{EE665} that 
\begin{align} \label{EE640}
& \ \ \ \ \breve{\nabla}^\diamondsuit (t^j \cdot \partial^{\langle  j  \rangle}_{\mr{non}\text{-}\mr{log}}) (t^{j' - n+1} \cdot  \partial^{\langle  j' \rangle}_{\mr{non}\text{-}\mr{log}} \otimes v) \\
& = t^j \cdot \sum_{\ell =0}^j   \frac{q_j!}{q_\ell ! \cdot q_{j-\ell}!} \cdot 
 \partial^{\langle  \ell  \rangle}_{\mr{non}\text{-}\mr{log}} (t^{j'-n+1}) \cdot
\breve{\nabla}^\diamondsuit  (\partial^{\langle  j -\ell  \rangle}_{\mr{non}\text{-}\mr{log}}) (\partial^{\langle  j'  \rangle}_{\mr{non}\text{-}\mr{log}} \otimes v) \\
& = \sum_{h=0}^j \frac{q_j!}{q_h !} \cdot \binom{j'-n+1}{j-h} \cdot  t^{h+ j' -n+1}  \cdot \breve{\nabla}^\diamondsuit (\partial_{\mr{non}\text{-}\mr{log}}^{\langle h\rangle}) ( \partial^{\langle  j' \rangle}_{\mr{non}\text{-}\mr{log}} \otimes v),
\end{align}
where $h := j - \ell$, and the coefficient ``$q_j!/q_{h}!$" in the rightmost is computed as an integer and then taken to be its mod $p$ reduction.
In the case  $h \geq n-1-j'$,
we have $t^{h+j'-n+1} \in \mcO_X$ and $\breve{\nabla}^\diamondsuit (\partial^{\langle  h  \rangle}_{\mr{non}\text{-}\mr{log}}) (\partial^{\langle  j' \rangle}_{\mr{non}\text{-}\mr{log}}\otimes v) \in \mcF_\varTheta$.
Therefore, $t^{h+j'-n+1} \cdot \breve{\nabla}^\diamondsuit (\partial^{\langle  h  \rangle}_{\mr{non}\text{-}\mr{log}}) (\partial^{\langle  j' \rangle}_{\mr{non}\text{-}\mr{log}}\otimes v) \in \mcF_\varTheta \subseteq \mcF_{\varTheta, +}$.
On the other hand, let us consider the case $0 \leq h < n-1-j'$.
Since $h < n-1-j' \leq n-1 <p$,  $\partial^{\langle  h  \rangle}_{\mr{non}\text{-}\mr{log}} =  c \cdot (\partial^{\langle  1  \rangle}_{\mr{non}\text{-}\mr{log}})^h$ for some $c \in \mbF_p^\times$ (cf. ~\cite[Proposition 2.3.2, (a)]{Mon}).
By the Griffiths transversality in the definition of an oper, 
the section  $\breve{\nabla}^\diamondsuit (\partial^{\langle  h  \rangle}_{\mr{non}\text{-}\mr{log}}) (\partial^{\langle  j' \rangle}_{\mr{non}\text{-}\mr{log}} \otimes v)$ lies in $\mcF_\varTheta^{n-1-j'-h}$,  and it is  described as $\sum_{s=0}^{j' +h} a_s \cdot (\partial^{\langle  s\rangle}_{\mr{non}\text{-}\mr{log}}\otimes v)$.
Observe that
\begin{align} \label{EE701}
t^{h+j'-n+1} \cdot \breve{\nabla}^\diamondsuit (\partial^{\langle  h  \rangle}_{\mr{non}\text{-}\mr{log}}) (\partial^{\langle  j' \rangle}_{\mr{non}\text{-}\mr{log}} \otimes v) &= \sum_{s=0}^{j' +h} a_s \cdot t^{h+j'-n+1}(\partial^{\langle  s\rangle}_{\mr{non}\text{-}\mr{log}}\otimes v) \\
& = \sum_{s=0}^{j' +h} a_s \cdot t^{j'+h-s} \left(t^{s-n+1}(\partial^{\langle  s\rangle}_{\mr{non}\text{-}\mr{log}}\otimes v)\right).
\end{align}
Since $s \leq j' + h$ ($s = 0, \cdots, j' +h$), i.e., $t^{j' + h-s} \in \mcO_X$, 
the leftmost of \eqref{EE701} lies in $\mcF_{\varTheta, +}$.
Thus,  the leftmost  of  \eqref{EE640} lies in $\mcF_{\varTheta, +}$.
The $\mcO_X$-module $\mcD_{X^\mr{log}/S^\mr{log}}^{(\N -1)}$ is locally generated by the section $t^j \cdot \partial^{\langle  j\rangle}_{\mr{non}\text{-}\mr{log}}$'s ($j \in \mbZ_{\geq 0}$),
so this completes the proof of the claim.

Now, denote by $\nabla_+^\diamondsuit$ the $\mcD_{X^\mr{log}/S}^{(\N -1)}$-module structure on $\mcF_{\varTheta, +}$ obtained by restricting $\breve{\nabla}^\diamondsuit$.
By the second ``$\cong$" in \eqref{eeQQ113},
$\nabla_+^\diamondsuit$ is verified to form a $(\mr{GL}_n^{(\N)}, \vartheta_+)$-oper, extending $\nabla^\diamondsuit$.
Since the uniqueness portion follows from the equality $\mcF_\varTheta |_U = \mcF_{\varTheta, +} |_U$,  
the proof of the first assertion is completed.

Next, to prove the second assertion,
we
suppose further  that $\nabla^\diamondsuit$ is dormant.
Since $\breve{\nabla}^\diamondsuit$ has vanishing $p^\N$-curvature, 
its restriction $\nabla_+^\diamondsuit$ turns out to be dormant.
To complete the proof, we may assume, after possibly taking the geometric fibers of points in $S$, that $S = \mr{Spec}(k)$ for an algebraically closed field $k$ over $\mbF_p$.
The formal neighborhood  of  $\mr{Im}(\sigma_{r+1})$ in $X$ may be identified with $U_\oslash := \mr{Spec} (k [\![t]\!])$
(cf.  \eqref{YY124})
 by using the function $t$.
 It follows from Proposition-Definition  \ref{P022}  that
there exists an isomorphism
\begin{align} \label{Eq201}
(\mcF_{\varTheta, +}, \nabla^\diamondsuit_+)|_{U_\oslash} \isom  \bigoplus_{i=1}^n (\mcO_\oslash, \nabla_{\oslash, d_i}^{(\N -1)})
\end{align}
for some $d_1, \cdots, d_n \in \mbZ/p^\N \mbZ$,
where $\nabla_{\oslash, d}^{(\N -1)}$ for each $d \in \mbZ$ denotes the  $\mcD_{\oslash}^{(\N -1)}$-module structure on $\mcO_{\oslash} := \mcO_{U_\oslash}$ defined in \eqref{YY52}.
Here, for each $\mcD_{\oslash}^{(\N -1)}$-module $(\mcF, \nabla)$, we shall write $\mr{Res}(\mcF, \nabla)$ for the cokernel of the natural morphism $F^{(\N)*}_{U_\oslash /k}(\mcS ol (\nabla)) \migi \mcF$.
 The assignment $(\mcF, \nabla) \mapsto \mr{Res}(\mcF, \nabla)$ is functorial, so
the isomorphism \eqref{Eq201} yields an isomorphism of $\mcO_\oslash$-modules
\begin{align} \label{Eq203}
\mr{Res}((\mcF_{\varTheta, +}, \nabla^\diamondsuit_+)|_{U_\oslash}) \isom \bigoplus_{i=1}^n \mr{Res}(\mcO_\oslash, \nabla_{\oslash, d_i}^{(\N -1)}).
\end{align}
The $\mcD_{\oslash}^{(\N -1)}$-module  $(\mcF_{\varTheta, +}, \nabla^\diamondsuit_+)|_{U_\oslash}$ restricts to
$(\mcF_\varTheta, \nabla^\diamondsuit) |_{U_\oslash}$, which is isomorphic to the direct sum of $n$ copies of $(\mcO_\oslash, \nabla^{(\N -1)}_{\oslash, 0})$.
It follows that
 $\mr{Res}((\mcF_{\varTheta, +}, \nabla^\diamondsuit_+)|_{U_\oslash}) $ is isomorphic to $\mcF_{\varTheta, +}/\mcF_\varTheta$.
On the other hand, by Proposition \ref{L093},
$\mr{Res}(\mcO_\oslash, \nabla_{\oslash, d_i}^{(\N -1)})$ is isomorphic to $\mcO_\oslash/(t^{\widetilde{d}_i})$.
Thus,  \eqref{Eq203} 
implies
\begin{align}
\bigoplus_{i=1}^n \mcO_\oslash/(t^{\widetilde{d}_i})
& \cong 
\mcF_{\varTheta, +}/\mcF_{\varTheta} \\
& \cong \bigoplus_{j=0}^{n-1} 
 (\mcO_X \cdot t^{j-n+1} \cdot \partial^{\langle j \rangle}_{\mr{non}\text{-}\mr{log}}  \otimes v)
 / (\mcO_X \cdot \partial^{\langle j \rangle}_{\mr{non}\text{-}\mr{log}}  \otimes v)
\notag \\
&\cong  \bigoplus_{j=0}^{n-1} \mcO_\oslash / (t^{n-1-j}).   \notag
\end{align}
This yields the equality of sets
 $\{ d_1, \cdots, d_n \} = \{ \overline{0}, \cdots, \overline{n-1} \}$, meaning  that 
 the exponent of $\nabla_+^\diamondsuit$ at $\sigma_{r+1}$ coincides with $[\overline{0}, \overline{1}, \cdots, \overline{n-1}]$.
This completes the proof of the second assertion.
\end{proof}
\SSP

\bpr \label{Prop2234}
Let $\nabla^\diamondsuit$ be a normal dormant $(\mr{GL}_n^{(\N)}, \vartheta_+)$-oper  on $\msX$ whose exponent at $\sigma_{r+1}$ coincides with $[\overline{0}, \overline{1}, \cdots, \overline{n-1}]$.
Then, there exists a unique  dormant $(\mr{GL}_n^{(\N)}, \vartheta)$-oper $\nabla^\diamondsuit_\mr{tail}$ on $\msX_{\mr{tail}}$
 such that the $(\mr{GL}_n^{(\N)}, \vartheta_+)$-oper  $(\nabla^\diamondsuit_\mr{tail})_+$ associated to $\nabla^\diamondsuit_\mr{tail}$ (by applying the construction of Proposition \ref{Prop15}) coincides with $\nabla^\diamondsuit$.
\epr
\begin{proof}
Let $\nabla_\mr{can}$ denote  the canonical $\mcD_{X^\mr{log}/S}^{(\N -1)}$-module structure on  $F_{X/S}^{(\N)*}(\mcS ol (\nabla^\diamondsuit))$ in the sense of   Definition \ref{dGGe11}, i.e., $\nabla_\mr{can} := \nabla^{(\N -1)}_{\mcS ol (\nabla^\diamondsuit), \mr{can}}$.
The inclusion $\mcS ol (\nabla^\diamondsuit) \migiincl \mcF_{\varTheta, +}$ induces 
 a  morphism of $\mcD_{X^\mr{log}/S}^{(\N -1)}$-modules
 \begin{align}
 \eta : 
 (F^{(\N)*}_{X/S}(\mcS ol (\nabla^\diamondsuit)), \nabla_{\mr{can}}) \migi (\mcF_{\varTheta, +}, \nabla^{\diamondsuit}).
 \end{align}
 This morphism becomes an isomorphism when restricted to $X \setminus \bigcup_{i=1}^{r+1} \mr{Im}(\sigma_i)$.
 Hence, one can glue together $(\mcF_{\varTheta, +}, \nabla^\diamondsuit) |_{X \setminus \mr{Im}(\sigma_{r +1})}$ and  
 $ (F^{(\N)*}_{X/S}(\mcS ol (\nabla^\diamondsuit)), \nabla_{\mr{can}})|_{X \setminus \bigcup_{i=1}^{r} \mr{Im}(\sigma_i)}$ by using the isomorphism  $\eta |_{X \setminus \bigcup_{i=1}^{r+1} \mr{Im}(\sigma_i)}$;
 the resulting  $\mcD_{X^\mr{log}/S}^{(\N -1)}$-module will be denoted by  $(\mcF, \nabla^\diamondsuit_\mr{tail})$.
 Since   $ (F^{(\N)*}_{X/S}(\mcS ol (\nabla^\diamondsuit)), \nabla_{\mr{can}})$ comes from  a (non-logarithmic) $\mcD^{(\N -1)}_{X/S}$-module,
 $(\mcF, \nabla^\diamondsuit_\mr{tail})$ specifies a $\mcD_{X^\mr{log}_\mr{tail}/S}^{(\N -1)}$-module.

 In what follows, we prove the equality $\mcF = \mcF_{\varTheta}$ of $\mcO_{X}$-submodules of $\mcF_{\varTheta, +}$.
To this end, we may assume, after possibly changing the base of  $\msX$,
that $S = \mr{Spec} (R)$ for a local ring $R$ over $\mbF_p$ such that the residue field is algebraically closed.
Let us fix a local function $t$  defining $\mr{Im}(\sigma_{r +1})$, which gives an identification of the formal neighborhood of $\mr{Im}(\sigma_{r+1}) \subseteq X$ with $U_\oslash$ (cf. \eqref{YY124}).
The assumption on $\nabla^\diamondsuit$ implies the existence of 
 an isomorphism between  $\mcD_{\oslash}^{(\N -1)}$-modules
\begin{align} \label{Eq205}
(\mcF_{\varTheta, +}, \nabla^\diamondsuit) |_{U_\oslash} \isom \bigoplus_{i=0}^{n-1} (\mcO_\oslash, \nabla_{\oslash, \overline{i}}^{(\N -1)}).
\end{align}
According to  Proposition \ref{NN77}, (ii), 
there exists a global section of the form $(v_1,  \cdots, v_n) \in H^0 (U_\oslash, (\mcO_\oslash^\times)^{\oplus n})$ which 
 formally generates its codomain  (as a $\mcD_{\oslash}^{(\N -1)}$-module).
The  $\mcO_\oslash$-module $\mcF_{\varTheta} |_{U_\oslash}$ is verified to be formally generated  by the sections corresponding, via \eqref{Eq205}, to the sections
 \begin{align}
& \ \ \ \  ((\nabla_{\oslash, \overline{0}}^{(\N -1)} ( \partial_{\mr{non}\text{-}\mr{log}}^{\langle 1 \rangle}))^j (t^{n-1}) \cdot v_1, \cdots, (\nabla_{\oslash, \overline{n -1}}^{(\N -1)} (\partial_{\mr{non}\text{-}\mr{log}}^{\langle 1 \rangle}))^j (t^{n-1})\cdot v_n) \\
 & = (j! \cdot \binom{n-1 - 0}{j}t^{n-1-j}\cdot v_1, \cdots, j! \cdot \binom{n- 1 - (n-1)}{j}t^{n-1-j}\cdot v_n) \in H^0 (U_\oslash, \mcO_\oslash^{\oplus n})
 \end{align}
 for $j=0, \cdots, n-1$, where 
 $\partial_{\mr{non}\text{-}\mr{log}}^{\langle 1 \rangle}$ is as in the proof of Proposition \ref{Prop15}, and 
 $\binom{m}{0} := 1$ for any $m \in \mbZ$.
 By  this fact together with  Proposition \ref{L093}, 
 $\mcF_\varTheta |_{U_\oslash}$
  coincides with 
  $F^{(\N)*}_{U_\oslash/S}(\mcS ol (\nabla^\diamondsuit |_{U_\oslash}))$.
 This proves  the desired equality $\mcF = \mcF_\varTheta$.
  
 The collection 
  $(\mcF_\varTheta, \nabla^\diamondsuit_\mr{tail}, \{ \mcF_\varTheta^j \}_{j=0}^n)$
  forms a $\mr{GL}_n^{(\N)}$-oper on $\msX_{\mr{tail}}$ because of the above discussion and the fact that its restriction 
   to $X \setminus \mr{Im}(\sigma_{r+1})$ defines a $\mr{GL}_n^{(\N)}$-oper.
   Moreover, it follows from the various definitions involved that
   $\nabla^\diamondsuit_\mr{tail}$  specifies a dormant $(\mr{GL}_n^{(\N)}, \vartheta)$-oper satisfying the required conditions. 
  The proof of the assertion is completed.
\end{proof}
\SSP

 Let $\msE^\spadesuit$ be a dormant  $\mr{PGL}_n^{(\N)}$-oper on $\msX_\mr{tail}$.
 This corresponds, via  the isomorphism $\Lambda^{^\mr{Zzz...}}_{\diamondsuit \Rightarrow \spadesuit, \vartheta}$ (cf. Theorem \ref{P14}), to
  a dormant $(\mr{GL}_n^{(\N)}, \vartheta)$-oper $\nabla^\diamondsuit$.
 We shall denote by $\msE^\spadesuit_+$ the dormant $\mr{PGL}_n^{(\N)}$-oper on $\msX$ corresponding to the $(\mr{GL}_n^{(\N)}, \vartheta_+)$-oper $\nabla_+^\diamondsuit$ constructed from $\nabla^\diamondsuit$ in the  manner of Proposition \ref{Prop15}.
 Note that the isomorphism classes of $\msE^\spadesuit_+$ depend only on that of $\msE^\spadesuit$ (i.e., does not depend on the choice of $\vartheta$).

 Moreover, by Proposition \ref{Prop15} again, 
 the radius of $\msE^\spadesuit_+$ at $\sigma_{r+1}$ coincides with
 $\varepsilon$.
 Here,   we set
 \begin{align} \label{Eq20}
\varepsilon := \pi_\Delta ([\overline{0}, \cdots, \overline{n -1}]) \in \mfS_n \backslash (\mbZ/p^\N \mbZ)^{\times n}/\Delta
\end{align}
(cf. \eqref{ssak} for the definition of $\pi_\Delta$).
Note that this element for $n=2$ corresponds to  
``$\epsilon$" (cf. \eqref{YY334})  via the bijection \eqref{wpaid}.

The following assertion follows immediately  from Proposition \ref{Prop2234}.

 \SSP
\bpr \label{Prop12}
Let $\rho$ be  an element of $\Xi_{n, \N}^{\times r}$.
Then, the  assignment
\begin{align}
(\msX, (\msX_\mr{tail}, \msE^\spadesuit)) \mapsto (\msX, (\msX, \msE^\spadesuit_+))
\end{align}
constructed above is functorial with respect to base-change over $S$-schemes, and 
determines an isomorphism
\begin{align} \label{Eq211}
\mcM_{g, r+1} \times_{\Phi_\mr{tail}, \overline{\mcM}_{g, r}} \mcO p_{\rho, g,r}^{^\mr{Zzz...}} \isom \mcM_{g, r+1} \times_{\overline{\mcM}_{g, r+1}} \mcO p_{(\rho, \varepsilon), g, r+1}^{^\mr{Zzz...}}
\end{align}
of stacks over $\mcM_{g, r+1}$.
\epr
\SSP

Moreover, the following assertion is a direct consequence of  the above proposition.

\SSP
\bco \label{Cor11}
Let $(g, r)$ and $\rho$ be as in Proposition \ref{Prop12}.
Also, let us keep the assumption in Theorem \ref{Theorem44}.
Then, 
 the following equality holds:
\begin{align} \label{eeQQ200}
\mr{deg} 
(\Pi_{(\rho, \varepsilon), g, r+1}) = \mr{deg}(\Pi_{\rho, g, r}).
\end{align}
\eco

\LSP
\subsection{$\mr{PGL}_n^{(\N)}$-opers on a $2$-pointed projective line} \label{SS083}

Let $S$ be a scheme over $\mbF_p$, and let $\mbP := \mcP roj (\mcO_S [x, y])$ denote the projective line over $S$.
Denote by $\sigma_{1}$ and $\sigma_{2}$ the marked points of $\mbP$ determined by the values 
$0$ and $\infty$, respectively (i.e., $\sigma_1 := [0]$ and $\sigma_2 := [\infty]$ in the terminology of Section \ref{SS30d2}).
In particular, we have a $2$-pointed curve 
\begin{align}
\msP' := (\mbP, \{ \sigma_{1}, \sigma_{2}\})
\end{align}
 over $S$.
The $S$-scheme $\mbP$ has two open subschemes $U_1 := \mbP \setminus \mr{Im}(\sigma_{2}) = \mcS pec (\mcO_S [t_1])$ (where $t_1 := x/y$) and $U_2 := \mbP \setminus \mr{Im}(\sigma_{1}) = \mcS pec (\mcO_S [t_2])$ (where $t_2 := y/x$).

Just as in the hyperbolic case, we can equip $\mbP$ with a log structure induced from the relative divisor defined as the union of $\mr{Im}(\sigma_{1})$ and $\mr{Im}(\sigma_{2})$; the resulting log curve  will be denoted by $\mbP^{\mr{log}'}$.
Also, the definition of radius $\rho_i (-)$ at $\sigma_i$ ($i=1, 2$)  can be formulated as in Definition \ref{epaddd}.

\SSP
\bpr \label{Prop13}
\begin{itemize}
\item[(i)]
Let $\msE^\spadesuit$ be a dormant $\mr{PGL}_n^{(\N)}$-oper on $\msP'$.
Then, the equality $\rho_1 (\msE^\spadesuit) = \rho_2 (\msE^\spadesuit)^\veebar$ holds.
\item[(ii)]
For each $\rho \in \Xi_{n, \N}$, there exists a unique (up to isomorphism) dormant $\mr{PGL}_n^{(\N)}$-oper on $\msP'$ with $\rho = \rho_1 (\msE^\spadesuit) = \rho_2 (\msE^\spadesuit)^\veebar$.
\end{itemize}
\epr
\begin{proof}
First, we shall prove assertion (i).
For each $i=1, 2$, denote by $\{ \partial_i^{\langle j \rangle} \}_{j \in \mbZ_{\geq 0}}$
the (globally defined) basis of $\mcD_{\mbP^{\mr{log}'}/S}^{(\N -1)}$ associated to the function $t_i$.
In particular, we have $\mcD_{\mbP^{\mr{log}'}/S}^{(\N -1)} = \bigoplus_{j \in \mbZ_{\geq 0}} \mcO_{\mbP} \cdot \partial_i^{\langle j \rangle}$.
The morphism $\mu_{(-)}$ and the sheaves $\mcP^\ell_{(-)}$ ($\ell \in \mbZ_{\geq 0}$) discussed in Sections \ref{SS040}-\ref{SS04f4} can be defined even in our situation here.
In particular, we obtain a section $\eta_i := \mu_{(\N -1)} (t_i) -1$ of $\mcP_{(\N -1)}^\ell$.
Since $\mu_{(\N -1)}$ preserves the monoid structure, we have
\begin{align}
(1 + \eta_1)(1 + \eta_2) = \mu_{(\N -1)} (t_1) \cdot \mu_{(\N -1)} (t_2) = \mu_{(\N -1)}(t_1 \cdot t_2) = \mu_{(\N -1)}(1) = 1.
\end{align}
Just as in the proof of Lemma \ref{P238},
the resulting equality $(1 + \eta_1)(1 + \eta_2) =1$  induces
\begin{align}
\partial_2^{\langle p^a \rangle} = - \partial_1^{\langle p^a \rangle} - 1 + \prod_{b= 0}^{a-1}(1- (\partial_1^{\langle p^b \rangle})^{p-1}).
\end{align}
Then, the assertion follows from \eqref{dE10122} and Proposition \ref{L090}.

Next, we shall prove assertion (ii).
Denote by $\nabla_{\mr{triv}}^{(\N -1)}$ the trivial $\mcD^{(\N-1)}_{\mbP^{\mr{log}'}/S}$-module structure on $\mcO_{\mbP}$.
By  fixing  an identification  $\Omega_{\mbP^{\mr{log}'}/S} = \mcO_\mbP$,
we regard the pair $\vartheta := (\mcO_\mbP, \nabla_{\mr{triv}}^{(\N -1)})$ 
as a dormant $n^{(\N)}$-theta characteristic of $\mbP^{\mr{log}'}/S$.
 Since $n < p$, there exists   an $n$-tuple  of integers $(d_1, \cdots, d_n)$ such that
 $\pi_\Delta ([d_1, \cdots, d_n]) = \rho$, $p^\N \mid \sum_{j=1}^n d_j$, and $0 \leq d_j \leq p^\N-1$ for every $j$.
Given each $j =1, \cdots, n$, we denote by $\nabla_{j}$ the 
 $\mcD^{(\N-1)}_{\mbP^{\mr{log}'}/S}$-module structure  on $\mcO_\mbP (d_j [\sigma_{1}] -d_j [\sigma_{2}])$ whose restriction to $\mbP \setminus (\mr{Im}(\sigma_{1}) \cup \mr{Im}(\sigma_{2}))$ coincides with the trivial one.
It is verified that the exponents of $\nabla_j$ at $\sigma_{1}$ and $\sigma_{2}$ are $d_j$ and $-d_j$ (as elements of $\mbZ/p^\N \mbZ$), respectively.
The direct sum $\nabla := \bigoplus_{j=1}^n \nabla_{j}$ defines  a $\mcD^{(\N -1)}_{\mbP^{\mr{log}'}/S}$-module structure on the rank $n$ vector bundle
\begin{align}
\mcF := \bigoplus_{j=1}^n \mcO_\mbP (d_j [\sigma_{1}] -d_j [\sigma_{2}]).
\end{align}

Let us fix  isomorphisms of $\mcO_{\mbP}$-modules $\tau_j : \mcO_{\mbP} \isom \mcO_\mbP (d_j [\sigma_{1}] - d_j [\sigma_{2}])$ ($j =1, \cdots, n$).
Denote by  $\mcF^{n-1}$ the line subbundle of $\mcF$ defined to as the image of
the diagonal embedding $(\tau_1, \cdots, \tau_n) : \mcO_\mbP \migiincl \mcF$.
Also, for each $j = 0, \cdots, n-2$, let  $\mcF^j$ denote the subbundle of $\mcF$ generated locally by  $v, \nabla (v), \cdots, \nabla^{n- 1 - j}(v)$ for  local sections  $v \in \mcF^{n-1}$.
Since the mod $p$ reductions of $d_1, \cdots, d_n$ are mutually distinct, we see that
the  collection
\begin{align}
(\mcF, \nabla, \{ \mcF^j \}_{j=0}^n)
\end{align}
forms a dormant $\mr{GL}_n^{(\N)}$-oper on $\msP'$ (cf. Proposition \ref{NN77}, (ii)), and its  exponent at $\sigma_1$ coincides with  $[d_1, \cdots, d_n]$.

The composite
\begin{align} \label{eeQQ301}
\mcF_{\mcO_\mbP} \left(=\mcD^{(\N -1)}_{\mbP^{\mr{log}'}/S, \leq n-1} \otimes \mcO_\mbP = \mcD^{(\N -1)}_{\mbP^{\mr{log}'}/S, \leq n-1} \otimes \mcF^{n-1} \right) \xrightarrow{\mr{inclusion}} \mcD^{(\N -1)}_{\mbP^{\mr{log}'}/S} \otimes \mcF \xrightarrow{\nabla} \mcF
\end{align}
is an isomorphism, so $\nabla$ is transposed into a $\mcD^{(\N -1)}_{\mbP^{\mr{log}'}/S}$-module structure $\nabla^\diamondsuit$ on $\mcF_{\mcO_\mbP}$ via this composite.
Then, $\nabla^\diamondsuit$ specifies a dormant $(\mr{GL}_n^{(\N)}, \vartheta)$-oper, and 
the induced $\mr{PGL}_n^{(\N)}$-oper  $\msE^\spadesuit$ satisfies the required condition.
This completes 
the existence assertion of (ii).

To prove the uniqueness portion, 
let us take  another dormant 
$\mr{PGL}_n^{(\N)}$-oper $\msE^\spadesuit_1$ on $\msP'$ with $\rho = \rho_1 (\msE_1^\spadesuit) = \rho_2 (\msE^\spadesuit_1)^\veebar$;
it corresponds to a dormant $(\mr{GL}_n^{(\N)}, \vartheta)$-oper $\nabla_1^\diamondsuit$ on $\msP'$, whose exponent at $\sigma_1$ coincides with $[d_1, \cdots, d_n]$.
Under the identification $\mcF_{\mcO_\mbP} = \mcO_\mbP^{\oplus n}$ given by
composing \eqref{eeQQ301} and $\bigoplus_{j=1}^n \tau_j^{-1} : \bigoplus_{j=1}^n\mcO_\mbP (d_j [\sigma_1] -d_j [\sigma_2]) \isom \mcO_\mbP^{\oplus n}$,
 both $\nabla_1^\diamondsuit$ and $\nabla^\diamondsuit$ may be regarded as   $\mcD_{\mbP^{\mr{log}'}/S}^{(\N -1)}$-module structures on $\mcO_\mbP^{\oplus n}$.
After possibly applying a suitable  gauge transformation of $\mcO_\mbP^{\oplus n}$,
we may assume that
the monodromy operator $\mu_1 (\nabla_1^\diamondsuit)$ of $\nabla_1^\diamondsuit$ at $\sigma_1$ (cf. Definition \ref{D19}) coincide with that of $\nabla^\diamondsuit$;
let $(v_1, \cdots, v_n) \in H^0 (S, \mcO_S)^{\oplus n}$ be the image of the composite  of the diagonal embedding $\mcO_\mbP \hookrightarrow \mcO_\mbP^{\oplus n}$ and this gauge transformation.

For each $a = 0, \cdots, \N -1$ and each $\mcD_{\mbP^{\mr{log}'}/S}^{(\N -1)}$-module structure $\nabla'$ on $\mcO_\mbP^{\oplus n} \left(= \mcF_{\mcO_\mbP} \right)$,  we shall set
\begin{align}
M^a (\nabla') := (\nabla' (\partial_1^{\langle p^a \rangle}) (\vec{e}_1), \nabla' (\partial_1^{\langle p^a \rangle})  (\vec{e}_2), \cdots, \nabla'
  (\partial_1^{\langle p^a \rangle}) (\vec{e}_n)) \in  \mr{End}_{\mcO_S} (H^0 (\mbP, \mcO_\mbP)^{\oplus n}),
\end{align} 
where  
  $\vec{e}_1, \cdots, \vec{e}_n$ denote the canonical basis vectors. 
Because of ~\cite[Eq.\,(2.5)]{Mon}, 
each such $\nabla'$ 
 is uniquely determined by  $\{ M^a (\nabla')\}_{a= 0}^{\N -1}$.
On the other hand, the morphism 
\begin{align}
\mr{End}_{\mcO_S}(H^0 (\mbP, \mcO_\mbP)^{\oplus n})  \left(= \mr{End}_{\mcO_S}(H^0 (\mbP, \mcF_{\mcO_\mbP})) \right)\rightarrow \mr{End}_{\mcO_S} (\sigma_1^* (\mcF_{\mcO_\mbP}))
\end{align}
 given by restriction  to $\sigma_1$ is bijective and the image of $M^a (\nabla')$ coincides with $\mu_1 (\nabla')^{\langle p^a \rangle}$ (cf. \eqref{YY121}).
 Hence, the equality $\mu_1 (\nabla_1^\diamondsuit) = \mu_1 (\nabla^\diamondsuit)$ implies $\nabla^\diamondsuit_1 = \nabla^\diamondsuit \left(= \bigoplus_{j=1}^n \nabla_j \right)$.

Moreover, since
the section $(v_1, \cdots, v_n) \in H^0 (S, \mcO_S)^{\oplus n} = H^0 (\mbP, \mcO_{\mbP})^{\oplus n}$ generates the $\mcD^{(\N -1)}_{\mbP^{\mr{log}'}/S}$-module  $(\mcO_\mbP^{\oplus n}, \nabla^\diamondsuit)$,
each section $v_j$ lies in $H^{0} (S, \mcO_S^\times)$.
Hence, 
by the definition of a $\mr{GL}_n^{(\N)}$-oper, the equality $(\mcO_\mbP^{\oplus n}, \nabla_1^\diamondsuit) = (\mcO_\mbP^{\oplus n}, \nabla^\diamondsuit)$ preserves the filtration $\{ \mcF^j_{\mcO_\mbP} \}_j$ of $\mcO_{\mbP}^{\oplus n} \left(=\mcF_{\mcO_\mbP}\right)$
after applying a gauge transformation by the diagonal matrix with diagonal entries $v_1, \cdots, v_n$.
This completes the proof of the uniqueness assertion.
\end{proof}
\SSP

By combining Propositions \ref{Prop2234} and \ref{Prop13}, (i) and (ii),
we obtain the following assertion.
(Note that the constructions asserted in Proposition \ref{Prop15} and \ref{Prop2234} are available even when the hyperbolic condition $2g-2+r >0$ does not hold.)

\SSP
\bco \label{Cor209}
Let $\rho_1, \rho_2$ be elements of $\Xi_{n, \N}$.
Then,
we have
\begin{align}
\mcO p_{(\rho_1, \rho_2, \varepsilon), 0, 3}^{^\mr{Zzz...}} \cong \begin{cases} \mr{Spec}(\mbF_p)& \text{if $\rho_1 = \rho_2^\veebar$}; \\ \emptyset & \text{if otherwise}. \end{cases}
\end{align}
In particular,  the following equality holds:
\begin{align}
\mr{deg}
(\Pi_{(\rho_1, \rho_2, \varepsilon), 0, 3}) =
\begin{cases} 1 & \text{if $\rho_1 = \rho_2^\veebar$}; \\ 0  & \text{if otherwise}. \end{cases}
\end{align}
\eco

\LSP
\subsection{Dual of $\mr{GL}_n^{(\N)}$-opers} \label{SS080}
Let $(g, r)$ be a pair of nonnegative integers with $2g-2+r > 0$ and $\msX := (X, \{ \sigma_i \}_{i=1}^r)$ an $r$-pointed  stable curve of genus $g$ over an $\mbF_p$-scheme $S$.
 For simplicity, we write $\Omega := \Omega_{X^\mr{log}/S^\mr{log}}$, $\mcT := \mcT_{X^\mr{log}/S^\mr{log}}$, and $\mcD^{(\N -1)}:= \mcD^{(\N -1)}_{X^\mr{log}/S^\mr{log}}$.

Consider a $\mr{GL}_n^{(\N)}$-oper
$\msF^\heartsuit := (\mcF, \nabla, \{ \mcF^j \}_{j=0}^n)$   on $\msX$.
For each $j=0, \cdots, n$, we regard $\mcF^{j \vee} := (\mcF/\mcF^{n-j})^\vee$ as a subbundle  of $\mcF^\vee$.
According to ~\cite[Corollaire 2.6.1, (ii)]{Mon},
$\nabla$ induces a $\mcD^{(\N -1)}$-module structure $\nabla^\vee$ on $\mcF^\vee$.
The resulting collection
\begin{align}
\msF^{\heartsuit \blacktriangledown} := (\mcF^\vee, \nabla^\vee, \{ \mcF^{\vee j}\}_{j=0}^n)
\end{align}
 forms  a $\mr{GL}_n^{(\N)}$-oper on $\msX$, which will be  called  the {\bf dual} of $\msF^\heartsuit$.
Note that 
$(\msF^{\heartsuit \blacktriangledown})^\blacktriangledown$ is isomorphic to $\msF^\heartsuit$ itself, and  that $\msF^{\heartsuit \blacktriangledown}$ is dormant when $\msF^\heartsuit$ is dormant.

Next, let $\vartheta := (\varTheta, \nabla_\vartheta)$ be an $n^{(\N)}$-theta characteristic of $X^\mr{log}/S^\mr{log}$.
Write
\begin{align}
\varTheta^\blacktriangledown := \Omega^{\otimes (n-1)} \otimes \varTheta^\vee.
\end{align}
Then, we have a composite isomorphism
\begin{align}
\mcT^{\otimes \frac{n (n-1)}{2}} \otimes (\varTheta^\blacktriangledown)^{\otimes n}
\isom \mcT^{\otimes \frac{n (n-1)}{2}} \otimes \mcT^{\otimes (-n (n-1))} \otimes (\varTheta^{\vee})^{\otimes n} \isom (\mcT^{\otimes \frac{n(n-1)}{2}} \otimes \varTheta^{\otimes n})^\vee.
\end{align}
By this composite, the dual $\nabla_\vartheta^\vee$ of $\nabla_\vartheta$ may be regarded as a $\mcD^{(\N -1)}$-module structure  on $\mcT^{\otimes \frac{n(n-1)}{2}} \otimes (\varTheta^\blacktriangledown)^{\otimes n}$.
Thus, we obtain an $n^{(\N)}$-theta characteristic
\begin{align}
\vartheta^\blacktriangledown := (\varTheta^\blacktriangledown, \nabla_\vartheta^\vee)
\end{align}
of $X^\mr{log}/S^\mr{log}$.
We shall refer to  $\vartheta^\blacktriangledown$  as the {\bf dual} of $\vartheta$.
It is immediately verified that $(\vartheta^\blacktriangledown)^\blacktriangledown$ coincides with  $\vartheta$, and that $\vartheta^\blacktriangledown$ is dormant when $\vartheta$ is dormant.

Fix  a $(\mr{GL}_n^{(\N)}, \vartheta)$-oper $\nabla^\diamondsuit$  on $\msX$.
Since 
\begin{align}
\mcF_\varTheta^{\vee n-1} = (\mcF_\varTheta / \mcF_\varTheta^1)^\vee = (\mcT^{\otimes  (n-1)} \otimes \varTheta)^\vee = \varTheta^\blacktriangledown,
\end{align}
we have an inclusion $\varTheta^\blacktriangledown \migiincl \mcF_\varTheta^\vee$.
The composite
\begin{align}
\mcF_{\varTheta^\blacktriangledown} \left(=  \mcD_{\leq n-1}^{(\N -1)} \otimes \varTheta^\blacktriangledown \right) \xrightarrow{\mr{inclusion}} \mcD^{(\N -1)} \otimes \mcF_\varTheta^\vee \xrightarrow{\nabla^{\diamondsuit \vee}} \mcF_\varTheta^\vee
\end{align}
turns out to be an isomorphism, and
the dual $\nabla^{\diamondsuit \vee}$ of $\nabla^\diamondsuit$ corresponds, via this composite, to a $\mcD^{(\N-1)}$-module structure 
\begin{align} \label{eeQQ205}
\nabla^{\diamondsuit \blacktriangledown}
\end{align}
on $\mcF_{\varTheta^\blacktriangledown}$.
Moreover, $\nabla^{\diamondsuit \blacktriangledown}$ forms a $(\mr{GL}_n^{(\N)}, \vartheta^\blacktriangledown)$-oper on $\msX$ whose underlying $\mr{GL}_n^{(\N)}$-oper coincides with the dual of $\nabla^{\diamondsuit \Rightarrow \heartsuit}$, i.e., $(\nabla^{\diamondsuit \blacktriangledown})^{\Rightarrow \heartsuit} = (\nabla^{\diamondsuit \Rightarrow \heartsuit})^\blacktriangledown$.

If both $\vartheta$ and $\nabla^\diamondsuit$ are dormant, then 
$\nabla^{\diamondsuit \blacktriangledown}$
 is dormant.
Thus,
the assignment $\nabla^\diamondsuit \mapsto \nabla^{\diamondsuit \blacktriangledown}$ defines an equivalence of categories between the categories of  (dormant) $(\mr{GL}_n^{(\N)}, \vartheta)$-opers and (dormant) $(\mr{GL}_n^{(\N)}, \vartheta^\blacktriangledown)$-opers.

Suppose that we are given a $\mr{PGL}_n^{(\N)}$-oper $\msE^\spadesuit$ on $\msX$.
The dual of the $(\mr{GL}_n^{(\N)}, \vartheta)$-oper corresponding  to $\msE^\spadesuit$  induces a $\mr{PGL}_n^{(\N)}$-oper 
\begin{align} \label{eeQQ310}
\msE^{\spadesuit \blacktriangledown}.
\end{align}
 The isomorphism classes of $\msE^{\spadesuit \blacktriangledown}$ does not depend on the choice of $\vartheta$, and $\msE^\spadesuit$ is dormant if and only if $\msE^{\spadesuit \blacktriangledown}$ is dormant.
 Also, we have $(\msE^{\spadesuit \blacktriangledown})^\blacktriangledown \cong \msE^\spadesuit$.
 
\SSP
\bpr
\begin{itemize}
\item[(i)]
The assignment $\msE^\spadesuit \mapsto \msE^{\spadesuit \blacktriangledown}$ is functorial with respect to pull-back over $S$, and hence, determines an isomorphism of $\overline{\mcM}_{g, r}$-stacks 
\begin{align} \label{Eq209}
\DDual : \mcO p_{g, r} \isom \mcO p_{g, r} \ \left(\text{resp.,} \ {^p}\DDual : \mcO p_{g, r}^{^\mr{Zzz...}} \isom \mcO p_{g, r}^{^\mr{Zzz...}}  \right),
\end{align}
satisfying $\DDual \circ \DDual \cong \mr{id}$ (resp., ${^p}\DDual \circ {^p}\DDual \cong \mr{id}$).
\item[(ii)]
Let $\rho$ be an element of $\Xi_{n, \N}^{\times r}$,
Then,  \eqref{Eq209} restricts to an isomorphism   of $\overline{\mcM}_{g, r}$-stacks
\begin{align}
{^p}\DDual_{\rho} : \mcO p_{\rho, g,r}^{^\mr{Zzz...}} \isom \mcO p_{\rho^\veebar, g, r}^{^\mr{Zzz...}}
\end{align}
with ${^p}\DDual_{\rho^\veebar} \circ {^p}\DDual_{\rho} \cong \mr{id}$ (cf. \eqref{e56} for the definition of $(-)^\veebar$).
\end{itemize}
\epr
\begin{proof}
Assertion (i) follows from the above discussion.
Also, assertion (ii) follows from  assertion (i) together with the second assertion of Proposition \ref{YY30}, (i).
\end{proof}
\SSP

The following assertion is a direct consequence of the above proposition.

\SSP
\bco \label{Cor459}
Let us keep the assumption in Theorem \ref{Theorem44}.
Then, for each $\rho \in \Xi^{\times r}_{n, \N}$, the following equality holds:
\begin{align}
\mr{deg}(\Pi_{\rho, g, r}) = \mr{deg} (\Pi_{\rho^\veebar, g, r}).
\end{align}
\eco

\LSP
\subsection{$2$d TQFT for dormant $\mr{PGL}_n^{(\N)}$-opers} \label{SS322}

This section aims  to  describe the various factorization properties on  the generic degrees $\Pi_{\rho, g, r}$ proved so far  in terms of  $2$d TQFTs (= $2$-dimensional topological quantum field theories).
To begin with,  let us briefly  recall what a $2$d TQFT is.
For details on its  precise  definition, we refer the reader to ~\cite{Koc} (or  ~\cite{Ati}, ~\cite{DuMu1}, ~\cite{DuMu2}).

Let us fix a field $K$ of characteristic $0$ (say, $\mbQ$ or $\mbC$).
Denote by 
 $\mcV ect_K$ the category consisting of vector spaces over $K$ and $K$-linear maps between them.
 Equipped with the ordinary tensor product $\otimes_K$ as the multiplication,
 with the ground field $K$ as the unit, and with 
     the collection of maps $T_{\mcV ect} := \{ T_{V, V'} : V \otimes_K V' \migi  V' \otimes_K V \}_{V, V' \in \mr{Ob}(\mcV ect_K)}$  given by interchanging the two factors of $\otimes_K$ as the symmetric braiding,
   the quadruple  
   \begin{align} \label{eeQQ320}
   (\mcV ect_K, \otimes_K, K, T_{\mcV ect})
   \end{align}
    specifies a symmetric monoidal category.
 
 Next, let $\Sigma$ and $\Sigma'$ be closed oriented $(\ell -1)$-dimensional manifolds ($\ell \in \mbZ_{> 0}$).
 An {\it oriented cobordism} from $\Sigma$ to $\Sigma'$ is a compact oriented $\ell$-dimensional manifold $M$ together with smooth maps $\Sigma\migi M$, $\Sigma' \rightarrow M$
  such that $\Sigma$ maps diffeomorphically (preserving orientation) onto the in-boundary of $M$, and $\Sigma'$ maps diffeomorphically (preserving orientation) onto the out-boundary of $M$.
 We will denote it by $M : \Sigma \Rightarrow \Sigma'$.
 Two 
 oriented cobordisms $M, M' : \Sigma \Rightarrow \Sigma'$ are {\it equivalent} if there is an orientation-preserving diffeomorphism $\psi : M \isom M'$ inducing the identity morphisms of $\Sigma$ and $\Sigma'$.

 Denote by $\ell\text{-}\mcC ob$ the category defined as follows:
 \begin{itemize}
 \item
 The objects are $(\ell-1)$-dimensional closed oriented manifolds;
  \item
  Given two such objects $\Sigma$ and $\Sigma'$, a morphism  from $\Sigma$ to $\Sigma'$ is 
  an equivalence class
   of oriented cobordisms $M : \Sigma \Rightarrow  \Sigma'$ (in the above sense).
   The identity morphisms are just the cylinders, and
  the composition of morphisms is given by gluing cobordism classes.
 \end{itemize}
Equipped with the disjoint union $\sqcup$ as the multiplication, with the empty manifold $\emptyset$ as the unit, and with the collection  of twist diffeomorphisms 
 $T_{\ell\text{-}\mcC ob} := \{ T_{\Sigma, \Sigma'} : \Sigma \sqcup \Sigma' \Rightarrow \Sigma' \sqcup \Sigma \}_{\Sigma, \Sigma' \in \mr{Ob}(\ell\text{-}\mcC ob)}$ as 
 the symmetric braiding,  the quadruple
\begin{align}
(\ell\text{-}\mcC ob, \sqcup, \emptyset, T_{\ell\text{-}\mcC ob})
\end{align}
forms a symmetric monoidal category.

  \SSP
\bde[cf. ~\cite{Koc}, Section 1.3.32] \label{Def5932}
An {\bf $\ell$-dimensional topological quantum field theory} (over  $K$), or {\bf $\ell$d TQFT} for short,  
is 
a symmetric monoidal functor  from $(\ell\text{-}\mcC ob, \sqcup, \emptyset, T_{\ell\text{-}\mcC ob})$ to 
$(\mcV ect_K, \otimes_K, K,  T_{\mcV ect})$.
\ede
 \SSP

Hereinafter, we focus on $2$d TQFTs.
We can uniquely classify an isomorphism class of objects in $2\text{-}\mcC ob$ with an integer $n \in \mbZ_{\geq 0}$ indicating the number of connected components, i.e., the number of disjoint circles $\mbS:= \left\{ (x, y) \in \mbR^2 \, | \, x^2 + y^2 =1 \right\}$.
In other words, the full subcategory whose objects are $\{ \mbS^{r} \, | \, r \in \mbZ_{\geq 0} \}$, where $\mbS^{ 0} := \emptyset$ and  $\mbS^{r}$ denotes the disjoint union of $r$ copies of $\mbS$,  forms 
a skeleton of $2\text{-}\mcC ob$.

Also, 
each connected oriented cobordism in $2\text{-}\mcC ob$ may be represented by
 $\mbM^{r \Rightarrow  s}_g$  for some triple of nonnegative integers $(g, r, s)$,
 where $\mbM^{r \Rightarrow  s}_g$ denotes a connected, compact oriented surface whose in-boundary and out-boundary are  $\mbS^{r}$  and $\mbS^{s}$, respectively.
According to ~\cite[Lemma 1.4.19]{Koc},  every oriented cobordism in $2\text{-}\mcC ob$ factors as a permutation cobordism, followed by a disjoint union of  $\mbM^{r \Rightarrow  s}_g$'s (for various triples  $(g, r, s)$), followed by a permutation cobordisms.

It follows  that a $2$d TQFT $\mcZ :  2\text{-}\mcC ob \rightarrow \mcV ect_K$ is uniquely determined by the $K$-vector space $A := \mcZ (\mbS^1)$ and the collection of $K$-linear maps 
\begin{align}
\omega^{r \Rightarrow s}_g := \mcZ (\mbM^{r \Rightarrow  s}_g) : A^{\otimes r}  \left(= \mcZ (\mbS^r) \right)\rightarrow  A^{\otimes s}\left(= \mcZ (\mbS^s) \right)
\end{align}
 for $(g, r, s) \in \mbZ_{\geq 0}^{\times 3}$ (where $A^{\otimes 0} := K$).
This collection of data satisfies the following properties (cf. ~\cite[Section 3]{DuMu1}, ~\cite[Definition 2.1]{DuMu2}):
\begin{itemize}
\item[(1)]
The $K$-linear map $\omega^{1 \Rightarrow 1}_0$ coincides with $\mr{id}_A$ (because $\mcZ$ is a functor), 
and 
$\omega^{r \Rightarrow s}_g$ (for $r \geq 1$)
 is symmetric with respect to the  action of $\mfS_r$ (:=  the symmetric group of $r$ letters)  arising from permutations of the $r$ factors  in the domain $A^{\otimes r}$;
\item[(2)]
If $\breve{\omega}_0^{2 \Rightarrow 0}$ denotes the $K$-linear morphism $A \rightarrow A^\vee$ induced by $\omega^{2 \Rightarrow 0}_0 : A^{\otimes 2} \rightarrow K$,
then the square diagram
\begin{align} \label{eeQQ401}
\vcenter{\xymatrix@C=46pt@R=36pt{
 A^{\otimes r} \ar[r]^-{\omega^{r \Rightarrow s}_g} \ar[d]_-{(\breve{\omega}_0^{2 \Rightarrow 0})^{\otimes r}} & A^{\otimes s}\ar[d]^-{(\breve{\omega}_0^{2 \Rightarrow 0})^{\otimes s}} \\
 (A^\vee)^{\otimes r} \ar[r]_-{(\omega^{s \Rightarrow r}_g)^\vee} &(A^\vee)^{\otimes s}
 }}
\end{align}
is commutative for any $(g, r, s)$.
\item[(3)]
For each triples $(g_1, r_1, s_1), (g_2, r_2, s_2) \in \mbZ_{\geq 0}^{\times 3}$ and each positive  integer $\ell$ with $\ell \leq s_1$ and $\ell \leq r_2$, the  equality 
\begin{align} \label{eeQQ404}
(\mr{id}_A^{\otimes (s_1 -\ell)} \otimes \omega_{g_2}^{r_2 \Rightarrow s_2})\circ (\omega_{g_1}^{r_1 \Rightarrow s_1} \otimes \mr{id}_A^{\otimes (r_2 -\ell)})
= \omega_{g_1 + g_2 + \ell -1}^{(r_1 + r_2 - \ell) \Rightarrow (s_1 + s_2 - \ell)} 
\end{align}
of $K$-linear maps $A^{\otimes (r_1 + r_2 - \ell)} \rightarrow A^{\otimes (s_1 + s_2 -\ell)}$ 
(arising  from gluing only $\ell$ pairs of boundary circles in two cobordisms) holds;
\end{itemize}

By the condition (3) for 
$(g_1, r_1, s_1) = (0, 0, 2)$, $(g_2, r_2, s_2) = (0, 2, 0)$,
  and $\ell =1$,
 the morphism $\breve{\omega}_0^{2 \Rightarrow 0}$ introduced in (2) turns out to be an isomorphism (i.e., $\omega^{2 \Rightarrow 0}_0$ is nondegenerate).
On the other hand, the case of 
$(g_1, r_1, s_1) = (0, 1, 2)$, $(g_2, r_2, s_2)  = (0, 1, 0)$,
and $\ell =1$ implies that 
  $\omega^{1 \Rightarrow 0}_0$ is nontrivial.

Conversely, a $K$-vector space $A$ together with a collection of various $K$-linear maps $\omega_g^{r \Rightarrow s} : A^{\otimes r} \rightarrow A^{\otimes s}$ satisfying the conditions (1)-(3) above extends to  a unique $2$d TQFT over $K$.

Returning to our discussion, we obtain the following assertion.

\SSP
\bt \label{Theorem4f4}
Suppose that $\Pi_{\rho, g, r}$ (for any $(g, r, \rho)$) is finite and  \'{e}tale over an  \'{e}tale  neighborhood of the  point in $\overline{\mcM}_{g, r}$ classifying  each totally degenerate curve.
Then, there exists a unique  $2$d TQFT
 \begin{align} \label{DeeeQQ}
 \mcZ_{n, \N} : (2\text{-}\mcC ob, \sqcup, \emptyset, T_{2\text{-}\mcC ob}) \rightarrow (\mcV ect_K, \otimes_K, K, T_{\mcV ect})
 \end{align}
  over $K$ determined by the following conditions:
\begin{itemize}
\item
$\mcZ_{n, \N} (\mbS^1) = K^{\Xi_{n, \N}}$, i.e., the $K$-vector space with basis $\Xi_{n, \N}$;
\item
$\mcZ_{n, \N} (\mbM_0^{0 \Rightarrow 0}) = \mr{id}_K$, and 
$\mcZ_{n, \N} (\mbM_1^{0 \Rightarrow 0}) = \sharp (\Xi_{n, \N}) \cdot \mr{id}_K \left(= \binom{p}{n} \cdot p^{n \N -n -\N} \cdot \mr{id}_K \right)$;
\item
$\mcZ_{n, \N} (\mbM_0^{0 \Rightarrow 1}) : K \rightarrow K^{\Xi_{n, \N}}$ 
and 
$\mcZ_{n, \N} (\mbM_0^{0 \Rightarrow 2})  : K \rightarrow    (K^{\Xi_{n, \N}})^{\otimes 2}$ satisfy
\begin{align} \label{eeQwe91}
\mcZ_{n, \N} (\mbM_0^{0 \Rightarrow 1})(1) = \varepsilon 
\hspace{5mm} \text{and} \hspace{5mm}
\mcZ_{n, \N} (\mbM_0^{0 \Rightarrow 2}) (1) = \sum_{\lambda \in \Xi_{n, \N}} \lambda \otimes \lambda^\veebar, 
\end{align}
respectively.
\item
$\mcZ_{n, \N} (\mbM_0^{1 \Rightarrow 0}) : K^{\Xi_{n, \N}} \rightarrow K$ and 
$\mcZ_{n, \N} (\mbM_0^{2 \Rightarrow 0}) : (K^{\Xi_{n, \N}})^{\otimes 2} \rightarrow K$ satisfy
\begin{align} \label{eeQwe90}
\mcZ_{n, \N} (\mbM_0^{1 \Rightarrow 0}) (\lambda) = 
\begin{cases} 1 & \text{if $\lambda = \varepsilon$}; \\
0 & \text{if otherwise},
\end{cases}
\hspace{5mm} \text{and} \hspace{5mm}
\mcZ_{n, \N} (\mbM_0^{2 \Rightarrow 0}) (\lambda \otimes \eta) = 
\begin{cases}
1 & \text{if $\eta = \lambda^\veebar$}; \\
0 & \text{if otherwise},
\end{cases}
\end{align}
respectively.
\item
For any triple of nonnegative integers $(g, r, s)$ with $2g-2 + r +s > 0$,
the $K$-linear map $\mcZ_{n, \N} (\mbM_g^{r \Rightarrow s}) : (K^{\Xi_{n, \N}})^{\otimes r} \rightarrow (K^{\Xi_{n, \N}})^{\otimes s}$ is given  by 
\begin{align} \label{eeQQ405}
\mcZ_{n, \N} (\mbM_g^{r \Rightarrow s}) (\bigotimes_{i=1}^r \rho_i) = \sum_{(\lambda_j)_j \in \Xi_{n, \N}^{\times s}} \mr{deg}(\Pi_{((\rho_i)_i, ( \lambda_{j}^\veebar)_{j}), g, r +s}) \bigotimes_{j=1}^{s} \lambda_{j}.
\end{align}
\end{itemize}
\et 
\begin{proof}
For simplicity, we write $\omega_g^{r \Rightarrow s} := \mcZ_{n, \N} (\mbM_g^{r \Rightarrow s})$, $\Xi := \Xi_{n, \N}$,  $\Xi^r := \Xi_{n, \N}^{\times r}$, and moreover, we abbreviate   
 $\mr{deg}(\Pi_{(\rho_i)_i, g, r})$ to the notation 
$D_{(\rho_i)_i, g, r}$ or $D_{(\rho_i)_i}$ if there is no fear of confusion.

The condition (1) described above is fulfilled because of the definition of $\mcZ_{n, \N}$ (together with the fact that $D_{(\rho_i)_i}$ does not depend on  the ordering of $(\rho_i)_i$).

Given an element  $\lambda \in \Xi$, we shall write $\lambda^\vee$ for the element of $K^{\Xi \vee}\left(= \mr{Hom}_K (K^\Xi, K) \right)$ determined by $\lambda^\vee (\rho) = 1$ if $\lambda = \rho$ and $\lambda^\vee (\rho) = 0$ if $\lambda \neq \rho \in \Xi$.
For each triple $(g, r, s) \in \mbZ_{\geq 0}^{\times 3}$ with $2g-2+r +s > 0$ and each $(\rho_i)_i \in \Xi^{r}$,  the following sequence of equalities holds:
\begin{align}
(\omega_g^{s \Rightarrow r})^\vee \circ (\breve{\omega}_0^{2 \Rightarrow 0})^{\otimes r} ( \bigotimes_{i=1}^r \rho_i) 
&= (\omega_g^{s \Rightarrow r})^\vee ( \bigotimes_{i=1}^r (\rho_i^\veebar)^\vee) \\
& =  \sum_{(\lambda_j)_j} 
D_{((\lambda_j)_j, (\rho_i)_i)}
  \bigotimes_{j=1}^s \lambda_j^\vee
\notag \\
&=  (\breve{\omega}_0^{2 \Rightarrow 0})^{\otimes s} (\sum_{(\lambda_j)_j} 
D_{((\rho_i)_i, (\lambda_j)_j)}
 \bigotimes_{j=1}^s \lambda^\veebar_j)  \notag \\
&= (\breve{\omega}_0^{2 \Rightarrow 0})^{\otimes s} \circ  \omega_g^{r \Rightarrow s} (\bigotimes_{i=1}^r \rho_i), \notag
\end{align}
where the third follows from the fact that 
$D_{((\rho_i)_i, (\lambda_j)_j)}$ 
 does not depend on  the ordering of $\rho_1, \cdots, \rho_r, \lambda_1, \cdots, \lambda_s$.
That is to say,  the square diagram \eqref{eeQQ401} is commutative.
The commutativities   for the remaining cases of $(g, r, s)$'s  are  immediately  verified, so $\mcZ_{n, \N}$ satisfies the condition (2).

Finally, let us take two triples $(g_j, r_j, s_j) \in \mbZ_{\geq 0}^{\times 3}$ ($j=1, 2$) with $2g_j - 2 + r_j + s_j > 0$ and  a positive integer $\ell$ with $\ell \leq s_1$ and $\ell \leq r_2$.
Consider the clutching data
\begin{align}
\mbG := (\GR, \{ (g_j, r_j + s_j) \}_{j=1}^2, \{ \lambda_j  \}_{j=1}^2)
\end{align}
determined by the following conditions:
\begin{itemize}
\item
$\mbG$ is of type $(g_1 + g_2 + \ell -1, r_1 + r_2 + s_1 + s_2 -2\ell)$;
\item
$\GR := (\{ v_1, v_2 \}, \{e_{i} \}_{i=1}^{\ell} \sqcup \{ e_{1, i} \}_{i=1}^{r_1 + s_1-\ell} \sqcup \{ e_{2, i} \}_{i=1}^{r_2 + s_2 -\ell}, \zeta)$, where 
$v_j$'s, $e_{i}$'s, and $e_{j, i}$'s are abstract symbols, and 
$\zeta$ is given by $\zeta (e_i) = \{ v_1, v_2 \}$ (for $i=1, \cdots, \ell$) and 
$\zeta (e_{j, i}) = \{ \circledast, v_j \}$ (for $j=1, 2$ and $i=1, \cdots, r_j + s_j -\ell$).
\item
For each  $j=1, 2$, the bijection $\lambda_j : B_{v_j} \isom \{1, \cdots, r_j + s_j \}$ is given by $\lambda_j (b) = i$ if $b \in e_{j, i}$ ($1 \leq i \leq r_j + s_j -\ell$) and $\lambda_j (b) = r_j + s_j -\ell + i$ if $b \in e_i$.
\end{itemize}
By applying Theorem \ref{Theorem44} to this clutching data,
we obtain, for each 
$(\rho_{1, i_1})_{i_1} \in \Xi^{r_1 + s_1 -\ell}$
and $(\rho_{2, i_2})_{i_2} \in \Xi^{r_2 + s_2 -\ell}$,
  the equality 
\begin{align} \label{eeQQ420}
D_{((\rho_{1, i_1})_{i_1}, (\rho_{2, i_2})_{i_2})} = \sum_{\rho_3 \in \Xi^{\ell}}
D_{((\rho_{1, i_1})_{i_1}, \rho_3)} \cdot D_{((\rho_{2, i_2})_{i_2}, \rho_3^\vee)}.
\end{align}
It follows that, for  $(\rho_i)_i \in \Xi^{r_1 + r_2 -\ell}$,
the following sequence of equalities holds:
\begin{align}
& \ \ \ \ (\mr{id}_A^{\otimes (s_1 -\ell)} \otimes \omega_{g_2}^{r_2 \Rightarrow s_2})\circ (\omega_{g_1}^{r_1 \Rightarrow s_1} \otimes \mr{id}_A^{\otimes (r_2 -\ell)} ) (\bigotimes_{i =1}^{r_1 + r_2 -\ell} \rho_i) \\
& = (\mr{id}_A^{\otimes (s_1 -\ell)} \otimes \omega_{g_2}^{r_2 \Rightarrow s_2}) (\sum_{(\lambda_j)_{j=1}^{s_1} \in \Xi^{s_1}} D_{((\rho_i)_{i=1}^{r_1}, (\lambda_j^\veebar)_{j=1}^{s_1})})\bigotimes_{j=1}^{s_1} \lambda_j \otimes \bigotimes_{i=r_1+1}^{r_1+r_2-\ell}\rho_i) \notag \\
& = \sum_{(\lambda_j)_{j=1}^{s_1} \in \Xi^{s_1}}
\sum_{(\eta_{j'})_{j' =1}^{s_2} \in \Xi^{s_2}}  
D_{((\rho_i)_{i=1}^{r_1}, (\lambda_j^\veebar)_{j=1}^{s_1})} \cdot D_{((\lambda_j)_{j= s_1 -\ell +1}^{s_1}, (\rho_i)_{i=r_1 + 1}^{r_1 + r_2 -\ell},  (\eta_{j'}^\veebar)_{j'=1}^{s_2})} \bigotimes_{j=1}^{s_1 -\ell} \lambda_j \otimes \bigotimes_{j' =1}^{s_2} \eta_{j'}\\
& \stackrel{\eqref{eeQQ420}}{=} 
\sum_{(\lambda_j)_{j=1}^{s_1-\ell} \in \Xi^{s_1-\ell}}
\sum_{(\eta_{j'})_{j' =1}^{s_2} \in \Xi^{s_2}} 
D_{((\rho_i)_{i=1}^{r_1 +r_2 -\ell}, (\lambda_j^\veebar)_{j=1}^{s_1 -\ell}, (\eta_{j'}^\veebar)_{j' =1}^{s_2})}
\bigotimes_{j=1}^{s_1 -\ell} \lambda_j \otimes \bigotimes_{j' =1}^{s_2} \eta_{j'}
\notag \\
& = \omega_{g_1 + g_2 + \ell -1}^{(r_1 + r_2 - \ell) \Rightarrow (s_1 + s_2 - \ell)} (\bigotimes_{i =1}^{r_1 + r_2 -\ell} \rho_i).
\end{align}
This proves   \eqref{eeQQ404}.
Also, the  equalities  for the remaining cases of $(g_1, r_1, s_1)$ and $(g_2, r_2, s_2)$ can be immediately verified from the various results obtained in this section, so $\mcZ_{n, \N}$ satisfies the condition (3).
Thus, the proof of the assertion is completed.
\end{proof}

\LSP
\subsection{Dormant fusion ring} \label{SS064}

The factorization property of $\Pi_{\rho, g, r}$'s resulting from Theorem \ref{Theorem44}  is also described in terms of {\it fusion rules}, which can be  essentially regarded as a special type of $2$d TQFT in a certain sense.
Applying a discussion in the general theory of fusion rules (cf.  ~\cite{Beau}), we will see that 
 our fusion rule associates  a commutative ring encoding this factorization property, as well as the data of  the values  $\mr{deg}(\Pi_{\rho, g, r})$.

Let $I$ be a finite set with an involution $(-)^*$ (i.e., a bijection of order $2$).
Denote by $\mbN^{I} := \bigoplus_{a \in I} \mbN a$ be the free commutative monoid generated by the elements of $I$.
The involution of $I$ extends, in an evident manner, to an involution of $\mbN^{I}$.

\SSP
\bde[cf. ~\cite{Beau}, Section 5, Definition] \label{Def491}
A {\bf fusion rule} on $I$ is a map $F : \mbN^{I} \migi \mbZ$ satisfying the following three conditions:
\begin{itemize}
\item[(1)]
One has $F (0) = 1$, and $F (\alpha) > 0$ for some $\alpha \in I$;
\item[(2)]
$F(x^*) = F(x)$ for every $x \in \mbN^{I}$;
\item[(3)]
For $x$, $y \in \mbN^{(I)}$, one has $F (x + y) = \sum_{\lambda \in I} F(x + \lambda) \cdot F(y +\lambda^*)$.
\end{itemize}
Also, a fusion rule $F$ on $I$ is said to be {\bf nondegenerate} if it satisfies the following condition:
\begin{itemize}
\item[(4)]
For any $\alpha \in I$,
there exists an element $\beta \in I$ with $F (\alpha + \beta) \neq 0$.
\end{itemize}
\ede
\SSP

Now, suppose that $\Pi_{\rho, g, r} := \Pi_{n, \N, \rho, g, r, \mbF_p}$ (for any $(g, r, \rho)$) is finite and  \'{e}tale over an  \'{e}tale  neighborhood of the point in  $\overline{\mcM}_{g, r}$ classifying each  totally degenerate curve.
Write
\begin{align} \label{eeQQ439}
F_{n, \N} : \mbN^{\Xi_{n, \N}} \migi \mbZ
\end{align}
for the map  determined  by 
the following rules:
\begin{itemize}
\item
$F_{n, \N} (0) := 1$;
\item
For $\lambda \in \Xi_{n, \N}$, we set 
$F_{n, \N} (\lambda) := 1$ if $\lambda = \varepsilon$ and $F_{n, \N} (\lambda) = 0$ if 
otherwise;
\item
For $(\lambda, \eta) \in \Xi_{n, \N}^{\times 2}$, we set  
$F_{n, \N} (\lambda + \eta) := 1$ if $\eta = \lambda^\vee$ and $F_{n, \N} (\lambda + \eta) := 0$ if otherwise;
\item
For any $(\rho_i)_i \in \Xi_{n, \N}^{\times r}$ (with $r \geq 3$),
we set
$F_{n, \N} (\sum_{i=1}^r \rho_i) := \mr{deg}(\Pi_{(\rho_i)_i, 0, r})$.
\end{itemize}

Regarding this map, 
one can prove  the following assertion, which is a higher-level generalization of   ~\cite[Proposition 7.33]{Wak8}.

\SSP
\bt \label{Prop221}
(Recall that we have assumed that the morphism $\Pi_{\rho, g, r}$ for any $(g, r, \rho)$ is finite and \'{e}tale over an  \'{e}tale  neighborhood of the point in $\overline{\mcM}_{g, r}$ classifying each totally degenerate curve.)
The map  $F_{n, \N}$ forms a nondegenerate fusion rule on the finite set $\Xi_{n, \N}$ (with respect to the involution $(-)^\veebar$).
\et
\begin{proof}
The conditions (1) and (4) are  fulfilled because of  
the definition of $F_{n, \N}$ and Corollary \ref{Cor209}.
Also, Corollary \ref{Cor459}  and \eqref{Eq101} show that $F_{n, \N}$ satisfies the conditions (2) and (3), respectively.
\end{proof}
\SSP

According to the discussion in  ~\cite[Section 5]{Beau},  the fusion rule $F_{n, \N}$ associates a ring encoding its structure.
To be precise, 
we define a multiplication law $* : \mbZ^{\Xi_{n, \N}} \times \mbZ^{\Xi_{n, \N}} \migi \mbZ^{\Xi_{n, \N}}$ on $\mbZ^{\Xi_{n, \N}}$
 by putting
\begin{align}
\alpha * \beta = \sum_{\lambda \in \Xi_{n, \N}} F_{n, \N} (\alpha + \beta + \lambda^\veebar) \lambda
\end{align}
for any $\alpha, \beta \in \Xi_{n, \N}$,
and extending by bilinearity.
The abelian group $\mbZ^{\Xi_{n, \N}}$ together with this multiplication law defines a unital, associative, and commutative ring
\begin{align} \label{eeQQ45}
\Fus_{n, \N}
\end{align}
with  identity element  $\varepsilon$.

The tensor product $K^{\Xi_{n, \N}} := K \otimes_\mbZ \Fus_{n, \N}$  (for the field $K$ fixed at the beginning of Section \ref{SS322}) may be identified with  the Frobenius algebra corresponding to the $2$d TQFT $\mcZ_{n, \N}$, and the Frobenius  pairing is determined  by $(\alpha, \beta) \mapsto F_{n, \N}(\alpha  + \beta)$ for $\alpha, \beta \in \Xi_{n, \N}$.
(Recall  from, e.g., ~\cite[Section 2.2.5]{Koc}, that a {\bf Frobenius algebra} over $K$ is a $K$-algebra $A$ of finite dimension equipped with an associative  nondegenerate pairing $A \otimes_K A \migi K$, called the {\bf Frobenius pairing}.
As a classical result, there exists an equivalence of categories between the category of $2$d TQFTs over $K$ and the category of commutative Frobenius algebras over $K$. See ~\cite[Theorem 3.3.2]{Koc}.)
Since $\mbC \otimes_\mbZ \Fus_{n, \N}$ is isomorphic to a direct product of copies of $\mbC$ (cf. ~\cite[Proposition 6.1]{Beau}),
the Frobenius algebra defined  for $\mcZ_{n, \N}$ is semisimple.

\SSP
\bde \label{Def989}
We shall refer to $\Fus_{n, \N}$ as the {\bf dormant fusion ring of type $\mr{PGL}_n^{(\N)}$}.
 \ede
\SSP

We shall write
\begin{align}
\mfS_{n, \N} \left(:= \mr{Hom}(\Fus_{n, \N}, \mbC) \right)
\end{align}
for the set of ring homomorphisms $\Fus_{n, \N} \migi \mbC$, i.e., the  set of $\mbC$-rational points of  $\mr{Spec}(\Fus_{n, \N})$.
Also, write
\begin{align}
\mr{Cas}_{n, \N} := \sum_{\lambda \in \Xi_{n, \N}} \lambda * \lambda^\veebar   \left(\in \Fus_{n, \N}\right).
\end{align}

An explicit knowledge of the ring homomorphisms $\Fus_{n, \N} \migi \mbC$ 
 and the element $\mr{Cas}_{n, \N}$ of $\Fus_{n, \N}$ allow us to perform some computation that we need in the ring $\Fus_{n, \N}$, as follows.

\SSP
\bpr\label{Prop16}
(Recall that we have assumed that the morphism  $\Pi_{\rho, g, r}$ for any $(g, r, \rho)$ is finite and  \'{e}tale over an  \'{e}tale  neighborhood of the points in $\overline{\mcM}_{g, r}$ classifying each  totally degenerate curve.)
Let $(g, r)$ be a pair of nonnegative integers  with $2g- 2 +r > 0$ and $\rho := (\rho_i)_i$ an element of $\Xi_{n, \N}^{\times r}$.
Then, the following equality holds:
\begin{align} \label{eeQQ451}
\mr{deg}(\Pi_{\rho, g, r}) =
\sum_{\chi \in \mfS_{n, \N}} \chi (\mr{Cas}_{n, \N})^{g-1} \cdot \prod_{i=1}^r \chi (\rho_i).
\end{align}
Moreover, if $r = 0$ (which implies $g >1$), then this equality means 
\begin{align} \label{eeQQ450}
\mr{deg} (\Pi_{\emptyset, g, 0})= \sum_{\chi \in \mfS_{n, \N}} \chi (\mr{Cas}_{n, \N})^{g-1}.
\end{align}
\epr
\begin{proof}
The assertion follows from ~\cite[Proposition 6.3]{Beau} and the fact that, in our situation here,  the map ``$N_g$" introduced in ~\cite[Proposition 5.9]{Beau} coincides with the degree function $\sum_{i=1}^r \rho_i \mapsto \mr{deg}(\Pi_{(\rho_i), g, r})$.
\end{proof}

\vspace{10mm}
\section{Deformation theory of dormant $\mr{PGL}_2^{(\N)}$-opers} \label{S44} \SSP

This section focuses on the study of dormant $\mr{PGL}_2^{(\N)}$-opers in characteristic $p$ from 
the viewpoint of deformation theory.
By applying  cohomological descriptions of their deformation spaces,
we prove that  the moduli space
$\mcO p^\ZZZ_{2, \N, g, r, \mbF_p}$ 
  is smooth  (cf. Corollary \ref{T34}, (i)) and the projection onto $\overline{\mcM}_{g, r}$ is generically \'{e}tale (cf. Theorem  \ref{c43}).

Let $B$ denote  the Borel subgroup of $\mr{PGL}_2$ consisting of the images of invertible upper triangular $2 \times 2$ matrices via the quotient $\mr{GL}_2 \migisurj \mr{PGL}_2$.
 Also, denote by $\mfg$ and $\mfb$ the Lie algebras of $\mr{PGL}_2$ and $B$, respectively.

Next,  let us fix  a positive integer $\N$ and 
 a pair of nonnegative integers  $(g, r)$  with $2g-2 +r >0$.
Throughout  this section, 
we suppose that $p > 2$, and
the Deligne-Mumford stack
$\mcO p^\ZZZ_{2, \N, g,r, \mbF_p}$ (resp., $\mcO p^\ZZZ_{2, \N, \rho, g,r, \mbF_p}$ for $\rho \in \Xi_{2, \N}^{\times r}$)  will be denoted by
 $\mcO p^\ZZZ_{g,r}$ (resp.,  $\mcO p^\ZZZ_{\rho, g,r}$)
  if there is no fear of confusion.
  By using the bijection \eqref{wpaid},
we identify $\Xi_{2, \N}$ with the set $(\mbZ/p^\N \mbZ)^\times /\{ \pm 1 \}$.

\LSP
\subsection{Adjoint bundle of a dormant $\mr{PGL}_2^{(\N)}$-oper} \label{SS078}

Let 
 $S$ be  a  scheme over $\mbF_p$
 and $\msX := (f: X \migi S, \{ \sigma_i \}_{i=1}^r)$ an $r$-pointed stable curve of genus $g$ over $S$.
 In particular,  we obtain a log curve $X^\mr{log}/S^\mr{log}$.
 For simplicity, we write $\Omega := \Omega_{X^\mr{log}/S^\mr{log}}$, $\mcT := \mcT_{X^\mr{log}/S^\mr{log}}$, and $\mcD^{(\N -1)} :=\mcD^{(\N -1)}_{X^\mr{log}/S^\mr{log}}$.
 
Let $\rho := (\rho_i)_{i=1}^r$ be an element of $((\mbZ/p^\N \mbZ)^\times/\{ \pm 1 \})^{\times r}$ (where $\rho := \emptyset$ if $r = 0$).
For each $i = 1, \cdots, r$, we denote by $\widetilde{\lambda}_i$ the 
unique integer  with  $0 \leq \widetilde{\lambda}_i \leq  \frac{p^\N-1}{2}$  mapped to 
$\rho_i$ via the natural quotient $\mbZ \migisurj (\mbZ/p^\N \mbZ)/\{ \pm 1 \}$. 
Also, denote by $\lambda_i$ the image of $\widetilde{\lambda}_i$ via $\mbZ \migisurj \mbZ/p^\N \mbZ$.

Next, let  $\msE^\spadesuit := (\mcE_B, \STR)$ 
 be a dormant $\mr{PGL}_2^{(\N)}$-oper on $\msX$ of radii $\rho$, and let 
$\pi : \mcE \rightarrow X$ be the $\mr{PGL}_2$-bundle induced from  $\mcE_B$.
In particular, we obtain the $\mcO_X$-algebra $\mcO_\mcE$ corresponding to  $\mcE$ (cf. the discussion following Remark \ref{dGGg5}).
By pulling-back the log structure of $X^\mr{log}$ via $\pi$,
one may obtain a log structure on $\mcE$; we denote the resulting log scheme by $\mcE^\mr{log}$.
The $\mr{PGL}_2$-action on $\mcE$ induces a $\mr{PGL}_2$-action on the direct image $\pi_* (\mcT_{\mcE^\mr{log}/S^\mr{log}})$ of $\mcT_{\mcE^\mr{log}/S^\mr{log}}$, so we obtain 
\begin{align} \label{YY671}
\widetilde{\mcT}_{\mcE^\mr{log}/S^\mr{log}} :=  \pi_* (\mcT_{\mcE^\mr{log}/S^\mr{log}})^{\mr{PGL}_2},
\end{align}
i.e.,  $\widetilde{\mcT}_{\mcE^\mr{log}/S^\mr{log}}$ is the subsheaf of $\mr{PGL}_2$-invariant sections of $\pi_* (\mcT_{\mcE^\mr{log}/S^\mr{log}})$.

Denote by $\mfg_\mcE$ (resp., $\mfb_{\mcE_B}$) the adjoint vector bundle of $\mcE$ (resp., $\mcE_B$), i.e., the vector bundle on $X$ associated to $\mcE$ (resp., $\mcE_B$) via change of structure group by the adjoint representation $\mr{PGL}_2 \migi \mr{GL}(\mfg)$ (resp., $B \migi \mr{GL}(\mfb)$).
We regard  $\mfg_{\mcE}$  as an $\mcO_X$-submodule of $\mcE nd_{\mcO_S} (\mcO_\mcE)$ under the natural identification $\mfg_\mcE = \pi_*(\mcT_{\mcE/X})^{\mr{PGL}_2} \left(= \pi_* (\mcT_{\mcE^\mr{log}/X^\mr{log}})^{\mr{PGL}_2} \right)$.
Differentiating $\pi$ yields a short exact sequence of $\mcO_X$-modules 
\begin{align} \label{YY662}
0 \migi \mfg_\mcE \xrightarrow{\mr{inclusion}} \widetilde{\mcT}_{\mcE^\mr{log}/S^\mr{log}} \xrightarrow{d_\mcE} \mcT \migi 0
\end{align}
(cf. ~\cite[Section 1.2.5]{Wak8}).

Note that $\STR$ induces a $\mcD^{(\N -1)}$-module structure
 \begin{align} \label{YY788}
 \DMO^\mr{ad}_\STR : {^L}\mcD^{(\N -1)} \migi  \mcE nd_{\mcO_S}(\mfg_\mcE)
 \end{align}
on $\mfg_\mcE$; for local sections $\partial \in \mcT$ and $v \in \mfg_\mcE$,
 the equality  $\nabla_\phi^\mr{ad} (\partial) (v) = [\nabla_\phi (\partial), v]$ holds.
In particular, we obtain a $\mcD^{(\N-1)}$-module $(\mfg_\mcE, \DMO_\STR^\mr{ad})$,  which has vanishing    $p^\N$-curvature.

Recall from ~\cite[Section 2.1.2]{Wak8} that
there exists a canonical $3$-step decreasing filtration 
\begin{align}
0 = \mfg_\mcE^2 \subseteq \mfg_\mcE^1 \subseteq \mfg_\mcE^0 \subseteq \mfg_\mcE^{-1} = \mfg_\mcE
\end{align}
on $\mfg_\mcE$ together with natural  identifications 
$\mfg_\mcE^i/\mfg_\mcE^{i+1} = \Omega^{\otimes i}$ ($i=-1, 0, 1$) and $\mfg_\mcE^0 = \mfb_{\mcE_B}$.
Then, we  can prove  the following assertion.

\SSP
\bpr \label{P300}
The collection of data
\begin{align} \label{YY568}
\mfg_\mcE^\heartsuit := (\mfg_\mcE, \DMO_\phi^\mr{ad}, \{ \mfg_\mcE^{j-1} \}_{j=0}^{3})
\end{align}
forms a dormant $\mr{GL}_3^{(\N)}$-oper on $\msX$.
Moreover,  for each $i=1, \cdots, r$,  the exponent of $\mfg_\mcE^\heartsuit$ at $\sigma_i$ coincides with  $[0, 2\lambda_i, - 2\lambda_i]$.
\epr
\begin{proof}
Let us fix a dormant $2^{(\N)}$-theta characteristic $\vartheta := (\varTheta, \DMO_\vartheta)$ of $X^\mr{log}/S^\mr{log}$.
The dormant $\mr{PGL}_2^{(\N)}$-oper $\msE^\spadesuit$ corresponds to a dormant $(\mr{GL}_2^{(\N)}, \vartheta)$-oper $\msF^\heartsuit := (\mcF, \DMO, \{ \mcF^j \}_j)$ on $\msX$ via the isomorphism $\Lambda_{\diamondsuit \Rightarrow \spadesuit}^\ZZZ$ (cf. Theorem \ref{P14}).
Let  $\mcE nd^0 (\mcF)$ denote the sheaf of $\mcO_X$-linear endomorphisms of $\mcF$ with vanishing trace.
This sheaf  has 
a decreasing filtration 
\begin{align}
0 = \mcE nd^0 (\mcF)^2 \subseteq  \mcE nd^0 (\mcF)^1 \subseteq \mcE nd^0 (\mcF)^0 \subseteq \mcE nd^0 (\mcF)^{-1} = \mcE nd^0 (\mcF)
\end{align}
determined by the condition that
$\mcE nd^0 (\mcF)^1$ (resp., $\mcE nd^0 (\mcF)^0$) consists of local sections $h$ with $h (\mcF^1) = 0$ (resp., $h (\mcF^1) \subseteq \mcF^1$).
Also, $\mcE nd^0 (\mcF)$ admits a $\mcD^{(\N -1)}$-module structure $\DMO_{\mcE nd^0}$ induced naturally from $\DMO$.
One may verify that the resulting collection
\begin{align}
(\mcE nd^0 (\mcF), \DMO_{\mcE nd^0}, \{ \mcE nd^0 (\mcF)^{j-1} \}_{j=0}^3)
\end{align}
forms a dormant $\mr{GL}_3^{(\N)}$-oper on $\msX$.
On the other hand, there exists a canonical identification $\mcE nd^0 (\mcF) = \mfg_\mcE$, by which $\{ \mcE nd^0 (\mcF)^{j-1} \}_j$ and $\DMO_{\mcE nd^0}$ correspond to $\{ \mfg_\mcE^{j-1} \}_j$  and $\DMO_\STR^{\mr{ad}}$, respectively.
This implies that $\mfg_\mcE^\heartsuit$ is  a dormant $\mr{GL}_3^{(\N)}$-oper.
This proves the first assertion.

Next,  to prove the second assertion, we fix  $i \in \{ 1, \cdots, r \}$.
Let $[\EX_{i, 1}, \EX_{i, 2}]$ be the exponent of $\msF^\heartsuit$ at $\sigma_i$.
Since $\msF^{\heartsuit \Rightarrow \spadesuit} \cong \msE^\spadesuit$,
the equality $\EX_{i, 1} - \EX_{i, 2} = 2\lambda_i$ holds.
We see that the exponent of $\DMO_{\mcE nd^0}$ at $\sigma_i$ is given by $[0, \EX_{i, 1} -\EX_{i, 2}, \EX_{i, 2} - \EX_{i, 1}] = [0, 2\lambda_i, - 2\lambda_i]$ (cf. Proposition \ref{YY30}, (i)), which completes the proof of the second  assertion.
\end{proof}
\SSP

\SSP
 \bpr \label{Pee4}
Let us fix 
 $a \in \{ 0, 1, \cdots, \N -1\}$, and write $\Omega^{(a)} := \Omega_{X^{(a)\mr{log}}/S^\mr{log}}$.
 Recall  that $(\mfg_\mcE, \DMO_\STR^{\mr{ad}})$ induces  a flat module 
 $(\mfg_\mcE^{[a]}, \DMO_\STR^{\mr{ad}[a]} : \mfg_\mcE^{[a]} \migi \Omega^{(a)} \otimes \mfg_\mcE^{[a]})$ on $X^{(a)\mr{log}}/S^\mr{log}$ with vanishing $p$-curvature (cf. \eqref{J16}, Proposition \ref{UU577}, (i) and (ii)).
  \begin{itemize}
 \item[(i)]
 The
following equalities of $\mcO_S$-modules  hold: 
  \begin{align}
 \mbR^2 f_* (\mcK^\bullet 
 [\DMO^{\mr{ad}[a]}_\phi])
 =f_*(\mr{Ker}(\DMO_\phi^{\mr{ad} [a]})) = \mbR^1 f_* (\mr{Coker}(\DMO^{\mr{ad}[a]}_\phi)) = 0.
 \end{align}
 \item[(ii)]
 $\mbR^1 f_* (\mcK^\bullet [\DMO^{\mr{ad}[a]}_\phi])$
 is a vector bundle  on $S$ of rank $6g-6 + 3r$.
 Also,  
  $\mbR^1f_*(\mr{Ker}(\DMO_\STR^{\mr{ad} [a]}))$ and $f_* (\mr{Coker}(\DMO^{\mr{ad}[a]}_\STR))$) are  
 vector bundles on $S$   of 
 rank  $3g-3+r$ and $3g-3+2r$, respectively.
 In particular,  $\mbR^1 f_* (\mcS ol (\DMO_\STR^{\mr{ad}}))$ is a vector bundle  of rank $3g-3 +r$. 
   \end{itemize}
 \epr
 \begin{proof}
 We shall prove both assertions (i) and (ii) by induction on $a$.
The base step, i.e., 
 the case  of $a = 0$, was already proved in ~\cite[Propositions 6.5, (ii), and  6.18]{Wak8}.
 To consider the induction step, we shall suppose that the assertions with $a$ replaced by $a-1$ ($a \geq 1$) have been proved.
 For simplicity, we write 
 $\mcF^{[a]} := \mfg_\mcE^{[a]}$ and  $\nabla^{[a]} := \DMO_\STR^{\mr{ad}[a]}$.
 
First, we shall consider assertion (i).
Since $\mr{Ker} (\DMO^{[a]}) \subseteq \mr{Ker}(\DMO^{[a-1]})$,
the equality 
\begin{align} \label{YY900}
f_* (\mr{Ker}(\DMO^{[a]})) = 0
\end{align}
holds because of the induction hypothesis $f_* (\mr{Ker}(\DMO^{[a -1]})) = 0$.
On the other hand, let us consider the Hodge to de Rham spectral sequence
\begin{align} \label{dE180}
'E_{1}^{q_1, q_2} := \mbR^{q_2} f_* (\mcK^{q_1} [\DMO^{[a]}]) \Rightarrow \mbR^{q_1 + q_2} f_* (\mcK^\bullet [\DMO^{[a]}])
\end{align}
associated to the complex $\mcK^\bullet [\DMO^{[a]}]$ (cf. ~\cite[Eq.\,(755)]{Wak8}).
Since the relative dimension $\mr{dim}(X/S)$ of $X/S$ is  $1$, the equality  $\mbR^2 f_* (\mcF^{[a]}) = 0$  holds.
It follows that
the sequence
\begin{align} \label{UU1}
\mbR^1f_* (\mcF^{[a]}) \xrightarrow{\mbR^1 f_* (\nabla^{[a]})}\mbR^1 f_* (\Omega^{(a)} \otimes \mcF^{[a]}) \migi \mbR^2 f_* (\mcK^\bullet [\DMO^{[a]}]) \migi 0
\end{align}
 induced by 
\eqref{dE180} is exact.
By the comment following ~\cite[Proposition 1.2.4]{Og},
$\mbR^1 f_* (\Omega^{(a)} \otimes \mcF^{[a]})$ is isomorphic to
$\mbR^1 f_* (\mr{Coker}(\DMO^{[a-1]}))$.
But,  the equality  $\mbR^1 f_* (\mr{Coker}(\DMO^{[a-1]})) =0$ holds  by 
the induction hypothesis, so the exactness of  \eqref{UU1} implies
\begin{align} \label{dE189}
\mbR^2 f_* (\mcK^\bullet [\DMO^{[a]}]) = 0.
\end{align}
Moreover,  consider 
the conjugate spectral sequence
\begin{align} \label{dE200}
''E_2^{q_1, q_2} := \mbR^{q_1} f_*(\mcH^{q_2}(\mcK^\bullet [\DMO^{[a]}])) \Rightarrow \mbR^{q_1 + q_2} f_* (\mcK^\bullet [\DMO^{[a]}])
\end{align}
associated to $\mcK^\bullet [\DMO^{[a]}]$ (cf. ~\cite[Eq.\,(757)]{Wak8}).
Since the equality $\mbR^j f_* (\mr{Ker}(\DMO^{[a]})) = 0$ holds for every $j \geq 2$ because of  $\mr{dim}(X/S) =1$,  
the morphism
$\mbR^2 f_* (\mcK^\bullet [\DMO^{[a]}]) \migi \mbR^1 f_* (\mr{Coker}(\DMO^{[a]}))$
induced by \eqref{dE200} is surjective.
Hence, it follows from \eqref{dE189} that 
\begin{align} \label{dE193}
\mbR^1 f_* (\mr{Coker}(\DMO^{[a]})) = 0.
\end{align}
By \eqref{YY900}, \eqref{dE189}, and \eqref{dE193},
the proof of assertion (i) is completed.

Next, we shall consider assertion (ii).
It follows from Proposition \ref{PPer4}, (i), that  $\mr{Ker}(\nabla^{[a]})$ is a relatively torsion-free sheaf on $X^{(a+1)}$.
Also, since
 $\mr{Coker}(\nabla^{[a]})$ is isomorphic to $\Omega^{(a+1)} \otimes \mr{Ker}(\nabla^{[a]})$ (cf. the comment following ~\cite[Proposition 1.2.4]{Og}), 
 $\mr{Coker}(\nabla^{[a]})$ is a relatively torsion-free sheaf on $X^{(a+1)}$.
By the equality $\mbR^2 f_* (\mr{Ker}(\nabla^{[a]})) = 0$ and  \eqref{dE193},
both $\mbR^1 f_*(\mr{Ker}(\nabla^{[a]}))$ and $f_*(\mr{Coker}(\nabla^{[a]}))$
turn out to be vector bundles (cf. ~\cite[Chapter III, Theorem 12.11, (b)]{Har}).
Moreover, the exactness of  
the sequence 
\begin{align} \label{UU2}
0 \migi  f_* (\Omega^{(a)} \otimes \mcF^{[a]}) \left(= f_* (\mr{Coker}(\DMO^{[a-1]})) \right) &\migi \mbR^1 f_* (\mcK^\bullet [\DMO^{[a]}])\\
 &\migi \mbR^1 f_* (\mcF^{[a]}) \left(= \mbR^1 f_* (\mr{Ker}(\DMO^{[a-1]})) \right)\migi 0  \notag
\end{align}
induced by \eqref{dE180} and  the induction hypothesis together  imply that
$\mbR^1 f_* (\mcK^\bullet [\nabla^{[a]}])$ is a vector bundle.

 In what follows, we shall compute the ranks of the vector bundles under consideration.
 To this end, we may assume that $S = \mr{Spec}(k)$ for an algebraically closed  field $k$  over $\mbF_p$.
Denote by $\msX^{(a)}$ the pointed stable curve obtained as the $a$-th Frobenius twist of $\msX$ over $k$, and by $D^{(a)}$ the reduced effective  divisor on $X^{(a)}$ determined by the union of the marked points of $\msX^{(a)}$.
For each flat module $(\mcV, \DMO_\mcV)$ on $X^{(a)\mr{log}}/k^\mr{log}$,
we write ${^c}\mcV := \mcV (-D^{(a)})$ and write ${^c}\DMO_\mcV$ for the $S^\mr{log}$-connection on ${^c}\mcV \left(\subseteq \mcV \right)$ obtained by restricting  $\DMO_\mcV$.
By  ~\cite[Corollary 6.16]{Wak8} in the case where ``$(\mcF, \nabla)$" is taken to be $(\mcF^{[a]}, \DMO^{[a]})$,
there exists a canonical isomorphism of $k$-vector spaces
 \begin{align}
H^1 (X^{(a)},  \mr{Ker} ({^c}(\DMO^{[a]\vee}))) \isom H^0 (X^{(a)}, \mr{Coker}(\DMO^{[a]}))^\vee.
 \end{align}
 (Note that the result in the cited  reference concerns locally free sheaves,  but as can be easily verified, it holds more generally for relatively  torsion-free sheaves.)
 In particular,
 we have
 \begin{align} \label{ddE1}
 h^1 (\mr{Ker} ({^c}(\DMO^{ [a]\vee}))) = h^0 (\mr{Coker}(\DMO^{[a]})).
 \end{align}
 Next, the Killing form on $\mfg$ induces an isomorphism of $\mcD^{(\N -1)}$-modules
 $(\mfg_\mcE, \DMO_\phi^{\mr{ad}}) \isom (\mfg_\mcE^\vee, (\DMO_\phi^\mr{ad})^\vee)$.
 This isomorphism restricts to an isomorphism of flat modules
\begin{align} \label{ddE2}
({^c}\mcF^{[a]}, {^c}\DMO^{[a]}) \isom 
({^c}(\mcF^\vee)^{[a]}, {^c}(\DMO^\vee)^{[a]}).
\end{align}
This isomorphism gives  the equality $\mr{Ker}({^c}\DMO^{[a]}) = \mr{Ker}({^c}(\DMO^\vee)^{[a]})$, which implies
\begin{align} \label{dE155}
h^1 (\mr{Ker}({^c}\DMO^{[a]})) = h^1 (\mr{Ker}({^c}(\DMO^\vee)^{[a]})).
\end{align}
Hence,
the following sequence of equalities holds:
\begin{align} \label{dE149}
h^0 (\mr{Coker}(\DMO^{[a]})) & \stackrel{\eqref{ddE1}}{=} h^1 (\mr{Ker}({^c}(\DMO^{[a]\vee}))) \\
& = h^1 (\mr{Ker}({^c}(\DMO^\vee)^{[a]})) -R_\IN \notag \\
&\stackrel{\eqref{dE155}}{=} h^1 (\mr{Ker}({^c}\DMO^{[a]})) -R_\IN  \notag \\
& = (h^1 (\mr{Ker}(\DMO^{[a]})) + r + R_\IN) -R_\IN \notag \\
& = h^1 (\mr{Ker}(\DMO^{[a]})) +r \notag 
\end{align}
(cf. \eqref{dE140} for the definition of $R_\IN$), 
where the second and  fourth equalities follow from Lemma \ref{L010} described below.
On the other hand,  
there exists a short exact sequence
\begin{align} \label{ddE6}
0 \migi H^1 (X^{(a)}, \mr{Ker}(\DMO^{[a]})) \migi \mbH^1 (X^{(a)}, \mcK^\bullet [\DMO^{[a]}]) \migi H^0 (X^{(a)}, \mr{Coker}(\DMO^{[a]}))
\migi 0 
\end{align}
induced by  \eqref{dE200}.
 This implies 
\begin{align} \label{eq+}
h^1 (\mr{Ker}(\nabla^{[a]}))  + h^0 (\mr{Coker}(\nabla^{[a]}))  & = \mr{dim}(\mbH^1 (X^{(a)}, \mcK^\bullet [\nabla^{[a]}]))  \\
&= h^1 (\mr{Ker}(\nabla^{[a-1]})) + h^0 (\mr{Coker}(\nabla^{[a-1]})) \notag  \\
& = (3g-3+r) + (3g-3+2r) \notag \\
& =  6g-6 + 3r, \notag
\end{align} 
where the second equality follows from the exactness of \eqref{UU2} and the third equality follows from the induction hypothesis.
Consequently,     it follows from 
\eqref{dE149} and \eqref{eq+} that 
\begin{align}
h^1(\mr{Ker}(\nabla^{[a]})) = 3g-3+r, \ \ \ h^0(\mr{Coker}(\nabla^{[a]})) = 3g-3+2r.
\end{align}
This completes the proof of the proposition.
 \end{proof}
\SSP

The following lemma was applied in the proof of the above proposition.

\SSP
\ble \label{L010}
We shall  keep the notation in the proof of the above proposition.
Let $\IN \in \{1, \cdots, \N -1\}$, and 
suppose that $S = \mr{Spec}(k)$ for an algebraically closed field $k$ over $\mbF_p$.
Also, we shall set 
\begin{align} \label{dE140}
R_\IN := \sharp \left\{ i \, \Big| \, 
1 \leq i \leq r,   \ (2\lambda_{i})_{[\IN]} \in \{0, p-1 \} \left(\subseteq \mbF_p \right)
\right\}
   \end{align}
(cf. Section \ref{SS182} for the definition of $(-)_{[\IN]}$).
Then, 
the following two equalities hold:
\begin{align} \label{dE141}
h^1 (\mr{Ker}({^c}\DMO^{[\IN]})  &= h^1 (\mr{Ker}(\DMO^{[\IN]})) + r + R_\IN,   \\
h^1 (\mr{Ker}({^c}(\DMO^\vee)^{[\IN]})) &= h^1 (\mr{Ker}({^c}(\DMO^{[\IN]\vee})))   +  R_\IN. 
\notag
\end{align}
\ele 
\begin{proof}
To begin with, we prepare some notations.
Let $U_\oslash$ and $\mcD_\oslash^{(\N -1)}$ be as in Section \ref{SS1071}.
Let $\msV := (\mcV, \DMO)$ be a  $\mcD^{(\N-1)}$-module (resp., a $\mcD_\oslash^{(\N-1)}$-module) such that $\psi (\DMO) = 0$ and $\mcV$ is a vector bundle of rank $n >0$.
In the non-resp'd portion, $({^c}\mcV^{[a]}, {^c}\DMO^{[a]})$ is defined as in the proof of the above proposition.
Also,  
in the resp'd portion, we set ${^c}\mcV^{[a]} := t^{p^a} \cdot \mcV^{[a]} \left(\subseteq \mcV^{[a]} \right)$, which has an $S^\mr{log}$-connection  ${^c}\DMO^{[a]}$ obtained by restricting $\DMO^{[a]}$.
Then, the  inclusion 
${^c}\mcV^{[\IN]}  \migiincl \mcV^{[\IN]}$ restricts to an inclusion
\begin{align} \label{dE123}
\alpha_{\msV}^{[\IN]} : \mr{Ker} ({^c}\DMO^{[\IN]}) \migiincl \mr{Ker} (\DMO^{[\IN]}).
\end{align}
On the other hand, the natural morphism
 ${^c}(\mcV^\vee)^{[a]}
 \migi 
{^c}(\mcV^{[\IN]\vee})$
restricts to a morphism of $\mcO_{X^{(\IN+1)}}$-modules
\begin{align} \label{UU7}
\beta_\msV^{[\IN]} : \mr{Ker}({^c} (\DMO^\vee)^{[\IN]}) \migi \mr{Ker} ({^c}(\DMO^{[\IN]\vee})).
\end{align}
In what follows, we shall  examine  the morphisms $\alpha_\msV^{[\IN]}$ and $\beta_\msV^{[\IN]}$ in the case where   $\msV$ is taken to be the $\mcD_\oslash^{(\N-1)}$-module  $\msO_{\oslash, \EX}^{(\N -1)} := (\mcO_\oslash, \DMO_{\oslash, \EX}^{(\N -1)})$ for 
 $\EX \in (\mbZ/p^\N \mbZ)^\times$ (cf. \eqref{UU3}).
For simplicity, write $\msO_{d} := \msO_{\oslash, \EX}^{(\N -1)}$ and $\DMO_{\EX}:= \DMO_{\oslash, \EX}^{(\N -1)}$.
First, 
 a straightforward calculation  shows 
\begin{align} \label{UUjk6}
\mr{Ker} ({^c}\DMO^{[\IN]}_{\EX}) = \begin{cases} t^{\widetilde{\EX}_{[0, \IN]}} \cdot \mcO^{(\IN+1)}_{\oslash}& \text{if $\EX_{[\IN]} \neq 0$}; 
\\
 t^{\widetilde{d}_{[0, \IN]}+p^{\IN +1}} \cdot \mcO^{(\IN+1)}_{\oslash}
  & \text{if $\EX_{[\IN]} = 0$}. \end{cases}
\end{align}
Hence, it follows from Proposition \ref{L093} that $\alpha_{\msO_\EX}^{[\IN]}$ is injective and its cokernel satisfies
\begin{align} \label{dE128}
\mr{length}_{\mcO_\oslash^{(a+1)}}(\mr{Coker}(\alpha^{[\IN]}_{\msO_\EX})) = \begin{cases} 0 & \text{if $\EX_{[\IN]} \neq 0$};\\ 1 & \text{if $\EX_{[\IN]} =0$}.  \end{cases}
\end{align}
Also, by putting
 $c := -\EX$, we see (from $\EX \in (\mbZ/p^\N \mbZ)^\times$) 
that 
\begin{align} \label{UU6}
\mr{Ker} ({^c}(\DMO_{\EX}^{\vee})^{[\IN]}) \left(= \mr{Ker}({^c} \DMO_{c}^{[\IN]}) \right) = \begin{cases} t^{\widetilde{c}_{[0, \IN]}} \cdot \mcO_{\oslash}^{(\IN+1)} &  \text{if
$c_{[\IN]} \neq 0 \left(\Leftrightarrow \EX_{[\IN]} \neq p-1 \right)$};
\\
t^{\widetilde{c}_{[0, \IN]}+p^{\IN+1}} \cdot \mcO_\oslash^{(\IN+1)} &
\text{if
$c_{[\IN]} = 0 \left(\Leftrightarrow \EX_{[\IN]} = p-1 \right)$}.
\end{cases}
\end{align}
and that
\begin{align} \label{UU5}
 \mr{Ker}({^c}(\DMO_{\EX}^{[\IN]})^{\vee}) = 
t^{-\widetilde{d}_{[0, \IN]}+p^{\IN+1}} \cdot \mcO_\oslash^{(\IN+1)} \left(= t^{\widetilde{c}_{[0, \IN]}} \cdot \mcO_\oslash^{(\IN+1)} \right).
\end{align}
Hence, 
$\beta_{\msO_\EX}^{[\IN]}$ is injective and its cokernel satisfies
\begin{align} \label{dE124}
\mr{length}_{\mcO_\oslash^{(a+1)}} (\mr{Coker}(\beta^{[\IN]}_{\msO_\EX})) =  \begin{cases} 1 & \text{if   $\EX_{[\IN]} = p-1$}; \\
0 & \text{if   $\EX_{[\IN]}\neq  p-1$}.
\end{cases}
 \end{align}
 The same  arguments with $d$ replaced by $c$ give
 \begin{align} \label{dE125}
\mr{length}_{\mcO_\oslash^{(a+1)}} (\mr{Coker}(\alpha^{[\IN]}_{\msO_{-\EX}})) &= \begin{cases} 0 & \text{if $\EX_{[\IN]} \neq p-1$};\\ 1 & \text{if $\EX_{[\IN]} =p-1$},  \end{cases}  \\
\mr{length}_{\mcO_\oslash^{(a+1)}} (\mr{Coker}(\beta^{[\IN]}_{\msO_{-\EX}})) &=  \begin{cases} 1 & \text{if   $\EX_{[\IN]} =  0$}; \\
0 & \text{if   $\EX_{[\IN]} \neq  0$}. \notag
\end{cases}
 \end{align}

Now,  let us go back to our situation.
We shall set $(\mcF, \DMO) := (\mfg_\mcE, \DMO_\STR^{\mr{ad}})$.
Both $\alpha^{[a]}_{(\mcF, \DMO)}$ and $\beta^{[a]}_{(\mcF, \DMO)}$  become  isomorphisms
when restricted to $X \setminus \bigcup_{i=1}^r \mr{Im}(\sigma_i)$.
Given each $\EX \in \mbZ/p^\N \mbZ$, we shall set
\begin{align}
\mcQ_{\EX}^{[\IN]} := \begin{cases} \mcO_\oslash^{(\IN+1)} /(t^{p^{\IN+1}}) & \text{if $\EX_{[\IN]} \in \{0, p-1 \}$};  \\ 0  & \text{if $\EX_{[a]} \notin \{0, p-1 \}$}.\end{cases}
\end{align}
For each $i =1, \cdots, r$, denote by $\widehat{U}_i$ the formal neighborhood of $\mr{Im}(\sigma_i)$ in $X$, and by 
  $\mr{incl}_{i} : \widehat{U}_i  \migi X$ the natural morphism.
It follows from Proposition \ref{P300} that
the restriction $(\mcF|_{\widehat{U}_i}, \DMO |_{\widehat{U}_i})$ of $(\mcF, \DMO)$ to $\widehat{U}_i$
  is isomorphic to
 $(\mcO_\oslash^{\oplus 3}, \DMO_{\oslash, 0}^{(\N-1)}\oplus \DMO_{\oslash, 2\lambda_i}^{(\N-1)} \oplus \DMO_{\oslash, -2\lambda_i}^{(\N-1)})$ via a fixed identification $\widehat{U}_i = U_\oslash$.
By \eqref{dE128} and the first equality  of \eqref{dE125},
 the cokernel  of $\alpha_{(\mcF, \DMO)}^{[\IN]}$ restricted to $\widehat{U}_i$ is isomorphic 
 to $\mcQ^{[\IN]}_{2\lambda_i} \oplus \mcO_\oslash^{(\IN+1)} / (t^{p^{\IN+1}})$.
This implies that  $\alpha^{[\IN]}_{(\mcF, \DMO)}$ fits into the short exact sequence 
\begin{align} \label{dE132}
0 \migi \mr{Ker}({^c}\DMO^{[\IN]}) \xrightarrow{\alpha^{[\IN]}_{(\mcF, \DMO)}} \mr{Ker}(\DMO^{[\IN]}) \migi \bigoplus_{i=1}^r \mr{incl}_{i*} (\mcQ^{[\IN]}_{2\lambda_i}\oplus \mcO^{(\IN+1)}_{\oslash}/(t^{p^{\IN+1}})) \migi 0.
\end{align}
Here, we shall write  $H^j (-) := H^j (X^{(\IN+1)}, -)$ ($j \geq 0$) for simplicity.
Since $H^0 (\mr{Ker}(\DMO^{[\IN]})) 
 =  0$ (cf. \eqref{YY900}),
 the sequence \eqref{dE132} yields  a short exact sequence of $k$-vector spaces
\begin{align}
0 \migi  \bigoplus_{i=1}^r H^0 (\mr{incl}_{i*} (\mcQ^{[\IN]}_{2\lambda_i}\oplus \mcO^{(\IN+1)}_\oslash/(t^{p^{\IN+1}})))
\migi  H^1 (\mr{Ker}({^c}\DMO^{[\IN]}))
\migi H^1 (\mr{Ker}(\DMO^{[\IN]}))
\migi 0.
\end{align}
This sequence implies 
\begin{align}
h^1 (\mr{Ker}({^c}\DMO^{[\IN]})) &= h^1 (\mr{Ker}(\DMO^{[\IN]})) + \sum_{i=1}^r\mr{dim} (H^0(\mr{incl}_{i*} (\mcQ^{[\IN]}_{2\lambda_i} \oplus \mcO^{(\IN+1)}_\oslash /(t^{p^{\IN+1}})))) \\
& = h^1 (\mr{Ker}(\DMO^{[\IN]})) + \sum_{i=1}^r (\mr{length}_{\mcO_\oslash^{(\IN +1)}} (\mcQ^{[\IN]}_{2\lambda_i})+1) \notag \\
& =  h^1 (\mr{Ker}(\DMO^{[\IN]}))   +R_a +r. \notag
\end{align}
Thus, we have proved the first equality in \eqref{dE141}.
Moreover,  it follows from \eqref{dE124} and the second equality  of \eqref{dE125} that we have the following short exact sequence:
\begin{align} \label{dE130}
0 \migi \mr{Ker}({^c}\DMO^{\vee [\IN]}) \xrightarrow{\beta^{[\IN]}_{(\mcF, \DMO)}} \mr{Ker}({^c}\DMO^{[\IN]\vee})  \migi \bigoplus_{i=1}^r \mr{incl}_{i*} (\mcQ^{[\IN]}_{2\lambda_i} )
\migi  0.
\end{align}
Similarly  to the above argument, 
 \eqref{dE130} implies  the equality
\begin{align}
h^1 (\mr{Ker}({^c}(\DMO^\vee)^{[\IN]})) = h^1 (\mr{Ker}({^c}(\DMO^{[\IN]\vee}))) + R_{\IN},
\end{align}
which is precisely the second equality in \eqref{dE141}.
Thus, we have finished the proof of this lemma.
\end{proof}

\LSP
\subsection{Higher-level connections on  the adjoint bundle} \label{SS07ff8}

Let $\N'$ be a positive integer with $\N' \leq \N$.
By reducing the level of $\phi$ to $\N' -1$, we obtain an $(\N' -1)$-PD stratification 
$\phi^{(\N' -1)}$ on $\mcE$ (hence $\phi^{(\N -1)} = \phi$).

 Denote by
\begin{align} \label{e603}
\nabla^{\mr{ad} (\N' -1)}_\phi: \mfg_\mcE \rightarrow\mfg_\mcE  \otimes \mcP_{(\N'-1)}^{p^{\N' -1}} 
\end{align}
the  $f^{-1}(\mcO_S)$-linear morphism obtained  as the difference $d^2 - d^1$,
where 
\begin{itemize}
\item
$d^1$ denotes the morphism   $\mfg_\mcE \left(=    \mfg_\mcE \otimes \mcO_X\right)\hookrightarrow   \mfg_\mcE \otimes \mcP_{(\N'-1)}^{p^{\N'-1}}$ induced by $\mr{pr}_1^{p^{\N' -1}}$ (cf. Section \ref{SS040}), and
\item
$d^2$ denotes the morphism  $\mfg_\mcE \rightarrow\mfg_\mcE \otimes \mcP_{(\N'-1)}^{p^{\N' -1}} $
corresponding to 
the composite 
$\mcD^{(\N' -1)}_{\leq p^{\N'-1}} \otimes \mfg_\mcE \rightarrow \mfg_\mcE$
of the inclusion 
$\mcD^{(\N' -1)}_{\leq p^{\N'-1}} \otimes \mfg_\mcE \hookrightarrow \mcD^{(\N' -1)} \otimes \mfg_\mcE$ and the 
morphism 
$\mcD^{(\N' -1)} \otimes \mfg_\mcE \rightarrow \mfg_\mcE$ induced by reducing the level of 
$\nabla_\phi^\mr{ad}$ to $\N'-1$ (or equivalently, obtained from $\phi^{(\N' -1)}$ via change of structure group by the adjoint representation $\mr{PGL}_2 \rightarrow \mr{GL}(\mfg)$).
\end{itemize}
The image of this morphism lies in  $\mfg_\mcE \otimes \overline{\mcP}_{(\N' -1)}^{p^{\N'-1}}$, where 
\begin{align}
\overline{\mcP}_{(\N' -1)}^{p^{\N'-1}} := \mr{Ker}\left(\mcP_{(\N' -1)}^{p^{\N'-1}} \twoheadrightarrow  \mcP_{(\N' -1)}^{0}\right) \left(\cong (\mcD^{(\N' -1)}_{\leq p^{\N' -1}}/\mcD^{(\N' -1)}_{\leq 0})^\vee \right).
\end{align}
Thus, $\nabla_\phi^{\mr{ad} (\N' -1)}$ can be regarded as a morphism $\mfg_\mcE \rightarrow\mfg_\mcE \otimes  \overline{\mcP}_{(\N'-1)}^{p^{\N' -1}}$ 
(cf. ~\cite[Section 1.1]{LSQ2} for the non-logarithmic version of such morphisms).
When $\N' = \N = 1$,
this morphism coincides with  the $S^\mr{log}$-connection on $\mfg_\mcE$ introduced in ~\cite[Eq.\,(745)]{Wak8}.


Observe that the image of the endomorphism of $\widetilde{\mcT}_{\mcE^\mr{log}/S^\mr{log}}$ given by  $\mr{id}_{\widetilde{\mcT}_{\mcE^\mr{log}/S^\mr{log}}} - \nabla_\phi \circ d_\mcE$ is contained in $\mfg_\mcE$.
Hence, one can define an $f^{-1}(\mcO_S)$-linear morphism
\begin{align} \label{UU10}
\widetilde{\nabla}^{\mr{ad}(\N' -1)}_\STR :   \widetilde{\mcT}_{\mcE^\mr{log}/S^\mr{log}} \migi  \mfg_\mcE \otimes \overline{\mcP}_{(\N' -1)}^{p^{\N'-1}}
\end{align}
defined as the composite $\widetilde{\nabla}^{\mr{ad}(\N' -1)}_\STR := \nabla_\phi^{\mr{ad} (\N '-1)}  \circ (\mr{id}_{\widetilde{\mcT}_{\mcE^\mr{log}/S^\mr{log}}} - \nabla_\phi \circ d_\mcE)$.
For any $\partial \in \mcT$ and $v \in \widetilde{\mcT}_{\mcE^\mr{log}/S^\mr{log}}$, the following equality holds:
\begin{align}
\langle \partial, \widetilde{\nabla}^{\mr{ad}(\N' -1)}_\STR (v) \rangle = [\phi^{(\N'-1)\natural \natural} (\partial), v] - \phi^{(\N'-1)\natural \natural} ([\partial, d_\mcE (v)]),
\end{align}
where $\phi^{(\N' -1)}$ denotes the $(\N'-1)$-PD stratification induced by $\phi$, and 
$\langle -, - \rangle$ denotes the pairing $\mcD^{(\N' -1)}_{\leq p^{\N'-1}} \times (\mfg_\mcE \otimes \mcP_{(\N' -1)}^{p^{\N'-1}}) \migi \mfg_\mcE$ arising from the natural pairing $\mcD^{(\N' -1)}_{\leq p^{\N'-1}} \times  \mcP_{(\N' -1)}^{p^{\N'-1}}\migi \mcO_X$.
(Just as in the case of $\mfg_\mcE$,
the sheaf  $\widetilde{\mcT}_{\mcE^\mr{log}/S^\mr{log}}$ is regarded  as an $\mcO_X$-submodules of $\mcE nd_{\mcO_S} (\mcO_\mcE)$.)
Since the endomorphism $\mr{id} - \nabla_\phi \circ d_{\mcE}$ restricts to the identity morphism of $\mfg_\mcE$,
the restriction of $\widetilde{\nabla}^{\mr{ad}(\N'-1)}_\STR$ to $\mfg_\mcE \left(\subseteq \widetilde{\mcT}_{\mcE^\mr{log}/S^\mr{log}} \right)$ coincides with $\nabla^{\mr{ad}(\N' -1)}_\STR$.
Also, we set 
\begin{align} \label{UU12}
\widetilde{\nabla}_{\STR, B}^{\mr{ad}(\N'-1)} &: \widetilde{\mcT}_{\mcE_B^\mr{log}/S^\mr{log}} \xrightarrow{\mr{inclusion}} \widetilde{\mcT}_{\mcE^\mr{log}/S^\mr{log}} \xrightarrow{\widetilde{\nabla}^{\mr{ad}(\N'-1)}_\STR} \mfg_\mcE \otimes \overline{\mcP}_{(\N' -1)}^{p^{\N'-1}},
\end{align}
where $\widetilde{\mcT}_{\mcE_B^\mr{log}/S^\mr{log}}$ denotes  the $\mcO_X$-module obtained from the $B$-bundle $\mcE_B$ in the same manner as $\widetilde{\mcT}_{\mcE^\mr{log}/S^\mr{log}}$.

We shall write
\begin{align} \label{WW200}
\mfg_{\mcE, \phi^{(\N'-1)}}  \ \left(\text{resp.,} \ \widetilde{\mcT}_{\mcE^\mr{log}/S^\mr{log}, \phi^{(\N' -1)}}\right)
\end{align}
for the subsheaf of  
$\mfg_{\mcE}$ (resp., $\widetilde{\mcT}_{\mcE^\mr{log}/S^\mr{log}}$)  consisting of local sections $v$ satisfying  $[\STR^{(\N' -1)\natural \natural} (D),  v] = 0$ (resp., $[\STR^{(\N'-1)\natural \natural} (D), v] - \STR^{\natural \natural} ([D, d_\mcE (v)]) =0 $)  for any $D\in \mcD^{(\N' -1)}$.
In particular,  the equality  $\mfg_{\mcE, \STR} \left(= \mfg_{\mcE, \phi^{(\N-1)}} \right) = \mcS ol (\DMO_\STR^{\mr{ad}})$ holds.
Also, we set
\begin{align} \label{Equu12}
\widetilde{\mcT}_{\mcE^\mr{log}_B/S^\mr{log}, \phi^{(\N' -1)}}:=  \widetilde{\mcT}_{\mcE^\mr{log}_B/S^\mr{log}} \cap  \widetilde{\mcT}_{\mcE^\mr{log}/S^\mr{log}, \phi^{(\N' -1)}}.
\end{align}
The sheaves  $\mfg_{\mcE, \STR^{(\N'-1)}}$, $\widetilde{\mcT}_{\mcE^\mr{log}/S^\mr{log}, \STR^{(\N'-1)}}$, and $\widetilde{\mcT}_{\mcE^\mr{log}_B/S^\mr{log}, \phi^{(\N' -1)}}$  can be regarded as  $\mcO_{X^{(\N')}}$-modules via the underlying homeomorphism of 
  $F^{(\N')}_{X/S} : X \migi X^{(\N')}$.
Since $\mcD^{(\N' -1)}$ is generated  by sections of  $\mcD_{\leq p^{\N' -1}}^{(\N' -1)}$ as an $\mcO_X$-algebra,  we have 
\begin{align} \label{Equu101}
\mfg_{\mcE, \phi^{(\N' -1)}} &= \mr{Ker} (\nabla_\phi^{\mr{ad}(\N' -1)}), 
\\
\widetilde{\mcT}_{\mcE^\mr{log}/S^\mr{log}, \phi^{(\N' -1)}} &= \mr{Ker} (\widetilde{\nabla}_\phi^{\mr{ad}(\N' -1)}), \notag \\
\widetilde{\mcT}_{\mcE_B^\mr{log}/S^\mr{log}, \phi^{(\N' -1)}} &= \mr{Ker} (\widetilde{\nabla}_{\phi, B}^{\mr{ad}(\N' -1)}). \notag
\end{align}

Suppose further  that the $\mr{PGL}_2$-oper induced from  $\msE^\spadesuit$ is normal, in the sense of ~\cite[Definition 4.53]{Wak8}.
(According to ~\cite[Proposition 4.55]{Wak8}, any  $\mr{PGL}_2$-oper is isomorphic to a normal  one.)
Consider the composite surjection
\begin{align} \label{dE33}
d'_{\mcE} : \widetilde{\mcT}_{\mcE^\mr{log}/S^\mr{log}} \migisurj \widetilde{\mcT}_{\mcE^\mr{log}/S^\mr{log}}/\widetilde{\mcT}_{\mcE_B^\mr{log}/S^\mr{log}}   \isom \mfg_\mcE/\mfb_{\mcE_B}  \left(=\mfg_{\mcE}^{-1}/\mfg_{\mcE}^0 \right) \isom \mcT,
\end{align}
where the second and  third  arrows are the isomorphisms defined in 
~\cite[Eq.\,(163)]{Wak8}  and ~\cite[Eq.\,(617)]{Wak8}, respectively.
Using this morphism, we obtain an $\mcO_X$-linear endomorphism of  $\widetilde{\mcT}_{\mcE^\mr{log}/S^\mr{log}}$ defined as
\begin{align} \label{TTT77}
\chi := \mr{id}_{\widetilde{\mcT}_{\mcE^\mr{log}/S^\mr{log}}} -  \nabla_\STR \circ d_\mcE - \nabla_\STR \circ d'_{\mcE} : \widetilde{\mcT}_{\mcE^\mr{log}/S^\mr{log}} \migi \widetilde{\mcT}_{\mcE^\mr{log}/S^\mr{log}},
\end{align}
where
$\DMO_\STR$ denotes 
 the $S^\mr{log}$-connection $\mcT \migi \widetilde{\mcT}_{\mcE^\mr{log}/S^\mr{log}}$ on $\mcE$ obtained by reducing the level of $\STR$  to $0$ (cf.  ~\cite[Definition 1.17]{Wak8} and \eqref{Efjj32}).

\SSP
\ble \label{L22}
The automorphism $\chi$ of $\widetilde{\mcT}_{\mcE^\mr{log}/S^\mr{log}}$ is involutive (i.e., $\chi \circ \chi = \mr{id}$), and 
the following equalities hold:
\begin{align} \label{Equu83}
  d_\mcE \circ \chi = -d'_\mcE, \hspace{10mm}
 d'_\mcE \circ \chi = -d_\mcE, \hspace{10mm}
\widetilde{\nabla}_\phi^{\mr{ad} (\N' -1)} \circ  \chi =  \widetilde{\nabla}_\phi^{\mr{ad} (\N' -1)}
\end{align}
\ele
\begin{proof}
We set $\mr{pr}_\phi := \mr{id} - \nabla_\phi \circ d_\mcE$.
Since  $\mr{pr}_\phi^2 = \mr{pr}_\phi$ and $\mr{pr}_\phi \circ \nabla_\phi =0$,  the equality  $\mr{pr}_\phi \circ \chi = \mr{pr}_\phi$ holds.
Also, since we have assumed that the $\mr{PGL}_2$-oper induced from  $\msE^\spadesuit$ is normal, 
the equality $d'_\mcE \circ \nabla_\STR = \mr{id}_{\mcT}$ holds (cf. ~\cite[Example 2.7 and Remark 2.16]{Wak8}).
Hence, the following sequence of equalities holds:
\begin{align} \label{NN10}
\chi \circ \chi & 
= (\mr{pr}_\phi - \nabla_\phi \circ d'_\mcE) \circ \chi \\
& = \mr{pr}_\phi \circ \chi - \nabla_\phi \circ d'_\mcE \circ (\mr{id} - \nabla_\STR\circ d_\mcE - \nabla_\STR\circ d'_\mcE )\\
& =  \mr{pr}_\phi - \nabla_\phi \circ d'_\mcE + \nabla_\phi \circ (d'_\mcE \circ \nabla_\phi) \circ d_\mcE + \nabla_\phi \circ (d'_\mcE \circ \nabla_\phi) \circ d'_\mcE  \\
& = \mr{id} - \nabla_\phi \circ d_\mcE - \nabla_\phi \circ d'_\mcE + \nabla_\phi \circ d_\mcE + \nabla_\phi \circ d'_\mcE 
\\
& = \mr{id}.
\end{align}
Moreover,  we obtain the following equalities:
\begin{align}
d_\mcE \circ \chi &= d_\mcE \circ (\mr{id} - \nabla_\STR\circ d_\mcE - \nabla_\STR\circ d'_\mcE) \\
& = d_\mcE - (d_\mcE \circ \nabla_\STR) \circ d_\mcE - (d_\mcE \circ  \nabla_\STR) \circ d'_\mcE \notag  \\
& =   d_\mcE -d_\mcE - d'_\mcE \notag \\
& = -d'_\mcE. \notag
\end{align}
A similar  calculation shows  $d'_\mcE \circ \chi = -d_\mcE$.
Finally, 
we have
\begin{align}
\widetilde{\nabla}^{\mr{ad} (\N' -1)}_\phi \circ \chi &= \nabla_\phi^{\mr{ad}(\N'-1)} \circ \mr{pr}_\phi \circ \chi \\
& = \nabla_\phi^{\mr{ad} (\N'-1)} \circ \mr{pr}_\phi \\
& = \widetilde{\nabla}_\phi^{\mr{ad} (\N' -1)}.
\end{align}
This completes the proof.
\end{proof}
\SSP

Moreover, we can  prove the following assertion.

\SSP
\bpr \label{P44WW}
The automorphism $\chi$ restricts to an automorphism $\chi_\STR : \widetilde{\mcT}_{\mcE^\mr{log}/S^\mr{log}, \STR^{(\N' -1)}} \isom \widetilde{\mcT}_{\mcE^\mr{log}/S^\mr{log}, \STR^{(\N' -1)}}$ of $\widetilde{\mcT}_{\mcE^\mr{log}/S^\mr{log}, \STR}$, and restricts to  an isomorphism  $\chi_B : \mfg_\mcE \isom \widetilde{\mcT}_{\mcE_B^\mr{log}/S^\mr{log}}$.
That is to say, we obtain the following commutative square diagrams:
\begin{align} \label{dE445}
\vcenter{\xymatrix@C=56pt@R=36pt{
\widetilde{\mcT}_{\mcE^\mr{log}/S^\mr{log}, \STR^{(\N' -1)}} \ar[r]_-{\sim}^-{\chi_\STR} \ar[d]_-{\mr{inclusion}} & \widetilde{\mcT}_{\mcE^\mr{log}/S^\mr{log}, \STR^{(\N' -1)}} \ar[d]^{\mr{inclusion}}
\\
\widetilde{\mcT}_{\mcE^\mr{log}/S^\mr{log}} \ar[r]^-{\sim}_-{\chi} &
\widetilde{\mcT}_{\mcE^\mr{log}/S^\mr{log}},
}}
\hspace{10mm}
\vcenter{\xymatrix@C=56pt@R=36pt{
\mfg_\mcE \ar[r]^-{\chi_B}_-{\sim} \ar[d]_-{\mr{inclusion}} & \widetilde{\mcT}_{\mcE_B^\mr{log}/S^\mr{log}} \ar[d]^-{\mr{inclusion}} 
\\
\widetilde{\mcT}_{\mcE^\mr{log}/S^\mr{log}} \ar[r]^-{\sim}_-{\chi} &
\widetilde{\mcT}_{\mcE^\mr{log}/S^\mr{log}}.
}}
\end{align}
\epr
\begin{proof}
As established  in \eqref{Equu101},  we have $\widetilde{\mcT}_{\mcE^\mr{log}/S^\mr{log}, \phi^{(\N'-1)}} = \mr{Ker} (\widetilde{\nabla}_\phi^{\mr{ad}(\N'-1)})$.
Hence, 
the third equality $\widetilde{\nabla}_\phi^{\mr{ad}(\N' -1)}  \circ \chi =  \widetilde{\nabla}_\phi^{\mr{ad}(\N' -1)}$ in  \eqref{Equu83} implies that
$\chi$ restricts to an automorphism of $\widetilde{\mcT}_{\mcE^\mr{log}/S^\mr{log}, \STR^{(\N' -1)}}$.

Moreover,
since $\mr{Ker}(d_\mcE) = \mfg_\mcE$ and $\mr{Ker}(d'_\mcE) = \widetilde{\mcT}_{\mcE^\mr{log}_B/S^\mr{log}}$,
 the second equality $d'_\mcE \circ \chi = - d_\mcE$ in \eqref{Equu83} implies that
 $\chi$ restricts to an isomorphism $\mfg_\mcE \isom \widetilde{\mcT}_{\mcE_B^\mr{log}/S^\mr{log}}$.
 This completes the proof of the assertion.
\end{proof}
\SSP

By taking the intersection of the isomorphisms $\chi_\phi$,  $\chi_B$ resulting from the above proposition, we obtain an isomorphism
\begin{align}
\chi_{\phi \cap B} 
:
\mfg_{\mcE, \STR^{(\N'-1)}} \left(=  \widetilde{\mcT}_{\mcE^\mr{log}/S^\mr{log}, \phi^{(\N' -1)}} \cap \mfg_\mcE \right)\xrightarrow{\sim} \widetilde{\mcT}_{\mcE^\mr{log}_B/S^\mr{log}, \phi^{(\N' -1)}} \left(= \widetilde{\mcT}_{\mcE^\mr{log}/S^\mr{log}, \phi^{(\N' -1)}} \cap \widetilde{\mcT}_{\mcE^\mr{log}_B/S^\mr{log}} \right).
\end{align}
This morphism makes the following diagram commute:
\begin{align} \label{Equu106}
\vcenter{\xymatrix@C=46pt@R=36pt{
\mfg_{\mcE, \STR^{(\N'-1)}} 
 \ar[rr]_-{\sim}^-{\chi_{\phi \cap B}} 
  \ar[rd]_-{\nabla_{\phi}^{\mr{ad}(\N'-1)}}&&  \widetilde{\mcT}_{\mcE^\mr{log}_B/S^\mr{log}, \phi^{(\N' -1)}}\ar[ld]^-{\widetilde{\nabla}_{\phi, B}^{\mr{ad}(\N' -1)}}
\\
& \mfg_\mcE \otimes \overline{\mcP}_{(\N' -1)}^{p^{\N'-1}}.&
}}
\end{align}
Moreover,  $\chi_{\phi \cap B}$ yields, via the functor $\mbR^1 f_*(-)$,
an isomorphism  of $\mcO_S$-modules 
\begin{align} \label{Equu102}
\mbR^1 f_* (\chi_{\phi \cap B}) : \mbR^1 f_* (\mfg_{\mcE, \STR^{(\N'-1)}}) \xrightarrow{\sim} \mbR^1 f_*(\widetilde{\mcT}_{\mcE^\mr{log}_B/S^\mr{log}, \phi^{(\N' -1)}}).
\end{align}

Applying the functor $\mbR^1 f_* (-)$ to the composites
\begin{align} \label{NN29}
\widetilde{\mcT}_{\mcE^\mr{log}_B/S^\mr{log}, \phi} 
&\xrightarrow{\mr{inclusion}}
\left(\widetilde{\mcT}_{\mcE^\mr{log}_B/S^\mr{log}}  \migiincl  \right)\widetilde{\mcT}_{\mcE^\mr{log}/S^\mr{log}} \xrightarrow{d_\mcE} \mcT,  \\
 \mfg_{\mcE, \STR} &\xrightarrow{\mr{inclusion}}
\left(\mfg_\mcE \migiincl  \right)\widetilde{\mcT}_{\mcE^\mr{log}/S^\mr{log}} \xrightarrow{d'_\mcE}
\mcT, \notag
\end{align}
we obtain  $\mcO_S$-linear morphisms
\begin{align} \label{Equu818}
D_{\msE^\spadesuit} : \mbR^1 f_* (\widetilde{\mcT}_{\mcE^\mr{log}_B/S^\mr{log}, \phi} ) \migi \mbR^1 f_* (\mcT) \ \ \ \text{and} \ \ \ 
D'_{\msE^\spadesuit} : \mbR^1 f_* ( \mfg_{\mcE, \STR}) \migi \mbR^1 f_* (\mcT),
\end{align}
respectively.
By the first equality displayed in \eqref{Equu83},
the isomorphism $\mbR^1 f_* (\chi_{\phi \cap B})$ in the case $\N' = \N$ makes 
the following diagram commute:
\begin{align} \label{Equu10ss9}
\vcenter{\xymatrix@C=46pt@R=36pt{
\mbR^1 f_* ( \mfg_{\mcE, \STR^{(\N -1)}}) 
 \ar[rr]_-{\sim}^-{\mbR^1 f_*(\chi_{\phi \cap B})} 
  \ar[rd]_-{-D'_{\msE^\spadesuit} }&&   \mbR^1 f_* (\widetilde{\mcT}_{\mcE^\mr{log}_B/S^\mr{log}, \phi} ) \ar[ld]^-{D_{\msE^\spadesuit}}
\\
& \mbR^1 f_* (\mcT).&
}}
\end{align}

We here focus on the case $\N' = \N-1$ under the assumption that $\N \geq 2$.
Using the natural inclusion 
 $\Omega^{\otimes p^{\N-1}} \otimes \mfg_{\mcE} \left(\cong \mfg_\mcE \otimes  \Omega^{\otimes p^{\N-1}} \right) \hookrightarrow\mfg_{\mcE} \otimes \overline{\mcP}_{(\N -1)}^{p^{\N-1}}$, 
 we regard 
 $\Omega^{(\N -1)} \otimes \mfg_{\mcE, \phi^{(\N-2)}}$  as a subsheaf of  $\mfg_{\mcE} \otimes \overline{\mcP}_{(\N -1)}^{p^{\N-1}}$ via $F_{X/S}^{(\N -1)}$.
The image of $\widetilde{\mcT}_{\mcE^\mr{log}/S^\mr{log}, \phi^{(\N -2)}}$ under $\widetilde{\nabla}_{\phi}^{\mr{ad}(\N -1)}$ is verified to be contained in $\Omega^{(\N -1)} \otimes \mfg_{\mcE, \phi^{(\N-2)}}$.
Thus, $\widetilde{\nabla}_{\phi}^{\mr{ad}(\N -1)}$ restricts to a morphism 
$\widetilde{\mcT}_{\mcE^\mr{log}/S^\mr{log}, \phi^{(\N -2)}} \rightarrow \Omega^{(\N -1)} \otimes \mfg_{\mcE, \phi^{(\N-2)}}$,
 extending  an $S^\mr{log}$-connection $(\nabla_\phi^\mr{ad})^{[\N -1]} : \mfg_\mcE^{[\N -1]} \rightarrow \Omega^{(\N -1)} \otimes \mfg_\mcE^{[\N -1]}$ on $\mfg_\mcE^{[\N -1]} \left(= \mfg_{\mcE, \phi^{(\N -2)}} \right)$.
Let 
\begin{align}
(\widetilde{\nabla}_{\phi}^{\mr{ad}})_B^{[\N -1]} 
: \widetilde{\mcT}_{\mcE_B^\mr{log}/S^\mr{log}, \phi^{(\N -2)}} \rightarrow  \Omega^{(\N -1)} \otimes \mfg_\mcE^{[\N -1]}
\end{align}
denote the morphism 
 obtained by restricting $\widetilde{\nabla}_{\phi}^{\mr{ad}(\N -1)}$.
The diagram \eqref{Equu106} for $\N' = \N-1$ yields the following commutative diagram:
\begin{align} \label{Equu109}
\vcenter{\xymatrix@C=46pt@R=36pt{
\mfg_{\mcE, \STR^{(\N-2)}} 
 \ar[rr]_-{\sim}^-{\chi_{\phi \cap B}} 
  \ar[rd]_-{(\nabla_\phi^\mr{ad})^{[\N -1]}}&&  \widetilde{\mcT}_{\mcE^\mr{log}_B/S^\mr{log}, \phi^{(\N-2)}}\ar[ld]^-{(\widetilde{\nabla}^\mr{ad}_{\phi})_B^{[\N -1]}}
\\
& \Omega^{(\N -1)}  \otimes \mfg_\mcE^{[\N -1]}.&
}}
\end{align}

\LSP
\subsection{Local triviality of dormant $\mr{PGL}_2^{(\N)}$-opers} \label{SSG68}

In this subsection, we suppose that $S = \mr{Spec}(k)$ for an algebraically closed field $k$ over $\mbF_p$.

By an {\bf fs log local Artinian $k$-algebra}, we mean
a triple $(R, M, \alpha)$ consisting of a  local Artinian $k$-algebra $R$ with residue field $k$, where we write $\mfm_R$ for its maximal ideal,
an integral commutative monoid $M$, and a monoid homomorphism  $\alpha : M \rightarrow R$ such that $\alpha$ restricts to an isomorphism $\alpha^{-1} (R^\times) \xrightarrow{\sim} R^\times$ and $M/R^\times$ is a sharp fs monoid.
A {\bf  morphism of fs log local Artinian $k$-algebras} from $(R, M, \alpha)$ to $(R', M', \alpha')$ is a pair $(\nu_R, \nu_M)$ consisting of a $k$-algebra homomorphism  $\nu_R : R \rightarrow R'$ and a monoid homomorphism $\nu_M : M \rightarrow M'$ with $\nu_R \circ  \alpha = \alpha' \circ \nu_{M}$. 

For each fs log local Artinian $k$-algebra $(R, M, \alpha)$, 
the pair $(M, \alpha)$ specify an fs  log structure on $\mr{Spec} (R)$; we denote the resulting log scheme by $\mr{Spec} (R)^\mr{log}$.
The assignment $(R, M, \alpha) \mapsto \mr{Spec}(R)^\mr{log}$  determines an anti-equivalence between the category of fs log local Artinian $k$-algebras and the category of fs log schemes over $k$ whose underlying schemes are the affine schemes associated to local  Artinian $k$-algebras.

For a morphism of fs log local Artinian $k$-algebras $(\nu_R, \nu_M) : (R, M, \alpha) \rightarrow (R', M', \alpha')$,
we say that $(\nu_R, \nu_M)$ is {\bf strict} 
if 
 the morphism $M \oplus_{R^\times}R'^\times \rightarrow M'$ given by $(m, v) \mapsto \nu_M (m) v$ is an isomorphism (or equivalently, the morphism $\overline{\nu}_M : M/R^\times \rightarrow M'/R'^\times$ induced by $\nu_M$ is an isomorphism).
 Also,  $(\nu_R, \nu_M)$ is said to be a {\bf strict surjection} if it is strict and $\nu_R : R \rightarrow R'$ is surjective.
Then, $(\nu_R, \nu_M)$ is a strict surjection if and only if the associated morphism of log schemes $\mr{Spec}(R')^\mr{log} \rightarrow \mr{Spec}(R)^\mr{log}$  is a strict closed immersion.

A {\bf small extension} of $(R, M, \alpha)$ is a strict surjection 
$(\widetilde{R}, \widetilde{M}, \widetilde{\alpha}) \twoheadrightarrow (R, M, \alpha)$ such that the ideal $I := \mr{Ker}(\widetilde{R} \twoheadrightarrow R)$ satisfies $I \cdot  \mfm_{\widetilde{R}} =0$; we denote this small extension by  ``$0 \rightarrow I \rightarrow (\widetilde{R}, \widetilde{M}, \widetilde{\alpha}) \rightarrow (R, M, \alpha) \rightarrow 0$".

The fs log scheme $S^\mr{log}$ corresponds  to  a triple  $(k, M_0, \alpha_0)$ via the anti-equivalence of categories mentioned above.
Denote by 
\begin{align} \label{EE39}
\mcA rt_{/(k, M_0, \alpha_0)}^\mr{log}
\end{align}
the category whose objects are the fs log local Artinian $k$-algebras $(R, M, \alpha)$ equipped with a strict surjection  $(R, M, \alpha) \twoheadrightarrow (k, M_0, \alpha_0)$, and whose morphisms are the strict morphisms 
$(R, M, \alpha) \rightarrow (R', M', \alpha')$ preserving  the surjection onto $(k, M_0, \alpha_0)$. 
In particular, each object $(R, M, \alpha)$ in this category corresponds to an fs  log $k$-scheme $\mr{Spec}(R)^\mr{log}$ under  $S^\mr{log}$.

For example, let us set $k_\epsilon := k [\epsilon]/(\epsilon^2) = k \oplus k \cdot \epsilon$ (hence $k_\epsilon^\times = k^\times \oplus k \cdot \epsilon$).
The homomorphism $\alpha_0$ induces a monoid homomorphism $\alpha_\epsilon : M_0 \oplus_{k^\times} k_\epsilon^\times \rightarrow k_\epsilon$, and the resulting triple $(k_\epsilon, M_\epsilon, \alpha_\epsilon)$ defines an object of $\mcA rt_{/(k, M_0, \alpha_0)}^\mr{log}$.


Let $(R', M', \alpha') \rightarrow (R, M, \alpha) \leftarrow (R'', M'', \alpha'')$ be a diagram in $\mcA rt_{/(k, M_0, \alpha_0)}^\mr{log}$.
Then, the triple $(R' \times_R R'', M' \times_M M'', \alpha' \times \alpha'')$ forms an object in this category, where  $\alpha' \times \alpha''$ denotes the monoid homomorphism $M' \times_M M'' \rightarrow R' \times_R R''$ induced by $\alpha'$ and $\alpha''$.
It is also easy to see that $(R' \times_R R'', M' \times_M M'', \alpha' \times \alpha'')$ gives the fiber product of $(R', M', \alpha') \rightarrow (R, M, \alpha) \leftarrow (R'', M'', \alpha'')$ in  $\mcA rt_{/(k, M_0, \alpha_0)}^\mr{log}$.

Let $(R, M, \alpha)$ be an object of $\mcA rt^\mr{log}_{/(k, M_0, \alpha_0)}$, where $S^\mr{log}_R := \mr{Spec}(R)^\mr{log}$,  and $U^\mr{log}$ a log curve over $\mr{Spec}(k)^\mr{log} := \mr{Spec}(k) \times_{\mr{Spec}(R)} S_R^\mr{log}$ whose underlying $k$-scheme $U$ is affine.
According to ~\cite[Proposition (3.14)]{KaKa},
there exists a unique (up to isomorphism)  lifting   $U_{R, \mr{triv}}^\mr{log}$ of $U^\mr{log}$   over $S_R^\mr{log}$ that  forms  a log curve.
We refer to $U_{R, \mr{triv}}^\mr{log}$ 
as the {\bf trivial deformation} of $U^\mr{log}$ over $S^\mr{log}_R$.
The formation of the trivial deformation commutes with base-change along (morphisms of log schemes induced from) morphisms in $\mcA rt^\mr{log}_{/(k, M_0, \alpha_0)}$.
If $S_R^\mr{log} \cong  \mr{Spec}(R) \times_{\mr{Spec}(k)} S^\mr{log}$,
  then
 $U_{R, \mr{triv}}^\mr{log}$ 
 is  isomorphic to the product $S_R^\mr{log} \times_{\mr{Spec}(k)^\mr{log}} U^\mr{log}$,
  i.e., obtained by changing the base of $U^\mr{log}$ over $S^\mr{log}$.


The following proposition asserts  a certain  local triviality of  dormant $\mr{PGL}_2^{(\N)}$-opers.

\SSP
\bpr \label{Prop65}
Suppose that there exists a global logarithmic coordinate $x$ on $U^\mr{log}$.
Let $\msG^\spadesuit := (\mcG_{B}, \phi)$ be a dormant $\mr{PGL}_2^{(\N)}$-oper on $U^\mr{log}/S^\mr{log}$.
Then, the following assertions hold.
\begin{itemize}
\item[(i)]
There exists a dormant $\mr{PGL}_2^{(\N)}$-oper $\msG^\spadesuit_{R} := (\mcG_{B, R}, \phi_R)$ on $U^\mr{log}_{R, \mr{triv}}/S_R^\mr{log}$ lifting  $\msG^\spadesuit$.
\item[(ii)]
Let  $\msG_R^\spadesuit$ be as in (i) and $\breve{\msG}_R^\spadesuit := (\breve{\mcG}_B, \breve{\phi}_R)$ another dormant $\mr{PGL}_2^{(\N)}$-oper on $U_{R, \mr{triv}}^\mr{log}/S^\mr{log}_R$ lifting $\msG^\spadesuit$.
Then,  there exists an automorphism
$\tau_U : U^\mr{log}_{R, \mr{triv}} \xrightarrow{\sim} U^\mr{log}_{R, \mr{triv}}$ over $S^\mr{log}_R$ and a $B$-equivariant isomorphism
$\tau_\mcG : \mcG_{B} \xrightarrow{\sim} \breve{\mcG}_B$ satisfying the following conditions:
\begin{itemize}
\item
The following square diagram is commutative:
\begin{align} \label{Equu23}
  \vcenter{\xymatrix@C=46pt@R=36pt{
\mcG_{B, R}^\mr{log} \ar[r]_-{\sim}^{\tau_{\mcG}} \ar[d]_-{\mr{projection}}
& \breve{\mcG}_{B, R}^\mr{log}\ar[d]^-{\mr{projection}}
\\
U^\mr{log}_{R, \mr{triv}} \ar[r]^-{\sim}_{\tau_{U}}& U^\mr{log}_{R, \mr{triv}},
}}
\end{align}
where we equip $\mcG_{B, R}$ and $\breve{\mcG}_{B, R}$ with the log structures pulled-back from $U^\mr{log}_{R, \mr{triv}}$ along the natural structure morphisms as $B$-bundles, and the resulting log schemes are denoted by $\mcG_{B, R}^\mr{log}$ and $\breve{\mcG}_{B, R}^\mr{log}$, respectively.
\item
$\tau_\mcG$ defines an isomorphism $\msG^\spadesuit_R \xrightarrow{\sim} \tau^*_U(\breve{\msG}^\spadesuit_{R})$ of dormant $\mr{PGL}_2^{(\N)}$-opers  on $U^\mr{log}_{R, \mr{triv}}/S_R^\mr{log}$, where 
$\tau^*_U(\breve{\msG}^\spadesuit_{R})$ denotes the dormant $\mr{PGL}_2$-oper on $U^\mr{log}_{R, \mr{triv}}$ obtained by pulling-back $\breve{\msG}^\spadesuit_{R}$ along $\tau_U$.
\end{itemize}
\end{itemize}
\epr
\begin{proof}
We set $G := \mr{PGL}_2$ for simplicity.
We prove both assertions (i) and (ii)
by induction on $\ell := \mr{dim}_k (R)$.
Since there is nothing to prove when $\ell  = 1$ (i.e., $S^\mr{log}_R = S^\mr{log}$),
we only discuss  the induction step.
Suppose that $\ell \geq 2$, and that we have proved assertions (i) and (ii) in the case where the dimension of  $R$ is smaller than $\ell$.
Let us take a nonzero elements $a$ of $R$ with $a \cdot \mfm_R = 0$ and  $a^2 =0$.
 Write $\overline{R} := R/(a)$.
The triple $(\overline{R}, M \oplus_{R^\times}\overline{R}^\times, \overline{\alpha})$, where $\overline{\alpha}$ denotes the morphism $M \oplus_{R^\times}\overline{R}^\times \rightarrow \overline{R}$ induced by $\alpha$, defines an object of $\mcA rt_{/(k, M_0, \alpha_0)}^\mr{log}$.
This triple corresponds to a log scheme $S_{\overline{R}}^\mr{log} := \mr{Spec} (\overline{R})^\mr{log}$.

First, let us consider assertion (i).
By the induction hypothesis, there exists a dormant $G^{(\N)}$-oper $\msG^\spadesuit_{\overline{R}} := (\mcG_{B, \overline{R}}, \phi_{\overline{R}})$ on $U^\mr{log}_{\overline{R}, \mr{triv}} \left(\cong \mr{Spec}(\overline{R}) \times_{\mr{Spec}(R)}U^\mr{log}_{R, \mr{triv}} \right)$ lifting $\msG^\spadesuit$.
As discussed in Section \ref{SSrtq}, 
we obtain $\mu_{p^\N}$-actions on $U_{\overline{R}, \mr{triv}}$ and $U_{R, \mr{triv}}$ corresponding to
 the trivial $\mcD^{(\N -1)}$-module structures, and the inclusion $U_{\overline{R}, \mr{triv}} \hookrightarrow U_{R, \mr{triv}}$ preserves the $\mu_{p^\N}$-action.
In what follows, we prove the claim that there exist a $B$-bundle $\mcG_{B, R}$ on $U_{R, \mr{triv}}$ lifting $\mcG_{B, \overline{R}}$ and a $\mu_{p^\N}$-action on $\mcG_{B, R} \times^B G$ extending that on $U_{R, \mr{triv}}$.

The obstruction for lifting  the $B$-bundle  $\mcG_{B, \overline{R}}$ (without the $\mu_{p^\N}$-action) over $\mr{Spec}(R)$ lies in $H^2 (U, \mfb_{\mcG_B})$,  so the affineness of $U$ implies that  $\mcG_{B, \overline{R}}$ lifts to a $B$-bundle $\mcG_{B, R}$ on  $U_{R, \mr{triv}}$.
Set $\mcG_R := \mcG_{B, R} \times^B G$ and $\mcG_{\overline{R}} := \mcG_{B, \overline{R}} \times^B G$.
Next, for $\Box \in \{ R, \overline{R} \}$, let $\mr{act}_{\Box} : \mu_{p^{\N}} \times_k U_{\Box, \mr{triv}} \rightarrow U_{\Box, \mr{triv}}$ 
denote the $\mu_{p^\N}$-action on $U_{\Box, \mr{triv}}$ and let 
$\mr{pr}_\Box : \mu_{p^\N} \times_k U_{\Box, \mr{triv}} \rightarrow U_{\Box, \mr{triv}}$ 
 denote  the second projection.
The  $\mu_{p^\N}$-action on $\mcG_{\overline{R}}$ corresponding to $\phi_{\overline{R}}$ via the equivalence \eqref{EE36} determines a $G$-equivariant  isomorphism $\lambda_{\overline{R}} : \mr{pr}_{\overline{R}}^*(\mcG_{\overline{R}}) \xrightarrow{\sim} \mr{act}^*_{\overline{R}}(\mcG_{\overline{R}})$.
Consider the set-valued functor $\mcI som_{G}(\mr{pr}^*_R (\mcG_{R}), \mr{act}_R^*(\mcG_{R}))$ which assigns, to each $(\mu_{p^\N} \times_k U_{R, \mr{triv}})$-scheme $T$, the set of isomorphisms $\mr{pr}^*_R (\mcG_{R})|_{T} \xrightarrow{\sim} \mr{act}_R^*(\mcG_{R})|_T$ of $G$-bundles on $T$.
This functor can be represented by
a scheme $Z$ of finite type over $\mu_{p^\N} \times_k U_{R, \mr{triv}}$.
Given an fppf covering $T\rightarrow \mu_{p^\N} \times_k U_{R, \mr{triv}}$ that trivializes both $\mr{pr}^*_R (\mcG_{R})$ and $\mr{act}_R^*(\mcG_{R})$, we have 
\begin{align}
Z \times_{\mu_{p^\N}\times_k U_{R, \mr{triv}}} T \cong \mcI som_{G} (G \times_S T, G \times_S T) \cong G \times_S T.
\end{align}
This implies that the structure morphism $Z \rightarrow \mu_{p^\N} \times_k U_{R, \mr{triv}}$ is smooth.
It follows that $\lambda_{\overline{R}} \in Z (\overline{R})$ lifts to an isomorphism $\lambda_R : \mr{pr}^*_R (\mcG_{R}) \xrightarrow{\sim} \mr{act}_R^*(\mcG_{R})$; we abuse notation by wring $\lambda_R$ for the induced morphism $\mu_{p^\N} \times_k \mcG_{R} \rightarrow \mcG_{R}$.
The failure of the cocycle condition for $\lambda_R$ is measured by the second cohomology $H^2 (\mu_{p^\N}, H^0 (U, \mfg_{\mcG}))$ of the $\mu_{p^\N}$-module $H^0 (U, \mfg_{\mcG})$, where $\mcG:= \mcG_B \times^B G$,  with respect to the $\mu_{p^\N}$-action induced from that on $\mcG$ corresponding to $\phi$.
However, this cohomology group vanishes, so $\lambda_R$ defines an action of $\mu_{p^\N}$ after replacing it with another.
(In fact, a $\mu_{p^\N}$-action on a $k$-vector space $M$ corresponds to a decomposition $M = \bigoplus_{j \in \mbZ/p^\N \mbZ} M_j$ with $M^{\mu_{p^\N}} = M_0$, as discussed in Section \ref{SSrtq}. The functor of taking $\mu_{p^\N}$-invariant elements is exact, which implies $H^\ell (\mu_{p^\N}, M) =0$ for $\ell >0$.)
This proves the claim.

The $\mu_{p^\N}$-action on $\mcG_{R}$  just obtained corresponds,  via  \eqref{EE36},  to an $(\N -1)$-PD stratification $\phi_R$ with vanishing $p^\N$-curvature, which lifts $\phi_{\overline{R}}$.
Since the mod $a$ reduction of  the pair $\msG_R^\spadesuit: = (\mcG_B, \phi_R)$ forms a $G^{(\N)}$-oper, this pair itself defines a dormant $G^{(\N)}$-oper.
This completes the proof of assertion (i).

Next, let us consider assertion (ii).
We use the notation $(-)^{(\N'-1)}$ (with  $1 \leq \N' \leq  \N$) to denote the result of reducing the level of an $(\N-1)$-PD stratification $(-)$ to $\N' -1$.

By the induction hypothesis,
there exists a pair $(\overline{\tau}_U, \overline{\tau}_\mcG)$ satisfying the two conditions in the statement of (ii) with $U^\mr{log}_{R, \mr{triv}}$ replaced by $U^\mr{log}_{\overline{R}, \mr{triv}}$.
After pulling-back $\breve{\msG}^\spadesuit_R$ via an automorphism of $U_{R, \mr{triv}}^\mr{log}$ lifting $\overline{\tau}_U$, we may assume, without loss of generality, that
$(\mcG_{B, \overline{R}}, \phi_{\overline{R}})$ coincides with the mod $a$ reduction $(\breve{\mcG}_{B, \overline{R}}, \breve{\phi}_{\overline{R}})$ of $(\breve{\mcG}_{B, R}, \breve{\phi}_{R})$.

To complete the proof, it suffices to show, by induction on $\N$, the claim  that there exists a   pair $(\tau_U, \tau_\mcE)$ satisfying the two conditions in (ii) with $(\phi_R, \breve{\phi}_R)$ replaced by $(\phi_R^{(\N -1)}, \breve{\phi}_R^{(\N-1)})$ extending the identity morphisms over $S_{\overline{R}}^\mr{log}$.
Since  the proof in the case $\N= 1$
closely parallels the argument of the following  induction step  and is even more concise,
we omit it and leave the details to reader.
Thus, suppose that $\N\geq 2$, and that the assertion holds at all the levels less than $\N$.

By the induction hypothesis on $\N$,
there exists a required pair $(\tau'_U, \tau'_\mcG)$ for $\N -1$.
Since $U$ is affine (which implies $H^1 (U, \mfb_{\mcG_B}) =0$),
we may assume, after pulling-back $\breve{\msG}^\spadesuit_R$ via $\tau'_U$ and succesibly composing a suitable infinitesimal automorphism of $\breve{\mcG}_{B, R}$,  that 
$(\breve{\mcG}_{R}, \breve{\phi}_R^{(\N -2)}) = (\mcG_R, \phi_R^{(\N -2)})$.
It follows that there exists a well-defined $\mcO_U$-linear morphism 
$h : \mcD^{(\N -1)} \rightarrow  \mcE nd_{k}(\mcO_\mcG)$ satisfying
 \begin{align} \label{UU902}
 \breve{\phi}_R^{ \natural \natural}  = \STR_R^{\natural \natural} + h \cdot \ep
 \end{align}
  (cf. Remark \ref{Eruy78} for the definition of $(-)^{\natural \natural}$).
 Since  $\breve{\phi}_R^{(\N -2)} = \phi_R^{(\N -2)}$, 
 the composite
 \begin{align} \label{UU901}
 \mcD^{(\N -2)}_{\leq p^{\N -1}-1} \xrightarrow{\sim} \mcD^{(\N -1)}_{\leq p^{\N -1}-1}   \xrightarrow{\mr{inclusion}} \mcD_{\leq p^{\N-1}}^{(\N -1)} \xrightarrow{h} \mcE nd_{k} (\mcO_\mcG)
 \end{align}
 coincides with the zero map.
 Hence, $h$ induces 
 an $\mcO_U$-linear morphism  
 \begin{align} \label{UU900}
 \overline{h} : F^{(\N -1)*}_{U/S} (\mcT^{(\N -1)}) \left(= \mcT^{\otimes p^{\N -1}} = (\mcD_{\leq p^{\N-1}}^{(\N -1)}/\mcD_{\leq p^{\N-1}-1}^{(\N -1)})   \right) \rightarrow  \mcE nd_{k} (\mcO_\mcG),
 \end{align}
 where $\mcT^{(\N-1)}:= \mcT_{U^{(\N -1)\mr{log}}/S^\mr{log}}$.

Since $\ep^2 =0$ and $h$ is $G$-equivariant (in an evident sense),
the image of $\overline{h}$ lies in 
 $\mfg_\mcG \left(= \pi_* (\mcT_{\mcG/U})^{G} \right)$.
 Moreover,  
 by the commutativity relation \eqref{Equu88}, 
   $\overline{h}$ restricts to a morphism
 \begin{align}
 \overline{h}^\nabla : \mcT^{(\N -1)} \migi   \mfg_{\mcG, \phi^{(\N -2)}}   \left(=\mfg_\mcG^{[\N -1]} \right).
 \end{align}
 We shall use the same notation $\overline{h}^\nabla$ to denote the induced  $\mcO_{U^{(\N -1)}}$-linear  morphism $\mcO_\mcG^{[\N -1]} \migi \Omega^{(\N -1)} \otimes \mcO_\mcG^{[\N -1]}$ (where $\Omega^{(\N -1)} := \mcT^{(\N -1)\vee}$) via the natural morphism  $\mfg_{\mcG, \phi^{(\N -2)}} \migi \mcE nd_{\mcO_{U^{(\N -1)}}} (\mcO_\mcG^{[\N -1]})$.

If $\DMO$ (resp., $\breve{\DMO}$) denotes  the $S_R^\mr{log}$-connection on the $\mcO_{U_R^{(\N -1)}}$-module  $\mcO_{\mcG_R}^{[\N -1]}$ induced   by $\STR_R^{(\N -1)}$ (resp., $\breve{\STR}^{(\N -1)}_R$), then 
   \eqref{UU902} implies 
  $\breve{\DMO} = \DMO + \overline{h}^\nabla \cdot \ep$.
 Since both $\DMO$ and $\breve{\DMO}$ have vanishing $p$-curvature  (cf. Proposition \ref{UU577}, (ii)),
 it follows from Lemma \ref{UU888} described below that
 the section $\overline{h}^\nabla$
 of $\Omega^{(\N -1)} \otimes \mfg_{\mcG, \phi^{(\N -2)}}$ 
  lies in the image of $(\DMO_\STR^{\mr{ad}})^{[\N -1]}$. 
By the commutativity of \eqref{Equu109},
we can find  an element $v \in H^0 (U, \widetilde{\mcT}_{\mcG_B^\mr{log}/S^\mr{log}, \phi^{(\N -2)}})$ satisfying 
  \begin{align} \label{UU944}
  (\widetilde{\nabla}^\mr{ad}_{\phi})_B^{[\N -1]} (v) = \overline{h}^\nabla.
  \end{align}

Denote by $h_v$ the automorphism of 
$\mcG_{R}^\mr{log}$
 (inducing an automorphism of $U^\mr{log}_{R, \mr{triv}}$) corresponding to  the automorphism of $\mcO_{\mcG_R}$ described as   $\mr{id}_{\mcO_{\mcG_R}} - v \cdot \ep$.
Since  $v$ lies in $H^0(U, \widetilde{\mcT}_{\mcG_B^\mr{log}/S^\mr{log}})$,
 the automorphism $h_v$  preserves  the $B$-reduction $\mcG_{B, R}$.
Also,  
since $v$ lies in  $H^0 (U, \widetilde{\mcT}_{\mcG^\mr{log}_B/S^\mr{log}, \phi^{(\N -2)}})$, 
the $\mcD^{(\N -1)}$-module structure $h_v^* (\breve{\STR}_R^{\natural \natural})$  on $\mcO_{\mcG_R}$ defined as 
the pull-back of $\breve{\STR}_R^{\natural \natural}$ by  $h_v$ satisfies 
\begin{align} \label{UU945}
h_v^* (\breve{\STR}_R^{\natural \natural})^{(\N -2)} =
\STR_R^{\natural \natural (\N -2)}.
\end{align}
Moreover,   the equality \eqref{UU944} together with the same  argument as   ~\cite[Remark 6.2]{Wak8} shows that
\begin{align} \label{UU955}
 h_v^* (\breve{\STR}^{\natural \natural}_R)^{[\N -1]} =\STR_R^{\natural \natural [\N -1]}.
 \end{align} 
By   \eqref{UU945} and  \eqref{UU955}, we can apply Lemma \ref{NN1} to obtain 
the equality $h_v^* (\breve{\STR}_R^{\natural \natural}) = \STR_R^{\natural \natural}$.
This means that $\breve{\STR}_R^{\natural \natural}$ can be transformed into $\STR_R^{\natural \natural}$  via 
a suitable automorphism of $\mcG_{B, R}$.
This completes the induction step of the claim, so we finish the proof of assertion (ii).
\end{proof}
\SSP

The following lemma was  used  in the proof of the above proposition.

\SSP
\ble \label{UU888}
Let us keep the notation in the proof of Proposition \ref{Prop65} (in particular, $\breve{\nabla} = \nabla + \overline{h}^\nabla \cdot a$).
Then, the $p$-curvatures of $\nabla$ and $\breve{\nabla}$  satisfy the equality 
\begin{align}
\psi (\breve{\nabla}) - \psi (\DMO) = -\ep \cdot C_{(\mfg_\mcG^{[\N-1 ]}, (\DMO_\STR^\mr{ad})^{[\N -1 ]})} (\overline{h}^\nabla)
\end{align}
under the composite of natural  inclusions
\begin{align} \label{TTT91}
H^0 (U^{(\N)}, \Omega^{(\N)} \otimes \mfg_\mcG^{[\N]}) 
&\isom \mr{Hom}_{\mcO_{U^{(\N)}}} (\mcT^{[\N]}, \mfg_\mcG^{[\N]}) \\ 
&\migiincl
\mr{Hom}_{\mcO_{U^{(\N)}}} (\mcT^{[\N]}, \mcE nd_{k} (\mcO_\mcG)) \notag \\
& \isom  \mr{Hom}_{\mcO_{U^{(\N -1)}}} (F_{U^{(\N -1)}/k}^{*}(\mcT^{[\N]}), \mcE nd_{k}(\mcO_\mcG)|_{U^{(\N-1)}}). \notag
\end{align} 
In particular, if the $p$-curvatures of $\DMO$ and $\breve{\nabla}$ vanish, then $\overline{h}^\nabla$ may be expressed as $\overline{h}^\nabla = (\DMO_\STR^\mr{ad})^{[\N -1]}(v)$ for some global section $v$ of $\mfg_\mcG^{[\N-1]}$.
\ele
\begin{proof}
The assertion follows from computations similar to those in the proof of ~\cite[Proposition 6.11]{Wak8}.
So we omit the details of the proof.
\end{proof}
\SSP


\LSP
\subsection{Smoothness of the moduli space} \label{SS068}


By a {\it  functor on $\mcA rt_{/(k, M_0, \alpha_0)}^\mr{log}$} (or sometimes just a {\it functor} when there is no fear of confusion), we shall mean a covariant functor
 $D : \mcA rt_{/(k, M_0, \alpha_0)}^\mr{log} \migi \mcS et$ such that  $D ((k, M_0, \alpha_0))$ is a singleton.
Also, following ~\cite[Definition 2.8]{FaMa} (or ~\cite[Section 6.1]{Wak7}), we  refer to a functor $D$ on  $\mcA rt_{/(k, M_0, \alpha_0)}^\mr{log}$ as 
a {\bf functor with good deformation theory} (or simply, a {\bf functor with gdt}) if it satisfies the following two of  the Schlessinger conditions in the logarithmic setting:
 \begin{itemize}
 \item[$(\text{H}_1^\mr{log})$]
 For any morphism $\sigma' : (R', M', \alpha') \migi (R, M, \alpha)$ and any surjection  $\sigma'' : (R'', M'', \alpha'') \migisurj  (R, M, \alpha)$ both of which belong  to  $\mcA rt^\mr{log}_{/(k, M_0, \alpha_0)}$, 
 the map 
 \begin{align} \label{eaori}
 &D ((R' \times_R R'', M' \times_M M'', \alpha' \times \alpha'')) \\
 & \migi D ((R', M', \alpha')) \times_{D ((R, M, \alpha))} D ((R'', M'', \alpha'')) \notag
 \end{align}
induced naturally by $\sigma'$ and $\sigma''$ is surjective.
(This condition remains equivalent even if $\sigma''$ is moreover assumed to be a small extension.)
 \item[$(\text{H}_2^\mr{log})$]
 With the above notation, 
 \eqref{eaori} becomes bijective when $(R, M, \alpha) = (k, M_0, \alpha_0)$.
 (This condition remains equivalent even if $R''$ is moreover assumed to be  $(k_\epsilon, M_\epsilon, \alpha_\epsilon)$.)
 \end{itemize}

Let  $D : \mcA rt_{/(k, M_0, \alpha_0)}^\mr{log} \migi \mcS et$ be  a deformation functor with gdt  and write $T D := D ((k_\epsilon, M_\epsilon, \alpha_\epsilon))$.
 According to the discussion following ~\cite[Proposition 3.4]{KaFu98} (or ~\cite[\S\,2, Lemma 2.12]{FaMa} in the non-logarithmic setting),
 the following properties  hold:
 \begin{itemize}
 \item
 $TD$ has a canonical structure of $k$-vector space;
 \item
For any small extension in $\mcA rt_{/(k, M_0, \alpha_0)}^\mr{log}$ of the form $e := (0 \migi I \migi (R, M, \alpha) \migi (k, M_0, \alpha_0) \migi 0)$,
there exists a canonical  bijection $I \otimes_k TD \isom D ((R, M, \alpha))$;
 \item
 For any small extension $e := (0 \migi I \migi (\widetilde{R}, \widetilde{M}, \widetilde{\alpha}) \xrightarrow{\sigma} (R, M, \alpha) \migi 0)$ in  $\mcA rt_{/(k, M_0, \alpha_0)}^\mr{log}$, there exists an $I \otimes_k TD$-action on $D ((\widetilde{R}, \widetilde{M}, \widetilde{\alpha}))$ which is  restricted  to a transitive action on the fibers of $D (\sigma) : D((\widetilde{R}, \widetilde{M}, \widetilde{\alpha})) \migi D((R, M, \alpha) )$.
 Moreover, the formation of this action is compatible, in an evident sense,  with morphisms of small extensions.
 \end{itemize}
\SSP

Here, we recall several notions concerning  deformation functors used in our discussion.
 A functor $D$ 
  is called {\bf  formally smooth} 
  if for every small extension 
  $e := (0 \migi I \migi (\widetilde{R}, \widetilde{M}, \widetilde{\alpha}) \xrightarrow{\sigma} (R, M, \alpha) \migi 0)$ in $\mcA rt_{/(k, M_0, \alpha_0)}^\mr{log}$, the map of sets $D ((\widetilde{R}, \widetilde{M}, \widetilde{\alpha})) \rightarrow D ((R, M, \alpha))$ is surjective.
  
  An {\bf  obstruction theory} for a functor $D$ is  defined as a collection $(O, \{ o_e \}_e)$  consisting of a $k$-vector space $O$ and a map $o_e : D((R, M, \alpha)) \migi I \otimes_k O$ for every small extension 
  $e := (0 \migi I \migi (\widetilde{R}, \widetilde{M}, \widetilde{\alpha}) \xrightarrow{\sigma} (R, M, \alpha) \migi 0)$
   such that an element $a \in D ((R, M, \alpha))$ satisfies $o_e (a) = 0$ if and only if $a$ lifts to $D((\widetilde{R}, \widetilde{M}, \widetilde{\alpha}))$.
 The $k$-vector space $O$ is called an {\it  obstruction space} of $D$.
 We shall often abuse notation by just  writing  $O$ for $(O, \{ o_e \}_e)$.

 
 A {\bf deformation triple} is a collection
\begin{align}
(D, T, O)
\end{align}
consisting of  a functor $D$  with gdt,  a $k$-vector space $T$  equipped with an identification  $T \isom T D$, and    an obstruction theory $O$ for $D$.

Now, let $\msX$ and $\msE^\spadesuit$ be as before.

\SSP
\bde \label{UU444}
\begin{itemize}
\item[(i)]
Let $R$ be a local Artinian $k$-algebra with residue field $k$.
A {\bf deformation of $(\msX, \msE^\spadesuit)$ over $R$}
 is a pair 
$(\msX_R, \msE_R^\spadesuit)$,
where 
\begin{itemize}
\item
$\msX_R$ denotes an $r$-pointed stable curve of genus $g$ over $R$ equipped with an isomorphism $\nu :\msX \isom k \times_R \msX_R$ between  $\msX$ and the reduction $k \times_R \msX_R$ of $\msX_R$ modulo $\mfm_R$;
\item 
$\msE^\spadesuit_R$ denotes a dormant $\mr{PGL}_2^{(\N)}$-oper on $\msX_R$ together with an isomorphism
 of $\mr{PGL}_2^{(\N)}$-opers $\msE^\spadesuit \xrightarrow{\sim} k \times_R \msE^\spadesuit_R$, where we regard  $k \times_R \msE^\spadesuit_R$ as a dormant $\mr{PGL}_n^{(\N)}$-oper on $\msX$ 
 via
$\nu$.
\end{itemize}
\item[(ii)]
Let $(R, M, \alpha) \in \mr{Ob}(\mcA rt^\mr{log}_{/(k, M_0, \alpha_0)})$.
A {\bf deformation of $(\msX, \msE^\spadesuit)$ over $(R, M, \alpha)$}
 is a deformation  $(\msX_R, \msE^\spadesuit_R)$ of $(\msX, \msE^\spadesuit)$ over $R$ such that the   log structures on $\mr{Spec}(R)$ arising from $\msX_R$ coincides with that of $S_R^\mr{log}$.
 (In particular, $\nu$ induces an isomorphism between log curves $X^\mr{log} \xrightarrow{\sim} S^\mr{log} \times_{S_R^\mr{log}} X_R^\mr{log}$ over $S^\mr{log}$.)
\end{itemize}
One can define, in a natural manner, 
the notion of an isomorphism between two  deformations
 of $(\msX, \msE^\spadesuit)$.
\ede
\SSP

For each $(R, M, \alpha) \in \mr{Ob}(\mcA rt^\mr{log}_{/(k, M_0, \alpha_0)}$), we shall set
 \begin{align} \label{Equu153}
 \mr{Def}_{(\msX, \msE^\sss)} ((R, M, \alpha)) := 
 \left(\begin{matrix} \text{the set of isomorphism classes} \\ \text{of deformations of $(\msX, \msE^\sss)$ over $(R, M, \alpha)$} 
 \end{matrix} \right).
 \end{align}
If $\sigma : (R', M', \alpha') \migi (R, M, \alpha)$ is a morphism in $\mcA rt_{/(k, M_0, \alpha_0)}^\mr{log}$,
then the base-change along  the induced morphism of log schemes $S_R^\mr{log} \migi S_{R'}^\mr{log}$ yields  a map of sets
$\sigma^* : \mr{Def}_{(\msX, \msE^\spadesuit)}((R', M', \alpha)) \migi \mr{Def}_{(\msX, \msE^\sss)}((R, M, \alpha))$.
Thus, the assignments  
$(R, M, \alpha) \mapsto \mr{Def}_{(\msX, \msE^\sss)}((R, M, \alpha))$
 and $\sigma \mapsto \sigma^*$ together
  determine  a functor
  \begin{align} \label{Ew3}
   \mr{Def}_{(\msX, \msE^\sss)}  : \mcA rt^\mr{log}_{/(k, M_0, \alpha_0)} \migi \mcS et.
 \end{align}


\SSP
\bt \label{P47}
The functor  $\mr{Def}_{(\msX, \msE^\sss)}$ is a deformation functor with good deformation theory.
Moreover,  the triple 
\begin{align}
(\mr{Def}_{(\msX, \msE^\spadesuit)},  H^1 (X, \widetilde{\mcT}_{\mcE_B^\mr{log}/k^\mr{log}, \phi}),  H^2 (X, \widetilde{\mcT}_{\mcE_B^\mr{log}/k^\mr{log}, \phi}))
\end{align}
forms a deformation triple,
 i.e., 
 the tangent and obstruction spaces of 
$\mr{Def}_{(\msX, \msE^\spadesuit)}$  are given by 
$H^1 (X, \widetilde{\mcT}_{\mcE_B^\mr{log}/k^\mr{log}, \phi})$ ($\cong H^1 (X, \mfg_{\mcE, \phi})$ by \eqref{Equu102}) and 
$H^2 (X, \widetilde{\mcT}_{\mcE_B^\mr{log}/k^\mr{log}, \phi})$, respectively.
In particular, $\mr{Def}_{(\msX, \msE^\sss)}$ is formally smooth.
\et
\begin{proof}
The first assertion can be proved immediately from the various definitions involved.
In what follows, we prove that  the tangent space of $\mr{Def}_{(\msX, \msE^\spadesuit)}$ is isomorphic to $H^1 (X, \widetilde{\mcT}_{\mcE_B^\mr{log}/k^\mr{log}, \phi})$.
Let $\mfe := (0 \migi I \migi (\widetilde{R}, \widetilde{M}, \alpha_{\widetilde{R}}) \xrightarrow{\sigma} (R, M, \alpha_R) \migi 0)$ be a small extension  in 
$\mcA rt_{/(k, M_0, \alpha_0)}^\mr{log}$ and $(\msX_{\widetilde{R}}, \msE^\spadesuit_{\widetilde{R}})$ (where $\msX_{\widetilde{R}} := (X_{\widetilde{R}}, \{ \sigma_{\widetilde{R}, i} \}_i)$ and  $\msE^\spadesuit_{\widetilde{R}} := ((\mcE_B)_{\widetilde{R}}, \STR_{\widetilde{R}})$) a deformation of $(\msX, \msE^\spadesuit)$ over $(\widetilde{R}, \widetilde{M}, \widetilde{\alpha})$.
Set $\mcE_{\widetilde{R}} := (\mcE_B)_{\widetilde{R}} \times^B \mr{PGL}_2$.
Also, let us take an element $v$ of $I \otimes_k H^1 (X, \widetilde{\mcT}_{\mcE_B^\mr{log}/k^\mr{log}, \phi})$.
Note that 
$H^1 (X, I \otimes_k \widetilde{\mcT}_{\mcE^\mr{log}_B/k^\mr{log}, \STR})$
can be realized as the cohomology of the \v{C}ech  complex   of
$I \otimes_k \widetilde{\mcT}_{\mcE^\mr{log}_B/k^\mr{log}, \STR}$ associated with
an affine covering.
 That is to say, we can find an affine open covering $\mcU := \{ U_\alpha \}_{\alpha \in J}$ (where 
$J$ is a finite index set) of $X$,  and  $v$ may be represented
by
 a \v{C}ech $1$-cocycle 
 \begin{align} \label{NN4}
 v =  \{ \partial_{\mcE, \alpha \beta} \}_{(\alpha, \beta) \in J_2} \in \check{C}^1 (\mcU, I \otimes_k \widetilde{\mcT}_{\mcE_B^\mr{log}/k^\mr{log}, \STR}),
 \end{align}
where  $J_2$ denotes  the subset of $J \times J$ consisting of pairs $(\alpha, \beta)$ with $U_{\alpha \beta}:= U_\alpha \cap U_\beta \neq \emptyset$, and $\partial_{\mcE, \alpha\beta} \in H^0 (U_{\alpha \beta},  I \otimes_k \widetilde{\mcT}_{\mcE_B^\mr{log}/k^\mr{log}, \STR})$. 

  For each $(\alpha, \beta) \in J_2$,
 we set $\partial_{\alpha \beta} := d_{\mcE_B} (\partial_{\mcE, \alpha \beta}) \in H^0 (U_{\alpha \beta}, I \otimes_k \mcT)$.
Also, we set $U_{\widetilde{R}, \alpha}^\mr{log} := X_{\widetilde{R}}^\mr{log} |_{U_\alpha}$ for each $\alpha \in J$ and set 
$U_{\widetilde{R}, \alpha \beta}^\mr{log} := X_{\widetilde{R}}^\mr{log} |_{U_{\alpha\beta}}$ for each $(\alpha, \beta) \in J_2$.
 The element $\partial_{\alpha \beta}$ (resp.,  $\partial_{\mcE, \alpha \beta}$)  may be regarded as  a $k$-linear morphism 
 \begin{align} \label{WW34}
 \partial_{\alpha \beta} : \mcO_{U_{\alpha \beta}} \migi I \otimes_k \mcO_{U_{\alpha \beta}} \ 
    \left(\text{resp.,}  \ \partial_{\mcE, \alpha \beta} : \mcO_{\mcE_B |_{U_{\alpha \beta}}} \migi I \otimes_k \mcO_{\mcE_B |_{U_{\alpha \beta}}}\right)
    \end{align}
      satisfying
    the Leibniz rule.
 Then, 
 $\mr{id}_{\mcO_{U_{\widetilde{R}, \alpha \beta}}} + \partial_{\alpha \beta}$  (resp., $\mr{id}_{\mcO_{(\mcE_B)_{\widetilde{R}}|_{U_{\widetilde{R}, \alpha \beta}}}} + \partial_{\mcE, \alpha \beta}$) defines a well-defined  automorphism $\partial^\sharp_{\alpha \beta}$ (resp., $\partial^\sharp_{\mcE, \alpha \beta}$) of  $U_{\widetilde{R}, \alpha \beta}^\mr{log}$ (resp., $(\mcE_B)_{\widetilde{R}}^\mr{log}|_{U_{\widetilde{R}, \alpha \beta}}$).
We abuse  notation  by writing $\partial^\sharp_{\mcE, \alpha \beta}$ for the induced automorphism of $\mcE_{\widetilde{R}}^\mr{log} |_{U_{\widetilde{R}, \alpha \beta}}$.
 By means of the $\partial_{\alpha \beta}^\sharp$'s ($(\alpha, \beta) \in J_2$), 
 we can glue together  the $U_{\widetilde{R}, \alpha}^\mr{log}$'s ($\alpha \in J$) to
  obtain a log curve  $X^\mr{log}_{\widetilde{R}, v}$ over $S_{\widetilde{R}}^\mr{log}$.
  It determines a deformation $\msX_{\widetilde{R}, v}$ of $\msX$ over $\widetilde{R}$ (cf. ~\cite[Theorem 4.1]{KaFu}).
Also, the automorphism  $\partial^\sharp_{\mcE, \alpha \beta}$ is $B$-equivariant, and    
 the following square diagram is commutative and Cartesian:
  \begin{align} \label{DD010}
\vcenter{\xymatrix@C=46pt@R=36pt{
(\mcE_B)_{\widetilde{R}}^\mr{log}|_{U_{\widetilde{R}, \alpha \beta}} \ar[r]_-{\sim}^{\partial^\sharp_{\mcE, \alpha \beta}} \ar[d]_-{\mr{projection}}
& (\mcE_B)_{\widetilde{R}}^\mr{log}|_{U_{\widetilde{R}, \alpha \beta}} \ar[d]^-{\mr{projection}}
\\
U_{\widetilde{R}, \alpha \beta}^\mr{log} \ar[r]^-{\sim}_{\partial^\sharp_{\alpha \beta}}& U_{\widetilde{R}, \alpha \beta}^\mr{log}.
}}
\end{align}
Hence,  by means of the $\partial_{\mcE, \alpha \beta}^\sharp$'s ($(\alpha, \beta) \in J_2$), we can glue together  the $(\mcE_B)_{\widetilde{R}}|_{U_{\widetilde{R}, \alpha}}$'s ($\alpha \in J$) to
  obtain a $B$-bundle   $(\mcE_B)_{\widetilde{R}, v}$ on $X_{\widetilde{R}, v}$.

 For each local section $D \in \mcD^{(\N -1)}_{U_{\widetilde{R}, \alpha \beta}^\mr{log}/S_{\widetilde{R}}^\mr{log}}$, 
the following sequence of equalities holds: 
\begin{align}
& \hspace{5mm} (\partial^\sharp_{\mcE, \alpha \beta})^* (\STR_{\widetilde{R}}^{\natural \natural} |_{U_{\widetilde{R}, \alpha \beta}}) (D) \\
&  = (\mr{id} + \partial_{\mcE, \alpha \beta})\circ \STR_{\widetilde{R}}^{\natural \natural} |_{U_{\widetilde{R}, \alpha \beta}} (D- [\partial_{\alpha \beta}, D]) \circ (\mr{id} + \partial_{\mcE, \alpha \beta})^{-1}
\notag \\
&  = (\mr{id} + \partial_{\mcE, \alpha \beta})\circ \STR_{\widetilde{R}}^{\natural \natural} |_{U_{\widetilde{R}, \alpha \beta}} (D- [\partial_{\alpha \beta}, D]) \circ (\mr{id} - \partial_{\mcE, \alpha \beta})
\notag \\
&= \STR_{\widetilde{R}}^{\natural \natural} |_{U_{\widetilde{R}, \alpha \beta}} (D) - \left([\STR_{\widetilde{R}}^{\natural \natural} |_{U_{\widetilde{R}, \alpha \beta}} (D), \partial_{\mcE, \alpha \beta}] -
\STR_{\widetilde{R}}^{\natural \natural} |_{U_{\widetilde{R}, \alpha \beta}} ([D, \partial_{\alpha \beta}]) \right)\notag \\
& = \STR_{\widetilde{R}}^{\natural \natural} |_{U_{\widetilde{R}, \alpha \beta}} (D) \notag
\end{align}
(cf. ~\cite[Remark 6.2]{Wak8} for the first equality), where the last equality follows from $\partial_{\mcE, \alpha \beta} \in  H^0 (U_{\alpha \beta}, I \otimes_k\widetilde{\mcT}_{\mcE_B^\mr{log}/k^\mr{log}, \STR})$.
This means that $\partial^\sharp_{\mcE, \alpha \beta}$ preserves the $(\N -1)$-PD stratification $\STR_{\widetilde{R}} |_{U_{\widetilde{R}, \alpha \beta}}$.
Hence, 
 the   $\STR_{\widetilde{R}} |_{U_{\widetilde{R}, \alpha \beta}}$'s
  can  be glued together by means of 
 the $\partial^\sharp_{\mcE, \alpha \beta}$'s; 
 we denote by $\STR_{\widetilde{R}, v}$  the resulting $(\N -1)$-PD stratification    on $\mcE_{\widetilde{R}, v} := (\mcE_B)_{\widetilde{R}, v} \times^B \mr{PGL}_2$.
 By construction, $\STR_{\widetilde{R}, v}$ has vanishing $p^\N$-curvature.
It follows that
$\msE^\spadesuit_{\widetilde{R}, v} := ((\mcE_{B})_{\widetilde{R}, v}, \STR_{\widetilde{R}, v})$ forms 
a dormant $\mr{PGL}_2^{(\N)}$-oper on 
$\msX_{\widetilde{R}, v}$, 
 and the pair  $(\msX_{\widetilde{R}, v}, \msE_{\widetilde{R}, v}^\spadesuit)$ specifies an element of $\mr{Def}_{(\msX, \msE^\spadesuit)}((\widetilde{R}, \widetilde{M}, \alpha_{\widetilde{R}}))$.
 
 By Proposition \ref{Prop65}, (i) and (ii), 
 the assignment  
 $(\msX_{\widetilde{R}}, \msE^\spadesuit_{\widetilde{R}}) \mapsto (\msX_{\widetilde{R}, v}, \msE_{\widetilde{R}, v}^\spadesuit)$ (for every $v \in H^1 (X,  I \otimes_k \widetilde{\mcT}_{\mcE_B^\mr{log}/k^\mr{log}, \STR})$)
 defines a transitive  action of $H^1 (X, I \otimes_k \widetilde{\mcT}_{\mcE_B^\mr{log}/k^\mr{log}, \STR})$ on 
the fibers of $\mr{Def}_{(X, \msE^\sss)}(\sigma) : \mr{Def}_{(\msX, \msE^\spadesuit)}( (\widetilde{R}, \widetilde{M}, \alpha_{\widetilde{R}})) \migi \mr{Def}_{(\msX, \msE^\spadesuit)} ((R, M, \alpha))$.
This action is, by  construction,  verified to be free if $(R, M, \alpha) = (k, M_0, \alpha_0)$.
Consequently, $H^1 (X, \widetilde{\mcT}_{\mcE_B^\mr{log}/k^\mr{log}, \phi})$  is isomorphic to the tangent space of $\mr{Def}_{(\msX, \msE^\spadesuit)}$.


The 
  assertion concerning an obstruction space can be proved by a routine argument using Proposition \ref{Prop65}, (i) and (ii),  and  
  the \v{C}ech cohomological descriptions of relevant deformations as discussed above,
   so the proof is left to the reader.

Moreover, the final assertion follows from  $H^2 (X, \widetilde{\mcT}_{\mcE_B^\mr{log}/k^\mr{log}, \phi})= 0$ obtained by   $\mr{dim}(X) =1$.
This completes the proof of this theorem.
\end{proof}
\SSP

\begin{rema}[Pro-representability] \label{Rem332}
It is immediately verified that the functor $\mr{Def}_{(\msX, \msE^\spadesuit)}$ is {\it rigid in the first order}  (cf. ~\cite[Definition 3.10]{KaFu98}) and satisfies the conditions $(\text{H}_3^\mr{log})$ and $(\text{H}_4^\mr{log})$ (cf. the discussion following  ~\cite[Proposition 3.12]{KaFu98}.
Hence, it follows from ~\cite[Theorem 3.13, (2)]{KaFu98} that $\mr{Def}_{(\msX, \msE^\spadesuit)}$ is pro-representable.
This fact can be regarded as a local version of the representability of $\mcO p_{g, r}^{^\mr{Zzz...}}$ asserted in Theorem \ref{T40}.
\end{rema}
\SSP

As a consequence of   the above theorem, we obtain the following assertion, generalizing ~\cite[Chapter II, Theorem 2.8]{Mzk2}.

\SSP
 \bco\label{T34} 
 Let  $(\mcO p^\ZZZ_{g,r})^\mr{log}$ denote the (Deligne-Mumford) log stack defined as $\mcO p^\ZZZ_{g,r}$ equipped with the log structure pulled-back from $\overline{\mcM}_{g, r}^\mr{log}$.
  Then, the following assertions hold.
   \begin{itemize}
 \item[(i)]
 $(\mcO p^\ZZZ_{g,r})^\mr{log}$
    is  log smooth over $\mbF_p$ whose log tangent bundle has constant rank $3g-3+r$.
    In particular,  the underlying  stack  $\mcO p^\ZZZ_{g,r}$\! is smooth over $\mbF_p$    and equidimensional   of  dimension $3g-3+r$. 
   \item[(ii)]
 Let
 $k$ be an algebraically closed  field over  $\mbF_p$,
  $\msX$  an $r$-pointed stable curve of genus $g$ over $k$,   and $\msE^\spadesuit$ a dormant $\mr{PGL}_2^{(\N)}$-oper on $\msX$.
Denote by $s : \mr{Spec}(k) \migi \overline{\mcM}_{g,r}$ (resp.,  $\widetilde{s} : \mr{Spec}(k) \migi \mcO p_{g,r}^{^\mr{Zzz...}}$)
 the 
 $k$-rational point
    classifying $\msX$ (resp., $(\msX, \msE^\spadesuit)$).
Also,  denote by $T_{s} \overline{\mcM}_{g,r}^\mr{log}$ (resp., $T_{\widetilde{s}} (\mcO p_{g,r}^{^\mr{Zzz...}})^\mr{log}$)
the $k$-vector space defined as the log tangent space of  $\overline{\mcM}_{g,r}^\mr{log}$ (resp., $(\mcO p_{g,r}^{^\mr{Zzz...}})^\mr{log}$) at $s$ (resp., $\widetilde{s}$).
 Then, there exists a canonical  isomorphism of $k$-vector spaces
 \begin{align} \label{e101}
 \mr{KS}_{(\msX, \msE^\spadesuit)} : 
 T_{\widetilde{s}} (\mcO p_{g,r}^{^\mr{Zzz...}})^\mr{log}
  \isom H^1(X, \widetilde{\mcT}_{\mcE_B^\mr{log}/k^\mr{log}, \phi})
 \end{align}
 which makes the following diagram commute:
\begin{align} \label{UU49r9}
\vcenter{\xymatrix@C=46pt@R=36pt{
 T_{\widetilde{s}} (\mcO p_{g,r}^{^\mr{Zzz...}})^\mr{log} \ar[d]^-{\wr}_-{\mr{KS}_{(\msX, \msE^\spadesuit)}}  \ar[r]^-{d \Pi_{g, r}} & 
 T_{s} \overline{\mcM}_{g,r}^\mr{log}
  \ar[d]^-{\mr{KS}_{\msX}}_-{\wr}
\\
H^1(X, \widetilde{\mcT}_{\mcE_B^\mr{log}/k^\mr{log}, \phi})\ar[r]_-{D_{\msE^\spadesuit}}& 
 H^1(X, \mcT),
}}
\end{align}
where $\mr{KS}_{\msX}$ denotes the usual  Kodaira-Spencer morphism (cf., e.g., ~\cite[Section 6.1.1]{Wak8}) and the upper horizontal arrow $d \Pi_{g, r}$ denotes  the differential of the projection $\Pi_{g, r}$ (cf. \eqref{YY144}).
\end{itemize}
\eco
\begin{proof}
The log smoothness of $(\mcO p^\ZZZ_{g,r})^\mr{log}$  and assertion (ii) follow from Theorem \ref{P47} and 
the definition of $D_{\msE^\spadesuit}$.
The smoothness of  $\mcO p^\ZZZ_{g,r}$ follows from the log smoothness of $(\mcO p^\ZZZ_{g,r})^\mr{log}$  and the fact that
the log structure of $(\mcO p^\ZZZ_{g,r})^\mr{log}$  is locally given by a free monoid $\mbN^s$ (for $s \geq 0$).
The rank  of the log  tangent bundle of $(\mcO p^\ZZZ_{g,r})^\mr{log}$  is computed  by Propositions \ref{Pee4}, (ii).
\end{proof}
\bco \label{Cor87}
The projection $\Pi_{g, r}$ is generically finite.
\eco
\begin{proof}
Let $\mcI$ be an irreducible component of $\mcO p_{g, r}^{^\mr{Zzz...}}$ that dominates $\overline{\mcM}_{g, r}$.
It follows from Corollary \ref{T34}, (i), that the dimensions of $\overline{\mcM}_{g, r}$ and $\mcI$ are identical.
Hence, the generic fiber of the projection $\mcI \rightarrow \overline{\mcM}_{g, r}$ turns out to be finite.
\end{proof}

\LSP
\subsection{Ordinary dormant $\mr{PGL}_2^{(\N)}$-opers} \label{SS0d34}

Let $S$, $\msX$, and $\msE^\spadesuit$ be as introduced  at the beginning of Section \ref{SS078}.

\SSP
\bde\label{D20} 
We shall say that $\msE^\spadesuit$ is {\bf ordinary} if  the morphism $D_{\msE^\spadesuit}$, or equivalently $D'_{\msE^\spadesuit}$,  is an isomorphism.
 \ede
\SSP



\begin{rema}[Previous definition of ordinariness] \label{WW223}
When 
$\msX$ is unpointed and smooth, 
  the definition of ordinariness
 introduced here  is equivalent to the ordinariness  of the corresponding {\it indigenous $(\mr{PGL}_2, B)$-bundle}, in the sense of  ~\cite[Definition 6.7.1]{Wak7}.
\end{rema}
\SSP

Denote by
\begin{align} \label{dE204}
\mcO p_{2, \N, g,r,  \mbF_p,  \mr{ord}}^{^\mr{Zzz...}}, \ \text{or simply} \  \mcO p_{g,r, \mr{ord}}^{^\mr{Zzz...}},
\end{align}
the open substack   of $\mcO p^\ZZZ_{g, r}$ (i.e.,  $\mcO p_{2, \N, g,r, \mbF_p}^{^\mr{Zzz...}}$) classifying {\it ordinary} dormant $\mr{PGL}_2^{(\N)}$-opers.
According to Corollary \ref{T34},
this stack coincides with the log  \'{e}tale locus  of $(\mcO p_{g,r}^{^\mr{Zzz...}})^\mr{log}$ relative to $\overline{\mcM}_{g,r}^\mr{log}$, or equivalently, coincides with  
the \'{e}tale locus  of $\mcO p_{g,r}^{^\mr{Zzz...}}$ relative to $\overline{\mcM}_{g,r}$.

\SSP
\bpr \label{P3990}
For $\N' \leq \N$,
the projection $\Pi_{\N \Rightarrow \N'} : \mcO p^\ZZZ_{2, \N, g,r, \mbF_p}\migi \mcO p^\ZZZ_{2, \N', g,r, \mbF_p}$ obtained by reducing the level  of dormant $\mr{PGL}_2^{(\N)}$-opers to $\N'$ restricts to a morphism 
$\mcO p^\ZZZ_{2, \N, g,r, \mbF_p, \mr{ord}} \migi \mcO p^\ZZZ_{2, \N',  g, r, \mbF_p,  \mr{ord}}$.
\epr
\begin{proof}
Let $q$ be a geometric point of $\mcO p^\ZZZ_{2, \N, g,r, \mbF_p, \mr{ord}}$, which classifies a pair $(\msX, \msE^\spadesuit)$ with $\msE^\spadesuit := (\mcE_B, \phi)$.
Denote by $\msE^{\spadesuit (\N')} := (\mcE_B, \phi^{(\N'-1)})$ the dormant $\mr{PGL}_2^{(\N')}$-oper on $\msX$  obtained from $\msE^\spadesuit$ via level reduction.
By the definition of $D_{\msE^\spadesuit}$, the diagram
\begin{align} \label{UU499}
\vcenter{\xymatrix@C=46pt@R=36pt{
 H^1 (X, \widetilde{\mcT}_{\mcE^\mr{log}_B/k^\mr{log}, \phi})  \ar[dr]_-{D_{\msE^\spadesuit}} \ar[rr]&& H^1 (X, \widetilde{\mcT}_{\mcE^\mr{log}_B/k^\mr{log}, \phi^{(\N'-1)}}) \ar[dl]^-{D_{\msE^{\spadesuit (\N')}}} \\
 &H^1 (X, \mcT)&
 }}
\end{align}
is commutative, where the horizontal arrow is the morphism induced by the inclusion $\widetilde{\mcT}_{\mcE^\mr{log}_B/k^\mr{log}, \phi} \hookrightarrow \widetilde{\mcT}_{\mcE^\mr{log}_B/k^\mr{log}, \phi^{(\N'-1)}}$.
Note that $D_{\msE^\spadesuit}$ is an isomorphism by assumption,  and that  both the $k$-vector spaces $H^1 (X, \widetilde{\mcT}_{\mcE^\mr{log}_B/k^\mr{log}, \phi})$ and $H^1 (X, \widetilde{\mcT}_{\mcE^\mr{log}_B/k^\mr{log}, \phi^{(\N'-1)}})$ have dimension $3g-3+r$ (cf. Proposition \ref{Pee4}, (ii),  and \eqref{Equu102}).
This implies that $D_{\msE^{\spadesuit (\N')}}$ is an isomorphism.
That is, $\Pi_{\N \Rightarrow \N'} (q)$ lies in $\mcO p^\ZZZ_{2, \N',  g, r, \mbF_p,  \mr{ord}}$.
This completes the proof of the assertion.
\end{proof}
\SSP

Next, 
suppose that $\msX$ is unpointed (i.e., $r = 0$) and smooth (which implies $X^\mr{log} = X$ and $S^\mr{log} = S$).
We shall write $\msG := (\mfg_\mcE, \DMO^{\mr{ad}}_\STR)$ for simplicity.
Since the equivalence of categories \eqref{YY6} commutes with the formation of duals,
the identification $(\mfg_\mcE, \DMO_\STR^\mr{ad}) \isom (\mfg_\mcE^\vee, \DMO_\STR^{\mr{ad}\vee})$ induced from the Killing form on $\mfg$ 
gives 
an identification $\mfg_{\mcE, \STR} = \mfg_{\mcE, \STR}^\vee$.
Under these identifications,
the  diagram \eqref{dEEr2} where   ``$(\mcF, \DMO)$" is taken to be $\msG$ defines   a commutative square  diagram
\begin{align} \label{dEEr5}
\vcenter{\xymatrix@C=96pt@R=36pt{
f_* (\Omega \otimes \mfg_\mcE)
\ar[r]^-{\tiny{{\Dual}_{\mr{Coker}}} \left(= f_*(C_{\msG}) \right)} \ar[d]_-{\int_{\msG}^{(0)}}^-{\wr} &
f_*^{(\N)} (\Omega^{(\N)}\otimes \mfg_{\mcE, \STR})
 \ar[d]^-{\int_{\msG}^{(\N)}}_-{\wr}
\\
\mbR^1 f_{*}(\mfg_\mcE)^\vee
\ar[r]_-{\tiny{{\Dual}_{\mr{Ker}}}} &
\mbR^1 f_*^{(\N)}(\mfg_{\mcE, \STR})^\vee.
}}
\end{align}

Let us consider the composite
\begin{align}
\Omega^{\otimes 2} 
\migiincl \Omega\otimes \mfg_\mcE 
\xrightarrow{C_\msG} \Omega^{(\N)}\otimes \mfg_{\mcE, \STR},
\end{align}
 where the first arrow denotes the morphism induced by the inclusion $\mfg^1_\mcE \migiincl \mfg_\mcE$ under  the natural identification $\left(\mfg^1_\mcE/\mfg^2_\mcE =  \right) \mfg_\mcE^1  =   \Omega$ (cf. ~\cite[Eq.\,(617)]{Wak8}).
The direct image via $f$ of this composite defines 
 an $\mcO_S$-linear morphism
\begin{align} \label{NN25}
D''_{\msE^\spadesuit} : f_*(\Omega^{\otimes 2}) \migi f_*(\Omega^{(\N)} \otimes \mfg_{\mcE, \STR}).
\end{align}

\SSP
\bpr \label{P29}
The following square diagram is commutative up to multiplication by an invertible factor in $\mbF_p$:
\begin{align}
\vcenter{\xymatrix@C=46pt@R=36pt{
 f_*(\Omega^{\otimes 2}) \ar[d]_-{\int}^-{\wr}  \ar[r]^-{D''_{\msE^\spadesuit}} & f_* (\Omega^{(\N)} \otimes \mfg_{\mcE, \phi})   \ar[d]_{\wr}^{\int^{(\N )}_\msG}
\\
\mbR^1 f_*(\mcT)^\vee\ar[r]_{(D'_{\msE^\spadesuit})^\vee}&
\mbR^1 f_*(\mfg_{\mcE, \phi})^\vee,
}}
\end{align}
where $\int$ denotes the isomorphism obtained by applying 
  Grothendieck-Serre duality to $\Omega$.
In particular, $\msE^\spadesuit$ is ordinary if and only if $D''_{\msE^\spadesuit}$ is an isomorphism.
\epr
\begin{proof}
Consider 
 the diagram
\begin{align}
\vcenter{\xymatrix@C=46pt@R=36pt{
f_* (\Omega^{\otimes 2})\ar[r] \ar[d]^-{\wr}_-{\int} & f_* (\Omega \otimes \mfg_\mcE) \ar[d]_-{\wr}^-{\int_\msG^{(0)}} \\
\mbR^1 f_* (\mcT)^\vee\ar[r] & \mbR^1  f_* (\mfg_\mcE)^\vee,
}}
\end{align}
where the upper and lower horizontal arrows arise from the isomorphisms $\Omega \isom \mfg_\mcE^1 \left(\migiincl \mfg_\mcE \right)$ and $\left(\mfg_\mcE \migisurj \right) \mfg_\mcE /\mfg_\mcE^0 \isom \mcT$, respectively,  defined in ~\cite[Eq.\,(617)]{Wak8}.
By the definition of the Killing form on $\mfg$ (i.e., $(a, b) \mapsto \frac{1}{4} \cdot \mr{tr}(\mr{ad} (a)\cdot \mr{ad} (b)) \left(= \mr{tr}(ab) \right)$ for $a, b \in \mfg$), this diagram
is commutative up to multiplication by an invertible factor.
Hence, the assertion follows from the commutativity (in this sense)  of this diagram  and  that of \eqref{dEEr5}.
The second  assertion follows  from the first  assertion.
\end{proof}
\SSP

\LSP
\subsection{Dormant $\mr{PGL}_2^{(\N)}$-opers on  a $3$-pointed projective line} \label{SS042e}

The goal of this section is to show that the stack $\mcO p^\ZZZ_{\rho, 0, 3}$ defined for each $\rho \in ((\mbZ/p^\N \mbZ)^\times /\{\pm \})^{\times 3}$ is, if it is nonempty, isomorphic to $\mr{Spec}(\mbF_p)$ (cf. Theorem \ref{Per33}).
Since we have already obtained  the finiteness and  smoothness of the moduli space (cf. Theorem \ref{P91} and Corollary \ref{T34}, (i)), 
the problem is reduced to proving
that $\mcO p^\ZZZ_{\rho, 0, 3}$ consists of exactly one point in the  set-theoretic sense.
So the desired assertion follows  directly from ~\cite[Proposition 7.4.1]{Wak9}.
However, in this subsection,   we give an alternative (somewhat simpler) proof of ``$\mcO p^\ZZZ_{\rho, 0, 3} \cong \mr{Spec}(\mbF_p)$" without applying  the finiteness and smoothness, which were actually proved by technically   complicated arguments.
We do so by slightly generalizing the discussion in ~\cite{Wak9}.

We shall consider the case where $\msX$ is 
 the $3$-pointed projective line $\msP := (\mbP/S, \{ [0], [1], [\infty] \})$,  introduced in \eqref{1051}.
 Let $k$ be an algebraically closed field over $\mbF_p$, and 
suppose  that $S = \mr{Spec}(R)$ for 
a local Artinian $k$-algebra $R$ with residue field $k$.
To clarify the base space ``$\mr{Spec}(R)$",
we occasionally write $\mbP_R$ and  $\msP_R$ instead of $\mbP$ and  $\msP$, respectively.

Let us take a triple $(\rho_0, \rho_1, \rho_\infty)$ of elements of $(\mbZ /p^{\N} \mbZ)^{\times}/\{ \pm 1 \}$.
Then, there exists  a triple of integers $(\lambda_0, \lambda_1, \lambda_\infty)$   satisfying the following conditions:
\begin{itemize}
\item
$2  \cdot \rho_x = \lambda_x$ as elements of  $(\mbZ/p^\N \mbZ)/\{ \pm 1 \}$ and 
$0 < \lambda_x < p^{\N}$
 for every  $x=0,1, \infty$;
\item
The sum $\lambda_0 + \lambda_1 + \lambda_\infty$ is odd $< 2 \cdot p^{\N}$.
\end{itemize}
Set  $\mcO_{\mbP}^+ := \mcO_{\mbP} (\lambda_0 \cdot [0]+ \lambda_1 \cdot [1] + \lambda_{\infty} \cdot [\infty])$.
Note that we can find a unique $\mcD^{(\N -1)}$-module structure  $\DMO^+$ on $\mcO_{\mbP}^+$ whose restriction to    $U := \mbP \setminus \{ [0], [1], [\infty] \}$ coincides with  the trivial $\mcD_{U/S}^{(\N-1)}$-module structure  on $\mcO_{\mbP}^+ |_U = \mcO_U$.
Also, let $\mcL$ be a line bundle on  $\mbP$ of relative degree $\frac{\lambda_0 + \lambda_1 + \lambda_\infty +1}{2}$, which is uniquely determined up to isomorphism because $R$ is a  local ring.
Let us fix an identification  $\mcL^{\otimes 2} \otimes \mcT = \mcO_{\mbP}^+$.
Under this identification, the pair $\vartheta := (\mcL, \DMO^+)$ form a dormant $2^{(\N)}$-theta characteristic of $\mbP^\mr{log}/S$.

Now, let $\msE^\spadesuit$ be a dormant $\mr{PGL}_2^{(\N)}$-oper on $\msP$ of  radii $(\rho_0, \rho_1, \rho_\infty)$, i.e., whose radius at $[x]$ ($x \in \{ 0, 1, \infty \}$) coincides with $\rho_x$.
We shall write
\begin{align} \label{NN31}
\msF^{\heartsuit} : = (\mcF, \DMO, \{ \mcF^j \}_{j=0}^2)
\end{align}
for  
the dormant $\mr{GL}_2^{(\N)}$-oper  determined by 
a unique  (up to isomorphism)
dormant $(\mr{GL}_2^{(\N)}, \vartheta)$-oper $\DMO^\diamondsuit$ on $\msP$ 
corresponding to $\msE^\spadesuit$ via the isomorphism $\Lambda_{\diamondsuit \Rightarrow \spadesuit}^\ZZZ$ (cf. Theorem \ref{P14}).
That is to say, we set $\msF^\heartsuit := \DMO^{\diamondsuit \Rightarrow \heartsuit}$ (hence $\mcF = \mcF_\mcL$ and  $\DMO = \DMO^{\diamondsuit}$).
In particular, since $\mcF^1 = \mcL$, we can regard $\mcL$ as a line subbundle of $\mcF$.

\SSP
\bde \label{NN49}
We shall refer to $\msF^\heartsuit$ as the {\bf canonical dormant $\mr{GL}_2^{(\N)}$-oper} associated to $\msE^\spadesuit$.
(Note that the formation of $\msF^\heartsuit$ depends on the choice of $(\lambda_0, \lambda_1, \lambda_\infty)$, but its isomorphism class does not depend on the choice of the identification $\mcL^{\otimes 2} \otimes \mcT = \mcO_{\mbP}^+$.)
\ede
\SSP

\begin{rema}[Previous works on canonical $\mr{GL}_2^{(\N)}$-opers] \label{NN54}
The  construction of the  canonical $\mr{GL}_2$-opers was discussed  in ~\cite[Chapter IV, Section 2.1]{Mzk2} (for $S = \mr{Spec}(k)$ and $\N =1$) and ~\cite[Section 7.3]{Wak9} (for $S = \mr{Spec}(k)$) to 
 establish a bijective  correspondence between dormant $\mr{PGL}_2$-opers on $\msP$ and certain tamely ramified coverings on $\msP$ (i.e., dynamical Belyi maps).
\end{rema}

We shall write  $\mcF^\flat:=  F^{(\N)*}_{\mbP/S}(\mcS ol (\DMO))$ and write
\begin{align} \label{NN30}
\tau : \mcF^\flat 
  \migiincl  \mcF
\end{align}
for  the $\mcO_{\mbP}$-linear injection 
 extending the
  $\mcO_{\mbP^{(\N)}}$-linear  inclusion $\mcS ol (\DMO) \migiincl \mcF$;
 this morphism commutes with the $\mcD^{(\N -1)}$-module structures $\DMO^{(\N -1)}_{\mcS ol (\DMO), \mr{can}}$ (cf. \eqref{QQwkko}) and  $\DMO$.
Next, we shall set $\mcL^\flat := \mcL \cap \mr{Im}(\tau)$.
Since the restriction of $\tau$ to $U := \mbP \setminus \{ [0], [1], [\infty] \}$ is an isomorphism (cf. \eqref{YY6}),
the quotient sheaf $\mcL/\mcL^\flat$ is a torsion sheaf supported on $\{ [0], [1], [\infty] \}$.

\bpr \label{P100}
Let us keep the above notation.
\begin{itemize}
\item[(i)]
$\mcL/\mcL^\flat$ is flat over $S$, and 
the stalk of  $\mcL/\mcL^\flat$ at the marked point $x \in \{[0], [1], [\infty] \}$ is 
 of rank   $\lambda_x$ (as a free $\mcO_S$-module).
 Moreover, the injection 
 \begin{align} \label{NN40}
 \iota : \mcL/\mcL^\flat \migiincl \mr{Coker}(\tau)
 \end{align}
  induced from the inclusion $\mcL \migiincl \mcF$ is an isomorphism.
\item[(ii)]
There exists an $\mcO_{\mbP^{(\N)}}$-linear isomorphism $\gamma :  \mcO_{\mbP^{(\N)}}^{\oplus 2} \isom \mcS ol (\DMO)$.
\end{itemize}
\epr
\begin{proof}
In what follows, we shall use the notation $(-)_k$ to denote the result of reducing objects and morphisms over $R$  modulo $\mfm_R$.

First, let us consider assertion (i).
By Proposition \ref{Prop65} (and the isomorphism $\Lambda_{\diamondsuit \Rightarrow \spadesuit}^\ZZZ$), $(\mcF, \DMO)$ is locally 
isomorphic to the base-change of $(\mcF_k, \DMO_k)$ over $R$.
Hence,  $\mr{Coker}(\tau)$ can be locally identified with the base-change  of  $\mr{Coker}(\tau_k)$.
This implies that $\mr{Coker}(\tau)$ is flat over $S$.
Recall from  ~\cite[Lemma 7.3.2]{Wak9} that
the reduction modulo $\mfm_R$ of 
\eqref{NN40}
 is an isomorphism.
By  Nakayama's lemma,  the morphism
\eqref{NN40}
    turns out to be  an isomorphism.
In particular, $\mcL/\mcL^\flat$ is flat  over $S$, and the rank of  its stalk  at $x \in \{ [0], [1], [\infty]\}$  (as a free $\mcO_S$-module)
coincides with the rank of its reduction modulo $\mfm_R$, i.e., $\lambda_x$.
This completes the proof of assertion (i).

Next, we shall consider assertion (ii).
Since $(\mcF, \DMO)$ is locally isomorphic to the base-change of $(\mcF_k, \DMO_k)$ (as mentioned above),
the vector bundle  $\mcS ol (\DMO)$ on $\mbP^{(\N)}$ is locally isomorphic to the base-change  of $\mcS ol (\DMO_k)$. 
Hence,  $\mcS ol (\DMO)$ specifies a deformation of $\mcS ol (\DMO_k)$ over $R$.
Here, let us observe the following fact from  well-known generalities of  deformation theory: if 
we are given a small extension in 
the category of local Artinian $k$-algebras 
of the form $0 \migi k \migi R_1 \migi R_0 \migi 0$ and a deformation 
$\mcS ol (\DMO_k)_{R_0}$  of $\mcS ol (\DMO_k)$ over $R_0$, then 
the  set of deformations of $\mcS ol (\DMO_k)_{R_0}$ over $R_1$  has a structure of torsor modeled on $H^1 (\mbP_k^{(\N)}, \mcE nd (\mcS ol (\DMO_k)))$.
On the other hand,  $\mcS ol (\DMO_k)$ is  isomorphic to
$\mcO_{\mbP_k^{(\N)}}^{\oplus 2}$ as proved in
 ~\cite[Lemma 7.3.3]{Wak9},
 so we have 
\begin{align}
H^1 (\mbP_k^{(\N)}, \mcE nd (\mcS ol (\DMO_k))) \cong H^1 (\mbP^{(\N)}_k, \mcE nd (\mcO_{\mbP_k^{(\N)}}^{\oplus 2})) \cong H^1 (\mbP_k^{(\N)}, \mcO_{\mbP_k^{(\N)}})^{\oplus 4} = 0.
\end{align}
This implies that 
any deformation of  $\mcS ol (\DMO_k)$ over $R$  is trivial.
In particular, there exists an isomorphism  $\mcO_{\mbP^{(\N)}}^{\oplus 2} \isom \mcS ol (\DMO)$.
\end{proof}
\SSP

The following assertion is a slight generalization of ~\cite[Proposition 7.4.1]{Wak9}.
(The proof of the corresponding proposition in the cited reference contains a gap; therefore, it should be replaced by the proof presented here.)

\SSP
\bpr \label{P00245}
Let $\rho := (\rho_0, \rho_1, \rho_\infty)$ be   an element of $((\mbZ/p^\N \mbZ)^\times /\{\pm 1 \})^{\times 3}$.
Then, any two dormant $\mr{PGL}_2^{(\N)}$-opers on $\msP$ of radii $\rho$ are isomorphic.
\epr
\begin{proof}
Suppose that we are given  two dormant $\mr{PGL}_2^{(\N)}$-opers   $\msE_\circ^\spadesuit$,  $\msE^\spadesuit_\bullet$ on $\msP$ of radii $\rho$.
For each $\star \in \{ \circ, \bullet \}$,
denote by $\nabla_\star^{\diamondsuit}$ the dormant $(\mr{GL}_2^{(\N)}, \vartheta)$-oper corresponding to $\msE^\spadesuit_\star$.
For $\N' \leq \N$,
The level $\N'$ reduction of $\nabla_\star^{\diamondsuit}$ determines a dormant $(\mr{GL}_2^{(\N')}, \vartheta^{(\N')})$-oper
$\nabla_\star^{\diamondsuit (\N')}$,
where 
$\vartheta^{(\N')}$ denote the $2^{(\N')}$-theta characteristic induced by $\vartheta \left(= (\mcL, \nabla^+) \right)$.
According to  ~\cite[Chapter I, Theorem 4.4]{Mzk2},    dormant $\mr{PGL}_2$-opers on $\msP$ are  completely determined by their radii.
In particular, since the dormant $\mr{PGL}_2$-opers corresponding to $\nabla_\circ^{\diamondsuit (1)}$  and $\nabla_\bullet^{\diamondsuit (1)}$ are isomorphic,
we obtain an isomorphism  $\nabla_\circ^{\diamondsuit (1)} \xrightarrow{\sim} \nabla_\bullet^{\diamondsuit (1)}$.
By considering the pull-back of $\nabla_\bullet^\diamondsuit$ via this isomorphism,
we may assume that $\nabla_\circ^{\diamondsuit (1)} = \nabla_\bullet^{\diamondsuit (1)}$.

In what follows,  we prove by induction on $\N'$ the equality $\nabla_\circ^{\diamondsuit (\N')} = \nabla_\bullet^{\diamondsuit (\N')}$.
The base step, i.e.,  the case $\N' =1$,  has just been  proved.
  To complete the proof of   the induction step, it suffices to verify the equality $\nabla_\circ^{\diamondsuit (\N)} = \nabla_\bullet^{\diamondsuit (\N)}$ under the assumption  that  the case $\N' = \N -1$ has been proved. 

Suppose that $\nabla_\circ^{\diamondsuit (\N -1)} = \nabla_\bullet^{\diamondsuit (\N -1)} =: \nabla$.
Let $x \in \{ [0], [1], [\infty]\}$.
For $\N' \leq \N$, we set $M_{\N'} := \mu (\nabla_{\oslash, 0}^{(\N' -1)} \oplus \nabla_{\oslash, \overline{\lambda}_x}^{(\N' -1)})^{\langle p^{\N'-1}\rangle}$,
which coincides with the diagonal matrix  with diagonal entries $0$, $(-\overline{\lambda}_x)_{[\N' -1]}$ (cf. Proposition \ref{L090}),  where $\overline{\lambda}_x$ denotes the mod $p^\N$ reduction of $\lambda_x$.
Then,
the equality $M_{\N} = \frac{(-\lambda_x)_{[\N -1]}}{(-\lambda_x)_{[0]}} \cdot M_1$ holds.
Since the set $\{ 0, \overline{\lambda}_x \}$  is the exponents of both $\nabla_\circ^{\diamondsuit (\N)}$ and $\nabla_\bullet^{\diamondsuit (\N)}$ at $x$,
their  $(\N' -1)$-th monodromy operators $\mu_x (\nabla_\circ^{\diamondsuit (\N)})^{\langle p^{\N' -1} \rangle}$ and $\mu_x (\nabla_\bullet^{\diamondsuit (\N)})^{\langle p^{\N' -1} \rangle}$ at $x$ are conjugate to 
$M_{\N'}$, and a conjugating matrix can be chosen as a matrix independent of $\N'$.
Hence, the equality $\nabla_\circ^{\diamondsuit (1)} = \nabla_\bullet^{\diamondsuit (1)}$ together with $M_{\N} = \frac{(-\lambda_x)_{[\N -1]}}{(-\lambda_x)_{[0]}} \cdot M_1$ 
implies 
\begin{align} \label{EE854}
\mu_x (\nabla_\circ^{\diamondsuit (\N)})^{\langle p^{\N-1} \rangle} = \mu_x (\nabla_\bullet^{\diamondsuit (\N)})^{\langle p^{\N -1} \rangle}.
\end{align}

Next, let us consider the difference $h := (\nabla_\circ^{\diamondsuit (\N)} - \nabla_\bullet^{\diamondsuit (\N)}) |_{\mcD^{(\N -1)}_{\leq p^{\N -1}}} : \mcD^{(\N -1)}_{\leq p^{\N -1}} \rightarrow  \mcE nd_{\mcO_S} (\mcF_\mcL)$.
Since the natural morphism $\mcD_{\leq  j}^{(\N -2)} \rightarrow \mcD_{\leq j}^{(\N -1)}$ is an isomorphism for $j \leq p^{\N -1}-1$,
 the inductive hypothesis shows that $h$ factors through the quotient  $\mcD^{(\N -1)}_{\leq p^{\N -1}} \twoheadrightarrow \mcT^{\otimes (p^{\N -1})}$.
 By \eqref{EE665}, the image of  the resulting morphism 
 $\overline{h} : \mcT^{\otimes (p^{\N -1})} \rightarrow \mcE nd_{\mcO_S} (\mcF_\mcL)$ lies in $\mcE nd (\mcF_\mcL) := \mcE nd_{\mcO_\mbP} (\mcF_\mcL)$.
 The line bundle $\mcT^{\otimes (p^{\N -1})}$ carries the  $\mcD^{(\N -2)}$-module structure $\nabla_\mr{can} := \nabla_{\mcT^{(\N -1)}, \mr{can}}^{(\N -2)}$ (cf. \eqref{QQwkko}) under the natural identification $\mcT^{\otimes (p^{\N -1})} \cong F^{(\N -1)*}_{\mbP/S} (\mcT^{(\N -1)})$.
 On the other hand, let $\nabla_{\mcE nd}$ denote the  $\mcD^{(\N -2)}$-module structure on $\mcE nd (\mcF_\mcL)$ induced by $\nabla$.
The commutative relation \eqref{Equu88} implies that $\overline{h}$ defines a morphism of $\mcD^{(\N -2)}$-modules 
$(\mcT^{\otimes (p^{\N -1})}, \nabla_{\mr{can}}) \rightarrow (\mcE nd(\mcF_\mcL), \nabla_{\mcE nd})$.
Moreover, if $\{ \partial^{\langle j \rangle} \}_j$ is the local basis of $\mcD^{(\N -1)}$ associated to a logarithmic coordinate at $x$, then 
\begin{align}
\overline{h} (\partial^{\langle p^{\N -1}\rangle})|_x = \mu_x (\nabla_\circ^{\diamondsuit (\N)})^{\langle p^{\N-1} \rangle} -\mu_x (\nabla_\bullet^{\diamondsuit (\N)})^{\langle p^{\N -1} \rangle} \stackrel{\eqref{EE854}}{=}0.
\end{align}
This implies that  $\overline{h}$ turns out to be a section of $\mcH om_{\mcO_{\mbP}} (\mcT^{\otimes p^{\N -1}}, \mcE nd (\mcF_\mcL) (-D))$,
where $D$ the effective reduced divisor on $\mbP$ defined as the sum of the marked points $[0] + [1] + [\infty]$.

By  Lemma \ref{Lem444} proved below, $\overline{h}$ must be  the zero map.
Since $\mcD^{(\N -1)}$ is generated by the sections of $\mcD^{(\N -1)}_{\leq p^{\N -1}}$ as an $\mcO_\mbP$-algebra, 
the equality $\nabla_\circ^{\diamondsuit (\N)} = \nabla_\bullet^{\diamondsuit (\N)}$ holds.
This completes the proof of the induction step, and we finish the proof of this proposition.
  \end{proof}
\SSP

\ble \label{Lem444}
Let us keep the notation in the proof of the above proposition.
Denote by $\mr{Hom}_\mcD (\mcT^{\otimes (p^{\N -1})}, \mcE nd (\mcF_\mcL))$ the $R$-submodule of $\mr{Hom}_{\mcO_\mbP} (\mcT^{\otimes (p^{\N -1})}, \mcE nd (\mcF_\mcL))$ consisting of morphisms $\mcT^{\otimes (p^{\N -1})} \rightarrow  \mcE nd (\mcF_\mcL)$ compatible with the respective $\mcD^{(\N -2)}$-module structures $\nabla_\mr{can}$, $\nabla_{\mcE nd}$.  
Then, we have an equality 
\begin{align} \label{EE4409}
\mr{Hom}_\mcD (\mcT^{\otimes (p^{\N-1})}, \mcE nd (\mcF_\mcL)) \cap \mr{Hom}_{\mcO_\mbP} (\mcT^{\otimes (p^{\N-1})}, \mcE nd (\mcF_\mcL) (-D)) =0. 
\end{align}
\ele
\begin{proof}
Let us fix $x \in \{[0], [1], [\infty] \}$, and choose a local parameter  $t$ at $x$.
Write $x' := F_{\mbP/S}^{(\N -1)} (x)$, and denote by $u$ the local parameter  at $x'$ corresponding to $t$, so that $F^{(\N -1)*}_{\mbP/S} (u) = t^{p^{\N -1}}$.
The restriction of $\mcS ol (\nabla_{\mcE nd})$ to the formal neighborhood $\widehat{U}_x$ of $x$ decomposes, after choosing a suitable local basis of $\mcF_\mcL$, as
\begin{align}
\mcS ol (\nabla_{\mcE nd}) |_{\widehat{U}_x} = R [\![u]\!] e_{11} \oplus R [\![u]\!] e_{22} \oplus R [\![u]\!] t^{r} e_{12} \oplus R [\![u ]\!] t^{p^{\N -1} -r}e_{21}
\end{align}
for some integer $r$ with $0 < r < p^{\N -1}$, where $e_{ij}$ denotes the matrix units.
Thus, a local section $b$ of $\Omega^{(\N -1)} \otimes \mcS ol (\nabla_{\mcE nd}) \left(\cong \mcH om_{\mcO_{\mbP^{\N -1}}} (\mcT^{(\N -1)}, \mcS ol (\nabla_{\mcE nd})) \right)$ can be expressed as 
\begin{align}
b = \frac{du}{u} \otimes  \begin{pmatrix} A (u) & t^r B (u) \\ t^{p^{\N -1}-r} C (u) & D (u)\end{pmatrix}
\end{align}
for $A (u), B (u), C (u), D (u) \in R [\![ u]\!]$.
Here, let us consider the composite
\begin{align}
D_x :  (\Omega^{(\N -1)} \otimes \mcS ol (\nabla_{\mcE nd})) |_{x'} \xrightarrow{\sim} \mcS ol  (\nabla_{\mcE nd}) |_{x'}
\xrightarrow{\sim} F^{(\N -1)*}_{\mbP/S} (\mcS ol (\nabla_{\mcE nd})) |_{x} \rightarrow \mcE nd (\mcF_\mcL) |_{x},
\end{align}
where the first arrow arises from the residue isomorphism $\Omega^{(\N -1)} |_{x'} \cong \mcO_S$, and the last arrow arises from
the inclusion $F^{(\N -1)*}_{\mbP/S} (\mcS ol (\nabla_{\mcE nd}))  \hookrightarrow \mcE nd (\mcF_\mcL)$.
The image of the fiber $b |_{x'}$ of $b$ under this composite is given by  $\begin{pmatrix} A (0) & 0 \\ 0 & D (0)\end{pmatrix}$.
We set  $\mcG$ to be  the $\mcO_{\mbP^{(\N -1)}}$-submodule of $\Omega^{(\N -1)} \otimes \mcS ol (\nabla_{\mcE nd})$ consisting of such sections $b$ with $D_x (b |_{x'}) =0$ for any $x \in \{[0], [1], [\infty] \}$.
Then, $\mcG$ is a vector bundle of rank $4$, and locally at the point  $x$ as above, it has a   basis 
\begin{align}
du \otimes e_{11}, \hspace{3mm}
du \otimes e_{22},  \hspace{3mm} \frac{du}{u} \otimes t^r e_{12}, \hspace{3mm} \frac{du}{u} \otimes t^{p^{\N -1} -r} e_{21}.
\end{align}
The trace pairing $\mcG \otimes \mcS ol (\nabla_{\mcE nd}) \rightarrow \Omega_{\mbP^{(\N -1)}/S}$ is verified to be nondegenerate.
Indeed, with respect to the displayed bases, its matrix representation, using $du$ as a basis of $\Omega_{\mbP^{(\N -1)}/S}$,  is given by
\begin{align}
\begin{pmatrix}
1 & 0 & 0 & 0 \\
0 & 1 & 0 & 0 \\
0 & 0 & 0 & 1 \\
0 & 0 & 1 & 0
\end{pmatrix}.
\end{align}
The nondegeneracy of the pairing away from the marked points follows from the usual trace pairing and Cartier descent.
Consequently, we have $\mcG \cong \Omega_{\mbP^{(\N -1)}/S}\otimes \mcS ol (\nabla_{\mcE nd})^\vee$.

Next, 
let  $\mcE nd^0 (\mcF_\mcL)$ denote the subbundle of $\mcE nd (\mcF_\mcL)$ consisting of endomorphisms with vanishing trace.
The $\mcD^{(\N -2)}$-module structure $\nabla_{\mcE nd}$ restricts to a $\mcD^{(\N -2)}$-module structure $\nabla_{\mcE nd^0}$ on $\mcE nd^0 (\mcF_\mcL)$.
Since $p >2$, there exists a natural decomposition  $(\mcE nd (\mcF_\mcL, \nabla_{\mcE nd})) = (\mcE nd^0 (\mcF_\mcL), \nabla_{\mcE nd^0}) \oplus (\mcO_\mbP, \nabla^{(\N -2)}_{\mbP, \mr{triv}})$.
This decomposition yields  $\mcS ol (\nabla_{\mcE nd}) = \mcS ol (\nabla_{\mcE nd^0}) \oplus \mcO_{\mbP^{(\N -1)}}$.
It follows from Proposition \ref{Pee4}, (ii), that $H^1 (\mbP^{(\N -1)}, \mcS ol (\nabla_{\mcE nd^0})) =0$.
Hence,  by Serre duality, we have
\begin{align} \label{EE4312}
&  \ \ \ \ H^0 (\mbP^{(\N -1)}, \mcG) \\
&\cong H^0 (\mbP^{(\N -1)}, \Omega_{\mbP^{(\N -1)}/S} \otimes \mcS ol (\nabla_{\mcE nd})^\vee)\\
&\cong H^1 (\mbP^{(\N -1)},  \mcS ol (\nabla_{\mcE nd}))^\vee\\
&\cong H^1 (\mbP^{(\N -1)},  \mcS ol (\nabla_{\mcE nd^0}))^\vee \oplus H^1 (\mbP^{(\N -1)}, \mcO_{\mbP^{(\N -1)}})^\vee \\
& = 0.
\end{align}

Each morphism of $\mcD^{(\N -2)}$-modules  $a : (\mcT^{\otimes (p^{\N-1})}, \nabla_\mr{can})\rightarrow (\mcE nd (\mcF_\mcL), \nabla_{\mcE nd})$ restricts to an $\mcO_{\mbP^{(\N -1)}}$-linear morphism $a^\nabla : \mcT^{(\N -1)} \rightarrow \mcS ol (\nabla_{\mcE nd})$, and the assignment $a \mapsto a^\nabla$ defines an injection 
\begin{align}
&\mr{Hom}_\mcD (\mcT^{\otimes (p^{\N-1})}, \mcE nd (\mcF_\mcL)) \\
& \hookrightarrow \mr{Hom}_{\mcO_{\mbP^{(\N -1)}}} (\mcT^{(\N -1)}, \mcS ol (\nabla_{\mcE nd})) \left(\cong H^0 (\mbP^{(\N -1)}, \Omega^{(\N -1)} \otimes \mcS ol (\nabla_{\mcE nd})) \right). \notag
\end{align} 
Hence, the left-hand side of \eqref{EE4409} is contained in $H^0 (\mbP^{(\N -1)}, \mcG)$, and   \eqref{EE4312} implies the desired equality.
\end{proof}

\bt \label{Per33}
Let $\rho \in ((\mbZ/p^\N \mbZ)^\times /\{ \pm 1 \})^{\times 3}$.
Then, the   stack $\mcO p^\ZZZ_{\rho, 0, 3}$ is either empty or  isomorphic to $\mr{Spec}(\mbF_p)$.
In particular, any dormant $\mr{PGL}_2^{(\N)}$-oper on $\msP$ is ordinary, i.e., we have $\mcO p^\ZZZ_{0, 3, \mr{ord}} = \mcO p^\ZZZ_{0, 3}$. 
\et
\begin{proof}
The assertion follows from Proposition \ref{P00245}.
\end{proof}
\SSP

\bco \label{NN812}
For each $\N \in \mbZ_{>0}$, denote by $\pi_0 (\mcO p^\ZZZ_{2, \N, g,r, \mbF_p})$
the set of connected components of  $\mcO p^\ZZZ_{2, \N, g,r, \mbF_p}$.
Let us consider  the projective system of sets
\begin{align} \label{UU834}
\cdots \migi \pi_0(\mcO p^\ZZZ_{2, \N, g,r, \mbF_p}) \migi \cdots \migi \pi_0(\mcO p^\ZZZ_{2, 2, g,r, \mbF_p}) \migi \pi_0 (\mcO p^\ZZZ_{2, 1, g,r, \mbF_p})
\end{align}
arising from \eqref{NN8123}.
Then, the projective limit 
\begin{align} \label{UU835}
\varprojlim_{\N \in \mbZ_{>0}} \pi_0 (\mcO p^\ZZZ_{2, \N, g,r, \mbF_p})
\end{align}
of this system is nonempty. 
\eco
\begin{proof}
Let us take an $r$-pointed totally degenerate curve $\msX$ of genus $g$ over an algebraically closed field $k$ in characteristic $p$.
Choose trivalent 
clutching data $\mbG := (\GR,  \{ \lambda_j \}_{j=1}^J)$
corresponding to  $\msX$ (cf. Definition \ref{DD3WW}).
In particular,
 $\msX$ may be obtained by gluing together $J$ copies of $\msP$ by means of $\mbG$.
Denote by $\rho_\mbG := \{ \rho^j \}_{j=1}^J$ the set of $\mbG$-$\Xi_{2, \N}$-radii defined by $\rho^j := \varepsilon^{\times 3}$ (cf. \eqref{NN8122}).
By Theorem \ref{T52} and the equivalence of categories \eqref{YY391},
there exists   a unique (up to isomorphism) dormant  $\mbG\text{-}\mr{PGL}_2^{(\N)}$-oper  of radii $\rho_\mbG$ on $\msX$;
we shall denote by $\msE^\spadesuit_{\msX, \N}$  the corresponding dormant $\mr{PGL}_2^{(\N)}$-oper on $\msX$ via the equivalence of categories \eqref{YY291}.
Also,  denote by $\mcR_\N$ the connected component of $\mcO p^\ZZZ_{2, \N, g,r, \mbF_p}$ containing the point $s_\N$ classifying $(\msX, \msE^\spadesuit_{\msX, \N})$.
For each $\N' \leq \N$, the equality $q_{\N \Rightarrow \N'} (\varepsilon^{\times 3}) = \varepsilon^{\times 3}$ holds  (cf. \eqref{NN8121} for the definition of $q_{\N \Rightarrow \N'}$).
This implies that the dormant $\mr{PGL}_2^{(\N')}$-oper obtained from $\msE^\spadesuit_{\msX, \N}$ by reducing its level is isomorphic to $\msE^\spadesuit_{\msX, \N'}$.
In other words, the projection $\mcO p^\ZZZ_{2, \N, g,r, \mbF_p} \migi \mcO p^\ZZZ_{2, \N', g, r, \mbF_p}$ sends $s_\N$ to $s_{\N'}$.
Since the stacks $\mcR_\N$ are  irreducible because of their smoothness  proved in Corollary  \ref{T34}, (i),
the projective system \eqref{NN8123} restricts to a projective system
\begin{align}
\cdots \migi \mcR_\N \migi \cdots \migi \mcR_2 \migi \mcR_1.
\end{align}
In particular,  this system defines an element of the set $\varprojlim_{\N} \pi_0 (\mcO p^\ZZZ_{2, \N, g,r, \mbF_p})$, so we conclude that $\varprojlim_{\N} \pi_0 (\mcO p^\ZZZ_{2, \N, g,r, \mbF_p}) \neq \emptyset$, as desired.
 \end{proof}

\LSP
\subsection{Certain finiteness properties of the moduli space} \label{SS300}

In this subsection, we prove that the projection $\Pi_{g, r}$ for $n=2$ is finite  when restricted over an  \'{e}tale neighborhood of the point classifying  each totally degenerate curve (cf. Corollary \ref{T50}).
Although this result asserts a  local property, global finiteness was already established for $\N=1$ in ~\cite[Chapter II, Theorem 2.8]{Mzk2}  and subsequently extended to dormant $G$-opers (for a semisimple algebraic group $G$ of adjoint type) in ~\cite[Corollary 6.1.6]{JP} and ~\cite[Corollary 3.32, Theorem 3.38, (i)]{Wak8}.
However, for $\N \geq 2$, global finiteness unfortunately does not hold even in the case of $\mr{PGL}_2$ (cf. Proposition \ref{Prop100}).



Let $S$ and  $\msX$ be as introduced at the beginning of Section \ref{SS078}.
To begin with, we prove the following lemma.

\ble \label{Lem56}
Suppose that $S = \mr{Spec} (k)$ with $k$ algebraically closed, and that $\msX$ is taken to be the $3$-pointed projective line $\msP := (\mbP/k, \{ [0], [1], [\infty] \})$.
Let 
$\msF^\heartsuit := (\mcF, \nabla, \{ \mcF^j \}_{j=0}^2)$
be a dormant $\mr{GL}_2^{(\N)}$-oper on $\msP$, and $\varphi$ a global section of $F^{(\N)*}_{\mbP/k} (\Omega^{(\N)}) \otimes \mcE nd^0_{\mcO_\mbP} (\mcF)$.
If  $\varphi$ is nilpotent (when considered as a morphism $F^{(\N)*}_{\mbP/k} (\mcT^{(\N)}) \rightarrow \mcE nd^0_{\mcO_\mbP} (\mcF)$) and horizontal with respect to the $\mcD^{(\N -1)}$-module structure $\nabla^{(\N -1)}_{\Omega^{(\N)}, \mr{can}} \otimes \nabla^{(\N -1)}_{\mcE nd^0}$ (cf. \eqref{fE12}), then
 we have  $\varphi =0$.
\ele
\begin{proof}
Suppose that $\varphi$ is nilpotent and horizontal.
For simplicity, we write $\mcH := \mcS ol (\nabla^{(\N -1)}_{\mcE nd^0})$ (cf. \eqref{fE12}).
Let $x \in \{[0], [1], [\infty] \}$, and choose a local coordinate $t$ defining $x$.
Denote by $u$ the local coordinate  on $\mbP^{(\N)}$ corresponding to $t$ (hence $F^{(\N)*}_{\mbP/k} (u) = t^{p^\N}$).
By applying Proposition \ref{NN77}, we can verify that
the restriction of $\mcH$ to 
 the formal neighborhood  $\widehat{U}_x$ of $x$ is given, with respect to a suitable basis,  by 
\begin{align} \label{EEE4}
\mcH |_{\widehat{U}_x} =  k[\![u]\!] h \oplus k[\![u]\!] t^{a}  e_+\oplus k[\![u]\!]t^{p^\N -a} e_-
\end{align} 
for some integer $a$ with $0 < a < p^\N$, where $h := \begin{pmatrix} 1 & 0 \\ 0 & -1 \end{pmatrix}$, $e_+ := \begin{pmatrix} 0 & 1 \\ 0 & 0\end{pmatrix}$, and $e_- := \begin{pmatrix} 0 & 0 \\ 1 & 0 \end{pmatrix}$.

Now, let $b$ be an arbitrary section of  $\Omega^{(\N)} \otimes \mcH$.
Using   \eqref{EEE4},  its restriction $b |_{\widehat{U}_x}$ over $\widehat{U}_x$ can be expressed, if it is defined,   as 
\begin{align} \label{EEE5}
b |_{\widehat{U}_x} := \frac{du}{u} \otimes B, \, \text{where} \ B:= \left( b_0 (u)h + b_+ (u)t^a e_+ + b_- (u) t^{p^\N -a} e_-\right)
\end{align}
for some $b_0 (u), b_+ (u), b_- (u) \in k[\![u]\!]$.
Let $\mcG$ denote the $\mcO_{\mbP^{(\N)}}$-submodule of $\Omega^{(\N)} \otimes \mcH$ consisting of  such sections $b$ with $b_0 (0) = 0$ for every $x \in \{[0], [1], [\infty] \}$.
If $b$ is nilpotent, then $0 = \mr{det} (B) |_{u=0} =  \left(- b_0 (u)^2 - u \cdot b_+ (u) b_{-} (u) \right) |_{u=0} = - b_0 (0)^2$.
This implies $b_0 (0) = 0$, so $b$ lies in  $\mcG$.
In particular, 
since 
the section $\varphi \in H^0 (\mbP^{(\N)}, \Omega^{(\N)} \otimes \mcH)$  is  nilpotent, 
$\varphi$ belongs to $H^0 (\mbP^{(\N)}, \mcG)$.

Let us consider the bilinear pairing $B : \mcG \times \mcH \rightarrow \Omega^{(\N)}$ given by $B (b, b') = \mr{tr}(bb')$.
The image of $B$ lies in the dualizing sheaf $\omega^{(\N)} \left(= \Omega^{(\N)} (-[0] -[1]- [\infty]) \right)$ of $\mbP^{(\N)}$,
so it defines  a bilinear pairing $B' : \mcG \times \mcH \rightarrow \omega^{(\N)}$.
The matrix form of this pairing over  $\widehat{U}_x$ with respect to the basis $\left(\frac{du}{u} \otimes  u h, \frac{du}{u} \otimes t^a e_+, \frac{du}{u} \otimes t^{p^\N -a}e_- \right)$ of $\mcG  |_{\widehat{U}_x}$
 and the basis $(h, t^a e_+, t^{p^\N -a}e_-)$ of $\mcH |_{\widehat{U}_x}$
 is given by 
 $\begin{pmatrix} 2 & 0 & 0 \\ 0 & 0 & 1 \\ 0 & 1 & 0\end{pmatrix}$.
 It follows that $B'  |_{\widehat{U}_x}$ is nondegenerate.
 Since the nondegeneracy of $B'$ on $\mbP^{(\N)} \setminus \{ [0], [1], [\infty]\}$
 are immediately verified,
   $\mcG$ is isomorphic to $\omega^{(\N)}  \otimes \mcH^\vee$.
Hence, we have 
\begin{align}
H^0 (\mbP^{(\N)}, \mcG) \cong H^0 (\mbP^{(\N)}, \omega^{(\N)}  \otimes \mcH^\vee) \cong H^1 (\mbP^{(\N)}, \mcH)^\vee  =0,
\end{align}
where the last equality follows from Proposition \ref{Pee4}, (ii).
In particular, the equality $\varphi = 0$ holds.
\end{proof}
\SSP

By using the above lemma, we prove the following proposition.

\SSP
\bpr \label{Prop44}
Assume 
the following three conditions:
\begin{itemize}
\item
$n =2 < p$;
\item
$S = \mr{Spec}(R)$ for a discrete valuation ring of a field $K$ over $k$ with residue field $k$;
\item
The special fiber  $\msX_0 := (X_0, \{ \sigma_{0, i}\}_i)$ of $\msX$ is totally degenerate.
\end{itemize}
Then, the  scheme  (representing the functor) $\mcO p_{\spadesuit}^\ZZZ$
 is
   finite over $S$.
 In particular, there are only finitely many isomorphism classes of dormant $\mr{PGL}_n^{(\N)}$-opers on $\msX$.
\epr
\begin{proof}
Let us fix a dormant $2^{(\N)}$-theta characteristic $\vartheta := (\varTheta, \DMO_\vartheta)$ of $X^\mr{log}/S^\mr{log}$ (cf. Proposition \ref{P204}).
Denote by $\msX_K := (X_K, \{ \sigma_{K, i} \}_i)$  the generic   fiber  of $\msX$,  and by $\vartheta_K := (\varTheta_K, \DMO_{\vartheta_K})$  (resp., $\vartheta_0 := (\varTheta_0, \DMO_{\vartheta_0})$) the restriction of $\vartheta$ to $X_K$ (resp., $X_0$).
Let us take an arbitrary  dormant $(\mr{GL}_2^{(\N)}, \vartheta_K)$-oper $\DMO_K^\diamondsuit$
 on  $\msX_K$.
We shall prove by induction on $\N' \leq \N$ the claim that 
  the dormant $(\mr{GL}_2^{(\N')}, \vartheta^{(\N')})$-oper $\DMO_K^{\diamondsuit (\N')}$ (where $\vartheta^{(\N')}$ denotes the $2^{(\N')}$-theta characteristic obtained from $\vartheta$) induced by $\nabla_K$ via level reduction extends  over $\msX$. 

The base step, i.e., the case $\N' =1$, follows from the finiteness  of the projection $\Pi_{g, r}$ for $\N =1$ asserted in  ~\cite[Chapter II, Theorem 2.8]{Mzk2}.

To consider the induction step, we assume that the claim for $\N' \leq \N -1$ has been proved.
Thus, there exists a dormant $(\mr{GL}_2^{(\N -1)}, \vartheta)$-oper $\nabla_R^{\diamondsuit (\N -1)}$ on $\msX$ extending $\DMO_K^{\diamondsuit (\N' -1)}$.
Assume that $\mcF_\varTheta$  is not closed under the $\mcD^{(\N -1)}_K$-action  $\nabla_K^\diamondsuit$.
Choose a locally defined logarithmic coordinate $x$ on   $X^\mr{log}/S^\mr{log}$, and let $\{ \partial^{\langle j \rangle} \}_{j \in \mbZ_{\geq 0}}$ be the  local  basis of $\mcD^{(\N -1)}$ associated to $x$ (cf. Section \ref{SS04f4}).
By assumption,  it makes sense to define
 the minimum nonnegative  integer $d$ satisfying $\nabla_K^\diamondsuit (\partial^{\langle p^{\N -1} \rangle}) \subseteq u^{-d} \mcF_\varTheta$, where $u$ denotes a uniformizer of $R$.
This integer $d$ does not depend on the choice of  $x$.
The assignment $\partial^{\langle p^{\N -1} \rangle} \mapsto u^d \nabla_K^{\diamondsuit} (\partial^{\langle p^{\N -1} \rangle})$  determines a well-defined {\it nonzero} section 
$\varphi \in H^0 (X_0, \Omega_{0}^{\otimes p^{\N -1}} \otimes \mcE nd (\mcF_{\varTheta, 0}))$,
 where $\Omega_0 := \Omega |_{X_0}$ and $\mcF_{\varTheta, 0} := \mcF_\varTheta |_{X_0}$.
By using \eqref{e364}, we can verify the equality  $\nabla_K ^\diamondsuit (\partial^{\langle p^{\N -1} \rangle})^p = \nabla_K ^\diamondsuit (\partial^{\langle p^{\N -1} \rangle})$.
This implies $(u^d \nabla_K ^\diamondsuit (\partial^{\langle p^{\N -1} \rangle}))^p = u^{d (p-1)} (u^d \nabla_K ^\diamondsuit (\partial^{\langle p^{\N -1} \rangle}))$, and hence, we have $\varphi^p = 0$.
Since $X_0$ is reduced,  $\varphi$ has vanishing trace, i.e., belongs to $H^0 (X_0, \Omega_{0}^{\otimes p^{\N -1}} \otimes \mcE nd^0 (\mcF_{\varTheta, 0})) \left(=  H^0 (X_0, F^{(\N -1)*} (\Omega_0^{(\N -1)}) \otimes \mcE nd (\mcF_{\varTheta, 0})) \right)$.
By the commutative relation \eqref{Equu88}, 
$\varphi$ 
turns out to be  horizontal with respect to $\nabla_{\Omega_0^{(\N -1)}, \mr{can}}^{(\N -2)} \otimes \nabla_{\mcE nd^0}^{(\N -2)}$.

Let $\msX' \left(\cong \msP \right)$ be one of the  $3$-pointed projective lines constituting $\msX$ by means of  clutching data.
Denote by $\msF^\heartsuit := (\mcF, \nabla, \{ \mcF^j \}_{j=0}^2)$ be the dormant $\mr{GL}_2^{(\N -1)}$-oper on $\msX'$ obtained by restricting $(\nabla_R^{\diamondsuit (\N -1)})^{\Rightarrow \heartsuit}$.
Then, 
by restricting $\varphi$ to $\msX'$,
we obtain a nilpotent and horizontal section $\varphi'$ of
$F^{(\N -1)*}_{\mbP/k} (\Omega^{(\N -1)}) \otimes \mcE nd^0_{\mcO_\mbP} (\mcF)$.
Hence, it follows from Lemma \ref{Lem56} that $\varphi' =0$.
By considering this fact for every constituent  $\msX'$ of $\msX$,
we see that $\varphi =0$, which is a contradiction.

Hence, $\nabla_K^\diamondsuit$ restricts to a $\mcD_R^{(\N -1)}$-module structure $\nabla_R^{(\N -1)}$ on $\mcF_\varTheta$, which extends $\nabla_R^{(\N -1)}$.
This proves the induction step, and hence, this completes the proof of the claim.

The claim just proved  holds after changing the base over  $S$-schemes of the form $S' := \mr{Spec} (R')$ for any  discrete valuation ring $R'$.
Thus, by the standard valuative criterion using discrete valuation rings (cf. ~\cite[Chapter II, Exercise 4.11]{Har}), 
we see that $\mcO p_{\spadesuit}^{^\mr{Zzz...}}$ is proper over $S$.
It follows from the affineness assertion in Theorem \ref{T40} that $\mcO p_{\spadesuit}^{^\mr{Zzz...}}$ is finite over $S$.
We finish the proof of this proposition.
\end{proof}


\SSP
\bco \label{T50}
Suppose that $n=2 < p$ and $R = \mbF_p$.
Let $q$ be a geometric  point of $\overline{\mcM}_{g, r}$ classifying a totally degenerate curve.
Then,
there exists an  \'{e}tale neighborhood $U_q$ of $q$ in $\overline{\mcM}_{g, r}$ such that
 the  projection 
  $\Pi_{g, r} : \mcO p_{g,r}^\ZZZ \rightarrow \overline{\mcM}_{g, r}$ is finite when restricted over $U_q$.
In particular, 
after possibly shrinking $U_q$,
 the morphism of stacks  $\Pi_{\N \Rightarrow \N'} : \mcO p_{n, \N, g,r, \mbF_p}^\ZZZ \migi \mcO  p^\ZZZ_{n, \N', g,r, \mbF_p}$ 
obtained by reducing  the level of dormant $\mr{PGL}_n$-opers   to any $\N' \leq \N$ is  finite over $U_q$.
\eco
\begin{proof}
The first  assertion is a direct consequence of the above proposition.
The second assertion follows from the first assertion together with the affineness asserted in  Theorem \ref{T40}.
\end{proof}

With the above assertion in mind, 
it will be natural to ask whether or not $\Pi_{g, r}$ is finite over the entire space $\overline{\mcM}_{g, r}$.
However,   we can immediately verify that
this global finiteness does not hold in general, as proved in the proposition below.

Here, denote by $\overline{\mcM}_{1, 1}^\mr{ord}$ (resp., $\overline{\mcM}_{1, 1}^{\mr{ss}}$) the  open (resp., closed) substack of $\overline{\mcM}_{1, 1}$ classifying ordinary (resp., supersingular)  stable elliptic curves, in the sense of ~\cite[Sections 2.1 and  3..2]{KaWa}.

\bpr \label{Prop100}
Suppose $n =2 < p$ and $R = \mbF_p$.
Then, the restriction
\begin{align}
\Pi^\mr{ord}_{\varepsilon, 1, 1} : \mcO p^{^\mr{Zzz...}}_{\varepsilon, 1, 1} \times_{\overline{\mcM}_{1, 1}} \overline{\mcM}_{1, 1}^\mr{ord} \rightarrow \overline{\mcM}_{1, 1}^{\mr{ord}}  
\end{align}
of $\Pi_{\varepsilon, 1, 1}$  (cf. \eqref{Eq20} for the definition of $\varepsilon$) over $\overline{\mcM}_{1, 1}^\mr{ord} $ is quasi-finite and faithfully flat, and its generic degree $\mr{deg} (\Pi^\mr{ord}_{\varepsilon, 1, 1} )$ satisfies  $\mr{deg} (\Pi^\mr{ord}_{\varepsilon, 1, 1} ) \geq \frac{p^{\N -1}(p-1)}{2}$.
Moreover, for $\N \geq 2$, the fiber of $\Pi_{\varepsilon, 1, 1}$ over every point of $\overline{\mcM}_{1, 1}^\mr{ss}$  is empty.
In particular, $\Pi_{\varepsilon, 1, 1}$ is finite if and only if $\N =1$. 
\epr
\begin{proof}
Let $k$ be an algebraically closed field over $\mbF_p$ and $\msX := (X, \{ \sigma_1\})$  a pointed curve classified by a $k$-rational point  $q$ of $\overline{\mcM}_{1, 1}$.
Write $\Pi_{\varepsilon, 1, 1}^{-1} (q)$ for the fiber of $\Pi_{\varepsilon, 1, 1}$ over $q$.
According to Proposition \ref{Prop12}, 
each  dormant  $\mr{PGL}_2^{(\N)}$-oper  of radius $\varepsilon$ on $\msX$
can be identified with
a dormant $\mr{PGL}_2^{(\N)}$-oper on the curve $X$  without marked point, i.e., without logarithmic structure.
(Note that Proposition \ref{Prop12} can be applied even when the hyperbolic condition $2g-2+r >0$ does not hold.)

Now, let us consider the first assertion.
By the fact mentioned above and ~\cite[Theorem D]{KaWa}, we see that if $q \in \overline{\mcM}_{1, 1}^\mr{ord} \cap \mcM_{1, 1}$, then $\Pi_{\varepsilon, 1, 1}^{-1} (q)$ is the disjoint union of  $0$-dimensional $k$-schemes   indexed by the elements of $(\mbZ/p^\N \mbZ)^\times/ \{ \pm 1 \}$ (cf. $\eqref{wpaid}$).
On the other hand,  the pointed curve  classified by a unique  $k$-rational point of  $\overline{\mcM}_{1, 1}^\mr{ord}  \setminus \mcM_{1, 1}$ can be constructed from the $3$-pointed projective line $\msP$ (cf. \eqref{1051}) by attaching $[0]$ and $[\infty]$ to form a nodal point.
If $\msX$ coincides with that curve, then it follows 
 Corollary \ref{Cor209} that $\Pi_{\varepsilon, 1, 1}^{-1} (q)$ is nonempty and isomorphic to  the disjoint union of finite copies of $\mr{Spec} (k)$.
Thus, $\Pi_{\varepsilon, 1, 1}^\mr{ord}$ is quasi-finite and surjective.
Moreover,
since both 
$\mcO p^{^\mr{Zzz...}}_{\varepsilon, 1, 1} \times_{\overline{\mcM}_{1, 1}} \overline{\mcM}_{1, 1}^\mr{ord}$ and $\overline{\mcM}_{1, 1}^{\mr{ord}}$
 are smooth over $\mbF_p$ (cf. Corollary \ref{T34}, (i)), the morphism $\Pi^\mr{ord}_{\varepsilon, 1, 1}$ turns out to be flat  (cf. ~\cite[Chapter III, Exercise 10.9]{Har}).
The computation of  the generic degree $\mr{deg} (\Pi_{\varepsilon, 1, 1}^\mr{ord})$ is given by 
\begin{align}
\mr{deg} (\Pi_{\varepsilon, 1, 1}^\mr{ord}) \geq \sharp (\Pi_{\varepsilon, 1, 1}^{-1} (q)) = \sharp ((\mbZ/p^\N \mbZ)^\times / \{ \pm 1 \}) = \frac{p^{\N -1}(p-1)}{2}.
\end{align}
This completes the proof of assertion (i).

Next, we prove the second assertion.
Consider the case where $X$ is a supersingular elliptic curve.
To complete the proof, it suffices to show that there is no dormant $\mr{PGL}_2^{(2)}$-oper on $X$.

Suppose, on the contrary, that there exists a dormant $\mr{PGL}_2^{(2)}$-oper $\msE^\spadesuit$ on $X$.
As discussed in ~\cite[Section 2.3]{Wak25},
dormant $\mr{PGL}_2$-oper induced by $\msE^\spadesuit$ corresponds to a nontrivial  extension  
\begin{align}\label{EE332}
0 \rightarrow \mcO_{X^{(1)}} \rightarrow \mcV \rightarrow \mcO_{X^{(1)}} \rightarrow 0
\end{align}
 of vector bundles on $X^{(1)}$.
Then, $\msE^\spadesuit$ corresponds to 
a rank $2$ vector bundle $\mcW$ on $X^{(2)}$ with $F^{*} (\mcW) \cong \mcV$ and $\mr{det} (\mcW)\cong \mcO_{X^{(2)}}$, where $F := F_{X^{(1)}/k}$.
If $\mcW$ is decomposable, i.e., $\mcW \cong \mcL \oplus \mcL^{-1}$ for a line bundle $\mcL$, then $\mcV \left(\cong F^* (\mcW) \right) \cong F^*(\mcL) \oplus F^*(\mcL)^\vee$.
This is impossible because the extension \eqref{EE332} is nontrivial, so
 $\mcW$ must be indecomposable.

According to Atiyah's classification of indecomposable vector bundles (cf. ~\cite{Ati1}),
$\mcW$ fits into   an extension 
\begin{align} \label{EE666}
0 \rightarrow \mcN \rightarrow \mcW \rightarrow \mcN \rightarrow 0
\end{align}
 using a line bundle $\mcN$ with $\mcN^{\otimes 2}\cong \mcO_{X^{(2)}}$.
 Such an extension  is classified by an element of $\mr{Ext}^1_{\mcO_{X^{(2)}}} (\mcN, \mcN) \cong H^1 (X^{(2)}, \mcO_{X^{(2)}}) \cong k$.
 Since $X$ is supersingular, the morphism $H^1 (X^{(2)}, \mcO_{X^{(2)}}) \rightarrow H^1 (X^{(1)}, \mcO_{X^{(1)}})$ induced by $F$ coincides with the zero map.
 This means that the pull-back $0 \rightarrow F^* (\mcN) \rightarrow F^* (\mcW) \rightarrow F^* (\mcN) \rightarrow 0$ of  \eqref{EE666} splits.
 Hence, we have $\mcV \cong F^* (\mcW) \cong F^* (\mcN) \oplus F^{*}(\mcN)^\vee$.
 Since there exists a surjection $\mcV \twoheadrightarrow \mcO_X$,
 $F^* (\mcN)$ (as well as  $F^* (\mcN)^\vee$) must be isomorphic to $\mcO_{X^{(1)}}$.
 This contradicts the fact that the extension \eqref{EE332} is nontrivial.
 This completes the proof of the second assertion.
 \end{proof}




\LSP
\subsection{Generic \'{e}taleness of the moduli space}\label{y0177}

 When  $\N=1$,
the generic \'{e}taleness of $\mcO p^\ZZZ_{g,r}/\overline{\mcM}_{g, r}$ was already shown in ~\cite[Chapter II, Theorem 2.8]{Mzk2}.
By using Theorem  \ref{Per33}, we can prove  the following higher-level generalization.

\SSP
\bt \label{c43} 
Suppose that $n=2 < p$. 
Then, 
the inverse image via the projection $\Pi_{g, r} : \mcO p^\ZZZ_{g, r} \!\migi \overline{\mcM}_{g, r}$  of 
  the point  classifying 
any  totally degenerate curve is contained in
 $\mcO p^\ZZZ_{g,r, \mr{ord}}$ (cf. \eqref{dE204}).
 In particular, $\Pi_{g, r}$ is generically \'{e}tale, i.e., any irreducible component of $\mcO p^\ZZZ_{g, r}$ that dominates $\overline{\mcM}_{g, r}$ has a dense open substack \'{e}tale over $\overline{\mcM}_{g, r}$.
\et
\begin{proof}
First, we shall prove  the first assertion.
Let $k$ be an algebraically closed field over $\mbF_p$ and $\rho$ an element of $\Xi_{2, \N}^{\times r}$.
Also, let us take an arbitrary $r$-pointed  totally degenerate curve $\msX$
of genus $g$ over $k$.
We may assume, without loss of generality, that
there exists 
 trivalent clutching data
  $\mbG = (\GR, \{ \lambda_j\}_{j=1}^J)$ inducing  $\msX$ (cf. Remark \ref{Rem3391i}, Definition \ref{DD3WW}).
  The point $q_\mbG$ of $\overline{\mcM}_{g, r}$ determined by the morphism
  \begin{align}
 \mr{Clut}_\mbG : \left(\mr{Spec}(\mbF_p) =  \right) \prod_{j=1}^J \overline{\mcM}_{0, 3}  \migi  \overline{\mcM}_{g, r}
  \end{align}
  (cf. \eqref{WW81}) classifies $\msX$. 
By  applying  Theorem \ref{y0176},
we obtain the Cartesian square diagram \eqref{1050} for $\mbG$ and $\rho$. 
It follows from  Theorem  \ref{Per33} that
the left-hand vertical arrow  in  this diagram  is unramified.
Hence,  the projection  $\Pi_{\rho, g, r} : \mcO p^\ZZZ_{\rho, g,r} \migi \overline{\mcM}_{g,r}$, i.e., the right-hand vertical arrow in  \eqref{1050},  is unramified over the point $q_\mbG$.
Since both $\mcO p^\ZZZ_{\rho, g,r}$ and  $\overline{\mcM}_{g,r}$ are smooth over $\mbF_p$ and of dimension $3g-3+r$ (cf. Corollary \ref{T34}, (i)),
$\Pi_{\rho, g, r}$ is moreover  \'{e}tale over the same point.
This completes the proof of the  first assertion.

Moreover, by the first assertion  and Proposition \ref{Prop44},  the projection $\Pi_{g, r}$ is  finite and \'{e}tale  when restricted  over  an \'{e}tale  neighborhood of $q_\mbG$.
It follows that any irreducible component of $\mcO p^\ZZZ_{g,r}$ that dominates $\overline{\mcM}_{g, r}$ contains 
at least one point lying over $q_\mbG$.
Thus,  the second  assertion follows from the open nature of \'{e}taleness. 
\end{proof}
\SSP

\bco \label{Cor333}
 $\mcO p^\ZZZ_{g,r} \times_{\overline{\mcM}_{g, r}} \mcM_{g, r}$ is nonempty.
\eco
\begin{proof}
Let us choose a geometric point $\overline{q}$ of $\overline{\mcM}_{g, r}$ classifying a totally degenerate curve.
By Theorem \ref{T52}, there exists a point $q \in \mcO p_{g, r}^{^\mr{Zzz...}}$\! lying over $\overline{q}$.
Denote by $\widehat{U}_{\overline{q}}$ and $\widehat{U}_q$ the formal neighborhoods of $\overline{q} \in \overline{\mcM}_{g, r}$ and $q \in \mcO p_{g, r}^{^\mr{Zzz...}}$, respectively.
According to Corollary \ref{T50} and Theorem \ref{c43}, the projection $\Pi_{g, r}$ is finite and \'{e}tale over
 $\widehat{U}_{\overline{q}}$.
 Hence, $\Pi_{g, r}$ restricts to an isomorphism  $\widehat{U}_q \xrightarrow{\sim} \widehat{U}_{\overline{q}}$.
 Since the generic point of $\widehat{U}_q$ lies in  $\mcO p^\ZZZ_{g,r} \times_{\overline{\mcM}_{g, r}} \mcM_{g, r}$,
 the proof of the assertion is completed.
 \end{proof}

\begin{rema}[Nonordinary locus] \label{NN80}
Let us consider
the morphism $D_{\msE^\spadesuit}$ (cf. \ref{Equu818}) in the case where the pair ``$(\msX, \msE^\spadesuit)$" is taken to be the universal object
 over $\mcO p^\ZZZ_{g,r}$.
This morphism determines  a morphism
\begin{align}
\bigwedge^{3g-3 +r} D_{\msE^\spadesuit} : \bigwedge^{3g-3 +r}\mbR^1 f_* (\widetilde{\mcT}_{\mcE_B^\mr{log}/S^\mr{log}, \phi}) \migi \bigwedge^{3g-3 +r}\mbR^1 f_* (\mcT)
\end{align}
between line bundles on $\mcO p^\ZZZ_{g, r}$ (cf. Proposition \ref{Pee4}, (ii)), which corresponds to a global section of the line bundle
\begin{align}
\left(\bigwedge^{3g-3+r}\mbR^1 f_* (\mcT) \right) \otimes \left( \bigwedge^{3g-3 +r}\mbR^1 f_* (\widetilde{\mcT}_{\mcE_B^\mr{log}/S^\mr{log}, \phi}) \right)^\vee.
\end{align}
The morphism $D_{\msE^\spadesuit}$ is an isomorphism exactly on the complement of the closed subscheme 
$\mcD_{\cancel{\mr{ord}}} \left(\subseteq \mcO p^\ZZZ_{g,r} \right)$ defined by this section.
On any irreducible component  of  $\mcO p^\ZZZ_{g,r}$  dominating $\overline{\mcM}_{g, r}$, this closed subscheme defines an effective divisor (cf. Theorem \ref{c43}).
The nonordinary locus in $\mcO p^\ZZZ_{g, r}$ coincides with the support of $\mcD_{\cancel{\mr{ord}}}$.
\end{rema}
\SSP

\begin{rema}[Generic \'{e}taleness for higher rank cases] \label{WW1}
For a (sufficiently small) general $n>1$, 
the projection  $\mcO p^\ZZZ_{n, 1, g, r, \mbF_p}\migi \overline{\mcM}_{g, r}$ 
is generically \'{e}tale (cf. ~\cite[Theorem G]{Wak8}).
The generic \'{e}taleness
 is essential in the proofs of Joshi's conjecture described in ~\cite{Wak} and ~\cite{Wak8}.
Indeed, we used this geometric property to lift relevant moduli spaces to characteristic $0$ and then compared them with   
certain Quot schemes over $\mbC$; this argument  enables us to compute explicitly the number of all possible dormant $\mr{PGL}_n$-opers on a general curve.

To develop the enumerative geometry of higher-level dormant opers in such a way that this argument works,  we expect  that the following conjecture is true:
\begin{quote}
{\it Suppose that $p$ is sufficiently large relative to $n$, $g$, and $r$.
Then, the  stack $\mcO p^\ZZZ_{n, \N, g, r, \mbF_p}$ associated to such a collection  $(p, n, g, r)$   is generically \'{e}tale over $\overline{\mcM}_{g, r}$, i.e., any irreducible component of $\mcO p^\ZZZ_{n, \N, g, r, \mbF_p}$ that dominates $\overline{\mcM}_{g, r}$  admits a dense open substack  which is \'{e}tale over $\overline{\mcM}_{g, r}$.}
\end{quote}
\end{rema}
\SSP


Here, for an inclusion of stacks $\mcA \rightarrow \mcB$, we shall write $\mcB \setminus \mcA$ 
for the substack of $\mcB$ such that, for each scheme $T$, the fiber  $(\mcB \setminus \mcA) (T)$ is the full subcategory  of $\mcB (T)$ consisting of objects $a$ whose values $a (q)$ at each geometric point $q$ does not factor through $\mcA \hookrightarrow \mcB$.
For example, if $\mcA \hookrightarrow \mcB$ is a closed immersion of Deligne-Mumford stacks, then $(\mcB \setminus\mcA)$ coincides with the open substack of $\mcB$ defined as the complement of $\mcA$ in the usual sense. 

For each positive integer $\N' \leq \N$, we shall write 
$\mcJ_{\N'}$ for 
the stack-theoretic image of the restriction  $\mcO p^\ZZZ_{2, \N', g,r, \mbF_p} \setminus \mcO p^\ZZZ_{2, \N', g,r, \mbF_p,  \mr{ord}} \rightarrow \overline{\mcM}_{g, r}$ of $\Pi_{g, r}$, i.e.,
the stackification  of the essential image of this restriction.
Then, we obtain the substack
\begin{align} \label{fE34}
\overline{\mcM}_{g, r, N\text{-}\mr{ord}} := \bigcap_{\N' \leq \N}  \overline{\mcM}_{g, r} \setminus \mcJ_{\N'}
\end{align}
of  $\overline{\mcM}_{g, r}$.
By Theorem \ref{c43} (in the case where $\N$ is replaced by each positive integer $\N' \leq \N$), 
$\overline{\mcM}_{g, r, N\text{-}\mr{ord}}$
 contains a dense open substack of $\overline{\mcM}_{g, r}$.

If $\msX$ is  a pointed stable curve over a geometric point classified by $\overline{\mcM}_{g, r, N\text{-}\mr{ord}}$,
then all the dormant $\mr{PGL}_2^{(\N)}$-opers on $\msX$ and their reductions to level $\N' \leq \N$ are ordinary.
In particular, it follows from Corollary \ref{Cor87} that  the set of isomorphism classes of dormant $\mr{PGL}_2^{(\N)}$-opers on a general curve $\msX$ in $\overline{\mcM}_{g, r, N\text{-}\mr{ord}}$  has cardinality equal to the degree of the projection   $\Pi_{g, r}$.

\SSP
\bde \label{DD456}
An $r$-pointed stable curve   of genus $g$ is  called {\bf (dormant-)$\N$-ordinary} if it is classified by the substack $\overline{\mcM}_{g, r, N\text{-}\mr{ord}}$.
(Since  $\overline{\mcM}_{g, r, N\text{-}\mr{ord}}$ is dense in $\overline{\mcM}_{g, r}$, a general pointed stable curve  is $\N$-ordinary.)
\ede
\SSP



\begin{rema}[$\infty$-ordinariness] \label{NN100}
 By the definition of $\N$-ordinariness,  
 there exists   a sequence of inclusions
 \begin{align} \label{NN123}
 \cdots \subseteq \overline{\mcM}_{g, r, \N\text{-}\mr{ord}} \subseteq \cdots \subseteq \overline{\mcM}_{g, r, 2\text{-}\mr{ord}}  \subseteq \overline{\mcM}_{g, r, 1\text{-}\mr{ord}}  \subseteq \overline{\mcM}_{g, r}.
 \end{align}
 It will be natural to ask what the intersection $\bigcap_{\N \in \mbZ_{>0}} \overline{\mcM}_{g, r, \N\text{-}\mr{ord}}$ is.
 At the time of writing this manuscript, the author does not know the existence of  a {\it closed} point of $\mcM_{g, r} \left(\subseteq \overline{\mcM}_{g, r} \right)$ 
 classifying a pointed  curve that  is $\N$-ordinary for every $\N$. 
\end{rema}

\vspace{10mm}
\section{Canonical diagonal liftings} \label{dEr45}\SSP

In this section, we   construct the canonical diagonal lifting of a dormant  $\mr{PGL}_2^{(\N)}$-oper, as well as of  a Frobenius-projective structure, on a  general  smooth curve of genus $>1$ (cf. Theorem-Definitions \ref{T04} and \ref{T1dddd}).
A key observation for doing  this is that the ordinariness introduced in the previous section 
enables  us to connect
  various deformation  spaces involved.
As a consequence, we obtain an approach to solving  the counting problem of dormant $\mr{PGL}_2$-opers in characteristic $p^\N$ via reduction modulo $p$.

 Let us fix 
 a positive integer $\N$, and suppose that $p >2$.

\LSP
\subsection{Canonical diagonal liftings of dormant $(\mr{GL}_2^{(\N)}, \vartheta)$-opers}\label{SS107}

 Let $R$ be 
   a flat $\mbZ/p^\N\mbZ$-algebra  and  $X$  a geometrically connected, proper, and  smooth curve  over $S := \mr{Spec}(R)$ of genus $g >1$.
   For each positive integer $\N' \leq \N$, we shall denote by a subscript ``$\N'-1$" the result of reducing an object over $\mbZ/p^\N \mbZ$ modulo $p^{\N'}$.
For simplicity, we write $\Omega_{0}:= \Omega_{X_0/S_0}$ and  $\Omega^{(\N')}_0 := \Omega_{X_0^{(\N')}/S_0}$.

We shall fix  a dormant $2^{(1)}$-theta characteristic $\vartheta := (\varTheta, \DMO_\vartheta)$  of $X/S$ (cf. Definition \ref{D41} and Proposition \ref{P204}) and fix an integer $\N'$ with $1 \leq \N' < \N$.
The diagonal reduction of $\vartheta_{\N' -1}$ defines a 
$2^{(\N')}$-theta characteristic  $\vartheta_0^{(\N')}$ of $X_{0}/S_{0}$.

Also, let 
$\DMO^{\diamondsuit}_{\N' -1}$ be 
a dormant $(\mr{GL}_2^{(1)}, \vartheta_{\N' -1})$-oper on $X_{\N' -1}/S_{\N' -1}$, and 
write  $\msF := (\mcF_{\N' -1}, \DMO_{\N' -1}^{\diamondsuit})$, where $\mcF := \mcF_{\varTheta}$.
  The diagonal reduction of $\DMO^{\diamondsuit}_{\N'-1}$ defines a $(\mr{GL}_2^{(\N')}, \vartheta_0^{(\N')})$-oper  $\DMO_0^{\diamondsuit (\N')}$  on $X_0/S_0$.
  The $(\mr{GL}_2^{(1)}, \vartheta_0)$-oper on $X_0/S_0$ induced by $\DMO_0^{\diamondsuit (\N')}$  via level reduction will be denoted by $\DMO_0^{\diamondsuit (1)}$.

We shall write
\begin{align}
\mr{Lift} (\DMO_{\N'-1}^\diamondsuit)_\vartheta
\end{align}
for the set of isomorphism classes of $(\mr{GL}_2^{(1)}, \vartheta_{\N'})$-opers $\DMO^\diamondsuit_{\N'}$  on $X_{\N'}/S_{\N'}$ with $(\DMO^{\diamondsuit}_{\N'})_{\N' -1} = \DMO^{\diamondsuit}_{\N' -1}$.
Then,  $\mr{Lift} (\DMO_{\N'-1}^\diamondsuit)_\vartheta$ is nonempty 
 and  has a canonical structure of torsor modeled on the $R$-module $H^0 (X_0, \Omega_{0}^{\otimes 2})$ (cf. ~\cite[Chapter I, Corollary 2.9]{Mzk1}).

Let us consider
the composite
\begin{align} \label{fE14}
 H^0 (X_0, \Omega_{0}^{\otimes 2}) \migiincl  H^0 (\Omega_{0} \otimes \mcE nd^0 (\mcF_0)) \migi \mbH^1 (X, \mcK^\bullet [(\nabla^{\diamondsuit (1)}_0)_{\mcE nd^0}])
\end{align}
(cf. \eqref{fE12} for the definition of $(-)_{\mcE nd^0}$), 
where the first arrow denotes  the canonical  inclusion  (cf. ~\cite[Section 4.9.5]{Wak8}) and 
the second arrow denotes the morphism induced by the natural  inclusion of complexes $(\Omega_{0} \otimes \mcE nd^{0} (\mcF_0))[-1] \migi \mcK^\bullet [(\nabla_0^{\diamondsuit (1)})_{\mcE nd^0}]$.
Then, 
the assignment $\DMO_{\N'}^\diamondsuit \mapsto (\mcF_{\N'}, \DMO_{\N'}^\diamondsuit)$ defines a map of sets
\begin{align} \label{fE15}
\mr{Lift} (\DMO_{\N'-1}^\diamondsuit)_\vartheta \migi \mr{Lift} (\msF)_{\msL}
\end{align}
where $\msL := (\mr{det}(\mcF), \DMO_\vartheta)$ (cf. \eqref{NN203} for the definition of $\mr{Lift} (\msF)_{\msL}$).

Recall from the discussion preceding Proposition \ref{WWW1} that $\mr{Lift} (\msF)_{\msL}$ has a structure of $\mbH^1 (X, \mcK^\bullet [(\nabla_0^{\diamondsuit (1)})_{\mcE nd^0}])$-torsor.
The map \eqref{fE15} is verified to commute  with the respective torsor structures via the morphism  \eqref{fE14} (cf. ~\cite[Section 6.3.2]{Wak8}).
By combining this fact and Proposition \ref{P016},
we obtain the following assertion.

\SSP
\bpr \label{P037} 
Taking the diagonal reductions   $\DMO_{\N'}^{\diamondsuit}  \mapsto {^{\Diag}}\!\!\DMO_{\N'}^{\diamondsuit}$
yields  a map of sets
\begin{align} \label{fE24}
\mr{Lift}(\DMO_{\N' -1}^{\diamondsuit})_\vartheta \migi  \mr{Lift}(\DMO_0^{\diamondsuit (\N')})_\msL 
\end{align}
(cf. \eqref{NN288} for the definition of the codomain of this map),
and this map commutes with the respective torsor structures via the composite
\begin{align} \label{fE21}
D''_{\DMO_0^{\diamondsuit (\N' )}} : H^0 (X_0, \Omega_{0}^{\otimes 2}) &\xrightarrow{\eqref{fE14}}
\mbH^1 (X_0, \mcK^\bullet [(\nabla_0^{\diamondsuit (1)})_{\mcE nd^0}]) \\
&\xrightarrow{\eqref{fE8}}
H^0 (X_0^{(\N')}, \Omega^{(\N')}_0\otimes \mcE nd^0 (\mcS ol (\DMO_0^{\diamondsuit (\N')}))). \notag
\end{align}
In particular,   
the map \eqref{fE24} becomes bijective when $D''_{\DMO_0^{\diamondsuit (\N')}}$ is bijective.
 \epr
\SSP

Denote by 
\begin{align} \label{NN199}
\mr{Lift} (\DMO_{\N' -1}^{\diamondsuit})_\vartheta^{\psi}
\end{align}
the subset of 
$\mr{Lift} (\DMO_{\N' -1}^{\diamondsuit})_\vartheta$ consisting of {\it dormant} $(\mr{GL}_2^{(1)}, \vartheta_{\N'})$-opers.
Propositions \ref{Pr459} and \ref{P037} together imply the following proposition.

\SSP
\bpr \label{Pr5221}
\begin{itemize}
\item[(i)]
Taking the diagonal reductions   $\DMO_{\N'}^{\diamondsuit}  \mapsto {^{\Diag}}\!\!\DMO_{\N'}^{\diamondsuit}$
 yields a map of sets
\begin{align} \label{fE29}
\mr{Lift} (\DMO_{\N' -1}^{\diamondsuit})_\vartheta^{\psi} \migi \mr{Lift} (\DMO_0^{\diamondsuit (\N')})^{\psi}_\msL,
\end{align} 
and the following commutative square diagram is  Cartesian:
\begin{align}  \label{fE400}
\vcenter{\xymatrix@C=56pt@R=36pt{
\mr{Lift} (\DMO_{\N' -1}^{\diamondsuit})_\vartheta^{\psi}  \ar[r]^-{\eqref{fE29}} \ar[d]_-{\mr{inclusion}} & \mr{Lift} (\DMO_0^{\diamondsuit (\N')})_\msL^{\psi}\ar[d]^{\mr{inclusion}} 
\\
\mr{Lift} (\DMO_{\N' -1}^{\diamondsuit})_\vartheta \ar[r]_-{\eqref{fE24}} & \mr{Lift} (\DMO_0^{\diamondsuit (\N')})_\msL.
}}
\end{align}
\item[(ii)]
Suppose that  $D''_{\DMO_0^{\diamondsuit (\N' )}}$ is bijective.
Then, the map 
\eqref{fE29}
 is bijective.
 In particular, for each dormant $(\mr{GL}_2^{(\N'+1)}, \vartheta^{(\N' +1)}_0)$-oper $\DMO_0^{\diamondsuit (\N' +1)}$ of $X_0/S_0$
 that induces   $\DMO_0^{\diamondsuit (\N')}$ by reducing its level,
  there exists a unique (up to isomorphism) dormant $(\mr{GL}_2^{(1)}, \vartheta_{\N'})$-oper $\DMO_{\N'}^{\diamondsuit}$ on $X_{\N'}/S_{\N'}$ with $(\DMO_{\N'}^{\diamondsuit})_{\N'-1} = \DMO_{\N'-1}^{\diamondsuit}$ and ${^{\Diag}}\!\!\DMO_{\N'}^{\diamondsuit} = \DMO_0^{\diamondsuit (\N' +1)}$.
 \end{itemize}
\epr

\SSP
\bco \label{Coeedk}
Let $\nabla^{\diamondsuit (\N)}_0$ be  a dormant $(\mr{GL}_2^{(\N)}, \vartheta_0^{(\N)})$-oper on $X_0/S_0$.
 Suppose that
the morphism $D''_{\DMO_0^{\diamondsuit (\N')}}$ is bijective for every $\N' \leq \N$, where $\DMO_0^{\diamondsuit (\N')}$ denotes the dormant $(\mr{GL}_2^{(\N')}, \vartheta^{(\N')}_0)$-oper obtained by reducing the level of $\nabla^{\diamondsuit (\N)}_0$ to $\N'$.
Then,  there exists  a unique (up to isomorphism)
dormant $(\mr{GL}_2^{(1)}, \vartheta)$-oper $\DMO^{\diamondsuit}$ on $X/S$ whose diagonal reduction coincides with  $\nabla^{\diamondsuit (\N)}_0$.
\eco
\begin{proof}
The assertion can be proved by successively applying assertion (ii) of Proposition \ref{Pr5221}  with respect to $\N'$ ($1 \leq \N' < \N$).
\end{proof}

\LSP
\subsection{Canonical diagonal liftings of  dormant $\mr{PGL}_2^{(\N)}$-opers}\label{SS1300}

Denote by
\begin{align} \label{e209}
\mr{O p}_{1, X}^{^\mr{Zzz...}} \ \left(\text{resp.,} \  \mr{O p}_{\N, X_0}^{^\mr{Zzz...}}\right)
\end{align}
the set of isomorphism classes of dormant $\mr{PGL}_2$-opers on $X/S$ (resp., dormant $\mr{PGL}_2^{(\N)}$-opers on $X_0/S_0$).
By regarding it as a (discrete) category, we obtain, from Proposition \ref{P44}, an equivalence of categories between
$\mr{Op}_{1, X}^{^\mr{Zzz...}}$ (resp., $\mr{Op}_{\N, X_0}^{^\mr{Zzz...}}$) and the fiber of  $\mcO p_{2, 1, g, 0, \mbZ/p^\N \mbZ}^\ZZZ$   (resp., $\mcO p_{2, \N, g, 0, \mbF_p}^\ZZZ$) over
the classifying morphism $\sigma_X : S \migi \overline{\mcM}_{g, r, \mbZ/p^\N \mbZ}$ (resp., $\sigma_{X_0} : S_0 \migi \overline{\mcM}_{g, r, \mbF_p}$) of $X$ (resp., $X_0$).

Taking the diagonal reductions of dormant $\mr{PGL}_2$-opers yields  a map of sets
\begin{align} \label{e210www}
\Diag_{\!\!\spadesuit} : \mr{Op}_{1, X}^{^\mr{Zzz...}}
 \migi \mr{Op}_{\N, X_0}^{^\mr{Zzz...}}
\end{align}
(cf. \eqref{YY260}).
Then, we have the following assertion.

\SSP
\begin{tdef} \label{T04}
Suppose that $X_0/S_0$ is $\N$-ordinary (cf. Definition \ref{DD456}).
Then, the map $\Diag_{\!\!\spadesuit} : \mr{Op}_{1, X}^{^\mr{Zzz...}}
 \migi \mr{Op}_{\N, X_0}^{^\mr{Zzz...}}$ is bijective.
 In particular, for each dormant  $\mr{PGL}_2^{(\N)}$-oper $\msE^\spadesuit_0$ on $X_0/S_0$,
 there exists a unique (up to isomorphism) dormant $\mr{PGL}_2$-oper $\msE^\spadesuit$ on $X/S$ with 
 ${^{\Diag}}\!\!\msE^\spadesuit = \msE^\spadesuit_0$.
We shall refer to $\msE^\spadesuit$ as the {\bf canonical diagonal lifting} of $\msE^\spadesuit_0$.
 \end{tdef}
\begin{proof}
According to Proposition \ref{P204}, we can find  a dormant $2^{(1)}$-theta characteristic $\vartheta := (\varTheta, \DMO_\vartheta)$ of $X/S$.
Denote by $\vartheta_0^{(\N)}$ the $2^{(\N)}$-theta characteristic of $X_0/S_0$ obtained as the diagonal reduction of $\vartheta$.
Now, let $\msE_0^\spadesuit := (\mcE_{B, 0}, \STR_0^{(\N)})$ be a dormant $\mr{PGL}_2^{(\N)}$-oper on $X_0/S_0$.
For each positive integer $\N' \leq \N$,
we denote by $\STR^{(\N')}_0$ the $(\N'-1)$-PD stratification on $\mcE$ induced by $\STR^{(\N)}_0$.
The resulting dormant $\mr{PGL}_2^{(\N')}$-oper $\msE_0^{\spadesuit (\N')} := (\mcE_{B, 0}, \STR_0^{(\N')})$
 corresponds, via
 the bijection $\Lambda_{\diamondsuit \Rightarrow \spadesuit, \vartheta}^\ZZZ$ (cf.  Theorem \ref{P14}), to 
   a dormant $(\mr{GL}_2^{(\N')}, \vartheta_0^{(\N')})$-oper $\DMO^{\diamondsuit (\N')}_0$ on $X_0/S_0$.

Note that there exists a canonical identification $\mcE nd^0 (\mcF_{\varTheta_0}) = \mfg_{\mcE_0}$ (where $\mcE_{0} := \mcE_{B, 0} \times^B \mr{PGL}_2$), which is compatible with the respective $\mcD_{X_0/S_0}^{(\N' -1)}$-module structures, i.e., $(\nabla_{0}^{\diamondsuit (\N')})_{\mcE nd^0}$ and $\DMO_{\STR_0^{(\N')}}^{\mr{ad}}$.
In particular, we have 
\begin{align}
\mcE nd^0 (\mcS ol (\DMO_0^{\diamondsuit (\N')})) \left(= \mcS ol ((\nabla_{0}^{\diamondsuit (\N')})_{\mcE nd^0})\right) = \mfg_{\mcE_0, \STR_0^{(\N')}}.
\end{align}

Under this identification, the morphism $D''_{\DMO_0^{\diamondsuit (\N')}}$   coincides with  $D''_{\msE_0^{\spadesuit (\N')}}$ (cf. \eqref{NN25}).
Hence, since $X_0/S_0$ is $\N$-ordinary, the morphisms $D''_{\DMO_0^{\diamondsuit (\N')}}$  are bijective for all $\N'$'s.
By Corollary \ref{Coeedk},
there exists a unique  (up to isomorphism) dormant $(\mr{GL}_2^{(1)}, \vartheta)$-oper $\DMO^\diamondsuit$ on $X/S$ whose diagonal reduction coincides with  $\DMO^{\diamondsuit (\N)}_0$.
The dormant $\mr{PGL}_2$-oper $\msE^\spadesuit$ corresponding to  $\DMO^\diamondsuit$ specifies a diagonal lifting of  $\msE^\spadesuit_0$.
The resulting assignment  $\msE_0^\spadesuit \mapsto \msE^\spadesuit$ specifies a well-defined map of sets $\mr{Op}_{\N, X_0}^\ZZZ \migi \mr{Op}_{1, X}^\ZZZ$, and it gives   an inverse map of  $\Diag_{\!\!\spadesuit}$.
This completes the proof of the assertion.
\end{proof}
\SSP

\begin{rema}[Affineness assumption] \label{WW90}
By the uniqueness  assertion  in the above theorem, the formation of  the canonical diagonal liftings commutes with base-change over $S$-schemes.
Hence,  Theorem-Definition \ref{T04} is verified to be  true even when we remove the   affineness assumption on $S$.
\end{rema}
\SSP

\begin{rema}[Representability of an ordinary locus of $\mcO p^\ZZZ_{2, 1, g, 0, \mbZ/p^\N \mbZ}$] \label{WW91}
Choose a dense open substack $\mcM_{\mbF_p}$ of $\mcM_{g, 0, \mbF_p} \cap \overline{\mcM}_{g, 0, \N\text{-}\mr{ord}}$, and
let 
 $\mcM_{\mbZ/p^\N \mbZ}$ be the corresponding open substack of   $\mcM_{g, 0, \mbZ/p^\N \mbZ}$.
Since the projection $\mcO p_{2, \N, g,0, \mbF_p}^\ZZZ \times_{\overline{\mcM}_{g, 0, \mbF_p}} \mcM_{\mbF_p}\migi \mcM_{\mbF_p}$ is \'{e}tale,
there exists a unique (up to isomorphism) \'{e}tale stack $\mcU$ over $\mcM_{\mbZ/p^\N \mbZ}$ whose reduction modulo $p$ is isomorphic to 
$\mcO p_{2, \N, g,0, \mbF_p,  \mr{ord}}^\ZZZ \times_{\overline{\mcM}_{g, 0, \mbF_p}} \mcM_{\mbF_p}$.

Now,   let  us take   an \'{e}tale scheme $S$ over $\mcM_{\mbZ/p^\N \mbZ}$ and a section $u$ of $\mcU$ over $S$.
  The structure morphism $S \migi \mcM_{g, 0, \mbZ/p^\N \mbZ}$ of $S$ classifies a  smooth curve $X/S$ whose reduction $X_0/S_0$ modulo $p$ is $\N$-ordinary, and the reduction modulo $p$ of $u$ defines a dormant $\mr{PGL}_2^{(\N)}$-oper $\msE^\spadesuit_0$ on $X_0/S_0$.
  According to Theorem-Definition \ref{T04},
  we obtain  the canonical diagonal lifting  $\msE^\spadesuit$ of $\msE^\spadesuit_0$ defined on $X/S$.
  In particular, the pair $(X/S, \msE^\spadesuit)$ specifies an object of the category $\mcO p_{2, 1, g,0, \mbZ/p^\N \mbZ}^\ZZZ \times_{\overline{\mcM}_{g, 0, \mbZ/p^\N \mbZ}} \mcM_{\mbZ/p^\N \mbZ}$ over $S$.
 
  One may verify that  the functor 
  \begin{align}
  \mcU \migi \mcO p_{2, 1, g, 0, \mbZ/p^\N \mbZ}^\ZZZ \times_{\overline{\mcM}_{g, 0, \mbZ/p^\N \mbZ}} \mcM_{\mbZ/p^\N \mbZ}
  \end{align}
   given by assigning $u \mapsto (X/S, \msE^\spadesuit)$ becomes an isomorphism between the \'{e}tale sheaves on $\mcM_{\mbZ/p^\N \mbZ}$ represented by the respective fibered categories.
 In particular, by using this isomorphism, we can construct   a universal family, parametrized by  $\mcU$,  of
 dormant $\mr{PGL}_2$-opers whose diagonal reductions are ordinary.
\end{rema}
\SSP

\begin{rema}[Comparison with $p$-adic Teichm\"{u}ller theory] \label{WW92}
One of the main achievements in $p$-adic Teichm\"{u}ller theory asserts the existence of  
a canonical $p$-adic  lifting of a curve equipped with 
a certain additional structure involving   a $\mr{PGL}_2$-oper (=
a torally  indigenous bundle).

Each object classified by the {\it VF-stack}, or the {\it shifted  VF-stack}, {\it of pure tone} (cf. ~\cite[Chapter IV, Definition 2.6]{Mzk2}) is  related to the notion of a dormant $\mr{PGL}_2^{(\N)}$-oper because both are characterized by  vanishing  of  the $p$-curvature of the associated  flat bundles.
However, despite the similarity in concept, 
there seems to be no (at least direct) correspondence with 
 the canonical diagonal liftings asserted in the above theorem.

Also, other lifting constructions of $\mr{PGL}_2$-opers
  can be found in ~\cite{Mzk1} (which is also in the context of $p$-adic Teichm\"{u}ller theory, but in a more classical setting) and ~\cite{LSYZ}.
These constructions are entirely different from  ours because the $\mr{PGL}_2$-opers treated there  
never have vanishing $p$-curvature.
\end{rema}

\SSP
\bco \label{C998}
Let $\N$ be a positive integer and  $k$ an algebraically closed field over $\mbF_p$.
Denote by   
 $W_\N$ the ring of Witt vectors of length  $\N$ over $k$.
Also, let $X$ be a geometrically connected, proper, and  smooth curve of genus $g>1$ over $W_\N$.
Denote by $X_0$ the reduction modulo $p$ of $X$, and suppose that $X_0$ is $\N$-ordinary.
Then, there are only finitely many isomorphism classes of dormant $\mr{PGL}_2$-opers on $X$.
If, moreover, $X_0$ is general in $\mcM_{g, 0}$, then
the cardinality $\sharp (\mr{Op}^\ZZZ_{1, X})$ of $\mr{Op}^\ZZZ_{1, X}$ satisfies  the following equalities:
\begin{align} \label{NN347}
\sharp (\mr{Op}^\ZZZ_{1, X}) = \sharp (\mr{Op}^\ZZZ_{\N, X_0}) = \mr{deg}(\Pi_{2, \N, g, 0, \mbF_p}).
\end{align} 
\eco
\begin{proof}
The first assertion follows from Theorem-Definition \ref{T04}.
The second assertion follows from Corollary \ref{Cor87} and Theorem \ref{c43}.
\end{proof}
\SSP

\begin{rema}[$p$-adic liftings of  dormant $\mr{PGL}_2^{(\infty)}$-opers] \label{NN840}
Let $K$ be an algebraically closed field over $\mbF_p$, and  let $X_0$ be a connected  proper smooth curve of genus $g >1$ over $K$ classified  by a $K$-rational {\it generic} point of $\mcM_{g, 0}$.
Denote by $W$ the ring of Witt vectors over $K$.
 Also, choose
 a  deformation  $X$ of  $X_0$ over $W$.

According to  (the proof of) Corollary \ref{NN812} and Theorem \ref{c43},
we can find
a collection of data
\begin{align}
\{ s_\N \}_{\N \in \mbZ_{>0}},
\end{align}
where 
\begin{itemize}
\item
each $s_\N$ ($\N \in \mbZ_{>0}$) denotes a $K$-valued generic point of an irreducible component of $\mcO p^\ZZZ_{2, \N, g, 0, \mbF_p}$;
\item
the points $\{ s_\N\}_{\N \in \mbZ_{>0}}$ are compatible with respect to the projective system \eqref{NN8123}, i.e., $s_\N$ is mapped to $s_{\N'}$ via the natural projection $\Pi_{\N \Rightarrow \N'} : \mcO p^\ZZZ_{2, \N, g, 0, \mbF_p} \migi \mcO p^\ZZZ_{2, \N', g, 0, \mbF_p}$ for every $\N' \leq \N$.
\end{itemize}
The point $s_\N$ determines a dormant $\mr{PGL}_2^{(\N)}$-oper $\msE_\N^\spadesuit$ on $X_0/K$.
 The resulting  {\it compatible} collection $\{ \msE^\spadesuit_\N \}_{\N \in \mbZ_{>0}}$ should be  called  a dormant $\mr{PGL}_2$-oper {\it of level $\infty$}, or  {\it dormant $\mr{PGL}_2^{(\infty)}$-oper}.

Since $X_0$ is $\N$-ordinary for every $\N$, we can apply
 Theorem-Definition \ref{T04}
  to the $\msE_\N^\spadesuit$'s, and  
 obtain  a dormant $\mr{PGL}_2$-oper $\msE_\infty^\spadesuit$ defined on a curve $X/W$  via algebraization.
This construction would partially realize  the expected correspondence (at least in the generic situation)  between dormant opers of level $\infty$
on a curve  in characteristic $p$ and dormant opers of level $1$ on its $p$-adic lifting.
\end{rema}

\LSP
\subsection{Canonical diagonal liftings of  $F^\N$-projective structures}\label{SS101ff}

The notion of an
 $F^\N$-projective structures (i.e., a Frobenius-projective structures of level $\N$) on a curve in characteristic $p$ was   originally introduced  by
   Y. Hoshi  (cf.   ~\cite[Definition 2.1]{Hos2}) as an analogue of a complex projective structure on a Riemann surface,  and later investigated by the author (cf. ~\cite{Wak6}, ~\cite{Wak7}).
 Roughly speaking, 
 an $F^\N$-projective structure is a  maximal collection of \'{e}tale coordinate charts on
a curve
  valued in the projective line whose transition functions descend to its $\N$-th Frobenius twist.

The statement of  ~\cite[Theorem A]{Hos2} may be interpreted as the existence of  a natural bijection  between dormant $\mr{PGL}_2^{(\N)}$-opers and $F^\N$-projective structures.
(In  ~\cite[Definition 2.2.2]{Wak7}, the author introduced  Frobenius-Ehresmann structures, which are   multi-dimensional generalizations of  Frobenius-projective structures. 
The same kind of  bijection  as mentioned above holds for those structures, see
  ~\cite[Theorem A]{Wak6} and ~\cite[Theorem A]{Wak7}.)

In this final subsection, we extend the definition of  a projective structure  to 
 the case of  characteristic $p^\N$ (cf. Definition \ref{D0188} below).
 After generalizing Hoshi's bijection,
 we formulate the canonical diagonal liftings  of $F^\N$-projective structures as a direct consequence  of  Theorem-Definition \ref{T04}.

Let  $\CH$ be  a nonnegative integer, 
$S$ a flat $\mbZ/p^{\CH +1} \mbZ$-scheme,
and $X$ a geometrically connected, proper, and  smooth  
curve  of genus $g >1$ over $S$.
 Suppose that $S$ is equipped with an $(\N -1)$-PD structure extending $X$.
Denote by  $(\mr{PGL}_{2})_X$ the Zariski sheaf on $X$ represented by $\mr{PGL}_2$; this may be identified with the  sheaf of automorphisms of  the trivial $\mr{PGL}_2$-bundle $X \times \mr{PGL}_2$.
Also, denote by
\begin{align} \label{NN603}
(\mr{PGL}_2^{(\N)})_X
\end{align}
the subsheaf of $(\mr{PGL}_{2})_X$ consisting of (locally defined) automorphisms of
$X \times \mr{PGL}_2$ preserving the $(\N -1)$-PD stratification $\STR_\mr{triv}$  (cf. \eqref{NN600} for the definition of $\STR_\mr{triv}$).

Next, let us write
\begin{align} \label{Ep120}
\mcP^{\text{\'{e}t}}
\end{align}
for the Zariski  sheaf of sets on $X$ that assigns, to each open subscheme  $U$ of $X$, the set of \'{e}tale $S$-morphisms
  from $U$ to the projective line  $\mbP$ over $S$.

\SSP
 \ble
Let $U$ be an open subscheme of $X$, $\varphi : U \migi \mbP$ an element of  $\mcP^{\text{\'{e}t}}(U)$.
Also,  let $h$ be an element  of $(\mr{PGL}_2^{(\N)})_X (U)$, regarded as an $S$-morphism $U \migi \mr{PGL}_2$.
Then,  
  the composite 
\begin{align} \label{Eg69}
{^h}\varphi : U 
\xrightarrow{(h, \varphi)} \mr{PGL}_2 \times \mbP   \xrightarrow{\Psi}\mbP
\end{align}
belongs to $\mcP^{\text{\'{e}t}}(U)$, where the second arrow $\Psi$ denotes the usual  $\mr{PGL}_2$-action on $\mbP$.
\ele
\begin{proof}
After possibly replacing $U$ with its covering with respect to the \'{e}tale topology,
we may assume that 
there exists  an automorphism $\widetilde{h}$ of $(\mcO_X^{\oplus 2}, (\DMO_{X, \mr{triv}}^{(\N -1)})^{\oplus 2})$ inducing   $h$ via projectivization; we will use the same notation $\widetilde{h}$ to denote the associated morphism $U \migi \mr{GL}_2$.
Let us choose  a point $q \in U (k)$, where $k$ denotes a field.
The differential  of $(h, \varphi): U \migi \mr{PGL}_2 \times \mbP$ at $q$ yields  a morphism  of $k$-vector spaces
\begin{align} \label{e555}
(dh, d\varphi) : \mcT_{U/S} |_q  \migi \mcT_{\mr{PGL}_2} |_{h (q)} \oplus \mcT_{\mbP/S} |_{\varphi (q)} \left(= \mcT_{(\mr{PGL}_2 \times \mbP)/S} |_{(h (q), \varphi (q))} \right).
\end{align}
Note that the differential  of $\widetilde{h}$ coincides with  the zero map because $\widetilde{h}$ lies in $\mr{Aut}(\mcO_X^{\oplus 2}, (\DMO_{X, \mr{triv}}^{(\N-1)})^{\oplus 2})$.
Hence, since $h$ factors through $\widetilde{h}$, the morphism   $d h : \mcT_{U/S} |_q \migi \mcT_{\mr{PGL}_2} |_{h (q)}$  coincides with  the zero map.
On the other hand, by the \'{e}taleness assumption on $\varphi$,
the morphism  $d \varphi : \mcT_{U/S} |_q \migi \mcT_{\mbP/S} |_{\varphi (q)}$ is an isomorphism.
Also, the differential $d \Psi : \mcT_{\mr{PGL}_2} |_{h (q)} \oplus \mcT_{\mbP/S} |_{\varphi (q)}\migi \mcT_{\mbP/S} |_{h(\varphi)(q)}$ of the $\mr{PGL}_2$-action  $\Psi$
restricts  to an isomorphism $\mcT_{\mbP/S}|_{\varphi (q)} \left(=0 \oplus \mcT_{\mbP/S}|_{\varphi (q)}\right) \isom \mcT_{\mbP/S} |_{{^h}\varphi (q)}$.
It follows that the differential $d ({^h}\varphi) \left(=  d\Psi \circ (d h, d \varphi)\right)$ of ${^h}\varphi$ at $q$ is  an isomorphism.
Since both $U$ and $\mbP$ are smooth over $S$,  the morphism  ${^h}\varphi$ turns out to be  \'{e}tale.
This completes the proof of this lemma.
\end{proof}
\SSP

By the above lemma, the assignment $(h, \varphi) \mapsto {^h}\varphi$   defines a  $(\mr{PGL}_2^{(\N)})_X$-action
\begin{align}
\Psi^{(\N)}: (\mr{PGL}_2^{(\N)})_X \times \mcP^{\text{\'{e}t}} \migi \mcP^{\text{\'{e}t}}
\end{align}
 on the sheaf $\mcP^{\text{\'{e}t}}$.

\SSP\bde
 \label{D0188} 
 Let $\mcS^\ST$  be a subsheaf of $\mcP^{\text{\'{e}t}}$.
 We shall say that $\mcS^\ST$ is a {\bf Frobenius-projective} {\bf  structure of level $\N$}  (or simply,  an {\bf $F^\N$-projective  structure})  on $X/S$  if it  is closed under the  $(\mr{PGL}_2^{(\N)})_X$-action $\Psi^{(\N)}$ 
   and forms a $(\mr{PGL}_2^{(\N)})_X$-torsor  on $X$ with respect to the resulting $(\mr{PGL}_2^{(\N)})_X$-action  on $\mcS^\ST$.
     \ede

\SSP
\begin{rema}[Comparison with the classical definition] \label{NN607}
Suppose that $\CH =0$.
Then, since 
 $(\mr{PGL}_2^{(\N)})_X$ coincides with the subsheaf $(F_{X/S}^{(\N)})^{-1}((\mr{PGL}_{2})_{X^{(\N)}})$ of $(\mr{PGL}_2)_X$,
the definition of an $F^{\N}$-projective structure described 
above
  is the same as the classical  definition discussed   in  ~\cite[Definition 2.1]{Hos2} and  ~\cite[Definition 2.2.1]{Wak6}.
\end{rema}
\SSP

 Denote by
\begin{align} \label{Eg61}
F^\N\text{-}\mr{Pr}_{X}
\end{align}
the set of $F^\N$-projective  structures on $X/S$.
The following assertion generalizes  ~\cite[Theorem A]{Hos2} and ~\cite[Theorem B]{Wak7} (for $(G, H) = (\mr{PGL}_2, B)$) to characteristic $p^\N$.

\SSP
\bt \label{T14}
Suppose that either $\CH =0$ or $\N =1$ is satisfied (hence the set $\mr{Op}_{\N, X}^\ZZZ$ can be defined as in \eqref{e209}).
Then, there exists a canonical bijection of sets
\begin{align} \label{e778}
 \mr{Op}_{\N, X}^\ZZZ \isom  F^\N\text{-}\mr{Pr}_{X}.
\end{align}
\et
\begin{proof}
Since the assertion for  $\CH =0$ can be  proved by an argument entirely similar to the proof of  ~\cite[Theorem A]{Hos2} (or, to the following discussion),
it suffices to consider the case of $\CH \neq 0$ (and $\N =1$).
First, we shall construct a map of sets  $\mr{Op}_{1, X}^\ZZZ \migi F^1\text{-}\mr{Pr}_{X}$.
Let $\msE^\spadesuit := (\mcE_B, \DMO)$ be a dormant $\mr{PGL}_2$-oper on $X/S$, and write  $\mcE := \mcE_B \times^B \mr{PGL}_2$.
Denote by $\mbP (\mcE)$ the projective line bundle associated to $\mcE$.
The $B$-reduction $\mcE_B$ of $\mcE$ determines a global section $\sigma : X \migi \mbP (\mcE)$.
Next, let us take an open subscheme $U$ of $X$ on which  $(\mcE, \DMO)$ is 
trivialized
 (cf. 
  Corollary \ref{UU39}, (ii)).
Choose  a trivialization 
\begin{align} \label{NN677}
\alpha : (\mcE |_U, \DMO  |_U)\isom (U \times \mr{PGL}_2, \DMO_\mr{triv}),
\end{align} 
where $\DMO_\mr{triv}$ denotes the trivial $S$-connection on  $U \times \mr{PGL}_2$ (cf. ~\cite[Eq.\,(78)]{Wak8}); it  induces an isomorphism 
$\alpha_\mbP : \mbP (\mcE) |_U \isom U \times_S \mbP$.
Thus, we obtain the composite
\begin{align} \label{NN611}
\varphi_\alpha : U \xrightarrow{\sigma|_U} \mbP (\mcE) |_{U} \xrightarrow{\alpha_\mbP} U \times_S \mbP \xrightarrow{\mr{pr}_2} \mbP.
\end{align}
Similarly to  the proof of ~\cite[Lemma 2.2.1]{Wak7}, we see that $\varphi_\alpha$ is  \'{e}tale.
Let  $\msE^{\spadesuit \Rightarrow \ST} (U)$ denote the set of \'{e}tale morphisms $\varphi_\alpha$  constructed in this way for all possible  trivializations $\alpha$ as in \eqref{NN677}.
The set  $\msE^{\spadesuit \Rightarrow \ST} (U)$ has a structure of $(\mr{PGL}_2^{(1)})_X (U)$-torsor with respect to the $(\mr{PGL}_2^{(1)})_X (U)$-action defined by $\varphi_\alpha \mapsto \varphi_{h \circ \alpha}$ for each $h \in (\mr{PGL}_2^{(1)})_X (U)$.
Hence, the sheaf associated to the assignment $U \mapsto \msE^{\spadesuit \Rightarrow \ST} (U)$ specifies a subsheaf 
\begin{align} \label{NN772}
\msE^{\spadesuit \Rightarrow \ST}
\end{align}
of $\mcP^{\text{\'{e}t}}$.
One may verify that  $\msE^{\spadesuit \Rightarrow \ST}$ forms an $F^1$-projective structure on $X/S$, and 
the resulting assignment $\msE^\spadesuit \mapsto \msE^{\spadesuit \Rightarrow \ST}$ defines a map of sets  $\mr{Op}_{1, X}^\ZZZ \migi F^1\text{-}\mr{Pr}_{X}$.
Moreover, we can reverse the steps in the above construction, so
 the map  $\mr{Op}_{1, X}^\ZZZ \migi F^1\text{-}\mr{Pr}_{X}$  is verified to be bijective.
\end{proof}
\SSP

Hereinafter, we suppose that $\CH +1 = \N$.
Denote by $X_0/S_0$ the  curve obtained as the reduction modulo $p$ of $X/S$.
Let us define a map 
\begin{align} \label{NN763}
\Diag_{\!\!\ST} : F^1 \text{-} \mr{Pr}_{X} \migi F^\N\text{-}\mr{Pr}_{X_0}
\end{align}
to be the unique  map that  makes the following  square diagram commute:
\begin{align}  \label{fE66}
\vcenter{\xymatrix@C=56pt@R=36pt{
F^1 \text{-} \mr{Pr}_{X/S} \ar[r]^-{\Diag_{\!\!\ST}} \ar[d]^-{\wr}_-{\eqref{e778}}& 
F^\N\text{-}\mr{Pr}_{X_0}
\ar[d]_-{\wr}^{\eqref{e778}}\\
\mr{Op}^\ZZZ_{1, X}\ar[r]_-{\Diag_{\!\!\spadesuit}} & \mr{Op}^\ZZZ_{\N, X_0}.
}}
\end{align}

\SSP
\begin{tdef} \label{T1dddd}
Suppose that $X_0/S_0$ is $\N$-ordinary.
Then, 
the map $\Diag_{\!\!\ST}$ is bijective.
That is to say, 
for any $F^\N$-projective structure $\mcS^\ST_0$ on $X_0/S_0$,
there exists a unique $F^1$-projective structure $\mcS^\ST$ on $X/S$ which is mapped to $\mcS^\ST_0$ via $\Diag_{\!\!\ST}$.
We shall refer to $\mcS^\ST$ as the {\bf canonical diagonal lifting} of $\mcS^\ST_0$.
\end{tdef}
\begin{proof}
The assertion follows from  Theorem-Definition \ref{T04}, Remark \ref{WW90},  and Theorem \ref{T14}.
\end{proof}
\SSP

\bco \label{Cor5509}
Suppose that $X_0/S_0$ is $\N$-ordinary, and that there are only finitely many connected components of $S$.
Then,  there are only finitely many $F^1$-projective structures on $X/S$ and the following equality holds:
\begin{align}
\sharp (F^1 \text{-}\mr{Pr}_{X/S}) = \sharp (F^\N\text{-}\mr{Pr}_{X_0/S_0}).
\end{align}
\eco
\begin{proof}
The assertion follows Theorem \ref{T14} and Theorem-Definition \ref{T1dddd}.
\end{proof}

\vspace{10mm}
\section{Combinatorial description of dormant $\mr{PGL}_2^{(\N)}$-opers} \label{S10}\SSP

In this final section, the $2$d TQFT $\mcZ_{2, \N}$ resulting from Theorem \ref{Theorem4f4} is translated into some combinatorial objects  to solve our counting problem in a practical manner.
To this end, we study Gauss hypergeometric differential operators in characteristic $p^\N$ with a full set of root functions, which amounts to the study of  dormant $\mr{PGL}_2^{(\N)}$-opers on a $3$-pointed projective line via diagonal  lifting/reduction (cf. Proposition \ref{Prop332}).
Since such differential operators are determined by their exponent differences,
one can  describe    dormant $\mr{PGL}_2^{(\N)}$-opers by means of certain edge numberings on the trivalent semi-graph associated to a totally degenerate curve, as well as lattice points inside a generalized rational polytope (cf. Propositions \ref{LLL001}, \ref{Prop4378}).
As a consequence, we prove that the numbers of dormant $\mr{PGL}_2$-opers and  $2$nd order differential operators on a general curve in characteristic $p^\N$ can be expressed as a quasi-polynomial function (cf. Theorem \ref{Thm5}, \ref{Theorem895}).

 Let us fix 
 a positive integer $\N$, and suppose that $p >2$.

\LSP
\subsection{Gauss hypergeometric differential operators} \label{SS33}

(The following  discussion in the case of $\ell = 0$ can be found in ~\cite[Section 4.12]{Wak8}.)
Denote by $W_\N$ the ring of Witt vectors of length $\N$ over an algebraically closed field  $k$ of  characteristic $p$.
Let us consider  (linear) differential operators on 
the $3$-pointed projective line $\msP := (\mbP/S, \{ [0], [1], [\infty]\})$ over $S := \mr{Spec}(W_\N)$ (cf. \eqref{1051}).

Each triple $(a, b, c) \in W_\N^{\times 3}$ of elements of $W_\N$ determines   the  Gauss hypergeometric differential operator 
\begin{align}
\mpD_{a,b, c}^\clubsuit := \partial_x^2 + \left(\frac{c}{x} + \frac{1 -c + a + b}{x-1} \right) \cdot \partial_x +\frac{ab}{x (x-1)} 
\end{align}
on $\mr{Spec}(W_\N [x])  = \mbP \setminus \{\infty\}$, where
  $\partial_x := \frac{d}{dx}$.
It defines  a $2$nd  differential operator $\mcO_{\mbP} \migi \Omega^{\otimes 2}$ with unit principal symbol under the natural identification $\Omega^{\otimes 2} \otimes \mcT^{\otimes 2} = \mcO_\mbP$ (cf. Definition \ref{Def898}) in such a way that 
 $D_{a, b, c}^\clubsuit := dx^{\otimes 2} \otimes \mpD_{a, b, c}^\clubsuit$ specifies  a global section of 
$\mcD {\it iff}_{\!\leq 2} (\mcO_{\mbP}, \Omega^{\otimes 2})$.
If we write $y := x-1$, $z := 1/x$, then the following equalities hold:
\begin{align} \label{eQ221i}
 D_{a,b, c}^\clubsuit &= \left(\frac{dx}{x}\right)^{\otimes 2} \otimes \left((x \partial_x)^2 + \frac{(a+b) \cdot x+ 1-c}{x-1} \cdot x \partial_x + \frac{ab \cdot  x}{x-1} \right) \\
&= \left( \frac{dy}{y}\right)^{\otimes 2} \otimes \left((y\partial_y)^2 + \frac{(a+b) \cdot y -c + a+b}{y+1} \cdot y \partial_y + \frac{ab \cdot y}{y+1} \right) \notag \\
& = \left( \frac{dz}{z}\right)^{\otimes 2} \otimes \left((z \partial_z)^2 + \frac{(1-c) \cdot z  + a+b}{z-1} \cdot z \partial_z - \frac{ab}{z-1} \right). \notag
\end{align}
In this section, we shall  call such a differential operator (associated to some triple $(a, b, c) \in R^{\times 3}$)  a {\bf hypergeometric operator}.

The $S^\mr{log}$-connection
 $\mpD_{a, b, c}^{\clubsuit \Rightarrow \diamondsuit}$ on $\mcF_{\mcO_\mbP} |_{\mbP \setminus \{ \infty \}} = \mcO_{\mbP \setminus \{ \infty \}} \oplus \mcO_{\mbP \setminus \{ \infty \}} \cdot  x (x-1) \partial_x$ (cf. \eqref{GL14})
 satisfies 
 \begin{align}
\mpD_{a,b,c}^{\clubsuit \Rightarrow \diamondsuit} = d + \frac{dx}{x (x-1)} \otimes  \begin{pmatrix} 0& -ab \cdot x (x-1)\\  1 & (1-a-b)\cdot x -1+c  \end{pmatrix}.
\end{align}
Denote by  $\nabla_{a,b,c}$ the  
$S^\mr{log}$-connection
   on $\mcT$ 
expressed
  as 
\begin{align}
\nabla_{a,b, c} := d  + \frac{d x}{x(x-1)} \otimes  ((1-a-b) \cdot x -1 +c)
\end{align}
under   the identification $\mcO_{\mbP \setminus \{ \infty \}} = \mcT |_{\mbP \setminus \{ \infty \}}$ given by 
$v \leftrightarrow v \cdot x (x-1) \partial_x$.
Then, the pair 
  $\bb_{a,b, c} := (\mcO_{\mbP}, \nabla_{a,b, c})$ specifies   a $2^{(1)}$-theta characteristic of $\mbP^\mr{log}/S^\mr{log}$, and
$\mpD_{a,b,c}^{\clubsuit}$ (resp., $\mpD_{a,b,c}^{\clubsuit \Rightarrow \diamondsuit}$) lies in    $\mcD{\it iff}_{\!\clubsuit, \vartheta_{a, b, c}}$ (resp., $\mcO p_{\diamondsuit, \vartheta_{a, b, c}}$).

\SSP
\ble \label{LemLem}
Suppose that $a + b \in \mbZ/p^\N\mbZ$ and $c \in \mbZ /p^\N\mbZ$.
Then, the flat line bundle $(\mcT, \nabla_{a, b, c})$ (i.e., the $2^{(1)}$-theta characteristic $\vartheta_{a, b, c}$) is dormant.
\ele
\begin{proof}
We shall set $U := \mbP \setminus \{ [0], [1], [\infty] \}$.
By Proposition \ref{Prop277},
it suffices to show that the restriction $(\mcT |_U, \nabla_{a, b, c}|_U)$ over  $U$ is dormant. 
Let $s$ (resp., $t$) be   the  integer defined as the unique lifting of $a + b -c$ (resp.,  $-1 + c$), which  is an  element of $\mbZ /p^\N \mbZ$ by assumption, via the natural surjection $\mbZ \migisurj \mbZ/p^\N \mbZ$ satisfying $0 \leq s < p^\N$ (resp., $0 \leq t < p^\N$).
Consider the gauge transformation of $\nabla_{a, b, c}|_U$ by an element $x^{t} (x-1)^s \in  H^0 (U, \mcO_X)$.
If  $d + dx \otimes A$ denotes the resulting connection, then the section $A$ is computed as follows: 
\begin{align}
A & =
  \frac{\partial_x(x^t (x-1)^s)}{x^t (x-1)^s} + \frac{(1-a-b)\cdot x -1 + c}{x(x-1)} \\
& =
\left(\frac{a+ b -c}{x-1} + \frac{-1 + c}{x}\right) + \frac{(1-a-b)\cdot x -1 + c}{x(x-1)} \notag \\
& = 0.
\end{align}
This implies that $(\mcT|_U, \nabla_{a, b, c}|_U)$ can be trivialized, in particular,  it is dormant.
\end{proof}
\SSP

If $\msE^\spadesuit_{a, b, c}$ denotes the $\mr{PGL}_2$-oper  defined to be the image of  $\mpD^\clubsuit_{a, b, c}$ via the composite $\Lambda_{\diamondsuit \Rightarrow \spadesuit, \vartheta_{a, b, c}} \circ \Lambda_{\clubsuit \Rightarrow \diamondsuit, \vartheta_{a, b, c}}$, then 
the equalities in \eqref{eQ221i} show that  the radii  of $\msE^\spadesuit_{a, b, c}$ at the marked points $[0]$, $[1]$, $[\infty]$ in the sense of Remark \ref{Rem89e} are given, as elements of $W_\N$,  by  
\begin{align} \label{QG2000}
\varrho_{[0]} (\msE^\spadesuit_{a, b, c})= - \frac{(1-c)^2}{4}, \  \
\varrho_{[1]}  (\msE^\spadesuit_{a, b, c}) = -\frac{(c-a-b)^2}{4}, \  \
\varrho_{[\infty]} (\msE^\spadesuit_{a, b, c}) =- \frac{(b-a)^2}{4}. 
\end{align}

Here, recall that  the {\bf exponent differences} of $\mpD^\clubsuit_{a,b,c}$ at $0$, $1$, $\infty$ are $1-c$, $c-a-b$, $b-a$, respectively, and 
set
\begin{align} \label{eQ33m}
\mr{Ex}_{a,b,c} :=(1-c, c-a-b, b-a).
\end{align}
This triple  will be regarded as a triple of elements in $W_\N /\{ \pm 1 \}$ ($:=$ the quotient set of $W_\N$ by the equivalence relation generated by $v \sim -v$ for every $v \in W_\N$).

Applying the fact mentioned at the end of Remark \ref{Rem89e} to the case of $n=2$, we see that (the isomorphism class of)  an $\mr{PGL}_2$-oper is uniquely determined by its radii.
Hence,   \eqref{QG2000} implies 
the following assertion.

\SSP
\bpr  \label{Prop554}
\begin{itemize}
\item[(i)]
Let $\msE^\spadesuit$ be a $\mr{PGL}_2$-oper on $\msP$ whose   radius at every marked point   belongs to $W_N^\times \left(= W_\N \setminus pW_\N \right)$.
  Then, the  preimage   of $\{ \msE^\spadesuit \}$ via  
the map of sets
 \begin{align} \label{Eq333}
\left\{\begin{matrix} \text{the set of} \\
\text{hypergeometric operators} \end{matrix} \right\} \migi 
\left\{\begin{matrix} \text{the set of isomorphism classes} \\
\text{of $\mr{PGL}_2$-opers on $\msP$}  \end{matrix} \right\}
\end{align}
given by 
 assigning  $\mpD^{\clubsuit}_{a, b, c} \mapsto \msE^\spadesuit_{a, b, c}$ is nonempty. 
 \item[(ii)]
 Let 
$(a, b, c)$ and $(a', b', c')$ be triples of  elements in $W_\N$.
Then,  $ \msE^\spadesuit_{a, b, c}\cong \msE^\spadesuit_{a', b', c'}$ 
  if and only if
  the equalities
   \begin{align}
   (1-c)^2 = (1 -c')^2, \hspace{3mm}(c-a-b)^2 = (c' -a' -b')^2,  \hspace{3mm} (b-a)^2 = (b'-a')^2
   \end{align}
    hold.
  In particular, 
  if all the factors in $\mr{Ex}_{a,b,c}$ and $\mr{Ex}_{a', b', c'}$ belong to $W_\N^\times$, then 
  $ \msE^\spadesuit_{a, b, c}\cong \msE^\spadesuit_{a', b', c'}$  if and only if 
   the equality  $\mr{Ex}_{a,b,c} = \mr{Ex}_{a', b', c'}$ holds in $(W_\N/\{ \pm 1 \})^{\times 3}$.
  \end{itemize}
 \epr
\SSP
 
 Next, we discuss {\it dormant} $\mr{PGL}_2$-opers arising from hypergeometric  operators.
 To this end, we first prove the following assertion.

\SSP
\ble \label{Prop11}
Let $\msE^\spadesuit$ be a $\mr{PGL}_2$-oper on $\msP$ and $U$ a dense open subscheme of $\mbP \setminus \{ [0], [1], [\infty]\}$.
Then,
$\msE^\spadesuit$ is dormant if and only if its restriction $\msE^\spadesuit |_U$ is dormant
 \ele
\begin{proof}
For simplicity, we shall write $\mcD^{(m)}$ (for each $m \geq 0$) instead of $\mcD^{(m)}_{\mbP_0^\mr{log}/S_0^\mr{log}}$.
The ``only if" part of the required equivalence is clear.

We shall prove  the ``if" part by induction on $\N$.
By Proposition \ref{P016dd}, the second assertion of (i),  we may assume, without loss of generality, that
$U = \mbP \setminus \{ [0], [1], [\infty]\}$.
The base step, i.e., the case of $\N = 1$,  follows from the density of $U \left(\subseteq \mbP \right)$ because 
the $p$-curvature of a flat bundle on $\mbP^\mr{log}_0/S^\mr{log}_0$ can be regarded as a global section of a certain associated  vector  bundle.
Next,  to discuss the induction step, suppose that we have proved the required assertion with $\N$ replaced  with $\N -1$ ($\N >1$).
Also,  suppose that $\msE^\spadesuit |_U$ is dormant.
Let $\rho_x \in (\mbZ /p^{\N -1}\mbZ)^\times / \{ \pm 1 \}$ ($x \in \{ 0, 1, \infty \}$) be the radius of $\msE_{\N -2}^\spadesuit$ at $[x]$.
Then, there exists a triple of integers $(\lambda_0, \lambda_1, \lambda_\infty)$ satisfying the following conditions (cf. the discussion at the beginning of Section \ref{SS042e}):
\begin{itemize}
\item
$2 \cdot \rho_x = \lambda_x$ as elements of $(\mbZ/p^{\N -1}\mbZ)/\{ \pm 1 \}$ and $0 < \lambda_x < p^{\N -1}$ for every $x = 0, 1, \infty$;
\item
The sum $\lambda_0 + \lambda_1 + \lambda_\infty$ is odd $< 2 \cdot p^{\N -1}$.
\end{itemize}
(The following discussion is entirely similar to the construction of the canonical dormant $\mr{GL}_2^{(\N)}$-oper, see Section \ref{SS042e}.)
We can find a unique $S^\mr{log}$-connection $\nabla^+$ on $\mcO_{\mbP}^+ := \mcO_\mbP (\lambda_0 \cdot [0] + \lambda_1 \cdot [1] + \lambda_\infty \cdot [\infty])$ extending the trivial connection on $\mcO_\mbP$.
Also, let $\mcL$ be a unique (up to isomorphism) line bundle on $\mbP$ of relative degree $\frac{\lambda_0 + \lambda_1 + \lambda_\infty +1}{2}$, and fix an identification $\mcL^{\otimes 2} \otimes \mcT = \mcO^+_\mbP$.
Under this identification, the pair $\vartheta := (\mcL, \nabla^+)$ forms a dormant $2^{(\N)}$-theta characteristic of $\mbP^\mr{log}/S^\mr{log}\left( =S\right)$ (cf. Proposition \ref{P72e}, (i)).
Denote by $\nabla^\diamondsuit$ be the $(\mr{GL}_{2}^{(\N)}, \vartheta)$-oper on $\msP$ corresponding to $\msE^\spadesuit$ via the isomorphism $\Lambda_{\diamondsuit \Rightarrow \spadesuit, \vartheta}$ (cf. Theorem \ref{P14}).
By the induction assumption,  $\msE^\spadesuit_{\N -2}$ is dormant, and hence, 
the reduction $\nabla^\diamondsuit_{\N -2}$ of $\nabla^\diamondsuit$ modulo $p^{\N -1}$ is dormant.
It follows from the definition of $\nabla^\diamondsuit$ that
the diagonal reduction of 
$\nabla_{\N -2}^{\diamondsuit \Rightarrow \heartsuit}$
  is canonical in the sense of  Definition \ref{NN49}; we denote it by 
 $\msF^\heartsuit_0 := (\mcF_0, \nabla_0^{(\N -2)}, \{ \mcF^j_0 \}_{j=0}^2)$. 
Just as in \eqref{J13}, we have an $S^\mr{log}$-connection $\nabla_\msF$ on $\mcV_\msF := \mcS ol (\nabla_0^{(\N -2)})$ associated to $\msF := (\mcF_\mcL, \nabla^\diamondsuit)$, and $\nabla_\msF$ has vanishing $p$-curvature because $\nabla^\diamondsuit |_U$ is dormant.
According to ~\cite[Corollaire 3.3.1]{Mon},
$\nabla_\msF$ induces a $\mcD^{(\N -1)}$-module structure $\nabla_0^{(\N -1)}$ on 
$\mcF^\flat_0 := F^{(N-1)*}_{\mbP_0/S_0} (\mcV_\msF)$ with vanishing $p^\N$-curvature that are compatible with $\nabla_0^{(\N -2)}$ via the inclusion $\tau : \mcF^\flat_0 \hookrightarrow \mcF_0$ and the natural morphism $\mcD^{(\N -2)} \rightarrow \mcD^{(\N -1)}$.
Denote by $\breve{\nabla}^{(\N -1)}_0$ the $\mcD^{(\N -1)}$-module structure on $u_* (\mcF^\flat_0 |_U)$ extending $\nabla_0^{(\N -1)}$ via the open immersion $u : U \hookrightarrow \mbP$.
Here, note that 
 $\mcF_0$ can be obtained as 
 the extension of $\mcF_0^\flat$ along the inclusion $\mcF_0^\flat \cap \mcL_0    \left(= \mcF^\flat_0 \cap \mcF_0^1 \right) \hookrightarrow \mcL_0 \left(= \mcF^1_0\right)$, i.e., 
 the push-forward of the diagram
\begin{align} \label{Eq1000}
\vcenter{\xymatrix@C=46pt@R=36pt{
\mcF^\flat_0 \cap \mcL_0  \ar[r]^-{\mr{inclusion}} \ar[d]_-{\mr{inclusion}} &\mcF^\flat_0 \\
 \mcL_0  & 
}}
\end{align}
(cf. Proposition  \ref{P100}, (i)).
By applying this fact together with an argument similar to the proof of Proposition \ref{Prop277},
we see that 
$\mcF_0 \left(\subseteq  u_* (\mcF^\flat_0 |_U) \right)$ is  closed under $\breve{\nabla}_0^{(\N -1)}$.
 The resulting $\mcD^{(\N -1)}$-module structure $\breve{\nabla}_0^{(\N -1)} |_{\mcF_0}$ on $\mcF_0$ satisfies
$(\msF, \breve{\nabla}_0^{(\N -1)} |_{\mcF_0}) \in \mr{Diag}_{\N -1}$ (cf. \eqref{ee1}), which means  that $\msE^\spadesuit $ is dormant. 
This completes the proof  of  the  ``if" part of the desired equivalence. 
\end{proof}

\SSP
\bpr \label{Prop21}
The assignment $(a, b, c) \mapsto \msE^\spadesuit_{a, b, c}$ determines an $8$-$1$ correspondence 
\begin{align} \label{QG2010}
&\left\{ (a, b, c) \in W_\N^{\times 3} \Biggl| \begin{matrix}
\text{the restriction  of $\mpD_{a, b, c}^\clubsuit$ to some dense} \\
\text{open subschemes of $\mbP \setminus \{ [0], [1], [\infty] \}$} \\
\text{have  full sets of root functions}
\end{matrix}\right\}  \\
&\stackrel{8 : 1}{\longleftrightarrow}
\left\{\begin{matrix}
\text{the set of isomorphism classes } \\
\text{of dormant $\mr{PGL}_2$-opers on $\msP$}
\end{matrix} \right\}
\end{align}
(cf. Definition \ref{D01fgh} for the definition of having a full set of root functions).
In particular, if a hypergeometric operator $\mpD_{a, b, c}^\clubsuit$ has a full set of root functions when restricted to some dense open subscheme of $\mbP \setminus \{ [0], [1], [\infty] \}$, then we have $(a, b, c) \in (\mbZ/p^\N \mbZ)^{\times 3}$.
\epr
\begin{proof}
The first assertion follows from  Proposition \ref{P39}, Proposition  \ref{Prop554}, (i) and (ii),  Lemma \ref{Prop11}, and the fact that the morphism $\Lambda^\mr{full}_{\clubsuit \Rightarrow \diamondsuit, (-)}$ (cf. \eqref{QQ379}) is an isomorphism.

Moreover, the first assertion and  
a comment in Remark \ref{Rem89e}
 together imply that $\frac{1- c}{2}$, $\frac{c- a - b}{2}$, $\frac{b-a}{2} \in \mbZ/p^\N \mbZ$.
This proves 
the second assertion.
\end{proof}

\LSP
\subsection{Combinatorial patterns of radii} \label{SS30}

We try  to understand 
  which triples of integers  $(a, b, c)$ yield  hypergeometric operators $\mpD^\clubsuit_{a, b, c}$ having a full set of solutions.

Given nonnegative integers $m$, $\ell$ with $m  < \ell$,
we set $\mbZ_{[m, \ell]} := \{m, m + 1, \cdots, \ell \}$.
The assignments $(a, b, c) \mapsto \left(\overline{\frac{1-c}{2}},  \overline{\frac{c-a-b}{2}}, \overline{\frac{b-a}{2}}\right)$ and  $(s_1, s_2, s_3) \mapsto \left(\overline{\frac{2s_1 +1}{2}}, \overline{\frac{2s_2 +1}{2}}, \overline{\frac{2s_3 +1}{2}}\right)$ determine  surjective maps of sets
\begin{align}
\xi_\N : \mbZ_{[1, p^\N]}^{\times 3} \migi ((\mbZ/p^\N \mbZ)/\{\pm 1 \})^{\times 3} 
\hspace{3mm} \text{and} \hspace{3mm}
\zeta_\N : \mbZ_{[0, p^\N-1]}^{\times 3} \migi ((\mbZ/p^\N \mbZ)/\{\pm 1 \})^{\times 3},
\end{align}
respectively.
Also, given  an integer $a$, 
we denote by $[a]_\N$ the remainder obtained by dividing $a$ by $p^\N$ (taken to be an element of $\mbZ_{[0, p^\N -1]}$ even when $a < 0$), and write 
\begin{align} \label{e2}
[a]'_\N := \begin{cases}
[a]_\N
 & \text{if $a \notin p^N \mbZ$;}\\ p^N & \text{if $a \in p^N \mbZ$.} \end{cases} 
\end{align}

If $\N'$ is a  positive integer  with $\N' \leq \N$,
then the assignment 
$(s_1, s_2, s_3) \mapsto ([s_1]'_{\N'}, [s_2]'_{\N'}, [s_3]'_{\N'})$  and 
$(s_1, s_2, s_3) \mapsto ([s_1]_{\N'}, [s_2]_{\N'}, [s_3]_{\N'})$ determine  surjections of sets
\begin{align} \label{e200ddd}
\tau'_{\N \Rightarrow \N'} : \mbZ_{[1, p^\N]}^{\times 3} \migisurj 
\mbZ_{[1, p^{\N'}]}^{\times 3} \hspace{3mm} \text{and} \hspace{3mm}
\tau_{\N \Rightarrow \N'} : \mbZ_{[0, p^\N-1]}^{\times 3} \migisurj 
\mbZ_{[0, p^{\N'}-1]}^{\times 3},
\end{align}
respectively, and
these make the following square diagrams commute:
\begin{align} \label{E456}
\vcenter{\xymatrix@C=46pt@R=36pt{
\mbZ_{[1, p^\N]}^{\times 3} \ar[r]^-{\xi_\N} \ar[d]_{\tau'_{\N \Rightarrow \N'}} & ((\mbZ/p^\N \mbZ)/\{\pm 1 \})^{\times 3} \ar[d]^-{\mr{quotient}} \\
\mbZ_{[1, p^{\N'}]}^{\times 3} \ar[r]_-{\xi_{\N'}}  & ((\mbZ/p^{\N'} \mbZ)/\{\pm 1 \})^{\times 3},
}}
\hspace{3mm} 
\vcenter{\xymatrix@C=46pt@R=36pt{
\mbZ_{[0, p^\N-1]}^{\times 3} \ar[r]^-{\zeta_\N} \ar[d]_-{\tau_{\N \Rightarrow \N'}} & ((\mbZ/p^\N \mbZ)/\{\pm 1 \})^{\times 3} \ar[d]^-{\mr{quotient}} \\
\mbZ_{[0, p^{\N'}-1]}^{\times 3} \ar[r]_-{\zeta_{\N'}} &  ((\mbZ/p^{\N'} \mbZ)/\{\pm 1 \})^{\times 3},
}}
\end{align}
where the right-hand vertical arrows in both diagrams  arise  from  the natural quotient $\mbZ/p^\N \mbZ \migisurj \mbZ/p^{\N'}\mbZ$.

We shall write
\begin{align} \label{e46}
B_\N := \left\{ (a, b, c) \in \mbZ_{[1, p^\N]}^{\times 3} \, \Big| \, \text{$a < c \leq b$ or  $b < c \leq a$}\right\}.
\end{align}
The map  $\xi_\N$ restricts to  an $8$-to-$1$ correspondence between  $B_\N$ and $\xi_{\N}(B_\N)$.
To be precise,
two triples $(a, b, c)$, $(a', b', c')$ of elements of $B_\N$  have the same image via $\xi_\N$ 
if and only if $(a', b', c')$ is one of the following $8$ triples:
\begin{align} 
&(a, b, c), \hspace{10mm}  (b, a, c),  \hspace{10mm}    ([ c-b ]'_\N, [c-a]'_\N, c), \hspace{10mm}    ([c-a ]'_\N, [c-b ]'_\N, c), \notag  \\
&(1 +p^\N -a, 1 + p^\N -b, 2 + p^\N -c), \hspace{5mm}  (1 + p^\N -b, 1 + p^\N - a, 2+p^\N -c),  \notag \\ 
 &([1+a -c ]'_\N, [ 1 + b-c ]'_\N, 2 + p^\N-c), \hspace{5mm} 
   ([ 1 + b -c ]'_\N, [ 1 + a - c ]'_\N, 2+ p^\N -c). \notag
\end{align}

On the other hand, we set
\begin{align}\label{eekode}
C_\N := \left\{ (s_1, s_2, s_3) \in \mbZ_{[0, p^\N -1]}^{\times 3} \, \Big| \, \text{$\sum_{i=1}^3 s_i \leq p^N -2$ and $|s_2 - s_3|  \leq s_1 \leq s_2 + s_3$}\right\}.
\end{align}
The restriction $\zeta_\N |_{C_\N} : C_\N \migi ((\mbZ/p^\N \mbZ)/\{\pm 1 \})^{\times 3}$ of $\zeta_\N$   is injective.
Then, the following assertion can be proved by a straightforward calculation.

\SSP
\ble \label{P06}
The equality $\xi_{\N}(B_\N) = \zeta_N (C_N)$ holds as  subsets of $((\mbZ/p^\N \mbZ)/\{\pm 1 \})^{\times 3}$.
In particular,  we obtain the map of sets
\begin{align} \label{e210}
B_\N \migi C_\N 
\end{align}
determined  by assigning $(a, b, c) \mapsto \zeta_\N^{-1}(\xi_\N (a, b, c))$, and this map gives  an $8$-to-$1$ correspondence
 between  $B_\N$ and $C_\N$.
\ele
\begin{proof}
We first prove the inclusion $\xi_\N (B_\N) \subseteq \zeta_\N (C_\N)$.
Let $(a, b, c)$ be an element of $B_\N$.
By the definition of $B_\N$, we have 
\begin{align} \label{dYY1}
2 \leq c \leq p^\N, \hspace{3mm}
-p^\N < c-a -b  < 0,
\hspace{3mm}
0 < |b-a| < p^\N.
\end{align}
In particular, the elements of $\mbZ/p^\N \mbZ$ represented by $1-c$, $c-a-b$, and $b-a$ are all nonzero.
Also, after replacing $(a, b, c)$ by one of the eight triples (with the same image under $\xi_\N$) listed before this lemma,
we may assume that
\begin{align}
u_1 := [1-c]'_\N, \hspace{3mm}
u_2 := [c-a-b]'_\N, \hspace{3mm}
u_3:= [b-a]'_\N
\end{align} 
are all odd.

The inequalities \eqref{dYY1} imply
$u_1 = p^\N + 1-c$ and $u_2 = p^\N + c - a -b$.
Since $u_1$ and $u_2$ are odd, 
both $c$ and $a+b$ turn out to be odd.
Hence, $b - a$ is odd.
If $a > b$, then we obtain the equality $u_3 = p^\N + b -a$, which contradicts the fact that $u_3$ is odd.
It follows that $a < b$, and the condition $(a, b, c) \in B_\N$ gives 
$1 \leq a < c \leq b \leq p^\N$.

For $i=1, 2, 3$, we set $s_i := \frac{u_i -1}{2}$.
These are nonnegative integers, and we have 
\begin{align}\label{YddY2}
p^\N +1 -c = 2 s_1 +1,
\hspace{3mm}
p^\N + c-a- b = 2s_2 +1,
\hspace{3mm}
b-a = 2s_3 +1.
\end{align}
In particular, we have $\xi_\N (a, b, c) = \zeta_\N (s_1, s_2, s_3)$.
The equalities in \eqref{YddY2} yield
\begin{align}
s_1 + s_2 + s_3 &= p^\N -a -1, \\
s_2 + s_3 - s_1 &= c-a- 1, \\
s_3 + s_1 -s_2 &= b-c, \\
s_1 + s_2 -s_3 &= p^\N -b.
\end{align}
Thus, we have 
\begin{align}
s_1 + s_2 + s_3 \leq p^\N -2  \ &\Longleftrightarrow \  1 \leq a, \\
s_1 \leq s_2 + s_3 \ &\Longleftrightarrow \ a < c, \\
s_2 \leq s_3 + s_1  \ & \Longleftrightarrow \ c \leq b, \\
s_3 \leq s_1 + s_2 \ &\Longleftrightarrow \ b \leq p^\N.
\end{align}
All these inequalities hold.
The  inequality $s_1 + s_2 + s_3 \leq p^\N -2$ implies that 
 each $s_i$ lies in  $\mbZ_{[0, p^\N -1]}$, and that $(s_1, s_2, s_3) \in C_\N$.
 This proves the desired inclusion.
 
 Conversely, 
 let $(s_1, s_2, s_3)$ be an element of $C_\N$, and define
 \begin{align}
 a &:= p^\N -s_1 -s_2 -s_3 -1, \\
 b &:= p^\N -s_1 -s_2 + s_3, \\
 c &:= p^\N -2s_1.
 \end{align}
 These integers satisfy  \eqref{YddY2}.
 By the defining inequalities for $C_\N$,  the following inequalities hold:
 \begin{align}
 a &\geq 1,
\\
c-a &= s_2 + s_3 -s_1 + 1 \geq 1, \\
b-c & = s_1 + s_3 -s_2 \geq 0, \\
p^\N -b &=s_1 + s_2 - s_3 \geq 0. 
 \end{align}
Hence, we have $1 \leq a < c \leq b \leq p^\N$, so $(a, b, c)$ belongs to $B_\N$.
This means $\xi_\N (a, b, c) = \zeta_\N (s_1, s_2, s_3)$, which proves the reverse inclusion.
Consequently, the equality $\xi_\N (B_\N) = \zeta_\N  (C_\N)$ holds.

Finally, as observed before this lemma, the restriction $\zeta_\N |_{C_\N}$ is injective, whereas $\xi_\N |_{B_\N}$ is $8$-to-$1$ onto its image.
Thus, the map  \eqref{e210} is well-defined and $8$-to-$1$.
This completes the proof of this lemma.
\end{proof}
\SSP

We shall set
\begin{align} \label{eeQQ112}
{^\dagger}B_\N := \bigcap_{\N' \leq  \N} ({\tau'_{\N \Rightarrow \N'}})^{-1}(B_{\N'}), \hspace{5mm}
{^\dagger}C_\N  
  := C_\N \cap \bigcap_{\N' < \N} \tau_{\N \Rightarrow \N'}^{-1} (\zeta^{-1}_{\N'} (\zeta_{\N'} (C_{\N'}))). 
\end{align}
In other words, an element $(a, b, c)$ of $B_\N$ (resp., an element $(s_1, s_2, s_3)$ of $C_\N$) belongs to ${^\dagger}B_\N$  (resp., ${^\dagger}C_\N$)
if and only if, for every positive   $\N'$ with $\N' < \N$,
either $[a]'_{\N'} < [c]'_{\N'} \leq [b]'_{\N'}$ or $[b]'_{\N'} < [c]'_{\N'} \leq [a]'_{\N'}$ is fulfilled (resp., 
$(s'_1, s'_2, s'_3)$ belongs to $C_{\N'}$ for some  $s'_i \in \{ [s_i]_{\N'}, p^{\N'}-1-[s_i]_{\N'} \}$).
Note that the subset ${^\dagger}C_\N$ of $\mbZ^{\times 3}$ is invariant under permutations of factors. 

By  Lemma \ref{P06} (applied to various $\N' \leq \N$) together with the definitions of ${^\dagger}B_\N$ and ${^\dagger}C_\N$,
 we obtain the following assertion.

\SSP
\bpr \label{C03}
The map \eqref{e210} restricts to a surjective map 
\begin{align} \label{e20efr9}
{^\dagger}B_\N \migisurj  {^\dagger}C_\N,
\end{align}
and it gives an $8$-to-$1$ correspondence between 
${^\dagger}B_\N$ and ${^\dagger}C_\N$.
\epr

\LSP
\subsection{Hypergeometric operators with a full set of root functions} \label{SS17}
In this section, we examine the relationship between the elements of ${^\dagger}B_\N$ and the exponent differences of hypergeometric operators with a full set of root functions.

Let $\widehat{U} := \mr{Spec} (W_\N [\![x ]\!] \left[ x^{-1}\right])$ be the  formal punctured disc in $\mbP$ centered at the origin    $[0]$.
If $\nabla^{(0)}_{\widehat{U}, \mr{triv}}$ denotes the trivial connection on $\mcO_{\widehat{U}}$, then we have $\mcS ol (\nabla^{(0)}_{\widehat{U}, \mr{triv}}) =   \sum_{s = 1}^\N p^{\N -s} \cdot W_\N [\![x^{p^s} ]\!] \left[ x^{-p^s} \right] =:R_\N$.

For each $(a, b, c) \in \mbZ_{[1, p^\N]}^{\times 3}$,
denote by $\widehat{\mpD}_{a, b, c}^\clubsuit$
be the differential operator on $\widehat{U}$ defined as 
\begin{align} \label{Eww3q}
\widehat{\mpD}_{a, b, c}^\clubsuit := x (1-x) \partial_x^2 + (c - (a+b+1)x) \partial_x - ab,
\end{align}
i.e., the restriction of $x (1-x) \mpD_{a, b, c}^\clubsuit$ to  $\widehat{U}$.
The kernel  $\mr{Ker}(\widehat{\mpD}_{a, b, c}^\clubsuit)$ of $\widehat{\mpD}_{a, b, c}^\clubsuit$ has a structure of $R_\N$-module.
 If $v := \sum_{j \geq n} q_j x^j$  is an element of 
 $\mr{Ker}(\widehat{\mpD}_{a, b, c}^\clubsuit)$, coefficient comparison gives
 the following recurrence relation:
 \begin{itemize}
\item
$n (n-1 +c) q_n=0$;
\item
 $(1+\ell)(c+\ell) \cdot q_{\ell+1} = (a+\ell)(b+\ell) \cdot q_{\ell}$ for every $\ell \geq n$.
\end{itemize}

Suppose that $(a, b, c)$ belongs to ${^\dagger}B_\N$.
By symmetry in $a$ and $b$, we  assume that $1 \leq a < c \leq b \leq p^\N$.
Let $P (x)$ and $Q (x)$ be the  truncated hypergeometric series in the polynomial ring  $\mbQ [x]$ defined as follows:
\begin{align} \label{EE1005}
P (x) := \sum_{j=0}^{p^\N -b} \frac{(a)_j\cdot (b-p^{\N})_j}{(c)_j \cdot j!} \cdot x^j, \hspace{3mm} Q (x) := \sum_{j=0}^{c-a-1} \frac{(a-c+1)_j \cdot (b-c+1)_j}{(2-c)_j \cdot j!} \cdot x^j,
\end{align}
where $(u)_j$ (for $u \in \mbZ$ and $j \in \mbZ_{\geq 0}$) denotes the Pochhammer symbol, i.e.,  $(u)_j := u (u+1) \cdots (u+j-1)$ and $(u)_0 := 1$.
Note that the denominators in both expressions are nonzero integers in the specified range.
In particular, in the case of $Q (x)$, we have $j \leq c-a-1 \leq c-2$, so the sum terminates before $(2-c)_j$ vanishes.

Here, let us observe the following lemma.

\SSP
\ble \label{Lem567}
Both the polynomials $P (x)$, $Q (x)$ belong to $\mbZ_{(p)} [x]$.
\ele
\begin{proof}
In general, the number of multiples of $p^\ell$ (for each $\ell \geq 1$) occurring in a product of nonzero  integers $(A)_j$  is $\left\lfloor \frac{j + [A]'_\ell -1}{p^\ell}\right\rfloor$.
Therefore, we have
\begin{align} \label{EE914}
& \ \ \ \ \mr{ord}_p \left(\frac{(A)_j \cdot (B)_j}{(C)_j \cdot j!} \right)\\
& = \sum_{\ell=1}^{\N-1} \left( \left\lfloor \frac{j + [A]'_\ell -1}{p^\ell}\right\rfloor
+ \left\lfloor \frac{j + [B]'_\ell -1}{p^\ell}\right\rfloor
- \left\lfloor \frac{j + [C]'_\ell -1}{p^\ell}\right\rfloor
- \left\lfloor \frac{j}{p^\ell}\right\rfloor
\right).  \notag
\end{align}

In the specified ranges of the sums in the definitions of $P (x)$ and  $Q (x)$, the absolute values of all the nonzero integers occurring in the product $\frac{(A)_j \cdot (B)_j}{(C)_j \cdot j!}$ are less than $p^\N$, so there are no contributions from higher powers of $p$.
If $[A]'_\ell < [C]'_\ell \leq [B]'_\ell$, then the inequalities  $\left\lfloor \frac{j + [B]'_\ell -1}{p^\ell}\right\rfloor \geq \left\lfloor \frac{j + [C]'_\ell -1}{p^\ell}\right\rfloor$ and $\left\lfloor \frac{j + [A]'_\ell -1}{p^\ell}\right\rfloor \geq \left\lfloor \frac{j}{p^\ell}\right\rfloor$ hold.
Also, in the case $[B]'_\ell < [C]'_\ell \leq [A]'_\ell$, we obtain the corresponding inequalities.
It follows that both sides of \eqref{EE914} are nonnegative.

Applying this result to $(A, B, C) = (a, b -p^\N, c)$, we have $P (x) \in \mbZ_{(p)}[x]$.
Moreover,   since the representative in $\mbZ^{\times 3}_{[1, p^\N]}$ of $(\overline{a-c+1}, \overline{b-c +1}, \overline{2-c}) \in (\mbZ/p^\N \mbZ)^{\times 3}$ lies in ${^\dagger}B_\N$,
the same result shows that $Q (x) \in \mbZ_{(p)} [x]$.
This completes the proof of this lemma.
\end{proof}
\SSP

\begin{rema}[$p$-integrality of $P (x)$] \label{Rem550r}
The polynomial $P (x)$ can be defined  without assuming  that $(a, b, c) \in {^\dagger}B_\N$.
More generally, suppose that $\N \geq 2$, and that $(a, b, c)$ satisfies $1 \leq c \leq a \leq b\leq p^\N$ and 
$([a]'_{\N -1}, [b]'_{\N -1}, [c]'_{\N -1}) \in {^\dagger}B_{\N -1}$.
Then,   the proof of Lemma \ref{Lem567} applies without modification and shows  that $P (x) \in \mbZ_{(p)}[x]$.
\end{rema}
\SSP

Thus, by reducing the coefficients of $P (x)$ and $Q (x)$ via the natural map $\mbZ_{(p)} \rightarrow W_\N$, we obtain well-defined polynomials $\overline{P} (x)$ and $\overline{Q} (x)$, respectively.

\SSP
\bpr \label{Propt4} 
The elements 
\begin{align} \label{EE9100}
f_1 (x) = \overline{P} (x), \hspace{3mm} f_2 (x) = x^{1-c} \overline{Q} (x)
\end{align}
of $W_\mbN [\![x]\!] [x^{-1}]$ form a basis of the $R_\N$-module $\mr{Ker}(\widehat{\mpD}_{a, b, c}^\clubsuit)$.
\epr
\begin{proof}
For integers $A$, $B$, and $C$, we set 
\begin{align} \label{Euty4}
\mpD_{A, B, C}^\mbZ := x (1-x) \partial_x^2 + (C-(A + B + 1)x)\partial_x - AB
\end{align}
  as a differential  operator on $\mbQ (x)$.
That is, under the assumption that $(A, B, C) \in \mbZ_{[1, p^\N]}^{\times 3}$,  the operator $\mpD_{A, B, C}^\mbZ$ (restricted to $\mbZ_{(p)}[x]$) is a lifting of $\widehat{\mpD}_{A, B, C}^\clubsuit$ along 
 $\mbZ_{(p)} \rightarrow W_\N$.
It follows from the recurrence relation for the coefficients that $\mpD^\mbZ_{a, b-p^\N, c} (P(x)) =0$ in $\mbQ (x)$.
Since the equality $\mpD^\mbZ_{a, b, c} - \mpD^\mbZ_{a, b-p^\N, c} = -p^\N (x \partial_x + a)$ holds,
 we obtain 
 \begin{align}
 \widehat{\mpD}_{a, b, c}^\clubsuit (f_1 (x)) = \widehat{\mpD}_{a, b, c}^\clubsuit (\overline{P} (x)) = \widehat{\mpD}_{a, b-p^\N, c}^\clubsuit (\overline{P} (x)) =0
 \end{align}
  via reduction of  the coefficients via $\mbZ_{(p)} \rightarrow W_\N$.
On the other hand, by the recurrence relation for the coefficients, the equality $\mpD^\mbZ_{a-c+1, b-c+1, 2-c} (Q (x)) =0$ holds. 
Since $\mpD_{a, b, c}^\mbZ (x^{1-c} Q (x)) = x^{1-c} \mpD_{a-c+1, b-c+1, 2-c}^\mbZ (Q (x))$,
we have 
\begin{align}
\widehat{\mpD}_{a, b, c}^{\clubsuit} (f_2 (x)) = \widehat{\mpD}_{a, b, c}^{\clubsuit} (x^{1-c} \overline{Q} (x)) = x^{1-c}\widehat{\mpD}_{a-c+1, b-c+1, 2-c}^{\clubsuit} (\overline{Q} (x)) =0.
\end{align}

Next,  
the Wronskian is given by
\begin{align}
W (f_1 (x), f_2 (x)) &:= f_1 (x) \cdot f'_2 (x) - f'_1 (x) \cdot f_2 (x) \\
& \  = x^{-c} \left((1-c) \overline{P} (x) \overline{Q} (x) + x (\overline{P}(x) \overline{Q}' (x) -\overline{P}' (x)  \overline{Q} (x)) \right). \notag
\end{align}
The constant term in the parenthesis is $1 -c$, which is  invertible in $W_\N^\times$ because of  the inequality  
$[c]_1' \geq 2$.
Hence, $W (f_1 (x), f_2 (x))$ is invertible in $W_\N [\![x]\!] [x^{-1}]$.
When we express any element of $\mr{Ker}(\widehat{\mpD}_{a, b, c}^\clubsuit)$ as a linear combination of $f_1 (x)$, $f_2 (x)$,
its coefficients turn out to be elements of $R_\N$.
Thus, we have $\mr{Ker}(\widehat{\mpD}_{a, b, c}^\clubsuit) = R_\N \cdot f_1 (x) \oplus R_\N \cdot f_2 (x)$, and this proves the claim.
\end{proof}
\SSP


Conversely, we can prove the following assertion.

\SSP
\bpr \label{Pro78ju}
Suppose $\mr{Ker}(\widehat{\mpD}_{a, b, c}^\clubsuit)$ is a free $R_\N$-module of rank $2$.
  Then,  the triple of integers $(a, b, c)$ belongs to ${^\dagger}B_\N$.
 \epr
\begin{proof}
We argue  by induction on $\N$.
The base step, i.e., $\N =1$, follows from ~\cite[Section 1.6]{Ihara1} or ~\cite[Section 6.4]{Katz2}.

To discuss the induction step, we suppose that $\N \geq 2$, and that the assertion holds at all smaller levels.
Also, suppose that $\mr{Ker}(\widehat{\mpD}_{a, b, c}^\clubsuit)$ is a free $R_\N$-module of rank $2$.
Choose
 an $R_\N$-basis  $u_1 (x), u_2 (x)$ of  $\mr{Ker}(\widehat{\mpD}_{a, b, c}^\clubsuit)$.
For a positive integer $\N'$ with $\N' < \N$,  multiplication by $p^{\N -\N'}$, applied to arbitrary liftings from level $\N'$ to level $\N$, induces  canonical identifications
 $R_\N [p^{\N'}] \cong R_{\N'}$ and $\mr{Ker}(\widehat{\mpD}_{a, b, c}^\clubsuit) [p^{\N'}]  \cong \mr{Ker}(\widehat{\mpD}_{[a]'_{\N'}, [b]'_{\N'}, [c]'_{\N'}}^\clubsuit)$, where $[p^{\N'}]$ denotes the kernel of multiplication by $p^{\N'}$.
 Under these identifications,
 the decomposition  $\mr{Ker}(\widehat{\mpD}_{a, b, c}^\clubsuit) [p^{\N'}] = R_\N[p^{\N'}] \cdot u_1 (x) \oplus R_\N [p^{\N'}] \cdot u_2 (x)$ shows that the mod $p^{\N'}$ reductions  of $u_1 (x)$ and  $u_2 (x)$ form an $R_{\N'}$-basis of $\mr{Ker}(\widehat{\mpD}_{[a]'_{\N'}, [b]'_{\N'}, [c]'_{\N'}}^\clubsuit)$. 
The induction hypothesis therefore  implies that  $([a]'_{\N'}, [b]'_{\N'}, [c]'_{\N'})\in {^\dagger}B_{\N'}$.
In particular, all the conditions defining  ${^\dagger}B_\N$, except possibly the condition at level $\N$, are satisfied.

We also record that the Wronskian of any $R_\N$-basis of $\mr{Ker}(\widehat{\mpD}_{a, b, c}^\clubsuit)$ is a unit in $W_\N [\![x ]\!] [ x^{-1}]$. 
Indeed, let $\overline{u}_1$, $\overline{u}_2$ denote the reductions of $u_1$, $u_2$ modulo $p$.
They form an $R_1$-basis of $\mr{Ker}(\widehat{\mpD}_{[a]'_1, [b]'_1, [c]'_1}^\clubsuit)$.
If their Wronskian were zero, then $(\frac{\overline{u}_2}{\overline{u}_1})' =0$ in the field $k (\!(x)\!)$, so that $\overline{u}_2/\overline{u}_1 \in R_1$, contrary to their $R_1$-linear independence.
Thus, $W (u_1 (x), u_2 (x))$ has nonzero reduction in $k (\!(x)\!)$ and is a unit in  $W_\N [\![x ]\!] [ x^{-1}]$.
Since a change of $R_\N$-basis multiplies the Wronskian by the determinant of the change-of-basis matrix, the same conclusion holds for every $R_\N$-basis.

It remains to prove that $(a, b, c) \in B_\N$.
Suppose, on the contrary, that $(a, b, c) \notin B_\N$.
Then, either $c \leq \mr{min} \{ a, b\}$ or $\mr{max} \{a, b\} <c$ is fulfilled.
In the latter case,  the identity 
\begin{align} \label{EEg3}
\widehat{\mpD}_{a, b, c}^\clubsuit (x^{1-c} z) = x^{1-c} \cdot \widehat{\mpD}_{a-c+1, b-c+1, 2-c}^\clubsuit (z)
\end{align}
shows that
multiplication by $x^{1-c}$ gives an $R_\N$-linear isomorphism $\mr{Ker}(\widehat{\mpD}_{a-c+1, b-c+1, 2-c}^\clubsuit) \xrightarrow{\sim} \mr{Ker}(\widehat{\mpD}_{a, b, c}^\clubsuit)$.
The representatives in $\mbZ_{[1, p^\N]}$ of the transformed parameters $(a-c+1, b-c+1, 2-c)$ modulo $p^\N$ are
 \begin{align}
 (\widetilde{a}, \widetilde{b}, \widetilde{c}) := (p^\N +1 + a-c, p^\N +1 + b-c, p^\N + 2- c),
 \end{align}
and they satisfy $\widetilde{c} \leq \widetilde{a}, \widetilde{b}$.

Moreover, the conditions at all lower levels remain satisfied after this transformation.
To see this, fix a positive integer $\N' < \N$.
 Reducing the identity \eqref{EEg3} modulo $p^{\N'}$, we obtain an $R_{\N'}$-linear isomorphism
 $\mr{Ker}(\widehat{\mpD}_{[\widetilde{a}]'_{\N'}, [\widetilde{b}]'_{\N'}, [\widetilde{c}]'_{\N'}}^\clubsuit) \xrightarrow{\sim} \mr{Ker}(\widehat{\mpD}_{[a]'_{\N'}, [b]'_{\N'}, [c]'_{\N'}}^\clubsuit)$ given by $z \mapsto x^{1-c}z$.
As shown at the beginning of the induction step,
 the module on the right is free of rank $2$ over $R_{\N'}$.
  Hence, the module on the left is also free of rank $2$ over $R_{\N'}$.
   Applying the induction hypothesis at level $\N'$ to the transformed equation therefore gives
   $([\widetilde{a}]'_{\N'}, [\widetilde{b}]'_{\N'}, [\widetilde{c}]'_{\N'}) \in {^\dagger}B_{\N'}$.
In particular, this triple belongs to $B_{\N'}$.
 Since $\N'$ was arbitrary, the transformed parameters satisfy all the required conditions at levels strictly smaller than $\N$.
Thus, after replacing $(a, b, c)$ by $(\widetilde{a}, \widetilde{b}, \widetilde{c})$, if necessary, and interchanging $a$ and $b$, we may assume that
\begin{align}
1 \leq c \leq a \leq b \leq p^\N,
\end{align}
This replacement preserves both the assumption that the solution module is free of rank $2$ over $R_\N$ and all the required conditions at levels $\N' < \N$.


Let $P (x)$ be the element of  $\mbQ [x]$ defined as  \eqref{EE1005}, which belongs to  $\mbZ_{(p)}[x]$ (cf. Lemma \ref{Lem567} and Remark \ref{Rem550r}).
By the recurrence relation for the coefficients, 
the polynomial  $f (x) = \overline{P} (x)$ in $W_\N [x]$ belongs to $\mr{Ker} (\widehat{\mpD}_{a, b, c}^\clubsuit)$ (cf. Proposition \ref{Propt4}).
Since $P (0) =1$, we have $f (x) \in W_\N [\![x]\!][x^{-1}]^\times$.
Write $f (x) = r_1 \cdot  u_1 (x) + r_2 \cdot u_2 (x)$ with $r_1$, $r_2 \in R_\N$.
At least one of $r_1$ and  $r_2$ is invertible in $W_\N [\![x]\!][x^{-1}]^\times$, since  $W_\N [\![x]\!][x^{-1}]$ is local and $f (x)$ is invertible.
An element of $R_\N$ invertible in $W_\N [\![x]\!][x^{-1}]$ is also invertible in $R_\N$, because the derivative of its inverse is zero.
Thus, $f (x)$ extends  to an $R_\N$-basis $f (x)$, $g (x)$ of $\mr{Ker} (\widehat{\mpD}_{a, b, c}^\clubsuit)$.
By the observation above, $W (f(x), g (x))$ is a unit in  $W_\N [\![x]\!][x^{-1}]$.

Let us set $w (x):= x^{-c} (1-x)^{c-a-b+p^\N -1} \in W_\N [\![x]\!][x^{-1}]^\times$.
Abel's identity shows that $W (f(x), g(x))/w (x)$ belongs to $R_\N^\times$. 
After multiplying $g (x$) by a suitable element of $R_\N^\times$, 
we may therefore assume that $W (f(x), g(x)) = w (x)$. 
It follows that
$(g(x)/f(x))' = w(x)/f(x)^2$.
Since the derivative of a Laurent series has zero residue, we obtain
\begin{align} \label{EF54}
\mr{Res}_{x=0} \left(  \frac{w (x)}{f(x)^2}dx\right) = 0.
\end{align}

We now compute this residue before reduction modulo $p^N$. 
Observe that
\begin{align} \label{GGG45}
\lambda := \mr{Res}_{x=0} \left(\frac{x^{-c}(1-x)^{c-a-b + p^\N-1}}{P(x)^2} dx\right) &=  (-1)^{c-1}\cdot \binom{a-1}{c-1}  \cdot \binom{b-p^\N-1}{c-1} \\
&= \binom{a-1}{c-1} \cdot \binom{p^\N - b + c-1}{c-1},\notag
\end{align}
where the first equality follows from Lemma \ref{Lemtie9} for $(A, B, C) = (a, b-p^\N, c)$ and $F (x) = P (x)$.
Both upper entries of the binomial coefficients in the last expression  lie between $c-1$ and $p^\N -1$.
The summands in 
Legendre's formula can be written as 
\begin{align}
\left\lfloor \frac{a-1}{p^{\N'}} \right\rfloor -\left\lfloor \frac{c-1}{p^{\N'}} \right\rfloor - \left\lfloor \frac{a-c}{p^{\N'}} \right\rfloor &=1_{[a]'_{\N'} <[c]'_{\N'}}, \\
\left\lfloor \frac{p^\N - b+c -1}{p^{\N'}} \right\rfloor -\left\lfloor \frac{c-1}{p^{\N'}} \right\rfloor - \left\lfloor \frac{p^{\N}-b}{p^{\N'}} \right\rfloor &=1_{[b]'_{\N'}< [c]'_{\N'}}, 
\end{align}
where $1_{\mcP}$ denotes the indicator function, i.e., its value is $1$ if the condition $\mcP$ holds  and $0$ otherwise.
Therefore, we have
\begin{align}
\mr{ord}_p (\lambda) &= \mr{ord}_p \left( \binom{a-1}{c-1}\right) + \mr{ord}_p \left(\binom{p^\N - b + c-1}{c-1} \right) \\
& = \sum_{\N'=1}^{\N -1} 1_{[a]'_{\N'} <[c]'_{\N'}} + \sum_{\N'=1}^{\N -1} 1_{[b]'_{\N'}< [c]'_{\N'}}.
\end{align}
Since exactly one of $[a]'_{\N'}$ and $[b]'_{\N'}$ is strictly less than $[c]'_{\N'}$, the  equality 
\begin{align} \label{GG779}
\mr{ord}_{p} (\lambda) = \N -1
\end{align}
holds.
Since $P (x) \in \mbZ_{(p)}[x]$ and $P (0) =1$, the Laurent series in \eqref{GGG45} has $p$-integral coefficients.
Its reduction therefore has residue equal to the image of $\lambda$ in $W_\N$, which is nonzero  by \eqref{GG779}.
This contradicts \eqref{EF54}. 

This contradiction shows  that the original triple $(a, b, c)$ belongs to $B_\N$.
Together with the conditions already established at the lower levels, this proves $(a, b, c) \in {^\dagger}B_\N$ and completes the induction.
\end{proof}
\SSP

The following lemma was used in the proof of the above proposition.

\SSP
\ble \label{Lemtie9}
Let $(A, B, C)$ be a triple of integers with $C > 0$ and $F (x)$ an element of  $\mbQ [\![x]\!]$ satisfying $F(0) =1$ and $\mpD_{A, B, C}^\mbZ (F (x)) =0$ (cf. \eqref{Euty4} for the definition of $\mpD_{A, B, C}^\mbZ$).
Then,  we have 
\begin{align} \label{Ett59}
\mr{Res}_{x=0} \left(\frac{x^{-C}(1-x)^{C-A-B-1}}{F(x)^2} \right) = (-1)^{C-1} \cdot \binom{A-1}{C-1} \cdot \binom{B-1}{C-1}.
\end{align}
\ele
\begin{proof}
When $C=1$, both sides of the asserted identity  are equal to $1$.
Suppose that $C\geq 2$. 
We first assume that $A$ and $B$ are generic complex parameters.

By the standard formula for a logarithmic solution of the hypergeometric equation (cf.  ~\cite[Eq.\,(15.10.8)]{Old} or ~\cite[Eq.\,(2.3.18)]{AAR}), 
there exists a second solution of the form
 $G (x) = F (x) \cdot \log x + S (x)$ for some $S (x) \in \mbC (\!(x)\!)$,
 where
 \begin{align}
 S (x) = L \cdot x^{1-C} + O (x^{2-C}), \hspace{5mm} L:= \frac{(-1)^{C} \cdot (C-1)! \cdot (C-2)!}{(1-A)_{C-1} \cdot  (1-B)_{C-1}}.
 \end{align}

By Abel's identity, the Wronskian $W (F (x), G(x)) \left(:= F (x) G' (x) -F' (x) G (x) \right)$ satisfies $W (F (x), G (x)) = K \cdot x^{-C} (1-x)^{C-A-B-1}$ for some constant $K \in \mbC$.
Since $F (0) =1$, comparison of the coefficients of $x^{-C}$ gives 
\begin{align}
K = - (C-1) L = \frac{(-1)^{C-1} \cdot ((C-1)!)^2}{(1-A)_{C-1} \cdot (1-B)_{C-1}}.
\end{align}
Therefore, we have
\begin{align} \label{Ett57}
\left(\frac{G (x)}{F(x)} \right)' =\frac{W (F (x), G (x))}{F(x)^2} = \frac{(-1)^{C-1} \cdot ((C-1)!)^2}{(1-A)_{C-1} \cdot (1-B)_{C-1}} \cdot \frac{x^{-C} (1-x)^{C-A-B-1}}{F(x)^2}.
\end{align}

On the other hand, note that $\frac{G (x)}{F(x)} = \log x + \frac{S (x)}{F (x)}$ and $S (x)/F(x) \in \mbC (\!(x )\!)$.
Since the derivative of a Laurent series has zero residue, we obtain 
\begin{align} \label{Ett56}
\mr{Res}_{x=0} \left(\left(\frac{G (x)}{F(x)} \right)' dx \right) &= \mr{Res}_{x=0} \left((\log x)' dx+ \left(\frac{S (x)}{F(x)}\right)' dx\right) \\
& = \mr{Res}_{x=0} \left(\frac{dx}{x}\right) \\
&= 1.
\end{align}
Combining these equalities yields
\begin{align}
\mr{Res}_{x=0} \left(\frac{x^{-C}(1-x)^{C-A-B-1}}{F(x)^2} dx\right)  &=
(-1)^{C-1} \cdot \frac{(1-A)_{C-1} \cdot (1-B)_{C-1}}{((C-1)!)^2} \\
&= (-1)^{C-1} \cdot \binom{A-1}{C-1} \cdot \binom{B-1}{C-1}. \notag
\end{align}
Finally, for fixed $C$, both sides of this identity are polynomials in $A$ and $B$.
Indeed, the equality $F (x) = \sum_{j \geq 0} \frac{(A)_j \cdot (B)_j}{(C)_j \cdot j!} \cdot x^j$ holds, and the residue on
 the left-hand side is the coefficient   of $x^{C-1}$ in $(1-x)^{C-A-B -1}F(x)^{-2}$.
 Since $F (0) =1$, this coefficient is  a polynomial in $A$ and $B$.
 Thus, the identity, established above for generic parameters, holds for all $A$, $B$. 
This completes the proof of this lemma.
\end{proof}


\SSP
\bpr \label{P05}
Let $(a, b, c)$ be an element of 
$(\mbZ/p^\N \mbZ)^{\times 3}$.
Then, 
the following conditions (1), (2) are equivalent to each other:
\begin{itemize}
\item[(1)]
The hypergeometric operator $\mpD_{a, b, c}^\clubsuit$
has a full set of root functions on some dense open subscheme $U$ of $\mbP \setminus \{[0], [1], [\infty] \}$; 
\item[(2)]
The element of $\mbZ_{[1, p^\N]}^{\times 3}$ corresponding to  $(a, b, c)$ via the natural bijection $\mbZ_{[1, p^\N]} \isom \mbZ/p^\N \mbZ$ belongs to ${^\dagger}B_\N$.
\end{itemize}
\epr
\begin{proof}
One may verify from Proposition \ref{P016dd}, (i) and (ii), that
 (1) is equivalent to the  condition that
 $\mr{Ker}(\widehat{\mpD}_{a, b, c}^\clubsuit)$ is a free  $R_\N$-module  of rank $2$.
Hence, the assertion follows from Propositions \ref{Propt4} and  \ref{Pro78ju}.
\end{proof}
\SSP

\LSP
\subsection{Explicit computations for $(g, r) = (0, 3)$} \label{SS02k92}

We shall write
\begin{align} \label{e51}
T_{\leq, \leq} &:= \left\{ (r_1, r_2, r_3) \in \mbZ_{[0, p-1]}^{\times 3} \, \Big| \,  r_1 \leq  r_2 \leq r_3\right\}, \\
T_{<, \leq} &:= \left\{ (r_1, r_2, r_3) \in \mbZ_{[0, p-1]}^{\times 3} \, \Big| \,  r_1 <  r_2 \leq r_3\right\}, \notag \\
T_{>, >} &:= \left\{ (r_1, r_2, r_3) \in \mbZ_{[0, p-1]}^{\times 3} \, \Big| \,  r_1 >  r_2 > r_3\right\}. \notag
\end{align}

Let us  take an $\N$-tuple  
$(\vec{r}_{\ell})_{\ell =1}^\N \in T_{<, \leq} \times (T_{\leq, \leq} \sqcup T_{>, >})^{\times (\N -1)}$, where $\vec{r}_1 := (r_{1, 1}, r_{1, 2}, r_{1, 3}) \in T_{<, \leq}$ and $\vec{r}_j := (r_{j, 1}, r_{j, 2}, r_{j, 3}) \in T_{\leq, \leq} \sqcup T_{>, >}$ ($j =2, \cdots, \N$).
This  $\N$-tuple defines inductively   triples 
$(q_{j, 1}, q_{j, 2}, q_{j, 3}) \in \mbZ_{> 0}^{\times 3}$ ($j =1, \cdots, \N$) starting with $(q_{1, 1}, q_{1, 2}, q_{1, 3}) :=  (1 + r_{1, 1}, 1 + r_{1, 3}, 1 + r_{1, 2})$, as follows:
\begin{itemize}
\item
If
the inequalities  $q_{j, 1} < q_{j, 3} \leq q_{j, 2}$ hold, then we set 
\begin{align} \label{e54}
(q_{j+1, 1}, q_{j+1, 2}, q_{j+1, 3}):= (q_{j, 1}+ p^{j} \cdot r_{j+1, 1}, q_{j, 2}+ p^{j} \cdot r_{j+1, 3}, q_{j, 3}+ p^j \cdot r_{j+1, 2}).
\end{align}
\item
If 
the inequalities   $q_{j, 2} < q_{j, 3} \leq q_{j,1}$ hold, then we set 
\begin{align} \label{e55}
(q_{j+1, 1}, q_{j+1, 2}, q_{j+1, 3}):= (q_{j, 1}+ p^{j} \cdot r_{j+1, 3}, q_{j, 2}+ p^{j} \cdot r_{j+1, 1}, q_{j, 3}+ p^j \cdot r_{j+1, 2}).
\end{align}
\end{itemize}
(Note that the resulting triple $(q_{j+1, 1}, q_{j+1, 2}, q_{j+1, 3})$ satisfies  
either 
 $q_{j+1, 1} < q_{j+1, 3} \leq q_{j+1, 2}$ or $q_{j+1, 2} < q_{j+1, 3} \leq q_{j+1, 1}$.)

Since $([q_{\N, 1}]'_{\N'}, [q_{\N, 2}]'_{\N'}, [q_{\N, 3}]'_{\N'}) = (q_{\N', 1}, q_{\N', 2}, q_{\N', 3})$ for $\N' \leq \N$, 
the assignment  $(\vec{r}_\ell)_\ell \mapsto (q_{\N, 1}, q_{\N, 2}, q_{\N, 3})$ (resp., $(\vec{r}_\ell)_\ell \mapsto (q_{\N, 2}, q_{\N, 1}, q_{\N, 3})$) defines a map of sets
\begin{align}
\delta_1 \ \left(\text{resp.,} \ \delta_2\right) : T_{<, \leq} \times (T_{\leq, \leq} \sqcup T_{>, >})^{\times (\N -1)}\migi {^\dagger}B_\N.
\end{align}

\SSP
\bpr \label{Prop53k}
The map of sets
\begin{align}
\delta_1 \sqcup \delta_2 :  (T_{<, \leq} \times (T_{\leq, \leq} \sqcup T_{>, >})^{\times (\N -1)})^{\sqcup 2}\migi {^\dagger}B_\N
\end{align}
induced from $\delta_1$ and $\delta_2$ is bijective. 
In particular, the cardinality of ${^\dagger}B_\N$ is explicitly given by the equality
\begin{align} \label{eQ331}
\sharp ({^\dagger}B_\N) = \frac{(p^2 -1)p^{\N} (p^2 +2)^{\N -1}}{3^{\N}}.
\end{align}
\epr
\begin{proof}
The first assertion follows from the various definitions involved.
Also, 
it is immediately verified that
\begin{align}
\sharp (T_{\leq, \leq}) = \frac{p (p+1)(p+2)}{6}, \  \ \ 
\sharp (T_{<, \leq}) = \frac{(p-1)p (p+1)}{6}, \ \ \ 
\sharp (T_{>, >}) = \frac{(p-2) (p-1)p}{6},
\end{align}
 This implies 
\begin{align}
\sharp ({^\dagger}B_\N) &= 2 \cdot \sharp (T_{<, \leq}) \cdot (\sharp (T_{\leq, \leq} \sqcup T_{>, >}))^{\N-1} \\
& =  2 \cdot  \frac{(p-1)p (p+1)}{6} \cdot \left(\frac{p (p+1)(p+2)}{6} + \frac{(p-2) (p-1)p}{6}\right)^{\N-1}\notag  \\
& = \frac{(p^2 -1) \cdot p^{\N} \cdot  (p^2 +2)^{\N -1}}{3^{\N}},\notag
\end{align}
thus completing the proof of the second assertion.
\end{proof}
\SSP

By combining the results proved so far, we obtain the following assertion.

\SSP
\bt \label{Pte2}
Let $\rho := (\rho_0, \rho_1, \rho_\infty)$ be an element of $((\mbZ/p^\N \mbZ)^{\times}/\{\pm 1\})^{\times 3}$.
Then, there exists a dormant $\mr{PGL}_2$-oper on $\msP$ of radii $\rho$ if and only if $\rho \in \zeta_\N ({^\dagger}C_\N)$.
In particular, the cardinality of  the set of isomorphism classes of dormant $\mr{PGL}_2$-opers on $\msP$ coincides with the value
\begin{align} \label{e47}
\frac{(p^2 -1)\cdot p^\N \cdot  (p^2 +2)^{\N-1}}{8 \cdot 3^\N}.
\end{align}
\et
\begin{proof}
The first assertion follows from  Propositions \ref{Prop21}, \ref{C03}, and  \ref{P05}.
The second assertion follows from the first one, Proposition \ref{Prop53k}, and the decomposition \eqref{e888}.
\end{proof}
\SSP

Next, we consider the relationship between dormant $\mr{PGL}_2$-opers on $\msP$ and dormant $\mr{PGL}_2^{(\N)}$-opers on the mod $p$ reduction $\msP_0$ of $\msP$.
To do this, 
we shall write 
\begin{align}
\mr{Cov}
\end{align}
for the set of equivalence classes of finite, separable, and tamely ramified coverings $\phi : \mbP_0 \rightarrow \mbP_0$ satisfying the following conditions:
\begin{itemize}
\item
The set of ramification points of $\phi$ is contained in  $\{[0], [1], [\infty] \}$; 
\item
If $\lambda_x$ ($x = 0, 1, \infty$) denotes the ramification index of $\phi$ at $[x]$,
then $\lambda_0, \lambda_1, \lambda_\infty$ are all odd and satisfy the inequality $\lambda_0 + \lambda_1 + \lambda_\infty < 2 \cdot p^\N$.
\end{itemize}
Here, the equivalence relation is defined in such a way that two coverings $\phi_1, \phi_2 : \mbP_0 \rightarrow \mbP_0$ are equivalent if there exists an element $h \in \mr{PGL}_2 (\mbP_0) \left(= \mr{Aut}_k (\mbP_0) \right)$ with $\phi_2 = h \circ \phi_1$.

Since the identity morphism $\mr{id}_{\mbP_0}$ of $\mbP_0$ defines a tamely ramified covering with ramification indices $(1, 1, 1)$, the set $\mr{Cov}$ is nonempty.

\SSP
\bpr \label{Prop332}
Let $\rho := (\rho_0, \rho_1, \rho_\infty)$ be an element of $((\mbZ/p^\N \mbZ)^{\times}/\{\pm 1\})^{\times 3}$, and $(\lambda_0, \lambda_1, \lambda_\infty)$ the triple of integers induced from  $\rho$ as discussed  at the beginning of  Section \ref{SS042e}.
Suppose that $\lambda_0$, $\lambda_1$, and $\lambda_\infty$ are all odd.
Then, the following three conditions (1)-(3) are equivalent to each other:
\begin{itemize}
\item[(1)]
There is a dormant $\mr{PGL}_2$-oper on $\msP$ of radii $\rho$;
\item[(2)]
There is a dormant $\mr{PGL}_2^{(\N)}$-oper on $\msP_0$ of radii $\rho$;
\item[(3)]
There is a covering $\phi_0 : \mbP_0 \rightarrow \mbP_0$ classified by $\mr{Cov}$  whose ramification index at $[x]$ ($x = 0, 1, \infty$) coincides with $\lambda_x$.
\end{itemize}
\epr
\begin{proof}
The equivalence (2) $\Leftrightarrow$ (3) follows from  ~\cite[Theorem 7.4.3]{Wak9}.

The implication (1) $\Rightarrow$ (2) follows from Propositions \ref{P72e}, (i), and  \ref{P00245}, which imply that the diagonal reduction of a dormant $\mr{PGL}_2$-oper on $\msP$ of radii $\rho$ specifies a dormant $\mr{PGL}_2^{(\N)}$-oper on $\msP_0$ of radii $\rho$.

Finally, we shall prove the implication (3) $\Rightarrow$ (1).
Let us take a covering $\phi_0 : \mbP_0 \rightarrow \mbP_0$ as required in (3).
After possibly composing it with an automorphism of $\mbP_0$, we may assume that $\phi_0 ([0]) = [0]$, $\phi_0 ([1]) = [1]$, $\phi_0 ([\infty]) = [\infty]$ (cf. ~\cite[Proposition 7.2.1]{Wak9}).
Put $D := [0] + [1] + [\infty]$ on each copy of $\mbP$, and let $D_0$ denote its special fiber.
Also, set $D'_0 := (\phi_0^{-1}(D_0))_\mr{red}$.
By the lifting theorem for tamely ramified coverings, the covering $\phi_0$ lifts to a tamely ramified  covering of $\mbP$ over $W_\N$, with branch locus contained in $D$.
Its domain is a smooth proper genus-$0$ curve with three distinguished sections lifting $[0]$, $[1]$, and $[\infty]$.
Identifying this $3$-pointed curve with the standard $3$-pointed projective line, we obtain a morphism $\phi : \mbP \rightarrow \mbP$ lifting $\phi_0$.

Let $D'$ be the sum of the disjoint sections above $D$; thus $D'$ lifts $D'_0$ and each section occurs with multiplicity $1$.
The morphism $\phi$ then induces a log \'{e}tale morphism $\phi^\mr{log} : \mbP^{\mr{log}'} \rightarrow \mbP^\mr{log}$, where $\mbP^{\mr{log}'}$ denotes the log scheme defined as $\mbP$ equipped with the log structure determined by $D'$.
Write $D' = D + E$.
The divisors $D$ and $E$ are disjoint, and $\phi$ is \'{e}tale on $\mbP \setminus D$.
In particular, every section occurring in $E$ has ramification index $1$.

Write $\mcL := \mcO_{\mbP} (-1) \otimes \mcO_\mbP ([0]+ [1]+ [\infty])$, and write $\tau_0$ for the $\mcO_\mbP$-linear injection $\mcO_\mbP (-1) \hookrightarrow \mcO_\mbP^{\oplus 2}$ given by $w \mapsto (wx, wy)$ (where $\mbP = \mr{Proj}(W_\N [x, y])$) for each local section $w \in \mcO_\mbP (-1)$.
Also, let $\mcF$ be a rank $2$ vector bundle on $\mbP$ which makes the following square diagram cocartesian:
\begin{align} \label{eQQe}
\vcenter{\xymatrix@C=46pt@R=36pt{
\mcO_\mbP (-1) \ar[r]^-{\tau_0} \ar[d]_-{\mr{inclusion}} & \mcO_\mbP^{\oplus 2} \ar[d]
\\
\mcL \ar[r] & \mcF.
}}
\end{align}
The trivial $S^\mr{log}$-connection on  $\mcO_\mbP^{\oplus 2}$ extends uniquely to an $S^\mr{log}$-connection $\nabla_\mcF$ on $\mcF$.
It follows from the various definitions involved that the composite
\begin{align}
\mcL \xrightarrow{\mr{inclusion}} \mcF \xrightarrow{\nabla_\mcF} \Omega \otimes \mcF \migisurj \Omega \otimes (\mcF/\mcL)
\end{align}
is $\mcO_\mbP$-linear and injective.
Moreover, since $\mr{deg}(\mcL) = \mr{deg}(\Omega \otimes (\mcF/\mcL)) \left(=2 \right)$,
this morphism turns out to be an isomorphism.
This means that the triple $(\mcF, \nabla_\mcF, \mcL)$ forms a dormant $\mr{GL}_2$-oper on
$\msP$.

Set $U. := \mbP \setminus E$ and $V := \mbP\setminus D$.
These open subschemes cover $\mbP$, and $U \cap V = \mbP \setminus D'$.
On $U$, consider the pull-back of $(\mcF, \nabla_\mcF, \mcL)$ along $\phi^\mr{log}$.
Since $D' |_U = D|_U$ and $\phi^\mr{log}$ is log \'{e}tale, this pull-back forms a $\mr{GL}_2$-oper on the log curve $\mbP^\mr{log} |_{U}$.
On $V$, consider instead  $(\mcO_V^{\oplus 2}, d^{\oplus 2}, (\phi^*(\mcO_\mbP(-1)))|_V)$,
where the line subbundle is embedded by the pull-back of $\tau_0$.
Its Kodaira-Spencer map is an isomorphism because $\phi$ is \'{e}tale on $V$.

By the construction of $\mcF$, its restriction to the complement of the divisor $D$ on the codomain is canonically identified with $\mcO_\mbP^{\oplus 2}$ with its trivial connection, and the restriction of $\mcL$ is identified with $\mcO_\mbP (-1)$.
The two opers therefore agree on $U \cap V$ and glue to a $\mr{GL}_2$-oper $(\mcF_\phi, \nabla_\phi, \mcL_\phi)$ on $\msP$.
Its line subbundle is $\mcL_\phi = \phi^* (\mcL) (-E) \left(= \phi^* (\mcL) \otimes \mcO_\mbP (-E) \right)$.
The associated $\mr{PGL}_2$-oper has trivial underlying flat bundle on $V$, and hence is dormant by Lemma \ref{Prop11}.

Finally, at each $[x]$ with $\{ 0, 1, \infty \}$, the local tame model of $\phi$ is described as the morphism  $\mr{Spec}(W_\N [\![t]\!]) \rightarrow \mr{Spec} (W_\N [\![s]\!])$ given by $s \mapsto t^{\lambda_x}$.
In a suitable basis, the connection on $\mcF_\phi$ is locally given by 
\begin{align}
\nabla_\phi = d +  \frac{dt}{t} \otimes \begin{pmatrix} - \lambda_x & 0 \\ 0 & 0  \end{pmatrix}. 
\end{align}
Thus, its exponent difference at $[x]$ is $\lambda_x$ up to sign, and the radius of its projectivization is the class of $\lambda_x/2$, namely $\rho_x$.
This proves the implication (3) $\Rightarrow$ (1).
\end{proof}
\SSP

By Theorem \ref{Pte2} and Proposition \ref{Prop332} (and ~\cite[Theorem 7.4.3]{Wak9}), we obtain the following assertion.

\SSP
\bt \label{T04rr}
Let $\rho := (\rho_0, \rho_1, \rho_\infty)$ be an element of $((\mbZ/p^\N \mbZ)^{\times}/\{\pm 1\})^{\times 3}$.
Then, we have 
\begin{align} \label{e10}
\mcO p_{2, \N, \rho, 0, 3, \mbF_p}^{^\mr{Zzz...}} \cong \begin{cases} \mr{Spec}(\mbF_p) & \text{if $\rho \in \zeta_N({^\dagger}C_\N$)}; \\  \emptyset & \text{if otherwise}. \end{cases}
\end{align}
In particular,  $\mcO p^{^\mr{Zzz...}}_{2, \N, 0, 3, \mbF_p}$ is isomorphic to the disjoint union of finitely many copies of $\mr{Spec}(\mbF_p)$,  and  its degree $\mr{deg} (\Pi_{2, \N, 0, 3, \mbF_p})$ over $\mbF_p$ satisfies  the following  equalities:
\begin{align} \label{esss47}
\mr{deg} (\Pi_{2, \N, 0, 3, \mbF_p}) = \sharp (\mr{Cov})
= \frac{(p^2 -1)\cdot p^\N \cdot  (p^2 +2)^{\N-1}}{8 \cdot 3^\N}.
\end{align}
\et
\SSP

The description \eqref{e10} enables us to obtain a detailed understanding of the dormant fusion ring $\Fus_{2, \N}$  of type $\mr{PGL}_2^{(\N)}$ (cf. Definition \ref{Def989}).
In fact, $\Fus_{2, \N}$ may be identified with the free abelian group $\mbZ^{C}$ with basis $C := (\mbZ/p^\N \mbZ)^{\times}/\{ \pm 1\}$ equipped with the multiplication $\ast$ given by 
\begin{align} \label{eeQQ59}
\alpha \ast \beta = \sum_{\substack{\gamma \in C, \\ (\alpha, \beta, \gamma) \in \zeta_\N({^\dagger}C_\N)}}  \gamma
\end{align}
for any $\alpha, \beta \in C$.
Also, by Proposition \ref{Prop16},  
the degree $\mr{deg}(\Pi_{2, \N, g, r, \mbF_p})$  can be computed as follows:
\begin{align} \label{eeQQ61}
\mr{deg}(\Pi_{2, \N, (\rho_i)_i,  g, r, \mbF_p}) = \sum_{\chi \in \mfS_{2, \N}} \chi (\sum_{\alpha \in C} \alpha \ast \alpha)^{g-1} \cdot \prod_{i=1}^r \chi (\rho_i),
\end{align}
where $\mfS_{2, \N} :=  \mr{Hom}(\Fus_{2, \N}, \mbC)$.
This formula in the case of $r =0$ together with Corollary \ref{C998} induces  the following assertion.

\SSP
\bco \label{Eq223}
Let $k$ be an algebraically closed field over $\mbF_p$, and denote by $W_\N$ the ring of Witt vectors of length $\N$ over $k$.
Also, let $X$ be a geometrically connected, proper, and smooth curve of genus $g > 1$ over $W_\N$.
If, moreover, the mod $p$ reduction $X_0$ of $X$ is general in $\mcM_{g, 0}$, then
 the cardinality $\sharp (\mr{Op}^{^\mr{Zzz...}}_{1, X})$ of the set $\mr{Op}^{^\mr{Zzz...}}_{1, X}$ (cf. \eqref{e209}) is given by the formula
\begin{align} \label{eeQQ60}
\sharp (\mr{Op}^{^\mr{Zzz...}}_{1, X}) =  \sum_{\chi \in \mfS_{2, \N}} \chi (\sum_{\alpha \in C} \alpha \ast \alpha)^{g-1}. 
\end{align}
\eco
\SSP

\begin{rema}[Relationship with other enumerative geometries] \label{Rem442}
In ~\cite[Section 7.8.2]{Wak8},
we have related a certain  algebra   (which is essentially the same as  $\Fus_{2, 1}$)
 encoding  the factorization rule of  the values  $\mr{deg}(\Pi_{2, 1, (\rho_i)_i, g, r, \mbF_p})$   to the fusion ring for the conformal field theory of the affine Lie algebra $\widehat{\mfs \mfl}_2$ in an explicit manner.
As mentioned in Introduction, this is crucial in establishing an analogue of  the  Verlinde formula computing the number of dormant $\mr{PGL}_2$-opers of level $1$  (cf. ~\cite[Theorem 7.41]{Wak8}).
However, at the time of writing this manuscript, we do not know any relationship between $\Fus_{2, \N}$ for $\N >1$ and other enumerative geometries such as the CFT with $\widehat{\mfs \mfl}_2$-symmetry.
\end{rema}
\SSP

Note that   \eqref{e47} also gives a computation of
the degree $\mr{deg}(\Pi_{2, \N, 2, 0, \mbF_p})$, i.e., the case of $(g, r) = (2, 0)$.
 Indeed,  let us consider  the unpointed stable curve $\msX$ of genus $2$ obtained by gluing together two copies of $\msP_0$ at the  respective corresponding marked points.
Since any  $\mr{PGL}_2^{(\N)}$-oper on $\msP_0$ is uniquely determined by its radii (cf. Proposition \ref{P00245}), 
 the restrictions  of a dormant  $\mr{PGL}_2^{(\N)}$-oper on $\msX$ to  the respective   components are  necessarily isomorphic (cf. Proposition \ref{YY188}, (i) and (ii),  for $n =2$).
 Conversely, each dormant $\mr{PGL}_2^{(\N)}$-oper on $\msX$ can be obtained by gluing together two isomorphic dormant $\mr{PGL}_2^{(\N)}$-opers on $\msP_0$ at the points of attachment.
It follows  that there exists a canonical
correspondence between dormant $\mr{PGL}_2^{(\N)}$-opers on $\msP_0$ and those on $\msX$.

Moreover,  by Theorem \ref{Theorem44} and the first assertion of Theorem \ref{T04rr}, 
the cardinalities  of these sets coincides with $\mr{deg}(\Pi_{2, \N, 2, 0, \mbF_p})$.
That is to say, we obtain the following assertion, generalizing 
 ~\cite[Chapter V, Corollary 3.7]{Mzk2}, ~\cite[Theorem 2]{LP}, and ~\cite[Theorem 1.2]{O4}.

\SSP
\bco \label{C19}
Let $X$ be as in Corollary \ref{Eq223} with $g =2$.
Then, the cardinality $\sharp (\mr{Op}^{^\mr{Zzz...}}_{1, X})$ and the generic degree 
$\mr{deg}(\Pi_{2, \N, 2, 0, \mbF_p})$ of $\Pi_{2, \N, 2, 0, \mbF_p}$
 are  explicitly computed  by the equalities 
\begin{align} \label{eeQQ70}
\sharp (\mr{Op}^{^\mr{Zzz...}}_{1, X}) =  \mr{deg} (\Pi_{2, \N, 2, 0, \mbF_p}) 
=   \frac{(p^2 -1)\cdot p^\N \cdot  (p^2 +2)^{\N-1}}{8 \cdot 3^\N}.
\end{align}
\eco

\LSP
\subsection{Edge numberings on trivalent semi-graphs} \label{SS10220}

The procedure of deriving the formulas \eqref{eeQQ61} and \eqref{eeQQ60} from the factorization  property defining our $2$d TQFT (or  fusion rule)
  can be translated into counting 
 the combinatorial  patterns of radii of dormant $\mr{PGL}_2^{(\N)}$-opers on a totally degenerate curve.
To observe this, we 
describe such patterns
  in terms of  certain edge numberings on trivalent graphs (cf. ~\cite{LO}, ~\cite{Mzk2}, and ~\cite{Wak2} for $\N =1$). 
 Following the terminology of ~\cite[Definition 3.1]{Wak31},
 we will refer to
  these numberings as  {\it balanced $(p, \N)$-edge numberings} (cf. Definition \ref{De113}, (i)).

Let us fix
trivalent clutching data $\mbG := (\GR, \{ \lambda_j \}_{j=1}^J)$ of type $(g, r)$ (cf. Definition \ref{Def112}, (iii)), where $\GR := (V, E, \zeta)$.
In particular, $\sharp (V) = J =  2g-2+ r$, $\sharp (E) = 3g-3+2r$, and one can write  $B_\circledast = \{ b_{\circledast, 1}, \cdots, b_{\circledast, r} \}$ (cf. Remark \ref{Rem7878}).

\SSP
\bde \label{De113}
\begin{itemize}
\item[(i)]
 A {\bf balanced $(p, \N)$-edge numbering} on $\mbG$ is 
 a collection 
 \begin{align} \label{eeQQ50}
 (a_{e})_{e \in E}
  \end{align}
  of elements of $\mbZ_{[0, p^\N -1]}$
   indexed by $E$  such that, for each vertex $v \in V$,
   the triple $(a_{e_b})_{b \in B_v}$ belongs to ${^\dagger}C_\N$, where $e_b$ denotes the unique edge satisfying $b \in e_b$.
  \item[(ii)]
 Let 
 $(a_e)_{e \in E}$ be a balanced $(p, \N)$-edge numbering on $\mbG$ and $\rho :=(\rho_i)_{i=1}^r$ an $r$-tuple of element of $(\mbZ /p^\N \mbZ)^{\times}/\{ \pm 1 \}$.
 We say that $(a_e)_e$ is {\bf of radii  $\rho$} if 
 the equality $\frac{\overline{2a_{e_i} +1}}{2} = \rho_i$ holds in $(\mbZ/p^\N \mbZ)/\{ \pm1 \}$ for every $i=1, \cdots, r$, where $e_i$ denotes the  unique open  edge of $\GR$ satisfying $e_i \ni b_{\circledast, i}$.
 For convenience, (regardless of whether $B_{\circledast}$ is empty or not) we shall refer to any balanced $(p, \N)$-edge numbering as being {\bf of radius $\emptyset$}. 
 \end{itemize}
 \ede
\SSP

Let $\rho$ be an element of 
  $((\mbZ/p^\N\mbZ)^{\times}/\{\pm 1\})^{\times r}$ (where  $\rho := \emptyset$ if $r =0$).
Denote by 
\begin{align} \label{eeQQ79}
\mr{Ed}_{p, \N, \mbG} \ \left(\text{resp.,}  \  \mr{Ed}_{p, \N, \mbG, \rho} \right)
\end{align}
the set of balanced $(p, \N)$-edge numberings   (resp., balanced $(p, \N)$-edge  numberings of radii  $\rho$) on $\mbG$.
The set 
 $\mr{Ed}_{p, \N, \mbG}$ decomposes into the disjoint union
 \begin{align} \label{YYYY}
\mr{Ed}_{p, \N, \mbG} = \coprod_{\rho \in ((\mbZ/p^\N\mbZ)^{\times}/\{\pm 1\})^{\times r}}  \mr{Ed}_{p, \N, \mbG, \rho}.
 \end{align}

Then, we can prove the following  proposition, which  
ensures
that the degree of $\Pi_{2, \N, g, r, \mbF_p}$ for each triple $(p, g, r)$ can be explicitly computed  by hand after  choosing  trivalent clutching data $\mbG$.

\SSP
\bpr \label{LLL001}
Let $\rho$ be an element of $((\mbZ /p^\N \mbZ)^\times /\{ \pm 1 \})^{\times r}$.
Denote by $\msX$ the totally degenerate curve over $k$ induced by  $\mbG$  (cf. Definition \ref{DD3WW}) and by $\mr{Op}^{^\mr{Zzz...}}_{\N, \msX}$ (resp., $\mr{Op}^{^\mr{Zzz...}}_{\N, \rho, \msX}$) the set of isomorphism classes of dormant $\mr{PGL}_2^{(\N)}$-opers (resp., dormant $\mr{PGL}_{2}^{(\N)}$-opers of radii $\rho$) on $\msX$.
Then, there exists a canonical bijection
\begin{align} \label{eeQQ30}
\mr{Op}^{^\mr{Zzz...}}_{\N, \msX} \isom \mr{Ed}_{p, \N, \mbG} \ \left(\text{resp.,} \ \mr{Op}^{^\mr{Zzz...}}_{\N, \rho, \msX} \isom \mr{Ed}_{p, \N, \mbG, \rho} \right).
\end{align}
In particular, the degree $\mr{deg}(\Pi_{2, \N, g, r, \mbF_p})$ (resp., $\mr{deg}(\Pi_{2, \N, \rho, g, r, \mbF_p})$)   satisfies the equality 
\begin{align} \label{eeQQ31}
\mr{deg}(\Pi_{2, \N, g, r, \mbF_p}) = \sharp (\mr{Ed}_{p, \N, \mbG}) \ \left(\text{resp.,} \ \mr{deg}(\Pi_{2, \N, \rho, g, r, \mbF_p}) = \sharp (\mr{Ed}_{p, \N, \mbG, \rho})  \right),
\end{align}
 and the value  $\sharp (\mr{Ed}_{p, \N, \mbG})$ (resp., $ \sharp (\mr{Ed}_{p, \N, \rho, \mbG})$) does not depend on the choice of $\mbG$, i.e., depends only on the type $(g, r)$.
 \epr
\begin{proof} 
The first  assertion follows from   Proposition \ref{YY188}, (i) and (ii), Theorem \ref{T04rr}.
Also, the generic \'{e}taleness of $\Pi_{g, r}$ proved in  Theorem \ref{c43} 
implies that $\mr{deg}(\Pi_{2, \N, g, r, \mbF_p})$ (resp., $\mr{deg}(\Pi_{2, \N, \rho, g, r, \mbF_p})$) coincides with  $\sharp (\mr{Op}^{^\mr{Zzz...}}_{\N, \msX})$ (resp., $\sharp (\mr{Op}^{^\mr{Zzz...}}_{\N, \rho, \msX})$),  so 
the second assertion follows directly  from the first assertion.
\end{proof}
\SSP

\SSP
\begin{exa}[Case of $(p, \N) = (3, 2)$]\label{Example33}
We here perform a few computations of  
the values $\mr{deg}(\Pi_{2, \N, g, r, \mbF_p})$ in   the case of $p=3$ and $\N =2$.
Note that
the set ${^\dagger}C_{2}$ (for $p=3$) is explicitly given by 
\begin{align} \label{eeQQ71}
{^\dagger}C_{2} = \{ & (0, 0, 0), (0, 2,2), (2, 0, 2), (2,2,0),
(2, 2, 2), \\
& (0, 3, 3), (3, 0, 3), (3, 3, 0), (3, 2,2), (2, 3, 2), (2, 2, 3)  \}.\notag
\end{align}

Now, let $\mbG_{0, 3}$ denote  the trivalent  clutching data   consisting of one vertex  and three edges $e_{0}, e'_{0}, e''_{0}$, which corresponds to the $3$-pointed projective line $\msP_0$
 (cf.  Figure 3 below).
The set  ${^\dagger}C_{2}$ may be identified with the set of balanced $(3, 2)$-edge numberings on $\mbG_{0, 3}$.
In particular, we have  $\mr{deg} (\Pi_{2, 2, 0, 3, \mbF_3}) =11$,
 which is consistent with \eqref{e47}.
 Regarding other pairs of nonnegative integers $(g, r)$, 
 we make the following observations by means of  \eqref{eeQQ71}:
 \begin{itemize}
 \item
 To begin with,  we consider the case of $g = 0$ and $r \geq 3$.
 Let  $\mbG_{0, r+2}$ (for $r \in \mbZ_{>0}$) denote trivalent clutching data of type $(0, r +2)$ whose underlying semi-graph is as  displayed  in Figure 4 below.
 Denote by  $u_r^{[t]}$ ($t = 0, 2, 3$)
 the cardinality of the subset of $\mr{Ed}_{3, 2, \mbG_{0, r+2}}$ consisting of
 balanced $(3, 2)$-edge numberings $(a_e)_e$ with $a_{e_{\sharp, r}} = t$.
  For example, we have $u_1^{[0]}= 3$, $u_1^{[2]} = 5$, and  $u_1^{[3]} = 3$.
 
Since
  $\mbG_{0, r+2}$ ($r >1$) may be constructed from both $\mbG_{0, r +1}$ and  $\mbG_{0, 3}$  in such a way that  the edges $e_{\sharp, r}$ and $e_{0}$ are attached to form a single edge,
  we obtain a recurrence relation for $u_r^{[t]}$ with $t =0$, as follows:
 \begin{align} \label{eeQQ96}
 u_{r}^{[0]} = \sum_{t= 0, 2, 3}  \sharp \left\{ (a_e)_e \in \mr{Ed}_{3, 2, \mbG_{0, 3}} \, \big| \, a_{e_0} = t, a_{e'_0} = 0 \right\} \cdot u_{r-1}^{[t]}   = u_{r-1}^{[0]} + u_{r-1}^{[2]} + u_{r-1}^{[3]}.  
 \end{align}
The  recurrence relations  for $t = 2, 3$ are expressed  in similar manners, and the following sequence of equalities holds:
 \begin{align} \label{eeQQ95}
 \begin{pmatrix}u^{[0]}_{r} \\ u^{[2]}_{r} \\ u^{[3]}_{r}\end{pmatrix}= \begin{pmatrix} 1 & 1 & 1 \\ 1 & 3 & 1 \\ 1 & 1 & 1\end{pmatrix}
 \begin{pmatrix}u^{[0]}_{r-1} \\ u^{[2]}_{r-1} \\ u^{[3]}_{r-1}\end{pmatrix} = \cdots =\begin{pmatrix} 1 & 1 & 1 \\ 1 & 3 & 1 \\ 1 & 1 & 1\end{pmatrix}^{r-1} \begin{pmatrix}u^{[0]}_{1} \\ u^{[2]}_{1} \\ u^{[3]}_{1}\end{pmatrix}
 = \frac{1}{3} \cdot \begin{pmatrix} 1+ 2 \cdot 4^r \\ -1 + 4^{r+1} \\ 1 + 2 \cdot 4^r \end{pmatrix}.
 \end{align}
 
 Thus, 
 it follows from \eqref{eeQQ31} that  the degree of $\Pi_{2, \N, 0, r+2, \mbF_p}$ satisfies 
 \begin{align} \label{eeQQ94}
 \mr{deg}(\Pi_{2, 2, 0, r+2, \mbF_p}) = u_{r}^{[0]} + u_r^{[2]} + u_r^{[3]} = \frac{1 + 2 \cdot 4^{r+1}}{3}.
 \end{align}
 \hspace{-5mm} 
 \includegraphics[width=16cm,bb=0 0 1212 220,clip]{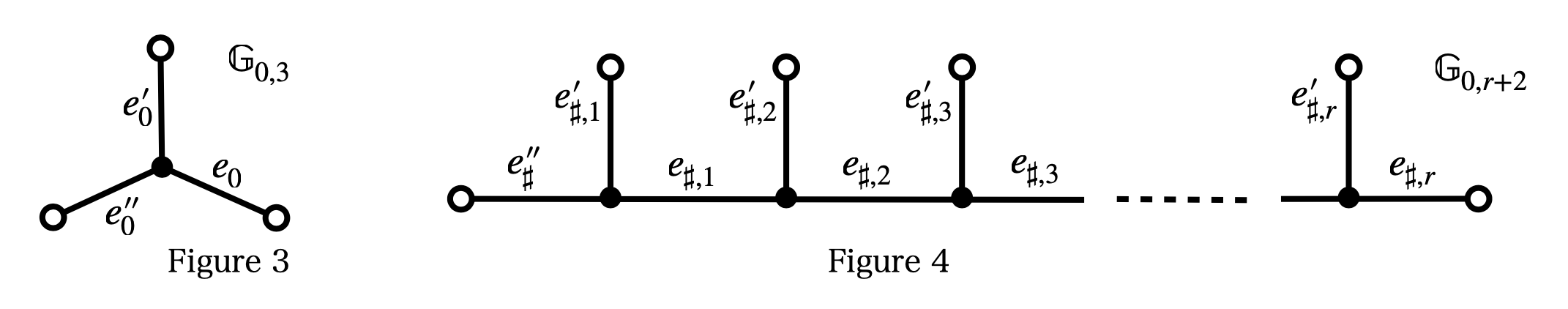}

\vspace{5mm}

 \item
 Next, we discuss the case of $g \geq 1$ and $r=1$.
 Denote by $\mbG_{g, 1}$ ($g \in \mbZ_{>0}$) clutching data of type $(g, 1)$   whose underlying semi-graph is as displayed in Figure 5 below.
 Also, denote by $v_{g}^{[t]}$ ($t =0, 2, 3$) the cardinality of the subset of $\mr{Ed}_{3, 2, \mbG_{g, 1}}$ consisting of balanced $(3, 2)$-edge numberings $(a_e)_e$ with $a_{e_{\flat, g}} = t$.
 For example, we have $v_{1}^{[0]} = 3$, $v_{1}^{[2]} = 1$, and $v_{1}^{[3]} = 1$.
 
 If $\mbG_{1, 2}$ denotes the trivalent clutching data   displayed in Figure 6, then
  $\mbG_{g, 1}$ ($g > 1$) may be obtained from both 
 $\mbG_{g-1, 1}$ and $\mbG_{1, 2}$
   in such a way that the edges $e_{\flat, g-1}$ and $e_1$ are attached to form a single edge.
 It follows that we obtain a recurrence relation for $v_{g}^{[t]}$ with $t = 0$, as follows:
 \begin{align} \label{eeQQ93}
 v_{g}^{[0]} = \sum_{t= 0, 2, 3}  \sharp \left\{ (a_e)_e \in \mr{Ed}_{3, 2, \mbG_{1, 2}} \, \big| \, a_{e_1} = t, a_{e'''_1} = 0 \right\} \cdot v_{g-1}^{[t]}   =  3 \cdot v_{g-1}^{[0]} + v_{g-1}^{[2]} + v_{g-1}^{[3]}.  
 \end{align}
 The recurrence relations for $t = 2, 3$ are expressed  in similar manners,
and  the following sequence of equalities holds:
 \begin{align} \label{eeQQ92}
 \begin{pmatrix}v^{[0]}_{g} \\ v^{[2]}_{g} \\ v^{[3]}_{g}\end{pmatrix}= \begin{pmatrix} 3 & 1 & 1 \\ 1 & 5 & 1 \\ 1 & 1 & 3 \end{pmatrix}
  \begin{pmatrix}v^{[0]}_{g-1} \\ v^{[2]}_{g-1} \\ v^{[3]}_{g-1}\end{pmatrix}
  = \cdots = \begin{pmatrix} 3 & 1 & 1 \\ 1 & 5 & 1 \\ 1 & 1 & 3 \end{pmatrix}^{g-1} \begin{pmatrix} v^{[0]}_{1} \\ v^{[2]}_{1} \\ v^{[3]}_{1}\end{pmatrix} = \begin{pmatrix} 2^{g-1} + 3^{g-1} + 6^{g-1} \\ -3^{g-1} + 2 \cdot 6^{g-1} \\ -2^{g-1} + 3^{g-1} + 6^{g-1}\end{pmatrix}.
 \end{align}
 Thus, the degree of $\Pi_{2, 2, g, 1, \mbF_p}$ satisfies  
 \begin{align} \label{eeQQ91}
 \mr{deg}(\Pi_{2, 2, g, 1, \mbF_p}) = v_g^{[0]} + v_g^{[2]} + v_g^{[3]} = 3^{g-1} + 4 \cdot 6^{g-1}.
 \end{align}
 
 Moreover,  note that 
 trivalent clutching data of type $(g, 0)$ ($g >1$) may be obtained by attaching 
 the respective unique open  edges of $\mbG_{g-1, 1}$ and $\mbG_{1, 1}$ (i.e., ``$\mbG_{g, 1}$" in the case of $g =1$) to form a single edge, so the following equalities hold: 
 \begin{align} \label{eeQQ90}
 \mr{deg}(\Pi_{2, 2, g, 0, \mbF_p}) = v_1^{[0]} \cdot v_{g-1}^{[0]} + v_1^{[2]} \cdot v_{g-1}^{[2]} + v_1^{[3]} \cdot v_{g-1}^{[3]} = 2^{g-1} + 3^{g-1} + 6^{g-1}.
 \end{align}
 \hspace{-5mm} 
 \includegraphics[width=16cm,bb=0 0 1212 220,clip]{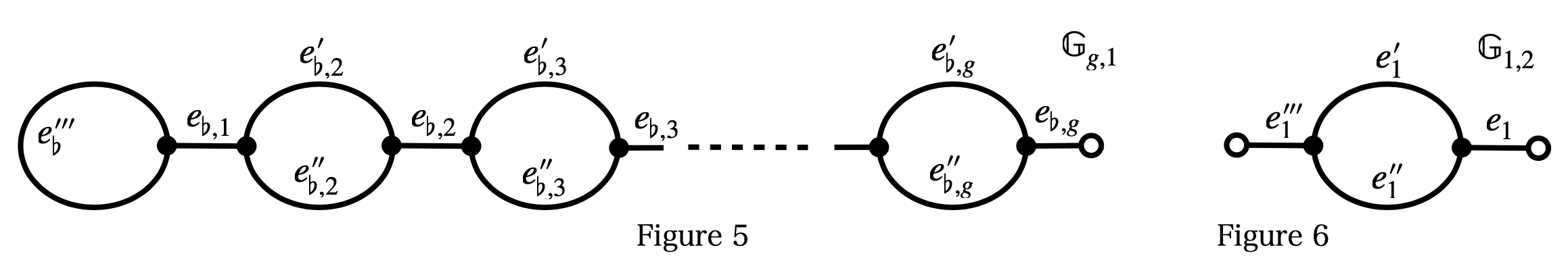}
 
 \vspace{5mm}
 
 \item
 Finally, we deal with  general $g$ and $r$.
 Observe  that some  trivalent clutching data of type $(g, r+1)$ ($g >1$) may be obtained from both $\mbG_{0, r+2}$ and $\mbG_{g, 1}$ by attaching 
 the respective last edges  to form a single edge.
 Hence,   we have 
 \begin{align} \label{eeQQ99}
 \mr{deg}(\Pi_{2, 2, g, r+1, \mbF_p}) = u_r^{[0]} \cdot v_{g}^{[0]} +u_r^{[2]} \cdot v_{g}^{[2]}+u_r^{[3]} \cdot v_{g}^{[3]} = 3^{g-1} + 6^{g-1} \cdot 4^{r+1}.
 \end{align}
 The resulting equality $ \mr{deg}(\Pi_{2, 2, g, r+1, \mbF_p}) = 3^{g-1} + 6^{g-1} \cdot 4^{r+1}$ 
  is true even when $g=0$ (resp., $r = 0$)  because of \eqref{eeQQ94} (resp., \eqref{eeQQ91}).

 \hspace{-5mm} 
 \includegraphics[width=16cm,bb=0 0 1212 240,clip]{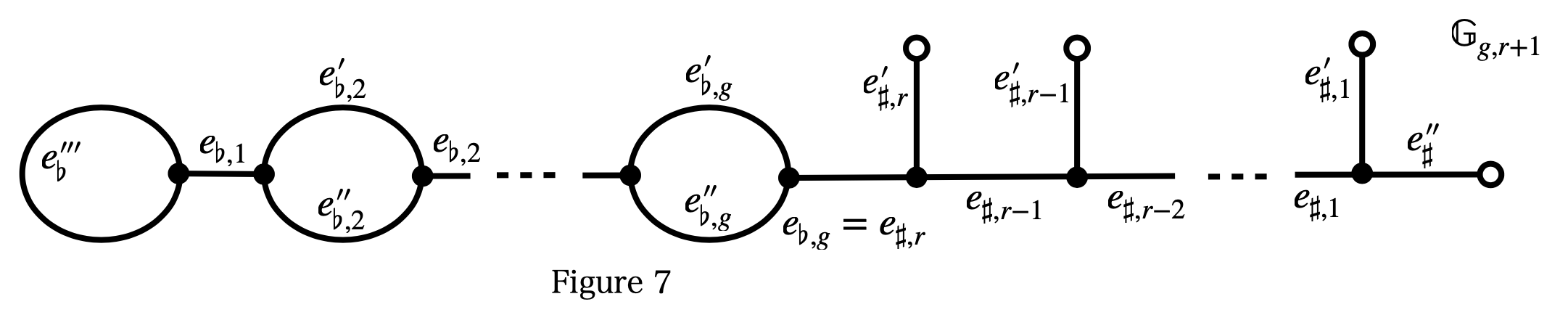}
 \end{itemize}
\end{exa}

\LSP
\subsection{Ehrhart quasi-polynomial counting dormant $\mr{PGL}_2^{(\N)}$-opers} \label{SS10223}

Next, 
in order to apply Ehrhart's theory,
we 
translate  balanced $(p, \N)$-edge numberings into
lattice points inside 
a certain generalized (rational) polytope.
 The main result shows that the values $\mr{deg}(\Pi_{2, \N, g, r, \mbF_p})$ may be expressed by using  a quasi-polynomial in $p$ (cf. Theorem \ref{Thm5}).

To begin with, we introduce the notion of a 
(rational) constructible subset of an $\mbR$-vector space.

\SSP
\bde \label{Def892}
Let  $V$ be a finite-dimensional $\mbR$-vector space equipped 
 with  a choice of a lattice  $L \left( \subseteq V \right)$, i.e., $L$ is a finitely generated free abelian group with $L \otimes_\mbZ \mbR = V$.
A subset $\mcP$ of $V$ is called  {\bf constructible} 
 if there exist a finite number of 
 convex polytopes $\mcQ_1, \cdots, \mcQ_\ell$ in $V$
 satisfying the equality 
  $\mcP = \bigcup_{i=1}^\ell (\mcQ_i \setminus \partial \mcQ_i)$, where each $\partial \mcQ_i$ denotes the relative  boundary  of $\mcQ_i$ in its affine hull (hence $\mcQ_i \setminus \partial Q_i$ is the relative interior of $\mcQ_i$).

 Moreover,  a constructible set $\mcP := \bigcup_{i=1}^\ell (\mcQ_i \setminus \partial \mcQ_i)$ is called {\bf rational} (with respect to $L$) if 
 one can choose  all such polytopes $\mcQ_1, \cdots, \mcQ_\ell$ as rational 
 (with respect to $L$) in the usual sense.
\ede
\SSP

The following property on constructible subsets can be  immediately verified, so the proofs are omitted.

\SSP
\bpr \label{Prop9920}
Let 
$V$ and $L$ be as in Definition \ref{Def892}, and let 
$\mcP_1$, $\mcP_2$  be  constructible subsets (resp., rational constructible subsets) of $V$.
Then, $\mcP_1 \cup \mcP_2$, $\mcP_1 \cap \mcP_2$, and $\mcP_1 \setminus \mcP_2$ are all constructible    (resp.,  constructible and rational).
\epr
\SSP

Given an element  $s$  in the unit interval $[0, 1] := \{ x \in \mbR \, | \, 0 \leq x \leq 1 \}$,
we shall set
$\langle s \rangle_{+1} := s$ and  $\langle s \rangle_{-1} := 1-s$.
Also, we write  $E_0 := \{1, 2, 3 \}$ and write
\begin{align} \label{eeQQ4}
\mr{Sgn}_\N
\end{align}
for the set of collections  $\mpa := (\mpa_{i, j})_{i \in E_0, j \in \{1, \cdots, \N \}}$
such that $\mpa_{i, j} \in \{ +1, -1 \}$ for any $i, j$, and that $(\mpa_{1, \N}, \mpa_{2, \N}, \mpa_{3, \N}) = (+1, +1, +1)$.

For each finite set $I$, we shall denote by 
$\mbR^{I}$ 
the set of all real-valued functions  $I \rightarrow \mbR$ on
$I$;
  it forms a $\sharp (I)$-dimensional $\mbR$-vector space, 
and the set of integer-valued functions $\mbZ^{I}$ forms  
 its   lattice. 
Each element of $\mbR^{I}$ (resp., $\mbZ^{I}$) may be identified with a collection $(s_{i})_{i \in I}$ with  $s_{i} \in \mbR$ (resp., $s_{i} \in \mbZ$).

To each   element $\mpa := (\mpa_{i, j})_{i, j} \in \mr{Sgn}_\N$, we associate the  subset
\begin{align} \label{eeQQ3}
\mcP_1 (\mpa)
\end{align}
of $\mbR^{E_0} \left(= \mbR^{E_0 \times \{1 \}} \right)$  consisting of 
collections  $(s_{i, 1})_{i \in E_0}$ that satisfies $s_{i, 1} \in [0, 1]$ for any $i$, and satisfies
 the following conditions:
\begin{itemize}
\item
 $\langle s_{1, 1} \rangle_{\mpa_{1, 1}} + \langle s_{2, 1} \rangle_{\mpa_{2, 1}} + \langle  s_{3, 1} \rangle_{\mpa_{3, 1}}< 1$,
$0 \leq \langle s_{1, 1}\rangle_{\mpa_{1, 1}} \leq \langle s_{2, 1} \rangle_{\mpa_{2, 1}} + \langle  s_{3, 1}\rangle_{\mpa_{3, 1}}$,
 $0 \leq\langle s_{2, 1}\rangle_{\mpa_{2, 1}} \leq  \langle s_{3, 1} \rangle_{\mpa_{3, 1}} + \langle s_{1, 1}\rangle_{\mpa_{1, 1}}$, and 
  $0 \leq \langle s_{3, 1} \rangle_{\mpa_{3, 1}} \leq  \langle s_{1, 1} \rangle_{\mpa_{1, 1}} + \langle s_{2, 1} \rangle_{\mpa_{2, 1}}$.
\end{itemize}

Moreover, when $\N \geq 2$,  we shall set
\begin{align}
\mcP_2 (\mpa)
\end{align}
to be the subset of $\mbR^{E_0 \times \{2, \cdots, \N \}}$ consisting of collections $(s_{i, j})_{i \in E_0, j \in \{2, \cdots, \N \}}$ that satisfies $s_{i, j} \in [0, 1]$ for any $i$ and $j$, and satisfies the following conditions:
\begin{itemize}
\item
If $(\mpa_{1, j}/\mpa_{1, j-1}, \mpa_{2, j}/\mpa_{2, j-1}, \mpa_{3, j}/\mpa_{3, j-1}) = (+1, +1, +1)$,
then  
$\langle s_{1, j} \rangle_{\mpa_{1, j}} + \langle s_{2, j} \rangle_{\mpa_{2, j}} + \langle  s_{3, j} \rangle_{\mpa_{3, j}}< 1$,
$0\leq\langle s_{1, j}\rangle_{\mpa_{1, j}} \leq \langle s_{2, j} \rangle_{\mpa_{2, j}} + \langle  s_{3, j}\rangle_{\mpa_{3, j}}$,
$0\leq\langle s_{2, j}\rangle_{\mpa_{2, j}} \leq  \langle s_{3, j} \rangle_{\mpa_{3, j}} + \langle s_{1, j}\rangle_{\mpa_{1, j}}$, and 
  $0\leq \langle s_{3, j} \rangle_{\mpa_{3, j}} \leq  \langle s_{1, j} \rangle_{\mpa_{1, j}} + \langle s_{2, j} \rangle_{\mpa_{2, j}}$;
  \item
If $(\mpa_{1, j}/\mpa_{1, j-1}, \mpa_{2, j}/\mpa_{2, j-1}, \mpa_{3, j}/\mpa_{3, j-1}) = (-1, +1, +1)$,
then  
$\langle s_{1, j} \rangle_{\mpa_{1, j}} + \langle s_{2, j} \rangle_{\mpa_{2, j}} + \langle  s_{3, j} \rangle_{\mpa_{3, j}}< 1$,
$0 <\langle s_{1, j}\rangle_{\mpa_{1, j}} \leq \langle s_{2, j} \rangle_{\mpa_{2, j}} + \langle  s_{3, j}\rangle_{\mpa_{3, j}}$,
$0\leq\langle s_{2, j}\rangle_{\mpa_{2, j}} <  \langle s_{3, j} \rangle_{\mpa_{3, j}} + \langle s_{1, j}\rangle_{\mpa_{1, j}}$, and 
  $0\leq \langle s_{3, j} \rangle_{\mpa_{3, j}} <  \langle s_{1, j} \rangle_{\mpa_{1, j}} + \langle s_{2, j} \rangle_{\mpa_{2, j}}$;
    \item
If $(\mpa_{1, j}/\mpa_{1, j-1}, \mpa_{2, j}/\mpa_{2, j-1}, \mpa_{3, j}/\mpa_{3, j-1}) = (+1, -1, +1)$,
then  
$\langle s_{1, j} \rangle_{\mpa_{1, j}} + \langle s_{2, j} \rangle_{\mpa_{2, j}} + \langle  s_{3, j} \rangle_{\mpa_{3, j}}< 1$,
$0\leq \langle s_{1, j}\rangle_{\mpa_{1, j}} < \langle s_{2, j} \rangle_{\mpa_{2, j}} + \langle  s_{3, j}\rangle_{\mpa_{3, j}}$,
$0<\langle s_{2, j}\rangle_{\mpa_{2, j}} \leq  \langle s_{3, j} \rangle_{\mpa_{3, j}} + \langle s_{1, j}\rangle_{\mpa_{1, j}}$, and
  $0\leq \langle s_{3, j} \rangle_{\mpa_{3, j}} <  \langle s_{1, j} \rangle_{\mpa_{1, j}} + \langle s_{2, j} \rangle_{\mpa_{2, j}}$;
  \item
  If $(\mpa_{1, j}/\mpa_{1, j-1}, \mpa_{2, j}/\mpa_{2, j-1}, \mpa_{3, j}/\mpa_{3, j-1}) = (+1, +1, -1)$,
then  
$\langle s_{1, j} \rangle_{\mpa_{1, j}} + \langle s_{2, j} \rangle_{\mpa_{2, j}} + \langle  s_{3, j} \rangle_{\mpa_{3, j}}< 1$,
$0\leq\langle s_{1, j}\rangle_{\mpa_{1, j}} < \langle s_{2, j} \rangle_{\mpa_{2, j}} + \langle  s_{3, j}\rangle_{\mpa_{3, j}}$,
$0\leq \langle s_{2, j}\rangle_{\mpa_{2, j}} <  \langle s_{3, j} \rangle_{\mpa_{3, j}} + \langle s_{1, j}\rangle_{\mpa_{1, j}}$, and
  $0<\langle s_{3, j} \rangle_{\mpa_{3, j}} \leq  \langle s_{1, j} \rangle_{\mpa_{1, j}} + \langle s_{2, j} \rangle_{\mpa_{2, j}}$;
    \item
  If $(\mpa_{1, j}/\mpa_{1, j-1}, \mpa_{2, j}/\mpa_{2, j-1}, \mpa_{3, j}/\mpa_{3, j-1}) = (-1, -1, +1)$,
then  
$\langle s_{1, j} \rangle_{\mpa_{1, j}} + \langle s_{2, j} \rangle_{\mpa_{2, j}} + \langle  s_{3, j} \rangle_{\mpa_{3, j}}\leq 1$,
$0<\langle s_{1, j}\rangle_{\mpa_{1, j}} \leq \langle s_{2, j} \rangle_{\mpa_{2, j}} + \langle  s_{3, j}\rangle_{\mpa_{3, j}}$,
$0<\langle s_{2, j}\rangle_{\mpa_{2, j}} \leq  \langle s_{3, j} \rangle_{\mpa_{3, j}} + \langle s_{1, j}\rangle_{\mpa_{1, j}}$, and 
  $0\leq \langle s_{3, j} \rangle_{\mpa_{3, j}} <  \langle s_{1, j} \rangle_{\mpa_{1, j}} + \langle s_{2, j} \rangle_{\mpa_{2, j}}$;
      \item
  If $(\mpa_{1, j}/\mpa_{1, j-1}, \mpa_{2, j}/\mpa_{2, j-1}, \mpa_{3, j}/\mpa_{3, j-1}) = (-1, +1, -1)$,
then  
$\langle s_{1, j} \rangle_{\mpa_{1, j}} + \langle s_{2, j} \rangle_{\mpa_{2, j}} + \langle  s_{3, j} \rangle_{\mpa_{3, j}}\leq 1$,
$0<\langle s_{1, j}\rangle_{\mpa_{1, j}} \leq \langle s_{2, j} \rangle_{\mpa_{2, j}} + \langle  s_{3, j}\rangle_{\mpa_{3, j}}$,
$0\leq\langle s_{2, j}\rangle_{\mpa_{2, j}} <  \langle s_{3, j} \rangle_{\mpa_{3, j}} + \langle s_{1, j}\rangle_{\mpa_{1, j}}$, and 
  $0<\langle s_{3, j} \rangle_{\mpa_{3, j}} \leq  \langle s_{1, j} \rangle_{\mpa_{1, j}} + \langle s_{2, j} \rangle_{\mpa_{2, j}}$;
      \item
  If $(\mpa_{1, j}/\mpa_{1, j-1}, \mpa_{2, j}/\mpa_{2, j-1}, \mpa_{3, j}/\mpa_{3, j-1}) = (+1, -1, -1)$,
then  
$\langle s_{1, j} \rangle_{\mpa_{1, j}} + \langle s_{2, j} \rangle_{\mpa_{2, j}} + \langle  s_{3, j} \rangle_{\mpa_{3, j}}\leq 1$,
$0\leq\langle s_{1, j}\rangle_{\mpa_{1, j}} < \langle s_{2, j} \rangle_{\mpa_{2, j}} + \langle  s_{3, j}\rangle_{\mpa_{3, j}}$,
$0<\langle s_{2, j}\rangle_{\mpa_{2, j}} \leq  \langle s_{3, j} \rangle_{\mpa_{3, j}} + \langle s_{1, j}\rangle_{\mpa_{1, j}}$, and
  $0<\langle s_{3, j} \rangle_{\mpa_{3, j}} \leq  \langle s_{1, j} \rangle_{\mpa_{1, j}} + \langle s_{2, j} \rangle_{\mpa_{2, j}}$;
       \item
  If $(\mpa_{1, j}/\mpa_{1, j-1}, \mpa_{2, j}/\mpa_{2, j-1}, \mpa_{3, j}/\mpa_{3, j-1}) = (-1, -1, -1)$,
then  
$\langle s_{1, j} \rangle_{\mpa_{1, j}} + \langle s_{2, j} \rangle_{\mpa_{2, j}} + \langle  s_{3, j} \rangle_{\mpa_{3, j}}\leq 1$,
$0<\langle s_{1, j}\rangle_{\mpa_{1, j}} < \langle s_{2, j} \rangle_{\mpa_{2, j}} + \langle  s_{3, j}\rangle_{\mpa_{3, j}}$,
$0<\langle s_{2, j}\rangle_{\mpa_{2, j}} <  \langle s_{3, j} \rangle_{\mpa_{3, j}} + \langle s_{1, j}\rangle_{\mpa_{1, j}}$, and
  $0<\langle s_{3, j} \rangle_{\mpa_{3, j}} <  \langle s_{1, j} \rangle_{\mpa_{1, j}} + \langle s_{2, j} \rangle_{\mpa_{2, j}}$.
\end{itemize}  
On the other hand, we write $\mcP_2 (\mpa) = \mbR^{E_0 \times\{2, \cdots, \N \}} := \{ 0\}$ when $\N = 1$.
Thus, the  subsets $\mcP_1 (\mpa)$, $\mcP_2 (\mpa)$  specify  full-dimensional rational constructible sets in $\mbR^{E_0}$ and $\mbR^{E_0 \times \{2, \cdots, \N \}}$, respectively.

Note that the assignments $\mpa \mapsto \mcP_\ell (\mfa)$  ($\ell =1,2$) commute with permutations of the elements in  $E_0$.
More precisely, let $\mfS_{E_0}$ denote the permutation group  of  $E_0$, and consider its actions on $\mr{Sgn}_\N$ and on the corresponding ambient spaces   defined  as $(\sigma \cdot \mpa)_{i, j} := \mpa_{\sigma^{-1}(i), j}$ and $(\sigma \cdot s)_{i, j} :=s_{\sigma^{-1}(i), j}$ for $\sigma \in \mfS_{E_0}$.
Then, it is verified that
\begin{align}
\sigma \cdot \mcP_1 (\mpa) = \mcP_1 (\sigma \cdot \mpa), \hspace{5mm} \sigma \cdot \mcP_2 (\mpa) = \mcP_2 (\sigma \cdot \mpa).
\end{align}
 



For each $\mpa := (\mpa_{i, j})_{i, j} \in \mr{Sgn}_\N$, we set
\begin{align} \label{YY4}
L_\mpa := ((p-1)
\mcP_1 (\mpa) \cap \mbZ^{E_0}) \times (p\mcP_2 (\mpa) \cap \mbZ^{E_0 \times \{2, \cdots, \N \}}).
\end{align}
where 
for a polytope $\mcP$ and a positive integer $m$ we denote by $m \mcP$
 the polytope $\mcP$ dilated by the factor $m$.
Under the natural identification $\mbR^{E_0} \times \mbR^{E_0 \times \{2, \cdots, \N \}} = \mbR^{E_0 \times \{1, \cdots, \N \}}$,
we regard an element of $L_\mpa$ as a collection of integers $t := (t_{i, j})_{i \in E_0, j \in \{1, \cdots, \N \}}$.
Also, set
$\varepsilon_{i, j} := \frac{1 - \mpa_{i, j}}{2} \in \{ 0, 1 \}$ and
$d_{i, 1} := t_{i, 1}$, $d_{i, j} := t_{i, j} - \varepsilon_{i, j-1}$ ($2 \leq j \leq \N$).
These integers determine a map
\begin{align}
\Phi_\mpa : L_\mpa \rightarrow \mbZ^3; t \mapsto \left(\sum_{j=1}^\N d_{i, j} p^{j-1} \right)_{i=1}^3.
\end{align}

\SSP
\ble \label{Lem78u8}
The maps $\Phi_\mpa$ take values in ${^\dagger}C_\N$ and induce a bijection
\begin{align}
\Phi_\N \left(:= (\Phi_\mpa)_{\mpa \in \mr{Sgn}_\N}\right) : \coprod_{\mpa \in \mr{Sgn}_\N}  L_\mpa \xrightarrow{\sim} {^\dagger}C_\N.
\end{align}
\ele
\begin{proof}
Let us fix $\mpa := (\mpa_{i, j})_{i, j}\in \mr{Sgn}_\N$ and $t := (t_{i, j})_{i, j}\in L_\mpa$.
For $1 \leq j \leq \N$, we set
\begin{align}
r_{i, j} := \sum_{h=1}^j d_{i, h} p^{h-1}, \hspace{5mm}
c_{i, j} := \begin{cases}  r_{i, j}  & \text{if $\mpa_{i, j} = + 1$}; \\ p^j -1 - r_{i, j} & \text{if $\mpa_{i, j} = -1$}. \end{cases}
\end{align}
We prove by induction on $j$  that $(c_{1, j}, c_{2, j}, c_{3, j}) \in C_j$ ($1 \leq j \leq \N$).

For $j=1$, this assertion follows directly from $(t_{1, 1}, t_{2, 1}, t_{3, 1}) \in (p-1) \mcP_1 (\mpa) \cap \mbZ^{E_0}$.
Indeed, $c_{i, 1}$ equals $t_{i, 1}$ or $p-1 - t_{i, 1}$, according to whether $\mpa_{i, 1} = +1$ or $-1$.

Now, suppose that $j \geq 2$ and the assertion holds at level  $j-1$.
We set
\begin{align}
v_i := \begin{cases} t_{i, j} & \text{if $\mpa_{i, j} = +1$}; \\ p-t_{i, j} & \text{if $\mpa_{i, j} =-1$}, \end{cases}
\hspace{5mm}
\delta_i := \begin{cases}0 & \text{if $\mpa_{i, j} = \mpa_{i, j-1}$}; \\ 1 & \text{if $\mpa_{i, j} \neq \mpa_{i, j-1}$}. \end{cases}
\end{align}
A direct calculation gives $c_{i, j} = p^{j-1}v_i + (1 -2\delta_i) c_{i, j-1} - \delta_i$.
Using the defining inequalities for $C_{j-1}$, substitution into the inequalities defining $C_j$ shows that $(c_{1, j}, c_{2, j}, c_{3, j})$ belongs to $C_j$ if and only if    the following conditions are fulfilled:
\begin{align}
v_i &\geq \delta_i  \ \ \ (i=1, 2,3),  \\
v_1 + v_2 + v_3 &\leq p-1 + 1_{\delta_1 + \delta_2 + \delta_3 \geq 2}, \notag \\
- v_i  + \sum_{\ell \neq i} v_\ell &\geq 1_{\delta_1 + \delta_2 + \delta_3 > 2 \delta_i}  \ \ \ (i=1, 2, 3), \notag
\end{align}
where $1_{(-)}$ denotes the indicator function.
These are precisely the integer inequalities obtained from the eight cases in the definition  of $p \mcP_2 (\mpa)$.
They therefore hold for $t$, completing the induction.

Since $(c_{1, j}, c_{2, j}, c_{3, j}) \in C_j$, we have $0 \leq c_{i, j} \leq \frac{p^j -3}{2}$, and $0 \leq r_{i, j} \leq p^j -1$.
The triple of integers $(s_1, s_2, s_3) := \Phi_\mpa (t)$ satisfies $r_{i, j} = [s_i]_j$.
Moreover, $\mpa_{i, \N} = +1$ for every $i$, so the equality $(s_1, s_2, s_3) = (c_{1, \N}, c_{2, \N}, c_{3, \N})$ holds.
The corresponding folded triples at all lower levels belong to $C_j$, as proved above.
Hence, $(s_1, s_2, s_3)$ turns out to be an element of ${^\dagger}C_\N$.

Conversely, let $(s_1, s_2, s_3)$ be an element of ${^\dagger}C_\N$.
We set 
\begin{align} \label{YY5}
\mpa_{i, j} := \begin{cases} +1 & \text{if $[s_i]_j \leq \frac{p^j -1}{2}$}; \\ -1 & \text{if $[s_i]_j > \frac{p^j -1}{2}$}.\end{cases}
\end{align}
Since $(s_1, s_2, s_3) \in C_\N$, we have $\mpa_{i, \N} = +1$, so the  resulting collection $\mpa := (\mpa_{i, j})_{i, j}$ belongs to $\mr{Sgn}_\N$.

Write the base-$p$ expansion of each $s_i$ as 
\begin{align} \label{YffY6}
s_i = \sum_{j=1}^\N d_{i, j} p^{j-1}
\end{align}
 ($0 \leq d_{i, j} \leq p-1$), and set
$t_{i, 1} := d_{i, 1}$, $t_{i, j} := d_{i, j} + \frac{1 - \mpa_{i, j-1}}{2}$ ($2 \leq j \leq \N$).
By the definition of ${^\dagger}C_\N$, all the folded triples determined by $(s_1, s_2, s_3)$ belong to the corresponding $C_j$.
Applying the preceding comparison of inequalities in reverse, we obtain  $t \in L_\mpa$.
By construction,  the equality $\Phi_\mpa (t) = s$ holds.

Finally, the signs $\mpa$ and the digits $d_{i, j}$ are uniquely recovered from $(s_1, s_2, s_3)$ by the displayed formulas \eqref{YY5}, \eqref{YffY6}.
These formulas therefore give an inverse to $\Phi_\N$, proving the asserted bijectivity.
\end{proof}
\SSP

Next, let
$\mbG$ and $\GR$ be as in Section \ref{SS10220}.
We denote by
\begin{align} \label{DFG1}
\mr{Sgn}_{\N, \mbG}
\end{align}
the set of collections $\mpa := (\mpa_{e, j})_{e \in E, j \in \{1, \cdots, \N \}}$ such that $\mpa_{e, j} \in \{ + 1, - 1 \}$ for any $e, j$,
 and that
 $\mpa_{e, \N} = +1$ for any $e$.
 Each $\mfa  := (\mpa_{e, j})_{e, j} \in \mr{Sgn}_{\N, \mbG}$ determines  the subset
\begin{align} \label{eeQQ25}
\mcP_{1, \mbG} (\mpa) \ \left(\text{resp.,} \ \mcP_{2, \mbG}(\mpa) \right)
\end{align}
of 
the $\mbR$-vector space $\mbR^E$ (resp.,  $\mbR^{E \times \{2, \cdots, \N \}}$) consisting of real-valued functions 
$(s_{e, 1})_{e \in E}$ (resp., $(s_{e, j})_{e \in E, j \in \{2, \cdots, \N \}}$)
 on 
 $E$ (resp.,
 $E \times \{2, \cdots, \N \}$) such that,
 for each $v \in V$, the collection $(s_{e_b, 1})_{b \in B_v}$  (resp., $(s_{e_b, j})_{b \in B_v, j \in \{2, \cdots, \N \}}$)
   belongs to $\mcP_1 ((\mpa_{e_b, 1})_{b \in B_v})$  (resp., $\mcP_2 ((\mpa_{e_b, j})_{b \in B_v, j \in \{2, \cdots, \N \}})$) under some (and hence, any) identification $E_0 = B_v $,
    where $e_b$ denotes the unique edge satisfying $b \in e_b$.
We set $\mcP_{2, \mbG} (\mpa) = \mbR^{E \times \{2, \cdots, \N \}} := \{ 0 \}$ when $\N =1$.
Note that the  closure of the set $\mcP_{1, \mbG}(\mpa)$
for $\N =1$, forming  a rational convex polytope,   
 was introduced  in ~\cite[Definition 2.3]{LO}.

\SSP
\bpr \label{Prop4378}
\begin{itemize}
\item[(i)]
Both 
$\mcP_{1, \mbG} (\mpa)$ and $\mcP_{2, \mbG} (\mpa)$ (for any $\mpa \in \mr{Sgn}_{\N, \mbG}$) form  full-dimensional  rational constructible subsets of
$\mbR^E$ and $\mbR^{E \times \{2, \cdots, \N \}}$, respectively.
\item[(ii)]
Write  $\msX$ for  the totally degenerate curve induced by $\mbG$ (cf. Definition \ref{DD3WW}) and $\mr{Op}^{^\mr{Zzz...}}_{\N, \msX}$ for  the set of isomorphism classes of dormant $\mr{PGL}_2^{(\N)}$-opers on $\msX$ (defined as in Proposition \ref{LLL001}).
Then, there exists a canonical bijection of sets
\begin{align}
\mr{Op}^{^\mr{Zzz...}}_{\N, \msX} \isom \coprod_{\mpa \in \mr{Sgn}_{\N, \mbG}} ((p-1)\mcP_{1, \mbG} (\mpa) \cap  \mbZ^E) \times (p\mcP_{2, \mbG} (\mpa)\cap \mbZ^{E \times \{2, \cdots, \N \}}).
\end{align}
\end{itemize}
\epr
\begin{proof}
Both sets are bounded and defined by finitely many strict or weak linear inequalities with rational coefficients. They are therefore rational constructible subsets of their respective ambient spaces.
To prove that they are full-dimensional, let us fix
$\mpa := (\mpa_{e, j})_{e, j} \in \mr{Sgn}_{\N, \mbG}$, and set
\begin{align}
s_{e, j}^* := \begin{cases} \frac{1}{4} & \text{if $\mpa_{e, j} = +1$}, \\  \frac{3}{4} & \text{if $\mpa_{e, j} = -1$}. \end{cases}
\end{align}
Then, $\langle s^*_{e, j} \rangle_{\mpa_{e, j}} = \frac{1}{4}$ for every $e$ and $j$.
At each vertex, the three resulting weights satisfy
\begin{align}
0 < \frac{1}{4} < \frac{1}{4} + \frac{1}{4}, \hspace{5mm} \frac{1}{4} + \frac{1}{4} + \frac{1}{4} < 1.
\end{align}
Consequently, all the defining inequalities remain satisfied after sufficiently small perturbations of the coordinates.
Hence, $(s^*_{e, 1})_{e \in E}$ and $(s^*_{e, j})_{e \in E, j \in \{ 2, \cdots, \N \}}$ are interior points of $\mcP_{1, \mbG} (\mpa)$ and $\mcP_{2, \mbG} (\mpa)$, respectively.
This proves their full-dimensionality.
When $\N =1$, the second assertion follows from the convention $\mcP_{2, \mbG} (\mpa) = \{ 0\}$ in the zero-dimensional ambient space.
We finish  the proof of assertion (i).


Assertion (ii) follows from 
Proposition \ref{LLL001} and Lemma \ref{Lem78u8}.
\end{proof}
\SSP

Here, recall a result in the theory of lattice-point counting for rational  polytopes by E. Ehrhart.
Let $V$ and $L$ be as in Definition \ref{Def892}.
Given  a rational convex polytope $\mcP$ in $V$, we denote by 
 \begin{align} \label{eeQQ100}
i_\mcP  : \mbZ_{> 0} \rightarrow \mbZ_{\geq 0}
\end{align}
the lattice-point counting function for $\mcP$, i.e., the function which, to any positive  integer $m$, assigns the cardinality
of lattice points in $\mcP$ dilated by $m$:
\begin{align} \label{eeQQ101}
i_\mcP (m) := \sharp \left(m \mcP \cap L\right).
\end{align}

Suppose that $\mcP$ is of dimension $d$, i.e., $d$ is the dimension of the  smallest affine space of $V$ containing $\mcP$.
Then, E. Ehrhart proved (cf. ~\cite{Ehr1}, ~\cite{Ehr2}, ~\cite{Ehr3}) that the function $i_\mcP$ is a quasi-polynomial function of degree $d$ with coefficients in $\mbQ$.
That is to say, there exist a positive integer $\ell$ and a (unique) sequence of polynomials
\begin{align}
H_\mcP (t) := (H_{\mcP, s} (t))_{s \in \mbZ},
\end{align}
where $H_{\mcP, s} (t)$ denotes a polynomial of degree at most  $d$ with coefficients in $\mbQ$ (i.e., an element of $\mbQ [t]$), such that 
\begin{align}
i_\mcP (m) = H_{\mcP, s} (m) \left(=:H_{\mcP} (m) \right)\ \text{for} \ m \equiv s\ (\mr{mod} \ \ell).
\end{align}
The sequence of polynomials  $H_\mcP (t)$ is called the {\it Ehrhart quasi-polynomial} of $\mcP$.
If  $\mr{Per} (\mcP)$  denotes the minimum period of $H_\mcP (t)$, then any positive integer $\ell$ as above is divided by  $\mr{Per} (\mcP)$.

One may immediately verify that the existence of such a quasi-polynomial can be extended to the case of rational constructible subsets of $V$.
In particular, by applying this result to $\mcP_{1, \mbG} (\mpa)$ and $\mcP_{2, \mbG} (\mpa)$ (for various $\mpa$'s), we obtain the following assertion, generalizing ~\cite[Theorem 2.1]{LO}.

\SSP
\bt \label{Thm5}
There exists 
a  quasi-polynomial  $H_{\N, \mbG} (t) := (H_{\N, \mbG, s} (t))_{s \in \mbZ}$ with coefficients in $\mbQ$ of degree $(3g-3+2r) \cdot \N$ 
such that its  minimum  period divides $4$ and satisfies
 the equalities
\begin{align} \label{eeQQ110}
\mr{deg}(\Pi_{2, \N, g, r, \mbF_p}) = \sharp (\mr{Ed}_{p, \N, \mbG})  = H_{\N, \mbG} (p)
\end{align}
for every odd prime $p$.
Moreover, the odd constituents of  $H_{\N, \mbG} (t)$ do not depend on the choice of $\mbG$ (i.e., depend only on the  type $(g, r)$ and the positive integer $\N$).
\et
\begin{proof}
The existence of a quasi-polynomial $H_{\N, \mbG} (t)$ satisfying 
the equalities \eqref{eeQQ110}  follows from Theorem \ref{c43} and Propositions \ref{LLL001},  \ref{Prop4378}.

Next, let us consider  periods of $H_{\N, \mbG} (t)$.
For each $\mpa := (\mpa_{e, j})_{e, j} \in \mr{Sgn}_{\N, \mbG}$,
we have $\mcP_{2, \mbG} (\mpa)= \prod_{j' =2}^\N \mcP_{2, \mbG, j'}(\mpa)$, where
 $\mcP_{2, \mbG, j'} (\mpa)$ ($j' =2, \cdots, \N$) denotes  the image of $\mcP_{2, \mbG} (\mpa)$ via the natural projection $\mbR^{E \times \{2, \cdots, \N \}} \twoheadrightarrow \mbR^{E \times \{ j' \}}$.
 The closure 
  $\overline{\mcP}_{2, \mbG, j'} (\mpa) \left(\subseteq\mbR^{E \times \{ j' \}}\right)$ of $\mcP_{2, \mbG, j'} (\mpa)$, which is  obtained from  $\mcP_{2, \mbG, j'} (\mpa)$ by adjoining its boundary, 
 decomposes as a disjoint union of rational convex polytopes $\mcQ_1, \mcQ_2, \cdots, \mcQ_M$ (for some $M \geq 1$).
After possibly applying reflections through coordinate axes and lattice translations,
 each $\mcQ_\ell$ coincides with  the polytope ``$\msP_G$" introduced in ~\cite[Definition 2.3]{LO} (or ~\cite[Definition 3.1]{Wak2})  for  $G = \mbG$.
Hence, it follows from  ~\cite[Proposition 3.5]{LO} that
the coordinates of any vertex of $\mcQ_\ell$ belongs to $\frac{1}{4}\mbZ$.
This fact together with  a well-known fact of Ehrhart's theory implies
that the minimum periods of the Ehrhart quasi-polynomials of $\mcQ_\ell$ and its faces  divide $4$.
It shows  that
 $\mr{Per}(\mcP_{2, \mbG} (\mpa)) \mid 4$, and the same is true for $\mcP_{1, \mbG} (\mpa)$ because of 
 a similar (but relatively simpler) argument.
 Therefore, by the bijection resulting from Proposition \ref{Prop4378}, (ii),  
  the minimal period of $H_{\N, \mbG} (t)$   turns out to 
  divide $4$.

Finally,  we prove 
 the last assertion.
 According to Dirichlet's theorem on arithmetic progressions,
there are infinitely many primes $p'$ with $p'  \equiv s$ (mod $4$) for every  fixed odd integer $s$.
In particular, 
if $H'_{\N, \mbG} (t) := (H'_{\N, \mbG, s} (t))_{s \in \mbZ}$ is another  quasi-polynomial of period $4$ satisfying the equalities \eqref{eeQQ110}, then $H_{\N, \mbG} (m) = H'_{\N, \mbG} (m)$  for infinitely many $m$'s with $m \equiv s$ (mod $4$).
It follows that $H'_{\N, \mbG, s} (t)$ must coincide with 
 $H_{\N, \mbG, s} (t)$ as a polynomial  whenever $s$ is odd.
 This completes the proof of this assertion.
\end{proof}


\LSP
\subsection{Counting $2$nd order ODE's with a full set of solutions} \label{SSfff}
As a consequence of the results obtained so far, a partial answer to the  question  displayed at the beginning of Section \ref{WW60} can be given as follows.

Let $X$ be a geometrically connected, proper, and smooth curve of genus $g >1$ over $S := \mr{Spec}(R)$ for a flat $\mbZ/p^\N \mbZ$-algebra $R$ such that $R/p R$ is an algebraically closed field over $\mbF_p$.
Given a line bundle  $\mcL$  on $X$,
we 
denote by
\begin{align} \label{eeQQ571}
\mr{Diff}_{2, \mcL}^\mr{full}
\end{align}
the set of $2$nd order linear differential operators (over $S$) on $\mcL$ with unit principal symbol and having a full set of root functions (cf. Definitions \ref{Def898} and \ref{D01fgh}).

\SSP
\bt \label{Theorem895}
\begin{itemize}
\item[(i)]
Suppose that  $\mr{obs}(\mcL^{\otimes 2} \otimes \Omega) \neq 0$ (cf.  \eqref{eeQQ468} for the definition of $\mr{obs}(-)$).
Then, we have $\mr{Diff}_{2, \mcL}^\mr{full} = \emptyset$.
\item[(ii)]
Suppose that 
$\mr{obs}(\mcL^{\otimes 2} \otimes \Omega) = 0$ and 
the mod $p$ reduction $X_0$ of $X$ is sufficiently general in $\mcM_{g, 0}$.
Then,  the set
$\mr{Diff}_{2, \mcL}^\mr{full}$ is finite and its 
 cardinality $\sharp (\mr{Diff}_{2, \mcL}^\mr{full})$ satisfies 
\begin{align}
\sharp (\mr{Diff}_{2, \mcL}^\mr{full}) 
  =  p^{g\N} \cdot \sharp (\mr{Ed}_{p, \N, \mbG})  = p^{g\N} \cdot H_{\N, \mbG} (p)
\end{align}
for any trivalent clutching data $\mbG$ of type $(g, 0)$,
where $H_{\N, \mbG} (t)$ denotes the quasi-polynomial resulting from Theorem \ref{Thm5}.
In particular,  $\sharp (\mr{Diff}_{2, \mcL}^\mr{full})$ may be expressed as a rational  quasi-polynomial in $p$  of degree $(4g-3) \cdot \N$. 
\end{itemize}
\et
\begin{proof}
For a $2^{(1)}$-theta characteristic $\vartheta$ of $X/S$, we 
denote by $\mr{Diff}_{\clubsuit, \vartheta}^\mr{full}$ the set of $(2, \vartheta)$-projective connections on $X/S$ having a full set of root functions (cf. Definitions \ref{WW288}, (ii), and \ref{D01fgh}).
Then, the  set $\mr{Diff}_{2, \mcL}^\mr{full}$ decomposes as 
\begin{align} \label{eeQQ6672}
\mr{Diff}_{2, \mcL}^\mr{full} = \coprod_{\vartheta \in W} \mr{Diff}^\mr{full}_{\clubsuit, \vartheta},
\end{align}
where 
$W$ denotes   the set of dormant $2^{(1)}$-theta characteristics $\vartheta := (\varTheta, \nabla_\vartheta)$ of $X/S$ with $\varTheta^\vee = \mcL$.

If $\mr{obs}(\mcL^{\otimes 2} \otimes \Omega) \neq 0$ (which implies that there is no $S$-connection on $\mcT \otimes (\mcL^\vee)^{\otimes 2}$), then $W$ must be empty.
Hence, in this case, the equality $\mr{Diff}_{2, \mcL}^\mr{full} = \emptyset$ holds.

On the other hand, suppose that $\mr{obs}(\mcL^{\otimes 2} \otimes \Omega) = 0$ and 
$X_0$ is sufficiently general, in particular, ordinary in the usual sense.
It follows from Proposition \ref{Prop7872} that $W$ is  in bijection with the set 
${^{\Diag}}\!\!W$
of dormant $2^{(\N)}$-theta characteristics  $\vartheta_0 := (\varTheta_0, \nabla_{\vartheta_0})$ of $X_0/S_0$ with $\varTheta_0^\vee = \mcL_0$.
Such an $\mcD_0^{(\N -1)}$-module structure $\nabla_{\vartheta_0}$ has vanishing $p^{\N}$-curvature, so the elements in ${^{\Diag}}\!\!W$ correspond bijectively to the set of isomorphism classes of line bundles $\mcN$ on $X_0^{(\N)}$ with $F^{(\N)*}_{X_0/S_0}(\mcN)\cong \mcT_0 \otimes \mcL_0^{\otimes (-2)}$ (cf. ~\cite[Corollary 3.2.4]{LSQ}).
It follows from $\mr{obs} (\mcL^{\otimes 2} \otimes \Omega) =0$ that $p^\N \mid \mr{deg} (\mcT_0 \otimes \mcL_0^{\otimes (-2)}) =:a$.
Since the twisted $\N$-th Verschiebung map $\mr{Pic}_{X^{(\N)}_0/S_0}^{a/p^\N} \rightarrow \mr{Pic}^a_{X_0/S_0}$ between    Picard schemes   of the relative curve $X_0/S_0$ is \'{e}tale, surjective,  and of degree $p^{g \N}$, the equalities  $\sharp ({^{\Diag}}\!\!W) = \sharp (W) = p^{g \N}$ hold.
It follows that
\begin{align}
\sharp (\mr{Diff}_{2, \mcL}^\mr{full}) = \sum_{\vartheta \in W} \sharp (\mr{Diff}^\mr{full}_{\clubsuit, \vartheta}) = p^{g \N} \cdot 
\sharp (\mr{Op}_{1, X}^{^\mr{Zzz...}})
\end{align}
(cf. \eqref{e209} for the definition of
 $\mr{Op}_{1, X}^{^\mr{Zzz...}}$),
where the first equality follows from \eqref{eeQQ6672} and the second equality follows from the sequence of isomorphisms \eqref{GGGjjG}.
Thus, the assertion follows from Corollary \ref{C998} and Theorem \ref{Thm5}.
\end{proof}

\LSP
\subsection{Conjectural relationship with the generalized Verschiebung maps} \label{SS101}

In this final subsection, we would like to formulate a conjecture concerning 
the relationship between the values $\mr{deg}(\Pi_{n, \N, g, 0, \mbF_p})$ and the generic degrees of the so-called ``{\it generalized Verschiebung}" rational maps, based on some intuitive observations.

Let $n$ be a positive integer with $n < p$,
$k$ an algebraically closed field   of characteristic $p$, and  $X$  a connected proper smooth  curve over $k$ of genus $g >1$.
For each integer $\N'$ with $0 \leq \N' \leq \N$, we  denote by $\mr{SU}^n_{X^{(\N')}}$ the moduli space of rank $n$ stable bundles with trivial determinant on $X^{(\N')}$ (where $X^{(0)} := X$).
Then, pulling-back  stable bundles on $X^{(\N')}$ ($\N' \geq 1$) via
the  relative  Frobenius morphism $F_{X^{(\N' -1)}/k} : X^{(\N'-1)} \migi X^{(\N' )}$ induces  a  rational   map
\begin{align} \label{e1}  
\mr{Ver}_{X^{(\N'-1)}}^n : \mr{SU}_{X^{(\N')}}^n \dashrightarrow \mr{SU}_{X^{(\N' -1)}}^n.
\end{align}
This is called  the {\bf generalized Verschiebung map} and 
   regarded as a higher-rank variant of the Verschiebung maps between Jacobians.

The geometry  of the rational map $\mr{Ver}_{X^{(\N')}}^n$, i.e., 
the dynamics of stable bundles with respect to Frobenius pull-back,  has been investigated 
for a long time.
In particular, the case of $(n, g)  =(2, 2)$ has already been investigated considerably.
We know that, 
if $X$ has genus $2$,
then the compactification of $\mr{SU}_{X}^2$ by semistable bundles is canonically isomorphic to the $3$-dimensional projective space  and the boundary locus can be identified with the Kummer surface associated to $X$.
By this description, 
the rational map $\mr{Ver}_{X}^2$ can be expressed  by polynomials of degree $p$ (cf. ~\cite[Proposition A.2]{LaPa2}), which are explicitly described in the cases $p=2$ (cf. ~\cite{LaPa}) and $p=3$ (cf. ~\cite{LaPa2}).
Moreover, the scheme-theoretic base locus of $\mr{Ver}_{X}^2$ were computed in  ~\cite[Theorem 2]{LP}.
This result enables us to specify the generic degree $\mr{deg}(\mr{Ver}_{X}^{2})$ of
$\mr{Ver}_{X}^{2}$, i.e.,  if $X$ is general in $\mcM_{2, 0}$, then we have
 \begin{align} \label{eQgght}
 \mr{deg}(\mr{Ver}_{X}^{2}) = \frac{p^3+2p}{3}
 \end{align}
  (cf. ~\cite[Theorem 1.3]{O3}, ~\cite[Corollary]{LP}). 

For a general $g$, we gave an upper bound for $\mr{deg}(\mr{Ver}_{X}^{2})$ (cf. ~\cite[Theorem A]{HoWa}) by applying Holla's formula computing the degrees of certain Quot schemes.
However, we have not yet reached a comprehensive understanding of 
this value
 because
not much seems to be known   for 
general $(p, n, g)$.

Let us make the following observation about computing it explicitly.
Denote by $\vartheta^{(\N)}$ the $n^{(\N)}$-theta characteristic of $X/k$ arising from a fixed theta characteristic in the manner of Example \ref{Ex20}.
Suppose that we are given a dormant $\mr{PGL}_n$-oper on $X$, which corresponds to a dormant $(\mr{GL}_n, \vartheta^{(1)})$-oper $\msF^\heartsuit := (\mcF, \nabla, \{ \mcF^j \}_{j=0}^n)$.
The sheaf of horizontal sections $\mcS ol (\nabla) \left(= \mr{Ker}(\nabla) \right)$
is verified to be stable and specifies a $k$-rational point $q$ of $\mr{SU}_{X^{(1)}}^n$.
As proved in ~\cite[Theorem H]{Wak8}, the number  of  
all such points $q$ obtained  in this way  (i.e., associated to some $\msF^\heartsuit$)
 is given  by
 the right-hand side of \eqref{eeQQ782} under some assumptions.
Also, it follows from the equivalence \eqref{YY6} that giving a dormant $(\mr{GL}^{(\N)}_n,  \vartheta^{(\N)})$-oper extending $\mcF^\heartsuit$ amounts to giving a rank $n$ vector bundle $\mcV$ on $X^{(\N)}$ with $\mr{det}(\mcV) \cong \mcO_{X^{(\N)}}$ and  $F^{(\N -1)*}_{X^{(1)}/k} (\mcV) \cong \mcS ol (\nabla)$, or equivalently, a vector bundle classified by a point in the fiber $\mr{Ver}_{(\N -1) \Rightarrow (1)}^{-1}(q)$ of  $q$  via the composite rational map
\begin{align} 
\mr{Ver}_{(\N -1) \Rightarrow (1)} :=\mr{Ver}^n_{X^{(1)}} \circ \mr{Ver}^n_{X^{(2)}} \circ \cdots \circ  \mr{Ver}^n_{X^{(\N -1)}}  : \mr{SU}_{X^{(\N)}}^n \dashrightarrow \mr{SU}_{X^{(1)}}^n.
\end{align}
Since $\mr{Ver}^n_{X^{(\N')}}$ (for $\N' \geq 1$) can be obtained as the base-change of $\mr{Ver}_X^n$, 
 the generic degree $\mr{deg}(\mr{Ver}_{(\N -1) \Rightarrow (1)})$ of $\mr{Ver}_{(\N -1) \Rightarrow (1)}$ coincides with $\mr{deg}(\mr{Ver}_X^n)^{\N -1} \left(= \prod_{\N' =1}^{\N-1}  \mr{deg}(\mr{Ver}_{X^{(\N')}}^n)\right)$.
On the other hand, if $q$ lies   in a general position (as an ideal setting), then the  total number of points in  $\mr{Ver}_{(\N -1) \Rightarrow (1)}^{-1}(q)$ would  amount to  the value  $\mr{deg}(\mr{Ver}_{(\N -1) \Rightarrow (1)})$.
With these points in mind,
   the degree  $\mr{deg}(\Pi_{n, \N, g, 0, \mbF_p})$ for a general $n$ might be expected
to satisfy the equality
\begin{align} \label{align22}
\mr{deg}(\Pi_{n, \N, g, 0, \mbF_p}) =  \mr{deg}(\mr{Ver}_X^n)^{\N -1} \cdot  \frac{p^{(n-1)(g-1)-1}}{n!} \cdot  \hspace{-5mm}
 \sum_{\genfrac{.}{.}{0pt}{}{(\zeta_1, \cdots, \zeta_n) \in \mbC^{\times n} }{ \zeta_i^p=1, \ \zeta_i \neq \zeta_j (i\neq j)}}
 \frac{(\prod_{i=1}^n\zeta_i)^{(n-1)(g-1)}}{\prod_{i\neq j}(\zeta_i -\zeta_j)^{g-1}}.
\end{align}

Even if not all the points $q$  as above
  lie in a general position, we expect that  the situation approximates the ideal setting   as  $\N$ goes to infinity.
Thus, our  conjecture is formulated as follows.

\SSP
\begin{conj} \label{Conj4}
Suppose that $X$ is  general in $\mcM_{g, 0}$ and $p$ is sufficiently large relative to $n$ and $g$, 
Then, the relationship between the generic degrees of $\mr{Ver}^n_X$ and $\Pi_{n, \N, g, 0, \mbF_p}$  can be described as the following formula:
\begin{align}
\mr{deg}(\mr{Ver}_X^n) = \lim_{\N \to \infty} \mr{deg}(\Pi_{n, \N, g, 0, \mbF_p})^{1/\N}.
\end{align}
More strongly, the equality \eqref{align22} holds.
\end{conj}
\SSP

After this conjecture for $n=2$ has been affirmatively solved, the value $\mr{deg}(\mr{Ver}_X^2)$ may be explicitly computed, via Proposition \ref{LLL001},  by counting  balanced $(p, \N)$-edge numberings on a trivalent graph.
Note that the case of    $(n, g) = (2, 2)$ follows directly from     Corollary \ref{C19},  \eqref{fffrr}, and  \eqref{eQgght}.
That is to say, we have the following assertion.

\SSP
\bco
Conjecture  \ref{Conj4} is true when $n=2$ and $g=2$.
\eco

\vspace{10mm}

\section*{List of notation}

The principal notation used in this manuscript is collected below, grouped
by topic. The references indicate where the notation is introduced or
explained. All objects are understood under the hypotheses of the cited
section. In particular, a $\mr{PGL}_n$- or $\mr{GL}_n$-oper of level $\N$
involves differential operators of level $\N-1$.

\begingroup
\footnotesize
\setlength{\tabcolsep}{4pt}
\setlength{\LTleft}{0pt}
\setlength{\LTright}{0pt}
\renewcommand{\arraystretch}{1.15}
\newcommand{\NotationRow}[3]{#1 & #2 & #3 \\[3pt]}
\newcommand{\NotationGroup}[1]{%
  \addlinespace[6pt]
  \multicolumn{3}{@{}l}{\textbf{#1}}\\*[4pt]}

\begin{longtable}{@{}>{\raggedright\arraybackslash}p{0.30\textwidth}%
  >{\raggedright\arraybackslash}p{\dimexpr0.55\textwidth-16pt\relax}%
  >{\raggedright\arraybackslash}p{0.15\textwidth}@{}}
\toprule
\textbf{Notation} & \textbf{Meaning} & \textbf{Reference} \\
\midrule
\endfirsthead
\multicolumn{3}{@{}l}{\textit{List of notation (continued)}}\\[3pt]
\toprule
\textbf{Notation} & \textbf{Meaning} & \textbf{Reference} \\
\midrule
\endhead
\midrule
\multicolumn{3}{r@{}}{\textit{Continued on the next page}}\\
\endfoot
\bottomrule
\endlastfoot

\NotationGroup{General conventions and indices}
\NotationRow{$p$, $\mbF_p$}
  {The fixed prime and the field $\mbZ/p\mbZ$.}
  {\S\,1.7}
\NotationRow{$\M$, $\N$, $\CH$}
  {Level of differential operators, level of opers, and the index in the base ring $\mbZ/p^{\CH+1}\mbZ$, respectively; $\M=\N-1$ when these levels correspond.}
  {\S\S\,\ref{SS040}, \ref{SS010}, \ref{SS041}}
\NotationRow{$(-)_\M$}
  {Reduction modulo $p^{\M+1}$; in particular, $(-)_0$ denotes reduction modulo $p$.}
  {\S\,\ref{SS010}}
\NotationRow{$X^{(a)}$, $F_{X/S}^{(a)}$}
  {The $a$-th Frobenius twist over $S$ and the $a$-th relative Frobenius morphism; $F_{X/S}=F_{X/S}^{(1)}$.}
  {\S\,\ref{SS058}}
\NotationRow{$\mcM_{S^\mr{log}}$, $\alpha_{S^\mr{log}}$}
  {The sheaf of monoids defining the log structure and its structural map to $\mcO_S$.}
  {\S\,1.7}
\NotationRow{$\overline{\mcM}_{X^\mr{log}/S^\mr{log}}$}
  {The relative characteristic of the morphism $X^\mr{log}\to S^\mr{log}$.}
  {\S\,1.7}
\NotationRow{$\Omega$, $\mcT$}
  {The logarithmic cotangent and tangent sheaves $\Omega_{X^\mr{log}/S^\mr{log}}$ and $\mcT_{X^\mr{log}/S^\mr{log}}$; $\mcT=\Omega^\vee$.}
  {\S\,\ref{SS040}}
\NotationRow{$\Omega^{(a)}$, $\mcT^{(a)}$}
  {The logarithmic cotangent and tangent sheaves on $X^{(a)\mr{log}}/S^\mr{log}$.}
  {\S\,\ref{SS058}}
\NotationRow{$\mcK^\bullet[\nabla]$}
  {The two-term complex defined by $\nabla:\mcK^0\to\mcK^1$, placed in degrees $0$ and $1$.}
  {\S\,1.7}
\NotationRow{$\mcE nd^0(\mcF)$}
  {The sheaf of $\mcO_X$-linear endomorphisms of $\mcF$ with vanishing trace.}
  {\S\,\ref{S21}}

\NotationGroup{Differential operators, stratifications, and Cartier operators}
\NotationRow{$P_{(\M)}^\mr{log}$, $\mcP_{(\M)}$}
  {The logarithmic $\M$-PD envelope of the diagonal and its structure sheaf.}
  {\S\,\ref{SS040}}
\NotationRow{$P_{(\M)}^{j,\mr{log}}$, $\mcP_{(\M)}^j$}
  {The truncation of the $\M$-PD envelope at order $j$, and its structure sheaf.}
  {\S\,\ref{SS040}}
\NotationRow{$\mcD^{(\M)}$, $\mcD_{\leq j}^{(\M)}$}
  {Logarithmic differential operators of level $\M$, and those of order at most $j$, on $X^\mr{log}/S^\mr{log}$.}
  {\S\,\ref{SS040}}
\NotationRow{${^L}\mcD_{\leq j}^{(\M)}$, ${^R}\mcD_{\leq j}^{(\M)}$}
  {The same sheaf with the $\mcO_X$-module structure given by left or right multiplication, respectively.}
  {\S\,\ref{SS040}}
\NotationRow{$\eta^{\{j\}}$, $\partial^{\langle j\rangle}$}
  {The local $\M$-PD monomial and its dual logarithmic differential operator.}
  {\S\,\ref{SS04f4}}
\NotationRow{$q_j$}
  {The quotient $\lfloor j/p^\M\rfloor$ used in the multiplication and Leibniz formulas for level-$\M$ operators.}
  {\S\,\ref{SS04f4}}
\NotationRow{$(\mcF,\nabla^{(\M)})$}
  {A left $\mcD^{(\M)}$-module; at level $0$, the structure is identified with an $S^\mr{log}$-connection.}
  {\S\,\ref{SS04f4}}
\NotationRow{$\nabla^{(\M')\Rightarrow(\M)}$}
  {Reduction of a $\mcD^{(\M')}$-module structure to level $\M\leq\M'$.}
  {\S\,\ref{SS04f4}}
\NotationRow{$\mcD_+^{(\M)}$, $\mcS ol(\nabla)$}
  {The kernel of the projection $\mcD^{(\M)}\to\mcO_X$, and the subsheaf of sections annihilated by this kernel.}
  {\S\,\ref{SS04f4}}
\NotationRow{$\nabla_{X,\mr{triv}}^{(\M)}$}
  {The natural $\mcD^{(\M)}$-module structure on $\mcO_X$.}
  {\S\,\ref{SS04f4}}
\NotationRow{$\phi=\{\phi_j\}_{j\geq0}$}
  {An $\M$-PD stratification on a $G$-bundle $\mcE$.}
  {Def.\,\ref{D057}}
\NotationRow{$\mcO_\mcE$, $\phi^\natural$, $\phi^{\natural\natural}$}
  {The $\mcO_X$-algebra of the affine $X$-scheme $\mcE$, the induced stratification on it, and the corresponding $\mcD^{(\M)}$-module structure.}
  {\S\,\ref{SS0239}}
\NotationRow{$\psi(\nabla)$, $\psi(\phi)$}
  {The $p^{\M+1}$-curvature of a $\mcD^{(\M)}$-module or $\M$-PD stratification. Its vanishing defines dormancy in characteristic $p$.}
  {Def.\,\ref{dGGd9}}
\NotationRow{$\nabla_{\mcG,\mr{can}}^{(\M)}$}
  {The canonical $\mcD^{(\M)}$-module structure on $F_{X/S}^{(\M+1)*}(\mcG)$.}
  {Def.\,\ref{dGGe11}}
\NotationRow{$\mcF^{[a]}$, $\nabla^{[a]}$}
  {The successive sheaves of horizontal sections on $X^{(a)}$ and their induced connections; $\mcF^{[\M+1]}=\mcS ol(\nabla^{(\M)})$.}
  {\S\,\ref{SS0ddd58}}
\NotationRow{$C_{(\mcF,\nabla^{(\M)})}$}
  {The iterated Cartier operator from $\Omega\otimes\mcF$ to $\Omega^{(\M+1)}\otimes\mcS ol(\nabla^{(\M)})$.}
  {Def.\,\ref{DDDf}}
\NotationRow{$\Dual_{\mr{Ker}}$, $\Dual_{\mr{Coker}}$}
  {The dual cohomological map induced by inclusion of horizontal sections, and the corresponding Cartier map under Grothendieck--Serre duality.}
  {\S\,\ref{SS44534}}

\NotationGroup{Diagonal reduction, liftings, and local models}
\NotationRow{$\msF=(\mcF,\nabla)$, $\mr{Diag}_\M$}
  {A flat module, and the inductively defined set of compatible pairs $(\msF,\nabla_0^{(\M)})$ used to define dormancy modulo $p^{\M+1}$.}
  {\S\,\ref{SS010}}
\NotationRow{$(\mcV_\msF,\nabla_\msF)$}
  {The flat module on $X_0^{(\M)}$ associated to a flat module modulo $p^{\M+1}$ whose reduction modulo $p^\M$ is dormant.}
  {\S\,\ref{SS010}}
\NotationRow{${^{\Diag}}\!\!\msF$, ${^{\Diag}}\!\!\msE$}
  {Diagonal reduction of a dormant flat module or a dormant flat $\mr{PGL}_n$-bundle to a level-$\M$ object in characteristic $p$.}
  {Def.\,\ref{D010}}
\NotationRow{$\mr{Lift}(\nabla_0^{(\M)})$, $\mr{Lift}(\nabla_0^{(\M)})_\msL$}
  {Liftings of a differential-module structure from level $\M$ to level $\M+1$, without or with the prescribed determinant $\msL$.}
  {\S\,\ref{S21}}
\NotationRow{$\mr{Lift}(\msF)$, $\mr{Lift}(\msF)_\msL$}
  {Liftings of a flat bundle from modulo $p^{\M+1}$ to modulo $p^{\M+2}$, without or with prescribed determinant.}
  {\S\,\ref{S21}}
\NotationRow{$\mr{obs}(\mcL)$}
  {The obstruction class $H^1(d\mr{log})([\mcL])$ in $H^1(X,\Omega)$ to the existence of a connection on $\mcL$.}
  {\S\,\ref{SSrrfe}}
\NotationRow{$U_\oslash$, $U_\otimes$}
  {The local models $\mr{Spec}(R[\![t]\!])$ and $\mr{Spec}(R[\![t_1,t_2]\!]/(t_1t_2))$ at a marked point and a node.}
  {\S\,\ref{SS1071}}
\NotationRow{$\mcB_S$, $\partial_\mcB^{\langle j\rangle}$}
  {The algebra encoding local monodromy and its distinguished basis.}
  {\S\,\ref{SS1076}}
\NotationRow{$\mu(\nabla)$, $\mu(\nabla)^{\langle\bullet\rangle}$}
  {The local monodromy operator and its tuple of components indexed by $1,p,\ldots,p^\M$.}
  {Def.\,\ref{D98}}
\NotationRow{$\widetilde d$, $\widetilde d_{[a]}$, $d_{[a]}$}
  {The representative of $d\in\mbZ/p^{\M+1}\mbZ$ in $[0,p^{\M+1}-1]$, its $a$-th base-$p$ digit, and that digit viewed in $\mbF_p$.}
  {\S\,\ref{SS182}}
\NotationRow{$\widetilde d_{[0,a]}$, $d_{[0,a]}$}
  {The truncation $\sum_{b=0}^a p^b\widetilde d_{[b]}$, and its class modulo $p^{a+1}$.}
  {\S\,\ref{SS182}}
\NotationRow{$\msO_{\oslash,d}^{(\M)}$, $\msO_{\otimes,d}^{(\M)}$}
  {The standard rank-one dormant level-$\M$ local models at a marked point and a node.}
  {\S\,\ref{SS182}}
\NotationRow{$e(\nabla^{(\M)})$}
  {The multiset of exponents classifying the local decomposition into the standard rank-one models.}
  {\S\,\ref{SS203}}
\NotationRow{$\msO_{\oslash,d}$}
  {The rank-one flat local model on $\mcO_\oslash$ whose connection has logarithmic residue $-d$.}
  {\S\,\ref{SS034}}

\NotationGroup{Opers, projective connections, and their moduli}
\NotationRow{$\msX=(f:X\to S,\{\sigma_i\}_i)$}
  {A pointed stable curve, with structure morphism $f$ and marked sections $\sigma_i$.}
  {\S\,\ref{QR780}}
\NotationRow{$\overline{\mcM}_{g,r}$, $\mcM_{g,r}$}
  {The moduli stack of $r$-pointed stable curves of genus $g$, and its open substack of smooth curves; the base is specified when needed.}
  {\S\,\ref{QR780}}
\NotationRow{$B\subseteq\mr{PGL}_n$}
  {The Borel subgroup given by the images of upper triangular matrices.}
  {\S\,\ref{DDe345}}
\NotationRow{$\msE^\spadesuit=(\mcE_B,\phi)$}
  {A $\mr{PGL}_n^{(\N)}$-oper; $\mcE_B$ is a $B$-bundle and $\phi$ is an $(\N-1)$-PD stratification on $\mcE:=\mcE_B\times^B\mr{PGL}_n$.}
  {Def.\,\ref{D39}}
\NotationRow{$\msF^\heartsuit=(\mcF,\nabla,\{\mcF^j\}_j)$}
  {A $\mr{GL}_n^{(\N)}$-oper, expressed as a filtered $\mcD^{(\N-1)}$-module.}
  {Def.\,\ref{D01234}}
\NotationRow{$\vartheta=(\varTheta,\nabla_\vartheta)$}
  {An $n^{(\N)}$-theta characteristic: a line bundle $\varTheta$ together with a level-$(\N-1)$ structure on $\mcT^{\otimes n(n-1)/2}\otimes\varTheta^{\otimes n}$.}
  {Def.\,\ref{D41}}
\NotationRow{$\mcF_\varTheta$, $\mcF_\varTheta^j$}
  {The bundle $\mcD_{\leq n-1}^{(\N-1)}\otimes\varTheta$ and its filtration by $\mcD_{\leq n-j-1}^{(\N-1)}\otimes\varTheta$.}
  {\S\,\ref{SS055}}
\NotationRow{$\nabla^\diamondsuit$}
  {A $(\mr{GL}_n^{(\N)},\vartheta)$-oper on $\mcF_\varTheta$, with determinant $\nabla_\vartheta$.}
  {Def.\,\ref{D42}}
\NotationRow{$\mpD^\clubsuit$}
  {An $n$-th order operator $\varTheta^\vee\to\Omega^{\otimes n}\otimes\varTheta^\vee$ with unit principal symbol; its prescribed subprincipal symbol defines an $(n,\vartheta)$-projective connection.}
  {\S\,\ref{SS0454}}
\NotationRow{$\Sigma_1(\mpD)$, $\Sigma_2(\mpD^\clubsuit)$}
  {The principal symbol of a differential operator and the subprincipal symbol defined via the determinant of its associated connection.}
  {Defs.\,\ref{Def898}, \ref{WW288}}
\NotationRow{$(-)^\ZZZ$}
  {The superscript marking dormant objects or their moduli, as in $\mcO p_\spadesuit^\ZZZ$.}
  {\S\,\ref{SS041}}
\NotationRow{$\mcO p_\spadesuit$, $\mcO p_\spadesuit^\ZZZ$}
  {The functors of isomorphism classes of $\mr{PGL}_n^{(\N)}$-opers and dormant opers on a fixed log curve, allowing base change.}
  {\S\,\ref{SS041}}
\NotationRow{$\mcO p_{\diamondsuit,\vartheta}$, $\mcO p_{\diamondsuit,\vartheta}^\ZZZ$}
  {The functors of $(\mr{GL}_n^{(\N)},\vartheta)$-opers and their dormant subfunctors.}
  {\S\,\ref{SS100}}
\NotationRow{$\mcO p_\heartsuit$, $\mcO p_\heartsuit^\ZZZ$}
  {The \'{e}tale sheafifications of the functors of $\mr{GL}_n^{(\N)}$-opers modulo tensoring with invertible differential modules, with both objects and twists dormant in the second case.}
  {\S\,\ref{SS042}}
\NotationRow{$\Lambda_{\heartsuit\Rightarrow\spadesuit}$, $\Lambda_{\diamondsuit\Rightarrow\spadesuit,\vartheta}$, $\Lambda_{\clubsuit\Rightarrow\diamondsuit,\vartheta}$}
  {The comparison maps between the descriptions by filtered vector bundles, projective bundles, fixed-theta opers, and projective connections.}
  {\S\S\,\ref{SS042}, \ref{SS100}, \ref{SS0454}}
\NotationRow{$\mcO p_{n,\N,g,r,R}$, $\mcO p_{n,\N,g,r,R}^\ZZZ$}
  {The universal moduli categories of level-$\N$ $\mr{PGL}_n$-opers and dormant opers over pointed stable curves over $R$.}
  {\S\,\ref{SSuniv}}
\NotationRow{$\Pi_{n,\N,g,r,R}$, $\Pi_{g,r}$}
  {The projection from the dormant-oper moduli category to $\overline{\mcM}_{g,r}$, in full and abbreviated notation.}
  {\S\,\ref{SSuniv}}
\NotationRow{${^{\Diag}}\!\!\msE^\spadesuit$}
  {The diagonal reduction of a dormant $\mr{PGL}_n$-oper modulo $p^\N$ to a dormant $\mr{PGL}_n^{(\N)}$-oper modulo $p$.}
  {Def.\,\ref{YY300}}
\NotationRow{${^\dagger}\msE^\spadesuit$, $\mr{nor}_{\msE^\spadesuit}$}
  {The normal oper associated to $\msE^\spadesuit$ and its canonical normalizing isomorphism.}
  {\S\,\ref{SSnormal}}

\NotationGroup{Radii, clutching, and the dormant TQFT}
\NotationRow{$\mu_i(\nabla)$, $e_i(\nabla)$}
  {The monodromy operator and exponent multiset at the marked section $\sigma_i$.}
  {Defs.\,\ref{D19}, \ref{D30}}
\NotationRow{$\Delta$, $\pi_\Delta$}
  {The diagonal subgroup of $(\mbZ/p^\CL\mbZ)^{\times n}$ and the quotient by simultaneous translation of all entries, where $\CL=\CH+\N$.}
  {\S\,\ref{SS044}}
\NotationRow{$\rho_i(\msE^\spadesuit)$, $\rho=(\rho_i)_i$}
  {The radius at $\sigma_i$, obtained from the exponent multiset modulo simultaneous translation, and the tuple of radii.}
  {Def.\,\ref{epaddd}}
\NotationRow{$\widetilde\Xi_{n,\CL}$, $\Xi_{n,\CL}$}
  {Exponent multisets whose reductions modulo $p$ are pairwise distinct, and their images under $\pi_\Delta$, respectively.}
  {\S\,\ref{SS044}}
\NotationRow{$\Xi_{2,\N}$}
  {The radius set in rank $2$, identified with $(\mbZ/p^\N\mbZ)^\times/\{\pm1\}$.}
  {\S\,\ref{SS044}}
\NotationRow{$\rho^\veebar$}
  {The radius obtained by negating every exponent; for tuples of radii, this operation is applied componentwise.}
  {\S\,\ref{QG1090}}
\NotationRow{$\mcO p_{n,\N,\rho,g,r,R}^\ZZZ$, $\Pi_{\rho,g,r}$}
  {The dormant-oper moduli category with prescribed radii $\rho$, and its projection to the moduli stack of curves.}
  {\S\,\ref{SS044}}
\NotationRow{$\GR=(V,E,\zeta)$}
  {A semi-graph with vertex set $V$, edge set $E$, and incidence map $\zeta$.}
  {Def.\,\ref{Defui3}}
\NotationRow{$B_v$, $B_\circledast$}
  {The branches incident to $v$, and the external branches, respectively.}
  {\S\,\ref{SS1070}}
\NotationRow{$\mbG$}
  {Clutching data: a semi-graph, the genera and numbers of markings of its components, and numberings of the incident branches.}
  {Def.\,\ref{Def112}}
\NotationRow{$\rho_\mbG=\{\rho^j\}_j$}
  {Radii assigned to the components of the clutching data, satisfying the matching condition at each glued pair of branches.}
  {Def.\,\ref{y0175}}
\NotationRow{$\mr{Clut}_\mbG$}
  {The clutching morphism on moduli stacks of pointed stable curves associated to $\mbG$.}
  {\S\,\ref{QG1090}}
\NotationRow{$\mr{Clut}_{\mbG,\rho_\mbG}$}
  {The corresponding clutching morphism for dormant opers with compatible component radii $\rho_\mbG$.}
  {\S\,\ref{QG1090}}
\NotationRow{$\msP=(\mbP/S,\{[0],[1],[\infty]\})$}
  {The standard three-pointed projective line.}
  {\S\,\ref{SS30d2}}
\NotationRow{$\varepsilon$}
  {The distinguished radius $\pi_\Delta([\overline0,\ldots,\overline{n-1}])$, which is the unit of the dormant fusion ring.}
  {\S\,\ref{SS060}}
\NotationRow{$\mcZ_{n,\N}$, $K^{\Xi_{n,\N}}$}
  {The dormant-oper $2$d TQFT and its state space, with distinguished basis indexed by $\Xi_{n,\N}$.}
  {\S\,\ref{SS322}}
\NotationRow{$\mbM_g^{r\Rightarrow s}$}
  {The connected oriented genus-$g$ cobordism with $r$ incoming and $s$ outgoing boundary circles.}
  {\S\,\ref{SS322}}
\NotationRow{$F_{n,\N}$, $\Fus_{n,\N}$}
  {The fusion rule determined by the genus-zero generic degrees and its associated dormant fusion ring.}
  {\S\,\ref{SS064}}
\NotationRow{$\mfS_{n,\N}$, $\mr{Cas}_{n,\N}$}
  {The ring homomorphisms $\Fus_{n,\N}\to\mbC$, and the element $\sum_{\lambda\in\Xi_{n,\N}}\lambda*\lambda^\veebar$ of the fusion ring.}
  {\S\,\ref{SS064}}

\NotationGroup{Deformation theory, ordinariness, and canonical liftings}
\NotationRow{$\mfg_\mcE$, $\mfb_{\mcE_B}$}
  {The adjoint bundles associated to the $\mr{PGL}_2$-bundle $\mcE$ and its $B$-reduction $\mcE_B$.}
  {\S\,\ref{SS078}}
\NotationRow{$\nabla_\phi^\mr{ad}$}
  {The $\mcD^{(\N-1)}$-module structure on $\mfg_\mcE$ induced by the adjoint representation.}
  {\S\,\ref{SS078}}
\NotationRow{$\widetilde{\mcT}_{\mcE^\mr{log}/S^\mr{log}}$, $d_\mcE$}
  {The sheaf of invariant logarithmic vector fields on $\mcE$, and its projection to $\mcT$.}
  {\S\,\ref{SS078}}
\NotationRow{$\nabla_\phi^{\mr{ad}(\N'-1)}$}
  {The difference map obtained from the adjoint stratification through order $p^{\N'-1}$; its kernel gives the horizontal adjoint sheaf at that level.}
  {\S\,\ref{SS07ff8}}
\NotationRow{$\widetilde\nabla_\phi^{\mr{ad}(\N'-1)}$, $\widetilde\nabla_{\phi,B}^{\mr{ad}(\N'-1)}$}
  {The extension of the preceding difference map to invariant vector fields, and its restriction to the $B$-bundle.}
  {\S\,\ref{SS07ff8}}
\NotationRow{$\mfg_{\mcE,\phi}$}
  {The horizontal adjoint sheaf $\mcS ol(\nabla_\phi^\mr{ad})$, regarded as a sheaf on $X^{(\N)}$.}
  {\S\,\ref{SS07ff8}}
\NotationRow{$\chi$, $\chi_{\phi\cap B}$}
  {The involution on invariant vector fields associated to a normal oper, and its induced isomorphism between the horizontal adjoint and $B$-tangent sheaves.}
  {\S\,\ref{SS07ff8}}
\NotationRow{$D_{\msE^\spadesuit}$, $D'_{\msE^\spadesuit}$}
  {The cohomological maps to $\mbR^1f_*(\mcT)$ used to describe the differential of the moduli projection and to define ordinariness.}
  {\S\S\,\ref{SS07ff8}, \ref{SS0d34}}
\NotationRow{$\mr{Def}_{(\msX,\msE^\spadesuit)}$}
  {The deformation functor of a pointed stable curve together with a dormant $\mr{PGL}_2^{(\N)}$-oper.}
  {\S\,\ref{SS068}}
\NotationRow{$\mcO p_{g,r,\mr{ord}}^\ZZZ$}
  {The open substack of ordinary dormant $\mr{PGL}_2^{(\N)}$-opers.}
  {\S\,\ref{SS0d34}}
\NotationRow{$\overline{\mcM}_{g,r,N\text{-}\mr{ord}}$}
  {The locus of $N$-ordinary pointed stable curves.}
  {Def.\,\ref{DD456}}
\NotationRow{$\mr{Op}_{1,X}^\ZZZ$, $\mr{Op}_{\N,X_0}^\ZZZ$}
  {Isomorphism classes of dormant $\mr{PGL}_2$-opers on $X/S$ and dormant level-$\N$ opers on its reduction $X_0/S_0$.}
  {\S\,\ref{SS1300}}
\NotationRow{$\Diag_{\!\!\spadesuit}$}
  {The diagonal-reduction map between these sets; its inverse gives canonical diagonal liftings when $X_0/S_0$ is $\N$-ordinary.}
  {\S\,\ref{SS1300}}
\NotationRow{$(\mr{PGL}_2^{(\N)})_X$}
  {The sheaf of automorphisms of the trivial $\mr{PGL}_2$-bundle preserving its trivial $(\N-1)$-PD stratification.}
  {\S\,\ref{SS101ff}}
\NotationRow{$\mcP^{\text{\'{e}t}}$, $\mcS^\ST$}
  {The sheaf of \'{e}tale maps to $\mbP$, and an $F^\N$-projective structure, viewed as a torsor subsheaf of it.}
  {\S\,\ref{SS101ff}}
\NotationRow{$F^\N\text{-}\mr{Pr}_X$}
  {The set of $F^\N$-projective structures on $X/S$.}
  {\S\,\ref{SS101ff}}

\NotationGroup{Hypergeometric operators and combinatorial counting}
\NotationRow{$W_\N$}
  {The ring of Witt vectors of length $\N$ over the fixed algebraically closed field $k$ of characteristic $p$.}
  {\S\,\ref{SS33}}
\NotationRow{$\mpD_{a,b,c}^\clubsuit$}
  {The monic Gauss hypergeometric differential operator with parameters $(a,b,c)$ over $W_\N$.}
  {\S\,\ref{SS33}}
\NotationRow{$\widehat\mpD_{a,b,c}^\clubsuit$, $R_\N$}
  {The restriction of $x(1-x)\mpD_{a,b,c}^\clubsuit$ to the formal punctured disc, and the ring of horizontal functions on that disc.}
  {\S\,\ref{SS17}}
\NotationRow{$\mr{Ex}_{a,b,c}$}
  {The triple of exponent differences $(1-c,c-a-b,b-a)$ at $0,1,\infty$, respectively.}
  {\S\,\ref{SS33}}
\NotationRow{$\mbZ_{[u,v]}$, $[a]_\N$, $[a]'_\N$}
  {The integer interval $\{u,\ldots,v\}$, and the representatives of $a$ modulo $p^\N$ in $[0,p^\N-1]$ and $[1,p^\N]$, respectively.}
  {\S\,\ref{SS30}}
\NotationRow{$\xi_\N$, $\zeta_\N$}
  {The maps to radius triples given by $(a,b,c)\mapsto((1-c)/2,(c-a-b)/2,(b-a)/2)$ and $(s_i)_i\mapsto((2s_i+1)/2)_i$, modulo $p^\N$ and signs.}
  {\S\,\ref{SS30}}
\NotationRow{$\tau'_{\N\Rightarrow\N'}$, $\tau_{\N\Rightarrow\N'}$}
  {Componentwise reduction of triples using the positive representatives $[-]'_{\N'}$ and the nonnegative representatives $[-]_{\N'}$, respectively.}
  {\S\,\ref{SS30}}
\NotationRow{$B_\N$}
  {Triples $(a,b,c)\in\mbZ_{[1,p^\N]}^3$ satisfying $a<c\leq b$ or $b<c\leq a$.}
  {\S\,\ref{SS30}}
\NotationRow{$C_\N$}
  {Nonnegative integer triples with $s_1+s_2+s_3\leq p^\N-2$ and the triangle inequalities $s_i\leq s_j+s_k$.}
  {\S\,\ref{SS30}}
\NotationRow{${^\dagger}B_\N$, ${^\dagger}C_\N$}
  {The subsets of $B_\N$ and $C_\N$ satisfying the corresponding compatibility conditions at every lower level.}
  {\S\,\ref{SS30}}
\NotationRow{$(u)_j$, $P(x)$, $Q(x)$}
  {The rising factorial, and the two truncated hypergeometric series used to construct solutions.}
  {\S\,\ref{SS17}}
\NotationRow{$\mr{Ed}_{p,\N,\mbG}$, $\mr{Ed}_{p,\N,\mbG,\rho}$}
  {Balanced $(p,\N)$-edge numberings on $\mbG$, with unrestricted or prescribed radii, respectively.}
  {\S\,\ref{SS10220}}
\NotationRow{$\langle s\rangle_{+1}$, $\langle s\rangle_{-1}$}
  {The functions $s$ and $1-s$ on $[0,1]$, used to express the inequalities for the counting regions.}
  {\S\,\ref{SS10223}}
\NotationRow{$\mr{Sgn}_\N$, $\mr{Sgn}_{\N,\mbG}$}
  {Sign arrays indexed by branches of one trivalent vertex, or by edges of $\mbG$, and levels $1,\ldots,\N$, with all signs at level $\N$ equal to $+1$.}
  {\S\,\ref{SS10223}}
\NotationRow{$\mcP_1(\mpa)$, $\mcP_2(\mpa)$}
  {The rational constructible sets encoding the first digit and the remaining digits at a trivalent vertex, for the sign array $\mpa$.}
  {\S\,\ref{SS10223}}
\NotationRow{$L_\mpa$, $\Phi_\mpa$, $\Phi_\N$}
  {The lattice-point set for the dilated regions $(p-1)\mcP_1(\mpa)$ and $p\mcP_2(\mpa)$, its digit-assembly map, and the resulting bijection onto ${^\dagger}C_\N$.}
  {\S\,\ref{SS10223}}
\NotationRow{$\mcP_{1,\mbG}(\mpa)$, $\mcP_{2,\mbG}(\mpa)$}
  {The graph-wide counting regions obtained by imposing the local inequalities at every vertex of $\mbG$.}
  {\S\,\ref{SS10223}}
\NotationRow{$i_\mcP(m)$, $H_\mcP(t)$, $\mr{Per}(\mcP)$}
  {The lattice-point count $\sharp(m\mcP\cap L)$, its Ehrhart quasi-polynomial, and the minimum period of that quasi-polynomial.}
  {\S\,\ref{SS10223}}
\NotationRow{$H_{\N,\mbG}(t)$}
  {The counting quasi-polynomial satisfying $H_{\N,\mbG}(p)=\sharp(\mr{Ed}_{p,\N,\mbG})=\mr{deg}(\Pi_{2,\N,g,r,\mbF_p})$ for odd primes $p$.}
  {Thm.\,\ref{Thm5}}
\NotationRow{$\mr{Diff}_{2,\mcL}^\mr{full}$}
  {Second-order differential operators on $\mcL$ with unit principal symbol and a full set of root functions.}
  {\S\,\ref{SSfff}}
\NotationRow{$\mr{SU}_X^n$, $\mr{Ver}_X^n$}
  {The moduli space of rank-$n$ stable bundles with trivial determinant, and the generalized Verschiebung rational map induced by Frobenius pull-back.}
  {\S\,\ref{SS101}}

\end{longtable}
\endgroup

\LSP
\subsection*{Acknowledgements} 
This paper grew out of the final research project that the author undertook during his appointment as an assistant professor at Tokyo Institute of Technology.
The author is grateful to the faculty and staff for their support and encouragement throughout his time there.

Finally, the author was partially  supported by JSPS KAKENHI Grant Numbers 21K13770 and 25K06933.


\end{document}